\documentclass[a4paper,DIV=16,twoside,abstracton,openany,BCOR=1.5cm,headsepline]{scrreprt}
\usepackage[dvips]{graphicx}
\usepackage{calrsfs}
\usepackage{setspace}
\usepackage{amsfonts}
\usepackage{amssymb}
\usepackage{amsthm}
\usepackage{amsmath}
\usepackage[shortcuts]{extdash}
\usepackage{color}
\usepackage[matrix,arrow,curve]{xy}
\allowdisplaybreaks

\newcommand{\uline}[1]{{\hbox to 0pt{\underline{\hphantom{{#1}}}\hss}{#1}}}

\newtheorem{theorem}{Theorem}[chapter]
\newtheorem{lemma}[theorem]{Lemma}
\newtheorem{remark}[theorem]{Remark}
\newtheorem{corollary}[theorem]{Corollary}
\newtheorem{proposition}[theorem]{Proposition}

\theoremstyle{definition}
\newtheorem{definition}[theorem]{Definition}

\newcommand{\QQ}{\mathbb Q}
\newcommand{\CC}{\mathbb C}
\newcommand{\ZZ}{\mathbb Z}
\newcommand{\NN}{\mathbb N}
\newcommand{\PP}{\mathbf P}
\newcommand{\OO}{\mathcal O}
\newcommand{\X}{\mathfrak X}
\newcommand{\6}{\sigma}
\newcommand{\sck}{{\sigma^\vee}}
\newcommand{\Pol}{\mathop{\mathrm{Pol}}}
\newcommand{\Div}{\mathop{\mathrm{CaDiv}}}
\newcommand{\Der}{\mathop{\mathrm{Der}}}
\newcommand{\Hom}{\mathop{\mathrm{Hom}}\nolimits}
\newcommand{\Tor}{\mathop{\mathrm{Tor}}\nolimits}
\newcommand{\coker}{\mathop{\mathrm{coker}}}

\newcommand{\im}{\mathop{\mathrm{im}}}
\newcommand{\tail}{\mathop{\mathrm{tail}}}
\newcommand{\eval}{\mathop{\mathrm{eval}}\nolimits}
\newcommand{\codim}{\mathop{\mathrm{codim}}\nolimits}
\newcommand{\ord}{\mathop{\mathrm{ord}}\nolimits}
\newcommand{\orb}{\mathop{\mathrm{orb}}\nolimits}
\renewcommand{\div}{\mathop{\mathrm{div}}\nolimits}
\newcommand{\inv}{\mathrm{inv}}
\newcommand{\id}{\mathrm{id}}
\newcommand{\Spec}{\mathop{\mathrm{Spec}}}
\newcommand{\Span}{\mathop{\mathrm{Span}}\nolimits}

\newcommand{\vfield}[1]{{\partial/\partial{#1}}}

\newcommand{\rawdivisorbasis}[1]{e_{#1}}
\newcommand{\olrawdivisorbasis}[1]{\overline{e}_{#1}}
\newcommand{\wtrawdivisorbasis}[1]{\widetilde{e}_{#1}}
\newcommand{\rawdivisortransition}[1]{c_{#1}}
\newcommand{\abstractdependentgenerator}[1]{x_{#1}}
\newcommand{\abstractindependentgenerator}[1]{\check{x}_{#1}}
\newcommand{\dependentgeneratorsdegree}[1]{\mathbf x_{#1}}
\newcommand{\oldependentgeneratorsdegree}[1]{\overline{\mathbf x}_{#1}}
\newcommand{\wtdependentgeneratorsdegree}[1]{\widetilde{\mathbf x}_{#1}}

\newcommand{\stdpolyhedronletter}{\Delta}
\newcommand{\stdprimitivepolyhedronletter}{\Xi}
\newcommand{\totalpolyhedronletter}{\overline{\mathbf\Delta}}
\newcommand{\stdvectorfiledxletter}{w}
\newcommand{\stdvectorfiledpletter}{v}
\newcommand{\normalvertexcone}[2]{\mathcal N(#1,#2)}
\newcommand{\primitivelattice}[1]{\mathbf b(#1)}
\newcommand{\latticelength}[1]{|{#1}|}
\newcommand{\sckboundarybasis}[1]{\alpha_{#1}}
\newcommand{\uidegree}[1]{\beta_{#1}}
\newcommand{\uithreadfunction}[1]{h_{#1}}
\newcommand{\oluithreadfunction}[1]{\overline{h}_{#1}}
\newcommand{\wtuithreadfunction}[1]{\widetilde{h}_{#1}}
\newcommand{\uismalltransition}[1]{C^\circ_{#1}}
\newcommand{\uilargetransition}[1]{C_{#1}}
\newcommand{\uidescriptionfunction}[1]{g_{#1}}
\newcommand{\uidescriptionfunctiondiff}[2]{g[#1]_{#2}}
\newcommand{\abstractvectorspace}{\nabla}
\newcommand{\stdpolynomialcoefficient}{a}
\newcommand{\stdpolynomialroot}{b}
\newcommand{\stdpolynomialrootmultiplicity}{c}
\newcommand{\coefficientspace}{V}
\newcommand{\defparpolynomial}{P}
\newcommand{\stdflatmorphism}{\xi}
\newcommand{\bigflatmorphism}{\boldsymbol{\xi}}
\newcommand{\bigtotalspace}{\mathbf S}
\newcommand{\stdbirationalgenerator}{y}
\newcommand{\stdbirationalmap}{b}

\newcommand{\numberofabstractgenerators}{n}
\newcommand{\numberoffixedgenerators}{\mathbf{n}}
\newcommand{\numberoflatticegenerators}{\mathbf{m}}
\newcommand{\numberofabstractopensets}{q}
\newcommand{\numberoffixedopensets}{\mathbf{q}}
\newcommand{\numberofdivisorpoints}{\mathbf{r}}
\newcommand{\numberofessentialdivisorpoints}{\mathbf{r}'}
\newcommand{\numberofprimitivepolyhedra}{\mathbf{R}}
\newcommand{\primitivepolyhedronmultiplicity}[1]{k_{#1}}
\newcommand{\primitivepolyhedronlocalmultiplicity}[2]{n_{#1,#2}}
\newcommand{\shiftstartingpoint}[2]{\ell_{#1,#2}}

\newcommand{\numberofvertices}[1]{\mathbf v(#1)}
\newcommand{\numberofverticespt}[1]{\mathbf v_{#1}}
\newcommand{\numberofverticesptdiff}[1]{\mathbf v'_{#1}}
\newcommand{\indexedvertex}[2]{\mathbf V_{#2}(#1)}
\newcommand{\indexedvertexpt}[2]{\mathbf V_{#1,#2}}
\newcommand{\indexededge}[2]{\mathbf E_{#2}(#1)}
\newcommand{\indexededgept}[2]{\mathbf E_{#1,#2}}
\newcommand{\opensetforvertex}[1]{\mathbf i_{#1}}

\newcommand{\twotorusdelimiter}{\chi_0}
\newcommand{\numberofedges}[1]{\mathbf e(#1)}
\newcommand{\numberofpositiveedges}[1]{\mathbf e^+(#1)}
\newcommand{\numberofnegativeedges}[1]{\mathbf e^-(#1)}
\newcommand{\indexedpositiveedge}[2]{\mathbf E^+_{#2}(#1)}
\newcommand{\indexednegativeedge}[2]{\mathbf E^-_{#2}(#1)}
\newcommand{\indexedfacet}[2]{\mathbf F_{#2}(#1)}
\newcommand{\indexedpositivefacet}[2]{\mathbf F^+_{#2}(#1)}
\newcommand{\indexednegativefacet}[2]{\mathbf F^-_{#2}(#1)}

\newcommand{\giinv}{{\mathcal G_{0,\Theta}^\inv}}
\newcommand{\givinv}{{\mathcal G_{1,\Theta,0}^\inv}}
\newcommand{\gvinv}{{\mathcal G_{0,\OO}^\inv}}
\newcommand{\gviiiinv}{{\mathcal G_{1,\OO,0}^\inv}}
\newcommand{\gvinvl}[1]{{\mathcal G_{0,\OO,#1}^\inv}}
\newcommand{\gviiiinvl}[1]{{\mathcal G_{1,\OO,0,#1}^\inv}}
\newcommand{\gi}{{\mathcal G_{0,\Theta}}}
\newcommand{\giiinv}{{\mathcal G_{1,\Theta,1}^\inv}}
\newcommand{\giipinv}{{\mathcal G_{1,\Theta,1}'^\inv}}
\newcommand{\giippinv}{{\mathcal G_{1,\Theta,1}''^\inv}}
\newcommand{\gii}{{\mathcal G_{1,\Theta,1}}}
\newcommand{\giip}{{\mathcal G_{1,\Theta,1}'}}
\newcommand{\giipp}{{\mathcal G_{1,\Theta,1}''}}
\newcommand{\giii}{{\mathcal G_{1,\Theta,2}}}
\newcommand{\giv}{{\mathcal G_{1,\Theta,0}}}
\newcommand{\gvl}[1]{{\mathcal G_{0,\OO,#1}}}
\newcommand{\gviinvl}[1]{{\mathcal G_{1,\OO,1,#1}^\inv}}
\newcommand{\gvipinvl}[1]{{\mathcal G_{1,\OO,1,#1}'^\inv}}
\newcommand{\gvippinvl}[1]{{\mathcal G_{1,\OO,1,#1}''^\inv}}
\newcommand{\gviinv}{{\mathcal G_{1,\OO,1}^\inv}}
\newcommand{\gvipinv}{{\mathcal G_{1,\OO,1}'^\inv}}
\newcommand{\gvippinv}{{\mathcal G_{1,\OO,1}''^\inv}}
\newcommand{\gvil}[1]{{\mathcal G_{1,\OO,1,#1}}}
\newcommand{\gvipl}[1]{{\mathcal G_{1,\OO,1,#1}'}}
\newcommand{\gvippl}[1]{{\mathcal G_{1,\OO,1,#1}''}}
\newcommand{\gviil}[1]{{\mathcal G_{1,\OO,2,#1}}}
\newcommand{\gviiil}[1]{{\mathcal G_{1,\OO,0,#1}}}
\newcommand{\gv}{{\mathcal G_{0,\OO}}}
\newcommand{\gviii}{{\mathcal G_{1,\OO,0}}}
\newcommand{\gvi}{{\mathcal G_{1,\OO,1}}}
\newcommand{\gvip}{{\mathcal G_{1,\OO,1}'}}
\newcommand{\gvipp}{{\mathcal G_{1,\OO,1}''}}
\newcommand{\giisi}[1]{{\mathcal G_{1,\Theta,1,#1}}}
\newcommand{\giipsij}[1]{{\mathcal G_{1,\Theta,1,#1}}}
\newcommand{\giippsijk}[1]{{\mathcal G_{1,\Theta,1,#1}}}
\newcommand{\giisiinv}[1]{{\mathcal G_{1,\Theta,1,#1}^\inv}}
\newcommand{\giipsijinv}[1]{{\mathcal G_{1,\Theta,1,#1}^\inv}}
\newcommand{\giippsijkinv}[1]{{\mathcal G_{1,\Theta,1,#1}^\inv}}
\newcommand{\gvisi}[1]{{\mathcal G_{1,\OO,1,#1}}}
\newcommand{\gvipsij}[1]{{\mathcal G_{1,\OO,1,#1}}}
\newcommand{\gvippsijk}[1]{{\mathcal G_{1,\OO,1,#1}}}
\newcommand{\gvisiinv}[1]{{\mathcal G_{1,\OO,1,#1}^\inv}}
\newcommand{\gvipsijinv}[1]{{\mathcal G_{1,\OO,1,#1}^\inv}}
\newcommand{\gvippsijkinv}[1]{{\mathcal G_{1,\OO,1,#1}^\inv}}
\newcommand{\giicircinv}{{\mathcal G_{1,\Theta,1}^{\circ\inv}}}
\newcommand{\gvicircinv}{{\mathcal G_{1,\OO,1}^{\circ\inv}}}

\newcommand{\gsi}[1]{{G_{0,\Theta,#1}}}
\newcommand{\gsii}[1]{{G_{1,\Theta,1,#1}}}
\newcommand{\gsiicirc}[1]{{G_{1,\Theta,1}^{\circ #1}}}
\newcommand{\gsiicircgen}[1]{{G_{1,\Theta,1}^{\circ #1}}}
\newcommand{\gsiicircgengen}{{G_{1,\Theta,1}^\circ}}
\newcommand{\gsvi}[1]{{G_{1,\OO,1,#1}}}
\newcommand{\gsvicirc}[1]{{G_{1,\OO,1}^{\circ #1}}}

\begin{document}
\title{Equivariant deformations of algebraic varieties with an action of an algebraic torus of complexity 1}
\author{Rostislav Devyatov}
\date{}
\maketitle

\begin{abstract}
Let $X$ be a 3-dimensional affine variety with a faithful action of a 2-dimensional
torus $T$. Then the space of first order infinitesimal deformations $T^1(X)$ is graded by the 
characters of $T$, and the zeroth graded component $T^1(X)_0$ consists of
all equivariant first order (infinitesimal) deformations.

Suppose that using the construction of such varieties from \cite{ahausen},
one can obtain $X$ from a proper polyhedral divisor $\mathcal D$ on $\mathbb P^1$ such that 
the tail cone of (any of) the used polyhedra is pointed and full-dimensional, 
and all vertices of all polyhedra are lattice points. Then we compute 
$\dim T^1(X)_0$ and find a formally versal equivariant deformation of $X$. 
We also establish a connection between our formula for 
$\dim T^1(X)_0$ and known formulas for the dimensions of the graded
components of $T^1$ of toric varieties.
\end{abstract}

\tableofcontents

\chapter{Introduction}
\section{T-varieties}
As proved and explained in \cite{ahausen}, normal affine varieties of dimension $d$ 
with a faithful action of a $k$-dimensional torus $T$ 
(which are called T-varieties in the sequel) are described by 
so-called proper polyhedral divisors. To define them, 
consider the character lattice $M=\X(T)$, the rational character lattice $M_\QQ=M\otimes_\ZZ\QQ$, 
the dual character lattice $N=\Hom_\ZZ(M,\ZZ)$, and the dual vector space (the dual rational character lattice) 
$N_\QQ=M_\QQ^*$.

A \textit{polyhedron} in a $\QQ$-vector space is the \underline{\smash{nonempty}} intersection of finitely many closed affine half-spaces.
A particular case of a polyhedron is a polyhedral cone, a polyhedron is called a \textit{polyhedral cone}
if it can be obtained as the intersection of finitely many closed \underline{linear} half-spaces, i.~e. 
the boundary of all these half-spaces should contain the origin.
If $\stdpolyhedronletter$ is a polyhedron in a $\QQ$-vector space $V$, its \textit{tail cone}
is defined as the set of vectors $v\in V$ such that for all $a\in\stdpolyhedronletter$
one has $v+a\in \stdpolyhedronletter$. It is denoted by $\tail(\stdpolyhedronletter)$. 
Fig.~\ref{figpolytrivial} shows an example of a polyhedron and of its tail cone.

\begin{figure}[!h]
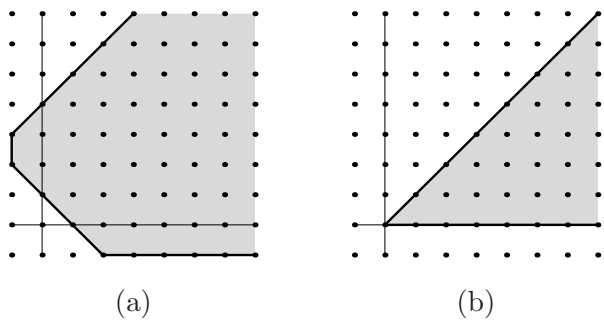

\begin{center}
\begin{tabular}{cc}
\hspace*{1em}\includegraphics{t1_3fold_figures-1.mps}\hspace*{1em}&\hspace*{1em}\includegraphics{t1_3fold_figures-2.mps}\hspace*{1em}\\[\medskipamount]
(a)&(b)
\end{tabular}
\end{center}
\caption{An example of (a) a polyhedron and (b) its tail cone.}
\label{figpolytrivial}
\end{figure}

All polyhedra in a given $\QQ$-vector space $V$ with a given tail cone $\6$ form 
a semigroup with the operation of Minkowski addition. $\6$ is the neutral element. 
Denote the Grothendick construction
for this semigroup by $\Pol_\6(V)$.

The next object we need to define to study $T$-varieties is a polyhedral divisor. Suppose that we have a normal variety $Y$.
A polyhedral divisor $\mathcal D$ is an element 
of the group $\Pol_\6(N)\otimes_\QQ\Div_\QQ(Y)$, where $\Div_\QQ$ 
is the group of $\QQ$-Cartier divisors.

Now we can say that a T-variety is determined by the following data:
\begin{enumerate}
\item A $(d-k)$-dimensional normal (not necessarily affine) variety $Y$.
\item A pointed cone $\6$ in the rational dual character lattice $N_\QQ=\X(T)^*_\QQ$.
\item A proper (see definition below, in Section \ref{tvarsreal}) polyhedral divisor $\mathcal D$.
\end{enumerate}

As we said, the definition of properness in the whole generality will be given later, 
but in the case we will need it now, namely when $Y=\PP^1$, it is easy to formulate an equivalent 
condition for properness. Namely, 
the polyhedral divisor $\mathcal D$ on $\PP^1$
is proper if and only if it can be written in the form 
$$\mathcal D=\sum_{i=1}^{\numberofdivisorpoints} p_i\otimes \stdpolyhedronletter_{p_i},$$
where $p_i\in\PP^1$ are points, and $\stdpolyhedronletter_{p_i}$ are \textit{polyhedra} 
(they should be "genuine" polyhedra, not elements of the Grothendick group), 
and the Minkowski sum of all polyhedra $\stdpolyhedronletter_{p_i}$
is strictly contained in $\6$.

The construction of a $T$-variety out of these data will be given in Section \ref{tvarsreal}.

\section{Deformations and first order deformations}

For a general reference on deformation theory, see \cite{hartdef}.

In general, a \textit{deformation} of a variety $X$ with a scheme $Z$ with a marked point $z\in Z$ being the 
\textit{parameter space} of the deformation is a flat morphism $\stdflatmorphism\colon Y\to Z$, where $Y$ is a scheme, together with 
an isomorphism $\iota$ between $X$ and $\stdflatmorphism^{-1}(z)$. 
Two deformations $(\stdflatmorphism\colon Y\to Z, \iota\colon X\to \stdflatmorphism^{-1}(z))$ and 
$(\stdflatmorphism'\colon Y'\to Z, \iota'\colon X\to \stdflatmorphism'^{-1}(z))$ with the same parameter space $Z$ 
and the same marked point $z$ are called equivalent if there exists an isomorphism $q\colon Z\to Z'$
such that $\stdflatmorphism=\stdflatmorphism'q$ and $q|_{\stdflatmorphism^{-1}(z)}\iota=\iota'$.

A deformation $(\stdflatmorphism\colon Y\to Z, \iota\colon X\to \stdflatmorphism^{-1}(z))$ 
with a torus action $T:Z$ is called \textit{equivariant}
if $\iota$ is $T$-equivariant and $\stdflatmorphism$ is $T$-invariant.

If $Z$ is the double point, i.~e. $Z=\Spec(\CC[\varepsilon]/\varepsilon^2)$,
and $X$ is affine ($X=\Spec A$), then the set of all possible deformations is denoted by $T^1(X)$, and 
one can define an $A$-module structure on it. See Section \ref{deformsreal} for details.

Deformations can be pulled back from one parameter space to another using fiber product.
In particular, if we have a vector space $Z$, then each tangent vector at the marked point
defines an embedding of the double point into $Z$. We can pullback the deformation from $Z$ 
to the double point and get an element of $T^1(X)$. So we get a map from the tangent space at the marked point 
to $T^1(X)$. In fact, this map is linear. It is called the Kodaira-Spencer map and is an important characteristic of 
the original deformation.

If $M$ is a lattice, and $A$ is an $M$-graded algebra, then $T^1(X)$ 
actually becomes a \textit{graded} $A$-module. Moreover, if $M=\mathfrak X(T)$, 
then an $M$-grading on $A$ is equivalent to a torus action on $\Spec A$, 
and the graded component of $T^1(X)$ of degree zero (it will be further denoted by $T^1(X)_0$)
contains exactly the equivariant deformations. (More precisely, $T^1(X)_0$
contains the set of deformations that can be made equivariant by the appropriate 
choice of a $T$-action on $Z$, but such a choice is unique up to an isomorphism of the deformation.)

\section{Problem setup and main results}\label{probsetup}
We start with a two-dimensional torus $T$. We choose a two-dimensional pointed cone $\6\subset N_{\QQ}$. 
Then we fix a proper polyhedral divisor $\mathcal D$ on $\PP^1$, where
$$\mathcal D=\sum_{i=1}^{\numberofdivisorpoints} p_i\otimes \stdpolyhedronletter_{p_i},$$
$p_i\in\PP^1$ are arbitrary points, and $\stdpolyhedronletter_{p_i}$ are (nonempty) polyhedra with tail cone $\6$.
Additionally, we suppose that all vertices of all polyhedra $\stdpolyhedronletter_{p_i}$ are 
lattice points.\footnote{The $T$-varieties obtained from 
polyhedral divisors without this "lattice point" condition can be obtained from the $T$-varieties under consideration 
by taking the quotient modulo a finite group action.}
In this case all divisors $\mathcal D(\chi)$ are integral, not rational.
The properness condition in this case 
means that the Minkowski sum of all polyhedra $\stdpolyhedronletter_{p_i}$ is strictly contained in $\6$. 
We are going to study the 3-dimensional variety $X$ with an action of the 2-dimensional torus $T$ defined as above 
by $\PP^1$, $\6$, and $\mathcal D$.
More specifically, we are going to find the dimension of the space of equivariant first order deformations of $X$, 
and to find a formally versal equivariant deformation space for $X$.

The dimension of $T^1(X)_0$ is computed in Chapters \ref{sectcohomformula} and \ref{combformula}. 
The answer is given by Theorem \ref{t1aslatticelength}. It can also be formulated as follows.

\begin{theorem}
We maintain the assumptions and the notation introduced above in this section. The dimension
of $T^1(X)_0$ is the sum of two summands:
\begin{enumerate}
\item 
The maximum of 0 and 
$$
-3+\#\{\text{points $p_i$ such that $\stdpolyhedronletter_{p_i}$ is not a shift of $\6$ (i.~e. $\stdpolyhedronletter_{p_i}$
has at least two vertices)}\}).
$$
\item The sum of the amounts of integral points \emph{inside} the \emph{finite part}
of the boundary of $\stdpolyhedronletter_{p_i}$ (i.~e. the boundary of $\stdpolyhedronletter_{p_i}$
except for the two rays). For example, if $\stdpolyhedronletter_{p_i}$ is a shift of $\6$, 
then the finite part of the boundary is just one vertex, and we count zero. 
If $\stdpolyhedronletter_{p_i}$ has one edge, which is a primitive lattice segment, 
then we still count zero since there are no lattice points inside it. But if, 
for example, $\stdpolyhedronletter_{p_i}$ has two edges, and both of them are primitive 
lattice segments, then we count one, because the vertex between these two edges is inside the finite part 
of the boundary.
\end{enumerate}
\end{theorem}

In Chapter \ref{sectksmtoric} we find a fromally versal deformation space for the equivariant deformations of $X$. 
The construction of the total space of this deformation requires more technical details and 
will be given in Section \ref{sectversalconstruction}, but the parameter space is just a vector space. 
In particular, it is smooth, so all equivariant first order deformations are unobstructed.

To prove that the deformation in question is formally versal, we will need to compute the 
Kodaira-Spencer map of a deformation defined by perturbation of generators 
of a \textit{subalgebra} of the polynomial algebra. See 
Section \ref{ksmgens} for more details.
The results about the Kodaira-Spencer map for such deformations may be of independent interest.

Some of these results were preliminarily announced in an arxiv.org preprint by the author, \cite{myarxiv}.

\chapter{Preliminaries}\label{prelimproofs}

\section{T-varieties and polyhedra}\label{tvarsreal}
We will need an explicit construction of a $T$-variety out of a polyhedral divisor.

First, we need one more definition concerning polyhedral cones. If $\6$ is a polyhedral cone in a $\QQ$-vector space $V$, 
its \textit{dual cone} is defined as the set of all vectors $w\in V^*$ such that for all $v\in V$ 
one has $w(v)\ge 0$.

Now we have to define the \textit{evaluation function} $\eval\colon \Pol_\6(V)\times \sck\to \QQ$ 
as follows: $\eval(\stdpolyhedronletter_1-\stdpolyhedronletter_2,f)=
\min_{v\in\stdpolyhedronletter_1}f(v)-\min_{v\in\stdpolyhedronletter_2}f(v)$ for all polyhedra 
$\stdpolyhedronletter_1, \stdpolyhedronletter_2$ with tail cone $\6$ and for all $f\in\sck$. 
One checks directly that this function is well-defined on $\Pol_\6(V)$, that
it is linear in the first argument and is piecewise-linear in the second argument. 
If we fix a polyhedron as the first argument (a real polyhedron, not an element of 
the Grothendick construction, i.~e. $\stdpolyhedronletter_2=\6$), then the resulting function
is also convex. If this polyhedron is of the form $\6+v$, where $v\in V$, 
then this function is linear, not just piecewise-linear. 
If $\stdpolyhedronletter$ is a polyhedron, we shortly call the function 
$\eval_\stdpolyhedronletter\colon \sck\to\QQ$ defined by $\eval_\stdpolyhedronletter(f)=\eval(\stdpolyhedronletter,f)$
the \textit{individual evaluation function of the polyhedron $\stdpolyhedronletter$}.

We are going to construct a $d$-dimensional variety with an action of a $k$-dimensional
torus $T$.
Suppose that we have a $(d-k)$-dimensional normal (not necessarily affine) variety $Y$,
a pointed cone $\6$ in the rational dual character lattice $N_\QQ=\X(T)^*_\QQ$, and
a polyhedral divisor $\mathcal D$ on $Y$.

For every element $\chi\in\sck\cap M$, $\mathcal D$ defines a rational divisor
$\mathcal D(\chi)$ as follows. Notice that $\chi$ can be considered as a function on $N$.
Let $\mathcal D=\sum a_i Z_i\otimes (\stdpolyhedronletter_i-\stdpolyhedronletter_i')$, where $a_i\in\QQ$, $Z_i$'s are irreducible 
hypersurfaces in $Y$, and $\stdpolyhedronletter_i$'s and $\stdpolyhedronletter_i'$'s are polyhedra with the tail cone $\6$. We put 
$\mathcal D(\chi):=\sum a_i\eval(\stdpolyhedronletter_i-\stdpolyhedronletter_i',\chi)Z_i=
\sum a_i(\min_{p\in \stdpolyhedronletter_i}\chi(p)-\min_{p\in \stdpolyhedronletter_i'}\chi(p))Z_i$.

\begin{definition}
A polyhedral divisor $\mathcal D$ is called \textit{principal}, if
it can be written in the form $\mathcal D=\sum \div(f_i)\otimes \alpha_i+\6$, where $f_i$'s are rational
functions on $Y$ and $\alpha_i\in N$.
\end{definition}

\begin{definition}
A polyhedral divisor $\mathcal D$ is called \textit{proper}, if
\begin{enumerate}
\item It can be written in the form $\mathcal D=\sum a_i Z_i\otimes \stdpolyhedronletter_i$, where $a_i\in\QQ$, 
efficient
Cartier divisors in $Y$, 
and $\stdpolyhedronletter_i$'s are polyhedra with the tail cone $\6$ \textbf{and $a_i\ge 0$}.
\item For every $\chi\in \sck\cap M$, $\mathcal D(\chi)$ is 
semiample,
and if $\chi$ is in the interior of $\sck$, $\mathcal D(\chi)$ is 
big.
\end{enumerate}
\end{definition}

Now, notice that if $\chi,\chi'\in \sck\cap M$, then $\mathcal D(\chi)+\mathcal D(\chi)-\mathcal D(\chi+\chi')$ is an effective divisor, so 
a product of (rational) functions from $\Gamma(\PP^1,\mathcal O(\mathcal D(\chi)))$ and from $\Gamma(\PP^1,\mathcal O(\mathcal D(\chi')))$ is in 
$\Gamma(\PP^1,\mathcal O(\mathcal D(\chi+\chi')))$. 
So we have a graded algebra 
$$
A=\bigoplus_{\chi\in\sck\cap M}\Gamma(\PP^1,\mathcal O(\mathcal D(\chi))).
$$
One can prove that if $\mathcal D$ is proper, this algebra is finitely generated. The T-variety in question is $X=\Spec A$. 
Since $A$ is graded, 
$T$ acts on $X$. 
If $\mathcal D$ is proper, $\dim X=d$. 

If we add a principal polyhedral divisor to $\mathcal D$, 
then $A$ will not change as a graded algebra, so $X$ will stay the same, and the action of the torus on $X$ 
will also stay the same.
Notice also that if $\chi,\chi'\in\sck\cap M$ are proportional, then $\mathcal D(\chi+\chi')=
\mathcal D(\chi)+\mathcal D(\chi')$, and in general the function $\chi\mapsto \mathcal D(\chi)$ is piecewise-linear.

Within the construction of $A$ we use, the elements of $\Gamma(\PP^1,\mathcal O(\mathcal D(\chi)))$ may be interpreted in two 
ways: they are rational functions on $Y$ and they are global algebraic functions on $X$. 
If $f\in \Gamma(\PP^1,\mathcal O(\mathcal D(\chi)))$, we will write $\overline f$ for a rational function on $Y$ 
and $\widetilde f$ for a global function on $X$.

\begin{proposition}\label{quotmorph}(see \cite[Theorem 3.1]{ahausen})
There exists a rational surjective map $\pi\colon X\to Y$ such that for every degree $\chi\in\sck\cap M$, for every 
point $x\in X$ such that $\pi$ is defined at $x$, and for every $f,g\in\Gamma(\PP^1,\mathcal O(\mathcal D(\chi)))$ the 
following conditions are equivalent:
\begin{enumerate}
\item $\overline f/\overline g$ is defined at $\pi (x)$ as a rational function.
\item $\widetilde f/\widetilde g$ is defined at $x$ as a rational function.
\end{enumerate}
In this case, $(\overline f/\overline g)(\pi (x))=(\widetilde f/\widetilde g)(x)$.
\end{proposition}

\section{Deformations}\label{deformsreal}

If $X=\Spec A$ is an affine algebraic variety, we will need to understand how to 
define an $A$-module structure on $T^1(X)$.
Namely, 
choose an embedding $X\hookrightarrow\CC^\numberofabstractgenerators$, then $A$ can be written as 
$A=\CC[\abstractindependentgenerator 1, \ldots, \abstractindependentgenerator{\numberofabstractgenerators}]/I$, where $I$ is an 
ideal. Then $I/I^2$ is an $A$-module. Consider also the following 
$\CC[\abstractindependentgenerator 1,\ldots, \abstractindependentgenerator{\numberofabstractgenerators}]$-module 
$\Theta=\Der \CC[\abstractindependentgenerator 1,\ldots, \abstractindependentgenerator{\numberofabstractgenerators}]$: its elements 
are of the form $\sum g_i\partial/\partial \abstractindependentgenerator i$, 
where $g_i\in\CC[\abstractindependentgenerator 1,\ldots, \abstractindependentgenerator{\numberofabstractgenerators}]$. Every such differential operator 
defines an $A$-homomorphism between $I/I^2$ and $A$: if $g\in I$, then $g/I^2\in I/I^2$ maps to 
$(\sum g_i\partial g/\partial \abstractindependentgenerator i)/I\in A$. If $g\in I^2$, $g=\sum g'_jg''_j$, then $\sum_{i,j} g_i\partial (g'_jg''_j)\partial \abstractindependentgenerator i
=\sum_{i,j}g_ig'_j\partial g''_j/\partial \abstractindependentgenerator i+\sum_{i,j}g_ig''_j\partial g'_j/\partial \abstractindependentgenerator i\in I$, so the map is well-defined. 
If $a\in\CC[\abstractindependentgenerator 1,\ldots, \abstractindependentgenerator{\numberofabstractgenerators}]$, $a/I \in A$, 
then $g\sum g_i\partial a/\partial \abstractindependentgenerator i\in I$, so 
$(\sum g_i\partial (ag)/\partial \abstractindependentgenerator i)/I=
(a\sum g_i\partial g/\partial \abstractindependentgenerator i)/I+(g\sum g_i \partial a/\partial \abstractindependentgenerator i)/I=
(a\sum g_i\partial g/\partial \abstractindependentgenerator i)/I$, and the map is $A$-linear. So in fact we have defined a 
map $\phi\colon \Theta\to\Hom_A(I/I^2,A)$.

Moreover, if $\sum g_i\partial/\partial \abstractindependentgenerator i\in I\Theta$, i.~e. if all $g_i$ are in $I$, then 
$\sum g_i\partial g/\partial \abstractindependentgenerator i\in I$ for all $g\in I$, so $\phi$ is well-defined on $\Theta/I\Theta$, which 
is an $A$-module. It is clear that $\phi$ is $A$-linear.

One can prove that $T^1(X)$ can be identified 
with $\coker\phi$ so that these identifications for all affine varieties together have good 
category-theoretical properties. We will not need these properties explicitly, and we will use 
this identification as a definition of $T^1(X)$. However, we will need to understand how the identification itself works 
exactly, because at some point a first order deformation will arise from a different source
(actually, as a restriction of a deformation over an affine line to a tangent vector at the origin), 
and we will need to understand how it is represented 
by an element of $\coker\psi$. Here is a brief description.

Suppose that we have a first order deformation with a total space $Y$.
It can be shown that $Y$ is an affine scheme.
Denote $B=\CC[Y]$. Then the flat morphism 
$\stdflatmorphism\colon Y\to \Spec \CC[\varepsilon]/\varepsilon^2$ means that $B$ is a 
$(\CC[\varepsilon]/\varepsilon^2)$-module, and the isomorphism $\iota$ determines an 
isomorphism $\iota^*\colon B/\varepsilon B\to A=\CC[X]$.
We keep the assumption that $A$ is generated by $\numberofabstractgenerators$ generators
$\abstractdependentgenerator1,\ldots,\abstractdependentgenerator{\numberofabstractgenerators}$, and 
that $I\subset \CC[\abstractindependentgenerator 1,\ldots, \abstractindependentgenerator{\numberofabstractgenerators}]$
is the ideal such that 
$\CC[\abstractindependentgenerator 1,\ldots, \abstractindependentgenerator{\numberofabstractgenerators}]/I=A$
and that this isomorphism maps $\abstractindependentgenerator i$ to $\abstractdependentgenerator i$ for each $i$.
The flatness of $\stdflatmorphism$ implies that 
there exist elements 
$\widetilde{\abstractdependentgenerator1}, \ldots, \widetilde{\abstractdependentgenerator{\numberofabstractgenerators}}\in B$
such that $\iota^*\widetilde{\abstractdependentgenerator i}=\abstractdependentgenerator i$
and that all elements 
$\varepsilon=\varepsilon\cdot 1, \widetilde{\abstractdependentgenerator1}, \ldots, \widetilde{\abstractdependentgenerator{\numberofabstractgenerators}}$
generate $B$.
It also follows from the flatness of $\stdflatmorphism$ that if $g\in I$, in other words, if 
$g$ is a polynomial in $\numberofabstractgenerators$ variables such that 
$g(\abstractdependentgenerator1,\ldots,\abstractdependentgenerator{\numberofabstractgenerators})=0$, 
then there exists a polynomial $g'\in \CC[\abstractindependentgenerator1,\ldots,\abstractindependentgenerator{\numberofabstractgenerators}]$
such that $g(\abstractdependentgenerator1,\ldots,\abstractdependentgenerator{\numberofabstractgenerators})=
\varepsilon g'(\abstractdependentgenerator1,\ldots,\abstractdependentgenerator{\numberofabstractgenerators})$.
Moreover, it can be shown that this $g'$ is unique modulo $I$. So, we have a well-defined map 
$I\to \CC[\abstractindependentgenerator1,\ldots,\abstractindependentgenerator{\numberofabstractgenerators}]/I=A$, 
and it can also be shown that it is well-defined on $I/I^2$ and is $A$-linear.
We say by definition that the isomorphism between $T^1(X)$ and $\coker\phi$ maps the 
deformation under consideration to the class of this map in $\coker\phi=\Hom_A(I/I^2,A)/\im\phi$.
One still has to show that this is really an isomorphism, but this is a known fact.
Note that a deformation itself only defines a class of a map $I/I^2\to A$ in $\coker\phi$, 
but if we lift the generators of $A$ to $B$, the map $I/I^2\to A$ itself
will be already uniquely determined (while it depends on the choice of the lift).

If $M$ is a lattice, $A$ is $M$-graded, and the generators $\abstractindependentgenerator 1,\ldots, \abstractindependentgenerator{\numberofabstractgenerators}$ are homogeneous,
then one has an $M$-grading on $\CC[\abstractindependentgenerator 1,\ldots, \abstractindependentgenerator{\numberofabstractgenerators}]$ as well. Then $I$ becomes an 
$M$-graded ideal, and $\Theta$ becomes an $M$-graded module with $\deg(\partial/\partial \abstractindependentgenerator i)=-\deg \abstractindependentgenerator i$.
The map $\phi$ preserves this grading, so we have a grading on $T^1(X)$.


\section{Schlessinger's formula for $T^1$}

Extending Schlessinger's result \cite[Lemma 2]{schless}, we prove the following theorem:
\begin{theorem}\label{schlessgen}
Let $X$ be an affine normal algebraic variety, and let $U$ be a non-singular open subset of $X$ 
such that $\codim_X(X\setminus U)\ge 2$. Then $T^1 (X)$ can be computed as follows.
Let $\Theta_X$ denote the tangent sheaf on $X$, and let $\abstractdependentgenerator 1, \ldots, \abstractdependentgenerator{\numberofabstractgenerators}\in\CC[X]$ be a set 
of generators. Consider the following map $\psi\colon\Theta_X\to\OO_X^{\oplus \numberofabstractgenerators}$: it maps 
a (locally defined) vector field 
$\stdvectorfiledxletter$ to 
$(d\abstractdependentgenerator1(\stdvectorfiledxletter), \ldots, d\abstractdependentgenerator{\numberofabstractgenerators}(\stdvectorfiledxletter))$.

Then $T^1(X)=\ker (H^1(U,\Theta_X)\stackrel{H^1(\psi|_U)}\longrightarrow H^1(U,\OO_X^{\oplus \numberofabstractgenerators}))$ as $\CC[X]$-modules.
\end{theorem}

The difference from Lemma 2 in \cite{schless} itself is the following. First, we speak about a normal affine variety $X$, 
while Lemma 2 in \cite{schless} speaks about local geometric schemes. Second, we allow $X$ to have any singularities 
as long as $X$ is normal, while Lemma 2 in \cite{schless} says that the singularity must be isolated. Finally,
here $U$ is an arbitrary smooth open subset of $X$ such that $\codim_X(X\setminus U)\ge 2$, 
while in Lemma 2 in \cite{schless} it must be the smooth locus of $X$. On the other hand, here $X$ is only embedded into 
a vector space, while in \cite{schless} it can be embedded into an arbitrary smooth local geometric scheme $Y$.
However, despite all these differences, the proof of Lemma 2 in \cite{schless} can be used as a proof of 
Lemma \ref{schlessgen} here without any significant changes. 

\begin{proof}[Proof of Theorem \ref{schlessgen}]
If $\mathcal F$ is a coherent sheaf on $X$, denote $\mathcal F^\vee=\Hom_X(\mathcal F, \mathcal O_X)$.
First, we prove three lemmas, which extend Lemma 1 from \cite{schless}. Here we use 
the following notion of \textit{sections and cohomology with support} (for more details, see, for example, \cite[Section II.1]{hartshorne} 
and \cite[Section III.2]{hartshorne}). Given a sheaf $\mathcal F$ on a variety $X$ and 
a 
closed subset $V\subset X$, we denote by $H^0_{V}(X,\mathcal F)$ (or by $\Gamma_V(X,\mathcal F)$) the space of all global sections $s$ 
of $\mathcal F$ that vanish outside $V$ ($s|_{X\setminus V}=0$). We call the space $H^0_{V}(X,\mathcal F)$ 
the space of \textit{global sections of $\mathcal F$ with support on $V$}.
The functor $H^0_V(X,-)$ is left exact, and it has classical right derived functors, 
which are called \textit{cohomology with support} and denoted by $H^i_V(X,-)$.

\begin{lemma}\label{schlesslemmao}
Let $X$ be a normal affine algebraic variety, $U$ be an open subset such that $\codim_X(X\setminus U)\ge 2$, and 
$\mathcal F$ be a free sheaf of finite rank on $X$. Then 
$H^0_{(X\setminus U)}(X,\mathcal F)=H^1_{(X\setminus U)}(X,\mathcal F)=0$. 
\end{lemma}

\begin{proof}
Write the long exact sequence for cohomology with support:
$$
0\to H^0_{(X\setminus U)}(X,\mathcal F)\to
H^0(X,\mathcal F)\to
H^0(U,\mathcal F)\to
H^1_{(X\setminus U)}(X,\mathcal F)\to
H^1(X,\mathcal F)\to\ldots
$$
$\mathcal F$ is a free sheaf of finite rank,
$X$ is normal, and $\codim_X(X\setminus U)\ge 2$, 
therefore the restriction map
$H^0(X,\mathcal F)\to\Gamma(U,\mathcal F)$
is an isomorphism. Hence, $H^0_{(X\setminus U)}(X,\mathcal F)=0$ and the map 
$H^1_{(X\setminus U)}(X,\mathcal F)\to
H^1(X,\mathcal F)$ is an embedding.
Since $X$ is affine, $H^1(X,\mathcal F)=0$, so
$H^1_{(X\setminus U)}(X,\mathcal F)=0$.
\end{proof}

The following lemma is known, but for convenience of the reader we give a proof here.

\begin{lemma}\label{schlesslemmai}
Let $X$ be a normal affine algebraic variety, $U$ be an open subset such that $\codim_X(X\setminus U)\ge 2$, and 
$\mathcal F$ be a coherent sheaf on $X$ such that there exists a coherent sheaf $\mathcal G$ on $X$ such that 
$\mathcal F=\mathcal G^\vee$. Then $H^0_{(X\setminus U)}(X,\mathcal F)=0$. 
\end{lemma}

\begin{proof}
Since $\mathcal G$ is a coherent sheaf, there exists an exact sequence of coherent sheaves on $X$
$$
0\to \mathcal G'\to\mathcal G''\to\mathcal G\to 0,
$$
where $\mathcal G''$ is free. Since $\Hom_X(\cdot, \mathcal O_X)$ and $H^0_{(X\setminus U)}(X,\cdot)$
are left exact functors, the corresponding map
$$
H^0_{(X\setminus U)}(X,\mathcal F)\to
H^0_{(X\setminus U)}(X,\mathcal G''^\vee)
$$
is an embedding. $\mathcal G''$ is free and coherent, i.~e. it is a free sheaf of finite rank, so 
$\mathcal G''^\vee$ is also a free sheaf of finite rank.
By Lemma \ref{schlesslemmao}, $H^0_{(X\setminus U)}(X,\mathcal G''^\vee)=0$. Hence,
$H^0_{(X\setminus U)}(X,\mathcal F)=0$.
\end{proof}

\begin{remark}
Strictly speaking, we will not need this fact later, but the statement of the lemma is closely related to
the notion of a \emph{reflexive} sheaf. Namely, a sheaf $\mathcal F$ is called \emph{reflexive} if 
$\mathcal F^{\vee\vee}=\mathcal F$. Clearly, if $\mathcal F$ is reflexive, then it satisfies the conditions of the lemma, we
can take $\mathcal G=\mathcal F^\vee$. One can prove that the contrary is also true, i.~e. if $\mathcal F$ can be
written as $\mathcal G^\vee$, then $\mathcal F$ is reflexive.
\end{remark}

\begin{lemma}\label{schlesslemmaii}
Let $X$ be a normal affine algebraic variety, $U$ be an open subset such that $\codim_X(X\setminus U)\ge 2$, and 
$\mathcal F$ be a coherent sheaf on $X$ such that there exists a coherent sheaf $\mathcal G$ on $X$ such that 
$\mathcal F=\Hom_X(\mathcal G, \mathcal O_X)$. Then the restriction map $\Gamma(X,\mathcal F)\to\Gamma(U,\mathcal F)$ 
is an isomorphism.
\end{lemma}

\begin{proof}
Again write an exact sequence of coherent sheaves on $X$
$$
0\to \mathcal G'\to\mathcal G''\to\mathcal G\to 0,
$$
where $\mathcal G''$ is free. The dualization functor is left exact, so the corresponding map $\mathcal F\to \mathcal G''\vee$ is 
an embedding, and its cokernel (denote it by $\mathcal Q$) is a subsheaf of $\mathcal G'^\vee$. By Lemma
\ref{schlesslemmai}, $H^0_{(X\setminus U)}(X,\mathcal G'^\vee)=0$.
Since $\mathcal Q$ is a subsheaf of $\mathcal G'^\vee$ and $H^0_{(X\setminus U)}(X,\cdot)$ is a left exact functor, 
$H^0_{(X\setminus U)}(X,\mathcal Q)=0$. Again, since $\mathcal G''$ is free and coherent,
$\mathcal G''^\vee$ is a free sheaf of finite rank. By Lemma \ref{schlesslemmao}, $H^1_{(X\setminus U)}(X,\mathcal G''^\vee)=0$.
We have the following exact sequence of cohomology:
\begin{multline*}
0\to H^0_{(X\setminus U)}(X,\mathcal F)\to
H^0_{(X\setminus U)}(X,\mathcal G''^\vee)\to
H^0_{(X\setminus U)}(X,\mathcal Q)\to\\
H^1_{(X\setminus U)}(X,\mathcal F)\to
H^1_{(X\setminus U)}(X,\mathcal G''^\vee)\to\ldots,
\end{multline*}
and we see that $H^1_{(X\setminus U)}(X,\mathcal F)=0$. By Lemma \ref{schlesslemmai},
$H^0_{(X\setminus U)}(X,\mathcal F)=0$.
Now write the following long exact sequence:
$$
0\to H^0_{(X\setminus U)}(X,\mathcal F)\to
H^0(X,\mathcal F)\to
H^0(U,\mathcal F)\to
H^1_{(X\setminus U)}(X,\mathcal F)\to\ldots
$$
We see that the restriction map $H^0(X,\mathcal F)\to
H^0(U,\mathcal F)$ is an isomorphism.
\end{proof}

Now we are ready to prove Theorem \ref{schlessgen}.
Denote $A=\mathbb C[X]$.
The generators $\abstractdependentgenerator1,\ldots,\abstractdependentgenerator{\numberofabstractgenerators}$ 
define an embedding $X\hookrightarrow \mathbb C^\numberofabstractgenerators=\Spec \mathbb C[\abstractindependentgenerator 1,\ldots,\abstractindependentgenerator{\numberofabstractgenerators}]$
and a morphism of algebras $\mathbb C[\abstractindependentgenerator 1,\ldots,\abstractindependentgenerator{\numberofabstractgenerators}]\to A$ 
so that $\abstractindependentgenerator i\mapsto \abstractdependentgenerator{i}$. Denote the kernel of this 
algebra morphism by $I$.
As we have previously seen, $I/I^2$ is an $A$-module. Denote the corresponding sheaf on $X$ by $\mathcal I$.
Observe that the $A$-module $\Theta/I\Theta$ introduced in the definition of $T^1(X)$ is isomorphic to the free 
$A$-module of rank $\numberofabstractgenerators$ as an $A$-module. The kernel of the map $\phi\colon \Theta/I\Theta \to \Hom_A(I/I^2,A)$ 
consists of all $\numberofabstractgenerators$-tuples $(g_1, \ldots, g_{\numberofabstractgenerators})$ of functions on $X$ such that for all $h\in I$ one has 
$\sum g_i\partial h/\partial \abstractindependentgenerator i=0$ in $A$ (to evaluate this expression, we take arbitrary representatives in the 
cosets corresponding to $g_i$ and to $h$ in $\mathbb C[\abstractindependentgenerator 1,\ldots,\abstractindependentgenerator{\numberofabstractgenerators}]$ 
and in $I$, respectively, we have seen previously
that its value in $A$ does not depend on this choice). In other words, the $\numberofabstractgenerators$-tuple $(g_1, \ldots, g_{\numberofabstractgenerators})$ defines a tangent 
vector field to $X$. The embedding of the tangent bundle on $X$ into the rank $\numberofabstractgenerators$ trivial bundle on $X$ we 
have just obtained coincides with the map $\psi$ in the statement of Theorem \ref{schlessgen}. So, we have the following exact 
sequence of $A$-modules:
$$
0\to \Gamma(X,\Theta_X)\stackrel{\Gamma(\psi|_U)}\longrightarrow A^{\oplus \numberofabstractgenerators}\to \Hom_A(I/I^2,A)\to T^1(X)\to 0.
$$
Since $X$ is affine, we also have an exact sequence of sheaves (denote the sheaf generated by the $A$-module $T^1(X)$ by $\mathcal T^1$):
$$
0\to\Theta_X\stackrel{\psi}\longrightarrow \mathcal O_X^{\oplus \numberofabstractgenerators}\to \mathcal I^\vee\to \mathcal T^1\to 0.
$$
Denote the map between sheaves $\mathcal O_X^{\oplus \numberofabstractgenerators}$ and $\mathcal I^\vee$ by $\widetilde \phi$.
It is known that 
(see, for example, \cite[Exercise 3.5 and Theorem 4.9]{hartdef})
if $U'\subseteq X$ is smooth, then $\Gamma(U',\mathcal T^1)=0$. So, if $U'$ is, 
in addition, affine, we have the following 
exact sequence:
$$
0\to\Gamma(U',\Theta_X)\stackrel{\Gamma(\psi|_{U'})}\longrightarrow \Gamma(U',\mathcal O_X^{\oplus \numberofabstractgenerators})
\stackrel{\Gamma(\widetilde{\phi}|_{U'})}\longrightarrow \Gamma(U',\mathcal I^\vee)\to 0.
$$
In particular, this holds for affine sets $U'$ forming an affine cover of $U$. Therefore, we have the following
exact sequence of sheaves on $U$:
$$
0\to\Theta_X|_U\stackrel{\psi|_U}\longrightarrow \mathcal O_X^{\oplus \numberofabstractgenerators}|_U\stackrel{\widetilde{\phi}|_U}\longrightarrow \mathcal I^\vee|_U\to 0,
$$
and we can write the long exact sequence of cohomology:
\begin{multline*}
0\to H^0(U,\Theta_X)\stackrel{H^0(\psi|_U)}\longrightarrow 
H^0(U,\mathcal O_X^{\oplus \numberofabstractgenerators})\stackrel{H^0(\widetilde{\phi}|_U)}\longrightarrow 
H^0(U,\mathcal I^\vee)\to\\
H^1(U,\Theta_X)\stackrel{H^1(\psi|_U)}\longrightarrow 
H^1(U,\mathcal O_X^{\oplus \numberofabstractgenerators})\to\ldots
\end{multline*}
Denote the map between $H^0(U,\mathcal I^\vee)$ and $H^1(U,\Theta_X)$ by $\delta$. 
We have $\ker H^1(\psi|_U)=\im \delta=H^0(U,\mathcal I^\vee)/\ker \delta=
H^0(U,\mathcal I^\vee)/\im H^0(\widetilde{\phi}|_U)=\coker H^0(\widetilde{\phi}|_U)$.

Recall the exact sequence of $A$-modules we started with:
$$
0\to \Gamma(X,\Theta_X)\stackrel{\Gamma(\psi)}\longrightarrow A^{\oplus \numberofabstractgenerators}\to \Hom_A(I/I^2,A)\to T^1(X)\to 0.
$$
We can write $A^{\oplus \numberofabstractgenerators}$ as $\Gamma(X,\mathcal O_X^{\oplus \numberofabstractgenerators})$. 
and
$\Hom_A(I/I^2,A)$ as $\Gamma(X,\mathcal I^\vee)$. Now we can apply Lemma \ref{schlesslemmaii}. $\Theta_X$ is 
dual to $\Omega_X$, $\mathcal O_X^{\oplus \numberofabstractgenerators}$ is dual to itself, and $\mathcal I^\vee$ is dual to $\mathcal I$ by 
construction. So we can rewrite the exact sequence as follows:
$$
0\to H^0(U,\Theta_X)\stackrel{H^0(\psi|_U)}\longrightarrow 
H^0(U,\mathcal O_X^{\oplus \numberofabstractgenerators})\stackrel{H^0(\widetilde{\phi}|_U)}\longrightarrow 
H^0(U,\mathcal I^\vee)\to
T^1(X)\to 0,
$$
and we see that $\coker H^0(\widetilde{\phi}|_U)=T^1(X)$.
\end{proof}

\begin{remark}
If a torus $T$ acts on $X$ and preserves $U$, all generators of $\CC[X]$ we have are homogeneous, 
and we find an affine covering of $U$ by sets preserved by $T$, 
then $\ker (H^1(U,\Theta_X)\stackrel{H^1(\psi|_U)}\longrightarrow H^1(U,\OO_X^{\oplus \numberofabstractgenerators}))$
becomes a graded $\CC[X]$-module. 
The $\CC[X]$-module $T^1(X)$ also becomes graded (see Section \ref{deformsreal}).

In this case, the argument above proves that $T^1(X)$ is isomorphic to 
$\ker (H^1(U,\Theta_X)\stackrel{H^1(\psi|_U)}\longrightarrow H^1(U,\OO_X^{\oplus \numberofabstractgenerators}))$
as a \emph{graded} $\CC[X]$-module.
\end{remark}

\section{Cech complexes cohomology}

We need 
two
more facts related to Cech complexes. The first proposition explains 
how to 
compute
derived direct images using Cech resolutions.

Let $\mathcal F$ be a quasicoherent sheaf on a separated algebraic variety $U$, and let $\{U_i\}_{i=1}^\numberofabstractopensets$ be an 
affine covering of $U$. 
Consider the following 
\textit{sheaf Cech resolution of $\mathcal F$}: it consists of sheaves $\mathcal F^i$ on $U$, $i\ge 0$, 
and 
$$
\mathcal F^i=\bigoplus_{1\le a_1<a_2<\ldots<a_{i+1}\le \numberofabstractopensets}\mathcal F_{a_1,\ldots,a_{i+1}},
$$
where if $V\subseteq U$ is an open subset, then $\Gamma(V,\mathcal F_{a_1,\ldots,a_{i+1}})=
\Gamma (V\cap U_{a_1}\cap\ldots\cap U_{a_{i+1}},\mathcal F)$. The differentials in 
the resolution are defined in the usual Cech sense: given a section
$$
(x_{a_1,\ldots,a_i})_{1\le a_1<a_2<\ldots<a_i\le \numberofabstractopensets}\in \Gamma (V,\mathcal F^{i-1}),
$$
the differential maps it to 
$$
(y_{a_1,\ldots,a_{i+1}})_{1\le a_1<a_2<\ldots<a_{i+1}\le \numberofabstractopensets}\in \Gamma (V,\mathcal F^i),
$$
where 
$$
y_{a_1,\ldots,a_{i+1}}=
\sum_{j=1}^{i+1}\!\left.\vphantom{\sum}(-1)^j(x_{a_1, \ldots, \widehat{a_j}, \ldots, a_{i+1}})
\right|_{V\cap U_{a_1}\cap\ldots\cap U_{a_{i+1}}}.
$$
Notice that if we take the global sections of all $\mathcal F^i$, we obtain a Cech complex 
of $\mathcal F$ in the "usual", non-sheaf sense.

Suppose we have a map $f\colon U\to Y$, where 
$Y$
is also a separated algebraic variety.

\begin{proposition}\label{computederived}\cite[Proposition III.8.7]{hartshorne}
$$
R^if_*(\mathcal F)=\mathcal H^i(f_*(\mathcal F^\bullet)),
$$
where $\mathcal H^i$ is the $i$th cohomology of the complex formed by $f_*(\mathcal F^i)$ for $i\ge 0$, 
not the $i$th cohomology of a particular sheaf.\qed
\end{proposition}

The second fact gives an easier way to compute the first cohomology of complexes that "look like a Cech complex"
under certain circumstances in any abelian category. Suppose that $\mathcal C$ is an abelian category, 
let $A$ be an object, let $\numberofabstractopensets\in \mathbb\NN$, and let for every $1\le i\le \numberofabstractopensets$ indices $a_i$ 
satisfying $1\le a_1 <\ldots< a_i\le \numberofabstractopensets$ $A_{a_1, \ldots, a_i}$ be a subobject of $A$ (i.~e. an object together with 
a morphism $A_{a_1, \ldots, a_i}\to A$ whose kernel is zero). 
Suppose also that if $(a_j)_{j=1}^i$ is a subsequence of $(b_j)_{j=1}^{i+1}$, 
then $A_{a_1, \ldots, a_i}$ is a subobject of $A_{b_1, \ldots, b_{i+1}}$, and the embedding 
$A_{a_1, \ldots, a_i}\to A_{b_1, \ldots, b_{i+1}}$ commutes with the embeddings of these 
objects into $A$.

Now consider the following complex $B^\bullet$: 
$$
B^i=\bigoplus_{1\le a_1<a_2<\ldots<a_{i+1}\le \numberofabstractopensets} A_{a_1,\ldots,a_{i+1}}, \quad i\ge -1.
$$
Here we allow $i=-1$ and say that $A_{\text{the empty sequence}}=0$, so $B^{-1}=0$
The differential $d\colon B^{i-1}\to B^i$ is defined using a sign-alternating sum, as it is usually 
defined in Cech complexes. Here we have objects in an abelian category, not necessarily abelian groups or 
modules over a ring, so we use universal properties of direct sums to interpret formulas with addition 
and subtraction signs.

We also need the following complex $B'^\bullet$:
$$
B'^i=\bigoplus_{1\le a_1<a_2<\ldots<a_{i+1}\le \numberofabstractopensets} A/A_{a_1,\ldots,a_{i+1}}, \quad i\ge -1.
$$
Here we also allow $i=-1$, and $B'^{-1}=A$.
Again, the differentials are defined "as usual" 
using universal properties of direct sums.


\begin{proposition}\label{hnasquotient}
For $i\ge 0$, $H^i(B^\bullet)=H^{i-1}(B'^\bullet)$. This isomorphism is functorial in $B$ and $B'$ if the embeddings 
$A_{a_1, \ldots, a_i}\to A$ are functorial in $A_{a_1, \ldots, a_i}$ and $A$.
\end{proposition}

\begin{proof}
Consider the following complex $B''^\bullet$: 
$$
B''^i=\bigoplus_{1\le a_1<a_2<\ldots<a_{i+1}\le \numberofabstractopensets} A, \quad i\ge -1.
$$
The differential is again the standard Cech differential. Clearly, 
we have an exact sequence of complexes: 
$$
0\to B^\bullet\to B''^\bullet\to B'^\bullet\to 0.
$$

Let us check that $B''^\bullet$ is acyclic.
\begin{lemma}
$B''^\bullet$ is acyclic.
\end{lemma}
\begin{proof}
First, consider the topological complex for a simplex with $\numberofabstractopensets$ vertices, 
i.~e. the following complex $C^\bullet$: 
$$
C^i=\bigoplus_{1\le a_1<a_2<\ldots<a_{i+1}\le \numberofabstractopensets} \ZZ, \quad i\ge 0.
$$
Here we do not allow $i=-1$, and the differential again coincides with the standard Cech differential
(although initially it is defined by topological means).
It is a well-known topological fact that $H^0(C^\bullet)=\ZZ$ and 
$H^i(C^\bullet)=0$ for $i\ne 0$. One easily checks directly that $H^0(C^\bullet)$
consists of classes of the elements of $C^0=\ZZ^{\oplus\numberofabstractgenerators}$
that have all coordinates equal, i.~e. of the elements of the form $(a,a,\ldots,a)$, where $a\in\ZZ$.

Therefore, the following complex $C'^\bullet$ of abelian groups is acyclic: 
$$
C'^i=\bigoplus_{1\le a_1<a_2<\ldots<a_{i+1}\le \numberofabstractopensets} \ZZ, \quad i\ge -1.
$$
Now let us use Mitchell's embedding theorem. Consider the abelian subcategory in $\mathcal C_0$ in $\mathcal C$ 
generated by $A$.
This is a small category, therefore by Mitchell's 
embedding theorem it is equivalent to an abelian subcategory in the category of left modules over a 
(not necessarily commutative) ring $R$.
So, we can consider $B''^\bullet$ as a complex of $R$-modules. In particular, $B''^\bullet$
also becomes a complex of abelian groups, 
and it is acyclic as a complex of $R$-modules iff it is acyclic as a complex of abelian groups.
And for complexes of abelian groups, we clearly have $B''^\bullet=C'^\bullet\otimes_{\ZZ} A$.

Let us deduce that $B''^\bullet$ is also acyclic. 
We cannot be sure that $A$ is a flat object in the category of abelian groups, but we can argue differently.
Since $C'^\bullet$ is acyclic and consists of free abelian groups of finite rank, it can be considered as a projective 
resolution for the abelian group $0$. Then $\Tor_i(0,A)=H^{-i}(C'^\bullet\otimes_\ZZ A)=H^{-i}(B''^\bullet)$.
But $\Tor_i(0,A)=0$, so $B''^\bullet$ is acyclic.
\end{proof}

Now let us write the long exact sequence for the exact triple
$$
0\to B^\bullet\to B''^\bullet\to B'^\bullet\to 0:
$$
$$
\ldots\to H^i(B^\bullet)\to H^i(B''^\bullet)\to H^i(B'^\bullet)\to H^{i+1}(B^\bullet)\to H^{i+1}(B''^\bullet)\to H^{i+1}(B'^\bullet)\to\ldots
$$
We have $H^i(B''^\bullet)=0$ and $H^{i+1}(B''^\bullet)=0$ for all $i\in\ZZ$, so 
$H^i(B'^\bullet)=H^{i+1}(B^\bullet)$.
\end{proof}

\begin{corollary}\label{h1compute}
If $A_{j,k}=A$ for all $1\le j<k\le \numberofabstractopensets$, then $H^1(B^\bullet)=(\bigoplus_{j=1}^\numberofabstractopensets A/A_j)/A$, 
where $A$ is mapped to $\bigoplus_{j=1}^\numberofabstractopensets A/A_j$ diagonally.\qed
\end{corollary}

\begin{corollary}\label{h1computegen}
In general, if it is not necessarily true that 
$A_{j,k}=A$ for all $1\le j<k\le \numberofabstractopensets$, then 
$$
H^1(B^\bullet)=\left.\!\left(\ker\Big(\bigoplus_{j=1}^\numberofabstractopensets(A/A_j)\to \bigoplus_{1\le j<k\le \numberofabstractopensets}(A/A_{j,k})\Big )\right)\right/A,
$$ 
where $A$ is mapped to $\bigoplus_{j=1}^\numberofabstractopensets A/A_j$ diagonally.\qed
\end{corollary}

\section{Leray spectral sequence}\label{sectleray}

We are going to use the following theorem:
\begin{theorem}
(see, for example, \cite[Section 3.3, page 74]{huybrechts} and \cite[\S III.7, Theorem 7]{gelman})
Let $f\colon X\to Y$ be a morphism of algebraic varieties, and let $\mathcal F$ be a 
quasicoherent sheaf on $X$.
Then there exists a spectral sequence called \textit{Leray spectral sequence} with the second sheet
$$
E_2^{p,q}=H^p(Y, R^qf_*\mathcal F),
$$
where the corresponding differentials map $E_r^{p,q}$ to $E_r^{p+r,q-r+1}$, $r\ge 2$, 
that converges to $H^{p+q}(X,\mathcal F)$. Denote the corresponding filtration 
on $H^{p+q}(X,\mathcal F)$ by $F^\bullet$.

The sheaves $R^qf_*\mathcal F$ can be considered as sheaves of $f_*\OO_X$-modules, 
and $H^p(Y, R^qf_*\mathcal F)$ can be therefore considered as $\CC[X]$-modules.
In this sense, the isomorphism
$$
F^p H^{p+q}(X,\mathcal F)/F^{p+1} H^{p+q}(X,\mathcal F)\cong E_\infty^{p,q}
$$
is an isomorphism of $\CC[X]$-modules.
\end{theorem}
Here $R^qf_*$ denotes the $q$th 
derived functors of the direct image functor in quasicoherent sheaf category
(or shortly, "$q$th derived direct image"). 

Notice that if $\dim Y=1$, then (since all sheaves $R^qf_*\mathcal F$ are coherent) 
$H^p(Y, R^qf_*\mathcal F)=0$ for $p\ge 2$ (and $p<0$), so all differentials 
vanish, $E_2^{p,q}=E_\infty^{p,q}$, and we have a short exact sequence
$$
0\to H^1 (Y, R^{q-1}f_*\mathcal F)\to H^q(X,\mathcal F)\to H^0 (Y, R^qf_*\mathcal F)\to 0.
$$

We will also need an explicit description of the maps in this exact sequence for $q=1$. 
By adopting the general construction from \cite[\S III.7]{gelman} one can check that 
the exact sequence above for $q=1$ looks as follows.

Choose an affine open covering $\{V_i\}$ of $Y$. 
Also choose an affine open covering $\{U_{i,j}\}$ of $X$ (here $(i,j)\in\mathfrak I$, 
where $\mathfrak I$ is a finite set of pairs of natural numbers)
so that if $(i,j)\in\mathfrak I$, then $1\le i\le \numberofabstractopensets$, and 
$U_{i,j}\subseteq f^{-1}(V_i)$.

Then the map $H^1 (Y, f_*\mathcal F)\to H^1(X,\mathcal F)$ works as follows:
An element of $H^1 (Y, f_*\mathcal F)$ is represented by a tuple of sections 
$(a_{i,i'})_{1\le i<i'\le \numberofabstractopensets}$, where $a_{i,i'}\in \Gamma(V_i\cap V_{i'},f_*\mathcal F)$
satisfy the cocycle conditions. By definition, 
$\Gamma(V_i\cap V_{i'},f_*\mathcal F)=\Gamma(f^{-1}(V_i\cap V_{i'}), \mathcal F)$, 
and, since $U_{i,i'}\subseteq f^{-1}(V_i)$ 
we can define restrictions $a_{i,i'}|_{U_{i,j}\cap U_{i',j'}}\in \Gamma(U_{i,j}\cap U_{i',j'}, \mathcal F)$
for all $j$ and $j'$ such that $(i,j),(i',j')\in\mathfrak I$.
These sections 
together with zeros for 
the sets $U_{i,j}\cap U_{i',j'}$, where $i=i'$,
form a class in $H^1(X,\mathcal F)$.

The map $H^1(X,\mathcal F)\to H^0 (Y, R^1f_*\mathcal F)$ works as follows:
Suppose that we have an element of $H^1(X,\mathcal F)$ defined by 
sections $a_{i,j,i',j'}\in \Gamma(U_{i,j}\cap U_{i',j'},f_*\mathcal F)$ 
($(i,j),(i',j')\in \mathfrak I$ and $(i,j)<(i',j')$ for some prefixed order on $\mathfrak I$)
satisfying the cocycle conditions. These sections together can be interpreted as 
a global section of the sheaf $\mathcal F^1$ from Proposition \ref{computederived} 
and, therefore, as a global section $s\in \Gamma(Y,f_*\mathcal F^1)$.
It follows from the cocycle conditions on $a_{i,j,i',j'}$ that 
$s\in\Gamma(Y,\ker(\mathcal F^1\to \mathcal F^2))$, so $s$ defines a class in 
$\Gamma(Y,H^1(\mathcal F^\bullet))=H^0(Y,R^1f_*\mathcal F)$.

Finally, consider an even more particular situation. Denote 
$U=\cap_{(i,j)\in\mathfrak I} U_{i,j}$ and $V=\cap_{i=1}^{\numberofabstractopensets}V_i$.
We keep all of the assumptions from the three previous paragraphs, 
but also suppose the following:
\begin{enumerate}
\item $X$ and $Y$ are irreducible.
\item $U$ and $V$ are nonempty.
\item $U\subseteq f^{-1}(V)$.
\item All restriction maps for the sheaf $\mathcal F$ 
to nonempty open subsets are injective (for example, this is true for vector bundles
of finite rank). I.~e., if $W\subseteq W'\subseteq X$ are nonempty open subsets, then 
the restriction map $\Gamma(\mathcal F,W')\to \Gamma(\mathcal F, W)$ is injective.
Then all restriction maps for the sheaf $\mathcal F$ to 
open sets containing $V$ are also injective. 
\item $V=V_i\cap V_{i'}$ if $1\le i<i'\le \numberofabstractopensets$.
This assumption enables us to use Corollary \ref{h1compute} to compute cohomology groups of $\mathcal F$.
\end{enumerate}

We can apply Corollary \ref{h1compute}
to $H^1 (Y, f_*\mathcal F)$ and apply Corollary \ref{h1computegen} 
to $H^1(X,\mathcal F)$.
Moreover, 
we can consider the sheaf $\mathcal F_U=(j_U)_*\mathcal F|_U$.
In other words, for all open subsets $U'\subset X$ set $\Gamma (U',\mathcal F_U)=\Gamma(F,U\cap U')$.
Then all sheaves $\mathcal F_{(i,j)}$ and $\mathcal F_{(i,j),(i',j')}$
from Proposition \ref{computederived} are 
subobjects of $\mathcal F_U$, and,
since the functor $f_*$ is left exact, each sheaf $f_*\mathcal F_{(i,j)}$ and $f_*\mathcal F_{(i,j),(i',j')}$
is a subobject of $f_*\mathcal F_U$.
So, we can also apply Corollary \ref{h1computegen} to 
the first cohomology of the complex $f_*\mathcal F^\bullet$.
We get the following isomorphisms: 
$$
H^1 (Y, f_*\mathcal F)=\left.\!\left(\bigoplus_{i=1}^{\numberofabstractopensets}
\Big(\Gamma(V,f_*\mathcal F)/\Gamma(V_i,f_*\mathcal F)\Big)\right)\right/\Gamma(V,f_*\mathcal F),
$$
\begin{multline*}
H^1 (X, \mathcal F)=\\
\ker \left(\bigoplus_{(i,j)\in \mathfrak I}\Big(\Gamma(U,\mathcal F)/\Gamma(U_{i,j},\mathcal F)\Big)\right.\!
\to \bigoplus_{\substack{(i,j),(i',j')\in \mathfrak I\\(i,j)<(i',j')}}\Big(\Gamma(U,\mathcal F)/\Gamma(U_{i,j}\cap U_{i',j'},\mathcal F)\Big)
\left.\!\left.\!\vphantom{\bigoplus_{(i,j)\in \mathfrak I}}
\right)\right/\Gamma(U,\mathcal F),
\end{multline*}
and
$$
H^1 (f_*\mathcal F^\bullet)=\left.\!\left(
\ker \bigoplus_{(i,j)\in\mathfrak I}\Big(f_*\mathcal F_U/f_*\mathcal F_{(i,j)}\Big)
\to \text{\raisebox{0pt}[0pt][0pt]{$\bigoplus_{\substack{(i,j),(i',j')\in\mathfrak I\\(i,j)<(i',j')}} $}}
\Big(f_*\mathcal F_U/f_*\mathcal F_{(i,j)}\Big)
\right)\right/f_*\mathcal F_U.
\vphantom{\bigoplus_{\substack{(i,j),(i',j')\in\mathfrak I\\(i,j)<(i',j')}}}
$$
These identifications enable us to write the 
maps $H^1 (Y, f_*\mathcal F)\to H^1(X,\mathcal F)$ and 
$H^1(X,\mathcal F)\to H^0 (Y, R^1f_*\mathcal F)$.

The map $H^1 (Y, f_*\mathcal F)\to H^1(X,\mathcal F)$
is induced by the following map 
$\bigoplus_{i=1}^{\numberofabstractopensets}\Gamma(V,f_*\mathcal F)\to 
\bigoplus_{(i,j)\in\mathfrak I}\Gamma(U,\mathcal F)$.
For each $i$ ($1\le i\le \numberofabstractopensets$), 
a section from the $i$th direct summand $\Gamma(V,f_*\mathcal F)=\Gamma(f^{-1}(V),\mathcal F)$
is restricted to $U$ and then
mapped diagonally to 
$\bigoplus_{j:(i,j)\in\mathfrak I}\Gamma(U,\mathcal F)$.
Note that 
$U_{i,j}\cap U_{i',j'}\subseteq f^{-1}(V)$ if $i\ne i'$, 
therefore, the image of this map indeed belongs to 
the correct subobject of 
$\bigoplus_{(i,j)\in \mathfrak I}\Gamma(U,f_*\mathcal F)$.

To get the map $H^1(X,\mathcal F)\to H^0 (Y, R^1f_*\mathcal F)$, note 
that each global section of 
$\bigoplus_{(i,j)\in\mathfrak I}f_*\mathcal F_U$
induces a global section of 
$(\bigoplus_{(i,j)\in\mathfrak I}
(f_*\mathcal F_U/f_*\mathcal F_{(i,j)}))/\mathcal F_U$.
On the other hand, by the definition of $\mathcal F_U$, 
$\Gamma(Y,\bigoplus_{(i,j)\in\mathfrak I}f_*\mathcal F_U)=
\bigoplus_{(i,j)\in\mathfrak I}\Gamma(U,\mathcal F)$,
and the map $H^1(X,\mathcal F)\to H^0 (Y, R^1f_*\mathcal F)$
is induced by this equality.

\begin{remark}
Suppose that a torus $T$ acts on $X$, the morphism $f$ is $T$-invariant, 
and each set $U_{i,j}$ is preserved by the action of $T$.
Then one can introduce grading on 
$H^1 (Y, f_*\mathcal F)$, on $H^1(X,\mathcal F)$, and on $H^0 (Y, R^qf_*\mathcal F)$
in the obvious way.

It follows from the above descriptions of the maps between these cohomology groups that 
this grading is preserved.
\end{remark}

\section{Stably dominant morphisms}
\begin{definition}
(This is definition 10.80.1, tag 058I in Stacks project \cite{stacks})
Let $A$ be an algebra.
A map of $A$-modules $f\colon K_1\to K_2$ is called \textit{universally injective}
if it is injective and for every $A$-module $K_3$, the map $f\otimes \id_{K_3}\colon K_1\otimes K_3\to K_2\otimes K_3$
is injective.
\end{definition}

\begin{remark}\label{directsummandinj}
A direct summand embedding is always universally injective.
\end{remark}

Consider the following situation. 
Let $X$ and $Y$ be two affine schemes with base $Z$, in other words, let $\stdflatmorphism_1\colon X\to Z$ and 
$\stdflatmorphism_2\colon Y\to Z$ be two morphisms of affine schemes.
Let $f\colon X\to Y$ be a relative morphism, i. e. a morphism such that $\stdflatmorphism_2\circ f=\stdflatmorphism_1$.
In other words, we have the following commutative diagram:
$$
\xymatrix{
X\ar[r]^{f}\ar[rd]_<<<<<{\stdflatmorphism_1} & Y\ar[d]^{\stdflatmorphism_2} \\
& Z
}
$$
Algebraically this means that $\CC[X]$ and $\CC[Y]$ are $\CC[Z]$-modules, and that 
$f^*\colon \CC[Y]\to \CC[X]$ is a morphism of $\CC[Z]$-algebras.
\begin{definition}
We call $f$ a \textit{stably dominant morphism} if $f^*$ is a universally injective morphism of $\CC[Z]$-modules.
\end{definition}

\begin{lemma}\label{basechangestabdom}
The functor of base change preserves relatively dominant morphisms. In other words, suppose that we have 
a morphism of schemes $Z_1\to Z$, 
and $f\colon X\to Y$ is a stably dominant morphism of $Z$-schemes.
Then $f\times_Z Z_1\colon X\times_Z Z_1 \to Y\times_Z Z_1$
is a relatively dominant morphism of $Z_1$-schemes.
\end{lemma}

\begin{proof}
In algebraic terms, we know that $f^*\colon \CC[X]\to \CC[Y]$ is universally injective.
Then $f^*\otimes_{\CC[Z]} \id_{\CC[Z_1]}\colon \CC[X]\otimes_{\CC[Z]}\CC[Z_1]\to \CC[Y]\otimes_{\CC[Z]}\CC[Z_1]$
is a universally injective morphism of $\CC[Z_1]$-modules.
Finally, $\CC[X\times_Z Z_1]=\CC[X]\otimes_{\CC[Z]}\CC[Z_1]$, 
$\CC[Y\times_Z Z_1]=\CC[Y]\otimes_{\CC[Z]}\CC[Z_1]$, 
and $(f\times_Z Z_1)^*=f^*\otimes_{\CC[Z]} \id_{\CC[Z_1]}$.
\end{proof}

Let $g\colon Z_1\to Z$ be a morphism of affine schemes.
If $X$ is an affine $Z$-scheme, $g$ induces a morphism 
$X\times_Z Z_1\to X$, which we will denote by $g_X$. 
Algebraically, if $x$ is a regular function on $X$, 
then $g_X^*(x)=x\otimes 1_{Z_1}$, where $1_{Z_1}$ is the unit of 
the algebra $\CC[Z_1]$.
This is illustrated by the following commutative diagram
$$
\xymatrix{
X\times_Z Z_1\ar[r]^-{g_X}\ar[d] & X\ar[d] \\
Z_1\ar[r]_g & Z
}
$$

For example, if $I\subset \CC[Z]$ is an ideal, and $Z_1$ is the vanishing locus of $I$, 
then $\CC[Z_1]=\CC[Z]/I$ as a $\CC[Z]$-module, 
$g$ is the embedding of $Z_1$ into $Z$, $g^*$ is the canonical projection
$\CC[Z]\to \CC[Z]/I$, and $g_X^*$ is the canonical projection 
$\CC[X]\to \CC[X]/(I\CC[X])$.

\begin{lemma}\label{computefiber}
Let $Z$ be an affine scheme,
let $f\colon X\to Y$ be a stably dominant morphism of affine $Z$-schemes, and let 
$g\colon Z_1\to Z$ be a closed embedding. Then 
$(f\times_Z Z_1)^*(\CC[Y\times_Z Z_1])=g_X^*(f^*(\CC[Y]))$ as a subalgebra of 
$\CC[X\times_Z Z_1]$.
\end{lemma}

Before we give a proof, let us provide a commutative diagram with all involved morphisms.
$$
\xymatrix{
X\times_Z Z_1\ar[rr]^{f\times_Z Z_1}\ar[rd]\ar[dd]^{g_X} && Y\times_Z Z_1\ar[ld]\ar[dd]^{g_Y} \\
& Z_1\ar[dd]^<<<<<g\\
X\ar[rr]^<<<<<<<<<f\ar[rd] && Y\ar[ld]\\
& Z
}
$$

\begin{proof}
Since $g$ is a closed embedding, $g^*$ is surjective, 
and then $(g_Y)^*$ is also surjective.
So, $(f\times_Z Z_1)^*(\CC[Y\times_Z Z_1])=(f\times_Z Z_1)^*((g_Y)^*(\CC[Y]))$. 
And the commutativity 
$(f\times_Z Z_1)^*\circ (g_Y)^*=(g_X)^*\circ f^*$ follows directly from the definitions 
of $g_X$, $g_Y$, and $f\times_Z Z_1$.
\end{proof}

In particular, we can use this lemma to compute fibers of the morphism $Y\to Z$ 
if we already know fibers of the morphism $X\to Z$ (such a fiber is a particular 
case of base change applied to $Y$, namely, we change the base of $Y$ from $Z$ to a point in $Z$).

%

\section{Notation and terminology}

First, we need some notation for lattice polyhedra.
Let $\stdpolyhedronletter$ be a polyhedron with tail cone $\6$ and with all vertices in $N$, 
where $\dim \6=\dim N=2$, and $\6$ is pointed.
We denote the number of 
vertices of $\stdpolyhedronletter$ by 
$\numberofvertices{\stdpolyhedronletter}$
and we denote the vertices of $\stdpolyhedronletter$ by 
$\indexedvertex{\stdpolyhedronletter}{1},\ldots,\indexedvertex{\stdpolyhedronletter}{\numberofvertices{\stdpolyhedronletter}}$
so that pairs of consecutive vertices in this enumeration form the finite edges of $\stdpolyhedronletter$.
We denote the finite edge between $\indexedvertex{\stdpolyhedronletter}{j}$ and $\indexedvertex{\stdpolyhedronletter}{j+1}$
by $\indexededge{\stdpolyhedronletter}{j}$. We denote the infinite edge with the endpoint $\indexedvertex{\stdpolyhedronletter}{1}$
by $\indexededge{\stdpolyhedronletter}{0}$ and the infinite edge with the endpoint $\indexedvertex{\stdpolyhedronletter}{\numberofvertices{\stdpolyhedronletter}}$
by $\indexededge{\stdpolyhedronletter}{\numberofvertices{\stdpolyhedronletter}}$.
For each vertex $\indexedvertex{\stdpolyhedronletter}{i}$ denote by $\normalvertexcone{\indexedvertex{\stdpolyhedronletter}{i}}{\stdpolyhedronletter}$
the subcone of $\sck$ consisting of all $\chi\in\sck$ such that $\chi(\indexedvertex{\stdpolyhedronletter}{i})=
\min_{a\in\stdpolyhedronletter}\chi (a)$. We call $\normalvertexcone{\indexedvertex{\stdpolyhedronletter}{i}}{\stdpolyhedronletter}$
the \textit{normal subcone of the vertex $\indexedvertex{\stdpolyhedronletter}{i}$}. 
One checks easily that this is really a subcone, 
that $\sck=\bigcup \normalvertexcone{\indexedvertex{\stdpolyhedronletter}{i}}{\stdpolyhedronletter}$, 
that 
the intersection of two such cones is either a ray or 
the origin, and it is a ray if and only if the two corresponding vertices form an edge $\indexededge{\stdpolyhedronletter}{j}$.
In the latter case this ray is exactly the set of all $\chi\in\sck$ whose minimum on $\stdpolyhedronletter$
is attained on $\indexededge{\stdpolyhedronletter}{j}$. We denote this ray by 
$\normalvertexcone{\indexededge{\stdpolyhedronletter}{i}}{\stdpolyhedronletter}$ and call it the \textit{normal ray 
of the edge $\indexededge{\stdpolyhedronletter}{i}$}. Finally, we extend this notation for 
infinite edges of $\stdpolyhedronletter$: we denote by $\normalvertexcone{\indexededge{\stdpolyhedronletter}{0}}{\stdpolyhedronletter}$
(resp. $\normalvertexcone{\indexededge{\stdpolyhedronletter}{\numberofvertices{\stdpolyhedronletter}}}{\stdpolyhedronletter}$)
the ray in $M$ consisting of all $\chi\in\sck$ whose minimum on $\stdpolyhedronletter$ is attained on 
$\indexededge{\stdpolyhedronletter}{0}$ (resp. $\indexededge{\stdpolyhedronletter}{\numberofvertices{\stdpolyhedronletter}}$).
These two rays are in fact the two rays forming $\partial(\sck)$, and they are also called the normal rays of the 
corresponding edges. The normal subcones of all vertices and the normal rays of all (finite and infinite)
edges form a fan, which is called the \textit{normal fan} of $\stdpolyhedronletter$. 

We always choose the order on vertices of $\stdpolyhedronletter$ so that 
$\indexededge{\stdpolyhedronletter}{0}$ is always the same one of the two rays forming $\partial(\6)$
(it must not depend on $\stdpolyhedronletter$). This ray is denoted by $\indexededge{\6}{0}$, and the 
other ray of $\partial(\6)$ is denoted by $\indexededge{\6}{1}$. 

The previous notation applies to polyhedra with tail cone $\6$, let us extend it to 
the vertex and edges of $\sck$. Namely, the boundary of $\sck$ consists of two infinite edges and one vertex
at the origin. Denote the vertex at the origin by $\indexedvertex{\sck}{1}$.
If $\stdpolyhedronletter$ is a polyhedron with tail cone $\6$, 
then $\normalvertexcone{\indexededge{\stdpolyhedronletter}{0}}{\stdpolyhedronletter}$
is always the same edge of $\sck$ (independently of $\stdpolyhedronletter$), 
and $\normalvertexcone{\indexededge{\stdpolyhedronletter}{\numberofvertices{\stdpolyhedronletter}}}{\stdpolyhedronletter}$
is always the other edge of $\sck$.
Denote $\indexededge{\sck}{0}=\normalvertexcone{\indexededge{\stdpolyhedronletter}{0}}{\stdpolyhedronletter}$
and $\indexededge{\sck}{1}=\normalvertexcone{\indexededge{\stdpolyhedronletter}{\numberofvertices{\stdpolyhedronletter}}}{\stdpolyhedronletter}$.
In particular, this is true for $\stdpolyhedronletter=\6$, i.~e. 
$\indexededge{\sck}{0}=\normalvertexcone{\indexededge{\6}{0}}{\6}$
and $\indexededge{\sck}{1}=\normalvertexcone{\indexededge{\6}{1}}{\6}$.
Denote the primitive lattice 
vectors on $\indexededge{\sck}{j}$ by $\sckboundarybasis{j}$.

Fig. \ref{polyhedronfig} shows an example of this notation for a polyhedron $\stdpolyhedronletter$, 
its tail cone $\6$, and its normal fan.

\begin{figure}[!h]
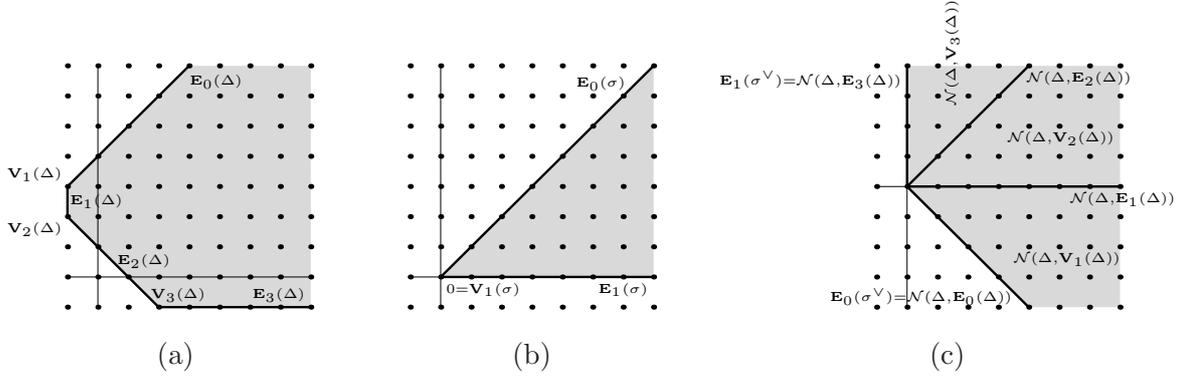

\begin{center}
\begin{tabular}{ccc}
\includegraphics{t1_3fold_figures-3.mps}\hspace*{1em}&\hspace*{1em}\includegraphics{t1_3fold_figures-4.mps}\hspace*{1em}&
\includegraphics{t1_3fold_figures-5.mps}\\[\medskipamount]
(a)&(b)&(c)
\end{tabular}
\end{center}
\caption{An example of notation for: (a) a polyhedron, (b) its tail cone, and (c) its normal fan.
The figure (c) also shows notation for the dual cone $\sck$.}
\label{polyhedronfig}
\end{figure}

If $\rho$ is a ray in $M_\QQ$, we denote the primitive lattice vector on $\rho$ by $\primitivelattice{\rho}$.
If $a$ is a vector or a segment in $N$, denote by $\latticelength{a}$ the lattice length of $a$, i.~e. the number of 
lattice points in $a$ including exactly one of the endpoints.

\subsection{List of notation introduced further}
The notation listed below will be properly introduced later, we list it now to ease reading and navigation only, without going into details of 
the underlying notions. 
\begin{enumerate}
\item We fix a two-dimensional pointed cone $\6\subset N_{\QQ}$ (recall that $\dim N_{\QQ}=2$).
\item We fix points $p_1,\ldots, p_{\numberofdivisorpoints}\in \PP^1$ and polyhedra 
$\stdpolyhedronletter_{p_1},\ldots, \stdpolyhedronletter_{p_\numberofdivisorpoints}\subset N_{\QQ}$
with tail cone $\6$.
\item We denote by $\mathcal D=\sum_{i=1}^{\numberofdivisorpoints} p_i\otimes \stdpolyhedronletter_{p_i}$ a divisor on $\PP^1$.
It will be used to construct a T-variety.
\item The T-variety will be denoted by $X$.
\item We shortly write $\numberofverticespt{p}$ instead of $\numberofvertices{\stdpolyhedronletter_{p}}$ and 
$\indexedvertexpt{p}{i}$ instead of $\indexedvertex{\stdpolyhedronletter_p}{i}$ for $p\in \PP^1$.
\item We also need a notion of an essential special point, which is a point $p_i$ such that $\stdpolyhedronletter_{p_i}$ is 
not a translation of $\6$. We denote the number of essential special points by $\numberofessentialdivisorpoints$, 
and we will assume that the points $p_1,\ldots, p_{\numberofessentialdivisorpoints}$
are essential.
\item We will introduce a set of degrees containing 
the union of Hilbert bases of several subcones of $\sck$, and we will denote 
the degrees in this set by $\lambda_1,\ldots,\lambda_{\numberoflatticegenerators}$.
\item Since 
$X$ is a $T$-variety, $\CC [X]$ is an $M$-graded algebra.
As usual, we will denote the degree of a homogeneous element $x\in \CC[X]$
with respect to this grading by $\deg(x)$.
\item We will choose homogeneous generators of this algebra, and denote them by 
\begin{multline*}
\dependentgeneratorsdegree{\lambda_1, 1}, \ldots, \dependentgeneratorsdegree{\lambda_1, \dim\Gamma(\PP^1,\OO(\mathcal D(\lambda_1)))}, \\
\dependentgeneratorsdegree{\lambda_2, 1}, \ldots, \dependentgeneratorsdegree{\lambda_2, \dim\Gamma(\PP^1,\OO(\mathcal D(\lambda_2)))}, \\ \ldots, \\
\dependentgeneratorsdegree{\lambda_{\numberoflatticegenerators}, 1}, \ldots, 
\dependentgeneratorsdegree{\lambda_{\numberoflatticegenerators}, \dim\Gamma(\PP^1,\OO(\mathcal D(\lambda_{\numberoflatticegenerators})))}.
\end{multline*}
Here $\deg(\dependentgeneratorsdegree{\lambda_i, j})=\lambda_i$. Also note that 
the $\lambda_i$th graded component of $\CC[X]$ is by construction identified with 
$\Gamma(\PP^1,\OO_{\PP^1}(\mathcal D(\lambda_i)))$. The generators 
$\dependentgeneratorsdegree{\lambda_i, 1}, \ldots, \dependentgeneratorsdegree{\lambda_i, \dim\Gamma(\PP^1,\OO(\mathcal D(\lambda_i)))}$
will span $\Gamma(\PP^1,\OO_{\PP^1}(\mathcal D(\lambda_i)))$.
\item We denote the 
total number of these generators by $\numberoffixedgenerators$.
\item We will fix a smooth open subset $U\subseteq X$ such that $\codim_X (X\setminus U)\ge 2$.
\item We will fix an affine open covering of $U$, which we will denote by 
$U_1,\ldots,U_{\numberoffixedopensets}$.
\end{enumerate}

\chapter{Formula for the graded component of $T^1$ of degree 0 in terms of sheaf cohomology}\label{sectcohomformula}

\section{Regularity locus and fiber structure of the map $\pi$}

Let $\6\subset N_\QQ$ be a pointed full-dimensional cone, $p_1,\ldots, p_{\numberofdivisorpoints}$ be points on $\PP^1$, 
$\stdpolyhedronletter_{p_i}\subseteq N_\QQ$ be polyhedra 
whose vertices are lattice points and whose tail cones are all $\6$. Unlike what is assumed sometimes, we do not allow $\varnothing$ 
to appear among these polyhedra.\footnote{In terms of the notation where $\varnothing$ is allowed among the coefficients, 
this means that the locus of the polyhedral divisor will be the whole $\PP^1$.} These data define a polyhedral divisor 
$\mathcal D=\sum_{i=1}^\numberofdivisorpoints \stdpolyhedronletter_{p_i}\otimes p_i$ and a graded algebra 
$A=\bigoplus_{\chi\in\sck\cap M}\Gamma(\PP^1,\mathcal O(\mathcal D(\chi)))$.
If $p\in \PP^1$ does not coincide with any of the points $p_i$, we denote $\stdpolyhedronletter_p=\6$.
Suppose in the sequel that $\mathcal D$ is proper, then $A$ defines a 3-dimensional variety $X=\Spec A$ with an 
action of a 2-dimensional torus. We use the notation $\pi$ for the rational map from $X$ to $\PP^1$ introduced 
in Proposition \ref{quotmorph}. It is known that all such varieties are normal.

In the sequel we will always keep in mind that very ample divisors on $\PP^1$ are exactly the divisors of positive degree 
and principal divisors are exactly the divisors of degree zero.
We call a point $p\in \PP^1$ \textit{ordinary} if it is not one of the points $p_i$, otherwise 
we call it \textit{special}. We require that the sum $\sum \stdpolyhedronletter_{p_i}\otimes p_i$ is finite, but we 
do not require that all summands are nontrivial, i.~e. we allow summands of the form 
$\6\otimes p_i$, which are zeros in the polyhedral divisor group. We call such points $p_i$ special anyway, according 
to the definition above. So in fact the notions 
of a special point and an ordinary point depend on the choice of exact presentation
$\mathcal D=\sum \stdpolyhedronletter_{p_i}\otimes p_i$, and we suppose that it is also fixed. If $\stdpolyhedronletter_{p_i}=\6+a$ for some $a\in N$, 
(including $a=0$), we call such $p_i$ a \textit{removable} special point, otherwise we call $p_i$ an
\textit{essential} special point.
If $\stdpolyhedronletter_{p_i}=\6$, we call $p_i$ a \textit{trivial} special point.

Fix a coordinate $t$ on $\PP^1$, i.~e. fix a rational function $t$ on $\PP^1$ that has one pole of order 1 and 
one zero of order 1. 
\begin{lemma}\label{sl2like}
Given two nonzero rational functions $f$ and $g$ on $\PP^1$ such that $f/g$ has one zero and one pole,
and both of them are of order one, 
there exist $a_1, b_1, a_2, b_2\in\CC$ such that $(a_1 f+b_1 g)/(a_2 f+b_2 g)=t$.
\end{lemma}
\begin{proof}
First, let us find $a'_1,b'_1,a'_2,b'_2\in\CC$ such that $(a'_1 f+b'_1 g)/(a'_2 f+b'_2 g)$ 
is regular at all points where $t$ is finite and has pole of order one at $t=\infty$. If $f/g=0$ at $t=\infty$, 
then this zero is of order one, and $a'_1=0,b'_1=1,a'_2=1,b'_2=0$ yield the function $g/f$, which has pole of degree one 
at $\infty$. It has no other poles since they would be other zeros of $f/g$, so this function has the desired properties.
Otherwise denote the value of $g/f$ at $t=\infty$ by $w_1$. Consider the following function: $g/f-w_1=(g-w_1f)/f$.
Clearly, it has a zero at $t=\infty$. 
Observe that $g/f$ has exactly one pole of order one, namely, at the point where $f/g$ has zero of order one. 
Hence, $g/f+w_1$ also has exactly one pole of order one. The sum of minus orders of all poles and of (plus) orders of all zeros 
of a rational function on $\PP^1$ is zero. Thus, $g/f+w_1$ has exactly one zero, and this zero is of order one. But we already know one zero of
$g/f+w_1$, namely, $t=\infty$. Therefore, this zero is of order one, and $f/(g-w_1f)$ has exactly one pole, this pole is of order one
and is at $t=\infty$.

Now we have a function $(a'_1 f+b'_1 g)/(a'_2 f+b'_2 g)$, which is regular at all points where $t$ takes finite value 
and has a pole of order one at $t=\infty$. Denote the value of this function at $t=0$ by $w_2$. Consider the following function:
$(a'_1 f+b'_1 g)/(a'_2 f+b'_2 g)-w_2=((a'_1 -w_2a'_2) f+(b'_1-w_2b'_2) g)/(a'_2 f+b'_2 g)$.
It has exactly one pole, this pole is at $t=\infty$ and of order one, and it has a zero at $t=0$. If we divide this function by $t$,
the resulting function $((a'_1 -w_2a'_2) f+(b'_1-w_2b'_2) g)/(t(a'_2 f+b'_2 g))$ has no poles on $\PP^1$, so
it is a constant. Therefore, if we multiply $((a'_1 -w_2a'_2) f+(b'_1-w_2b'_2) g)/(t(a'_2 f+b'_2 g))$ 
by the appropriate constant, it will be equal to $t$.
\end{proof}
\begin{corollary}\label{ratfuncexists}
For every divisor $D$ on $\PP^1$ of positive degree and for every non-zero rational function 
$f\in\Gamma(\PP^1,\mathcal O(D))$ there exist $g\in\Gamma(\PP^1,\mathcal O(D))$ and $a_1, b_1, a_2, b_2\in\CC$ such 
that $(a_1 f+b_1 g)/(a_2 f+b_2 g)=t$.
\end{corollary}
\begin{proof}
Since $f\in\Gamma(\PP^1,\OO(D))$, $\div(f)+D$ is an effective divisor. Write $\div(f)+D=\sum a'_i p'_i$, where 
$a'_i\in\ZZ_{\ge 0}$, $p'_i\in \PP^1$. Since $f$ is a rational function on $\PP^1$, $\deg\div(f)=0$, 
and $\sum a'_i=\deg\div(f)+\deg D=\deg D>0$. There exists a point $p'_i$ such that $a'_i>0$. Choose another point 
$p'_j$, and consider the following divisor: $D_1=\sum a'_i p'_i-p'_i+p'_j$. This is an effective divisor since $a'_i>0$.
Let $y$ be a rational function on $\PP^1$ such that $\div(y)=-p'_i+p'_j$. Then $D+div(fy)=D_1\ge 0$.
Hence, $g=fy\in\Gamma(\PP^1,\OO(D))$, and we can apply Lemma \ref{sl2like} to $f$ and $g$ since $\div(f/g)=\div(1/y)=p'_i-p'_j$.
\end{proof}
\begin{corollary}\label{quotmorphdefined}
Let $x\in X$. If there exists a degree $\chi\in\sck\cap M$ such that $\dim \Gamma(\PP^1,\mathcal O(\mathcal D(\chi)))\ge 2$ and 
$f\in\Gamma(\PP^1,\mathcal O(\mathcal D(\chi)))$ such that $\widetilde f(x)\ne 0$, then $\pi$ is defined at $x$.
\end{corollary}
\begin{proof}
Apply Corollary \ref{ratfuncexists} to $\mathcal D(\chi)$ and $\overline f$. 
There exists $\overline g\in\Gamma(\PP^1,\OO(\mathcal D(\chi)))$ and $a_1,a_2,b_1,b_2\in \CC$ 
such that $(a_1 \overline f+b_1 \overline g)/(a_2 \overline f+b_2 \overline g)=t$ 
on $\PP^1$.
The functions $\overline f$ and $\overline g$ cannot be proportional, otherwise 
$(a_1 \overline f+b_1 \overline g)/(a_2 \overline f+b_2 \overline g)$ would be a constant on $\PP^1$.
Then $\widetilde f$ and $\widetilde g$ cannot be proportional either, 
and $(a_1 \widetilde f+b_1 \widetilde g)/(a_2 \widetilde f+b_2 \widetilde g)$ is a rational function on $X$.
The rational function $(a_1 \widetilde f+b_1 \widetilde g)/(a_2 \widetilde f+b_2 \widetilde g)$, 
considered as a rational map from $X$ to $\PP^1$ (we suppose that it computes the coordinate $t$ of 
a point on $\PP^1$), coincides with $\pi$ by Proposition \ref{quotmorph}.
The pairs $(a_1,a_2)$ and $(b_1,b_2)$ cannot be proportional, and $\widetilde f(x)\ne0$, so the functions 
$(a_1 \widetilde f+b_1 \widetilde g)$ and 
$(a_2 \widetilde f+b_2 \widetilde g)$ cannot vanish simultaneously.
Therefore, the rational map from $X$ to $P^1$ 
defined by $t=(a_1 \widetilde f+b_1 \widetilde g)/(a_2 \widetilde f+b_2 \widetilde g)$
is defined at $x$. This rational map coincides with $\pi$, so we are done.
\end{proof}

So we define an open subset $U_0\subseteq X$ as follows: it consists of all points $x\in X$ such that 
there exists a degree $\chi\in\sck\cap M$ such that $\dim \Gamma(\PP^1,\mathcal O(\mathcal D(\chi)))\ge 2$ and there exists 
$f\in\Gamma(\PP^1,\mathcal O(\mathcal D(\chi)))$ such that $\widetilde f(x)\ne 0$. Corollary \ref{quotmorphdefined} shows 
that $\pi$ is defined on $U_0$. In fact, $\pi$ is not defined outside $U_0$, but we will not need this.

Our next goal is to understand fibers of $\pi$. First, consider an \textbf{ordinary} point $p\in \PP^1$. For 
every degree $\chi\in\sck\cap M$, the sections of $\Gamma(\PP^1,\mathcal O(\mathcal D(\chi)))$ do not have poles at $p$. 
For each $\chi\in\sck\cap M$, choose a basis 
$$
\rawdivisorbasis{p,\chi,1},\ldots,\rawdivisorbasis{p,\chi,\dim \Gamma(\PP^1,\mathcal O(\mathcal D(\chi)))}
$$
of $\Gamma(\PP^1,\mathcal O(\mathcal D(\chi)))$ such that 
$$
\olrawdivisorbasis{p,\chi,1}(p)=1, \quad \olrawdivisorbasis{p,\chi,2}(p)=
\ldots=\olrawdivisorbasis{p,\chi,\dim \Gamma(\PP^1,\mathcal O(\mathcal D(\chi)))}(p)=0.
$$
In particular, observe that for $\chi=0$ we have $\OO(\mathcal D(0))=\OO_{\PP^1}$, and 
the only global functions of degree $0$ are constants. The condition $\olrawdivisorbasis{p,0,1}(p)=1$
guarantees in this case that $\olrawdivisorbasis{p,0,1}=1$ and $\wtrawdivisorbasis{p,0,1}=1$ everywhere.
By Proposition \ref{quotmorph}, if $\pi(x)=p$ and $2\le i\le\dim\Gamma(\PP^1,\mathcal O(\mathcal D(\chi)))$, then 
$(\wtrawdivisorbasis{p,\chi,i}/\wtrawdivisorbasis{p,\chi,1}) (x)=(\olrawdivisorbasis{p,\chi,i}/\olrawdivisorbasis{p,\chi,1}) (p)=0$, so 
$\wtrawdivisorbasis{p,\chi,i} (x)=0$ since $\wtrawdivisorbasis{p,\chi,1}$ is a global function.


For every $\chi,\chi'\in\sck\cap M$, $a,a'\in \ZZ_{\ge 0}$, 
$(\rawdivisorbasis{p,\chi,1})^a (\rawdivisorbasis{p,\chi',1})^{a'}$ is an element of $\Gamma(\PP^1,\mathcal O(\mathcal D(a \chi+a' \chi')))$, so it can be written 
as 
$$
(\rawdivisorbasis{p,\chi,1})^a (\rawdivisorbasis{p,\chi',1})^{a'}=
\sum_i \rawdivisortransition{i,\chi,\chi',a,a'}\rawdivisorbasis{p,a\chi+a'\chi',i},\text{ where }\rawdivisortransition{i,\chi,\chi',a,a'}\in\CC.
$$
This equality holds for rational functions on $\PP^1$, and evaluation at $p$ shows that $\rawdivisortransition{1,\chi,\chi',a,a'}=1$.
The equality also holds for the corresponding global functions on $X$.

These computations prove the following lemma:
\begin{lemma}
For every $\chi,\chi'\in\sck\cap M$, $a,a'\in \ZZ_{\ge 0}$ and for every $x\in\pi^{-1}(p)$, 
$$
(\wtrawdivisorbasis{p,\chi,1}(x))^a (\wtrawdivisorbasis{p,\chi',1}(x))^{a'}=\wtrawdivisorbasis{p,a \chi+a' \chi',1}(x).
$$\qed
\end{lemma}

Recall that we have denoted the two rays on the boundary of $\sck$ by $\indexededge{\sck}{0}$
and $\indexededge{\sck}{1}$, and the primitive lattice vectors on these edges were denoted by
$\sckboundarybasis{0}$ and $\sckboundarybasis{1}$, respectively.
\begin{lemma}\label{genfiberstruct}
For a point $x\in X$, $x\in\pi^{-1} (p)\cap U_0$, there are at most three possibilities:
\begin{enumerate}
\item\label{gencase} For every $\chi\in\sck\cap M$, $\wtrawdivisorbasis{p,\chi,1}(x)\ne 0$.
\item\label{boundarya} For every 
$\chi\in\indexededge{\sck}{0}\cap M$,
$\wtrawdivisorbasis{p,\chi,1}(x)\ne 0$, 
and $\wtrawdivisorbasis{p,\chi,1}(x)=0$ for all other $\chi\in\sck\cap M$. This is possible if 
and only if 
$\deg\mathcal D(\sckboundarybasis{0})>0$
\item\label{boundaryb} For every 
$\chi\in\indexededge{\sck}{1}\cap M$,
$\wtrawdivisorbasis{p,\chi,1}(x)\ne 0$, 
and $\wtrawdivisorbasis{p,\chi,1}(x)=0$ for all other $\chi\in\sck\cap M$. This is possible if 
and only if $\deg \mathcal D(\sckboundarybasis{1})>0$.
\end{enumerate}
\end{lemma}
\begin{proof}
Until the end of the proof, denote the sublattice in $M$ generated by $\sckboundarybasis{0}$ and $\sckboundarybasis{1}$ by $M'$.
First, consider a degree $\chi'\in M'$. We know that if $\chi'=a_0 \sckboundarybasis{0}+a_1 \sckboundarybasis{1}$, 
then $\wtrawdivisorbasis{p,\chi',1}(x)=(\wtrawdivisorbasis{p,\sckboundarybasis{0},1}(x))^{a_0} (\wtrawdivisorbasis{p,\sckboundarybasis{1},1}(x))^{a_1}$.
So there can be four possibilities:
\begin{enumerate}
\item $\wtrawdivisorbasis{p,\sckboundarybasis{0},1}(x)\ne 0$ and $\wtrawdivisorbasis{p,\sckboundarybasis{1},1}(x)\ne 0$. Then $\wtrawdivisorbasis{p,\chi',1}(x)\ne 0$ 
for all $\chi'\in \sck\cap M'$.
\item $\wtrawdivisorbasis{p,\sckboundarybasis{0},1}(x)\ne 0$, but $\wtrawdivisorbasis{p,\sckboundarybasis{1},1}(x)=0$. Then for all $\chi'\in \sck\cap M'$ 
we have $\wtrawdivisorbasis{p,\chi',1}(x)\ne 0$ if and only if 
$\chi'\in\indexededge{\sck}{0}$.
\item $\wtrawdivisorbasis{p,\sckboundarybasis{0},1}(x)=0$, $\wtrawdivisorbasis{p,\sckboundarybasis{1},1}(x)\ne 0$. Similarly,
$\wtrawdivisorbasis{p,\chi',1}(x)\ne 0$ if and only if 
$\chi'\in\indexededge{\sck}{1}$.
\item $\wtrawdivisorbasis{p,\sckboundarybasis{0},1}(x)=\wtrawdivisorbasis{p,\sckboundarybasis{1},1}(x)=0$. Then $\wtrawdivisorbasis{p,\chi',1}(x)=0$ 
for all $\chi'\in \sck\cap M'$ except $\chi'=0$.
\end{enumerate}

Since $M'$ is a sublattice of finite index in $M$ (recall that $\dim M=2$), for every $\chi\in M$ there is $\chi'=a_0 \chi\in M'$, $a_0\in\NN$.
We have $\wtrawdivisorbasis{p,\chi',1}(x)=(\wtrawdivisorbasis{p,\chi,1}(x))^{a_0}$, so $\wtrawdivisorbasis{p,\chi,1}(x)=0$
if and only if $\wtrawdivisorbasis{p,\chi',1}(x)=0$. Therefore, the classification above also 
works for $\chi\in M$:
\begin{enumerate}
\item $\wtrawdivisorbasis{p,\sckboundarybasis{0},1}(x)\ne 0$ and $\wtrawdivisorbasis{p,\sckboundarybasis{1},1}(x)\ne 0$. Then $\wtrawdivisorbasis{p,\chi,1}(x)\ne 0$ 
for all $\chi\in \sck\cap M$.
\item\label{onlya} $\wtrawdivisorbasis{p,\sckboundarybasis{0},1}(x)\ne 0$, but $\wtrawdivisorbasis{p,\sckboundarybasis{1},1}(x)=0$. Then for all $\chi\in \sck\cap M$ 
we have $\wtrawdivisorbasis{p,\chi,1}(x)\ne 0$ if and only if 
$\chi\in\indexededge{\sck}{0}$.
\item\label{onlyb} $\wtrawdivisorbasis{p,\sckboundarybasis{0},1}(x)=0$, $\wtrawdivisorbasis{p,\sckboundarybasis{1},1}(x)\ne 0$. Similarly,
$\wtrawdivisorbasis{p,\chi,1}(x)\ne 0$ if and only if 
$\chi\in\indexededge{\sck}{1}$.
\item\label{allzeros} $\wtrawdivisorbasis{p,\sckboundarybasis{0},1}(x)=\wtrawdivisorbasis{p,\sckboundarybasis{1},1}(x)=0$. Then $\wtrawdivisorbasis{p,\chi,1}(x)=0$ 
for all $\chi\in \sck\cap M$ except $\chi=0$.
\end{enumerate}
Notice that case \ref{allzeros} is impossible in $U_0$, and case \ref{onlya} (resp. \ref{onlyb}) is 
possible if and only if there is a 
degree $\chi\in\indexededge{\sck}{0}\cap M$ (resp. $\chi\in\indexededge{\sck}{1}\cap M$) 
such that 
$\deg \mathcal D(\chi)>0$. Now recall that $\mathcal D(\chi)$ becomes a linear function
after a restriction to a line in $M$, so existence of such $\chi$ is equivalent to 
$\deg \mathcal D(\sckboundarybasis{0})>0$ (resp. $\deg \mathcal D(\sckboundarybasis{1})>0$).
\end{proof}

This lemma can be reformulated without mentioning bases of 
$\Gamma(\PP^1,\mathcal O(\mathcal D (\chi)))$ explicitly as follows:
\begin{proposition}\label{genfiberstructinvar}
For each $x\in\pi^{-1} (p)\cap U_0$, 
there exists a subcone $\tau\subseteq\sck$
such that if $\chi\in\sck\cap M$ 
and $f\in\Gamma(\PP^1,\mathcal O(\mathcal D(\chi)))$, then 
$$
\widetilde f(x)\ne 0\Leftrightarrow \chi\in \tau\text{ and }\ord_p (\overline f)=0.
$$
For the 
cone $\tau$
(which depends on $x$) there are at most three possibilities:
\begin{enumerate}
\item\label{gencaseinv} 
$\tau=\sck$.
\item\label{boundaryainv} 
$\tau=\indexededge{\sck}{0}$.
This is possible if and only if 
$\deg \mathcal D(\sckboundarybasis{0})>0$.
\item\label{boundarybinv} 
$\tau=\indexededge{\sck}{1}$.
This is possible if and only if 
$\deg \mathcal D(\sckboundarybasis{1})>0$.
\end{enumerate}
\end{proposition}
\begin{proof}
First, fix a degree $\chi\in\sck\cap M$. Notice that if $f\in\Gamma(\PP^1,\mathcal O(\mathcal D(\chi)))$, 
then $\ord_p (\overline f)=0$ if and only if the decomposition of $f$ into a linear combination 
of functions $\rawdivisorbasis{p,\chi,i}$ 
contains $\rawdivisorbasis{p,\chi,1}$ with a nonzero coefficient.
Now fix a point $x\in\pi^{-1}(p)\cap U_0$.
Recall that all functions $\wtrawdivisorbasis{p,\chi,i}$ for $i>1$ vanish on $\pi^{-1}(p)\cap U_0$.
We see that $\wtrawdivisorbasis{p,\chi,1}(x)\ne 0$
if and only if $\widetilde f(x)\ne 0$ for all $f\in\Gamma(\PP^1,\mathcal O(\mathcal D (\chi)))$ such that $\ord_p (\overline f)=0$.
We also see that, independently of the value of $\wtrawdivisorbasis{p,\chi,1}(x)$,
$\widetilde f(x)=0$ for all $f\in\Gamma(\PP^1,\mathcal O(\mathcal D (\chi)))$ such that $\ord_p (\overline f)>0$.
\end{proof}
Following \cite[Section 6.2]{ilten}, denote the set of all points $x\in \pi^{-1}(p)\cap U_0$ such that case \ref{gencaseinv} (resp. case \ref{boundaryainv}, \ref{boundarybinv}) holds
by $\orb(p, \indexedvertex{\6}{1})$ (resp. by $\orb(p, \indexededge{\6}{0})$, $\orb(p, \indexededge{\6}{1})$). 
In fact (see \cite[Section 6.2]{ilten}, \cite[Corollary 7.11, Theorem 10.1]{ahausen}), 
these sets are orbits of the torus, and their closures are affine toric varieties constructed by the standard toric construction from the cone $\sck$, 
but we will not need these facts.
Sometimes we can simply write $\orb(p, 0)$ instead of $\orb(p, \indexedvertex{\6}{1})$.

Now we are going to understand the structure of a fiber $\pi^{-1}(p)$ over a \textbf{special}
point $p=p_i$. The function $\chi\mapsto \min_{a\in \stdpolyhedronletter_p}\chi(a)$ (which defines the coefficient for 
$p$ in $\mathcal D(\chi)$, denote it shortly by $\mathcal D_p(\chi)$) is piecewise linear.
One checks easily that the maximal subcones of $\sck$ where $\mathcal D(\chi)$ is linear 
are exactly the cones $\normalvertexcone{\indexedvertex{\stdpolyhedronletter_p}{j}}{\stdpolyhedronletter_p}$ ($1\le j\le \numberofvertices{\stdpolyhedronletter_p}$), 
In what follows, we write $\numberofverticespt{p}$ instead of $\numberofvertices{\stdpolyhedronletter_p}$
and $\indexedvertexpt{p}{j}$ instead of $\indexedvertex{\stdpolyhedronletter_p}{j}$ for brevity.
Observe that $\numberofverticespt{p}=1$ if and only if $p$ is a removable special point.

This time we choose bases of $\Gamma(\PP^1,\mathcal O (\mathcal D(\chi)))$ as follows: let 
$$
\rawdivisorbasis{p,\chi,1},\ldots,\rawdivisorbasis{p,\chi,\dim \Gamma(\PP^1,\mathcal O(\mathcal D(\chi)))}
$$
be a basis of $\Gamma(\PP^1,\mathcal O (\mathcal D(\chi)))$
such that $\ord_p (\olrawdivisorbasis{p,\chi,1})=-\mathcal D_p(\chi)$ and 
$\ord_p (\olrawdivisorbasis{p,\chi,i})>-\mathcal D_p(\chi)$ for $i>1$. Then 
functions $\olrawdivisorbasis{p,\chi,i}/\olrawdivisorbasis{p,\chi,1}$ for $i>1$ are defined at $p$ and evaluate to $0$ there, 
so if $x\in\pi^{-1}(p)$, then by Proposition \ref{quotmorph}
$(\wtrawdivisorbasis{p,\chi,i}/\wtrawdivisorbasis{p,\chi,1}) (x)=0$, and 
$\wtrawdivisorbasis{p,\chi,i} (x)=0$ for $i>1$.
In this case we demand explicitly for $\chi=0$ that
$\olrawdivisorbasis{p,0,1}=1$ and $\wtrawdivisorbasis{p,0,1}=1$ everywhere.

Now let $\chi,\chi'\in\sck\cap M$, $a,a'\in \ZZ_{\ge 0}$, then 
$(\rawdivisorbasis{p,\chi,1})^a (\rawdivisorbasis{p,\chi',1})^{a'}$ is an element of $\Gamma(\PP^1,\mathcal O(\mathcal D(a \chi+a' \chi')))$, so it can be written 
as 
$$
(\rawdivisorbasis{p,\chi,1})^a (\rawdivisorbasis{p,\chi',1})^{a'}=
\sum_i\rawdivisortransition{i,\chi,\chi',a,a'}\rawdivisorbasis{p,a \chi+a' \chi',i},\text{ where }\rawdivisortransition{i,\chi,\chi',a,a'}\in\CC.
$$
We have $\ord_p (\olrawdivisorbasis{p,\chi,1})^a (\olrawdivisorbasis{p,\chi',1})^{a'}=-a \mathcal D_p(\chi)-a' \mathcal D_p(\chi')$, 
$\ord_p(\olrawdivisorbasis{p,a \chi+a' \chi',1})=-\mathcal D_p(a \chi+a' \chi')$ and 
$\ord_p(\olrawdivisorbasis{p,a \chi+a' \chi',i})>-\mathcal D_p(a \chi+a' \chi')$ for $i>1$. Therefore, 
$\rawdivisortransition{1,\chi,\chi',a,a'}\ne 0$ if and only if 
$a \mathcal D_p(\chi)+a' \mathcal D_p(\chi')=\mathcal D_p(a \chi+a' \chi')$
if and only if $a=0$ or $a'=0$ or $\chi$ and $\chi'$ are in the same subcone of $\sck$ where $\mathcal D_p(\cdot)$ is linear, 
i.~e. $\chi, \chi'\in \normalvertexcone{\indexedvertexpt{p}{j}}{\stdpolyhedronletter_p}$ for some $j$.

These computations prove the following lemma:
\begin{lemma}\label{productzero}
For every $\chi,\chi'\in\sck\cap M$, $a,a'\in \ZZ_{\ge 0}$ and for every $x\in\pi^{-1}(p)$, 
$(\wtrawdivisorbasis{p,\chi,1}(x))^a (\wtrawdivisorbasis{p,\chi',1}(x))^{a'}=\rawdivisortransition{1,\chi,\chi',a,a'}\wtrawdivisorbasis{p,a \chi+a' \chi',1}(x)$,
where $\rawdivisortransition{1,\chi,\chi',a,a'}$ depends on $p$ and on the choice of $\rawdivisorbasis{p,\chi,i}$, but not on $x$.
$\rawdivisortransition{1,\chi,\chi',a,a'}\ne 0$ if and only if $a=0$ or $a'=0$ or there exists 
a vertex $\indexedvertexpt{p}{j}$ of $\stdpolyhedronletter_p$
such that 
$\chi, \chi'\in \normalvertexcone{\indexedvertexpt{p}{j}}{\stdpolyhedronletter_p}$ 
(in other words, $\chi$ and $\chi'$ belong to the same cone of the normal fan of $\stdpolyhedronletter_p$).\qed
\end{lemma}
\begin{corollary}
Let $\chi,\chi'\in\sck\cap M$, $a,a'\in \NN$, $x\in\pi^{-1}(p)$. 
Suppose that there exist 
no vertex $\indexedvertexpt{p}{j}$
such that $\chi, \chi'\in \normalvertexcone{\indexedvertexpt{p}{j}}{\stdpolyhedronletter_p}$.
Then for every $f\in\Gamma(\PP^1,\mathcal O(\mathcal D(\chi)))$, $g\in\Gamma(\PP^1,\mathcal O(\mathcal D(\chi')))$
we have $\widetilde f(x)\widetilde g(x)=0$.\qed
\end{corollary}

\begin{lemma}
Let $x\in X$ be a point, $x\in\pi^{-1} (p)\cap U_0$. The set of degrees $\chi$ such that 
$\wtrawdivisorbasis{p,\chi,1} (x)\ne 0$ can be 
the set of all lattice points in one of the following cones: 
\begin{enumerate}
\item 
$\normalvertexcone{\indexedvertexpt{p}{j}}{\stdpolyhedronletter_p}$ for some $j$, $1\le j\le \numberofverticespt{p}$.
\item 
$\normalvertexcone{\indexededgept{p}{j}}{\stdpolyhedronletter_p}$ for some $j$, $0<j<\numberofverticespt{p}$.
\item 
$\normalvertexcone{\indexededgept{p}{j}}{\stdpolyhedronletter_p}$ for $j=0$ or $j=\numberofverticespt{p}$.
This is possible if and only if $\deg \mathcal D(\chi_j)>0$.
\end{enumerate}
\end{lemma}
\begin{proof}
Denote $\chi_j=\primitivelattice{\indexededgept{p}{j}}$ for $0\le j\le \numberofverticespt{p}$. (In particular, 
we have $\chi_0=\sckboundarybasis{0}$ and $\chi_{\numberofverticespt{p}}=\sckboundarybasis{1}$.
Consider all indices $j$ such that $\wtrawdivisorbasis{p,\chi_j,i} (x)\ne 0$. 
Since $\chi_j$ is in $\normalvertexcone{\indexedvertexpt{p}{j'}}{\stdpolyhedronletter_p}$ only for $j'=j$ or $j'=j-1$, there can be 
at most two such indices $j$, and if there are two of them, they should be two consecutive natural numbers. 

Suppose first that $\wtrawdivisorbasis{p,\chi_{j-1},i} (x)\ne 0$ and 
$\wtrawdivisorbasis{p,\chi_j,i} (x)\ne 0$ for some $j$. The argument is similar 
to the proof of Lemma \ref{genfiberstruct}. Namely, consider the sublattice 
in $M$ generated by $\chi_{j-1}$ and $\chi_j$. It is a sublattice of finite index, denote it 
by $M'$. For every $\chi'\in M'$, $\chi'=a \chi_{j-1}+a' \chi_j$ we have 
$$
\rawdivisortransition{1,\chi_{j-1},\chi_j,a,a'}\wtrawdivisorbasis{p,\chi',1}(x)=
(\wtrawdivisorbasis{p,\chi_{j-1},i}(x))^a (\wtrawdivisorbasis{p,\chi_j,i}(x))^{a'}\ne 0,
$$
so $\wtrawdivisorbasis{p,\chi',1}(x)\ne 0$. For every $\chi\in \normalvertexcone{\indexedvertexpt{p}{j}}{\stdpolyhedronletter_p}\cap M$ 
there exists $a''\in \NN$ such that $a'' \chi\in M'$, so 
$\wtrawdivisorbasis{p,a'' \chi,1}(x)\ne 0$. By lemma \ref{productzero}, 
$$
(\wtrawdivisorbasis{p,\chi,1}(x))^{a''}=\rawdivisortransition{1,\chi,0,a'',0}\wtrawdivisorbasis{p,a'' \chi,1}(x),
$$
and $\rawdivisortransition{1,\chi,0,a'',0}\ne 0$, so $\wtrawdivisorbasis{p,\chi,1}(x)\ne 0$.
Finally, for a degree $\chi\notin \normalvertexcone{\indexedvertexpt{p}{j}}{\stdpolyhedronletter_p}$ choose an arbitrary degree $\chi'$ in the 
interior of $\chi\in \normalvertexcone{\indexedvertexpt{p}{j}}{\stdpolyhedronletter_p}\cap M$. 
Then by Lemma \ref{productzero}, $\wtrawdivisorbasis{p,\chi,1}(x)\wtrawdivisorbasis{p,\chi',1}(x)=0$, 
we already know that $\wtrawdivisorbasis{p,\chi',1}(x)\ne 0$, so $\wtrawdivisorbasis{p,\chi,1}(x)=0$.

Now suppose that there exists a degree $\chi$ such that $\wtrawdivisorbasis{p,\chi,1} (x)\ne 0$
and $\chi$ is in the interior of a cone $\normalvertexcone{\indexedvertexpt{p}{j}}{\stdpolyhedronletter_p}$. Again denote the lattice generated 
by $\chi_{j-1}$ and $\chi_j$ by $M'$. There exists $a''\in \NN$ such that $\chi'=a'' \chi\in M'$.
We have 
$$
\rawdivisortransition{1,\chi,\chi,a'',0}\wtrawdivisorbasis{p,\chi',1}(x)=(\wtrawdivisorbasis{p,\chi,1}(x))^{a''},
$$
so $\wtrawdivisorbasis{p,\chi',1}(x)\ne0$. $\chi'$ is also in the interior of $\normalvertexcone{\indexedvertexpt{p}{j}}{\stdpolyhedronletter_p}$, so 
there exist $a,a'\in\NN$ such that $a \chi_{j-1}+a' \chi_j=\chi'$.
Again we have 
$$
(\wtrawdivisorbasis{p,\chi_{j-1},i}(x))^a (\wtrawdivisorbasis{p,\chi_j,i}(x))^{a'}=
\rawdivisortransition{1,\chi_{j-1},\chi_j,a,a'}\wtrawdivisorbasis{p,\chi',1}(x),
$$
where $\rawdivisortransition{1,\chi_{j-1},\chi_j,a,a'}\ne 0$, 
so $\wtrawdivisorbasis{p,\chi_{j-1},i}(x)\ne 0$ and $\wtrawdivisorbasis{p,\chi_j,i}(x)\ne 0$.
Therefore, if there exists a degree $\chi$ in the interior of a cone $\normalvertexcone{\indexedvertexpt{p}{j}}{\stdpolyhedronletter_p}$
such that $\wtrawdivisorbasis{p,\chi,1} (x)\ne 0$, then there are two indices $j'$ 
such that $\wtrawdivisorbasis{p,\chi_{j'},i} (x)\ne 0$.

Now consider the case when there is only one $j$ such that $\wtrawdivisorbasis{p,\chi_j,i} (x)\ne 0$. 
We already know that in this case for all degrees $\chi$ from the interiors of the cones $\normalvertexcone{\indexedvertexpt{p}{j}}{\stdpolyhedronletter_p}$,
we have $\wtrawdivisorbasis{p,\chi,1} (x)=0$. So the only possible degrees $\chi$ such that $\wtrawdivisorbasis{p,\chi,1} (x)\ne 0$
are multiples of $\chi_j=\primitivelattice{\indexededgept{p}{j}}$. And for these degrees we have 
$$
\rawdivisortransition{1,\chi_j,0,a,0}\wtrawdivisorbasis{p,a \chi_j,1}(x)=(\wtrawdivisorbasis{p,\chi_j,i}(x))^a,
$$
so $\wtrawdivisorbasis{p,a \chi_j,1}(x)\ne0$. Such $x$ can be in $U_0$ only if 
$\deg \mathcal D(\chi_j)>0$. Properness guarantees this for $0<j<\numberofverticespt{p}$, and 
for $j=0$ or $j=\numberofverticespt{p}$ we have to check this explicitly.
\end{proof}
And again this lemma can be reformulated without referring to bases of 
$\Gamma(\PP^1,\mathcal O(\mathcal D(\chi)))$.
\begin{proposition}\label{specfiberstructinvar}
For each $x\in\pi^{-1} (p)\cap U_0$, 
there exists 
there exists a subcone $\tau\subseteq\sck$
such that if $\chi\in\sck\cap M$ 
and $f\in\Gamma(\PP^1,\mathcal O(\mathcal D(\chi)))$, then 
$$
\widetilde f(x)\ne 0\Leftrightarrow \chi\in \tau\text{ and }\ord_p (\overline f)=-\mathcal D_p(\chi).
$$
$\tau$ can be one of the following cones:
\begin{enumerate}
\item\label{specfibergencaseinv} 
The normal subcone $\normalvertexcone{\indexedvertexpt{p}{j}}{\stdpolyhedronletter_p}$ of a vertex $\indexedvertexpt{p}{j}$
of $\stdpolyhedronletter_p$.
\item\label{specfiberinneredgeinv} 
The normal subcone $\normalvertexcone{\indexededgept{p}{j}}{\stdpolyhedronletter_p}$ of a finite edge $\indexededgept{p}{j}$
($0<j<\numberofverticespt{p}$).
\item\label{specfiberboundaryinv} 
The normal subcone $\normalvertexcone{\indexededgept{p}{j}}{\stdpolyhedronletter_p}$ of an infinite edge $\indexededgept{p}{j}$
($j=0$ or $j=\numberofverticespt{p}$, respectively).
This is possible if and only if $\deg \mathcal D(\sckboundarybasis{0})>0$ or $\deg \mathcal D(\sckboundarybasis{1})>0$, respectively.
\end{enumerate}
\end{proposition}

\begin{proof}
The proof is very similar to the proof of Proposition \ref{genfiberstructinvar}.
Again, we fix a degree $\chi\in\sck\cap M$ and notice that if $f\in\Gamma(\PP^1,\mathcal O(\mathcal D(\chi)))$, 
then $\ord_p (\overline f)=-\mathcal D_p(\chi)$ if and only if the decomposition of $f$ into a linear combination 
of functions $\rawdivisorbasis{p,\chi,i}$ 
contains $\rawdivisorbasis{p,\chi,1}$ with a nonzero coefficient.
Fix a point $x\in\pi^{-1}(p)\cap U_0$.
Again for all functions $\wtrawdivisorbasis{p,\chi,i}$, where $i>1$, we have $\wtrawdivisorbasis{p,\chi,i}(x)=0$.
Therefore, $\wtrawdivisorbasis{p,\chi,1}(x)\ne 0$
if and only if $\widetilde f(x)\ne 0$ for all $f\in\Gamma(\PP^1,\mathcal O(\mathcal D (\chi)))$ such that $\ord_p (\overline f)=-\mathcal D_p(\chi)$.
And, independently of the value of $\wtrawdivisorbasis{p,\chi,1}(x)$,
$\widetilde f(x)=0$ for all $f\in\Gamma(\PP^1,\mathcal O(\mathcal D (\chi)))$ such that $\ord_p (\overline f)>-\mathcal D_p(\chi)$.
\end{proof}

And again, following \cite[Section 6.2]{ilten}, we denote the set of all points $x\in \pi^{-1}(p)\cap U_0$ such that case \ref{specfibergencaseinv} 
(resp. case \ref{specfiberinneredgeinv} or \ref{specfiberboundaryinv}) holds
by $\orb(p, \indexedvertexpt{p}{j})$ (resp. by $\orb(p, \indexededgept{p}{j})$). 
In fact (see \cite[Section 6.2]{ilten}, \cite[Corollary 7.11, Theorem 10.1]{ahausen}), 
these sets are orbits of the torus.

It follows easily from Proposition \ref{specfiberstructinvar} that for each vertex $\indexedvertexpt{p}{j}$ ($1\le j\le \numberofverticespt{p}$),
$$
\overline{\orb(p,\indexedvertexpt{p}{j})}=\orb(p,\indexededgept{p}{j-1})\cup \orb(p,\indexedvertexpt{p}{j})\cup \orb(p,\indexedvertexpt{p}{j}).
$$
Moreover, all sets $\orb(p,\indexedvertexpt{p}{j})$ are two-dimensional, and all sets 
$\orb(p,\indexedvertexpt{p}{j})$ ($0\le j\le \numberofverticespt{p}$) are one-dimensional.
This is illustrated by Fig.~\ref{figfiberstruct}.

\begin{figure}[h!]
\begin{center}
\includegraphics{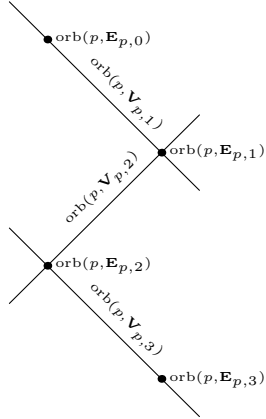}
\end{center}
\caption{Structure of a fiber of $\pi$ over a special point $p$: lines show two-dimensional components, points show one-dimensional curves inside.}
\label{figfiberstruct}
\end{figure}


\section{Sufficient systems of open subsets of $X$}

We are going to use Theorem \ref{schlessgen}, Leray spectral sequence for the map 
$\pi$ and Proposition 
\ref{computederived} 
to compute $T^1(X)$.
To do this, we need an open subset $U\subseteq X$ suitable for Theorem 
\ref{schlessgen} (i.~e. smooth and such that $\codim_X(X\setminus U)\ge 2$) 
and an affine covering of $U$. We first choose several affine subsets of $X$. The amount of these 
sets will be denoted by $\numberoffixedopensets$, the sets themselves will be denoted by
$U_i$ ($1\le i\le \numberoffixedopensets$). Then we will set $U=\bigcup U_i$.
As we will see later, the intersection of a set $U_i$ and a fiber of $\pi$ 
will be either an empty set, or a two-dimensional torus orbit,
or the union of a two-dimensional and a one-dimensional torus orbit. In the last base 
the one-dimensional orbit belongs to the closure of the two-dimensional orbit, 
and the entire intersection is isomorphic to $(\CC^*)\times \CC$.
Very roughly and informally speaking, each set $U_i$ will correspond to a choice 
of several special points and of two-dimensional orbits
in the fibers above these points, one orbit above each special point.

To define a set $U_i$, we fix the following data:
\begin{enumerate}
\item a pair of degrees $(\uidegree{i,1}, \uidegree{i,2})\in\sck\cap M$ generating $M$ as a lattice and such that
$\deg\mathcal D(\uidegree{i,1})>0$, $\deg\mathcal D(\uidegree{i,2})>0$, and $\uidegree{i,2}$ is in the interior of $\sck$,
\item two sections $\uithreadfunction{i,1}\in\Gamma(\PP^1,\OO (\mathcal D (\uidegree{i,1})))$, $\uithreadfunction{i,2}\in\Gamma(\PP^1,\OO (\mathcal D (\uidegree{i,2})))$.
\item Let $V_i\subseteq\PP^1$ be an arbitrary open subset of the set of all points $p\in\PP^1$ such that:
\begin{enumerate}
\item $\ord_p (\oluithreadfunction{i,1})=-\mathcal D_p(\uidegree{i,1})$, $\ord_p (\oluithreadfunction{i,2})=-\mathcal D_p(\uidegree{i,2})$ 
(in particular, if $p$ is an ordinary point, $\ord_p (\oluithreadfunction{i,1})=\ord_p (\oluithreadfunction{i,2})=0$).
\item If $p$ is a special point and $\uidegree{i,1}$ is in the interior of $\sck$, 
then $\uidegree{i,1}$ and $\uidegree{i,2}$ are in the interior of the same 
normal subcone $\normalvertexcone{\indexedvertexpt{p}{j}}{\stdpolyhedronletter_p}$
of the same vertex $\indexedvertexpt{p}{j}$.
\item If $p$ is a special point and $\uidegree{i,1}\in\indexededge{\sck}{0}$, then $\uidegree{i,2}$ is in the 
interior of 
$\normalvertexcone{\indexedvertexpt{p}{0}}{\stdpolyhedronletter_p}$.
\item If $p$ is a special point and $\uidegree{i,1}\in\indexededge{\sck}{1}$, then $\uidegree{i,2}$ is in the 
interior of 
$\normalvertexcone{\indexedvertexpt{p}{\numberofverticespt{p}}}{\stdpolyhedronletter_p}$.
\end{enumerate}
\end{enumerate}
After these data are fixed, we will denote the basis of $N$ dual to the basis $\uidegree{i,1},\uidegree{i,2}$ of $M$ 
by $\uidegree{i,1}^*,\uidegree{i,2}^*$. In other words, for each $\chi\in M$ we have $\chi=\uidegree{i,1}^*(\chi)\uidegree{i,1}+
\uidegree{i,2}^*(\chi)\uidegree{i,2}$.

$U_i$ is defined to be the set of points $x\in U_0\subseteq X$ such that:
\begin{enumerate}
\item $\pi(x)\in V_i$,
\item $\wtuithreadfunction{i,1}(x)\ne 0$,
\item if $\uidegree{i,1}$ is in the interior of $\sck$, then $\wtuithreadfunction{i,2}(x)\ne 0$.
\end{enumerate}

\begin{lemma}\label{uifiberstruct}
If $p\in V_i$ is an ordinary point, then:
\begin{enumerate}
\item If $\uidegree{i,1}\in\indexededge{\sck}{0}$, then $\pi^{-1}(p)\cap U_i=\orb(p,\indexededge{\6}{0})\cup \orb(p,0)$.
\item If $\uidegree{i,1}\in\indexededge{\sck}{1}$, then $\pi^{-1}(p)\cap U_i=\orb(p,\indexededge{\6}{1})\cup \orb(p,0)$.
\item If $\uidegree{i,1}$ is a degree in the interior of $\sck$, then $\pi^{-1}(p)\cap U_i=\orb(p,\orb(p,0)$.
\end{enumerate}
If $p\in V_i$ is a special point, then:
\begin{enumerate}
\item If $\uidegree{i,1}\in\indexededge{\sck}{0}$, then $\pi^{-1}(p)\cap U_i=\orb(p,\indexededgept{p}{0})\cup \orb(p,\indexedvertexpt{p}{1})$.
\item If $\uidegree{i,1}\in\indexededge{\sck}{1}$, then $\pi^{-1}(p)\cap U_i=\orb(p,\indexededgept{p}{\numberofverticespt p})\cup \orb(p,\indexedvertexpt{p}{\numberofverticespt p})$.
\item If $\uidegree{i,1}$ is a degree in the interior of $\sck$, and 
$\uidegree{i,1},\uidegree{i,2}\in\normalvertexcone{\stdpolyhedronletter_p,\indexedvertexpt{p}{j}}$, then 
$\pi^{-1}(p)\cap U_i=\orb(p,\orb(p,\indexedvertexpt{p}{j})$.
\end{enumerate}
\end{lemma}

\begin{proof}
This follows directly from the definitions of the $\orb(p,\cdot)$ sets and of $U_i$.
\end{proof}

Fig. \ref{figmanyfiberstuct} shows how a set $U_i$ can intersect the fibers of $\pi$ in $U_0$.

\begin{figure}[!h]
\begin{center}
\includegraphics{t1_3fold_figures-7.mps}
\end{center}
\caption{An example of the intersections of a set $U_i$ with the fibers of $\pi$ in $U_0$. Here $p$ is the only special 
point, $\numberofverticespt p=3$, $\deg \mathcal D(\sckboundarybasis{0})>0$, $\deg \mathcal D(\sckboundarybasis{1})>0$, 
and $\uidegree{i,1}=\sckboundarybasis{0}$. The gray point in $\PP^1$ is outside $V_i$.
The intersections of individual fibers with $U_i$ are shown in black, and their complements are shown in gray.}\label{figmanyfiberstuct}
\end{figure}

We say that sets $U_i$ defined this way \textit{form a sufficient system} if 
\begin{enumerate}
\item for every ordinary point $p\in \PP^1$ there exists $i$ such that $p\in V_i$,
\item for every special point $p\in\PP^1$ and for every 
normal subcone
$\normalvertexcone{\indexedvertexpt{p}{j}}{\stdpolyhedronletter_p}$
there exists an index $i$ such that $p\in V_i$ and $\uidegree{i,1}, \uidegree{i,2}\in\normalvertexcone{\indexedvertexpt{p}{j}}{\stdpolyhedronletter_p}$,
\item for every primitive degree $\chi\in\partial\sck$ such that $\deg\mathcal D(\chi)>0$
and for every 
point $p\in\PP^1$ there exists an index $i$ such that $\uidegree{i,1}=\chi$ and $p\in V_i$.
\end{enumerate}
Clearly, sufficient systems exist. 
An example of a sufficient system is constructed in Section \ref{sectsuffsystemconstruction}.
Fix a sufficient system and set $U=\bigcup U_i$. Denote the number of sets 
$U_i$ in the sufficient system we chose by $\numberoffixedopensets$.

We are going to prove that $\codim_X(X\setminus U)\ge 2$, i.~e. that 
$\dim (X\setminus U)\le 1$. 
\begin{lemma}
$\dim (X\setminus U_0)\le 1$.
\end{lemma}
\begin{proof}
Let $x\in X\setminus U_0$. For every degree $\chi\in\sck\cap M$ such that 
$\deg \mathcal D(\chi)>0$, for every $f\in\Gamma(\PP^1,\OO(\mathcal D (\chi)))$ we have $\widetilde f(x)=0$.
$\deg \mathcal D(\chi)$ can be zero only if $\chi\in\partial\sck$. 
If there are functions $f\in\Gamma(\PP^1,\OO(\mathcal D (\sckboundarybasis{0})))$, $g\in\Gamma(\PP^1,\OO(\mathcal D (\sckboundarybasis{1})))$ that do not vanish 
at $x$, then $fg\in\Gamma(\PP^1,\OO(\mathcal D(\sckboundarybasis{0}+\sckboundarybasis{1})))$, $\widetilde f(x)\widetilde g(x)\ne 0$, but 
$\sckboundarybasis{0}+\sckboundarybasis{1}\notin \partial\sck$. So for at most one of the degrees $\sckboundarybasis{0}$ and $\sckboundarybasis{1}$ there
are functions of this degree that vanish at $x$. Without loss of generality suppose that 
if $f\in\Gamma(\PP^1,\OO(\mathcal D (\sckboundarybasis{0})))$, then $\widetilde f(x)=0$. If $\deg \mathcal D(\sckboundarybasis{1})>0$, then 
$\deg \mathcal D(\chi)>0$ for all multiples $\chi$ of $\sckboundarybasis{1}$, so for every such $\chi$ all functions 
of degree $\chi$ vanish at $x$. Otherwise $\dim\Gamma(\PP^1,\OO(\mathcal D(\chi)))=1$ 
for every multiple $\chi$ of $\sckboundarybasis{0}$, and if $f\in\Gamma(\PP^1,\OO(\mathcal D(\sckboundarybasis{0})))$, $f\ne 0$, then 
$f^a$ generate $\Gamma(\PP^1,\OO(\mathcal D(a \sckboundarybasis{0})))$ as a vector space, so
all functions of degree $a \sckboundarybasis{0}$ vanish at $x$.
Summarizing, we conclude that if $\chi\in\indexededge{\sck}{0}\cap M$, then all functions of degree $\chi$ vanish at $x$.
Consequently, if $\deg \mathcal D(\sckboundarybasis{1})>0$, then all functions of nonzero degree, i.~e. all 
nonconstant functions on $X$ vanish at $x$. There exists only one such point $x$.
Otherwise, if $f$ forms a basis of $\Gamma(\PP^1,\OO(\mathcal D(\sckboundarybasis{1})))$, then $f^a$ forms a basis of 
$\Gamma(\PP^1,\OO(\mathcal D(a \sckboundarybasis{1})))$, so values of all functions of all degrees at $x$ are 
determined by $\widetilde f(x)$. Therefore, such points $x$ form a 1-dimensional subset.
\end{proof}

Now we are going to consider points from $U_0$.
\begin{lemma}\label{ordinaryfiberfull}
For every ordinary point $p\in\PP^1$ we have $\pi^{-1}(p)\cap U_0=\pi^{-1}(p)\cap U$.
\end{lemma}
\begin{proof}
Clearly, $\pi^{-1}(p)\cap U_0\subseteq\pi^{-1}(p)\cap U$. To prove the other inclusion, 
we use the description of $\pi^{-1}(p)\cap U_0$ from Proposition \ref{genfiberstructinvar}.
Recall that if $p\in V_i$ for some index $i$, then $\ord_p (\oluithreadfunction{i,1})=
\ord_p (\oluithreadfunction{i,2})=0$. 
If $x\in\orb(p,0)$,
then
it is sufficient to take any index $i$ such that $p\in V_i$ (it exists by the definition of 
a sufficient system). Then by Proposition \ref{genfiberstructinvar}, 
$\wtuithreadfunction{i,1}(x)\ne 0$, $\wtuithreadfunction{i,2}(x)\ne 0$, and $x\in U_i$.
If $x\in\orb (p,\indexededge{\6}{0})$,
then $\deg\mathcal D (\sckboundarybasis{0})>0$, and there 
exists an index $i$ such that $\sckboundarybasis{0}=\uidegree{i,1}$ and $p\in V_i$. Then $f_i$ is a function of degree $\sckboundarybasis{0}$, 
so Proposition \ref{genfiberstructinvar} says that $\wtuithreadfunction{i,1}(x)\ne 0$, and, since 
$\deg\mathcal D (\sckboundarybasis{0})>0$, this is enough for $x$ to be in $U_i$.
The case $x\in\orb (p,\indexededge{\6}{1})$
can be considered similarly.
\end{proof}

Now we are going to consider the fiber of $\pi$ over a special point $p\in\PP^1$. 
\begin{lemma}\label{specialfiberdense}
Let $p\in\PP^1$ be a special point. Then $\dim(\pi^{-1}(p)\cap(U_0\setminus U))\le 1$.
\end{lemma}
\begin{proof}
We use the description of $\pi^{-1}(p)\cap U_0$ from Proposition \ref{specfiberstructinvar}.
First, pick a vertex $\indexedvertexpt{p}{j}$ ($1\le j\le\numberofverticespt{p}$)
and consider a point $x\in\orb(p,\indexedvertexpt{p}{j})$.
Since the system $\{U_i\}$ is sufficient, there exists $i$ such that 
$\uidegree{i,1}, \uidegree{i,2}\in\normalvertexcone{\indexedvertexpt{p}{j}}{\stdpolyhedronletter_p}$ and $p\in V_i$. By the definition of $V_i$, 
$\ord_p (\oluithreadfunction{i,1})=\mathcal D_p(\uidegree{i,1})$ and $\ord_p (\oluithreadfunction{i,2})=\mathcal D_p(\uidegree{i,2})$, 
and by the definition 
of $\orb(p,\indexedvertexpt{p}{j})$,
$\wtuithreadfunction{i,1}(x)\ne 0$ and $\wtuithreadfunction{i,2}(x)\ne 0$. Hence, 
$x\in U_i$. Therefore, if $x\in\pi^{-1}(p)\cap U_0$, but $x\notin\pi^{-1}(p)\cap U$, 
then 
$x\in\orb(p,\indexededgept{p}{j})$ for some (finite or infinite) edge $\indexededgept{p}{j}$.

It is sufficient to prove that 
for each (finite or infinite) edge $\indexededgept{p}{j}$, 
we have $\dim \orb(p,\indexededgept{p}{j})\le 1$.
Denote $\chi=\primitivelattice{\normalvertexcone{\indexededgept{p}{j}}{\stdpolyhedronletter_p}}$ and 
choose a basis 
$$
\rawdivisorbasis{p,\chi,1}, \ldots, \rawdivisorbasis{p,\chi,\dim\Gamma(\PP^1,\OO(\mathcal D(\chi)))}
$$
of $\Gamma(\PP^1,\OO(\mathcal D (\chi)))$ 
as previously, i.~e. 
so that $\ord_p (\olrawdivisorbasis{p,\chi,1})=-\mathcal D_p(\chi)$, and
$\ord_p (\olrawdivisorbasis{p,\chi,l})>-\mathcal D_p(\chi)$ for $1<l\le \dim\Gamma(\PP^1,\OO(\mathcal D(\chi)))$.
Consider a degree $\chi'=a \chi$, $a\in\NN$.
Choose a basis of $\Gamma(\PP^1,\OO(\mathcal D(\chi')))$ as follows. Its first element is $\rawdivisorbasis{p,\chi',1}=(\rawdivisorbasis{p,\chi,1})^a$,
so we have $\ord_p (\olrawdivisorbasis{p,\chi',1})=-a\mathcal D_p(\chi)=-\mathcal D (\chi')$.
All other elements of the basis,
denoted by 
$$
\rawdivisorbasis{p,\chi',2},\ldots, \rawdivisorbasis{p,\chi',\dim\Gamma(\PP^1,\OO(\mathcal D(\chi')))},
$$
satisfy 
$\ord_p (\olrawdivisorbasis{p,\chi',l})>-\mathcal D_p(\chi')$. 
We have already seen for such a basis that 
$\wtrawdivisorbasis{p,\chi',l}(x)=0$ for all $x\in \pi^{-1}(p)\cap U_0$, $l>1$. So again values of 
all functions of all degrees at $x\in \orb(p,\indexededgept{p}{j})$ are determined by $\wtrawdivisorbasis{p,\chi,1}(x)$, 
and $W_j$ is at most one-dimensional.
\end{proof}

We are going to use $\{U_i\}$ to compute cohomology groups, so we are going to prove that 
all $U_i$ are affine. Fix an index $i$.
\begin{lemma}\label{convexity}
Let $\chi\in\sck\cap M$ be a degree. 
Let $p\in V_i$.
Then, independently of the signs of 
$\uidegree{i,1}^*(\chi)$ and $\uidegree{i,2}^*(\chi)$,
$\mathcal D_p(\chi)\le \uidegree{i,1}^*(\chi)\mathcal D_p(\uidegree{i,1})+\uidegree{i,2}^*(\chi)\mathcal D_p(\uidegree{i,2})$.
\end{lemma}

\begin{proof}
Recall that the function $\mathcal D_p(\cdot)$ is always linear on the cone spanned by $\uidegree{i,1}$ and $\uidegree{i,2}$ if $p\in V_i$. 
Hence, if $\uidegree{i,1}^*(\chi)\ge 0$ and $\uidegree{i,2}^*(\chi)\ge 0$, then 
$\mathcal D_p(\chi)=
\mathcal D_p(\uidegree{i,1}^*(\chi)\uidegree{i,1}+\uidegree{i,2}^*(\chi)\uidegree{i,2})=
\uidegree{i,1}^*(\chi)\mathcal D_p(\uidegree{i,1})+\uidegree{i,2}^*(\chi)\mathcal D_p(\uidegree{i,2})$. 
If $\uidegree{i,1}^*(\chi) <0$ or $\uidegree{i,2}^*(\chi) <0$, in other words, 
if $\chi$ is not in the cone generated by $\uidegree{i,1}$ and $\uidegree{i,2}$, 
then, since $\mathcal D_p(\cdot)$ is a convex function, $\mathcal D_p(\chi)\le 
\uidegree{i,1}^*(\chi)\mathcal D_p(\uidegree{i,1})+\uidegree{i,2}^*(\chi)\mathcal D_p(\uidegree{i,2})$.
\end{proof}

\begin{lemma}\label{uistructure}
$U_i$ is isomorphic to $V_i\times (\CC\setminus 0)\times L$, where $L$ is 
isomorphic to $\CC$ or $\CC\setminus 0$. More exactly, $L=\CC$ if 
and only if $\uidegree{i,1}\in\partial\sck$, otherwise $L=\CC\setminus 0$. 
$V_i$ is isomorphic to an open set in an affine line. The isomorphism
is given by $(\pi, \wtuithreadfunction{i,1}, \wtuithreadfunction{i,2})$.
(Note that despite $\pi$ is rational on $X$, it is defined everywhere on $U_i$ since $U_i\subseteq U_0$ by definition.)
\end{lemma}

\begin{proof}
We know that $V_i\subseteq \PP^1$, and to prove that $V_i$ is isomorphic to an open 
subset in an affine line, it is sufficient to prove that $V_i$ cannot be equal to $\PP^1$.
Indeed, if $p\in V_i$, then, in particular, $\ord_p(\oluithreadfunction{i,1})=\mathcal D_p(\uidegree{i,1})$. 
If $V_i=\PP^1$, this would mean that $\div(\oluithreadfunction{i,1})=\mathcal D(\uidegree{i,1})$. But 
$\deg \mathcal D(\uidegree{i,1})>0$, and $\deg\div(\oluithreadfunction{i,1})=0$.

Consider the map $U_i\to V_i\times (\CC\setminus 0)\times L$
given by $(\pi, \wtuithreadfunction{i,1}, \wtuithreadfunction{i,2})$ (recall that 
$\wtuithreadfunction{i,2}=0$ is possible in $U_i$ if and only if $\uidegree{i,1}\in\partial\sck$). 
To define 
its inverse, we need for every triple $(p,t_1, t_2)$, where 
$p\in V_i$, $t_1\in\CC\setminus 0$, $t_2\in L$, define a point $x\in U_i$. 
To do this, we define a homomorphism $\CC[X]\to \CC$. 
We define it on each graded component of $\CC[X]$.

Let $\chi\in\sck\cap M$ be a degree. 
%
By Lemma \ref{convexity},
$$
\mathcal D_p(\chi)\le \uidegree{i,1}^*(\chi)\mathcal D_p(\uidegree{i,1})+\uidegree{i,2}^*(\chi)\mathcal D_p(\uidegree{i,2})=
-\ord_p(\oluithreadfunction{i,1}^{\uidegree{i,1}^*(\chi)}\oluithreadfunction{i,2}^{\uidegree{i,2}^*(\chi)}).
$$
Therefore, if $f\in\Gamma(\PP^1,\OO(\mathcal D(\chi)))$, then 
$$
\ord_p(\overline f)\ge -\mathcal D_p(\chi)\ge 
\ord_p(\oluithreadfunction{i,1}^{\uidegree{i,1}^*(\chi)}\oluithreadfunction{i,2}^{\uidegree{i,2}^*(\chi)}),
$$
and the rational
function $\overline f/(\oluithreadfunction{i,1}^{\uidegree{i,1}^*(\chi)}\oluithreadfunction{i,2}^{\uidegree{i,2}^*(\chi)})$ is defined at $p$.

Now we define a map $\CC[X]\to\CC$ as follows: if $f\in\Gamma(\PP^1,\OO(\mathcal D(\chi)))$, then 
$$
f\mapsto t_1^{\uidegree{i,1}^*(\chi)} t_2^{\uidegree{i,2}^*(\chi)}
(\overline f/(\oluithreadfunction{i,1}^{\uidegree{i,1}^*(\chi)}\oluithreadfunction{i,2}^{\uidegree{i,2}^*(\chi)})) (p).
$$
Note that $\uidegree{i,2}^*(\chi)<0$ is possible if and only if $\uidegree{i,1}\notin\partial \sck$, i.~e. exactly 
if and only if $L=\CC\setminus 0$. 
It is clear from the construction that this map is an algebra homomorphism, so it 
defines a point $x\in X$.
If we choose a set of homogeneous generators of $\CC[X]$, we see that the values of these generators 
at $x$ depend algebraically on $p$, $t_1$, and $t_2$, so we have defined an algebraic morphism 
$\varphi\colon V_i\times (\CC\setminus 0)\times L\to X$. 

Now we are going to prove that two morphisms we have 
defined are mutually inverse. Fix points $p\in V_i$, $t_1\in \CC\setminus 0$, and $t_2\in L$, 
denote $x=\varphi(p,t_1,t_2)$. First, $x\in U_0$ 
since $\deg \mathcal D(\uidegree{i,1})>0$ and $\wtuithreadfunction{i,1}(x)=
t_1^1 t_2^0(\oluithreadfunction{i,1}/(\oluithreadfunction{i,1}^1\oluithreadfunction{i,2}^0)) (p)=t_1\ne 0$.
Now denote $\pi(x)=p'$. 
For every degree 
$\chi\in\sck\cap M$
and for every pair of functions $f_1,f_2\in\Gamma(\PP^1,\OO(\mathcal D(\chi)))$
we have the following equalities of rational functions ($p''\in V_i$, $t_1'\in\CC\setminus 0$, $t_2'\in L$ are 
arbitrary points): 
\begin{multline*}
(\widetilde {f_1}/\widetilde {f_2}) (\varphi (p'', t_1', t_2'))=\\
(t_1'^{\uidegree{i,1}^*(\chi)} t_2'^{\uidegree{i,2}^*(\chi)}
(\overline {f_1}/(\oluithreadfunction{i,1}^{\uidegree{i,1}^*(\chi)}\oluithreadfunction{i,2}^{\uidegree{i,2}^*(\chi)})) (p''))/
(t_1'^{\uidegree{i,1}^*(\chi)} t_2'^{\uidegree{i,2}^*(\chi)}
(\overline {f_2}/(\oluithreadfunction{i,1}^{\uidegree{i,1}^*(\chi)}\oluithreadfunction{i,2}^{\uidegree{i,2}^*(\chi)})) (p''))=\\
(\overline{f_1}/\overline{f_2})(p'').
\end{multline*}
Choose a degree $\chi$ such that $\deg\mathcal D(\chi)>0$. By Corollary
\ref{ratfuncexists}, there exist functions $f_1, f_2\in \Gamma(\PP^1,\OO(\mathcal D(\chi)))$ such that 
$\overline{f_1}/\overline{f_2}$ is defined at $p'$, and if $(\overline{f_1}/\overline{f_2})(p')=
(\overline{f_1}/\overline{f_2})(p'')$ for some $p''\in\PP^1$, then $p'=p''$. By Proposition \ref{quotmorph}, 
$\widetilde{f_1}/\widetilde{f_2}$ is defined at $x$, and 
$(\widetilde{f_1}/\widetilde{f_2})(x)=(\overline{f_1}/\overline{f_2})(p')$.
On the other hand, it follows from the computation above that 
$(\widetilde{f_1}/\widetilde{f_2})(x)=(\widetilde {f_1}/\widetilde {f_2}) (\varphi (p, t_1, t_2))=
(\overline{f_1}/\overline{f_2})(p)$, so $p=p'$, and $\pi (x)=p$. We have already checked that 
$\wtuithreadfunction{i,1}(x)=t_1$, a similar computation shows that $\wtuithreadfunction{i,2}(x)=t_2$. 
The conditions from the definition of $U_i$ are therefore satisfied, and $x\in U_i$.

Finally, check that the other composition of morphisms $X\to V_i\times (\CC\setminus 0)\times L\to X$
is also the identity morphism. 
To do this, fix a point $x\in U_i$, a degree 
$\chi\in\sck\cap M$
and 
a function $f\in\Gamma(\PP^1,\OO(\mathcal D(\chi)))$. We have the following equality of rational 
functions: 
$$
\widetilde f(x)=\wtuithreadfunction{i,1}(x)^{\uidegree{i,1}^*(\chi)}\wtuithreadfunction{i,2}(x)^{\uidegree{i,2}^*(\chi)}
(\widetilde f/(\wtuithreadfunction{i,1}^{\uidegree{i,1}^*(\chi)}\wtuithreadfunction{i,2}^{\uidegree{i,2}^*(\chi)}))(x),$$ 
and
$$
(\widetilde f/(\wtuithreadfunction{i,1}^{\uidegree{i,1}^*(\chi)}\wtuithreadfunction{i,2}^{\uidegree{i,2}^*(\chi)}))(x)=
(\overline f/(\oluithreadfunction{i,1}^{\uidegree{i,1}^*(\chi)}\oluithreadfunction{i,2}^{\uidegree{i,2}^*(\chi)}))(\pi(x))
$$
since $f$ and ${\uithreadfunction{i,1}}^{\uidegree{i,1}^*(\chi)} {\uithreadfunction{i,2}}^{\uidegree{i,2}^*(\chi)}$ are functions of the same degree.
\end{proof}

Since each set $U_i$ is affine and $X$ is separated, all intersections of sets $U_i$ are also affine, 
and we can use them to compute Cech cohomology on $U=\bigcup U_i$.
However, we will also need 
to understand the structure of intersections of $U_i$.
Fix several indices 
$a_1,\ldots, a_k$.
\begin{lemma}\label{uijstructure}
$U'=U_{a_1}\cap\ldots\cap U_{a_k}$ is isomorphic 
to 
$V'\times (\CC\setminus 0)\times L'$, where $V'$ is an open subset of 
$V_{a_1}$, and $L'$ is isomorphic to $\CC$ or $\CC\setminus 0$. The isomorphism
is given by $(\pi, \wtuithreadfunction{a_1,1}, \wtuithreadfunction{a_1,2})$ (this is exactly the restriction of the 
isomorphism from Lemma \ref{uistructure} to the subset $U'\subseteq U_{a_1}$).

In this case, $L'=\CC$ if and only if $\uidegree{a_1,1}=\ldots=\uidegree{a_k,1}\in\partial\sck$.

Here the set of ordinary points in $V'$ is the set of ordinary points in $V_{a_1}\cap\ldots\cap V_{a_k}$.
If $p\in\PP^1$ is a special point, then $p\in V'$ if and only if $p\in V_{a_1}\cap\ldots\cap V_{a_k}$ and
all degrees $\uidegree{a_1,1}, \ldots, \uidegree{a_k,1}, \uidegree{a_1,2},\ldots, \uidegree{a_k,2}$ belong to the normal subcone 
of the same vertex of $\stdpolyhedronletter_p$.
\end{lemma}
\begin{proof}
Consider a fiber $\pi^{-1}(p)\cap U'$, where $p\in V_{a_1}$. It is a subset of $\pi^{-1}(p)\cap U_{a_1}$, which
is isomorphic to $(\CC\setminus 0)\times L$ by Lemma \ref{uistructure}.
It is sufficient to prove that for each $p\in V_{a_1}$, in terms of this isomorphism, $\pi^{-1}(p)\cap U'$ either is 
the empty set, or equals $(\CC\setminus 0)\times L'\subseteq (\CC\setminus 0)\times L$.

First, let $p\in V_{a_1}$ be an ordinary point. 
If there exists an index $i$ such that $p\notin V_{a_i}$, then $\pi^{-1}(p)\cap U'=\varnothing$.
Otherwise,
consider a point $x\in \pi^{-1}(p)\cap U_{a_1}$. There are two possibilities: either $\wtuithreadfunction{a_1,2}(x)\ne 0$ 
(in other words, the last coordinate of $x$ in terms of the isomorphism $U_i\cong V_i\times (\CC\setminus 0)\times L$
from Lemma \ref{uistructure} is nonzero), or $\uidegree{a_1,1}\in\partial\sck$ and $\wtuithreadfunction{a_1,2}(x)=0$ (in other 
words, the last coordinate of $x$ in terms of the isomorphism from Lemma \ref{uistructure} is zero).
If the first possibility takes place, then, by Proposition \ref{genfiberstructinvar}, $x\in U_{a_i}$ for all $i$.
If the second possibility takes place, then it follows from Proposition \ref{genfiberstructinvar} that 
$x\in U_{a_i}$ if and only if $\uidegree{a_i,1}\in\partial\sck$ (i.~e. we have no condition for $\wtuithreadfunction{a_i,2}(x)$, 
which is in fact zero since $\uidegree{a_i,2}$ is in the interior of $\sck$) and $\uidegree{a_i,1}=\uidegree{a_1,1}$ (otherwise 
$\wtuithreadfunction{a_i,1}(x)=0$). This finishes the proof for an ordinary point.

Now let $p\in V_{a_i}$ be a special point. 
Again, if there exists an index $i$ such that $p\notin V_{a_i}$, 
then $\pi^{-1}(p)\cap U'=\varnothing$. Moreover, by Proposition \ref{specfiberstructinvar}, 
if there exist no vertex $\indexedvertexpt{p}{j}$ such that $\uidegree{a_i,1}\in \normalvertexcone{\indexedvertexpt{p}{j}}{\stdpolyhedronletter_p}$ for all $i$, then 
$\pi^{-1}(p)\cap U'=\varnothing$ (recall that we require that $\uidegree{a_i,1}$ is in the interior 
of the normal cone 
of a vertex of $\stdpolyhedronletter_p$,
unless $\uidegree{a_i,1}\in\partial\sck$, in the definition of $V_{a_i}$, 
so $\uidegree{a_i,1}$ cannot be in 
the normal cones of two different vertices
simultaneously).
And again, if $p\in V_{a_i}$ for all $i$ and there exists 
a vertex $\indexedvertexpt{p}{j}$
such that 
$\uidegree{a_i,1}\in\normalvertexcone{\indexedvertexpt{p}{j}}{\stdpolyhedronletter_p}$ for all $i$ (by the definition of $V_{a_i}$, this implies that 
$\uidegree{a_i,2}$ is in the interior of $\normalvertexcone{\indexedvertexpt{p}{j}}{\stdpolyhedronletter_p}$ for all $i$), then there are two possibilities.
Either $\wtuithreadfunction{a_1,2}(x)\ne 0$, (i.~e. the last coordinate of $x$ is nonzero), 
or $\uidegree{a_1,1}\in\partial\sck$ and $\wtuithreadfunction{a_1,2}(x)=0$, (i.~e. the last 
coordinate of $x$ is zero). The rest of the proof repeats the proof for an ordinary point. 
Namely, if the first possibility holds, it follows from Proposition \ref{specfiberstructinvar} 
that $x\in U_{a_i}$ for all $i$. If the second possibility holds, then, by 
Proposition \ref{specfiberstructinvar}, $x\in U_{a_i}$ if and only if $\uidegree{a_i,1}\in\partial\sck$ 
(i.~e. we have no condition for $\wtuithreadfunction{a_i,2}(x)$, while $\wtuithreadfunction{a_i,2}(x)=0$
since $\uidegree{a_i,2}$ is in the interior of $\normalvertexcone{\indexedvertexpt{p}{j}}{\stdpolyhedronletter_p}$) 
and $\uidegree{a_i,1}=\uidegree{a_1,1}$ (this is a criterion 
for $\wtuithreadfunction{a_i,1}(x)\ne 0$, nevertheless, this condition can only be violated if 
$\sck=\normalvertexcone{\indexedvertexpt{p}{j}}{\stdpolyhedronletter_p}$, i.~e. 
$p$ is a removable special point).
\end{proof}

\section{Computation of $T^1(X)_0$ in terms of cohomology of sheaves on $\PP^1$}

We know that $\codim_X(X\setminus U)\ge 2$, so Theorem \ref{schlessgen} can be applied.
To apply it, we need a set of generators of $\CC[X]$. We choose it as follows. 
For each special point $p$, the cone $\sck$ can be split into the union of 
normal cones of all vertices of $\stdpolyhedronletter_p$.
All intersections of these cones (for different special points)
split $\sck$ into a fan, which we call the \textit{total normal fan} of $\mathcal D$.
(It equals the normal fan of the Minkowski sum of all polyhedra $\stdpolyhedronletter_p$.)
For each cone 
$\tau$ in this fan, 
the function $\mathcal D(\cdot)|_{(\tau}\colon\tau\to \Div(\PP^1)$ is linear.
Choose a set of degrees $\lambda_1,\ldots,\lambda_{\numberoffixedgenerators}\in\sck\cap M$
satisfying the following conditions:
\begin{enumerate}
\item It contains the Hilbert bases of all cones of the total normal fan of $\mathcal D$.
\item For each special point $p$:
\begin{enumerate}
\item 
For each (finite or infinite) edge $\indexededgept{p}{j}$, 
$\primitivelattice{\normalvertexcone{\indexededgept{p}{j}}{\stdpolyhedronletter_p}}\in\{\lambda_1,\ldots,\lambda_{\numberoflatticegenerators}\}$.
\item 
For each vertex $\indexedvertexpt{p}{j}$ there exists a degree 
$\chi\in\{\lambda_1,\ldots,\lambda_{\numberoflatticegenerators}\}\cap\normalvertexcone{\indexedvertexpt{p}{j}}{\stdpolyhedronletter_p}$
such that $\chi$ and 
$\primitivelattice{\normalvertexcone{\indexededgept{p}{j-1}}{\stdpolyhedronletter_p}}$
form a lattice basis of $M$ 
\item 
For each vertex $\indexedvertexpt{p}{j}$ there exists a degree 
$\chi\in\{\lambda_1,\ldots,\lambda_{\numberoflatticegenerators}\}\cap\normalvertexcone{\indexedvertexpt{p}{j}}{\stdpolyhedronletter_p}$
such that $\chi$ and 
$\primitivelattice{\normalvertexcone{\indexededgept{p}{j}}{\stdpolyhedronletter_p}}$
form a lattice basis of $M$.
\end{enumerate}
\end{enumerate}
In fact, 
the first condition implies all three parts of the second one,
but we don't need this fact and we will not prove it.
For each $i$, $1\le i\le \numberoflatticegenerators$ 
let $\dependentgeneratorsdegree{\lambda_i, 1}, \ldots, \dependentgeneratorsdegree{\lambda_i, \dim \Gamma(\PP^1,\OO(\mathcal D(\lambda_i)))}$ be a basis of 
$\Gamma(\PP^1,\OO(\mathcal D(\lambda_i)))$.

\begin{lemma}\label{generateccx}
All $\wtdependentgeneratorsdegree{\lambda_i,j}$ (for $1\le i\le \numberoflatticegenerators$, 
$1\le j\le \dim\Gamma(\PP^1,\OO(\mathcal D(\lambda_i)))$) together generate $\CC[X]$.
\end{lemma}
\begin{proof}
It is sufficient to prove that every homogeneous element of $\CC[X]$ can be generated by $\dependentgeneratorsdegree{\lambda_i,j}$. 
So, fix a degree $\chi\in\sck\cap M$, and let $f\in\Gamma(\PP^1,\OO(\mathcal D(\chi)))$. If $\chi\in\{\lambda_1,\ldots,\lambda_{\numberoflatticegenerators}\}$, 
the claim is clear. Otherwise, choose a cone 
$\tau$
from the 
total normal fan
so that 
$\chi\in\tau$.
$\chi$ is not an element of the Hilbert basis of 
$\tau$,
so there exist 
$\chi',\chi''\in\tau\cap M$, 
$\chi'\ne 0$, $\chi''\ne 0$, such that $\chi'+\chi''=\chi$. Since $\mathcal D(\cdot)\colon \sck\to\Div(\PP^1)$
becomes a linear function after being restricted to 
$\tau$,
$\mathcal D(\chi)=\mathcal D(\chi')+\mathcal D(\chi'')$.


Let $r_1$ be the number of points $p\in \PP^1$ that are either special or are zeros of $\overline f$. Denote zeros of 
$\overline f$ that are ordinary points by $p_{\numberofdivisorpoints+1},\ldots, p_{r_1}$ (recall that we have $\numberofdivisorpoints$ special points $p_1,\ldots,p_{\numberofdivisorpoints}$).
Consider the following $r_1$ integers:
$a_i=\mathcal D_{p_i}(\chi)+\ord_{p_i}(\overline f)$. By the definition of $\Gamma(\PP^1,\OO(\mathcal D(\chi)))$, all these numbers are nonnegative integers.
Also, $a_1+\ldots+a_{r_1}=\mathcal D_{p_1}(\chi)+\ldots+\mathcal D_{p_{r_1}}(\chi)+\ord_{p_1}(\overline f)+\ldots+\ord_{p_{r_1}}(\overline f)=
\deg \mathcal D(\chi)+\deg\div(\overline f)=\deg\mathcal D(\chi)$. Then it is possible to split each of these numbers into 
a sum $a_i=a'_i+a''_i$ of two nonnegative integers so that $a'_1+\ldots+a'_{r_1}=\deg\mathcal D(\chi')$
and $a''_1+\ldots+a''_{r_1}=\deg\mathcal D(\chi'')$ (recall that $\mathcal D(\chi)=\mathcal D(\chi')+\mathcal D(\chi'')$).
Then $D_1=(a'_1-\mathcal D_{p_1}(\chi'))p_1+\ldots+(a'_{r_1}-\mathcal D_{p_{r_1}}(\chi'))p_{r_1}$ and 
$D_2=(a''_1-\mathcal D_{p_1}(\chi''))p_1+\ldots+(a''_{r_1}-\mathcal D_{p_{r_1}}(\chi''))p_{r_1}$ are 
divisors of degree 0, and $D_1\ge -\mathcal D(\chi')$, $D_2\ge -\mathcal D(\chi'')$. Therefore, there exist functions 
$f'\in\Gamma(\PP^1,\OO(\mathcal D(\chi')))$ and $f''\in\Gamma(\PP^1,\OO(\mathcal D(\chi'')))$ such that $\div(\overline{f'}=D_1$ and 
$\div(\overline{f''}=D_2$. Now, for every point $p_i$ we have the following:
$\ord_{p_i}(\overline{f'}\overline{f''})=a'_i-\mathcal D_{p_i}(\chi')+a''_i-\mathcal D_{p_i}(\chi'')=
a_i-\mathcal D_{p_i}(\chi)=\ord_{p_i}(\overline f)$. Hence, $\overline{f'}\overline{f''}/\overline f$ is a 
rational function on $\PP^1$ that does not have zeros or poles, so it is a constant, and $f$ is a multiple of $f'f''$.

Repeating this procedure by induction on 
$\chi\in\tau$,
we can write $f$ as a product of functions whose 
degrees are in the set $\{\lambda_1,\ldots,\lambda_{\numberoflatticegenerators}\}$.
\end{proof}

Now we construct 
a map $\psi\colon\Theta_X\to\OO_X^{\oplus \numberoflatticegenerators}$ required for Theorem \ref{schlessgen} using these generators. 
Recall that $\psi$ maps a vector field to the sequence of the derivatives of all generators $\wtdependentgeneratorsdegree{\lambda_i,j}$ 
along this vector field.
Denote the 
total number of these generators by $\numberoffixedgenerators$. By Theorem \ref{schlessgen}, we have the following isomorphism
of $\CC[X]$-modules:
$$
T^1(X)=\ker (H^1(U,\Theta_X)\stackrel{H^1(\psi|_U)}\longrightarrow H^1(U,\OO_X^{\oplus \numberoffixedgenerators})).
$$
By Lemma \ref{uistructure}, $\{U_i\}$ form an affine covering of $U$, so
it can be used to compute homology groups in this formula as Cech homology.
Moreover, all conditions defining $U_i$ as subsets of $X$ are formulated in terms of 
fibers of $\pi$ and inequalities of the form $f\ne 0$, where $f$ is a homogeneous 
function. Since $\pi$ is 
$T$-invariant
and the inequalities of 
form $f\ne 0$ are also invariant if $f$ is homogeneous, the sets $U_i$ are 
$T$-invariant.
The sheaves involved in the formula above are the tangent 
bundle and the trivial bundle, so 
$T$ acts on the modules of their sections on $U_i$.
Hence, these modules are $M$-graded. This enables us to introduce an $M$-grading 
on $H^1(U,\Theta_X)$ and on $H^1(U,\OO_X^{\oplus \numberoffixedgenerators})$. The map $\psi$ is defined by 
$\numberoffixedgenerators$ maps $\Theta_X\to\OO_X$, each of 
them corresponds to a generator $\wtdependentgeneratorsdegree{\lambda_j,k}$ of degree $\lambda_j$.
It maps 
the graded component of $\Gamma (U_i,\Theta_X)$ of degree  $\chi\in M$
to the graded component of $\Gamma (U_i, \OO_X)$ of degree $\chi+\lambda_j$.
Hence, $H^1(\psi|_U)$ maps different graded components of 
$H^1(U, \Theta_X)$ to different graded components of $H^1(U,\OO_X^{\oplus \numberoffixedgenerators})=H^1(U,\OO_X)^{\oplus \numberoffixedgenerators}$, 
and $\ker H^1(\psi|_U)$ is a graded submodule in $H^1(U, \Theta_X)$.
It follows from the proof of Theorem \ref{schlessgen} that 
the isomorphism $T^1(X)=\ker H^1(\psi|_U)$
is an isomorphism of \textbf{graded} $\CC[X]$-modules.
We are going to study the zeroth graded component of $T^1(X)$.

Now, we apply Leray spectral sequence for the map $\pi\colon U\to \PP^1$
and get the following short exact sequences of $\CC[X]$-modules (note that 
Lemmas \ref{ordinaryfiberfull} and \ref{specialfiberdense} guarantee that $\pi(U)=\PP^1$):
$$
0\to H^1(\PP^1,(\pi|_U)_*(\Theta_X|_U))\to H^1(U,\Theta_X)\to H^0(\PP^1,R^1(\pi|_U)_*(\Theta_X|_U))\to 0
$$
and
$$
0\to H^1(\PP^1,(\pi|_U)_*(\OO_X^{\oplus \numberoffixedgenerators}|_U))\to H^1(U,\OO_X^{\oplus \numberoffixedgenerators})\to H^0(\PP^1,R^1(\pi|_U)_*(\OO_X^{\oplus \numberoffixedgenerators}|_U))\to 0.
$$
The Snake lemma yields the following exact sequence:
\begin{multline*}
0\to\ker\left(H^1(\PP^1,(\pi|_U)_*(\Theta_X|_U))\stackrel{H^1((\pi|_U)_*\psi)}\longrightarrow H^1(\PP^1,(\pi|_U)_*(\OO_X^{\oplus \numberoffixedgenerators}|_U))\right)
\to T^1(X)\to\\
\ker\left(H^0(\PP^1,R^1(\pi|_U)_*(\Theta_X|_U))\stackrel{H^0(R^1(\pi|_U)_*\psi)}\longrightarrow H^0(\PP^1,R^1(\pi|_U)_*(\OO_X^{\oplus \numberoffixedgenerators}|_U))\right)\to\\
\coker\left(H^1(\PP^1,(\pi|_U)_*(\Theta_X|_U))\stackrel{H^1((\pi|_U)_*\psi)}\longrightarrow H^1(\PP^1,(\pi|_U)_*(\OO_X^{\oplus \numberoffixedgenerators}|_U))\right).
\end{multline*}
This is an isomorphism of $\CC[X]$-modules, and it is possible to introduce an $M$-grading 
on these modules. 
Indeed, in fact the sheaves $(\pi|_U)_*(\Theta_X|_U)$ and $(\pi|_U)_*(\OO_X^{\oplus \numberoffixedgenerators}|_U)$
are graded themselves, i.~e. they are direct sums of their graded components in the category of 
sheaves of $\OO_{\PP^1}$-modules, since their sections on any open subset $V\subseteq \PP^1$ are sections of 
the tangent bundle and of rank $\numberoffixedgenerators$ trivial bundle on a 
$T$-invariant
subset $\pi^{-1}(V)$, 
and multiplication by functions from $\Gamma (V, \OO_{\PP^1})$ does not change the grading 
of a section.
This is also true for the sheaves $R^1(\pi|_U)_*(\Theta_X|_U))$ and $R^1(\pi|_U)_*(\OO_X^{\oplus \numberoffixedgenerators}|_U)$
if we compute them using Proposition \ref{computederived} with $\{U_i\}$ being the 
required affine covering of $U$ since in this case the module of sections of any sheaf in the complex 
on any open subset $V\subseteq \PP^1$ is also a direct sum of modules of sections of 
the tangent bundle or of the trivial bundle on a 
$T$-invariant
subset of $X$, 
and the differentials in the complex preserve this grading.
So, again there is an $M$-grading on cohomology groups: on $H^1(\PP^1,(\pi|_U)_*(\Theta_X|_U))$, 
on $H^0(\PP^1,R^1(\pi|_U)_*(\Theta_X|_U)))$, on $H^1(\PP^1,(\pi|_U)_*(\OO_X^{\oplus \numberoffixedgenerators}|_U))$, 
and on $H^0(\PP^1,R^1(\pi|_U)_*(\OO_X^{\oplus \numberoffixedgenerators}|_U))$.
And again, the map $(\pi|_U)_*\psi\colon (\pi|_U)_*(\Theta_X|_U)\to((\pi|_U)_*(\OO_X|_U))^{\oplus \numberoffixedgenerators}$ 
is defined by $\numberoffixedgenerators$ maps $(\pi|_U)_*(\Theta_X|_U)\to((\pi|_U)_*(\OO_X|_U))$, each of 
them corresponds to a generator $\wtdependentgeneratorsdegree{\lambda_i,j}$.
It maps 
the graded component of $(\pi|_U)_*(\Theta_X|_U)$
of degree $\chi\in M$ to graded components of $(\pi|_U)_*(\OO_X^{\oplus \numberoffixedgenerators}|_U)$ 
of degree $\chi+\lambda_i$. 
So, 
$\ker H^1((\pi|_U)_*\psi)\oplus H^0(R^1(\pi|_U)_*\psi)$ is an $M$-graded $\CC[X]$-module.
This grading coincides (in terms of the isomorphisms mentioned above) with gradings on
$T^1(X)$ and on $\ker H^1(\psi|_U)$.


Now we are going to obtain a formula for the graded component of $T^1(X)$ of degree 0.
Denote it by $T^1(X)_0$.
Denote also the graded component of $(\pi|_U)_*\Theta_X$ of degree 0 by 
$\giinv$, the graded component of $R^1(\pi|_U)_*\Theta_X$ of degree 0 by 
$\givinv$. The superscript "$\inv$" here indicates that these sheaves by definition
are just pushforwards of sheaves on $X$, they are defined "invariantly"
in contrast with the sheaves we will define later using trivializations and transition matrices.

We need graded components of $(\pi|_U)_*\OO_X$ and of 
$R^1(\pi|_U)_*\OO_X$ of different degrees, so for a degree $\chi$ 
denote by $\gvinvl{\chi}$ the graded component of $(\pi|_U)_*\OO_X$
of degree $\chi$, and denote by $\gviiiinvl{\chi}$ the graded component of 
$R^1(\pi|_U)_*\OO_X$ of degree $\chi$. The morphism $H^1((\pi|_U)_*\psi)$
maps $H^1(\PP^1, \giinv)$ to $H^1(\PP^1, \gvinv)$, where
$$
\gvinv=\bigoplus_{i=1}^\numberoflatticegenerators\bigoplus_{j=1}^{\dim \Gamma(\PP^1,\OO(\mathcal D(\lambda_i)))}\gvinvl{\lambda_i}.
$$
The morphism $H^0(R^1(\pi|_U)_*\psi)$ maps $H^0(\PP^1, \givinv)$ to
$H^0(\PP^1,\gviiiinv)$, where
$$
\gviiiinv=\bigoplus_{i=1}^\numberoflatticegenerators\bigoplus_{j=1}^{\dim \Gamma(\PP^1,\OO(\mathcal D(\lambda_i)))}\gviiiinvl{\lambda_i}.
$$
So, the above exact sequence for $T^1(X)$ can be written in the graded form as 
follows:
\begin{proposition}\label{exactseqweak}
The following sequence is exact:
%
\begin{multline*}
0\to\ker\left(H^1(\PP^1,\giinv)
\stackrel{H^1(((\pi|_U)_*\psi)|_{\giinv})}\longrightarrow 
H^1(\PP^1,\gvinv)\right)\to T^1(X)_0\to\\
\ker\left(H^0(\PP^1,\givinv)
\stackrel{H^0((R^1(\pi|_U)_*\psi)|_{\givinv})}\longrightarrow 
H^0(\PP^1,\gviiiinv)\right)\to\\
\coker\left(H^1(\PP^1,\giinv)
\stackrel{H^1(((\pi|_U)_*\psi)|_{\giinv})}\longrightarrow 
H^1(\PP^1,\gvinv)\right).
\end{multline*}\qed
\end{proposition}

Our next goal is to find expressions for the sheaves 
$\giinv$, $\givinv$, $\gvinv$, and $\gviiiinv$ including 
only functions on $\PP^1$ and the combinatorics of $\mathcal D$.
%
Given an index $i$ and a point $p\in V_i$, Proposition \ref{uistructure}
provides an isomorphism between $\pi^{-1}(p)\cap U_i$ and $(\CC\setminus 0)\times L$, 
where $L$ is $\CC\setminus 0$ or $\CC$. Call the point identified by this isomorphism 
with $(1,1)\in(\CC\setminus 0)\times L$ the \textit{canonical point in the fiber $\pi^{-1}(p)$
with respect to $U_i$}. In other words, the canonical point in $\pi^{-1}(p)$ with respect 
to $U_i$ is the (unique) point $x\in\pi^{-1}(p)\cap U_i$ such that $\wtuithreadfunction{i,1}(x)=
\wtuithreadfunction{i,2}(x)=1$. 

\subsection{Computation of $\giinv$}
For each $i$ ($1\le i\le \numberoffixedopensets$)
fix an embedding $V_i\hookrightarrow \CC$. As long as such an embedding is fixed, we identify 
each point of $p\in V_i$ with its coordinate $t_0\in\CC$. Denote the coordinates of 
a point $x\in U_i$ provided by the isomorphism $U_i\cong V_i\times (\CC\setminus 0)\times L$
by $t_0\in V_i$, $t_1\in \CC\setminus 0$, $t_2\in L$.  

We are going to study homogeneous vector fields of degree 0 (i.~e. $T$-invariant vector fields) 
on open sets 
$U_i'\subset X$ of the form $V_i'\times(\CC\setminus 0)\times L' \subseteq U_i$, 
where $V_i'\subseteq V_i$ is an open subset, $L'\subseteq L$ is $\CC$ or $(\CC\setminus 0)$, $L$ is defined in Lemma \ref{uistructure}, 
and $U_i'$ is embedded in $U_i$ as a subset of $V_i\times (\CC\setminus 0)\times L$ via isomorphism 
from Lemma \ref{uistructure}.
\begin{lemma}\label{uidescvfieldnoninv}
Let $V_i'\subseteq V_i$ be an open subset, $L'\subseteq L$ be an open subset 
that can be equal $\CC$ or $(\CC\setminus 0)$, $U_i'=V_i'\times(\CC\setminus 0)\times L'\subseteq U_i$.
A homogeneous vector field of degree 0 on $U_i'$ is uniquely determined by 
its values at canonical points in all fibers $\pi^{-1}(t_0)$ (for $t_0\in V_i'$) with 
respect to $U_i$. 
These values can be arbitrary vectors depending algebraically 
on $t_0\in V_i'$.
\end{lemma}

\begin{proof}
Let $\stdvectorfiledxletter$ be a vector field of degree 0 on $U_i'$, and suppose that 
$\stdvectorfiledxletter(t_0,1,1)=f_0(t_0)\vfield{t_0}+f_1(t_0)\vfield{t_1}+f_2(t_0)\vfield{t_2}$, where $f_j\colon V_i'\to\CC$ 
are algebraic functions. Since $M$ is the character lattice of $T$, and $\uidegree{i,1}$ and $\uidegree{i,2}$ form a basis 
of $M$, every pair $(t_1,t_2)\in(\CC\setminus 0)\times (\CC\setminus 0)$ uniquely and algebraically 
determines an element $\tau\in T$ such that $\uidegree{i,1}(\tau)=t_1$, $\uidegree{i,2}(\tau)=t_2$. This element acts on $U_i'$, 
i.~e. it defines an automorphism of $U_i'$, which we also denote by $\tau$. Recall that $t_j=\wtuithreadfunction{i,j}|_{U_i}$, 
$j=1,2$, and $\wtuithreadfunction{i,1}$ (resp. $\wtuithreadfunction{i,2}$) is a function of degree $\uidegree{i,1}$ (resp. $\uidegree{i,2}$), 
so $\tau (t_0,1,1)=(t_0,t_1,t_2)$ for every $t_0\in V_i'$. By the definition of a $T$-invariant vector field, 
$\stdvectorfiledxletter$ is a field of degree 0 if and only if $\stdvectorfiledxletter(\tau' x)=d\tau' \stdvectorfiledxletter(x)$ for every $x\in U_i'$, $\tau'\in T$.
In particular, this holds for $x=(t_0,1,1)$, $\tau'=\tau$, 
so $\stdvectorfiledxletter$ is uniquely determined on 
$V_i'\times (\CC\setminus 0)\times (\CC\setminus 0)$, which is at least an open subset in $U_i'$, 
so it is determined uniquely on $U_i'$.

We still have to check that if we start with arbitrary functions $f_0, f_1, f_2\colon V_i'\to \CC$,  
the vector field on $V_i'\times (\CC\setminus 0)\times (\CC\setminus 0)$
constructed this way can be extended to the whole $U_i'$
if and only if $f_0, f_1, f_2$ satisfy the statement of the Lemma and that the resulting 
vector field on $U_i'$ is $T$-invariant. To do this, let us first write 
the vector field we have constructed in terms of $f_j$ and $\vfield{t_j}$.
Take a point $x=(t_0,t_1,t_2)\in V_i'\times(\CC\setminus 0)\times (\CC\setminus 0)$, $t_0=\pi(x)$.
We have $\stdvectorfiledxletter(t_0,t_1,t_2)=d\tau \stdvectorfiledxletter(t_0,1,1)=d\tau (f_0(t_0)\vfield{t_0}+f_1(t_0)\vfield{t_1}
+f_2(t_0)\vfield{t_2})=(f_0(t_0)\vfield{t_0}+t_1f_1(t_0)\vfield{t_1}+t_2f_2(t_0)\vfield{t_2})$.
Clearly, functions of the form $f_j(t_0) t_1^{a_1}t_2^{a_2}$ with $a_1\ge 0$, $a_2\ge 0$ 
can be extended to the whole $U_i'$.

Observe that to check homogeneity, 
we have to check an equality of two vector fields for each $\tau\in T$. This equality holds 
if it holds on an open subset of $U_i'$, in particular, it is sufficient to check homogeneity 
of the resulting vector field on $V_i'\times (\CC\setminus 0)\times (\CC\setminus 0)$.
Take a point $x=(t_0,t_1,t_2)\in V_i'\times (\CC\setminus 0)\times (\CC\setminus 0)$ and an element 
$\tau'\in T$. Denote by $\tau\in T$ the element of $T$ such that $\uidegree{i,1}(\tau)=t_1$, $\uidegree{i,2}(\tau)=t_2$. 
We have $\stdvectorfiledxletter(\tau' x)=\stdvectorfiledxletter(\tau' \tau (t_0,1,1))=
d(\tau'\tau) \stdvectorfiledxletter(t_0,1,1)=d\tau'd\tau \stdvectorfiledxletter(t_0,1,1)=d\tau'\stdvectorfiledxletter(t_0,t_1,t_2)$, 
and the vector field is $T$-invariant. 
\end{proof}

\begin{corollary}\label{uidescvfield}
A homogeneous vector field $w$ of degree 0 on $U_i'$ is also uniquely determined by the following 
data:
\begin{enumerate}
\item The derivatives of $\wtuithreadfunction{i,j}$ ($j=1,2$) along $\stdvectorfiledxletter$ at canonical points, considered as two 
algebraic functions $V_i\to\CC$.
\item The vector field on $V_i'$ obtained by applying $d\pi$ to the values of $\stdvectorfiledxletter$ at canonical points, 
$d_{(t_0,1,1)}\pi \stdvectorfiledxletter(t_0,1,1)$.
\end{enumerate}
The vector field and two functions 
can be arbitrary algebraic.
\end{corollary}
\begin{proof}
Write $\stdvectorfiledxletter(t_0,1,1)=f_0(t_0)\vfield{t_0}+f_1(t_0)\vfield{t_1}+f_2(t_0)\vfield{t_2}$. Then 
$d_{(t_0,1,1)}\pi \stdvectorfiledxletter(t_0,1,1)=f_0(t_0)\vfield{t_0}$, $d\wtuithreadfunction{i,j}\stdvectorfiledxletter(t_0,1,1)=f_j(t_0)$ ($j=1,2$).
\end{proof}

Note that these data (the image of a vector at a canonical point under $d\pi$, the derivatives of functions 
along $\stdvectorfiledxletter$) do not depend on the choice of an embedding $V_i\to\CC$.
Given a vector field $\stdvectorfiledxletter$ of degree 0 on $U_i'$, we call the data 
from Corollary \ref{uidescvfield} the \textit{$U_i$-description of $\stdvectorfiledxletter$}.
Also, the $U_i$-description only depends on the data we used to define the set $U_i$ 
(the degrees $\uidegree{i,1}$ and $\uidegree{i,2}$ and the sections $\uithreadfunction{i,1}$ and $\uithreadfunction{i,2}$), 
not on the whole sufficient system $U_1, \ldots, U_{\numberoffixedopensets}$.

Observe also that the operation of taking the $U_i$-description is compatible
with replacing $U_i'$ by a smaller subset $U_i''$ of the same form, or, more precisely, 
we can say the following:

\begin{remark}\label{uidescvfieldrestriction}
Let $U_i''\subset U_i'$ be a subset of $U_i'$ of the same form, i.~e. 
let $V_i''\subseteq V_i'$ be an open subset, let $L''\subseteq L'$ be an open subset
that can be equal $\CC$ or $\CC\setminus0$, and let
$U_i''=V_i''\times (\CC\setminus0)\times L''$
be embedded into $\subseteq V_i'\times (\CC\setminus0)\times L'=U_i'$ via 
the embeddings $V_i''\subseteq V_i'$ and $L''\subseteq L'$ above. Let $\stdvectorfiledxletter'$ be the restriction of 
$\stdvectorfiledxletter$ to $U_i''$. Then the $U_i''$-description of $\stdvectorfiledxletter'$ consists of the restrictions from 
$V_i'$ to $V_i''$ of the vector field and two functions forming the $U_i'$-description of $\stdvectorfiledxletter$.
\end{remark}


Note that there are many possible descriptions for a given vector field 
on $X$, 
each one corresponds to one of the chosen open subsets $U_i$. Sometimes we will need many descriptions of a given 
vector field 
on $X$ simultaneously. And sometimes we will simultaneously deal with 
descriptions of many different 
vector fields. To distinguish between these situations clearly, 
we will usually use "standard" subscripts to enumerate different descriptions of the same 
vector field,
for example:
$$
(\uidescriptionfunction{1,1},\uidescriptionfunction{1,2},\stdvectorfiledpletter_1,\ldots,
\uidescriptionfunction{i,1},\uidescriptionfunction{i,2},\stdvectorfiledpletter_i,\ldots,
\uidescriptionfunction{\numberoffixedopensets,1},\uidescriptionfunction{\numberoffixedopensets,2},\stdvectorfiledpletter_{\numberoffixedopensets}).
$$
Here $(\uidescriptionfunction{i,1},\uidescriptionfunction{i,2},\stdvectorfiledpletter_i)$ is the $U_i$-description 
of a vector field that does not depend on $i$.
If we have several different vector fields and one description of each of them, we enumerate them using indices 
in brackets, for example:
$$
(\uidescriptionfunctiondiff 11,\uidescriptionfunctiondiff 12,\stdvectorfiledpletter[1],\ldots,
\uidescriptionfunctiondiff i1,\uidescriptionfunctiondiff i2,\stdvectorfiledpletter[i],\ldots,
\uidescriptionfunctiondiff{\numberofdivisorpoints}1,\uidescriptionfunctiondiff{\numberofdivisorpoints}2,\stdvectorfiledpletter[\numberofdivisorpoints]).
$$
Here $(\uidescriptionfunctiondiff i1,\uidescriptionfunctiondiff i2,\stdvectorfiledpletter[i])$ can be, for example, 
the $U_1$-description of a vector field $\stdvectorfiledxletter[i]$ on $X$, and these 
vector fields may vary independently. These are only generic rules, they are stated here to 
demonstrate what kind of notation will be used later. Every time, when we consider a description of a vector field,
we say explicitly which set $U_i$ we use, which vector field or function on $X$ we describe, 
and how we denote the description.

Later we will introduce $U_i$-descriptions of homogeneous functions on $X$ in a similar way, the only difference will be that 
the $U_i$-description of a homogeneous function consists of only one function on $V_i$, not of two functions and a vector field. 
When we have several $U_i$-descriptions of functions, we will use the same generic rules to write their indices.

Choose two indices $i$ and $j$ ($1\le i,j\le\numberoffixedopensets$).
The following lemma relates 
the $U_i$-description with the $U_j$-description of a vector field $\stdvectorfiledxletter$
of degree 0. We need some more notation to formulate it.
Denote by $\uismalltransition{i,j}$ the following $2\times 2$-matrix:
$$
\uismalltransition{i,j}=
\left(
\begin{array}{cc}
\uidegree{i,1}^*(\uidegree{j,1}) & \uidegree{i,2}^*(\uidegree{j,1})\\
\uidegree{i,1}^*(\uidegree{j,2}) & \uidegree{i,2}^*(\uidegree{j,2})
\end{array}
\right),
$$
Denote
\newlength{\cijlength}
\settoheight{\cijlength}{\ensuremath{\displaystyle{\uismalltransition{i,j}}}}
$$
\renewcommand{\arraystretch}{0}
\uilargetransition{i,j}(p)=
\left(
\begin{array}{cc|c}
&&\frac{\oluithreadfunction{i,1}(p)^{\uidegree{i,1}^*(\uidegree{j,1})}\oluithreadfunction{i,2}(p)^{\uidegree{i,2}^*(\uidegree{j,1})}}{\oluithreadfunction{j,1}(p)} 
d\left(\frac{\oluithreadfunction{j,1}(p)}{\oluithreadfunction{i,1}(p)^{\uidegree{i,1}^*(\uidegree{j,1})}\oluithreadfunction{i,2}(p)^{\uidegree{i,2}^*(\uidegree{j,1})}}\right) \\
\multicolumn{2}{c|}{\raisebox{-0.5\cijlength}[0.5pt][0.5pt]{\ensuremath{
\uismalltransition{i,j}
}}}&\\
&&\frac{\oluithreadfunction{i,1}(p)^{\uidegree{i,1}^*(\uidegree{j,2})}\oluithreadfunction{i,2}(p)^{\uidegree{i,2}^*(\uidegree{j,2})}}{\oluithreadfunction{j,2}(p)} 
d\left(\frac{\oluithreadfunction{j,2}(p)}{\oluithreadfunction{i,1}(p)^{\uidegree{i,1}^*(\uidegree{j,2})}\oluithreadfunction{i,2}(p)^{\uidegree{i,2}^*(\uidegree{j,2})}}\right) \\
\vrule width 0pt height 1pt&
\\
\hline
\mathstrut \strut 0 & 0 & 1
\end{array}
\right),
$$
where $p\in\PP^1$ is an arbitrary point, and the first and the second entry in the third column are understood as 
\textit{rational} covector fields on $\PP^1$. In particular, if $i=j$, $\uismalltransition{i,i}$ and $\uilargetransition{i,i}$ are unit matrices.
By Lemma \ref{uijstructure}, $U_i\cap U_j$ is isomorphic to $V'\times (\CC\setminus 0)\times L'$, 
where $V'$ is an open subset of $V_i\cap V_j$, and $L'$ is $\CC$ or $(\CC\setminus0)$. This product is 
embedded into $U_i$ via the isomorphism from Lemma \ref{uistructure}.
\begin{lemma}\label{vfieldtransition}
Let $V''$ be an open subset of $V'$, $L''$ be an open subset of $L'$, $L''=\CC$ or $L''=\CC\setminus 0$,
and let $U''=V''\times (\CC\setminus 0)\times L''$ be embedded into $U_i\cap U_j$ via the map from 
Lemma \ref{uijstructure}. Let $\stdvectorfiledxletter$ be a vector field on $U''$ of degree 0, and let $\uidescriptionfunction{i,1}$, $\uidescriptionfunction{i,2}$, $\stdvectorfiledpletter_i$ be the
$U_i$-description of $\stdvectorfiledxletter$, and $\uidescriptionfunction{j,1}$, $\uidescriptionfunction{j,2}$, $\stdvectorfiledpletter_j$ be the $U_j$-description of $\stdvectorfiledxletter$. Then 
for every $p\in V''$
$$
\left(
\begin{array}{c}
\uidescriptionfunction{j,1}(p)\\
\uidescriptionfunction{j,2}(p)\\
\stdvectorfiledpletter_j(p)
\end{array}
\right)
=
\uilargetransition{i,j}(p)
\left(
\begin{array}{c}
\uidescriptionfunction{i,1}(p)\\
\uidescriptionfunction{i,2}(p)\\
\stdvectorfiledpletter_i(p)
\end{array}
\right).
$$
In particular, $\stdvectorfiledpletter_i(p)=\stdvectorfiledpletter_j(p)$.
\end{lemma}
\begin{proof}
It is sufficient to check this equality on an arbitrary open subset of $V''$, so let $p\in V''$ be 
an ordinary point. Let $x$ be the canonical point in $\pi^{-1}(p)$ with respect to $U_i$. It follows
from the definition of the canonical point that $x\in U''$. Let $x'$ be the canonical point 
in $\pi^{-1}(p)$ with respect to $U_j$. By Proposition \ref{genfiberstructinvar}, $\wtuithreadfunction{i,1}(x')\ne 0$, 
$\wtuithreadfunction{i,2}(x')\ne 0$, so $x'\in U''$.

Let $\tau\in T$ be the element of $T$ such that $\uidegree{i,1}(\tau)=\wtuithreadfunction{i,1}(x')$, 
$\uidegree{i,2}(\tau)=\wtuithreadfunction{i,2}(x')$. It defines an automorphism of $U''$, and we 
also denote this automorphism by $\tau$. Then $\wtuithreadfunction{i,1}(\tau x)=\wtuithreadfunction{i,1}(x')$, 
$\wtuithreadfunction{i,2}(\tau x)=\wtuithreadfunction{i,2}(x')$, $\pi(\tau x)=p=\pi(x')$, so $\tau x=x'$.

Since $\stdvectorfiledxletter$ is a vector field of degree 0, $\stdvectorfiledxletter(x')=d_x\tau \stdvectorfiledxletter(x)$.
Since $\pi=\pi\tau$, we have $d_x\pi=d_{\tau x}\pi d_x\tau=d_{x'}\pi d_x\tau$, 
and $\stdvectorfiledpletter_j(p)=d_{x'}\pi \stdvectorfiledxletter(x')=(d_{x'}\pi)(d_x\tau \stdvectorfiledxletter(x))=d_x\pi \stdvectorfiledxletter(x)=\stdvectorfiledpletter_i(p)$.

Now we are going to compute $\uidescriptionfunction{j,1}(p)=d_{x'}\wtuithreadfunction{j,1}\stdvectorfiledxletter(x')$. 
Until the end of the proof, denote 
$a_{1,1}=\uidegree{i,1}^*(\uidegree{j,1})$,
$a_{1,2}=\uidegree{i,2}^*(\uidegree{j,1})$, 
$a_{2,1}=\uidegree{i,1}^*(\uidegree{j,2})$, and
$a_{2,2}=\uidegree{i,2}^*(\uidegree{j,2})$.
We have 
$$
\wtuithreadfunction{j,1}=\wtuithreadfunction{i,1}^{a_{1,1}}\wtuithreadfunction{i,2}^{a_{1,2}}
\frac{\wtuithreadfunction{j,1}}{\wtuithreadfunction{i,1}^{a_{1,1}}\wtuithreadfunction{i,2}^{a_{1,2}}},
$$
and
\begin{multline*}
d_{x'}\wtuithreadfunction{j,1}=
a_{1,1}(d_{x'}\wtuithreadfunction{i,1})\wtuithreadfunction{i,2}(x')^{a_{1,2}}
\frac{\wtuithreadfunction{j,1}(x')}{\wtuithreadfunction{i,1}(x')^{a_{1,1}}\wtuithreadfunction{i,2}(x')^{a_{1,2}}}+\\
a_{1,2}\wtuithreadfunction{i,1}(x')^{a_{1,1}}(d_{x'}\wtuithreadfunction{i,2})
\frac{\wtuithreadfunction{j,1}(x')}{\wtuithreadfunction{i,1}(x')^{a_{1,1}}\wtuithreadfunction{i,2}(x')^{a_{1,2}}}+
\wtuithreadfunction{i,1}(x')^{a_{1,1}}\wtuithreadfunction{i,2}(x')^{a_{1,2}}
d_{x'}\left(\frac{\wtuithreadfunction{j,1}}{\wtuithreadfunction{i,1}^{a_{1,1}}\wtuithreadfunction{i,2}^{a_{1,2}}}\right).
\end{multline*}
Taking into account that $\wtuithreadfunction{j,1}(x')=1$, we get
$$
d_{x'}\wtuithreadfunction{j,1}=
\frac{a_{1,1}d_{x'}\wtuithreadfunction{i,1}}{\wtuithreadfunction{i,1}(x')^{a_{1,1}}}+
\frac{a_{1,2}d_{x'}\wtuithreadfunction{i,2}}{\wtuithreadfunction{i,2}(x')^{a_{1,2}}}+
\frac{\wtuithreadfunction{i,1}(x')^{a_{1,1}}\wtuithreadfunction{i,2}(x')^{a_{1,2}}}{\wtuithreadfunction{j,1}(x')}
d_{x'}\left(\frac{\wtuithreadfunction{j,1}}{\wtuithreadfunction{i,1}^{a_{1,1}}\wtuithreadfunction{i,2}^{a_{1,2}}}\right).
$$

We are computing $d_{x'}\wtuithreadfunction{i,1}\stdvectorfiledxletter(x')$. We have $d_{x'}\wtuithreadfunction{i,1}\stdvectorfiledxletter(x')=
d_{x'}\wtuithreadfunction{i,1}d_x\tau \stdvectorfiledxletter(x)$. Since $\wtuithreadfunction{i,1}$ is a homogeneous 
function of degree $\uidegree{i,1}$, we have the following equality of maps $X\to\CC$: $\wtuithreadfunction{i,1}\circ\tau=
\uidegree{i,1}(\tau)\wtuithreadfunction{i,1}=\wtuithreadfunction{i,1}(x')\wtuithreadfunction{i,1}$. So, 
$d_{x'}\wtuithreadfunction{i,1}d_x\tau \stdvectorfiledxletter(x)=\wtuithreadfunction{i,1}(x')d_x\wtuithreadfunction{i,1}\stdvectorfiledxletter(x)=
\wtuithreadfunction{i,1}(x')\uidescriptionfunction{i,1}(p)$. 
Similarly, $d_{x'}\wtuithreadfunction{i,2}\stdvectorfiledxletter(x')=\wtuithreadfunction{i,2}(x')\uidescriptionfunction{i,2}(p)$.

Now we are going to deal with the last summand in the formula for $d_{x'}\wtuithreadfunction{j,1}$ above.
Since $\wtuithreadfunction{j,1}$ and $\wtuithreadfunction{i,1}^{a_{1,1}}\wtuithreadfunction{i,2}^{a_{1,2}}$
are functions of the same degree $\uidegree{j,1}$, by Proposition \ref{quotmorph} we have the following 
equalities of maps from the open subset where they are defined as regular functions, not only as rational functions, 
to $\CC$:
$$
\frac{\wtuithreadfunction{j,1}}{\wtuithreadfunction{i,1}^{a_{1,1}}\wtuithreadfunction{i,2}^{a_{1,2}}}=
\frac{\oluithreadfunction{j,1}}{\oluithreadfunction{i,1}^{a_{1,1}}\oluithreadfunction{i,2}^{a_{1,2}}}\circ\pi
\text{ and }
\frac{\wtuithreadfunction{i,1}^{a_{1,1}}\wtuithreadfunction{i,2}^{a_{1,2}}}{\wtuithreadfunction{j,1}}=
\frac{\oluithreadfunction{i,1}^{a_{1,1}}\oluithreadfunction{i,2}^{a_{1,2}}}{\oluithreadfunction{j,1}}\circ\pi.
$$
As we already know, $\wtuithreadfunction{i,1}(x')\ne 0$, $\wtuithreadfunction{i,2}(x')\ne 0$. Also, 
$\wtuithreadfunction{j,1}(x')=1$ by the definition of $x'$, so these maps are defined at $x'$,
and we get
\begin{multline*}
\frac{\wtuithreadfunction{i,1}(x')^{a_{1,1}}\wtuithreadfunction{i,2}(x')^{a_{1,2}}}{\wtuithreadfunction{j,1}(x')}
d_{x'}\left(\frac{\wtuithreadfunction{j,1}}{\wtuithreadfunction{i,1}^{a_{1,1}}\wtuithreadfunction{i,2}^{a_{1,2}}}\right)\stdvectorfiledxletter(x')=\\
\frac{\oluithreadfunction{i,1}(p)^{a_{1,1}}\oluithreadfunction{i,2}(p)^{a_{1,2}}}{\oluithreadfunction{j,1}(p)}
d_p\left(\frac{\oluithreadfunction{j,1}}{\oluithreadfunction{i,1}^{a_{1,1}}\oluithreadfunction{i,2}^{a_{1,2}}}\right)d_{x'}\pi \stdvectorfiledxletter(x')=
\frac{\oluithreadfunction{i,1}(p)^{a_{1,1}}\oluithreadfunction{i,2}(p)^{a_{1,2}}}{\oluithreadfunction{j,1}(p)}
d_p\left(\frac{\oluithreadfunction{j,1}}{\oluithreadfunction{i,1}^{a_{1,1}}\oluithreadfunction{i,2}^{a_{1,2}}}\right)\stdvectorfiledpletter_j(p)=\\
\frac{\oluithreadfunction{i,1}(p)^{a_{1,1}}\oluithreadfunction{i,2}(p)^{a_{1,2}}}{\oluithreadfunction{j,1}(p)}
d_p\left(\frac{\oluithreadfunction{j,1}}{\oluithreadfunction{i,1}^{a_{1,1}}\oluithreadfunction{i,2}^{a_{1,2}}}\right)\stdvectorfiledpletter_i(p).
\end{multline*}

Finally, we get the following formula for $\uidescriptionfunction{j,1}(p)$:
\begin{multline*}
\uidescriptionfunction{j,1}(p)=d_{x'}\wtuithreadfunction{j,1}\stdvectorfiledxletter(x')=\\
\frac{a_{1,1}d_{x'}\wtuithreadfunction{i,1}\stdvectorfiledxletter(x')}{\wtuithreadfunction{i,1}(x')^{a_{1,1}}}+
\frac{a_{1,2}d_{x'}\wtuithreadfunction{i,2}\stdvectorfiledxletter(x')}{\wtuithreadfunction{i,2}(x')^{a_{1,2}}}+
\frac{\wtuithreadfunction{i,1}(x')^{a_{1,1}}\wtuithreadfunction{i,2}(x')^{a_{1,2}}}{\wtuithreadfunction{j,1}(x')}
d_{x'}\left(\frac{\wtuithreadfunction{j,1}}{\wtuithreadfunction{i,1}^{a_{1,1}}\wtuithreadfunction{i,2}^{a_{1,2}}}\right)\stdvectorfiledxletter(x')=\\
\frac{a_{1,1}\wtuithreadfunction{i,1}(x')\uidescriptionfunction{i,1}(p)}{\wtuithreadfunction{i,1}(x')^{a_{1,1}}}+
\frac{a_{1,2}\wtuithreadfunction{i,2}(x')\uidescriptionfunction{i,2}(p)}{\wtuithreadfunction{i,2}(x')^{a_{1,2}}}+
\frac{\oluithreadfunction{i,1}(p)^{a_{1,1}}\oluithreadfunction{i,2}(p)^{a_{1,2}}}{\oluithreadfunction{j,1}(p)}
d_p\left(\frac{\oluithreadfunction{j,1}}{\oluithreadfunction{i,1}^{a_{1,1}}\oluithreadfunction{i,2}^{a_{1,2}}}\right)\stdvectorfiledpletter_i(p)=\\
a_{1,1}\uidescriptionfunction{i,1}(p)+a_{1,2}\uidescriptionfunction{i,2}(p)+
\frac{\oluithreadfunction{i,1}(p)^{a_{1,1}}\oluithreadfunction{i,2}(p)^{a_{1,2}}}{\oluithreadfunction{j,1}(p)}
d_p\left(\frac{\oluithreadfunction{j,1}}{\oluithreadfunction{i,1}^{a_{1,1}}\oluithreadfunction{i,2}^{a_{1,2}}}\right)\stdvectorfiledpletter_i(p).
\end{multline*}
Similarly,
$$
\uidescriptionfunction{j,2}(p)=
a_{2,1}\uidescriptionfunction{i,1}(p)+a_{2,2}\uidescriptionfunction{i,2}(p)+
\frac{\oluithreadfunction{i,1}(p)^{a_{2,1}}\oluithreadfunction{i,2}(p)^{a_{2,2}}}{\oluithreadfunction{j,2}(p)}
d_p\left(\frac{\oluithreadfunction{j,2}}{\oluithreadfunction{i,1}^{a_{2,1}}\oluithreadfunction{i,2}^{a_{2,2}}}\right)\stdvectorfiledpletter_i(p).
$$
\end{proof}

Now we are ready to describe the sheaf $\giinv$ only using functions on $\PP^1$ and the notion of 
a sufficient system of $U_i$ (which uses only combinatorics of $\mathcal D$ and functions on $\PP^1$).
We will prove that it is isomorphic to another sheaf (denoted by $\gi$), which will be defined
using functions and vector fields on $\PP^1$ satisfying certain conditions. This is similar to the approach using transition matrices, 
but the sheaf we will define does not have to be locally free.

Namely, consider the following sheaf $\gi$. \label{g1intropage}
Let $V\subseteq \PP^1$ be an open subset. The space of sections $\Gamma (V,\gi)$ is the space of 
sequences of length $2 \numberoffixedopensets+1$
$$
(\uidescriptionfunction{1,1}, \uidescriptionfunction{1,2}, \ldots, \uidescriptionfunction{i,1}, \uidescriptionfunction{i,2}, 
\ldots, \uidescriptionfunction{\numberoffixedopensets,1},\uidescriptionfunction{\numberoffixedopensets,2},\stdvectorfiledpletter),
$$
where $\uidescriptionfunction{i,j}\in\Gamma (V_i\cap V,\OO_{\PP^1})$, $\stdvectorfiledpletter\in\Gamma (V,\Theta_{\PP^1})$
satisfy the following condition:
For every indices $i, i'$:
$$
\left(
\begin{array}{c}
\uidescriptionfunction{i',1}(p)\\
\uidescriptionfunction{i',2}(p)\\
\stdvectorfiledpletter(p)
\end{array}
\right)
=
\uilargetransition{i,i'}(p)
\left(
\begin{array}{c}
\uidescriptionfunction{i,1}(p)\\
\uidescriptionfunction{i,2}(p)\\
\stdvectorfiledpletter(p)
\end{array}
\right)
$$
\begin{proposition}\label{vfieldpushforward}
$\giinv$ is isomorphic to $\gi$. For an open set $V\subseteq \PP^1$, the isomorphism maps a vector field $\stdvectorfiledxletter$ defined on 
$\pi^{-1}(V)\cap U$ to the sequence 
$$
(\uidescriptionfunction{1,1}, \uidescriptionfunction{1,2}, \ldots, \uidescriptionfunction{i,1}, \uidescriptionfunction{i,2}, 
\ldots, \uidescriptionfunction{\numberoffixedopensets,1},\uidescriptionfunction{\numberoffixedopensets,2},\stdvectorfiledpletter),
$$
such that $(\uidescriptionfunction{i,1}, \uidescriptionfunction{i,2}, \stdvectorfiledpletter)$ is the $U_i$-description of $\stdvectorfiledxletter$.
\end{proposition}
\begin{proof}
This is a direct consequence of Lemma \ref{vfieldtransition}, Lemma \ref{uistructure}, and the definition of a pushforward of a sheaf.
\end{proof}

The following three lemmas make it easier to construct sections of $\gi$ explicitly.
\begin{lemma}\label{rijregular}
All entries of $\uilargetransition{i,j}$ are regular at ordinary points $p$ such that $p\in V_i\cap V_j$.
\end{lemma}
\begin{proof}
For constant entries the claim is clear, and non-constant entries are logarithmic derivatives of functions
$$
\frac{\oluithreadfunction{j,1}}{\oluithreadfunction{i,1}^{\uidegree{i,1}^*(\uidegree{j,1})}\oluithreadfunction{i,2}^{\uidegree{i,2}^*(\uidegree{j,1})}}
\quad\text{and}\quad
\frac{\oluithreadfunction{j,2}}{\oluithreadfunction{i,1}^{\uidegree{i,1}^*(\uidegree{j,2})}\oluithreadfunction{i,2}^{\uidegree{i,2}^*(\uidegree{j,1})}}.
$$
If $p$ is an ordinary point and $p\in V_i\cap V_j$, then, by the definition of $V_i$ and of $V_j$,
$\ord_p\oluithreadfunction{i,1}=\ord_p\oluithreadfunction{i,2}=\ord_p\oluithreadfunction{j,1}=\ord_p\oluithreadfunction{j,2}=0$.
Hence, both functions
$$
\frac{\oluithreadfunction{j,1}}{\oluithreadfunction{i,1}^{\uidegree{i,1}^*(\uidegree{j,1})}\oluithreadfunction{i,2}^{\uidegree{i,2}^*(\uidegree{j,1})}}
\quad\text{and}\quad
\frac{\oluithreadfunction{j,2}}{\oluithreadfunction{i,1}^{\uidegree{i,1}^*(\uidegree{j,2})}\oluithreadfunction{i,2}^{\uidegree{i,2}^*(\uidegree{j,1})}}
$$
are defined at $p$ and do not vanish at $p$, so their logarithmic derivatives are regular at $p$.
\end{proof}

\begin{lemma}\label{logderpole}
Let $p$ be a special point, and let $i$ and $j$ be two indices such that 
$p\in V_i\cap V_j$, and $\uidegree{i,1}$ and $\uidegree{j,1}$ belong to 
the normal vertex cones of two different vertices of $\stdpolyhedronletter_p$.
Then each non-constant entry of $\uilargetransition{i,j}$ has pole of degree exactly 1 at $p$.
\end{lemma}

\begin{proof}
We know that each of the degrees $\uidegree{i,1}$ and $\uidegree{i,2}$ belongs to 
the normal subcone of exactly one vertex of $\stdpolyhedronletter_p$, and this vertex is the same one
for $\uidegree{i,1}$ and for $\uidegree{i,2}$. 
$\uidegree{j,1}$ belong to 
the normal subcone of a different vertex of $\stdpolyhedronletter_p$,
which is also unique.
Since $\mathcal D_p(\cdot)$ is a convex function, it cannot be linear on the union of these two 
subcones, and $\mathcal D_p(\uidegree{j,1})<\uidegree{i,1}^*(\uidegree{j,1})\mathcal D_p(\uidegree{i,1})+\uidegree{i,2}^*(\uidegree{j,1})\mathcal D_p(\uidegree{i,2})$. Therefore, 
$$
\ord_p\left(\frac{\oluithreadfunction{j,1}}{\oluithreadfunction{i,1}^{\uidegree{i,1}^*(\uidegree{j,1})}\oluithreadfunction{i,2}^{\uidegree{i,2}^*(\uidegree{j,1})}}\right)=
-\mathcal D_p(\uidegree{j,1})+\uidegree{i,1}^*(\uidegree{j,1})\mathcal D_p(\uidegree{i,1})+\uidegree{i,2}^*(\uidegree{j,1})\mathcal D_p(\uidegree{i,2})>0,
$$
and, by a property of logarithmic derivative,
$$
\ord_p\left(\frac{\oluithreadfunction{i,1}^{\uidegree{i,1}^*(\uidegree{j,1})}\oluithreadfunction{i,2}^{\uidegree{i,2}^*(\uidegree{j,1})}}{\oluithreadfunction{j,1}}
d\left(\frac{\oluithreadfunction{j,1}}{\oluithreadfunction{i,1}^{\uidegree{i,1}^*(\uidegree{j,1})}\oluithreadfunction{i,2}^{\uidegree{i,2}^*(\uidegree{j,1})}}\right)
\right)=-1.
$$
The argument for the second non-constant entry of $\uilargetransition{i,j}$ is similar.
\end{proof}

\begin{lemma}\label{rijcocycle}
For the matrices $\uismalltransition{i,j}$ and $\uilargetransition{i,j}$ defined above, one has 
$\uismalltransition{i,k}=\uismalltransition{j,k}\uismalltransition{i,j}$ and $\uilargetransition{i,k}=\uilargetransition{j,k}\uilargetransition{i,j}$ for every triple of indices $(i,j,k)$.
\end{lemma}

\begin{proof}
The equality $\uismalltransition{i,k}=\uismalltransition{j,k}\uismalltransition{i,j}$
can be proved by a direct computation using linear algebra. We omit this computation.

Now, to prove that $\uilargetransition{i,k}=\uilargetransition{j,k}\uilargetransition{i,j}$, it is sufficient to check that 
\begin{multline*}
\left(\begin{array}{c}
\frac{\oluithreadfunction{i,1}^{\uidegree{i,1}^*(\uidegree{k,1})}\oluithreadfunction{i,2}^{\uidegree{i,2}^*(\uidegree{k,1})}}{\oluithreadfunction{k,1}} 
d\left(\frac{\oluithreadfunction{k,1}}{\oluithreadfunction{i,1}^{\uidegree{i,1}^*(\uidegree{k,1})}\oluithreadfunction{i,2}^{\uidegree{i,2}^*(\uidegree{k,1})}}\right) \\
\frac{\oluithreadfunction{i,1}^{\uidegree{i,1}^*(\uidegree{k,2})}\oluithreadfunction{i,2}^{\uidegree{i,2}^*(\uidegree{k,2})}}{\oluithreadfunction{k,2}} 
d\left(\frac{\oluithreadfunction{k,2}}{\oluithreadfunction{i,1}^{\uidegree{i,1}^*(\uidegree{k,2})}\oluithreadfunction{i,2}^{\uidegree{i,2}^*(\uidegree{k,2})}}\right)
\end{array}\right)=\\
\uismalltransition{j,k}
\left(\begin{array}{c}
\frac{\oluithreadfunction{i,1}^{\uidegree{i,1}^*(\uidegree{j,1})}\oluithreadfunction{i,2}^{\uidegree{i,2}^*(\uidegree{j,1})}}{\oluithreadfunction{j,1}} 
d\left(\frac{\oluithreadfunction{j,1}}{\oluithreadfunction{i,1}^{\uidegree{i,1}^*(\uidegree{j,1})}\oluithreadfunction{i,2}^{\uidegree{i,2}^*(\uidegree{j,1})}}\right) \\
\frac{\oluithreadfunction{i,1}^{\uidegree{i,1}^*(\uidegree{j,2})}\oluithreadfunction{i,2}^{\uidegree{i,2}^*(\uidegree{j,2})}}{\oluithreadfunction{j,2}} 
d\left(\frac{\oluithreadfunction{j,2}}{\oluithreadfunction{i,1}^{\uidegree{i,1}^*(\uidegree{j,2})}\oluithreadfunction{i,2}^{\uidegree{i,2}^*(\uidegree{j,2})}}\right)
\end{array}\right)\\
+
\left(\begin{array}{c}
\frac{\oluithreadfunction{j,1}^{\uidegree{j,1}^*(\uidegree{k,1})}\oluithreadfunction{j,2}^{\uidegree{j,2}^*(\uidegree{k,1})}}{\oluithreadfunction{k,1}} 
d\left(\frac{\oluithreadfunction{k,1}}{\oluithreadfunction{j,1}^{\uidegree{j,1}^*(\uidegree{k,1})}\oluithreadfunction{j,2}^{\uidegree{j,2}^*(\uidegree{k,1})}}\right) \\
\frac{\oluithreadfunction{j,1}^{\uidegree{j,1}^*(\uidegree{k,2})}\oluithreadfunction{j,2}^{\uidegree{j,2}^*(\uidegree{k,2})}}{\oluithreadfunction{k,2}} 
d\left(\frac{\oluithreadfunction{k,2}}{\oluithreadfunction{j,1}^{\uidegree{j,1}^*(\uidegree{k,2})}\oluithreadfunction{j,2}^{\uidegree{j,2}^*(\uidegree{k,2})}}\right)
\end{array}\right).
\end{multline*}
By a property of logarithmic derivatives, if $f_1,f_2$ are (rational) functions, 
$$
\frac{d(f_1^{a_1}f_2^{a_2})}{f_1^{a_1}f_2^{a_2}}=
a_1\frac{df_1}{f_1}+a_2\frac{df_2}{f_2}.
$$
Hence, the left-hand side of the equality we are proving can be written as
$$
\left(\begin{array}{c}
\frac{d\oluithreadfunction{k,1}}{\oluithreadfunction{k,1}}
-\uidegree{i,1}^*(\uidegree{k,1})\frac{d\oluithreadfunction{i,1}}{\oluithreadfunction{i,1}}
-\uidegree{i,2}^*(\uidegree{k,1})\frac{d\oluithreadfunction{i,2}}{\oluithreadfunction{i,2}} \\
\frac{d\oluithreadfunction{k,2}}{\oluithreadfunction{k,2}}
-\uidegree{i,1}^*(\uidegree{k,2})\frac{d\oluithreadfunction{i,1}}{\oluithreadfunction{i,1}}
-\uidegree{i,2}^*(\uidegree{k,2})\frac{d\oluithreadfunction{i,2}}{\oluithreadfunction{i,2}}
\end{array}\right)=
\left(\begin{array}{c}
\frac{d\oluithreadfunction{k,1}}{\oluithreadfunction{k,1}}\\
\frac{d\oluithreadfunction{k,2}}{\oluithreadfunction{k,2}}
\end{array}\right)
-\uismalltransition{i,k}
\left(\begin{array}{c}
\frac{d\oluithreadfunction{i,1}}{\oluithreadfunction{i,1}}\\
\frac{d\oluithreadfunction{i,2}}{\oluithreadfunction{i,2}}
\end{array}\right).
$$
Similarly, the right-hand side can be written as
\begin{multline*}
\uismalltransition{j,k}\Big(\left(\begin{array}{c}
\frac{d\oluithreadfunction{j,1}}{\oluithreadfunction{j,1}}\\
\frac{d\oluithreadfunction{j,2}}{\oluithreadfunction{j,2}}
\end{array}\right)
-\uismalltransition{i,j}
\left(\begin{array}{c}
\frac{d\oluithreadfunction{i,1}}{\oluithreadfunction{i,1}}\\
\frac{d\oluithreadfunction{i,2}}{\oluithreadfunction{i,2}}
\end{array}\right)
\Big)+
\left(\begin{array}{c}
\frac{d\oluithreadfunction{k,1}}{\oluithreadfunction{k,1}}\\
\frac{d\oluithreadfunction{k,2}}{\oluithreadfunction{k,2}}
\end{array}\right)
-\uismalltransition{j,k}
\left(\begin{array}{c}
\frac{d\oluithreadfunction{j,1}}{\oluithreadfunction{j,1}}\\
\frac{d\oluithreadfunction{j,2}}{\oluithreadfunction{j,2}}
\end{array}\right)=\\
\left(\begin{array}{c}
\frac{d\oluithreadfunction{k,1}}{\oluithreadfunction{k,1}}\\
\frac{d\oluithreadfunction{k,2}}{\oluithreadfunction{k,2}}
\end{array}\right)
-\uismalltransition{j,k}\uismalltransition{i,j}
\left(\begin{array}{c}
\frac{d\oluithreadfunction{i,1}}{\oluithreadfunction{i,1}}\\
\frac{d\oluithreadfunction{i,2}}{\oluithreadfunction{i,2}}
\end{array}\right).
\end{multline*}
By taking into account that $\uismalltransition{i,k}=\uismalltransition{j,k}\uismalltransition{i,j}$, we obtain the desired equality.
\end{proof}

\subsection{Computation of $\givinv$}
Recall that we have denoted the graded component of $R^1(\pi|_U)_*\Theta_X$ of degree 0 by 
$\givinv$.
Now we are going to compute $\givinv$ using Proposition \ref{computederived}. We can use $\{U_i\}$ as an affine 
covering of $U$. We have to consider a complex of sheaves on $U$ that we temporarily denote by $\mathcal F_\bullet$.
For an open subset $U'\subseteq U$, $\Gamma (U',\mathcal F_0)$ consists of sequences $(\stdvectorfiledxletter_1,\ldots, \stdvectorfiledxletter_{\numberoffixedopensets})$, where $\stdvectorfiledxletter_i$ is a 
vector field on $U_i\cap U'$, $\Gamma (U',\mathcal F_1)$ consists of sequences $(\stdvectorfiledxletter_{i,j})_{1\le i< j\le \numberoffixedopensets}$, 
where $\stdvectorfiledxletter_{i,j}$ is a vector field on $U_i\cap U_j\cap U'$, and $\Gamma (U',\mathcal F_2)$
consists of sequences $(\stdvectorfiledxletter_{i,j,k})_{1\le i<j<k\le \numberoffixedopensets}$, where $\stdvectorfiledxletter_{i,j,k}$ is a vector field on $U_i\cap U_j\cap U_k\cap U'$.
Denote the graded components of degree 0 of 
the pushforwards of these sheaves by $\giiinv$, $\giipinv$, $\giippinv$, respectively.
Using 
Corollary \ref{uidescvfield}
we get the following description of these sheaves:

Consider the following sheaves $\gii$, $\giip$, and $\giipp$. For an open subset $V\subseteq \PP^1$, $\Gamma (V,\gii)$ consists of sequences 
$$
(\uidescriptionfunctiondiff 11, \uidescriptionfunctiondiff 12,\stdvectorfiledpletter[1], \ldots, 
\uidescriptionfunctiondiff i1, \uidescriptionfunctiondiff i2, \stdvectorfiledpletter[i], \ldots, 
\uidescriptionfunctiondiff{\numberoffixedopensets}{1},\uidescriptionfunctiondiff{\numberoffixedopensets}{2},\stdvectorfiledpletter[\numberoffixedopensets]),
$$
where $\uidescriptionfunctiondiff{i}{j}\in\Gamma (V_i\cap V,\OO_{\PP^1})$, $\stdvectorfiledpletter[i]\in\Gamma (V_i\cap V,\Theta_{\PP^1})$.
Then $\gii$ is isomorphic to $\giiinv$, and the isomorphism maps a sequence of $\numberoffixedopensets$ vector fields 
$(\stdvectorfiledxletter[1],\ldots, \stdvectorfiledxletter[\numberoffixedopensets])$
to the sequence
$$
(\uidescriptionfunctiondiff 11, \uidescriptionfunctiondiff 12,\stdvectorfiledpletter[1], \ldots, 
\uidescriptionfunctiondiff i1, \uidescriptionfunctiondiff i2, \stdvectorfiledpletter[i], \ldots, 
\uidescriptionfunctiondiff{\numberoffixedopensets}{1},\uidescriptionfunctiondiff{\numberoffixedopensets}{2},\stdvectorfiledpletter[\numberoffixedopensets]),
$$
where $\uidescriptionfunctiondiff i1, \uidescriptionfunctiondiff i2, \stdvectorfiledpletter[i]$ form the $U_i$-description of $\stdvectorfiledxletter[i]$.

$\Gamma (V,\giip)$ consists of sequences $(\uidescriptionfunctiondiff{i,j}{1}, \uidescriptionfunctiondiff{i,j}{2}, \stdvectorfiledpletter[i,j])_{1\le i< j\le \numberoffixedopensets}$,
where $\uidescriptionfunctiondiff{i,j}{1}, \uidescriptionfunctiondiff{i,j}{2} \in\Gamma (V_i\cap V_j\cap V,\OO_{\PP^1})$, $\stdvectorfiledpletter[i,j]\in\Gamma (V_i\cap V_j\cap V,\Theta_{\PP^1})$.
Similarly, $\giip$ is isomorphic to $\giipinv$, and the isomorphism maps a sequence 
$(\stdvectorfiledxletter[i,j])_{1\le i< j\le \numberoffixedopensets}$
of vector fields on open subsets of $U\cap \pi^{-1}(V)$ to the sequence 
$(\uidescriptionfunctiondiff{i,j}{1}, \uidescriptionfunctiondiff{i,j}{2}, \stdvectorfiledpletter[i,j])_{1\le i< j\le \numberoffixedopensets}$,
where
$\uidescriptionfunctiondiff{i,j}{1}$, $\uidescriptionfunctiondiff{i,j}{2}$ and 
$\stdvectorfiledpletter[i,j]$ form the $U_i$-description of a vector field defined on $U_i\cap U_j\cap \pi^{-1}(V)$.
(In fact, at this point we can choose arbitrarily whether this is the $U_i$-description or the $U_j$-description of 
$\stdvectorfiledxletter[i,j]$, and we choose that this is the $U_i$-description, and \textbf{not} the $U_j$-description.)

Finally, $\Gamma (V,\giipp)$ consists of sequences 
$(\uidescriptionfunctiondiff{i,j,k}{1}, \uidescriptionfunctiondiff{i,j,k}{2}, \stdvectorfiledpletter[i,j,k])_{1\le i< j\le \numberoffixedopensets}$,
where 
$$
\uidescriptionfunctiondiff{i,j,k}{1},\uidescriptionfunctiondiff{i,j,k}{2}\in\Gamma (V_i\cap V_j\cap V_k\cap V,\OO_{\PP^1}),\quad
\stdvectorfiledpletter[i,j,k]\in\Gamma (V_i\cap V_j\cap V_k\cap V,\Theta_{\PP^1}).
$$
The isomorphism between $\giippinv$ and $\giipp$ is constructed similarly, and here we again say (we choose) that 
$\uidescriptionfunctiondiff{i,j,k}{1},\uidescriptionfunctiondiff{i,j,k}{2}, \stdvectorfiledpletter[i,j,k]$ is the $U_i$-description of a vector field on 
$U_i\cap U_j\cap U_k\cap \pi^{-1}(V)$, not its $U_j$- or $U_k$-description.

Let us compute the kernel $\ker(\giip\to\giipp)$. Denote it by $\giii$. A kernel 
of a sheaf map can be computed on each open subset independently, and the map here comes from the standard 
Cech map $\mathcal F_1\to \mathcal F_2$
via the pushforward and the isomorphisms $\giipinv\cong\giip$ and $\giippinv\cong \giipp$ 
defined above. Summarizing these definitions (and choices between $U_i$-descriptions made there), we get the following 
formula for the map $\giip\to\giipp$,
where we have to calculate a $U_i$-description from a $U_j$-description once:
$$
\left(
\begin{array}{c}
\uidescriptionfunctiondiff{i,j,k}{1}(p)\\
\uidescriptionfunctiondiff{i,j,k}{2}(p)\\
\stdvectorfiledpletter[i,j,k](p)
\end{array}
\right)
=
\left(
\begin{array}{c}
\uidescriptionfunctiondiff{i,j}{1}(p)\\
\uidescriptionfunctiondiff{i,j}{2}(p)\\
\stdvectorfiledpletter[i,j](p)
\end{array}
\right)
+\uilargetransition{j,i}(p)
\left(
\begin{array}{c}
\uidescriptionfunctiondiff{j,k}{1}(p)\\
\uidescriptionfunctiondiff{j,k}{2}(p)\\
\stdvectorfiledpletter[j,k](p)
\end{array}
\right)
-\left(
\begin{array}{c}
\uidescriptionfunctiondiff{i,k}{1}(p)\\
\uidescriptionfunctiondiff{i,k}{2}(p)\\
\stdvectorfiledpletter[i,k](p)
\end{array}
\right).
$$
So we get the following description for $\giii$. The space of sections of $\giii$ over an 
open subset $V\subseteq \PP^1$ is the space of sequences of length $3 \numberoffixedopensets(\numberoffixedopensets-1)/2$ of the form
$(\uidescriptionfunctiondiff{i,j}{1}, \uidescriptionfunctiondiff{i,j}{2}, \stdvectorfiledpletter[i,j])_{1\le i<j\le \numberoffixedopensets}$, where 
$\uidescriptionfunctiondiff{i,j}{k}\in\Gamma (V\cap V_i\cap V_j,\OO_{\PP^1})$ and 
$\stdvectorfiledpletter[i,j]\in\Gamma (V\cap V_i\cap V_j,\Theta_{\PP^1})$
satisfy the following condition:
For every indices $i<j<k$:
$$
\left(
\begin{array}{c}
\uidescriptionfunctiondiff{i,j}{1}(p)\\
\uidescriptionfunctiondiff{i,j}{2}(p)\\
\stdvectorfiledpletter[i,j](p)
\end{array}
\right)
+\uilargetransition{j,i}(p)
\left(
\begin{array}{c}
\uidescriptionfunctiondiff{j,k}{1}(p)\\
\uidescriptionfunctiondiff{j,k}{2}(p)\\
\stdvectorfiledpletter[j,k](p)
\end{array}
\right)
-\left(
\begin{array}{c}
\uidescriptionfunctiondiff{i,k}{1}(p)\\
\uidescriptionfunctiondiff{i,k}{2}(p)\\
\stdvectorfiledpletter[i,k](p)
\end{array}
\right)=0.
$$
Finally, by Proposition \ref{computederived}, $\givinv$ is isomorphic to\label{g4intropage} 
$\giv=\coker(\gii\to\giii)$, where the map $\gii\to\giii$ can be written as follows:
$$
\left(
\begin{array}{c}
\uidescriptionfunctiondiff{i,j}{1}(p)\\
\uidescriptionfunctiondiff{i,j}{2}(p)\\
\stdvectorfiledpletter[i,j](p)
\end{array}
\right)
=
\left(
\begin{array}{c}
\uidescriptionfunctiondiff{i}{1}(p)\\
\uidescriptionfunctiondiff{i}{2}(p)\\
\stdvectorfiledpletter[i](p)
\end{array}
\right)
-
\uilargetransition{j,i}(p)
\left(
\begin{array}{c}
\uidescriptionfunctiondiff{j}{1}(p)\\
\uidescriptionfunctiondiff{j}{2}(p)\\
\stdvectorfiledpletter[j](p)
\end{array}
\right).
$$

\subsection{Computation of $\gvinv$}
The sheaves $\gvinvl{\chi}$ can be computed similarly to $\giinv$. 
We start with the following Lemma.
\begin{lemma}\label{uidescfunct}
Let $V_i'\subseteq V_i$ be an open subset, $L'\subseteq L$ be an open subset 
that can be equal $\CC$ or $(\CC\setminus 0)$, $U_i'=V_i'\times(\CC\setminus 0)\times L'\subseteq U_i$.
A homogeneous function of degree $\chi\in M$
on $U_i'$ 
is uniquely determined by 
its values at canonical points in all fibers $\pi^{-1}(t_0)$ (for $t_0\in V_i'$) with 
respect to $U_i$. 
\begin{enumerate}
\item If $L'=\CC\setminus 0$ or $\uidegree{i,2}^*(\chi) \ge 0$, these values can form an arbitrary function depending algebraically 
on $p\in V_i'$.
\item If $L'=\CC$ and $\uidegree{i,2}^*(\chi) < 0$, these values must vanish. This is only possible if $\chi\notin\sck$.
\end{enumerate}
\end{lemma}
\begin{proof}
The proof is similar to the proof of Lemma \ref{uidescvfieldnoninv}.
Denote the coordinates of 
a point $x\in U_i$ provided by the isomorphism $U_i\cong V_i\times (\CC\setminus 0)\times L$
by $t_0\in V_i$, $t_1\in \CC\setminus 0$, $t_2\in L$.
Let $f$ be a function of degree $\chi$ on $U_i'$, and suppose that $f(t_0,1,1)=f_0(t_0)$, 
where $f_0\colon V_i\to\CC$ is an algebraic function. Fix a pair $(t_1,t_2)\in(\CC\setminus 0)\times (\CC\setminus 0)$
and let $\tau\in T$ be the element of $T$ such that $\uidegree{i,1}(\tau)=t_1$, $\uidegree{i,2}(\tau)=t_2$. 
Denote by $\tau$ the automorphism of $U_i'$ provided by $\tau$ as well. By the definition 
of a homogeneous function of degree $\chi$, $f(t_0,t_1,t_2)=f(\tau\cdot(t_0,1,1))=\chi(\tau)f(t_0,1,1)=
\chi(\tau)f_0(t_0)$, so $f_0$ determines $f$ uniquely on $V_i'\times (\CC\setminus 0)\times (\CC\setminus 0)$, 
which is at least an open subset in $U_i'$.

We still have to check that if we start with an arbitrary functions $f_0\colon V_i'\to \CC$,  
the resulting function on $V_i'\times (\CC\setminus 0)\times (\CC\setminus 0)$
can be extended to the whole $U_i'$
if and only if $\uidegree{i,2}^*(\chi) <0$ or $L'=\CC\setminus0$ (in the last case there is nothing to extend) and that the resulting 
function on $U_i'$ is homogeneous of degree $\chi$. The function we have constructed 
can be written as follows: $f(t_0,t_1,t_2)=\chi(\tau)f_0(t_0)=\uidegree{i,1}(\tau)^{\uidegree{i,1}^*(\chi)} \uidegree{i,2}(\tau)^{\uidegree{i,2}^*(\chi)}=
t_1^{\uidegree{i,1}^*(\chi)} t_2^{\uidegree{i,2}^*(\chi)} f_0(t_0)$. Recall that $t_1$ (resp. $t_2$) is a function on $X$ of degree 
$\uidegree{i,1}$ (resp. $\uidegree{i,2}$), so this function is clearly homogeneous of degree $\uidegree{i,1}^*(\chi) \uidegree{i,1}+\uidegree{i,2}^*(\chi) \uidegree{i,2}=\chi$ 
on $V_i'\times (\CC\setminus 0)\times (\CC\setminus 0)$. If the function can be extended to the whole 
$U_i'$, it remains homogeneous there since homogeneity means an equality of two functions for 
each element of $T$, and this equality holds if it holds on an open subset.

If $L'=\CC\setminus 0$, there is nothing to extend. If $L'=\CC$, $f$ can be extended to $U_i'$ if
and only if $\uidegree{i,2}^*(\chi) \ge 0$.

Finally, $L'=\CC$, then $\uidegree{i,1}\in\partial\sck$, and if $\uidegree{i,2}^*(\chi) <0$ in this case, then $\chi\notin\sck$.
\end{proof}

Given a homogeneous function $f$ of degree $\chi\in M$ defined on a set $U_i'$ as described in Lemma \ref{uidescfunct},
we call the function $f_0\colon V_i'\to\CC$ such that $f_0(p)=f(x)$, where $x$ is the canonical point in 
$\pi^{-1}(p)$ with respect to $U_i$ the \textit{$U_i$-description of $f$}. 
Again, the $U_i$-description of a function only depends on the data we used to define the set $U_i$ 
(the degrees $\uidegree{i,1}$ and $\uidegree{i,2}$ and the sections $\uithreadfunction{i,1}$ and $\uithreadfunction{i,2}$), 
not on the whole sufficient system $U_1, \ldots, U_{\numberoffixedopensets}$.
And again we can make a remark similar to 
Remark \ref{uidescvfieldrestriction}:

\begin{remark}\label{uidescfunctrestriction}
Let $V_i''\subseteq V_i'$ and $L''\subseteq L'$ be open subset, and $L''=\CC$ or $L''=\CC\setminus0$.
These embeddings give rise to an embedding of $U_i''=V_i''\times (\CC\setminus0)\times L''$ into 
$V_i'\times (\CC\setminus0)\times L'=U_i'$. Let $f'$ be the restriction of $\chi$ to $U_i''$.
Then the $U_i$-description of $f'$ is the restriction of the $U_i$-description of $f$ to $V_i''$.
\end{remark}

Now we are going to relate the $U_i$-description of a homogeneous function of degree $\chi$ defined on 
an open subset of $U_i\cap U_j$ with its $U_j$-description. To formulate this relation, 
we need to introduce some notation. 
Denote the following rational function of $p\in\PP^1$:
$$
\mu_{i,j,\chi}(p)=
\frac{\oluithreadfunction{i,1}(p)^{\uidegree{i,1}^*(\chi)}\oluithreadfunction{i,2}(p)^{\uidegree{i,2}^*(\chi)}}
{\oluithreadfunction{j,1}(p)^{\uidegree{i,1}^*(\chi)}\oluithreadfunction{j,2}(p)^{\uidegree{i,2}^*(\chi)}}.
$$
In particular, if $i=j$, then $\mu_{i,j,\chi}=1$.
This time it is a trivial observation that these functions satisfy conditions similar to 
Lemma \ref{rijcocycle} for matrices $\uismalltransition{i,j}$ and $\uilargetransition{i,j}$:
\begin{remark}\label{mijcocycle}
For every three indices $i,j,k$ one has $\mu_{i,k,\chi}=\mu_{i,j,\chi}\mu_{j,k,\chi}$.
\end{remark}

By Lemma \ref{uijstructure}, $U_i\cap U_j$ can be written as $V'\times (\CC\setminus 0)\times L'$, 
where $V'\subseteq V_i\cap V_j$ is an open subset, and $L'$ equals $\CC$ or $(\CC\setminus0)$. This product is 
embedded into $U_i$ via the isomorphism from Lemma \ref{uistructure}.
\begin{lemma}\label{functtransition}
Let $V''$ be an open subset of $V'$, $L''$ be an open subset of $L'$, $L''=\CC$ or $L''=\CC\setminus 0$,
and let $U''=V''\times (\CC\setminus 0)\times L''$ be embedded into $U_i\cap U_j$ via the map from 
Lemma \ref{uijstructure}.

Let $f$ be a homogeneous function on $V''$ of degree $\chi$, and let $\uidescriptionfunction i$ (resp. $\uidescriptionfunction j$) be the 
$U_i$-description (resp. $U_j$-description) of $f$. Then for every $p\in V''$:
$$
\uidescriptionfunction j(p)=\mu_{i,j,\chi}\uidescriptionfunction i(p).
$$
\end{lemma}
\begin{proof}
As in the proof of Lemma \ref{vfieldtransition}, it is sufficient to prove the equality for all ordinary points $p\in V''$.
So let $p\in V''$ be an ordinary point and let $x$ (resp. $x'$) be the canonical point in $\pi^{-1}(p)$ with respect 
to $U_i$ (resp. to $U_j$). It follows from Proposition \ref{genfiberstructinvar} that $\wtuithreadfunction{i,1}(x')\ne 0$, 
$\wtuithreadfunction{i,2}(x')\ne 0$, hence $x'\in U''$.

Let $\tau$ be the element of $T$ such that $\uidegree{i,1}(\tau)=\wtuithreadfunction{i,1}(x')$, 
$\uidegree{i,2}(\tau)=\wtuithreadfunction{i,2}(x')$. As usual, denote the corresponding automorphism of $U''$ by $\tau$ as well.
Since $\wtuithreadfunction{i,1}$ (resp. $\wtuithreadfunction{i,2}$) is a homogeneous function of degree $\uidegree{i,1}$ (resp. $\uidegree{i,2}$), 
$\wtuithreadfunction{i,1}(\tau x)=\wtuithreadfunction{i,1}(x')$, $\wtuithreadfunction{i,2}(\tau x)=\wtuithreadfunction{i,2}(x')$, 
so $\tau x=x'$.

Since $f$ is a homogeneous function of degree $\chi$,
$$
f(x')=f(\tau x)=\chi(\tau) f(x)=
\uidegree{i,1}(\tau)^{\uidegree{i,1}^*(\chi)} \uidegree{i,2}(\tau)^{\uidegree{i,2}^*(\chi)} f(x)=
\wtuithreadfunction{i,1}(x')^{\uidegree{i,1}^*(\chi)} \wtuithreadfunction{i,2}(x')^{\uidegree{i,2}^*(\chi)} f(x).
$$
Recall that $\wtuithreadfunction{j,1}(x')=\wtuithreadfunction{j,2}(x')=1$. We have
$$
f(x')=\frac{\wtuithreadfunction{i,1}(x')^{\uidegree{i,1}^*(\chi)} \wtuithreadfunction{i,2}(x')^{\uidegree{i,2}^*(\chi)}}
{\wtuithreadfunction{j,1}(x')^{\uidegree{j,1}^*(\chi)} \wtuithreadfunction{j,2}(x')^{\uidegree{i,2}^*(\chi)}} f(x).
$$
Since the numerator and the denominator of this fraction are homogeneous functions of degree 
$\uidegree{i,1}^*(\chi)\uidegree{i,1}+\uidegree{i,2}^*(\chi)\uidegree{i,2}=
\uidegree{j,1}^*(\chi)\uidegree{j,1}+\uidegree{j,2}^*(\chi)\uidegree{j,2}=\chi$, by Proposition \ref{quotmorph},
$$
f(x')=\frac{\oluithreadfunction{i,1}(\pi(x'))^{\uidegree{i,1}^*(\chi)} \oluithreadfunction{i,2}(\pi(x'))^{\uidegree{i,2}^*(\chi)}}
{\oluithreadfunction{j,1}(\pi(x'))^{\uidegree{j,1}^*(\chi)} \oluithreadfunction{j,2}(\pi(x'))^{\uidegree{j,2}^*(\chi)}} f(x)=
\mu_{i,j,\chi}(p)\uidescriptionfunction j(p).
$$
\end{proof}

Recall that for a degree $\chi\in M$ 
we have denoted by $\gvinvl{\chi}$ the graded component of $(\pi|_U)_*\OO_X$
of degree $\chi$.
Lemma \ref{functtransition} enables us to formulate a description of $\gvinvl{\chi}$ similar to the description of 
$\giinv$ above. Namely, define a sheaf $\gvl{\chi}$ as follows:
Let $V\subseteq \PP^1$ be an open subset. The space of sections $\Gamma(V,\gvl{\chi})$ is the space of 
sequences $(\uidescriptionfunction1,\ldots, \uidescriptionfunction{\numberoffixedopensets})$
of functions on $V$ satisfying the following conditions:
\begin{enumerate}
\item $\uidescriptionfunction{i'}=\mu_{i,i',\chi}\uidescriptionfunction i$ for all indices $i, i'$.
\item \label{functpushforwardmust0} If $\uidegree{i,1}\in\partial\sck$ and 
$\uidegree{i,2}^*(\chi)<0$,
then $\uidescriptionfunction i=0$.
\end{enumerate}

\begin{lemma}\label{functpushforward}
$\gvinvl{\chi}$ is isomorphic to $\gvl{\chi}$. If $f$ is a function on $\pi^{-1}(V)\cap U$ of degree 
$\chi$, then the isomorphism maps it to $(\uidescriptionfunction1,\ldots, \uidescriptionfunction{\numberoffixedopensets})$, 
where $\uidescriptionfunction i$ is the $U_i$-description of $f$.\qed
\end{lemma}

The following lemma gives an alternative description of 
$\gvl{\chi}$ if $\chi\in\sck\cap M$.
\begin{lemma}\label{pushforwardisdlambda}
If $\chi\in\sck\cap M$, then $\gvl{\chi}\cong\OO(\mathcal D(\chi))$. 
The isomorphism $\OO(\mathcal D(\chi))\leftrightarrow\gvl{\chi}$ is given by
$$
f\leftrightarrow\left(\frac f{\oluithreadfunction{1,1}^{\uidegree{1,1}^*(\chi)}\oluithreadfunction{1,2}^{\uidegree{1,2}^*(\chi)}},\ldots, 
\frac f{\oluithreadfunction{i,1}^{\uidegree{i,1}^*(\chi)}\oluithreadfunction{i,2}^{\uidegree{i,1}^*(\chi)}},\ldots,
\frac f
{\oluithreadfunction{\numberoffixedopensets,1}^{\uidegree{\numberoffixedopensets,1}^*(\chi)}
\oluithreadfunction{\numberoffixedopensets,2}^{\uidegree{\numberoffixedopensets,2}^*(\chi)}}\right),
$$
where
$f\in\Gamma(V,\OO(\mathcal D(\chi)))$, $V\subseteq\PP^1$ is an open subset.
\end{lemma}
\begin{proof}
First, let $f\in\Gamma(V,\OO(\mathcal D(\chi)))$ be a function. Then it is clear that 
$\uidescriptionfunction i=f/(\oluithreadfunction{i,1}^{\uidegree{i,1}^*(\chi)}\oluithreadfunction{i,2}^{\uidegree{i,}^*(\chi)})$ satisfy the conditions 
$\uidescriptionfunction j=\mu_{i,j,\chi}\uidescriptionfunction i$ from Lemma \ref{functpushforward} by construction. The condition 
\ref{functpushforwardmust0} from 
the definition of $\gvl{\chi}$
is void since $\chi\in\sck$.
We have to check that $\uidescriptionfunction i$ are well-defined at points $p\in V\cap V_i$. If $p\in V\cap V_i$,
then by Lemma \ref{convexity},
$\mathcal D_p(\chi)\le \uidegree{i,1}^*(\chi)\mathcal D_p(\uidegree{i,1})+\uidegree{i,2}^*(\chi) \mathcal D_p(\uidegree{i,2})$.
By the definition of $V_i$, $\ord_p(\oluithreadfunction{i,1})=-\mathcal D_p(\uidegree{i,1})$, 
$\ord_p(\oluithreadfunction{i,2})=-\mathcal D_p(\uidegree{i,2})$. Since $f\in\Gamma(V,\OO(\mathcal D(\chi)))$, 
$\ord_p(f)\ge -\mathcal D_p(\chi)$, so $\ord_p (f)\ge \ord_p(\oluithreadfunction{i,1}^{\uidegree{i,1}^*(\chi)}\oluithreadfunction{i,2}^{\uidegree{i,2}^*(\chi)})$,
and $\uidescriptionfunction i$ is well-defined on $V\cap V_i$. Therefore, $(\uidescriptionfunction1,\ldots, \uidescriptionfunction{\numberoffixedopensets})$ defines an element 
of $\gvl{\chi}$.

Now, let $(\uidescriptionfunction1,\ldots, \uidescriptionfunction{\numberoffixedopensets})\in\Gamma(V,\gvl{\chi})$.
The condition 
$\uidescriptionfunction j=\mu_{i,j,\chi}\uidescriptionfunction i$ guarantees that 
$f=\uidescriptionfunction i\oluithreadfunction{i,1}^{\uidegree{i,1}^*(\chi)}\oluithreadfunction{i,2}^{\uidegree{i,2}^*(\chi)}$
does not depend on $i$ as a rational function. We have to check that 
$f\in\Gamma(V,\OO(\mathcal D(\chi)))$. Let $p\in V$ be an ordinary point. By the definition
of a sufficient system, there exists an index $i$ such that $p\in V_i$. Then $\uidescriptionfunction i$ is 
well-defined as $p$, and $\oluithreadfunction{i,1}$ and $\oluithreadfunction{i,2}$ are defined at $p$ 
since $p$ is an ordinary point.

Now suppose that $p\in V$ is a special point. 
Let $\indexedvertexpt{p}{j}$ be a vertex such that 
$\chi\in\normalvertexcone{\indexedvertexpt{p}{j}}{\stdpolyhedronletter_p}$
By the definition of 
a sufficient system, there exists an index $i$ such that $p\in V_i$ and 
$\uidegree{i,1},\uidegree{i,2}\in\normalvertexcone{\indexedvertexpt{p}{j}}{\stdpolyhedronletter_p}$. The function $\mathcal D_p(\cdot)$ is linear on 
$\normalvertexcone{\indexedvertexpt{p}{j}}{\stdpolyhedronletter_p}$
so $\mathcal D_p(\chi)=\uidegree{i,1}^*(\chi)\mathcal D_p(\uidegree{i,1})+\uidegree{i,2}^*(\chi)\mathcal D_p(\uidegree{i,2})$.
Then $\ord_p(f)=\ord_p(\uidescriptionfunction i)+\uidegree{i,1}^*(\chi)\ord_p(\oluithreadfunction{i,1})+\uidegree{i,2}^*(\chi)\ord_p(\oluithreadfunction{i,2})
\ge \uidegree{i,1}^*(\chi)\ord_p(\oluithreadfunction{i,1})+\uidegree{i,2}^*(\chi)\ord_p(\oluithreadfunction{i,2})=
-\uidegree{i,1}^*(\chi)\mathcal D_p(\uidegree{i,1})-\uidegree{i,2}^*(\chi)\mathcal D_p(\uidegree{i,2})=-\mathcal D_p(\chi)$. 
Therefore, $f\in\Gamma(V,\OO(\mathcal D(\chi)))$.
\end{proof}

\begin{corollary}\label{h1g5zero}
$H^1(\PP^1, \gv)=0$.
\end{corollary}
\begin{proof}
Recall that 
$$
\gv=\bigoplus_{i=1}^\numberoflatticegenerators\bigoplus_{j=1}^{\dim \Gamma(\PP^1,\OO(\mathcal D(\lambda_i)))}\gvl{\lambda_i},
$$
where $\lambda_i$ form the Hilbert basis of $\sck\cap M$, in particular, $\lambda_i\in\sck\cap M$. 
Therefore, 
$$
\gv=\bigoplus_{i=1}^\numberoflatticegenerators\bigoplus_{j=1}^{\dim \Gamma(\PP^1,\OO(\mathcal D(\lambda_i)))}\OO(\mathcal D(\lambda_i)).
$$
In particular, $\mathcal D(\lambda_i)$ are divisors of non-negative degree on $\PP^1$, and 
$H^1(\PP^1, \gv)=0$. 
\end{proof}

\subsection{Computation of $\gviiiinv$}
We can compute $\gviiiinvl{\chi}$ using Proposition \ref{computederived} with $\{U_i\}$ being the required 
affine covering of $U$.
Recall that for each $\chi\in M$, $\gviiiinvl{\chi}$ is the graded component of 
$R^1(\pi|_U)_*\OO_X$ of degree $\chi$.
Again denote temporarily the complex of sheaves on $U$ we have to consider in Proposition \ref{computederived}
by $\mathcal F_\bullet$. Let $U'$ be an open subset of $U$. Then $\Gamma(U',\mathcal F_0)$ consists of sequences $(f_1,\ldots, f_{\numberoffixedopensets})$,
where $f_i\in\Gamma(U_i\cap U',\OO_X)$, $\Gamma(U',\mathcal F_1)$ consists of sequences $(f_{i,j})_{1\le i<j\le \numberoffixedopensets}$, 
where $f_{i,j}\in\Gamma(U_i\cap U_j\cap U',\OO_X)$, and $\Gamma(U',\mathcal F_2)$ consists of sequences $(f_{i,j,k})_{1\le i<j<k\le \numberoffixedopensets}$, 
where $f_{i,j,k}\in\Gamma(U_i\cap U_j\cap U_k\cap U',\OO_X)$.
Denote the graded components of degree $\chi$ of 
the pushforwards of these sheaves by $\gviinvl{\chi}$, $\gvipinvl{\chi}$, $\gvippinvl{\chi}$, respectively.
Denote also 
$$
\gviinv=\bigoplus_{i=1}^\numberoflatticegenerators\bigoplus_{j=1}^{\dim \Gamma(\PP^1,\OO(\mathcal D(\lambda_i)))}\gviinvl{\lambda_i},
$$
$$
\newlength{\andlength}\settowidth{\andlength}{\text{\qquad and}}\hspace{\andlength}
\gvipinv=\bigoplus_{i=1}^\numberoflatticegenerators\bigoplus_{j=1}^{\dim \Gamma(\PP^1,\OO(\mathcal D(\lambda_i)))}\gvipinvl{\lambda_i},
\text{\qquad and}
$$
$$
\gvippinv=\bigoplus_{i=1}^\numberoflatticegenerators\bigoplus_{j=1}^{\dim \Gamma(\PP^1,\OO(\mathcal D(\lambda_i)))}\gvippinvl{\lambda_i}.
$$
We get the following descriptions of these sheaves from 
Lemma \ref{uidescfunct}:

Define sheaves $\gvil{\chi}$, $\gvipl{\chi}$, $\gvippl{\chi}$ as follows. 
Fix an an open subset $V\subseteq \PP^1$. Let $\Gamma(V,\gvil{\chi})$ be the space of sequences of the form 
$(\uidescriptionfunctiondiff 1{},\ldots, \uidescriptionfunctiondiff{\numberoffixedopensets}{})$, 
where $\uidescriptionfunctiondiff i{}\in\Gamma(V\cap V_i,\OO_{\PP^1})$ and $\uidescriptionfunctiondiff i{}=0$ if $\uidegree{i,1}\in\partial\sck$ and 
$\uidegree{i,2}^*(\chi)<0$. 
Then $\gviinvl{\chi}\cong\gvil{\chi}$, and the isomorphism maps a sequence 
$(f[1],\ldots, f[\numberoffixedopensets])$ of functions of degree $\chi$ 
defined on open subsets of $\pi^{-1}(V)\cap U$
to 
$(\uidescriptionfunctiondiff1{},\ldots, \uidescriptionfunctiondiff{\numberoffixedopensets}{})$, 
where $\uidescriptionfunctiondiff i{}$ is the $U_i$-description of $f[i]$.

Let $\Gamma(V,\gvipl{\chi})$ be the space of 
sequences $(\uidescriptionfunctiondiff{i,j}{})_{1\le i<j\le \numberoffixedopensets}$, where $\uidescriptionfunctiondiff{i,j}{}\in\Gamma(\OO_{\PP^1}, V\cap V_i\cap V_j)$.
These functions should be zero in some cases if $\uidegree{i,1}=\uidegree{j,1}\in\partial\sck$ (see Lemma \ref{uijstructure}).
To define these cases, note first that 
if $\uidegree{i,1}=\uidegree{j,1}$, then 
$\uidegree{i,2}^*=\uidegree{j,2}^*$.
So, the condition is: 
If
$\uidegree{i,2}^*(\chi)<0$,
then $\uidescriptionfunctiondiff{i,j}{}=0$.
Again, $\gvipinvl{\chi}\cong\gvipl{\chi}$, and 
the isomorphism maps a sequence $(f[i,j])_{1\le i<j\le \numberoffixedopensets}$ of functions of 
degree $\chi$
defined on open subsets of $\pi^{-1}(V)\cap U$ to the sequence 
$(\uidescriptionfunctiondiff{i,j}{})_{1\le i<j\le \numberoffixedopensets}$
of functions on $V$ 
such that $\uidescriptionfunctiondiff{i,j}{}$ is the $U_i$-description 
of 
$f[i,j]$.
(Again, we could choose the $U_j$-description here, as well, but we choose the $U_i$-description.)

Finally, let $\Gamma(V,\gvippl{\chi})$ be the space of sequences $(\uidescriptionfunctiondiff{i,j,k}{})_{1\le i<j<k\le \numberoffixedopensets}$, where 
$\uidescriptionfunctiondiff{i,j,k}{}\in\Gamma(V\cap V_i\cap V_j\cap V_k,\OO_{\PP^1})$ and, as in the previous case, $\uidescriptionfunctiondiff{i,j,k}{}=0$ 
if $\uidegree{i,1}=\uidegree{j,1}=\uidegree{k,1}\in\partial\sck$ and 
$\uidegree{i,2}^*(\chi)<0$.
Then $\gvippinvl{\chi}\cong\gvippl{\chi}$, the isomorphism is constructed similarly, and again we say that
$\uidescriptionfunctiondiff{i,j,k}{}$ is the $U_i$-description 
of a function defined on $U_i\cap U_j\cap U_k\cap \pi^{-1}(V)$, not its 
$U_j$- or $U_k$-description.

Denote $\gviil{\chi}=\ker (\gvipl{\chi}\to\gvippl{\chi})$, where the map $\gvipl{\chi}\to\gvippl{\chi}$
comes from the standard Cech map $\mathcal F_1\to \mathcal F_2$ via the pushforward, then the restriction to the degree $\chi$, and 
then the isomorphisms $\gvipinvl{\chi}\cong\gvipl{\chi}$ and $\gvippinvl{\chi}\cong\gvippl{\chi}$
defined above.
To compute a kernel of 
a map between sheaves, it is sufficient to compute the kernels of the corresponding maps between modules on each 
open subset. So let $V\subseteq\PP^1$ be an open subset. Taking into account the choice of $U_i$-description in 
the definition of the isomorphisms $\gvipinvl{\chi}\cong\gvipl{\chi}$
and $\gvippinvl{\chi}\cong\gvippl{\chi}$, we see that the corresponding map $\Gamma(V,\gvipl{\chi})\to
\Gamma(V,\gvippl{\chi})$ can be written as follows:
$$
\uidescriptionfunctiondiff{i,j,k}{}=\uidescriptionfunctiondiff{i,j}{}+\mu_{j,i,\chi}\uidescriptionfunctiondiff{j,k}{}-\uidescriptionfunctiondiff{i,k}{},
$$
and $\Gamma(V,\gviil{\chi})$ is the space of sequences of the form $(\uidescriptionfunctiondiff{i,j}{})_{1\le i<j\le \numberoffixedopensets}$, 
where $\uidescriptionfunctiondiff{i,j}{}\in\Gamma(V\cap V_i\cap V_j,\OO_{\PP^1})$ satisfy the following conditions:
\begin{enumerate}
\item $\uidescriptionfunctiondiff{i,j}{}+\mu_{j,i,\chi}\uidescriptionfunctiondiff{j,k}{}-\uidescriptionfunctiondiff{i,k}{}=0$
for all indices $i<j<k$.
\item If $\uidegree{i,1}=\uidegree{j,1}\in\partial\sck$ and 
$\uidegree{i,2}^*(\chi)<0$
then $\uidescriptionfunctiondiff{i,j}{}=0$.
\end{enumerate}
Now, by Proposition \ref{computederived}, $\gviiiinvl{\chi}$ 
is isomorphic to
$\gviiil{\chi}=\coker(\gvil{\chi}\to\gviil{\chi})$, 
where the map $\gvil{\chi}\to\gviil{\chi}$ can be written as follows:
$\uidescriptionfunctiondiff{i,j}{}(p)=\uidescriptionfunctiondiff i{}(p)-\mu_{j,i,\chi}\uidescriptionfunctiondiff j{}(p)$.
After we have defined the sheafs $\gv{\chi}$ and $\gviii{\chi}$ isomorphic to 
$\gvinv{\chi}$ and $\gviiiinv{\chi}$ (respectively) for each degree $\chi$, 
we define\label{g8intropage}\label{g5intropage}
$$
\gv=\bigoplus_{i=1}^\numberoflatticegenerators\bigoplus_{j=1}^{\dim \Gamma(\PP^1,\OO(\mathcal D(\lambda_i)))}\gvl{\lambda_i}
\text{\qquad and\qquad}
\gviii=\bigoplus_{i=1}^\numberoflatticegenerators\bigoplus_{j=1}^{\dim \Gamma(\PP^1,\OO(\mathcal D(\lambda_i)))}\gviiil{\lambda_i}.
$$

We can also shortly write
$$
\gvi=\bigoplus_{i=1}^\numberoflatticegenerators\bigoplus_{j=1}^{\dim \Gamma(\PP^1,\OO(\mathcal D(\lambda_i)))}\gvil{\lambda_i},
$$
$$
\settowidth{\andlength}{\text{\qquad and}}\hspace{\andlength}
\gvip=\bigoplus_{i=1}^\numberoflatticegenerators\bigoplus_{j=1}^{\dim \Gamma(\PP^1,\OO(\mathcal D(\lambda_i)))}\gvipl{\lambda_i},
\text{\qquad and}
$$
$$
\gvipp=\bigoplus_{i=1}^\numberoflatticegenerators\bigoplus_{j=1}^{\dim \Gamma(\PP^1,\OO(\mathcal D(\lambda_i)))}\gvippl{\lambda_i}.
$$
Then $\gviii$ is the cohomology in the middle of the complex $\gvi\to\gvip\to \gvipp$.

\subsection{Final remarks for the computation of $T^1(X)_0$}
Proposition \ref{exactseqweak} involves (in particular) the map
$H^0((R^1(\pi|_U)_*\psi)|_{\givinv})\colon H^0(\PP^1,\givinv)\to H^0(\PP^1,\gviiiinv)$.
The isomorphisms 
$\givinv\cong\giv$
and $\gviiiinv\cong\gviii$
constructed above enable us to consider 
a map
$H^0(\PP^1,\giv)\to H^0(\PP^1,\gviii)$
instead.
Denote 
it by $H^0((R^1(\pi|_U)_*\psi)|_{\givinv})^\circ$, 
The following lemma establishes relations between $U_i$-descriptions of sections of $\Theta_X$ and their images under $\psi$, so it
will help us to understand 
this map.
\begin{lemma}\label{uidescpsi}
Let $V'$ be an open subset of $V_i$, $L'$ be an open subset of $L$, $L'=\CC$ or $L'=\CC\setminus 0$,
and let $U'=V'\times (\CC\setminus 0)\times L'$ be embedded into $U_i$ via the map from 
Lemma \ref{uistructure}.

Let $(\uidescriptionfunction{i,1},\uidescriptionfunction{i,2},\stdvectorfiledpletter_i)$ be 
the $U_i$-description of a vector field $\stdvectorfiledxletter$ defined on $V'$, 
$\chi\in\sck\cap M$ be a degree, 
$f\in\Gamma(\PP^1,\OO(\mathcal D(\chi)))$.
Then the $U_i$-description of $(d\widetilde f)\stdvectorfiledxletter$ is 
$$
\frac{\overline f}{\oluithreadfunction{i,1}^{\uidegree{i,1}^*(\chi)}\oluithreadfunction{i,2}^{\uidegree{i,2}^*(\chi)}}
(\uidegree{i,1}^*(\chi)\uidescriptionfunction{i,1}+\uidegree{i,2}^*(\chi)\uidescriptionfunction{i,2})
+d_p\frac{\overline f}{\oluithreadfunction{i,1}^{\uidegree{i,1}^*(\chi)}\oluithreadfunction{i,2}^{\uidegree{i,1}^*(\chi)}} \stdvectorfiledpletter_i.
$$
\end{lemma}

\begin{proof}
The proof is similar to the proof of Lemma \ref{vfieldtransition}. 
It is sufficient to prove the equality 
for an arbitrary open subset of $V'$, so let $p\in V'$ be an arbitrary point, and let $x$ be the canonical
point in $\pi^{-1}(p)$ with respect to $U_i$. Denote 
$a_1=\uidegree{i,1}^*(\chi)$, $a_2=\uidegree{i,2}^*(\chi)$,
and denote by $h$ the $U_i$-description of the function 
$(df)\stdvectorfiledxletter$. Then $h(p)=(d_x\widetilde f) \stdvectorfiledxletter(x)$. We have
\begin{multline*}
d_x\widetilde f=d_x\left(\wtuithreadfunction{i,1}^{a_1}\wtuithreadfunction{i,2}^{a_2}
\frac{\widetilde f}{\wtuithreadfunction{i,1}^{a_1}\wtuithreadfunction{i,2}^{a_2}}\right)=\\
a_1 (d_x\wtuithreadfunction{i,1})\wtuithreadfunction{i,2}^{a_2}(x)\frac{\widetilde f(x)}{\wtuithreadfunction{i,1}^{a_1}(x)\wtuithreadfunction{i,2}^{a_2}(x)}
+a_2\wtuithreadfunction{i,1}^{a_1}(x)(d_x\wtuithreadfunction{i,2})\frac{\widetilde f(x)}{\wtuithreadfunction{i,1}^{a_1}(x)\wtuithreadfunction{i,2}^{a_2}(x)}\\
+\wtuithreadfunction{i,1}^{a_1}(x)\wtuithreadfunction{i,2}^{a_2}(x)d_x\left(\frac{\widetilde f}{\wtuithreadfunction{i,1}^{a_1}\wtuithreadfunction{i,2}^{a_2}}\right).
\end{multline*}
Since $\wtuithreadfunction{i,1}(x)=\wtuithreadfunction{i,2}(x)=1$,
$$
d_x\widetilde f=\frac{\widetilde f(x)}{\wtuithreadfunction{i,1}^{a_1}(x)\wtuithreadfunction{i,2}^{a_2}(x)}
(a_1 d_x\wtuithreadfunction{i,1}+a_2 d_x\wtuithreadfunction{i,2})
+d_x\left(\frac{\widetilde f}{\wtuithreadfunction{i,1}^{a_1}\wtuithreadfunction{i,2}^{a_2}}\right).
$$
$\widetilde f$ and $\wtuithreadfunction{i,1}^{a_1}\wtuithreadfunction{i,2}^{a_2}$ are homogeneous 
functions of degree $\chi$, so by Proposition 1 we have the following equality of rational maps 
from $X$ to $\CC$:
$$
\frac{\widetilde f}{\wtuithreadfunction{i,1}^{a_1}\wtuithreadfunction{i,2}^{a_2}}=
\frac{\overline f}{\oluithreadfunction{i,1}^{a_1}\oluithreadfunction{i,2}^{a_2}}\circ \pi.
$$
Therefore,
$$
d_x\widetilde f=\frac{\overline f(p)}{\oluithreadfunction{i,1}^{a_1}(p)\oluithreadfunction{i,2}^{a_2}(p)}
(a_1 d_x\wtuithreadfunction{i,1}+a_2 d_x\wtuithreadfunction{i,2})+
d_p\frac{\overline f}{\oluithreadfunction{i,1}^{a_1}\oluithreadfunction{i,2}^{a_2}}d_x\pi.
$$
Finally, we get 
\begin{multline*}
h(p)=(d_x\widetilde f) \stdvectorfiledxletter(x)=\frac{\overline f(p)}{\oluithreadfunction{i,1}^{a_1}(p)\oluithreadfunction{i,2}^{a_2}(p)}
(a_1 d_x\wtuithreadfunction{i,1}\stdvectorfiledxletter(x)+a_2 d_x\wtuithreadfunction{i,2}\stdvectorfiledxletter(x))
+d_p\frac{\overline f}{\oluithreadfunction{i,1}^{a_1}\oluithreadfunction{i,2}^{a_2}}d_x\pi \stdvectorfiledxletter(x)=\\
\frac{\overline f(p)}{\oluithreadfunction{i,1}^{a_1}(p)\oluithreadfunction{i,2}^{a_2}(p)}
(a_1 \uidescriptionfunction{i,1}(p)+a_2 \uidescriptionfunction{i,2}(p))
+d_p\frac{\overline f}{\oluithreadfunction{i,1}^{a_1}\oluithreadfunction{i,2}^{a_2}} \stdvectorfiledpletter_i (p).
\end{multline*}
\end{proof}

Summarizing, we have found an explicit description for the sheaves $\gi$, $\gii$, $\giii$, 
$\gvl{\chi}$, $\gvil{\chi}$, and $\gviil{\chi}$ and for the map $\psi$, 
and all sheaves involved in 
Proposition \ref{exactseqweak} can be obtained from these sheaves by taking a cokernel of a map we have explicitly 
described and 
forming a direct sum. 

By Corollary \ref{h1g5zero},
$$
\coker(H^1(\PP^1,\gi)
\stackrel{H^1(((\pi|_U)_*\psi)|_{\gi})^\circ}\longrightarrow 
H^1(\PP^1,\gv))=0,
$$
$$
\ker(H^1(\PP^1,\gi)
\stackrel{H^1(((\pi|_U)_*\psi)|_{\gi})^\circ}\longrightarrow 
H^1(\PP^1,\gv))=H^1(\PP^1,\gi),
$$
and the exact sequence from Proposition \ref{exactseqweak} can be written in the following form:
\begin{theorem}\label{thmsheavesdownstairs}
Let $\gi$, $\giv$, and $\gviii$ be the sheaves on $\PP^1$ introduced above, on pages \pageref{g1intropage}, \pageref{g4intropage}, 
and \pageref{g8intropage} (respectively).
Then the following sequence is exact:
$$
0\to H^1(\PP^1,\gi)
\to T^1(X)_0\to
H^0(\PP^1,\giv)
\stackrel{H^0((R^1(\pi|_U)_*\psi)|_{\giv})^\circ}\longrightarrow 
H^0(\PP^1,\gviii).
$$\qed
\end{theorem}

Let us prove one more lemma about functions defined on $U_i$. We will need it later.

\begin{lemma}\label{uiregularity}
Let $\chi\in\sck\cap M$ be a degree, and let $f_1\in\Gamma(\PP^1,\OO(\mathcal D(\chi)))$.
Let $f_2$ be a rational function on $\PP^1$. The rational function $f_3(x)=f_2(\pi(x))\widetilde{f_1}(x)$
on $X$ is regular on $U_i$ if and only if:
\begin{enumerate}
\item A rational function $f_2 \overline{f_1}$ is defined at all regular points of $V_i$.
\item For each special point $p\in V_i$, $\ord_p(f_2\overline{f_1})\ge -\uidegree{i,1}^*(\chi)\mathcal D_p(\uidegree{i,1})
-\uidegree{i,2}^*(\chi)\mathcal D_p(\uidegree{i,2})$
\end{enumerate}
\end{lemma}
\begin{proof}
Let us suppose that conditions 1 and 2 are satisfied and prove that 
$f_3$
is regular on $U_i$.

First, note that
$$
\uidescriptionfunction{}=\frac{f_2\overline{f_1}}{\oluithreadfunction{i,1}^{\uidegree{i,1}^*(\chi)}\oluithreadfunction{i,2}^{\uidegree{i,2}^*(\chi)}}
$$
is a regular function on $V_i$. Indeed, if $p\in V_i$ is an ordinary point, 
then $f_2\overline{f_1}$ has no pole at $p$, and the functions 
$\uithreadfunction{i,1}$ $\uithreadfunction{i,2}$ do not have poles or zeros at $p$.
And if $p$ is a special point, then 
$\ord_p(\oluithreadfunction{i,j})=-\mathcal D_p(\uidegree{i,j})$
for $j=1,2$.
So, $\uidescriptionfunction{}$ is the $U_i$-description of a regular homogeneous 
function of degree $\chi$
on $U_i$, which we will denote by $f_4$.

Let $p\in V_i$ be an ordinary point. Suppose that $f_2$ is defined at $p$. 
Then $f_3$ is defined at each point of $\pi^{-1}(p)\cap U_i$.
Let $x_0$ be the canonical point in $\pi^{-1}(p)\cap U_i$.
By the definition of an $U_i$-description, $f_4(x_0)=\uidescriptionfunction{}(p)$.
On the other hand,
by Proposition \ref{quotmorph}, 
$$
\uidescriptionfunction{}(p)=
f_2(\pi(x_0))\frac{\widetilde f_1(x_0)}
{\wtuithreadfunction{i,1}^{\uidegree{i,1}^*(\chi)}(x_0)\wtuithreadfunction{i,2}^{\uidegree{i,2}^*(\chi)}}(x_0).
$$
Since $x_0$ is the canonical point in $\pi^{-1}(p)\cap U_i$, 
$\wtuithreadfunction{i,1}(x_0)=\wtuithreadfunction{i,2}(x_0)=1$, and 
$\uidescriptionfunction{}(p)=f_3(x_0)$. So, $f_4(x_0)=f_3(x_0)$.
Moreover, both $f_3$ and $f_4$ are homogeneous functions of degree $\chi$ with respect to the torus action, 
so, $f_3$ and $f_4$ coincide on the whole fiber $U_i\cap \pi^{-1}(p)$.

All ordinary points $p\in V_i$ such that $f_2$ is defined at $p$ form a non-empty open subset of $V_i$. 
Therefore, $f_3$ coincides with $f_4$ on an open subset of $U_i$. But $f_4$ is regular on $U_i$, 
so $f_3$ is regular on $U_i$ as well.

Now suppose that $f_3$ is regular on $U_i$. 
Again consider the following rational function $\uidescriptionfunction{}$ on $\PP^1$:
$$
\uidescriptionfunction{}=\frac{f_2\overline{f_1}}{\oluithreadfunction{i,1}^{\uidegree{i,1}^*(\chi)}\oluithreadfunction{i,2}^{\uidegree{i,2}^*(\chi)}}
$$
If $p\in V_i$ is an ordinary point, $f_2$ is defined at $p$, and $x_0$ is the canonical point 
in $\pi^{-1}(p)\cap U_i$, 
then by Proposition \ref{quotmorph},
$$
\uidescriptionfunction{}(p)=
f_2(\pi(x_0))\frac{\widetilde f_1(x_0)}
{\wtuithreadfunction{i,1}^{\uidegree{i,1}^*(\chi)}(x_0)\wtuithreadfunction{i,2}^{\uidegree{i,2}^*(\chi)}}(x_0).
$$
Since $x_0$ is the canonical point, $\uidescriptionfunction{}(p)=f_2(\pi(x_0))\widetilde f_1(x_0)=f_3(x_0)$.
Therefore, $\uidescriptionfunction{}$ coincides with the $U_i$-description of $f_3$ on a non-empty open 
subset of $V_i$.
But then $\uidescriptionfunction{}$ is the $U_i$-description of $f_3$, and $\uidescriptionfunction{}$ is defined everywhere on $V_i$.
Again, recall that 
$\uithreadfunction{i,1}$ $\uithreadfunction{i,2}$ do not have poles or zeros at ordinary points of $V_i$, 
and if $p\in V_i$ is a special point, 
then 
$\ord_p(\oluithreadfunction{i,j})=-\mathcal D_p(\uidegree{i,j})$
for $j=1,2$.
\end{proof}

\chapter{Combinatorial formula for the dimension of the graded component of $T^1$ of degree zero}\label{combformula}

\section{Construction of a particular sufficient system}\label{sectsuffsystemconstruction}

Without loss of generality, in this section we will assume that there are at least two special points
(we always can add trivial special points).
Recall that we have a coordinate function $t$ on $\PP^1$.
Now we will need more coordinate functions on $\PP^1$
(i.~e. rational functions with one pole and one zero, 
both are of order 1).
Namely, for each special point $p\in\PP^1$, we will need a coordinate function on $\PP^1$ that vanishes at $p$. 
Choose such coordinate functions and denote them by $t_p$.
We are also going to construct a sufficient system of sets $U_i$ more explicitly.

\begin{lemma}\label{giexist}
Let $p\in \PP^1$ be a special point and let $\chi\in\sck\cap M$ be a degree. 
There exists a rational function $f\in\Gamma(\PP^1,\OO(\mathcal D(\chi)))$ such that 
$\ord_p(f)=-\mathcal D_p(\chi)$, and $f$ does not have zeros or poles at 
ordinary points.
\end{lemma}

\begin{proof}
Choose a rational function $f$ on $\PP^1$ that has one simple zero and one simple pole, and 
that takes finite values at all special points. (For example, if $t=\infty$ is an ordinary point, we can take 
$f=t$, otherwise we can take $f=1/(t-a)$, where $a\in\CC$ and $t=a$ is an ordinary point.) 
Then each function $f-a$, where $a\in\CC$ again has one simple zero and one simple pole.

Recall that we have denoted all special points by $p_1,\ldots,p_{\numberofdivisorpoints}$.
Let $p=p_i$.
Denote $a_i=\mathcal D_p(\chi)$.
Since $\deg \mathcal D(\chi)\ge 0$, there exist $a_1, \ldots, a_{i-1}, a_{i+1},\ldots, a_{\numberofdivisorpoints}\in\ZZ$ 
such that $a_1+\ldots+a_{\numberofdivisorpoints}=0$ and $a_j\le \mathcal D_{p_j}(\chi)$ for $1\le j\le \numberofdivisorpoints$.
Consider the following function: 
$f_1=(f-f(p_1))^{-a_1}\ldots(f-f(p_{\numberofdivisorpoints}))^{-a_{\numberofdivisorpoints}}$
Since the sum of the exponents is zero, $f_1$ is defined and takes value 1 
at the (ordinary) point of $\PP^1$ where $f=\infty$.
Clearly, $f_1$ has no zeros or poles at other ordinary points. At $p$, we have 
$\ord_p (f_1)=-a_i=-\mathcal D_p(\chi)$, and at $p_j$ ($j\ne i$), we have $\ord_{p_j}(f_1)=-a_j\ge-\mathcal D_{p_j}(\chi)$.
\end{proof}

We are going to use a sufficient system 
$U_1,\ldots, U_{\numberoffixedopensets}$
constructed as follows.
We have several (in fact, up to two) sets $U_i$ for every pair $(p,j)$, where $p\in\PP^1$ is a special point, and 
$j$ corresponds to a vertex $\indexedvertexpt{p}{j}$ of $\stdpolyhedronletter_p$ ($1\le j\le \numberofverticespt{p}$, we 
write $(p,j)$ instead of $(p,\indexedvertexpt{p}{j})$ to simplify notation).
Each of 
these sets $U_i$ chosen for $(p,j)$ 
corresponds to a 
face of $\normalvertexcone{\indexedvertexpt{p}{j}}{\stdpolyhedronletter_p}$ (which can be 
$\normalvertexcone{\indexededgept{p}{j-1}}{\stdpolyhedronletter_p}$, $\normalvertexcone{\indexededgept{p}{j-1}}{\stdpolyhedronletter_p}$, 
or the interior of $\normalvertexcone{\indexedvertexpt{p}{j}}{\stdpolyhedronletter_p}$).
These sets together are $U_1,\ldots, U_{\numberoffixedopensets-1}$
Additionally, 
we will use one more set, which is $U_\numberoffixedopensets$, and which does not correspond to any special point.

More precisely, for every special point $p$,
for every 
vertex $\indexedvertexpt{p}{j}$ of $\stdpolyhedronletter_p$
and for each of the two rays 
$\normalvertexcone{\indexededgept{p}{j-1}}{\stdpolyhedronletter_p}$
and $\normalvertexcone{\indexededgept{p}{j}}{\stdpolyhedronletter_p}$
forming 
$\partial\normalvertexcone{\indexedvertexpt{p}{j}}{\stdpolyhedronletter_p}$
we choose a basis of $M$ 
as follows. First, let $\chi\in M$ be the lattice basis of the chosen ray. If $\chi\notin\partial\sck$, we do not choose a basis 
for this pair $(p,j)$ and for this ray. If $\deg\mathcal D(\chi)=0$, we do not choose a basis 
for this pair $(p,j)$ and for this ray. Otherwise, we choose a basis $\uidegree{i,1},\uidegree{i,2}$ of $M$, where $\uidegree{i,1}=\chi$ and $\uidegree{i,2}$ is a lattice point 
in the interior of $\normalvertexcone{\indexedvertexpt{p}{j}}{\stdpolyhedronletter_p}$.

We choose a basis of $M$ corresponding to a pair $(p,j)$ and to the interior of 
$\normalvertexcone{\indexedvertexpt{p}{j}}{\stdpolyhedronletter_p}$
only if at the previous 
step we finally did not choose any basis corresponding to the pair $(p,j)$ and 
to one of the two rays 
$\normalvertexcone{\indexededgept{p}{j-1}}{\stdpolyhedronletter_p}$
and $\normalvertexcone{\indexededgept{p}{j}}{\stdpolyhedronletter_p}$
(for example, this can happen if 
$\partial\sck\cap\partial\normalvertexcone{\indexedvertexpt{p}{j}}{\stdpolyhedronletter_p}=0$). 
In this case, we choose a basis 
$\uidegree{i,1},\uidegree{i,2}$ of $M$ such that $\uidegree{i,1}$ and $\uidegree{i,2}$ are lattice points in the interior of 
$\normalvertexcone{\indexedvertexpt{p}{j}}{\stdpolyhedronletter_p}$.
We continue using the notation $\uidegree{i,1}^*,\uidegree{i,2}^*$ for the dual bases.

Observe that we chose exactly one or two bases for each pair $(p,j)$. We chose two bases if and only if 
$p$ is a removable special point and
$\deg\mathcal D(\sckboundarybasis{0})>0$ and $\deg\mathcal D(\sckboundarybasis{1})>0$.

Now for every chosen basis, we choose functions $\uithreadfunction{i,1}\in\Gamma(\PP^1,\OO(\mathcal D(\uidegree{i,1})))$ 
and $\uithreadfunction{i,2}\in\Gamma(\PP^1,\OO(\mathcal D(\uidegree{i,2})))$
satisfying the conditions of Lemma \ref{giexist} for the corresponding special point $p$ and the degree 
$\uidegree{i,1}$ or $\uidegree{i,2}$, respectively. Then we may set $V_i$ to consist of all ordinary points and $p$.

We enumerate the sets $U_i$ so that the sets corresponding to $p_1\in \PP^1$ go first, then the sets corresponding to $p_2$, etc.
Among the sets corresponding to a single essential special point $p$, we have exactly one set $U_i$ for each vertex of 
$\stdpolyhedronletter_p$. We enumerate them along with the enumeration of vertices, i.~e. first 
we take the set corresponding to $(p,1)$, then the set corresponding to $(p,2)$, etc, then the set corresponding 
to $(p,\numberofverticespt{p})$.
If we have two sets $U_i$ corresponding to the same removable special point, we enumerate them arbitrarily.

These were the sets $U_1,\ldots,U_{\numberoffixedopensets-1}$.
To define the set for the sufficient system, i.~e. $U_\numberoffixedopensets$, 
choose an arbitrary basis $\uidegree{\numberoffixedopensets,1},\uidegree{\numberoffixedopensets,2}$ 
of $M$ such that $\uidegree{\numberoffixedopensets,1},\uidegree{\numberoffixedopensets,2}$
are in the interior of $\sck$, and 
choose functions 
$\uithreadfunction{\numberoffixedopensets,1}\in\Gamma(\PP^1,\OO(\mathcal D(\uidegree{\numberoffixedopensets,1})))$ 
and $\uithreadfunction{\numberoffixedopensets,2}\in\Gamma(\PP^1,\OO(\mathcal D(\uidegree{\numberoffixedopensets,2})))$ that 
do not have zeros or poles at ordinary points (such functions exist by Lemma \ref{giexist}). 
In this case, let $V_\numberoffixedopensets$ be the set of all ordinary points. 


Note that if we remove 
$U_{\numberoffixedopensets}$,
we will 
still get a sufficient system.
We will use whole system $U_1,\ldots,U_{\numberoffixedopensets}$
to compute $H^1(\PP^1,\gi)$, 
and the smaller sufficient system $U_1,\ldots, U_{\numberoffixedopensets-1}$
to compute $H^0(\PP^1,\giv)$ and $H^0(\PP^1,\gviii)$, 
During the computation of 
$H^0(\PP^1,\giv)$ and $H^0(\PP^1,\gviii)$,
the set $U_{\numberoffixedopensets}$
will only be used sometimes
to define $U_{\numberoffixedopensets}$-descriptions sometimes.

\section{Computation of the dimension of $H^1(\PP^1,\gi)$}

We start with $H^1(\PP^1,\gi)$. To compute this space, 
we need an affine covering of $\PP^1$. So, for each special point $p\in\PP^1$, we denote 
by $W_p$ the set consisting of all ordinary points of $\PP^1$ and $p$. 
These sets $W_p$ really form an affine covering 
since we have at least two special points. 
Denote also the set of 
all ordinary points by $W$. 
It follows directly from the definition of $\gi$
that the restriction maps to nonempty open sets are injective.
Note also that if $p\ne p'$ are special points, then $W_p\cap W_{p'}=W$.
So we can use Corollary \ref{h1compute} for Cech cohomology. By Corollary \ref{h1compute},
$$
H^1(\PP^1,\gi)=\left.\!\left(\bigoplus_{p\text{ special point}}\Big(\Gamma(W,\gi)/\Gamma(W_p,\gi)\Big)\right)\right/\Gamma(W,\gi).
$$

For an essential special point $p$, denote by $\gsi{p,\numberoffixedopensets}$ the space 
of triples $(\uidescriptionfunction{\numberoffixedopensets,1}, \uidescriptionfunction{\numberoffixedopensets,2}, \stdvectorfiledpletter)$, 
where $\uidescriptionfunction{\numberoffixedopensets,1},\uidescriptionfunction{\numberoffixedopensets,2}\in
\Gamma(W_p,\OO_{\PP^1})$, $\stdvectorfiledpletter\in\Gamma(W_p,\Theta_{\PP^1})$, and 
$\stdvectorfiledpletter(p)=0$.
The last index $\numberoffixedopensets$ indicates that these triples will be considered as $U_{\numberoffixedopensets}$-descriptions 
of vector fields on $\pi^{-1}(W_p)\cap U$.

\begin{lemma}\label{wpvfieldpushforward}
Let $p\in\PP^1$ be an essential special point.
Then $\Gamma(W_p, \gi)$ can be identified with $\gsi{p,\numberoffixedopensets}$.

The isomorphism here maps 
$(\uidescriptionfunction{\numberoffixedopensets,1},\uidescriptionfunction{\numberoffixedopensets,2},\stdvectorfiledpletter)\in \gsi{p,\numberoffixedopensets}$
to 
$(\uidescriptionfunction{1,1},\uidescriptionfunction{1,2}, 
\ldots,\uidescriptionfunction{\numberoffixedopensets,1},\uidescriptionfunction{\numberoffixedopensets,2},\stdvectorfiledpletter)$, 
where 
$$
\left(
\begin{array}{c}
\uidescriptionfunction{i,1}\\
\uidescriptionfunction{i,2}\\
\stdvectorfiledpletter
\end{array}
\right)
=
\uilargetransition{\numberoffixedopensets,i}
\left(
\begin{array}{c}
\uidescriptionfunction{\numberoffixedopensets,1}\\
\uidescriptionfunction{\numberoffixedopensets,2}\\
\stdvectorfiledpletter
\end{array}
\right).
$$
\end{lemma}

\begin{proof}
First, we have to check that the $(2 \numberoffixedopensets+1)$-tuple obtained this way 
from an element of $\gsi{p,\numberoffixedopensets}$ really defines an element of $\Gamma(W_p,\gi)$. 
The equalities
$$
\left(
\begin{array}{c}
\uidescriptionfunction{i,1}\\
\uidescriptionfunction{i,2}\\
\stdvectorfiledpletter
\end{array}
\right)
=
\uilargetransition{j,i}
\left(
\begin{array}{c}
\uidescriptionfunction{j,1}\\
\uidescriptionfunction{j,2}\\
\stdvectorfiledpletter
\end{array}
\right)
$$
for arbitrary indices $i,j$ follow from Lemma \ref{rijcocycle}. All functions $\uidescriptionfunction{i,1}$ and $\uidescriptionfunction{i,2}$ 
are regular at ordinary points by Lemma \ref{rijregular}. Let $i$ be an index such that 
$U_i$ corresponds to the special point $p$ in the above construction. We have
$$
\uidescriptionfunction{i,1}=\uidegree{\numberoffixedopensets,1}^*(\uidegree{i,1})\uidescriptionfunction{\numberoffixedopensets,1}+
\uidegree{\numberoffixedopensets,2}^*(\uidegree{i,1})\uidescriptionfunction{\numberoffixedopensets,2}+
\frac{\oluithreadfunction{\numberoffixedopensets,1}^{\uidegree{\numberoffixedopensets,1}^*(\uidegree{i,1})}
\oluithreadfunction{\numberoffixedopensets,2}^{\uidegree{\numberoffixedopensets,2}^*(\uidegree{i,1})}}{\oluithreadfunction{i,1}} 
d\left(\frac{\oluithreadfunction{i,1}}{\oluithreadfunction{\numberoffixedopensets,1}^{\uidegree{\numberoffixedopensets,1}^*(\uidegree{i,1})}
\oluithreadfunction{\numberoffixedopensets,2}^{\uidegree{\numberoffixedopensets,2}^*(\uidegree{i,1})}}\right)\stdvectorfiledpletter.
$$
The covector field in the last summand 
is a logarithmic derivative of a rational function on $\PP^1$, so it cannot have 
a pole of order more than 1. 
Since $\stdvectorfiledpletter(p)=0$, $\uidescriptionfunction{i,1}$ is defined at $p$. The argument for $\uidescriptionfunction{i,2}$ is similar.

Clearly, this map from the space of triples to $\Gamma(W_p,\gi)$ is injective. To prove surjectivity, 
we have to check that if $(\uidescriptionfunction{1,1},\uidescriptionfunction{1,2}, \ldots,
\uidescriptionfunction{\numberoffixedopensets,1},\uidescriptionfunction{\numberoffixedopensets,2},\stdvectorfiledpletter)\in\Gamma(W_p,\gi)$, 
then $\stdvectorfiledpletter(p)=0$ and $\uidescriptionfunction{\numberoffixedopensets,1}$ and 
$\uidescriptionfunction{\numberoffixedopensets,2}$ have no poles at $p$. Let $U_i$ and $U_j$ be two open subsets 
corresponding to the special point $p$ and two 
normal subcones of two different vertices of $\stdpolyhedronletter_p$.
If $\stdvectorfiledpletter(p)\ne 0$, then by Lemma \ref{logderpole}, 
$$
\ord_p\left(\frac{\oluithreadfunction{i,1}^{\uidegree{i,1}^*(\uidegree{j,1})}
\oluithreadfunction{i,2}^{\uidegree{i,2}^*(\uidegree{j,1})}}
{\oluithreadfunction{j,1}}
d\left(\frac{\oluithreadfunction{j,1}}{\oluithreadfunction{i,1}^{\uidegree{i,1}^*(\uidegree{j,1})}
\oluithreadfunction{i,2}^{\uidegree{i,2}^*(\uidegree{j,1})}}\right)\stdvectorfiledpletter
\right)=-1,
$$
and $\uidescriptionfunction{j,1}$, $\uidescriptionfunction{i,1}$ and $\uidescriptionfunction{i,2}$ cannot be defined at $p$ simultaneously. Therefore, $\stdvectorfiledpletter(p)=0$. Finally, 
$$
\uidescriptionfunction{\numberoffixedopensets,1}=
\uidegree{i,1}^*(\uidegree{\numberoffixedopensets,1})\uidescriptionfunction{i,1}+
\uidegree{i,2}^*(\uidegree{\numberoffixedopensets,1})\uidescriptionfunction{i,2}+
\frac{\oluithreadfunction{i,1}^{\uidegree{i,1}^*(\uidegree{\numberoffixedopensets,1})}
\oluithreadfunction{i,2}^{\uidegree{i,2}^*(\uidegree{\numberoffixedopensets,1})}}{\oluithreadfunction{\numberoffixedopensets,1}} 
d\left(\frac{\oluithreadfunction{\numberoffixedopensets,1}}{\oluithreadfunction{i,1}^{\uidegree{i,1}^*(\uidegree{\numberoffixedopensets,1})}
\oluithreadfunction{i,2}^{\uidegree{i,2}^*(\uidegree{\numberoffixedopensets,1})}}\right)\stdvectorfiledpletter.
$$
Again, covector field in the last summand here is a logarithmic derivative of a rational function on $\PP^1$, so it cannot have 
a pole of order more than 1. Since $\stdvectorfiledpletter(p)=0$, $\uidescriptionfunction{\numberoffixedopensets,1}$ has no pole at $p$. 
Similarly, $\uidescriptionfunction{\numberoffixedopensets,2}$ has no pole at $p$.
\end{proof}

Now let $p$ be an essential special point. Recall that $t_p$ is a coordinate function on $\PP^1$
that has a (simple) zero at $p$.
Denote by $\abstractvectorspace_{0,0,p}$
the space of triples of Laurent polynomials of the form 
$(a_{1,-1}t_p^{-1}+\ldots+a_{1,-n_1}t_p^{-n_1},
a_{2,-1}t_p^{-1}+\ldots+a_{2,-n_2}t_p^{-n_2}, 
(b_0+b_{-1}t_p^{-1}+\ldots+b_{-n_3}t_p^{-n_3})\partial/\partial t_p)$.

\begin{lemma}\label{wwpquotient}
If $p\in\PP^1$ is an essential special point,
then 
$\abstractvectorspace_{0,0,p}$ is isomorphic to $\Gamma(W,\gi)/\Gamma(W_p,\gi)$. The isomorphism here 
is the composition of the map 
$$
(\uidescriptionfunction{\numberoffixedopensets,1},
\uidescriptionfunction{\numberoffixedopensets,2}, 
\stdvectorfiledpletter)\mapsto(\uidescriptionfunction{1,1},\uidescriptionfunction{1,2}, 
\ldots,\uidescriptionfunction{\numberoffixedopensets,1},\uidescriptionfunction{\numberoffixedopensets,2},\stdvectorfiledpletter)
\in\Gamma(W,\gi),
$$ 
where
$$
\left(
\begin{array}{c}
\uidescriptionfunction{i,1}\\
\uidescriptionfunction{i,2}\\
\stdvectorfiledpletter
\end{array}
\right)
=
\uilargetransition{\numberoffixedopensets,i}
\left(
\begin{array}{c}
\uidescriptionfunction{\numberoffixedopensets,1}\\
\uidescriptionfunction{\numberoffixedopensets,2}\\
\stdvectorfiledpletter
\end{array}
\right),
$$
and the canonical projection 
$\Gamma(W,\gi)\to\Gamma(W,\gi)/\Gamma(W_p,\gi)$.

If $(\uidescriptionfunction{1,1}',\uidescriptionfunction{1,2}',\ldots, 
\uidescriptionfunction{\numberoffixedopensets,1}', \uidescriptionfunction{\numberoffixedopensets,2}',\stdvectorfiledpletter')\in 
\Gamma(W,\gi)$ is a section that belongs to the same
coset in 
$\Gamma(W,\gi)/\Gamma(W_p,\gi)$ 
as 
the image of 
$(\uidescriptionfunction{\numberoffixedopensets,1},
\uidescriptionfunction{\numberoffixedopensets,2}, 
\stdvectorfiledpletter)
\in \abstractvectorspace_{0,0,p}$ 
under the isomorphism above,
then
$\uidescriptionfunction{\numberoffixedopensets,1}'-\uidescriptionfunction{\numberoffixedopensets,1}$ and 
$\uidescriptionfunction{\numberoffixedopensets,2}'-\uidescriptionfunction{\numberoffixedopensets,1}$ are functions regular at $p$, and 
$\stdvectorfiledpletter'-\stdvectorfiledpletter$ is a vector field that 
vanishes at $p$.
\end{lemma}

\begin{proof}
The proof is similar to the proof of the previous lemma. 
Let $(\uidescriptionfunction{\numberoffixedopensets,1},
\uidescriptionfunction{\numberoffixedopensets,2}, 
\stdvectorfiledpletter)\in\abstractvectorspace_{0,0,p}$. 
Denote its image in 
$\Gamma(W,\gi)$
by
$(\uidescriptionfunction{1,1},\uidescriptionfunction{1,2}, 
\ldots,\uidescriptionfunction{\numberoffixedopensets,1},\uidescriptionfunction{\numberoffixedopensets,2},\stdvectorfiledpletter)$.
The functions $\uidescriptionfunction{\numberoffixedopensets,1}$ and $\uidescriptionfunction{\numberoffixedopensets,2}$ 
and the vector field $\stdvectorfiledpletter$ have no poles except $p$, the entries of $\uilargetransition{\numberoffixedopensets,i}$
have no poles at ordinary points by Lemma \ref{rijregular}, so $\uidescriptionfunction{i,1}, \uidescriptionfunction{i,2}\in \Gamma(W,\OO_{\PP^1})$. Therefore,
$(\uidescriptionfunction{1,1},\uidescriptionfunction{1,2},\ldots, \uidescriptionfunction{\numberoffixedopensets,1}, 
\uidescriptionfunction{\numberoffixedopensets,2},\stdvectorfiledpletter)$ really defines an element of $\Gamma(W,\gi)$ since all 
necessary equations are satisfied by Lemma \ref{rijcocycle}.

Now let $(\uidescriptionfunction{1,1}',\uidescriptionfunction{1,2}',\ldots, \uidescriptionfunction{\numberoffixedopensets,1}', 
\uidescriptionfunction{\numberoffixedopensets,2}',\stdvectorfiledpletter')\in 
\Gamma(W,\gi)$ be a section.
Let
$$
\uidescriptionfunction{\numberoffixedopensets,1}'=\sum_{k=-n_1}^\infty a_{1,k}t_p^k,\quad 
\uidescriptionfunction{\numberoffixedopensets,2}'=\sum_{k=-n_2}^\infty a_{2,k}t_p^k,\quad 
\stdvectorfiledpletter'=\left(\sum_{k=-n_3}^\infty b_kt_p^k\right)\frac\partial{\partial t_p}
$$
be the Laurent series for $\uidescriptionfunction{\numberoffixedopensets,1}$, $\uidescriptionfunction{\numberoffixedopensets,2}$, and 
$\stdvectorfiledpletter$, respectively (in the sense of complex analysis). Denote
$$
\uidescriptionfunction{\numberoffixedopensets,1}=\sum_{k=-n_1}^{-1}a_{1,k}t_p^k,\quad 
\uidescriptionfunction{\numberoffixedopensets,2}=\sum_{k=-n_2}^{-1}a_{2,k}t_p^k,\quad 
\stdvectorfiledpletter=\left(\sum_{k=-n_3}^0b_kt_p^k\right)\frac\partial{\partial t_p}.
$$
These sums are finite, so they define algebraic rational functions and an algebraic rational vector field. Hence, 
$\uidescriptionfunction{\numberoffixedopensets,1}-\uidescriptionfunction{\numberoffixedopensets,1}'$, 
$\uidescriptionfunction{\numberoffixedopensets,2}-\uidescriptionfunction{\numberoffixedopensets,2}'$, and 
$\stdvectorfiledpletter-\stdvectorfiledpletter'$ are also algebraic rational. They are defined at $p$ in complex-analytic sense, hence they have no poles at $p$ in 
algebraic sense. Note also that $(\stdvectorfiledpletter-\stdvectorfiledpletter')(p)=0$.
By Lemma \ref{wpvfieldpushforward}, the triple 
$(\uidescriptionfunction{\numberoffixedopensets,1}-\uidescriptionfunction{\numberoffixedopensets,1}', 
\uidescriptionfunction{\numberoffixedopensets,2}-\uidescriptionfunction{\numberoffixedopensets,2}', 
\stdvectorfiledpletter-\stdvectorfiledpletter')$ defines an element of $\Gamma(W_p,\gi)$, so 
$(\uidescriptionfunction{\numberoffixedopensets,1}, \uidescriptionfunction{\numberoffixedopensets,2}, \stdvectorfiledpletter)$ is equivalent to 
$(\uidescriptionfunction{\numberoffixedopensets,1}', \uidescriptionfunction{\numberoffixedopensets,2}', \stdvectorfiledpletter')$ in $\Gamma(W,\gi)/\Gamma(W_p,\gi)$. 
But 
$(\uidescriptionfunction{\numberoffixedopensets,1}, \uidescriptionfunction{\numberoffixedopensets,2}, \stdvectorfiledpletter)\in\abstractvectorspace_{0,0,p}$,
so 
the map from 
$\abstractvectorspace_{0,0,p}$
to $\Gamma(W,\gi)/\Gamma(W_p,\gi)$
is surjective. The injectivity of the map 
$\abstractvectorspace_{0,0,p}\to\Gamma(W,\gi)$ 
is clear, 
and it follows from Lemma \ref{wpvfieldpushforward} that the only triple that maps to $\Gamma(W_p,\gi)$ is $(0,0,0)$.
\end{proof}

Let $p\in\PP^1$ be a removable special point,
and
let $U_i$ be a subset of $X$ corresponding to $p$.
Denote by $\gsi{p,i}$
the space of triples $(\uidescriptionfunction{i,1}, \uidescriptionfunction{i,2}, \stdvectorfiledpletter)$, 
where $\uidescriptionfunction{i,1},\uidescriptionfunction{i,2}\in\Gamma(W_p,\OO_{\PP^1})$, $\stdvectorfiledpletter\in\Gamma(W_p,\Theta_{\PP^1})$, but 
this time it is not necessarily true that $\stdvectorfiledpletter(p)=0$.
The last index $i$ in the notation $\gsi{p,i}$ indicates that these triples will be considered as $U_i$-descriptions 
of vector fields on $\pi^{-1}(W_p)\cap U$.

\begin{lemma}\label{wplinearvfieldpushforward}
Let $p\in\PP^1$ be a removable special point,
and
let $U_i$ be a subset of $X$ corresponding to $p$.
Then $\Gamma(W_p, \gi)$ can be identified with $\gsi{p,i}$.

The isomorphism here maps 
$(\uidescriptionfunction{i,1}, \uidescriptionfunction{i,2}, \stdvectorfiledpletter)$
to
$(\uidescriptionfunction{1,1},\uidescriptionfunction{1,2}, 
\ldots,\uidescriptionfunction{\numberoffixedopensets,1},\uidescriptionfunction{\numberoffixedopensets,2},\stdvectorfiledpletter)$, 
where 
$$
\left(
\begin{array}{c}
\uidescriptionfunction{j,1}\\
\uidescriptionfunction{j,2}\\
\stdvectorfiledpletter
\end{array}
\right)
=
\uilargetransition{i,j}
\left(
\begin{array}{c}
\uidescriptionfunction{i,1}\\
\uidescriptionfunction{i,2}\\
\stdvectorfiledpletter
\end{array}
\right).
$$
\end{lemma}

\begin{proof}
The proof is similar to the proofs of two previous lemmas. 
All necessary linear equations in the definition of $\gi$ 
are satisfied by Lemma \ref{rijcocycle}. We only have to check that 
if $U_j$ is another subset of $X$ corresponding to $p$, then $\uidescriptionfunction{j,1}$ and $\uidescriptionfunction{j,2}$ do not have poles at $p$.
The only entries of $\uilargetransition{i,j}$ that could have poles at $p$ are
$$
\frac{\oluithreadfunction{i,1}^{\uidegree{i,1}^*(\uidegree{j,1})}\oluithreadfunction{i,2}^{\uidegree{i,2}^*(\uidegree{j,1})}}
{\oluithreadfunction{j,1}} 
d\left(\frac{\oluithreadfunction{j,1}}
{\oluithreadfunction{i,1}^{\uidegree{i,1}^*(\uidegree{j,1})}\oluithreadfunction{i,2}^{\uidegree{i,2}^*(\uidegree{j,1})}}\right)
$$
and
$$
\frac{\oluithreadfunction{i,1}^{\uidegree{i,1}^*(\uidegree{j,2})}\oluithreadfunction{i,2}^{\uidegree{i,2}^*(\uidegree{j,2})}}
{\oluithreadfunction{j,2}} 
d\left(\frac{\oluithreadfunction{j,2}}
{\oluithreadfunction{i,1}^{\uidegree{i,1}^*(\uidegree{j,2})}\oluithreadfunction{i,2}^{\uidegree{i,2}^*(\uidegree{j,2})}}\right).
$$
Consider the first one of them, the second one is considered similarly. We have 
$\ord_p(\oluithreadfunction{i,1})=-\mathcal D_p(\uidegree{i,1})$, 
$\ord_p(\oluithreadfunction{i,2})=-\mathcal D_p(\uidegree{i,2})$, 
$\ord_p(\oluithreadfunction{j,1})=-\mathcal D_p(\uidegree{j,1})$.
Since $p$ is a removable special point
and $\uidegree{j,1}=\uidegree{i,1}^*(\uidegree{j,1})\uidegree{i,1}+\uidegree{i,2}^*(\uidegree{j,1})\uidegree{i,2}$, 
$\mathcal D_p(\uidegree{j,1})=\uidegree{i,1}^*(\uidegree{j,1})\mathcal D_p(\uidegree{i,1})+\uidegree{i,2}^*(\uidegree{j,1})\mathcal D_p(\uidegree{i,2})$, and
$$
\ord_p\left(\frac{\oluithreadfunction{j,1}}
{\oluithreadfunction{i,1}^{\uidegree{i,1}^*(\uidegree{j,1})}\oluithreadfunction{i,2}^{\uidegree{i,2}^*(\uidegree{j,1})}}\right)=0.
$$
Therefore, the logarithmic derivative of this function does not have a zero or a pole at $p$, and $\uidescriptionfunction{j,1}$ is well-defined at $p$.
The argument for $\uidescriptionfunction{j,2}$ is similar.

The injectivity of the map from 
$\gsi{p,i}$
to 
$\Gamma(W_p,\gi)$
is again clear since 
a $(2 \numberoffixedopensets+1)$-tuple 
defines the zero section only if all entries are zeros, and the surjectivity is also clear this time since in every section from 
$\Gamma(W_p,\gi)$, $\uidescriptionfunction{i,1}$, $\uidescriptionfunction{i,2}$, and $\stdvectorfiledpletter$ should be well-defined at $p$.
\end{proof}

Note that in this lemma, 
we use an affine open set $U_i$, 
which depends on $p$, and in fact used the $U_i$-description of a vector field, while in Lemma \ref{wpvfieldpushforward}
we used 
$U_\numberoffixedopensets$,
which did not depend on $p$, and used the $U_\numberoffixedopensets$-description. 
However, in the next lemma, we are going to use 
$U_\numberoffixedopensets$
again.

For a removable special point $p$, denote by $\abstractvectorspace_{0,0,p}$
the space 
of triples 
of Laurent polynomials of the form 
$(a_{1,-1}t_p^{-1}+\ldots+a_{1,-n_1}t_p^{-n_1},
a_{2,-1}t_p^{-1}+\ldots+a_{2,-n_2}t_p^{-n_2}, 
(b_{-1}t_p^{-1}+\ldots+b_{-n_3}t_p^{-n_3})\partial/\partial t_p)$.
 
\begin{lemma}\label{wwplinearquotient}
If $p\in\PP^1$ is a removable special point,
the space $\Gamma(W,\gi)/\Gamma(W_p,\gi)$ can be identified with $\abstractvectorspace_{0,0,p}$.

More exactly, these three Laurent polynomials are three \textbf{last} entries in a $(2 \numberoffixedopensets+1)$-tuple 
defining an element of $\Gamma(W,\gi)$, which in turn defines a coset in $\Gamma(W,\gi)/\Gamma(W_p,\gi)$.
In other words, the isomorphism is the composition of the map
$$
(\uidescriptionfunction{\numberoffixedopensets,1},
\uidescriptionfunction{\numberoffixedopensets,2}, 
\stdvectorfiledpletter)\mapsto(\uidescriptionfunction{1,1},\uidescriptionfunction{1,2}, 
\ldots,\uidescriptionfunction{\numberoffixedopensets,1},\uidescriptionfunction{\numberoffixedopensets,2},\stdvectorfiledpletter)
\in\Gamma(W,\gi),
$$
where
$$
\left(
\begin{array}{c}
\uidescriptionfunction{i,1}\\
\uidescriptionfunction{i,2}\\
\stdvectorfiledpletter
\end{array}
\right)
=
\uilargetransition{\numberoffixedopensets,i}
\left(
\begin{array}{c}
\uidescriptionfunction{\numberoffixedopensets,1}\\
\uidescriptionfunction{\numberoffixedopensets,2}\\
\stdvectorfiledpletter
\end{array}
\right),
$$
and the canonical projection 
$\Gamma(W,\gi)\to\Gamma(W,\gi)/\Gamma(W_p,\gi)$.

The vector field here always differs from the last entry of \textbf{any} element of $\Gamma(W,\gi)$
from the same coset in $\Gamma(W,\gi)/\Gamma(W_p,\gi)$ by a vector field that has no pole at $p$.
This is true for the two functions if the vector field is zero in both representatives of the coset.
\end{lemma}

\begin{proof}
First, if 
$(\uidescriptionfunction{\numberoffixedopensets,1},
\uidescriptionfunction{\numberoffixedopensets,2}, 
\stdvectorfiledpletter)\mapsto(\uidescriptionfunction{1,1},\uidescriptionfunction{1,2}, 
\ldots,\uidescriptionfunction{\numberoffixedopensets,1},\uidescriptionfunction{\numberoffixedopensets,2},\stdvectorfiledpletter)$, then 
$\uidescriptionfunction{i,1}$, $\uidescriptionfunction{i,2}$, and $\stdvectorfiledpletter$ 
have no poles outside $p$, entries of $\uilargetransition{i,j}$ have no poles at ordinary points, 
and all necessary equations are satisfied by Lemma \ref{rijcocycle}, 
so these functions and this vector field really define an element of $\Gamma(W,\gi)$, and hence an 
element of $\Gamma(W,\gi)/\Gamma(W_p,\gi)$.

The proof of injectivity is quite easy. If 
$(\uidescriptionfunction{\numberoffixedopensets,1},
\uidescriptionfunction{\numberoffixedopensets,2}, 
\stdvectorfiledpletter)\mapsto(\uidescriptionfunction{1,1},\uidescriptionfunction{1,2}, 
\ldots,\uidescriptionfunction{\numberoffixedopensets,1},\uidescriptionfunction{\numberoffixedopensets,2},\stdvectorfiledpletter)\in \Gamma(W_p,\gi)$,
then $\stdvectorfiledpletter=0$ since otherwise it has pole at $p$. But then we can choose an open set $U_i$ corresponding to $p$ 
and write
$$
\left(
\begin{array}{c}
\uidescriptionfunction{\numberoffixedopensets,1}\\
\uidescriptionfunction{\numberoffixedopensets,2}\\
\end{array}
\right)=
\uismalltransition{i,\numberoffixedopensets}
\left(
\begin{array}{c}
\uidescriptionfunction{i,1}\\
\uidescriptionfunction{i,2}\\
\end{array}
\right).
$$
The matrix $\uismalltransition{i,\numberoffixedopensets}$ has only constant entries, so 
if $\uidescriptionfunction{i,1}$ and $\uidescriptionfunction{i,2}$ are regular at $p$, then 
$\uidescriptionfunction{\numberoffixedopensets,1}$ and $\uidescriptionfunction{\numberoffixedopensets,2}$ are 
regular at $p$ as well. 
But then $\uidescriptionfunction{\numberoffixedopensets,1}=\uidescriptionfunction{\numberoffixedopensets,2}=0$.

Now we prove surjectivity.
Let
$(\uidescriptionfunction{1,1}',\uidescriptionfunction{1,2}',\ldots, 
\uidescriptionfunction{\numberoffixedopensets,1}',\uidescriptionfunction{\numberoffixedopensets,2}', \stdvectorfiledpletter')\in\Gamma(W,\gi)$
be a section.
Choose an index $i$ such that $U_i$ corresponds to $p$ and 
write complex-analytic Laurent series:
$$
\uidescriptionfunction{i,1}'=\sum_{k=-n_1}^\infty a'_{1,k}t_p^k,\quad 
\uidescriptionfunction{i,2}'=\sum_{k=-n_2}^\infty a'_{2,k}t_p^k,\quad 
\stdvectorfiledpletter'=\left(\sum_{k=-n_3}^\infty b'_kt_p^k\right)\frac\partial{\partial t_p}.
$$
Set
$$
\uidescriptionfunction{i,1}''=\sum_{k=-n_1}^{-1}a'_{1,k}t_p^k,\quad 
\uidescriptionfunction{i,2}''=\sum_{k=-n_2}^{-1}a'_{2,k}t_p^k,\quad 
\stdvectorfiledpletter''=\left(\sum_{k=-n_3}^{-1}b'_kt_p^k\right)\frac\partial{\partial t_p},
$$
and
$$
\left(
\begin{array}{c}
\uidescriptionfunction{j,1}''\\
\uidescriptionfunction{j,2}''\\
\stdvectorfiledpletter''
\end{array}
\right)
=
\uilargetransition{i,j}
\left(
\begin{array}{c}
\uidescriptionfunction{i,1}''\\
\uidescriptionfunction{i,2}''\\
\stdvectorfiledpletter''
\end{array}
\right)
$$
for all $j$ ($1\le j\le \numberoffixedopensets$).
Observe that
$\uidescriptionfunction{i,1}''-\uidescriptionfunction{i,1}'$, 
$\uidescriptionfunction{i,2}''-\uidescriptionfunction{i,2}'$ and 
$\stdvectorfiledpletter''-\stdvectorfiledpletter'$
are well-defined at $p$, 
so 
$(\uidescriptionfunction{i,1}''-\uidescriptionfunction{i,1}',\uidescriptionfunction{i,2}''-\uidescriptionfunction{i,2}',
\stdvectorfiledpletter''-\stdvectorfiledpletter')\in \gsi{p,i}$.
The image of this element of $\gsi{p,i}$
under the isomorphism from Lemma \ref{wplinearvfieldpushforward}
equals
$(\uidescriptionfunction{1,1}''-\uidescriptionfunction{1,1}',\uidescriptionfunction{1,2}''-\uidescriptionfunction{1,2}', 
\ldots,\uidescriptionfunction{\numberoffixedopensets,1}''-\uidescriptionfunction{\numberoffixedopensets,1}',
\uidescriptionfunction{\numberoffixedopensets,2}''-\uidescriptionfunction{\numberoffixedopensets,2}',
\stdvectorfiledpletter''-\stdvectorfiledpletter')\in\Gamma(W_p,\gi)$, hence,
by Remark \ref{uidescvfieldrestriction}, 
$(\uidescriptionfunction{1,1}''-\uidescriptionfunction{1,1}',\uidescriptionfunction{1,2}''-\uidescriptionfunction{1,2}', 
\ldots,\uidescriptionfunction{\numberoffixedopensets,1}''-\uidescriptionfunction{\numberoffixedopensets,1}',
\uidescriptionfunction{\numberoffixedopensets,2}''-\uidescriptionfunction{\numberoffixedopensets,2}',
\stdvectorfiledpletter''-\stdvectorfiledpletter')\in\Gamma(W,\gi)$
defines the zero coset in 
$\Gamma(W,\gi)/\Gamma(W_p,\gi)$.
It is sufficient to prove that 
$(\uidescriptionfunction{1,1}'',\uidescriptionfunction{1,2}'',\ldots,
\uidescriptionfunction{\numberoffixedopensets,1}'',\uidescriptionfunction{\numberoffixedopensets,2}'',\stdvectorfiledpletter'')$
is in the image of the morphism $\abstractvectorspace_{0,0,p}\to \Gamma(W,\gi)/\Gamma(W_p,\gi)$.

Now write 
$$
\uidescriptionfunction{\numberoffixedopensets,1}''=\sum_{k=-n_1}^{\infty}a''_{1,k}t_p^k,\quad 
\uidescriptionfunction{\numberoffixedopensets,2}''=\sum_{k=-n_2}^{\infty}a''_{2,k}t_p^k
$$
(without loss of generality, we may suppose that $n_1$ and $n_2$ did not change, we may add more zeros in the negative part of Laurent series)
and recall that
$$
\stdvectorfiledpletter''=\left(\sum_{k=-n_3}^{-1}b'_kt_p^k\right)\frac\partial{\partial t_p}.
$$
Set
$$
\uidescriptionfunction{\numberoffixedopensets,1}=\sum_{k=-n_1}^{-1}a''_{1,k}t_p^k,\quad 
\uidescriptionfunction{\numberoffixedopensets,2}=\sum_{k=-n_2}^{-1}a''_{2,k}t_p^k,\quad 
\stdvectorfiledpletter=\stdvectorfiledpletter'',
$$
and
$$
\left(
\begin{array}{c}
\uidescriptionfunction{j,1}\\
\uidescriptionfunction{j,2}\\
\stdvectorfiledpletter
\end{array}
\right)
=
\uilargetransition{\numberoffixedopensets,j}
\left(
\begin{array}{c}
\uidescriptionfunction{\numberoffixedopensets,1}\\
\uidescriptionfunction{\numberoffixedopensets,2}\\
\stdvectorfiledpletter
\end{array}
\right)
$$
for all $j$ ($1\le j\le \numberoffixedopensets$).
Then 
$(\uidescriptionfunction{\numberoffixedopensets,1},\uidescriptionfunction{\numberoffixedopensets,2},\stdvectorfiledpletter)
\in\abstractvectorspace_{0,0,p}$.
Since $v=v''$, we have
$$
\left(
\begin{array}{c}
\uidescriptionfunction{i,1}\\
\uidescriptionfunction{i,2}\\
\stdvectorfiledpletter
\end{array}
\right)
-
\left(
\begin{array}{c}
\uidescriptionfunction{i,1}''\\
\uidescriptionfunction{i,2}''\\
\stdvectorfiledpletter''
\end{array}
\right)=
\uilargetransition{\numberoffixedopensets,i}
\left(
\begin{array}{c}
\uidescriptionfunction{\numberoffixedopensets,1}-\uidescriptionfunction{\numberoffixedopensets,1}''\\
\uidescriptionfunction{\numberoffixedopensets,2}-\uidescriptionfunction{\numberoffixedopensets,2}''\\
0
\end{array}
\right),
$$
and
$$
\left(
\begin{array}{c}
\uidescriptionfunction{i,1}\\
\uidescriptionfunction{i,2}\\
\end{array}
\right)
-
\left(
\begin{array}{c}
\uidescriptionfunction{i,1}''\\
\uidescriptionfunction{i,2}''\\
\end{array}
\right)=
\uismalltransition{\numberoffixedopensets,i}
\left(
\begin{array}{c}
\uidescriptionfunction{\numberoffixedopensets,1}-\uidescriptionfunction{\numberoffixedopensets,1}''\\
\uidescriptionfunction{\numberoffixedopensets,2}-\uidescriptionfunction{\numberoffixedopensets,2}''\\
\end{array}
\right).
$$
Hence, 
$\uidescriptionfunction{i,1}''-\uidescriptionfunction{i,1}$ and $\uidescriptionfunction{i,2}''-\uidescriptionfunction{i,2}$
are regular at $p$,
$(\uidescriptionfunction{i,1}''-\uidescriptionfunction{i,1},\uidescriptionfunction{i,2}''-\uidescriptionfunction{i,2},
\stdvectorfiledpletter''-\stdvectorfiledpletter)\in \gsi{p,i}$, 
and the isomorphism from Lemma \ref{wplinearvfieldpushforward}
maps this triple to
$(\uidescriptionfunction{1,1}''-\uidescriptionfunction{1,1},\uidescriptionfunction{1,2}''-\uidescriptionfunction{1,2},\ldots,
\uidescriptionfunction{\numberoffixedopensets,1}''-\uidescriptionfunction{\numberoffixedopensets,1},
\uidescriptionfunction{\numberoffixedopensets,2}''-\uidescriptionfunction{\numberoffixedopensets,2},
\stdvectorfiledpletter''-\stdvectorfiledpletter)\in\Gamma(W_p,\gi)$.
Therefore, 
$(\uidescriptionfunction{1,1}'',\uidescriptionfunction{1,2}'',\ldots,
\uidescriptionfunction{\numberoffixedopensets,1}'',\uidescriptionfunction{\numberoffixedopensets,2}'',\stdvectorfiledpletter'')$
defines the same coset in $\Gamma(W,\gi)/\Gamma(W_p,\gi)$
as
$(\uidescriptionfunction{1,1},\uidescriptionfunction{1,2},\ldots,
\uidescriptionfunction{\numberoffixedopensets,1},\uidescriptionfunction{\numberoffixedopensets,2},\stdvectorfiledpletter)$, 
which is the image of 
$(\uidescriptionfunction{\numberoffixedopensets,1},\uidescriptionfunction{\numberoffixedopensets,2},\stdvectorfiledpletter)
\in\abstractvectorspace_{0,0,p}$.

During the proof of surjectivity, we have changed the vector field from $\stdvectorfiledpletter'$ to 
$\stdvectorfiledpletter''=\stdvectorfiledpletter$, and we chose $\stdvectorfiledpletter''$ so that
$\stdvectorfiledpletter'-\stdvectorfiledpletter''$ is regular at $p$. If we started with 
$\stdvectorfiledpletter'=0$, then $\stdvectorfiledpletter''=0$ as well. In this case
$$
\left(
\begin{array}{c}
\uidescriptionfunction{\numberoffixedopensets,1}''\\
\uidescriptionfunction{\numberoffixedopensets,2}''\\
\end{array}
\right)
-
\left(
\begin{array}{c}
\uidescriptionfunction{\numberoffixedopensets,1}'\\
\uidescriptionfunction{\numberoffixedopensets,2}'\\
\end{array}
\right)=
\uismalltransition{i,\numberoffixedopensets}
\left(
\begin{array}{c}
\uidescriptionfunction{i,1}''-\uidescriptionfunction{i,1}'\\
\uidescriptionfunction{i,2}''-\uidescriptionfunction{i,2}'\\
\end{array}
\right).
$$
$\uidescriptionfunction{i,1}''-\uidescriptionfunction{i,1}'$ and 
$\uidescriptionfunction{i,2}''-\uidescriptionfunction{i,2}'$ are regular at $p$ by construction, 
all entries in $\uismalltransition{i,\numberoffixedopensets}$ are constants, so 
$\uidescriptionfunction{\numberoffixedopensets,1}''-\uidescriptionfunction{\numberoffixedopensets,1}'$ and 
$\uidescriptionfunction{\numberoffixedopensets,2}''-\uidescriptionfunction{\numberoffixedopensets,2}'$ are regular at $p$.
Recall also that 
$\uidescriptionfunction{\numberoffixedopensets,1}-\uidescriptionfunction{\numberoffixedopensets,1}''$ and 
$\uidescriptionfunction{\numberoffixedopensets,2}-\uidescriptionfunction{\numberoffixedopensets,2}''$
are regular at $p$ by construction, so finally we see that 
$\uidescriptionfunction{\numberoffixedopensets,1}-\uidescriptionfunction{\numberoffixedopensets,1}'$ and 
$\uidescriptionfunction{\numberoffixedopensets,2}-\uidescriptionfunction{\numberoffixedopensets,2}'$
are regular at $p$ if $v'=0$.
\end{proof}

Note that in this lemma, we do not claim (and this is not true in general) that 
if $(\uidescriptionfunction{1,1}', \uidescriptionfunction{1,2}',\ldots,
\uidescriptionfunction{\numberoffixedopensets,1}',\uidescriptionfunction{\numberoffixedopensets,2}',\stdvectorfiledpletter')$ is any representative of 
the coset in $\Gamma(W,\gi)/\Gamma(W_p,\gi)$ defined by three 
Laurent polynomials in lemma, then, for example, the difference between the first 
of these Laurent polynomials and $\uidescriptionfunction{\numberoffixedopensets,1}'$ is regular at $p$. We only claim that 
this is true if $\stdvectorfiledpletter'=0$ and the third Laurent polynomial is also $0$, and we 
also claim that independently of $\stdvectorfiledpletter'$, there always exists such a representative in the coset.

Using Lemmas \ref{wwpquotient} and \ref{wwplinearquotient}, we identify 
the direct sum
$$
\bigoplus_{i=1}^\numberofdivisorpoints\Gamma(W,\gi)/\Gamma(W_{p_i},\gi)
$$
with the space 
$$
\bigoplus_{i=1}^\numberofdivisorpoints\abstractvectorspace_{0,0,p_i}
$$
of $3\numberofdivisorpoints$-tuples of Laurent polynomials of a certain form, 
where the first three polynomials correspond to $p_1$, the second three polynomials correspond
to $p_2$, etc.
\begin{lemma}\label{vfieldreduction}
Let 
$$
(\uidescriptionfunctiondiff{1}{1},\uidescriptionfunctiondiff{1}{2},\stdvectorfiledpletter[1],\ldots,
\uidescriptionfunctiondiff{\numberofdivisorpoints}{1},\uidescriptionfunctiondiff{\numberofdivisorpoints}{2},\stdvectorfiledpletter[\numberofdivisorpoints])
$$
be an element of 
$\bigoplus_{i=1}^\numberofdivisorpoints\abstractvectorspace_{0,0,p_i}$.
Then there exists another element 
$$
(\uidescriptionfunctiondiff{1}{1}',\uidescriptionfunctiondiff{1}{2}',\stdvectorfiledpletter'[1],\ldots,
\uidescriptionfunctiondiff{\numberofdivisorpoints}{1}',\uidescriptionfunctiondiff{\numberofdivisorpoints}{2}',\stdvectorfiledpletter[\numberofdivisorpoints]')
\in\bigoplus_{i=1}^\numberofdivisorpoints\abstractvectorspace_{0,0,p_i}
$$
such that 
these two elements 
represent the same class in 
$$
\left.\!\left(\bigoplus_{i=1}^\numberofdivisorpoints\Big(\Gamma(W,\gi)/\Gamma(W_{p_i},\gi)\Big)\right)\right/\Gamma(W,\gi),
$$
and $\stdvectorfiledpletter[i]'$ is a vector field regular at $p_i$ for all $i$.
\end{lemma}

\begin{proof}
All $\stdvectorfiledpletter[i]$'s can be written using Laurent polynomials as follows: 
$\stdvectorfiledpletter[i]=(b_{i,-1}t_{p_i}^{-1}+\ldots+b_{i,-k}t_{p_i}^{-k})\partial/\partial t_{p_i}$ or 
$\stdvectorfiledpletter[i]=(b_{i,0}+b_{i,-1}t_{p_i}^{-1}+\ldots+b_{i,-k}t_{p_i}^{-k})\partial/\partial t_{p_i}$, 
the exact form depends on whether 
$p_i$ is a removable special point or an essential special point.
Denote $\stdvectorfiledpletter[i]''=(b_{i,-1}t_{p_i}^{-1}+\ldots+b_{i,-k}t_{p_i}^{-k})\partial/\partial t_{p_i}$ 
(if $p_i$ is removable,
then $\stdvectorfiledpletter[i]=\stdvectorfiledpletter[i]''$).
This 
vector field is regular at all points of $\PP^1$ except $p_i$ (including the point $t_{p_i}=\infty$, 
where it has a zero of of order 3). Then $\stdvectorfiledpletter''=\stdvectorfiledpletter[1]''+\ldots+\stdvectorfiledpletter[\numberofdivisorpoints]''\in\Gamma(W,\Theta_{\PP^1})$, 
and we can construct an element of $\Gamma(W,\gi)$ similarly to what we did in 
previous proofs: we set $\uidescriptionfunction{\numberoffixedopensets,1}''=\uidescriptionfunction{\numberoffixedopensets,2}''=0$, and 
$$
\left(
\begin{array}{c}
\uidescriptionfunction{i,1}''\\
\uidescriptionfunction{i,2}''\\
\stdvectorfiledpletter''
\end{array}
\right)
=
\uilargetransition{\numberoffixedopensets,i}
\left(
\begin{array}{c}
\uidescriptionfunction{\numberoffixedopensets,1}''\\
\uidescriptionfunction{\numberoffixedopensets,2}''\\
\stdvectorfiledpletter''
\end{array}
\right).
$$
By Lemma \ref{rijregular},
all entries in $\uilargetransition{\numberoffixedopensets,i}$ are regular at ordinary points, and 
$$
(\uidescriptionfunction{1,1}'',\uidescriptionfunction{1,2}'', \ldots,
\uidescriptionfunction{\numberoffixedopensets,1}'',\uidescriptionfunction{\numberoffixedopensets,2}'',\stdvectorfiledpletter'')\in\Gamma(W,\gi).
$$
Now, by Lemmas \ref{wwpquotient} and \ref{wwplinearquotient}, this section defines elements of 
$\abstractvectorspace_{0,0,p_i}$
of the form $(\uidescriptionfunctiondiff{i}{1}''', \uidescriptionfunctiondiff{i}{2}''', \stdvectorfiledpletter[i]''')$, 
where $\stdvectorfiledpletter[i]'''-\stdvectorfiledpletter''$ is regular at $p_i$. Recall that $\stdvectorfiledpletter[j]''$ is regular at $p_i$ if $i\ne j$, so 
$\stdvectorfiledpletter[i]'''-\stdvectorfiledpletter[i]''$ is regular at $p_i$ as well. 
Also, $\stdvectorfiledpletter[i]''-\stdvectorfiledpletter[i]$ is regular at $p_i$, so $\stdvectorfiledpletter[i]'''-\stdvectorfiledpletter[i]$ is 
regular at $p$. Finally, we set 
$$
\uidescriptionfunctiondiff{i}{1}'=\uidescriptionfunctiondiff{i}{1}-\uidescriptionfunctiondiff{i}{1}''',\quad
\uidescriptionfunctiondiff{i}{2}'=\uidescriptionfunctiondiff{i}{2}-\uidescriptionfunctiondiff{i}{2}''',\quad\text{and}\quad
\stdvectorfiledpletter[i]'=\stdvectorfiledpletter[i]-\stdvectorfiledpletter[i]'''.
$$
The sequence $(\uidescriptionfunctiondiff{1}{1}''',\uidescriptionfunctiondiff{1}{2}''', \stdvectorfiledpletter[1]''',
\ldots,\uidescriptionfunctiondiff{\numberofdivisorpoints}{1}''', \uidescriptionfunctiondiff{\numberofdivisorpoints}{2}''', \stdvectorfiledpletter[\numberofdivisorpoints]''')$
defines
an element of the zero coset in 
$$
\left.\!\left(\bigoplus_{i=1}^\numberofdivisorpoints\Big(\Gamma(W,\gi)/\Gamma(W_{p_i},\gi)\Big)\right)\right/\Gamma(W,\gi)
$$
by construction. Therefore, $(\uidescriptionfunctiondiff{1}{1}',\uidescriptionfunctiondiff{1}{2}', \stdvectorfiledpletter[1],\ldots,
\uidescriptionfunctiondiff{\numberofdivisorpoints}{1}', \uidescriptionfunctiondiff{\numberofdivisorpoints}{1}', \stdvectorfiledpletter[\numberofdivisorpoints]')$ 
defines the same coset as
$(\uidescriptionfunctiondiff{1}{1},\uidescriptionfunctiondiff{1}{2},\stdvectorfiledpletter[1],\ldots,
\uidescriptionfunctiondiff{\numberofdivisorpoints}{1},\uidescriptionfunctiondiff{\numberofdivisorpoints}{2},\stdvectorfiledpletter[\numberofdivisorpoints])$ in 
$$
\left.\!\left(\bigoplus_{i=1}^\numberofdivisorpoints\Big(\Gamma(W,\gi)/\Gamma(W_{p_i},\gi)\Big)\right)\right/\Gamma(W,\gi).
$$
\end{proof}

\begin{lemma}\label{vfieldfinitereduction}
Suppose that
$$
(\uidescriptionfunctiondiff{1}{1},\uidescriptionfunctiondiff{1}{2},\stdvectorfiledpletter[1],\ldots,
\uidescriptionfunctiondiff{\numberofdivisorpoints}{1},\uidescriptionfunctiondiff{\numberofdivisorpoints}{2},\stdvectorfiledpletter[\numberofdivisorpoints])
\in\bigoplus_{i=1}^\numberofdivisorpoints
\abstractvectorspace_{0,0,p_i}
$$
and
$$
(\uidescriptionfunctiondiff{1}{1}',\uidescriptionfunctiondiff{1}{2}',\stdvectorfiledpletter[1]',\ldots,
\uidescriptionfunctiondiff{\numberofdivisorpoints}{1}',\uidescriptionfunctiondiff{\numberofdivisorpoints}{2}',\stdvectorfiledpletter[\numberofdivisorpoints]')
\in\bigoplus_{i=1}^\numberofdivisorpoints
\abstractvectorspace_{0,0,p_i}
$$
define the same class in
$(\bigoplus_{i=1}^\numberofdivisorpoints\Gamma(W,\gi)/\Gamma(W_{p_i},\gi))/\Gamma(W,\gi)$, 
and for every $i$, $\stdvectorfiledpletter[i]$ and $\stdvectorfiledpletter[i]'$ are regular at $p_i$.
Then there exists a globally defined vector field $\stdvectorfiledpletter\in\Gamma(\PP^1,\Theta_{\PP^1})$ such that 
$\stdvectorfiledpletter(p_i)=\stdvectorfiledpletter[i](p_i)-\stdvectorfiledpletter[i]'(p_i)$ if 
$p_i$ is an essential special point.

And vice versa, if 
$$
(\uidescriptionfunctiondiff{1}{1},\uidescriptionfunctiondiff{1}{2},\stdvectorfiledpletter[1],\ldots,
\uidescriptionfunctiondiff{\numberofdivisorpoints}{1},\uidescriptionfunctiondiff{\numberofdivisorpoints}{2},\stdvectorfiledpletter[\numberofdivisorpoints])
\in\bigoplus_{i=1}^\numberofdivisorpoints
\abstractvectorspace_{0,0,p_i}
$$
is such that every $\stdvectorfiledpletter[i]$ is regular at $p_i$, and $\stdvectorfiledpletter\in\Gamma(\PP^1,\Theta_{\PP^1})$ is a 
globally defined vector field, then there exists
$$
(\uidescriptionfunctiondiff{1}{1}',\uidescriptionfunctiondiff{1}{2}',\stdvectorfiledpletter[1]',\ldots,
\uidescriptionfunctiondiff{\numberofdivisorpoints}{1}',\uidescriptionfunctiondiff{\numberofdivisorpoints}{2}',\stdvectorfiledpletter[\numberofdivisorpoints]')
\in\bigoplus_{i=1}^\numberofdivisorpoints
\abstractvectorspace_{0,0,p_i}
$$
equivalent to $(\uidescriptionfunctiondiff{1}{1},\uidescriptionfunctiondiff{1}{2},\stdvectorfiledpletter[1],\ldots,
\uidescriptionfunctiondiff{\numberofdivisorpoints}{1},\uidescriptionfunctiondiff{\numberofdivisorpoints}{2},\stdvectorfiledpletter[\numberofdivisorpoints])$ in 
$$
\left.\!\left(\bigoplus_{i=1}^\numberofdivisorpoints\Big(\Gamma(W,\gi)/\Gamma(W_{p_i},\gi)\Big)\right)\right/\Gamma(W,\gi)
$$ and such that 
$\stdvectorfiledpletter[i]'$ is regular at $p_i$ for every $i$. 
Here $\stdvectorfiledpletter[i](p_i)-\stdvectorfiledpletter[i]'(p_i)=\stdvectorfiledpletter(p_i)$ for all $i$ 
such that 
$p_i$ is an essential special point.
\end{lemma}

\begin{proof}
The first statement follows easily from Lemmas \ref{wwpquotient} and \ref{wwplinearquotient}.
Namely, all triples 
$(\uidescriptionfunctiondiff{i}{1}-\uidescriptionfunctiondiff{i}{1}',
\uidescriptionfunctiondiff{i}{2}-\uidescriptionfunctiondiff{i}{2}',
\stdvectorfiledpletter[i]-\stdvectorfiledpletter[i]')$
define
the same section from $\Gamma(W,\gi)$ in the sense of Lemmas \ref{wwpquotient} and 
\ref{wwplinearquotient} applied at $p_i$. This element of $\Gamma(W,\gi)$ can be written as
$(\uidescriptionfunction{1,1}'', \uidescriptionfunction{1,2}'',\ldots, 
\uidescriptionfunction{\numberoffixedopensets,1}'',\uidescriptionfunction{\numberoffixedopensets,2}'',\stdvectorfiledpletter)$. 
Let us prove that $\stdvectorfiledpletter$ is the desired vector field.
We know that $\stdvectorfiledpletter$ is defined at all ordinary points. If $p_i$ is a removable special point,
then by Lemma \ref{wwplinearquotient}, 
$\stdvectorfiledpletter[i]-\stdvectorfiledpletter[i]'-\stdvectorfiledpletter$ is regular at $p_i$, 
but we already know that $\stdvectorfiledpletter[i]-\stdvectorfiledpletter[i]'$ is regular at $p_i$, 
so $\stdvectorfiledpletter$ is regular at $p_i$. If $p_i$ is an essential special point,
then by Lemma \ref{wwpquotient}, $\stdvectorfiledpletter[i]-\stdvectorfiledpletter[i]'-\stdvectorfiledpletter$ 
is defined at $p_i$ and equals 0 there.
Hence, $\stdvectorfiledpletter$ is defined at $p_i$, and $\stdvectorfiledpletter[i](p)-\stdvectorfiledpletter[i]'(p)=\stdvectorfiledpletter(p)$.

The proof of the second statement is similar to the proof of the previous lemma. Namely, 
we start with $\uidescriptionfunction{\numberoffixedopensets,1}''=\uidescriptionfunction{\numberoffixedopensets,2}''=0$ and construct a section
$(\uidescriptionfunction{1,1}'',\uidescriptionfunction{1,2}'', \ldots,\uidescriptionfunction{\numberoffixedopensets,1}'',
\uidescriptionfunction{\numberoffixedopensets,2}'',\stdvectorfiledpletter)\in\Gamma(W,\gi)$ via
$$
\left(
\begin{array}{c}
\uidescriptionfunction{i,1}''\\
\uidescriptionfunction{i,2}''\\
\stdvectorfiledpletter
\end{array}
\right)
=
\uilargetransition{\numberoffixedopensets,i}
\left(
\begin{array}{c}
\uidescriptionfunction{\numberoffixedopensets,1}''\\
\uidescriptionfunction{\numberoffixedopensets,2}''\\
\stdvectorfiledpletter
\end{array}
\right).
$$
This section defines elements of $\Gamma(W,\gi)/\Gamma(W_{p_i},\gi)$, 
and the isomorphisms from Lemmas \ref{wwpquotient} and \ref{wwplinearquotient}
map them to
$(\uidescriptionfunctiondiff{i}{1}''',\uidescriptionfunctiondiff{i}{2}''',\stdvectorfiledpletter[i]''')$ .
Both Lemmas say that $\stdvectorfiledpletter[i]'''-\stdvectorfiledpletter$ is defined at $p_i$, and, since $\stdvectorfiledpletter$ is defined globally, 
$\stdvectorfiledpletter[i]'''$ is defined at $p_i$. So we can set
$\uidescriptionfunction{i}{1}'=\uidescriptionfunction{i}{1}+\uidescriptionfunction{i}{1}'''$,
$\uidescriptionfunction{i}{2}'=\uidescriptionfunction{i}{2}+\uidescriptionfunction{i}{2}'''$, and 
$\stdvectorfiledpletter[i]'=\stdvectorfiledpletter[i]+\stdvectorfiledpletter[i]'''$.
If $p_i$ is an essential special point,
Lemma \ref{wwpquotient} says that $\stdvectorfiledpletter[i]'''(p_i)=\stdvectorfiledpletter(p_i)$, 
so $\stdvectorfiledpletter[i]'(p_i)-\stdvectorfiledpletter[i](p_i)=\stdvectorfiledpletter(p_i)$.
\end{proof}

\begin{lemma}\label{functreduction}
Let 
$$
(\uidescriptionfunctiondiff{1}{1},\uidescriptionfunctiondiff{1}{2},\stdvectorfiledpletter[1],\ldots,
\uidescriptionfunctiondiff{\numberofdivisorpoints}{1},\uidescriptionfunctiondiff{\numberofdivisorpoints}{2},\stdvectorfiledpletter[\numberofdivisorpoints])
\in\bigoplus_{i=1}^\numberofdivisorpoints
\abstractvectorspace_{0,0,p_i}
$$
Then this element of 
$\bigoplus_{i=1}^\numberofdivisorpoints
\abstractvectorspace_{0,0,p_i}$
and
$$
(0,0,\stdvectorfiledpletter[1],\ldots,0,0,\stdvectorfiledpletter[\numberofdivisorpoints])
$$
define the same class in 
$$
\left.!\left(\bigoplus_{i=1}^\numberofdivisorpoints\Big(\Gamma(W,\gi)/\Gamma(W_{p_i},\gi)\Big)\right)\right/\Gamma(W,\gi).
$$
\end{lemma}

\begin{proof}
The proof is similar to the proof of Lemma \ref{vfieldreduction}.
Since all $\uidescriptionfunctiondiff{i}{j}$ here are Laurent polynomials of the form
$a_{i,j,-1}t_{p_i}^{-1}+\ldots+a_{i,j,-n}t_{p_i}^{-n}$ (we do not mean here that $a_{i,j,-n}\ne 0$, so 
we can use the same $n$ for all polynomials),
they have no poles except $p_i$, and functions
$\uidescriptionfunction{\numberoffixedopensets,1}'=\uidescriptionfunctiondiff{1}{1}+\ldots+\uidescriptionfunctiondiff{\numberofdivisorpoints}{1}$ and 
$\uidescriptionfunction{\numberoffixedopensets,2}'=\uidescriptionfunctiondiff{1}{2}+\ldots+\uidescriptionfunctiondiff{\numberofdivisorpoints}{2}$ have no poles at 
ordinary points. Using these functions, we can construct a section 
$(\uidescriptionfunction{1,1}', \uidescriptionfunction{1,2}', \ldots, 
\uidescriptionfunction{\numberoffixedopensets,1}', \uidescriptionfunction{\numberoffixedopensets,2}', 0)
\in\Gamma(W,\gi)$ as in proofs of previous lemmas, 
namely
$$
\left(
\begin{array}{c}
\uidescriptionfunction{i,1}'\\
\uidescriptionfunction{i,2}'\\
0
\end{array}
\right)
=
\uilargetransition{\numberoffixedopensets,i}
\left(
\begin{array}{c}
\uidescriptionfunction{\numberoffixedopensets,1}'\\
\uidescriptionfunction{\numberoffixedopensets,2}'\\
0
\end{array}
\right),
$$
or, in other words,
$$
\left(
\begin{array}{c}
\uidescriptionfunction{i,1}'\\
\uidescriptionfunction{i,2}'\\
\end{array}
\right)=
\uismalltransition{\numberoffixedopensets,i}
\left(
\begin{array}{c}
\uidescriptionfunction{\numberoffixedopensets,1}'\\
\uidescriptionfunction{\numberoffixedopensets,2}'\\
\end{array}
\right).
$$
Since all entries in $\uismalltransition{\numberoffixedopensets,i}$ are constants, all functions $\uidescriptionfunction{i,j}'$ are defined on $W$, 
and they define an element of $\Gamma(W,\gi)$.

A function $\uidescriptionfunctiondiff{i}{1}$ or $\uidescriptionfunctiondiff{i}{2}$ that has pole at $p_j$ only if $i=j$. Hence, 
the class of $(\uidescriptionfunction{1,1}', \uidescriptionfunction{1,2}', \ldots, 
\uidescriptionfunction{\numberoffixedopensets,1}', \uidescriptionfunction{\numberoffixedopensets,2}', 
0)$ in $\Gamma(W,\gi)/\Gamma(W_{p_i},\gi)$
is mapped by the isomorphism from Lemma \ref{wwpquotient} or \ref{wwplinearquotient} to
$(\uidescriptionfunctiondiff{i}{1},\uidescriptionfunctiondiff{i}{2},0)$. 
Therefore, $(\uidescriptionfunctiondiff{1}{1},\uidescriptionfunctiondiff{1}{2},0,\ldots, 
\uidescriptionfunctiondiff{\numberofdivisorpoints}{1},\uidescriptionfunctiondiff{\numberofdivisorpoints}{2},0)$ 
defines
the zero coset in 
$$
\left.\!\left(\bigoplus_{i=1}^\numberofdivisorpoints\Big(\Gamma(W,\gi)/\Gamma(W_{p_i},\gi)\Big)\right)\right/\Gamma(W,\gi),
$$
and $(\uidescriptionfunctiondiff{1}{1},\uidescriptionfunctiondiff{1}{2},\stdvectorfiledpletter[1],\ldots, 
\uidescriptionfunctiondiff{\numberofdivisorpoints}{1},\uidescriptionfunctiondiff{\numberofdivisorpoints}{2},\stdvectorfiledpletter[\numberofdivisorpoints])$ 
and $(0,0,\stdvectorfiledpletter[1],\ldots,0,0,\stdvectorfiledpletter[\numberofdivisorpoints])$ define the same coset in 
$$
\left.\!\left(\bigoplus_{i=1}^\numberofdivisorpoints\Big(\Gamma(W,\gi)/\Gamma(W_{p_i},\gi)\Big)\right)\right/\Gamma(W,\gi).
$$
\end{proof}

Denote now by $\numberofessentialdivisorpoints$ the number of essential special points.
Denote these special points by $p'_1, \ldots, p'_{\numberofessentialdivisorpoints}$.
\begin{lemma}\label{p1vfields}
If $\numberofessentialdivisorpoints\ge 3$, then every globally defined vector field on $\PP^1$ is uniquely determined by 
its values at $p'_1,\ldots,p'_{\numberofessentialdivisorpoints}$. If $\numberofessentialdivisorpoints\le 3$, then for every set of tangent vectors at 
$p'_1, \ldots, p'_{\numberofessentialdivisorpoints}$ there exists a globally defined vector field that takes these values 
at these points.
\end{lemma}

\begin{proof}
Every globally defined vector field on $\PP^1$ can be written as $(a_0+a_1 t+a_2 t^2)\partial/\partial t$.
(If the polynomial here is of higher degree, the vector field has a pole at infinity.) A polynomial of 
degree 2 is completely determined by its values at at least three points (if there are more than three points, 
these values cannot be arbitrary, but a polynomial of degree two is still unique if it exists). A polynomial
of degree 2 can take arbitrary prescribed values at at most three points.
\end{proof}

\begin{proposition}\label{honepushforwardnumber}
If $\numberofessentialdivisorpoints\le 3$, then $H^1(\PP^1,\gi)=0$.

If $\numberofessentialdivisorpoints\ge 3$, then there exists a vector space $\abstractvectorspace_{0,1}$ of dimension $\numberofessentialdivisorpoints$ and 
an embedding $\Gamma(\PP^1, \Theta_{\PP^1})\hookrightarrow \abstractvectorspace_{0,1}$ such 
that $H^1(\PP^1,\gi)\cong \abstractvectorspace_{0,1}/\Gamma(\PP^1, \Theta_{\PP^1})$. Therefore, 
$\dim H^1(\PP^1,\gi)=\numberofessentialdivisorpoints-3$ in this case.
\end{proposition}
\begin{proof}
By applying first Lemma \ref{vfieldreduction}, and then Lemma \ref{functreduction} to an element 
of $\bigoplus_{i=1}^\numberofdivisorpoints\abstractvectorspace_{0,0,p_i}$, we can get 
another element of $\bigoplus_{i=1}^\numberofdivisorpoints\abstractvectorspace_{0,0,p_i}$ of 
the form
$(0,0,\stdvectorfiledpletter[1],\ldots,0,0,\stdvectorfiledpletter[\numberofdivisorpoints])$ 
equivalent to the original element of $\bigoplus_{i=1}^\numberofdivisorpoints\abstractvectorspace_{0,0,p_i}$
in $(\bigoplus_{i=1}^\numberofdivisorpoints\Gamma(W,\gi)/\Gamma(W_{p_i},\gi))/\Gamma(W,\gi)$.
Here, $\stdvectorfiledpletter[i]$ are Laurent polynomials regular at $p_i$, i.~e. they don't have non-zero coefficients 
at negative degrees. But Lemmas \ref{wwpquotient} and \ref{wwplinearquotient}
describe exact form of these polynomials, and the highest possible degree of a non-zero term is 
0 if $p_i$ is an essential special point, and $-1$ if it is removable.
We conclude that 
if $p_i$ is a removable special point,
then 
$\stdvectorfiledpletter[i]=0$. Otherwise, $\stdvectorfiledpletter[i]$ is a vector field of the form $a\partial/\partial t_{p_i}$ ($a\in\CC$), 
which is completely determined by its value at $p_i$.

Therefore, we have constructed a surjective linear map from
$$
\abstractvectorspace_{0,1}=\bigoplus_{i=1}^{\numberofessentialdivisorpoints}\Theta_{\PP^1,p'_i}
$$
to $H^1(\PP^1,\gi)$. Denote this map by $\zeta_1$. 
$\Gamma(\PP^1, \Theta_{\PP^1})$ can be mapped to
$\abstractvectorspace_{0,1}$ via evaluation of a vector field at   points $p'_1,\ldots, p'_{\numberofessentialdivisorpoints}$.
Denote this map by $\zeta_2$.
Let us prove that $\ker\zeta_1=\zeta_2(\Gamma(\PP^1, \Theta_{\PP^1}))$.
First, if $\stdvectorfiledpletter$ is a globally defined vector field, by the second part of Lemma \ref{vfieldfinitereduction},
there exists
$(\uidescriptionfunctiondiff{1}{1}',\uidescriptionfunctiondiff{1}{2}',\stdvectorfiledpletter[1]',\ldots, 
\uidescriptionfunctiondiff{\numberofdivisorpoints}{1}',\uidescriptionfunctiondiff{\numberofdivisorpoints}{2}',\stdvectorfiledpletter[\numberofdivisorpoints]')
\in\bigoplus_{i=1}^\numberofdivisorpoints\abstractvectorspace_{0,0,p_i}$
equivalent to 0 
in $(\bigoplus_{i=1}^\numberofdivisorpoints\Gamma(W,\gi)/\Gamma(W_{p_i},\gi))/\Gamma(W,\gi)$
and such that $\stdvectorfiledpletter[i]'$ is defined at $p_i$ and $\stdvectorfiledpletter[i]'(p_i)=\stdvectorfiledpletter(p_i)$ for all essential special points $p_i$.
As we have already seen, $\stdvectorfiledpletter[i]'=0$ if $p_i$ is removable.
By Lemma \ref{functreduction}, 
$(\uidescriptionfunctiondiff{1}{1}',\uidescriptionfunctiondiff{1}{2}',\stdvectorfiledpletter[1],\ldots, 
\uidescriptionfunctiondiff{\numberofdivisorpoints}{1}',\uidescriptionfunctiondiff{\numberofdivisorpoints}{2}',\stdvectorfiledpletter[\numberofdivisorpoints]')$ is equivalent to 
$(0,0,\stdvectorfiledpletter[1]',\ldots, 0,0,\stdvectorfiledpletter[\numberofdivisorpoints]')$, so $\zeta_2(\Gamma(\PP^1, \Theta_{\PP^1}))\subseteq\ker\zeta_1$.
On the other hand, if $(0,0,\stdvectorfiledpletter[1],\ldots,0,0,\stdvectorfiledpletter[\numberofdivisorpoints])\in\ker\zeta_1$, 
then by the first part of Lemma \ref{vfieldfinitereduction},
there exists a vector field $\stdvectorfiledpletter\in\Gamma(\PP^1, \Theta_{\PP^1})$ 
such that $\stdvectorfiledpletter[i](p_i)=\stdvectorfiledpletter(p_i)$ for all essential special points $p_i$.
This means that $\ker\zeta_1\subseteq\zeta_2(\Gamma(\PP^1, \Theta_{\PP^1}))$, and we finally conclude that 
$\ker\zeta_1=\zeta_2(\Gamma(\PP^1, \Theta_{\PP^1}))$.

Now, by Lemma \ref{p1vfields}, $\zeta_2$ is surjective if $\numberofessentialdivisorpoints\le 3$, and $\zeta_2$ is injective if $\numberofessentialdivisorpoints\ge 3$, 
and the claim follows.
\end{proof}

\section{Computation of the dimension of $\ker H^0(\PP^1,\giv)\to H^0(\PP^1,\gviii)$}

Now we continue with $\ker H^0(\PP^1,\giv)\to H^0(\PP^1,\gviii)$. Recall that 
we use the sufficient system $U_1,\ldots, U_{\numberoffixedopensets-1}$ to compute $\giv$ and $\gviii$ and that 
$\giv$ (resp. $\gviii$) is the first cohomology of the complex $\gii\to \giip\to\giipp$ 
(resp. $\gvi\to \gvip\to\gvipp$). The map between $\giv$ and $\gviii$
can also be written as the cohomology in the middle of a map between these two complexes.
Here $\gii$ can be written as a direct sum of 
sheaves, each of them corresponds to an 
open subset $U_i$,
namely, its sections over an open set $V\subseteq \PP^1$ are 
the $U_i$-descriptions of 
homogeneous vector fields of degree 0 defined on $\pi^{-1}(V)\cap U_i$. 
Denote 
this direct summand by
$\giisi{i}$. 
The sheaves $\giip$ and 
$\giipp$ can be decomposed into direct sums similarly, and each direct summand corresponds to two or three 
of the sets $U_i$, respectively.
Denote these direct summands by $\giipsij{i,j}$ 
and by
$\giippsijk{i,j,k}$, respectively.
Similarly, we can define decompositions $\giiinv=\bigoplus \giisiinv{i}$, 
$\giipinv=\bigoplus \giipsijinv{i,j}$, $\giippinv=\bigoplus \giippsijkinv{i,j,k}$.

$\gvi$, $\gvip$, and $\gvipp$ can also be decomposed into 
sums of direct summands corresponding to one, two, or three sets $U_i$, respectively. The sections of 
each of these summands over an open subset $V\subseteq \PP^1$ are sequences of length 
$\numberoffixedgenerators=\sum_j \dim \Gamma(\PP^1,\OO(\mathcal D(\lambda_j)))$, where each entry is the 
$U_i$-description of a function of degree $\lambda_j$ defined on 
the intersection of $\pi^{-1}(V)$ and one, two, or three of the sets $U_i$, respectively.
Denote these direct summands by $\gvisi{i}$, 
$\gvipsij{i,j}$, $\gvippsijk{i,j,k}$, respectively.
Again, we can also decompose $\gviinv=\bigoplus \gvisiinv{i}$, 
$\gvipinv=\bigoplus \gvipsijinv{i,j}$, $\gvippinv=\bigoplus \gvippsijkinv{i,j,k}$.

The maps $\gii\to\gvi$, $\giip\to\gvip$, $\giipp\to\gvipp$
map each of these direct summands in $\gii$, $\giip$, $\giipp$ to the 
corresponding direct summand in $\gvi$, $\gvip$, $\gvipp$, respectively.

Our next goal is to simplify the expressions for $\giv$ and $\gviii$ we have now.
For this goal, it will be more convenient to deal with the "invariant" versions of the sheaves, i.~e. with 
$\giiinv$, $\giisiinv{i}$, $\giipsijinv{i,j}$, $\giippsijkinv{i,j,k}$, 
$\gviinv$, $\gvisiinv{i}$, $\gvipsijinv{i,j}$, and $\gvippsijkinv{i,j,k}$, 
which do not involve any $U_i$-descriptions explicitly.
By Lemma \ref{uijstructure}, $U_{\numberoffixedopensets}\subseteq U_i$ is a dense open subset for all $i$.
Hence, each of the 
sheaves $\giisiinv{i}$, $\giipsijinv{i,j}$, and $\giippsijkinv{i,j,k}$
can be embedded into
the following sheaf that we denote by $\giicircinv$: $\Gamma (V, \giicircinv)$
is the space of $T$-invariant vector fields on $\pi^{-1}(V)\cap U_{\numberoffixedopensets}$.
Similarly, each of 
sheaves $\gvisiinv{i}$, $\gvipsijinv{i,j}$, and $\gvippsijkinv{i,j,k}$
can be embedded into the following sheaf $\gvicircinv$: $\Gamma (V, \gvicircinv)$
is the space of sequences of length $\sum_j \dim \Gamma(\PP^1,\OO(\mathcal D(\lambda_j)))$ of functions of degree $\lambda_j$ 
defined on $\pi^{-1}(V)\cap U_{\numberoffixedopensets}$. 
Then by 
Corollary \ref{h1computegen}
we have the following formulas for $\givinv$ and $\gviiiinv$:
$$
\givinv=\left.\!\left(\ker\Big(\bigoplus_{i=1}^{\numberoffixedopensets-1}(\giicircinv/\giisiinv{i})\to
\bigoplus_{1\le i<j\le \numberoffixedopensets-1}(\giicircinv/\giipsijinv{i,j})\Big)\right)\right/\giicircinv,
$$
$$
\gviiiinv=\left.\!\left(\ker\Big(\bigoplus_{i=1}^{\numberoffixedopensets-1}(\gvicircinv/\gvisiinv{i})\to
\bigoplus_{1\le i<j\le \numberoffixedopensets-1}(\gvicircinv/\gvipsijinv{i,j})\Big)\right)\right/\gvicircinv.
$$
And again, the map $\givinv\to \gviiiinv$ maps each direct summand of $\givinv$ in 
this formula to the corresponding direct summand of $\gviiiinv$.
Note that Corollary \ref{h1compute} cannot be applied here directly because it is not always true that 
$\gvicircinv=\gvipsijinv{i,j}$. However, we can prove the following two lemmas.
Recall that by Lemma \ref{uistructure}, 
$U_i$ is isomorphic to $V_i\times (\CC\setminus 0)\times L$, where $L$ is isomorphic to $\CC$ or $\CC\setminus 0$.
\begin{lemma}\label{restrictisomorphism}
Let $V'_i\subseteq V_i$ be an open subset. The space of
$T$-invariant vector fields defined on 
$V'_i\times (\CC\setminus 0)\times L$ and on $V'_i\times (\CC\setminus 0)\times (\CC\setminus 0)$
coincide, in other words, the restriction homomorphism from the space of $T$-invariant vector fields
on $V'_i\times (\CC\setminus 0)\times L$ to the space of $T$-invariant vector fields 
on $V'_i\times (\CC\setminus 0)\times (\CC\setminus 0)$ is in fact an isomorphism.
This is also true for functions of degree $\chi$ instead of vector fields of degree 0, 
if $\chi\in\sck\cap M$.
\end{lemma}
\begin{proof}
The claim for vector fields follows directly from Corollary \ref{uidescvfield}, namely, the 
description of the space of vector fields there does not depend on whether $L'=\CC$ or $L'=\CC\setminus0$
(in terms of the notation used in Corollary \ref{uidescvfield}). For functions 
of degree $\chi$, Lemma \ref{uidescfunct} gives the same description for 
$L'=\CC$ and for $L'=\CC\setminus0$, if $\chi\in\sck\cap M$
\end{proof}

\begin{lemma}\label{removeexcessive}
The embeddings $\giipsijinv{i,j}\to\giicircinv$ and $\gvipsijinv{i,j}\to\gvicircinv$
are isomorphisms for $1\le i<j\le \numberoffixedopensets-1$, except for the following case: both indices $i$ and 
$j$ correspond to the same \textbf{removable} special point $p$.
In this case, the embeddings $\giisiinv{i}\to \giipsijinv{i,j}$, $\giisiinv{j}\to \giipsijinv{i,j}$,
$\gvisiinv{i}\to \gvipsijinv{i,j}$, $\gvisiinv{j}\to \gvipsijinv{i,j}$
are isomorphisms.
\end{lemma}
\begin{proof}
If $U_i$ and $U_j$ correspond to different special points, then $V_i\cap V_j=W$, and by Lemma \ref{uijstructure}, 
$U_i\cap U_j=W\times (\CC\setminus 0)\times L$, where $L$ is isomorphic to $\CC$ or $\CC\setminus 0$.
If $U_i$ and $U_j$ correspond to the same essential special point $p$,
then they must correspond to 
the normal subcones of different vertices of $\stdpolyhedronletter_p$,
so 
Lemma \ref{uijstructure} says that $U_i\cap U_j$ is isomorphic to $W\times (\CC\setminus 0)\times L$
again. If $L=\CC\setminus0$, then $U_i\cap U_j\cap U_{\numberoffixedopensets}=W\times (\CC\setminus 0)\times L$ as well, 
and the isomorphism here, as well as in the equality $U_i\cap U_j=W\times (\CC\setminus 0)\times L$, 
is given by the isomorphism defined by Lemma \ref{uistructure} for $U_i$, so $U_i\cap U_j\cap U_{\numberoffixedopensets}=U_i\cap U_j$. 
We already know that $U_{\numberoffixedopensets}\subseteq U_i$, $U_{\numberoffixedopensets}\subseteq U_j$, so $U_{\numberoffixedopensets}=U_i\cap U_j$ if $L=\CC\setminus 0$.
If $L=\CC$, then $U_i\cap U_j=W\times (\CC\setminus 0)\times \CC$ and 
$U_i\cap U_j\cap U_{\numberoffixedopensets}=W\times (\CC\setminus 0)\times (\CC\setminus 0)$, where the isomorphism in 
both equalities is 
given by the isomorphism defined by Lemma \ref{uistructure} for $U_i$.
Let $V\subseteq \PP^1$ be an open subset. 
Now it follows from Lemma \ref{restrictisomorphism} that we 
always have $\Gamma (V,\giipsijinv{i,j})=\Gamma (V,\giicircinv)$ and 
$\Gamma (V,\gvipsijinv{i,j})=\Gamma (V,\gvicircinv)$
if $U_i$ and $U_j$ correspond to different special points or 
$U_i$ and $U_j$ correspond to the same essential special point $p$.

Suppose now that both $U_i$ and $U_j$ correspond to the same removable special point $p$.
Let us prove that the embeddings $\giisiinv{i}\to \giipsijinv{i,j}$ and 
$\gvisiinv{i}\to \gvipsijinv{i,j}$ are isomorphisms, the situation 
for $\giisiinv{j}\to \giipsijinv{i,j}$ and $\gvisiinv{j}\to \gvipsijinv{i,j}$ 
is completely symmetric.
We have
$\uidegree{i,1}, \uidegree{j,1}\in\partial\sck$, but $\uidegree{i,1}\ne \uidegree{j,1}$, so by Lemmas \ref{uistructure} and \ref{uijstructure},
$U_i=V_i\times (\CC\setminus 0)\times \CC$,
$U_i\cap U_j=V_i\times (\CC\setminus 0)\times (\CC\setminus 0)$,
and the isomorphism in the second equality is 
a restriction of the isomorphism in the first equality. 
The claim again follows from Lemma \ref{restrictisomorphism}.
\end{proof}


Since kernels of sheaf maps can be computed on each open subset independently, Lemma \ref{removeexcessive}
implies that 
$$
\ker\left(\bigoplus_{i=1}^{\numberoffixedopensets-1}(\giicircinv/\giisiinv{i})\to
\bigoplus_{1\le i<j\le \numberoffixedopensets-1}(\giicircinv/\giipsijinv{i,j})\right)
$$
can be computed as follows. Its sections over an open subset $V\subseteq \PP^1$ 
are sequences of the form
$(\stdvectorfiledxletter_1,\ldots,\stdvectorfiledxletter_{\numberoffixedopensets-1})
\in\bigoplus_{i=1}^{\numberoffixedopensets-1}\Gamma(V, \giicircinv/\giisiinv{i})$
satisfying the following conditions: if indices $i$ and $j$ correspond to 
the same removable special point $p$,
then 
$\stdvectorfiledxletter_i=\stdvectorfiledxletter_j$. So we can do the following. For each removable special point $p$,
if there are two indices $i$ and $j$ corresponding 
to $p$, choose one of them and call it \textit{excessive}. Then the kernel 
is isomorphic to the following sheaf:
$$
\bigoplus_{\substack{1\le i\le \numberoffixedopensets-1\\ i\text{ is not excessive}}}(\giicircinv/\giisiinv{i}).
$$
Similarly, 
$$
\ker\left(\bigoplus_{i=1}^{\numberoffixedopensets-1}(\gvicircinv/\gvisiinv{i})\to
\bigoplus_{1\le i<j\le \numberoffixedopensets-1}(\gvicircinv/\gvipsijinv{i,j})\right)
\cong\bigoplus_{\substack{1\le i\le \numberoffixedopensets-1\\ i\text{ is not excessive}}}(\gvicircinv/\gvisiinv{i}).
$$
So we get the following formulas for $\givinv$ and $\gviiiinv$:
$$
\givinv\cong
\left.\!\left(\bigoplus_{\substack{1\le i\le \numberoffixedopensets-1\\ i\text{ is not excessive}}}(\giicircinv/\giisiinv{i})\right)\right/\giicircinv,
$$
$$
\gviiiinv\cong
\left.\!\left(\bigoplus_{\substack{1\le i\le \numberoffixedopensets-1\\ i\text{ is not excessive}}}(\gvicircinv/\gvisiinv{i})\right)\right/\gvicircinv.
$$
Sections of quotients of sheaves can only be computed directly on affine subsets. 
To compute the space of global sections on $\PP^1$, we should first compute sections 
for an affine covering of $\PP^1$, then global sections are tuples of local sections 
that coincide on the intersections of the sets from the affine covering. We already have 
an affine covering of $\PP^1$, namely, we have sets $W_p$.
Recall that $W_p\cap W_{p'}=W$ for every pair of special points $p\ne p'$.

\begin{lemma}\label{g4g8vanish}
Let $V\subseteq \PP^1$ be an open subset and $p\in\PP^1$ be a special point such that 
$V\cap W_p=W$. Let an index $i$ correspond to $p$. Then 
$\Gamma (V, \giicircinv)=\Gamma (V, \giisiinv{i})$ and 
$\Gamma (V, \gvicircinv)=\Gamma (V, \gvisiinv{i})$.
\end{lemma}
\begin{proof}
$\Gamma (V, \giisiinv{i})$ (resp. $\Gamma (V, \gvisiinv{i})$) is the space of $T$-invariant 
vector fields (resp. sequences of functions of certain degrees) defined on $\pi^{-1}(V)\cap U_i=
\pi^{-1}(V)\cap(W_p\times (\CC\setminus0)\times L)=(V\cap W_p)\times (\CC\setminus0)\times L=W\times(\CC\setminus0)\times L$, 
where $L=\CC$ or $L=\CC\setminus0$.
By Lemma \ref{restrictisomorphism}, these spaces are isomorphic 
to the spaces of (respectively) vector fields and sequences of functions of certain degrees
defined on $W\times (\CC\setminus0)\times (\CC\setminus0)\subseteq U_i$, where the embedding 
is given by the isomorphism for $U_i$ in Lemma \ref{uistructure}.
On the other hand, 
by Lemma \ref{uijstructure}, $U_i\cap U_{\numberoffixedopensets}$ is also isomorphic to $W\times (\CC\setminus0)\times (\CC\setminus0)$,
and the isomorphism here is also the restriction of the isomorphism in Lemma \ref{uistructure} for 
$U_i$ to $U_i\cap U_{\numberoffixedopensets}$. Therefore, in fact we have proved that the restriction of spaces 
of $T$-invariant vector fields and of functions of the required degrees 
from $\pi^{-1}(V)\cap U_i$ to $U_i\cap U_{\numberoffixedopensets}$ are isomorphisms. 
But $U_i\cap U_{\numberoffixedopensets}=U_{\numberoffixedopensets}=\pi^{-1}(V)\cap U_{\numberoffixedopensets}$, 
and $\Gamma (V, \giicircinv)$ (resp. $\Gamma (V, \gvicircinv)$) is the space of 
$T$-invariant vector fields (resp. of sequences functions of the required degrees) defined 
on $\pi^{-1}(V)\cap U_{\numberoffixedopensets}$.
\end{proof}

\begin{corollary}\label{g4g8directsum}
$\Gamma (W, \givinv)=0$, $\Gamma (W, \gviiiinv)=0$.
\end{corollary}
\begin{proof}
$W\cap W_p=W$ for all special points $p$, so all direct summands of the from $\giicircinv/\giisiinv{i}$
and $\gvicircinv/\gvisiinv{i}$ in the formulas above vanish.
\end{proof}

This corollary enables us to omit the condition that sections of $\givinv$ (or of $\gviiiinv$) over 
different sets $W_p$ should coincide on intersections to form a global section.
Therefore,
$$
\Gamma(\PP^1,\givinv)\cong\bigoplus_{\substack{p\text{ special}\\\text{point}}}\left(\left.\!\Bigg(
\bigoplus_{\substack{1\le i\le \numberoffixedopensets-1\\ i\text{ is not excessive}}}\left(\Gamma(W_p,\giicircinv)/
\Gamma(W_p,\giisiinv{i\vphantom{\opensetforvertex{p,j}}})\right)\Bigg)\right/\Gamma(W_p,\giicircinv)\right),
$$
$$
\Gamma(\PP^1,\gviiiinv)\cong\bigoplus_{\substack{p\text{ special}\\\text{point}}}\left(\left.\!\Bigg(
\bigoplus_{\substack{1\le i\le \numberoffixedopensets-1\\ i\text{ is not excessive}}}\left(\Gamma(W_p,\gvicircinv)/
\Gamma(W_p,\gvisiinv{i\vphantom{\opensetforvertex{p,j}}})\right)\Bigg)\right/\Gamma(W_p,\gvicircinv)\right).
$$

This formulas can be simplified more. Namely, recall that every set $V_i$ equals $W_p$ or $W$.
If $p$ is a special point, and $V_i=W$ or $V_i=W_{p'}$, where $p'\ne p$, then $V_i\cap W_p=W$, 
and by Lemma \ref{g4g8vanish}, $\Gamma(W_p,\giicircinv)/\Gamma(W_p,\giisiinv{i})=0$ and 
$\Gamma(W_p,\gvicircinv)/\Gamma(W_p,\gvisiinv{i})=0$. So, we can write global sections 
of $\givinv$ and of $\gviiiinv$ as follows:
$$
\Gamma(\PP^1,\givinv)\cong\bigoplus_{\substack{p\text{ special}\\\text{point}}}\left(\left.\!\Bigg(
\bigoplus_{\substack{1\le i\le \numberoffixedopensets-1\\ i\text{ is not excessive}\\ V_i=W_p}}\left(\Gamma(W_p,\giicircinv)/
\Gamma(W_p,\giisiinv{i\vphantom{\opensetforvertex{p,j}}})\right)\Bigg)\right/\Gamma(W_p,\giicircinv)\right),
$$
$$
\Gamma(\PP^1,\gviiiinv)\cong\bigoplus_{\substack{p\text{ special}\\\text{point}}}\left(\left.\!\Bigg(
\bigoplus_{\substack{1\le i\le \numberoffixedopensets-1\\ i\text{ is not excessive}\\ V_i=W_p}}\left(\Gamma(W_p,\gvicircinv)/
\Gamma(W_p,\gvisiinv{i\vphantom{\opensetforvertex{p,j}}})\right)\Bigg)\right/\Gamma(W_p,\gvicircinv)\right).
$$
Now each sheaf $\giisiinv{i}$ and $\gvisiinv{i}$ occurs only once in these summations.
Each direct summand in the first direct sum in the formula for $\Gamma(\PP^1,\givinv)$ is mapped to the 
corresponding direct summand of $\Gamma(\PP^1,\gviiiinv)$, so we have proved the following lemma:
\begin{lemma}\label{consideronespecialpoint}
The kernel $\ker(\Gamma(\PP^1,\givinv)\to\Gamma(\PP^1,\gviiiinv))$ is isomorphic to the following direct sum:
\begin{multline*}
\bigoplus_{p\text{ special point}}\ker\left(\left.\!\Bigg(
\bigoplus_{\substack{1\le i\le \numberoffixedopensets-1\\ i\text{ is not excessive}\\ V_i=W_p}}
\left(\Gamma(W_p,\giicircinv)/
\Gamma(W_p,\giisiinv{i\vphantom{\opensetforvertex{p,j}}})\right)\Bigg)\right/\Gamma(W_p,\giicircinv)\right.\!\\
\longrightarrow
\left.\!\left.\!\Bigg(
\bigoplus_{\substack{1\le i\le \numberoffixedopensets-1\\ i\text{ is not excessive}\\ V_i=W_p}}
\left(\Gamma(W_p,\gvicircinv)/
\Gamma(W_p,\gvisiinv{i\vphantom{\opensetforvertex{p,j}}})\right)\Bigg)\right/\Gamma(W_p,\gvicircinv)\right).
\end{multline*}\qed
\end{lemma}

Fix a special point 
$p$.
Recall that we have a coordinate function $t_p$ on $\PP^1$ with the only zero at $p$.
Our next goal is
to compute the kernel
\begin{multline*}
\ker\left(\left.\!\Bigg(
\bigoplus_{\substack{1\le i\le \numberoffixedopensets-1\\ i\text{ is not excessive}\\ V_i=W_p}}
\left(\Gamma(W_p,\giicircinv)/
\Gamma(W_p,\giisiinv{i\vphantom{\opensetforvertex{p,j}}})\right)\Bigg)\right/\Gamma(W_p,\giicircinv)\right.\!\\
\longrightarrow
\left.\!\left.\!\Bigg(
\bigoplus_{\substack{1\le i\le \numberoffixedopensets-1\\ i\text{ is not excessive}\\ V_i=W_p}}\left(\Gamma(W_p,\gvicircinv)/
\Gamma(W_p,\gvisiinv{i\vphantom{\opensetforvertex{p,j}}})\right)\Bigg)\right/\Gamma(W_p,\gvicircinv)\right).
\end{multline*}
Recall first that if $p$ is a removable special point, then there exists only one non-excessive index $i$ corresponding
to $p$.
But then each direct sum in the formula above contains only one summand, and 
$
(\Gamma(W_p,\giicircinv)/
\Gamma(W_p,
\giisiinv{i}
))/\Gamma(W_p,\giicircinv)=0$.
So in the sequel we suppose that $p$ is an \textbf{essential} special point.
Then there are no excessive indices corresponding to $p$.
Moreover, 
in this case 
we chose exactly one set $U_i$ for each pair $(p,j)$, where $1\le j\le \numberofverticespt{p}$.
In other words, these sets $U_i$ (and the summands in each of the direct sums in the formula above)
are in bijection with 
the vertices $\indexedvertexpt{p}{j}$ of $\stdpolyhedronletter_p$.
For each pair $(p,j)$, denote the index $i$ such that $U_i$ corresponds to $(p,j)$ 
by $\opensetforvertex{p,j}$.
So, now we are computing the kernel
\begin{multline*}
\ker\left(\left.\!\Bigg(
\bigoplus_{j=1}^{\numberofverticespt p}
\left(\Gamma(W_p,\giicircinv)/
\Gamma(W_p,\giisiinv{\opensetforvertex{p,j}})\right)
\Bigg)\right/\Gamma(W_p,\giicircinv)
\right.\!\\
\left.\!
\longrightarrow
\left.\!\Bigg(
\bigoplus_{j=1}^{\numberofverticespt p}
\left(\Gamma(W_p,\gvicircinv)/
\Gamma(W_p,\gvisiinv{\opensetforvertex{p,j}})\right)
\Bigg)\right/\Gamma(W_p,\gvicircinv)\right).
\end{multline*}

Fix an index $j$, $1\le j\le \numberofverticespt p$.
Now we come back to using $U_i$-descriptions, namely,
We are going to use 
$U_{\opensetforvertex{p,j}}$-descriptions 
to compute
$$
\Gamma(W_p,\giicircinv)/\Gamma(W_p,\giisiinv{\opensetforvertex{p,j}})
$$
and
$$\Gamma(W_p,\gvicircinv)/
\Gamma(W_p,\gvisiinv{\opensetforvertex{p,j}}).
$$
$\Gamma(W_p,\giisiinv{\opensetforvertex{p,j}})$ is the space of $T$-invariant vector fields defined on $U_{\opensetforvertex{p,j}}$, and, 
by Corollary \ref{uidescvfield}, they are determined by triples of a vector field and two functions 
defined on $W_p$, which form $\Gamma(W_p,\giisi{\opensetforvertex{p,j}})$. We shortly write 
$\gsii{p,j}=\Gamma(W_p,\giisi{\opensetforvertex{p,j}})$.
$\Gamma(W_p,\giicircinv)$ is the space of $T$-invariant vector fields 
defined on 
$U_{\numberoffixedopensets}=\pi^{-1}(W_p)\cap U_{\numberoffixedopensets}$, 
and by Lemma \ref{uijstructure} and by Corollary \ref{uidescvfield},
they are determined by triples of a vector field and two functions defined on $W$. 
Denote this space of 
triples of a vector field and two functions defined on $W$ by $\gsiicirc{p,j}$.
Observe that 
the space itself does not depend 
on $p$ and $j$, but the isomorphism between $\gsiicirc{p,j}$ and $\Gamma(W_p,\giicircinv)$
we use depends on $p$ and $j$.
Denote this isomorphism 
by 
$$
\kappa_{\Theta,p,j}\colon \gsiicirc{p,j}\to \Gamma(W_p,\giicircinv).
$$ 
By Remark \ref{uidescvfieldrestriction},
the embedding $\Gamma(W_p,\giisiinv{\opensetforvertex{p,j}})\to \Gamma(W_p,\giicircinv)$ 
after applying these isomorphisms becomes
the restriction of vector fields and functions from $W_p$ to $W$.

Similarly, $\Gamma(W_p,\gvisi{\opensetforvertex{p,j}})$ is the space of sequences of functions of certain 
degrees from $\sck\cap M$ defined on $U_{\opensetforvertex{p,j}}$. Lemma \ref{uidescfunct} for $i=\opensetforvertex{p,j}$ identifies this space 
with the space of sequences of functions defined on $W_p$ (each function is identified with its $U_{\opensetforvertex{p,j}}$-description), 
denote this space of sequences of functions by $\gsvi{p,j}$.
$\Gamma(W_p,\gvicircinv)$ is the space of sequences of functions of the same degrees, but 
they are defined on 
$U_{\numberoffixedopensets}=\pi^{-1}(W_p)\cap U_{\numberoffixedopensets}$. 
Again, Lemma \ref{uidescfunct} for $i=\opensetforvertex{p,j}$ identifies this space
with the space of sequences of functions defined on $W$ 
(again, each function is identified with its $U_{\opensetforvertex{p,j}}$-description, not with its $U_{\numberoffixedopensets}$-description).
Denote this space of sequences of functions by $\gsvicirc{p,j}$, and
denote this isomorphism between 
$\gsvicirc{p,j}$
and 
$\Gamma(W_p,\gvicircinv)$ by $\kappa_{\OO,p,j}$. And again, despite the spaces
themselves do not depend on $p$ and $j$, the isomorphism is based on $U_{\opensetforvertex{p,j}}$-descriptions
and depends on $p$ and $j$.
By Remark \ref{uidescfunctrestriction},
the embedding $\Gamma(W_p,\gvisi{\opensetforvertex{p,j}})\to\Gamma(W_p,\gvicircinv)$
after these identifications becomes the restriction of functions from $W_p$ to $W$.

Finally, the formula in Lemma \ref{uidescpsi} (for different indices $\opensetforvertex{p,j}$) 
defines morphisms 
$$
\Gamma(W_p,\giicircinv)\to\Gamma(W_p,\gvicircinv).
$$
Denote the corresponding morphisms between 
$\gsiicirc{p,j}$
and 
$\gsvicirc{p,j}$
by $\psi_{p,j}$.
Denote also the map 
$$
\bigoplus_{j=1}^{\numberofverticespt{p}}\gsiicirc{p,j} \to
\bigoplus_{j=1}^{\numberofverticespt{p}}\gsvicirc{p,j}
$$
formed by maps $\psi_{p,j}$ for all $j$ 
($1\le j\le \numberofverticespt p$) by $\psi_p$.
It follows from functoriality of the isomorphism in Proposition \ref{hnasquotient} that
$\psi_p$
induces the morphism in 
question between 
$$
\left.\!\left(
\bigoplus_{j=1}^{\numberofverticespt p}
\left(\Gamma(W_p,\giicircinv)/
\Gamma(W_p,\giisiinv{\opensetforvertex{p,j}})\right)\right)\right/\Gamma(W_p,\giicircinv)
$$
and
$$
\left.\!\left(
\bigoplus_{j=1}^{\numberofverticespt p}\left(\Gamma(W_p,\gvicircinv)/
\Gamma(W_p,\gvisiinv{\opensetforvertex{p,j}})\right)\right)\right/\Gamma(W_p,\gvicircinv).
$$

\begin{lemma}\label{wwiquotient}
Let 
$p$
be a special point, $j$ be an index, $1\le j\le \numberofverticespt p$.

The composition of the restriction of $\kappa_{\Theta,p,j}$ to the space of triples of the form 
$$
(a_{1,-1}t_p^{-1}+\ldots+a_{1,-n_1}t_p^{-n_1},
a_{2,-1}t_p^{-1}+\ldots+a_{2,-n_2}t_p^{-n_2}, 
(b_{-1}t_p^{-1}+\ldots+b_{-n_3}t_p^{-n_3})\partial/\partial t_p)
$$
and the natural projection 
$\Gamma(W_p,\giicircinv)\to\Gamma(W_p,\giicircinv)/\Gamma(W_p,\giisiinv{\opensetforvertex{p,j}})$
is an isomorphism.
\end{lemma}

\begin{proof}
The proof is similar to the proof of Lemma \ref{wwpquotient}. Namely, let $\uidescriptionfunction1, \uidescriptionfunction2\in\Gamma (W, \OO_{\PP^1})$, 
$\stdvectorfiledpletter\in\Gamma (W,\Theta_{\PP^1})$
Consider complex-analytic Laurent series: 
$$
\uidescriptionfunction l=\sum_{k=-n_l}^\infty a_{l,k}t_p^{k},\quad(l=1,2),\quad \stdvectorfiledpletter=\left(\sum_{k=-n_3}^\infty b_{k}t_p^{k}\right)\frac\partial{\partial t_p}.
$$
Set
$$
\uidescriptionfunction l'=\sum_{k=-n_l}^{-1} a_{l,k}t_p^{k},\quad(l=1,2),\quad \stdvectorfiledpletter'=\left(\sum_{k=-n_3}^{-1} b_{k}t_p^{k}\right)\frac\partial{\partial t_p}.
$$
These functions and this vector field are algebraic since the sums are finite. The functions have zero of degree at least 1 at $\infty$, 
the vector field has zero of degree at least 3 at $\infty$, so $\uidescriptionfunction1',\uidescriptionfunction2'\in\Gamma (W, \OO_{\PP^1})$ and $\stdvectorfiledpletter'\in\Gamma (W, \Theta_{\PP^1})$.
Hence, $\uidescriptionfunction1'-\uidescriptionfunction1$, $\uidescriptionfunction2'-\uidescriptionfunction2\in\Gamma (W, \OO_{\PP^1})$ and 
$\stdvectorfiledpletter'-\stdvectorfiledpletter\in\Gamma (W, \Theta_{\PP^1})$, 
but $\uidescriptionfunction1'-\uidescriptionfunction1$, $\uidescriptionfunction2'-\uidescriptionfunction2$, 
and $\stdvectorfiledpletter'-\stdvectorfiledpletter$
have no poles at $p$, so they define an element of 
$\gsii{p,j}$
Hence,
$\kappa_{\Theta,p,j}(\uidescriptionfunction1,\uidescriptionfunction2,\stdvectorfiledpletter)$ and 
$\kappa_{\Theta,p,j}(\uidescriptionfunction1',\uidescriptionfunction2',\stdvectorfiledpletter')$
define the same element of $\Gamma(W_p,\giicircinv)/\Gamma(W_p,\giisiinv{\opensetforvertex{p,j}})$, and 
the composition of the restriction of $\kappa_{\Theta,p,j}$ and the natural projection under consideration 
is surjective. Injectivity is also clear since if a Laurent polynomial of the considered form has 
no pole at $p$, then it is zero.
\end{proof}

Note that despite the proof is similar to the proof of Lemma \ref{wwpquotient}, Laurent polynomials here and in Lemma \ref{wwpquotient}
have different meanings: here they from the $U_{\opensetforvertex{p,j}}$-description of a vector field on 
$U_{\numberoffixedopensets}=U_{\opensetforvertex{p,j}}\cap U_{\numberoffixedopensets}$, while in 
Lemma \ref{wwpquotient} they formed the $U_{\numberoffixedopensets}$-description of a vector field on $U_{\numberoffixedopensets}$.

Denote the direct sum of maps $\kappa_{\Theta,p,j}$, which maps 
$\bigoplus_{j=1}^{\numberofverticespt p}\gsiicirc{p,j}$
to $\bigoplus_{j=1}^{\numberofverticespt p}\Gamma(W_p,\giicircinv)$, by $\kappa_{\Theta,p}$.
Denote by $\abstractvectorspace_{1,0,p}$ the subspace of $\bigoplus_{j=1}^{\numberofverticespt p}\gsiicirc{p,j}$
formed by the $3\numberofverticespt p$-tuples of the 
form
$$
(0,0,0,\uidescriptionfunctiondiff{2}{1},\uidescriptionfunctiondiff{2}{2}, \stdvectorfiledpletter[2],\ldots,
\uidescriptionfunctiondiff{\numberofverticespt p}{1},\uidescriptionfunctiondiff{\numberofverticespt p}{1}, \stdvectorfiledpletter[\numberofverticespt p]),
$$
where
$$
\uidescriptionfunctiondiff{j}{l}=\sum_{k=-n_{j,l}}^{-1}a_{j,l,k}t_p^{k}, \quad 
\stdvectorfiledpletter[j]=\left(\sum_{k=-n_{j,3}}^{-1}b_{j,k}t_p^{k}\right)\frac\partial{\partial t_p}.
$$

\begin{lemma}\label{delta0suffices}
The restriction of the composition of $\kappa_{\Theta,p}$ and the natural projection
$$
\bigoplus_{j=1}^{\numberofverticespt p}\Gamma(W_p,\giicircinv)\to
\left.\!\left(
\bigoplus_{j=1}^{\numberofverticespt p}\left(\Gamma(W_p,\giicircinv)/\Gamma(W_p,\giisiinv{\opensetforvertex{p,j}})\right)
\right)\right/\Gamma(W_p,\giicircinv)
$$
to $\abstractvectorspace_{1,0,p}$ is surjective. 
Its kernel is one-dimensional.
\end{lemma}

\begin{proof}
To prove surjectivity, consider a 
$3\numberofverticespt p$-tuple $(\uidescriptionfunctiondiff{1}{1}',\uidescriptionfunctiondiff{1}{2}', \stdvectorfiledpletter[1]',\ldots, 
\uidescriptionfunctiondiff{\numberofverticespt p}{1}', \uidescriptionfunctiondiff{\numberofverticespt p}{2}',\stdvectorfiledpletter[\numberofverticespt p])\in 
\bigoplus_{j=1}^{\numberofverticespt p}\gsiicirc{p,j}$.
For every $j$, $1\le j\le \numberofverticespt p$, set
$$
\left(
\begin{array}{c}
\uidescriptionfunction{j,1}''\\
\uidescriptionfunction{j,1}''\\
\stdvectorfiledpletter''_j
\end{array}
\right)=
\uilargetransition{\opensetforvertex{p,j},\opensetforvertex{p,1}}
\left(
\begin{array}{c}
\uidescriptionfunctiondiff{1}{1}'\\
\uidescriptionfunctiondiff{1}{2}'\\
\stdvectorfiledpletter[1]'
\end{array}
\right).
$$
By Lemma \ref{rijregular}, these functions and vector fields are regular on $W$.
By Lemma \ref{vfieldtransition}, 
$\kappa_{\Theta,p,j}(\uidescriptionfunction{j,1}'',\uidescriptionfunction{j,2}'',\stdvectorfiledpletter''_j)=
\kappa_{\Theta,p,1}(\uidescriptionfunctiondiff{1}{1}',\uidescriptionfunctiondiff{1}{2}',\stdvectorfiledpletter[1]')$.
Hence, 
$$
\kappa_{\Theta,p}(\uidescriptionfunctiondiff{1}{1}',\uidescriptionfunctiondiff{2}{1}', \stdvectorfiledpletter[1],\ldots, 
\uidescriptionfunctiondiff{\numberofverticespt p}{1}', \uidescriptionfunctiondiff{\numberofverticespt p}{2}',\stdvectorfiledpletter[\numberofverticespt p]')
$$ 
and 
$$
\kappa_{\Theta,p}(\uidescriptionfunctiondiff{1}{1}'-\uidescriptionfunction{1,1}'',
\uidescriptionfunctiondiff{1}{2}'-\uidescriptionfunction{1,2}'', 
\stdvectorfiledpletter[1]'-\stdvectorfiledpletter''_1,\ldots, 
\uidescriptionfunctiondiff{\numberofverticespt p}{1}'-\uidescriptionfunction{\numberofverticespt p,1}'', 
\uidescriptionfunctiondiff{\numberofverticespt p}{2}'-\uidescriptionfunction{\numberofverticespt p,2}'',
\stdvectorfiledpletter[\numberofverticespt p]'-\stdvectorfiledpletter''_{\numberofverticespt p})
$$
define the same coset in 
$$
\left.\!\left(\bigoplus_{j=1}^{\numberofverticespt p}\left(\Gamma(W_p,\giicircinv)/\Gamma(W_p,\giisiinv{\opensetforvertex{p,j}})\right)\right)\right/\Gamma(W_p,\giicircinv).
$$
Observe that 
$\uidescriptionfunctiondiff{1}{1}'=\uidescriptionfunction{1,1}''$, 
$\uidescriptionfunctiondiff{1}{2}'=\uidescriptionfunction{1,2}''$, 
and $\stdvectorfiledpletter[1]'=\stdvectorfiledpletter''_1$. 
Now, by Lemma \ref{wwiquotient}, every triple 
$(\uidescriptionfunctiondiff{j}{1}'-\uidescriptionfunction{j,1}'',
\uidescriptionfunctiondiff{j}{2}'-\uidescriptionfunction{j,2}'', 
\stdvectorfiledpletter[j]'-\stdvectorfiledpletter''_j)\in \gsiicirc{p,j}$ can be replaced with 
$(\uidescriptionfunctiondiff{j}{1}, \uidescriptionfunctiondiff{j}{2}, \stdvectorfiledpletter[j])\in \gsiicirc{p,j}$, where
$$
\uidescriptionfunctiondiff{j}{l}=\sum_{k=-n_{j,l}}^{-1}a_{j,l,k}t_p^{k}, \quad 
\stdvectorfiledpletter[j]=\left(\sum_{k=-n_{j,3}}^{-1}b_{j,k}t_p^{k}\right)\frac\partial{\partial t_p},
$$
so that 
$\kappa_{\Theta,p,j}(\uidescriptionfunctiondiff{j}{1}'-\uidescriptionfunction{j,1}'',
\uidescriptionfunctiondiff{j}{2}'-\uidescriptionfunction{j,2}'', 
\stdvectorfiledpletter[j]'-\stdvectorfiledpletter''_j)$ 
and 
$\kappa_{\Theta,p,j}(\uidescriptionfunctiondiff{j}{1}, \uidescriptionfunctiondiff{j}{2}, \stdvectorfiledpletter[j])$
define the same coset in $\Gamma(W_p,\giicircinv)/\Gamma(W_p,\giisi{\opensetforvertex{p,j}})$. 
Hence, 
$$
\kappa_{\Theta,p}(0,0,0,\uidescriptionfunctiondiff{2}{1},\uidescriptionfunctiondiff{2}{2}, \stdvectorfiledpletter[2],\ldots,
\uidescriptionfunctiondiff{\numberofverticespt p}{1},\uidescriptionfunctiondiff{\numberofverticespt p}{2}, \stdvectorfiledpletter[\numberofverticespt p])
$$ 
and 
$$
\kappa_{\Theta,p}(\uidescriptionfunctiondiff{1}{1}',\uidescriptionfunctiondiff{1}{2}', \stdvectorfiledpletter[1]',\ldots, 
\uidescriptionfunctiondiff{\numberofverticespt p}{1}', \uidescriptionfunctiondiff{\numberofverticespt p}{2}',\stdvectorfiledpletter[\numberofverticespt p]')
$$
define the same element of 
$$
\left.\!\left(\bigoplus_{j=1}^{\numberofverticespt p}\left(\Gamma(W_p,\giicircinv)/\Gamma(W_p,\giisiinv{\opensetforvertex{p,j}})\right)\right)\right/\Gamma(W_p,\giicircinv),
$$ 
and 
$$
(0,0,0,\uidescriptionfunctiondiff{2}{1},\uidescriptionfunctiondiff{2}{2}, \stdvectorfiledpletter[2],\ldots,
\uidescriptionfunctiondiff{\numberofverticespt p}{1},\uidescriptionfunctiondiff{\numberofverticespt p}{2}, \stdvectorfiledpletter[\numberofverticespt p])\in\abstractvectorspace_{1,0,p},
$$ 
therefore, the restriction of the composition to $\abstractvectorspace_{1,0,p}$ is surjective.

Now suppose that 
$$
(0,0,0,\uidescriptionfunctiondiff{2}{1},\uidescriptionfunctiondiff{2}{2}, \stdvectorfiledpletter[2],\ldots,
\uidescriptionfunctiondiff{\numberofverticespt p}{1},\uidescriptionfunctiondiff{\numberofverticespt p}{2}, \stdvectorfiledpletter[\numberofverticespt p])
\in\abstractvectorspace_{1,0,p}
$$ 
and 
$$
\kappa_{\Theta,p}(0,0,0,\uidescriptionfunctiondiff{2}{1},\uidescriptionfunctiondiff{2}{2}, \stdvectorfiledpletter[2],\ldots,
\uidescriptionfunctiondiff{\numberofverticespt p}{1},\uidescriptionfunctiondiff{\numberofverticespt p}{2}, 
\stdvectorfiledpletter[\numberofverticespt p])
$$
defines the zero
coset in 
$$
\left.\!\left(\bigoplus_{j=1}^{\numberofverticespt p}\left(\Gamma(W_p,\giicircinv)/\Gamma(W_p,\giisiinv{\opensetforvertex{p,j}})\right)\right)\right/
\Gamma(W_p,\giicircinv).
$$
For simplicity of notation, denote $\uidescriptionfunctiondiff{1}{1}=\uidescriptionfunctiondiff{1}{2}=0$, $\stdvectorfiledpletter[1]=0$.
Then there exist 
$(\uidescriptionfunctiondiff{j}{1}',\uidescriptionfunctiondiff{j}{2}',\stdvectorfiledpletter[j]')\in \gsii{p,j}$
and 
$(\uidescriptionfunction{j,1}'',\uidescriptionfunction{j,2}'',\stdvectorfiledpletter_j'')\in \gsiicirc{p,j}$
($1\le j\le \numberofverticespt{p}$)
such that 
$\uidescriptionfunctiondiff{j}{l}=\uidescriptionfunctiondiff{j}{l}'+\uidescriptionfunction{j,l}''$, 
$\stdvectorfiledpletter[j]=\stdvectorfiledpletter[j]'+\stdvectorfiledpletter''_j$ and 
$$
\kappa_{\Theta,p,1}(\uidescriptionfunction{1,1}'',\uidescriptionfunction{2,1}'',\stdvectorfiledpletter''_1)=
\kappa_{\Theta,p,2}(\uidescriptionfunction{1,2}'',\uidescriptionfunction{2,2}'',\stdvectorfiledpletter''_2)=
\ldots=
\kappa_{\Theta,p,\numberofverticespt p}(\uidescriptionfunction{1,\numberofverticespt p}'',\uidescriptionfunction{2,\numberofverticespt p}'',
\stdvectorfiledpletter''_{\numberofverticespt p}).
$$
By Lemma \ref{vfieldtransition}
this means that 
$$
\left(
\begin{array}{c}
\uidescriptionfunction{j,1}''\\
\uidescriptionfunction{j,2}''\\
\stdvectorfiledpletter''_j
\end{array}
\right)=
\uilargetransition{\opensetforvertex{p,j},\opensetforvertex{p,1}}
\left(
\begin{array}{c}
\uidescriptionfunction{1,1}''\\
\uidescriptionfunction{1,2}''\\
\stdvectorfiledpletter''_1
\end{array}
\right).
$$
In particular, $\stdvectorfiledpletter''_j=\stdvectorfiledpletter''_1$ and all functions 
$\uidescriptionfunction{j,l}''$ and all vector fields $\stdvectorfiledpletter''_j$ are determined by 
$(\uidescriptionfunction{1,1}'', \uidescriptionfunction{1,2}'', \stdvectorfiledpletter''_1)$.
On the other hand, the conditions 
$\uidescriptionfunctiondiff{j}{l}=\uidescriptionfunctiondiff{j}{l}'+\uidescriptionfunction{j,l}''$, 
$\stdvectorfiledpletter[j]=\stdvectorfiledpletter[j]+\stdvectorfiledpletter''_j$ 
for $j=1$ mean that 
$\uidescriptionfunctiondiff{1}{l}'=-\uidescriptionfunction{1,l}''$, 
$\stdvectorfiledpletter[1]'=-\stdvectorfiledpletter''_1$. 
Therefore, $\uidescriptionfunction{1,1}''$, $\uidescriptionfunction{1,2}''$, and $\stdvectorfiledpletter''_1$
are regular at $p$.
By Lemma \ref{logderpole},
$\ord_p(\uidescriptionfunction{l,j}'')\ge-1$ for $l=1,2$, $1\le j\le \numberofverticespt p$. Now it follows from the definition of 
$\abstractvectorspace_{1,0,p}$ that 
$\uidescriptionfunctiondiff{j}{l}=a_{-1,j,l}t_p^{-1}$ for some $a_{-1,j,l}\in\CC$, and $\stdvectorfiledpletter[j]=0$.
Moreover, it follows from a consideration of Laurent series of $\stdvectorfiledpletter''_1$, of entries of 
$\uilargetransition{\opensetforvertex{p,j},\opensetforvertex{p,1}}$, and of $\uidescriptionfunction{j,l}''$
that all numbers $a_{-1,j,l}$ are uniquely determined by the value of $\stdvectorfiledpletter''_1$ at $p$, which is 
an element of a one-dimensional space (the tangent space of $\PP^1$ at $p$). Therefore, the kernel 
of the composition of $\kappa_{\Theta,p}$ and the projection is at most one-dimensional.

Now let us prove that the kernel contains a nonzero element. Set 
$\uidescriptionfunction{1,1}''=\uidescriptionfunction{1,2}''=0$, $\stdvectorfiledpletter''_1=\partial/\partial t_p$, and
$$
\left(
\begin{array}{c}
\uidescriptionfunction{j,1}''\\
\uidescriptionfunction{j,2}''\\
\stdvectorfiledpletter''_j
\end{array}
\right)=
\uilargetransition{\opensetforvertex{p,j},\opensetforvertex{p,1}}
\left(
\begin{array}{c}
\uidescriptionfunction{1,1}''\\
\uidescriptionfunction{1,2}''\\
\stdvectorfiledpletter''_1
\end{array}
\right).
$$
By Lemma \ref{rijregular}, all functions $\uidescriptionfunction{j,l}''$ are regular on $W$.
By Lemma \ref{vfieldtransition}, 
$$
\kappa_{\Theta,p,1}(\uidescriptionfunction{1,1}'',\uidescriptionfunction{1,2}'',\stdvectorfiledpletter''_1)=
\kappa_{\Theta,p,j}(\uidescriptionfunction{j,1}'',\uidescriptionfunction{j,2}'',\stdvectorfiledpletter''_j)
\text{ for }1\le j\le \numberofverticespt p,
$$
and 
$$
\kappa_{\Theta,p}(\uidescriptionfunction{1,1}'',\uidescriptionfunction{1,2}'',\stdvectorfiledpletter''_1, \ldots, 
\uidescriptionfunction{\numberofverticespt p,1}'',\uidescriptionfunction{\numberofverticespt p,2}'',\stdvectorfiledpletter''_{\numberofverticespt p})
$$
defines the zero coset in 
$$
\left.\!\left(\bigoplus_{j=1}^{\numberofverticespt p}\Big(\Gamma(W_p,\giicircinv)/\Gamma(W_p,\giisiinv{\opensetforvertex{p,j}})\Big)\right)\right/\Gamma(W_p,\giicircinv).
$$
By Lemma \ref{logderpole}, 
$$\ord_p(\uidescriptionfunction{j,l}'')=-1\text{ for }2\le j\le\numberofverticespt p\text{ and }l=1,2.$$ 
So we can write 
$\uidescriptionfunction{j,l}''=\uidescriptionfunctiondiff{j}{l}+\uidescriptionfunctiondiff{j}{l}'$, where $\uidescriptionfunctiondiff{j}{l}=a_{-1,j,l}t_p^{-1}$ 
(here $a_{-1,j,l}\in\CC$, and for $j\ge 2$ we also have
$a_{l,-1,j}\ne 0$) and 
$\uidescriptionfunctiondiff{j}{l}'$ is regular at $p$ (and hence on $W_p$).
By the definition of matrices 
$\uilargetransition{\opensetforvertex{p,j},\opensetforvertex{p,1}}$, 
$\stdvectorfiledpletter''_j=\stdvectorfiledpletter''_1$, and we can set $\stdvectorfiledpletter[j]'=\stdvectorfiledpletter''_j$, 
$\stdvectorfiledpletter[j]=0$. Then $\stdvectorfiledpletter[j]'\in\Gamma(W_p,\Theta_{\PP^1})$,
and 
$$
\kappa_{\Theta,p,j}(\uidescriptionfunctiondiff{j}{1}',\uidescriptionfunctiondiff{j}{2}',\stdvectorfiledpletter[j]')
$$ 
defines the zero coset in 
$$
\Gamma(W_p,\giicircinv)/\Gamma(W_p,\giisiinv{\opensetforvertex{p,j}}).
$$ By construction, 
$$
\uidescriptionfunction{}=
(\uidescriptionfunctiondiff{1}{1},\uidescriptionfunctiondiff{1}{2},\stdvectorfiledpletter[1],\ldots, 
\uidescriptionfunctiondiff{\numberofverticespt p}{1},\uidescriptionfunctiondiff{\numberofverticespt p}{2},\stdvectorfiledpletter[\numberofverticespt p])
\in\abstractvectorspace_{1,0,p}.
$$
Since 
$$
a_{-1,j,l}\ne0\text{ for all }2\le j\le \numberofverticespt p\text{ and }l=1,2,
$$
we have $\uidescriptionfunction{}\ne 0$. Since $\uidescriptionfunctiondiff{j}{l}=\uidescriptionfunction{j,l}''-\uidescriptionfunctiondiff{j}{l}'$ and 
$\stdvectorfiledpletter[j]=\stdvectorfiledpletter[j]'-\stdvectorfiledpletter''_j$, 
$\uidescriptionfunction{}$ is an element of the kernel of the composition of $\kappa_{\Theta,p}$ and the 
projection from $\bigoplus_{j=1}^{\numberofverticespt p}\Gamma(W_p,\giicircinv)$
to 
$$
\left.\!\left(\bigoplus_{j=1}^{\numberofverticespt p}\Big(\Gamma(W_p,\giicircinv)/\Gamma(W_p,\giisiinv{\opensetforvertex{p,j}})\Big)\right)\right/\Gamma(W_p,\giicircinv).
$$
\end{proof}

Now we are going to use the map
$$
\psi_{p,j}\colon \gsiicirc{p,j}\to \gsvicirc{p,j}
$$
we have introduced before using Lemma \ref{uidescpsi}.
Each of the functions it computes is the $U_{\opensetforvertex{p,j}}$-description of a function of a degree 
$\chi$ on $U_{\numberoffixedopensets}$ ($\chi=\lambda_1,\ldots,\lambda_{\numberoflatticegenerators}$), and exactly
$\dim \Gamma(\PP^1,\OO(\mathcal D(\chi)))$ of these functions are of this degree. Denote 
the morphism 
$\gsiicirc{p,j}\to\Gamma(W,\OO_{\PP^1})$ 
computing the $k$th of the functions of degree $\chi$ by
$\psi_{p,j,\chi,k}$. Denote also by $\kappa_{\OO,p}$ the direct sum of morphisms $\kappa_{\OO,p,j}$, which
maps 
$\bigoplus_{j=1}^{\numberofverticespt p} \gsvicirc{p,j}$
to $\bigoplus_{j=1}^{\numberofverticespt p}\Gamma(W_p,\gvicircinv)$.
We are computing the kernel of the map 
\begin{multline*}
\left.\!\left(
\bigoplus_{j=1}^{\numberofverticespt p}
\left(\Gamma(W_p,\giicircinv)/
\Gamma(W_p,\giisiinv{\opensetforvertex{p,j}})\right)\right)\right/\Gamma(W_p,\giicircinv)\\
\longrightarrow
\left.\!\left(
\bigoplus_{j=1}^{\numberofverticespt p}\left(\Gamma(W_p,\gvicircinv)/
\Gamma(W_p,\gvisiinv{\opensetforvertex{p,j}})\right)\right)\right/\Gamma(W_p,\gvicircinv)
\end{multline*}
induced by $\psi_p$, so 
by Lemma \ref{delta0suffices}, it is sufficient to consider the restriction of $\psi_p$
to $\abstractvectorspace_{1,0,p}$. Then $\kappa_{\Theta,p}$ maps the kernel of the composition of $(\kappa_{\OO,p}\circ\psi_p|_{\abstractvectorspace_{1,0,p}})$ 
and the natural projection 
$$
\bigoplus_{j=1}^{\numberofverticespt p}(\Gamma(W_p,\gvicircinv)
\to\left.\!\left(
\bigoplus_{j=1}^{\numberofverticespt p}\left(\Gamma(W_p,\gvicircinv)/
\Gamma(W_p,\gvisiinv{\opensetforvertex{p,j}})\right)\right)\right/\Gamma(W_p,\gvicircinv)
$$
to 
\begin{multline*}
\ker\left(\left.\!\Bigg(
\bigoplus_{j=1}^{\numberofverticespt p}
\left(\Gamma(W_p,\giicircinv)/
\Gamma(W_p,\giisiinv{\opensetforvertex{p,j}})\right)\Bigg)\right/\Gamma(W_p,\giicircinv)
\right.\!\\
\longrightarrow
\left.\!
\left.\!\Bigg(
\bigoplus_{j=1}^{\numberofverticespt p}\left(\Gamma(W_p,\gvicircinv)/
\Gamma(W_p,\gvisiinv{\opensetforvertex{p,j}})\right)\Bigg)\right/\Gamma(W_p,\gvicircinv)\right)
\end{multline*}
surjectively, and the kernel of this composition contains the one-dimensional kernel $\ker \kappa_{\Theta,p}|_{\abstractvectorspace_{1,0,p}}$
since $\psi_p$ induces the map
\begin{multline*}
\left.\!\left(
\bigoplus_{j=1}^{\numberofverticespt p}
\left(\Gamma(W_p,\giicircinv)/
\Gamma(W_p,\giisiinv{\opensetforvertex{p,j}})\right)\right)\right/\Gamma(W_p,\giicircinv)\\
\longrightarrow
\left.\!\left(
\bigoplus_{j=1}^{\numberofverticespt p}\left(\Gamma(W_p,\gvicircinv)/
\Gamma(W_p,\gvisiinv{\opensetforvertex{p,j}})\right)\right)\right/\Gamma(W_p,\gvicircinv)
\end{multline*}
via $\kappa_{\Theta,p}$ and $\kappa_{\OO,p}$.
So we have to find the preimage 
$$\textstyle
\abstractvectorspace_{1,0,p}\cap\psi_p^{-1}\left(\kappa_{\OO,p}^{-1}\left((\bigoplus_{j=1}^{\numberofverticespt p}
\Gamma(W_p,\gvisiinv{\opensetforvertex{p,j}}))+\Gamma(W_p,\gvicircinv)\right)\right).$$

\begin{remark}\label{commdiagremark}
This is illustrated by the following commutative diagram:
$$
\xymatrix@C=0pt
{
\scriptstyle
\left.\!\left(
\bigoplus_{j=1}^{\numberofverticespt p}
\left(\Gamma(W_p,\giicircinv)/
\Gamma(W_p,\giisiinv{\opensetforvertex{p,j}})\right)\right)\right/\Gamma(W_p,\giicircinv)\ar[dr] \\
& 
\scriptstyle
\left.\!\left(
\bigoplus_{j=1}^{\numberofverticespt p}\left(\Gamma(W_p,\gvicircinv)/
\Gamma(W_p,\gvisiinv{\opensetforvertex{p,j}})\right)\right)\right/\Gamma(W_p,\gvicircinv)
\\
\scriptstyle
\bigoplus_{j=1}^{\numberofverticespt p}
\Gamma(W_p,\giicircinv)
\ar[r] 
\ar@{->>}[uu]_{\substack{\text{canonical}\\\text{projection}}}
& 
\scriptstyle
\bigoplus_{j=1}^{\numberofverticespt p}
\Gamma(W_p,\gvicircinv)
\ar@{->>}[u]_{\substack{\text{canonical}\\\text{projection}}}
\\
\scriptstyle
\bigoplus_{j=1}^{\numberofverticespt p}
\gsiicirc{p,j}
\ar[r]^{\psi_p}
\ar@{^(->>}[u]_{\kappa_{\Theta,p}}
&
\scriptstyle
\bigoplus_{j=1}^{\numberofverticespt p}
\gsvicirc{p,j}
\ar@{^(->>}[u]_{\kappa_{\OO,p}}
\\
{\vphantom{\big(}\abstractvectorspace_{1,0,p}}
\ar@{^(->}[u]
\ar@/^5em/@{->>}[uuuu]^{\dim\ker=1}
}
$$
\end{remark}

\begin{lemma}\label{diffregular}
Let 
$$
\uidescriptionfunction{}=(0,0,0,\uidescriptionfunctiondiff{2}{1},\uidescriptionfunctiondiff{2}{2},\stdvectorfiledpletter[2],\ldots,
\uidescriptionfunctiondiff{\numberofverticespt p}{1},\uidescriptionfunctiondiff{\numberofverticespt p}{2},\stdvectorfiledpletter[\numberofverticespt p])
\in \abstractvectorspace_{1,0,p}.
$$
Suppose that 
$$
\kappa_{\OO,p}(\psi_p(\uidescriptionfunction{}))
\in(\bigoplus_{j=1}^{\numberofverticespt p}\Gamma(W_p,\gvisiinv{\opensetforvertex{p,j}}))+\Gamma(W_p,\gvicircinv)
\subseteq\bigoplus_{j=1}^{\numberofverticespt p}\Gamma(W_p,\gvicircinv),
$$
where 
the last summand is embedded into $\bigoplus_{j=1}^{\numberofverticespt p}\Gamma(W_p,\gvicircinv)$ diagonally.
Pick two vertices $\indexedvertexpt{p}{j_1}$ and $\indexedvertexpt{p}{j_2}$ of $\stdpolyhedronletter_p$ and denote 
$i_1=\opensetforvertex{p,j_1}$, $i_2=\opensetforvertex{p,j_2}$. Also choose 
$\chi\in\{\lambda_1,\ldots,\lambda_{\numberoflatticegenerators}\}$
and an 
index $k$ ($1\le k\le \dim \Gamma(\PP^1,\OO(\mathcal D(\chi)))$).

Then 
it is possible to write
$$
\psi_{p,j_1,\chi,k}(\uidescriptionfunction{})-
\mu_{i_2,i_1,\chi}\psi_{p,j_2,\chi,k}(\uidescriptionfunction{})
$$
as 
$$
f[j_1]_{\chi,k}-\mu_{i_2,i_1,\chi}f[j_2]_{\chi,k},
$$
where 
$f_{j,\chi,k}\in\Gamma(W_p,\OO_{\PP^1})$ for $j=j_1,j_2$.
\end{lemma}

\begin{proof}
Since 
$$
\kappa_{\OO,p}(\psi_p(\uidescriptionfunction{}))\in(\bigoplus_{j=1}^{\numberofverticespt p}\Gamma(W_p,\gvisi{\opensetforvertex{p,j}}))+\Gamma(W,\gvicircinv),
$$
$\psi_p(\uidescriptionfunction{})$ can be written as $f+f'$, where 
\begin{multline*}
f=(f[j]_{\chi',k'})_{1\le j\le \numberofverticespt p, \chi'\in\{\lambda_1,\ldots,\lambda_{\numberoflatticegenerators}\}, 1\le k'\le \dim\Gamma(\PP^1,\OO(\mathcal D(\chi)))}\in
\bigoplus_{j=1}^{\numberofverticespt{p}}\gsvi{p,j},\\
f'=(f'_{j,\chi',k'})_{1\le j\le \numberofverticespt p, \chi'\in\{\lambda_1,\ldots,\lambda_{\numberoflatticegenerators}\}, 1\le k'\le \dim\Gamma(\PP^1,\OO(\mathcal D(\chi)))}\in
\bigoplus_{j=1}^{\numberofverticespt{p}}\gsvicirc{p,j},
\end{multline*}
and, in addition,
$$
\kappa_{\OO,p,j}((f'_{j,\chi',k'})_{\chi'\in\{\lambda_1,\ldots,\lambda_{\numberoflatticegenerators}\}, 1\le k'\le \dim\Gamma(\PP^1,\OO(\mathcal D(\chi)))})
$$ 
does not depend on $j$.
By the definition of $\gsvi{p,j}$,
$f[j]_{\chi,k}\in\Gamma(W_p,\OO_{\PP^1})$ for all $j$.
It also follows from Lemma \ref{functtransition}
that $f'_{j_1,\chi,k}=\mu_{i_2,i_1,\chi}f'_{j_2,\chi,k}$.
Thus, 
\begin{multline*}
\psi_{p,j_1,\chi,k}(\uidescriptionfunction{})-\mu_{i_2,i_1,\chi}\psi_{p,j_2,\chi,k}(\uidescriptionfunction{})
=\\
(f[j_1]_{\chi,k}+f'_{j_1,\chi,k})-\mu_{i_2,i_1,\chi}(f[j_2]_{\chi,k}+f'_{j_2,\chi,k})=
f[j_1]_{\chi,k}-\mu_{i_2,i_1,\chi}f[j_2]_{\chi,k}.
\end{multline*}
\end{proof}

\begin{corollary}\label{diffregularcor}
If the hypothesis of Lemma \ref{diffregular} holds, then
$$
\ord_p(\psi_{p,j_1,\chi,k}(\uidescriptionfunction{})-
\mu_{i_2,i_1,\chi}\psi_{p,j_2,\chi,k}(\uidescriptionfunction{}))\ge 
\min (0,\ord_p(\mu_{i_2,i_1,\chi})).
$$
\qed
\end{corollary}


\begin{corollary}\label{diffregularcorii}
Suppose that the hypothesis of Lemma \ref{diffregular} holds.
Let $f\in\Gamma(\PP^1,\OO(\mathcal D(\chi)))$.
Denote $a_{1,1}=\uidegree{i_1,1}^*(\chi)$, $a_{1,2}=\uidegree{i_1,2}^*(\chi)$, 
$a_{2,1}=\uidegree{i_2,1}^*(\chi)$, and $a_{2,2}=\uidegree{i_2,2}^*(\chi)$.
Then
\begin{multline*}
\ord_p\left(\frac{\overline f}{\oluithreadfunction{i_1,1}^{a_{1,1}}\oluithreadfunction{i_1,2}^{a_{1,2}}}
(a_{1,1}\uidescriptionfunctiondiff{j_1}{1}+a_{1,2}\uidescriptionfunctiondiff{j_1}{1}-a_{2,1}\uidescriptionfunctiondiff{j_2}{1}-a_{2,2}\uidescriptionfunctiondiff{j_2}{1})
\right.\!\\
\left.\!
+d\left(\frac{\overline f}{\oluithreadfunction{i_1,1}^{a_{1,1}}\oluithreadfunction{i_1,2}^{a_{1,2}}}\right)\stdvectorfiledpletter_{j_1}
-\mu_{i_2,i_1,\chi}d\left(\frac{\overline f}{\oluithreadfunction{i_2,1}^{a_{2,1}}\oluithreadfunction{i_2,2}^{a_{2,2}}}\right)\stdvectorfiledpletter_{j_2}
\right)\ge \min (0,\ord_p(\mu_{i_2,i_1,\chi})).
\end{multline*}
\end{corollary}

\begin{proof}
Observe that the function under the $\ord$ sign in the left-hand side of the inequality is 
linear in $f$, and the right-hand side does not depend on $f$, so 
it is sufficient to prove the inequality for all functions $f$ forming 
a basis of $\Gamma(\PP^1,\OO(\mathcal D(\chi)))$. For example, we can use the functions of degree $\chi$ 
among the generators of $\CC[X]$ we have chosen to define the map $\psi$ for 
Theorem \ref{schlessgen}. Recall that we have denoted these generators by 
$\dependentgeneratorsdegree{\chi,1},\ldots, \dependentgeneratorsdegree{\chi,\dim\Gamma(\PP^1,\OO(\mathcal D(\chi)))}$ 
and that they form a basis of $\Gamma(\PP^1,\OO(\mathcal D(\chi)))$.
So, set $f=\dependentgeneratorsdegree{\chi,k}$.
By Lemma \ref{uidescpsi},
\begin{multline*}
\psi_{p,j_1,\chi,k}(\uidescriptionfunction{})-\mu_{i_2,i_1,\chi}\psi_{p,j_2,\chi,k}(\uidescriptionfunction{})=\\
\frac{\oldependentgeneratorsdegree{\chi,k}}{\oluithreadfunction{i_1,1}^{a_{1,1}}\oluithreadfunction{i_1,2}^{a_{1,2}}}
(a_{1,1} \uidescriptionfunctiondiff{j_1}{1}+a_{1,2} \uidescriptionfunctiondiff{j_1}{2})
+d\left(\frac{\oldependentgeneratorsdegree{\chi,k}}{\oluithreadfunction{i_1,1}^{a_{1,1}}\oluithreadfunction{i_1,2}^{a_{1,2}}}\right) \stdvectorfiledpletter_{j_1}\\
-\mu_{i_2,i_1,\chi}\left(
\frac{\oldependentgeneratorsdegree{\chi,k}}{\oluithreadfunction{i_2,1}^{a_{2,1}}\oluithreadfunction{i_2,2}^{a_{1,2}}}
(a_{2,1} \uidescriptionfunctiondiff{j_2}{1}+a_{1,2} \uidescriptionfunctiondiff{j_2}{2})
+d\left(\frac{\oldependentgeneratorsdegree{\chi,k}}{\oluithreadfunction{i_2,1}^{a_{2,1}}\oluithreadfunction{i_2,2}^{a_{1,2}}}\right) \stdvectorfiledpletter_{j_2}
\right)=\\
\frac{\oldependentgeneratorsdegree{\chi,k}}{\oluithreadfunction{i_1,1}^{a_{1,1}}\oluithreadfunction{i_1,2}^{a_{1,2}}}
(a_{1,1} \uidescriptionfunctiondiff{j_1}{1}+a_{1,2} \uidescriptionfunctiondiff{j_1}{2})
-\frac{\oluithreadfunction{i_2,1}^{a_{2,1}}\oluithreadfunction{i_2,2}^{a_{1,2}}}{\oluithreadfunction{i_1,1}^{a_{1,1}}\oluithreadfunction{i_1,2}^{a_{1,2}}}
\frac{\oldependentgeneratorsdegree{\chi,k}}{\oluithreadfunction{i_2,1}^{a_{2,1}}\oluithreadfunction{i_2,2}^{a_{1,2}}}
(a_{2,1} \uidescriptionfunctiondiff{j_2}{1}+a_{1,2} \uidescriptionfunctiondiff{j_2}{2})\\
+d\left(\frac{\oldependentgeneratorsdegree{\chi,k}}{\oluithreadfunction{i_1,1}^{a_{1,1}}\oluithreadfunction{i_1,2}^{a_{1,2}}}\right) \stdvectorfiledpletter_{j_1}
-\mu_{i_2,i_1,\chi}
d\left(\frac{\oldependentgeneratorsdegree{\chi,k}}{\oluithreadfunction{i_2,1}^{a_{2,1}}\oluithreadfunction{i_2,2}^{a_{1,2}}}\right) \stdvectorfiledpletter_{j_2}=\\
\frac{\oldependentgeneratorsdegree{\chi,k}}{\oluithreadfunction{i_1,1}^{a_{1,1}}\oluithreadfunction{i_1,2}^{a_{1,2}}}
(a_{1,1} \uidescriptionfunctiondiff{j_1}{1}+a_{1,2} \uidescriptionfunctiondiff{j_1}{2}-a_{2,1} \uidescriptionfunctiondiff{j_2}{1}-a_{1,2} \uidescriptionfunctiondiff{j_2}{2})\\
+d\left(\frac{\oldependentgeneratorsdegree{\chi,k}}{\oluithreadfunction{i_1,1}^{a_{1,1}}\oluithreadfunction{i_1,2}^{a_{1,2}}}\right) \stdvectorfiledpletter_{j_1}
-\mu_{i_2,i_1,\chi}
d\left(\frac{\oldependentgeneratorsdegree{\chi,k}}{\oluithreadfunction{i_2,1}^{a_{2,1}}\oluithreadfunction{i_2,2}^{a_{1,2}}}\right) \stdvectorfiledpletter_{j_2},
\end{multline*}
and the claim follows from Corollary \ref{diffregularcor}.
\end{proof}


Now we need more information about the behavior of $\ord_p(\mu_{\opensetforvertex{p,j_2},\opensetforvertex{p,j_1},\chi})$
depending on $j_1,j_2,\chi$.
Here we perform arithmetic actions on vertices of $\stdpolyhedronletter_p$, they are understood as arithmetic actions in $N$.

\begin{lemma}\label{linfuncdecomp}
For each degree $\chi$ and for 
any two vertices $\indexedvertexpt{p}{j_1}$, $\indexedvertexpt{p}{j_2}$
($1\le j_1, j_2\le \numberofverticespt p$) one has
$\ord_p(\mu_{\opensetforvertex{p,j_2},\opensetforvertex{p,j_1},\chi})=\chi(\indexedvertexpt{p}{j_1}-\indexedvertexpt{p}{j_2})$.
\end{lemma}

\begin{proof}
Again denote $i_1=\opensetforvertex{p}{j_1}$, $i_2=\opensetforvertex{p}{j_2}$, $a_{1,1}=\uidegree{i_1,1}^*(\chi)$, $a_{1,2}=\uidegree{i_1,2}^*(\chi)$, 
$a_{2,1}=\uidegree{i_2,1}^*(\chi)$, and $a_{2,2}=\uidegree{i_2,2}^*(\chi)$.
We chose 
$\uithreadfunction{i_1, 1}$, $\uithreadfunction{i_1, 2}$, $\uithreadfunction{i_1, 1}$, and $\uithreadfunction{i_1, 2}$ so that 
$\ord_p(\oluithreadfunction{i_1,1})=-\mathcal D_p(\uidegree{i_1,1})$, $\ord_p(\oluithreadfunction{i_1,2})=-\mathcal D_p(\uidegree{i_1,2})$,
$\ord_p(\oluithreadfunction{i_2,1})=-\mathcal D_p(\uidegree{i_2,1})$, $\ord_p(\oluithreadfunction{i_2,2})=-\mathcal D_p(\uidegree{i_2,2})$.
By the definition of $\mu_{i_2,i_1,\chi}$, we have
\begin{multline*}
\ord_p(\mu_{i_2,i_1,\chi})=
\ord_p\left(
\frac{\oluithreadfunction{i_2,1}^{a_{2,1}}\oluithreadfunction{i_2,2}^{a_{2,2}}}{\oluithreadfunction{i_1,1}^{a_{1,1}}\oluithreadfunction{i_1,2}^{a_{1,2}}}
\right)=\\
a_{1,1}\mathcal D_p(\uidegree{i_1,1})+a_{1,2}\mathcal D_p(\uidegree{i_1,2})-a_{2,1}\mathcal D_p(\uidegree{i_2,1})-a_{2,2}\mathcal D_p(\uidegree{i_2,2}).
\end{multline*}
Since $\uidegree{i_1,1}\in \normalvertexcone{\indexedvertexpt{p}{j_1}}{\stdpolyhedronletter_p}$, 
the minimum $\min_{b\in \stdpolyhedronletter_{p_l}}\uidegree{i_1,1}(b)$ is attained at 
$\indexedvertexpt{p}{j_1}$. 
In other words, $\mathcal D_p(\uidegree{i_1,1})=\uidegree{i_1,1}(\indexedvertexpt{p}{j_1})$.
Similarly, $\mathcal D_p(\uidegree{i_1,2})=\uidegree{i_1,2}(\indexedvertexpt{p}{j_1})$ (since $\uidegree{i_1,2}\in \normalvertexcone{\indexedvertexpt{p}{j_1}}{\stdpolyhedronletter_p}$),
$\mathcal D_p(\indexedvertexpt{p}{j_2})=\uidegree{i_2,1}(\indexedvertexpt{p}{j_2})$, and 
$\mathcal D_p(\uidegree{i_2,2})=\uidegree{i_2,2}(\indexedvertexpt{p}{j_2})$
(since $\uidegree{i_2,1},\uidegree{i_2,2}\in \normalvertexcone{\indexedvertexpt{p}{j_2}}{\stdpolyhedronletter_p}$).
Hence,
\begin{multline*}
a_{1,1}\mathcal D_p(\uidegree{i_1,1})+a_{1,2}\mathcal D_p(\uidegree{i_1,2})
-a_{2,1}\mathcal D_p(\uidegree{i_2,1})-a_{2,2}\mathcal D_p(\uidegree{i_2,2})=\\
(a_{1,1}\uidegree{i_1,1}+a_{1,2}\uidegree{i_1,2})(\indexedvertexpt{p}{j_1})-(a_{2,1}\uidegree{i_2,1}+a_{2,2}\uidegree{i_2,2})(\indexedvertexpt{p}{j_2})=\\
\chi(\indexedvertexpt{p}{j_1})-\chi(\indexedvertexpt{p}{j_2})=\chi(\indexedvertexpt{p}{j_1}-\indexedvertexpt{p}{j_2}).
\end{multline*}
\end{proof}


\begin{lemma}\label{linfuncestimate}
Let $\indexededgept{p}{j}$ be a finite edge of $\stdpolyhedronletter_p$ ($1\le j<\numberofverticespt p$), 
let $\chi=\primitivelattice{\normalvertexcone{\indexededgept{p}{j}}{\stdpolyhedronletter_p}}$.
Choose $\chi'\in\cap \normalvertexcone{\indexedvertexpt{p}{j}}{\stdpolyhedronletter_p} M$ 
so that $\chi$ and $\chi'$ form a 
lattice basis of $M$. 
Then 
$\chi(\indexedvertexpt{p}{j}-\indexedvertexpt{p}{j+1})=0$ and $\chi'(\indexedvertexpt{p}{j}-\indexedvertexpt{p}{j+1})=-\latticelength{\indexededgept{p}{j}}$.
%
\end{lemma}

\begin{proof}
Since $\chi\in\normalvertexcone{\indexededgept{p}{j}}{\stdpolyhedronletter_p}$, 
the minimum $\min_{a\in \stdpolyhedronletter_p}\chi(a)$ is attained at both $a=\indexedvertexpt{p}{j}$ and 
$a=\indexedvertexpt{p}{j+1}$, so 
$\chi(\indexedvertexpt{p}{j})=\chi(\indexedvertexpt{p}{j+1})$, $\chi(\indexedvertexpt{p}{j}-\indexedvertexpt{p}{j+1})=0$, and 
$\chi((1/\latticelength{\indexededgept{p}{j}})(\indexedvertexpt{p}{j}-\indexedvertexpt{p}{j+1}))=0$. 
It follows from the definition of $\latticelength{\indexededgept{p}{j}}$ that $(1/\latticelength{\indexededgept{p}{j}})(\indexedvertexpt{p}{j}-\indexedvertexpt{p}{j+1})$ is a primitive lattice vector.
Hence, 
elements of $M$ can take arbitrary 
values at it. Since $\chi$ and $\chi'$ form a lattice basis of $M$ and $\chi((1/\latticelength{\indexededgept{p}{j}})(\indexedvertexpt{p}{j}-\indexedvertexpt{p}{j+1}))=0$, we
conclude that $\chi'((1/\latticelength{\indexededgept{p}{j}})(\indexedvertexpt{p}{j}-\indexedvertexpt{p}{j+1}))=\pm 1$, and 
$\chi'(\indexedvertexpt{p}{j}-\indexedvertexpt{p}{j+1}))=\pm\latticelength{\indexededgept{p}{j}}$.
But the minimum $\min_{a\in \stdpolyhedronletter_p}\chi'(a)$ is attained at $\indexedvertexpt{p}{j}$ since $\chi'\in\normalvertexcone{\indexedvertexpt{p}{j}}{\stdpolyhedronletter_p}$, so 
$\chi'(\indexedvertexpt{p}{j}-\indexedvertexpt{p}{j+1}))=-\latticelength{\indexededgept{p}{j}}$.
\end{proof}

The following lemma is proved completely similarly to Lemma \ref{linfuncestimate}, one only has 
to interchange $\indexedvertexpt{p}{j}$ and $\indexedvertexpt{p}{j+1}$.

\begin{lemma}\label{linfuncestimateii}
Let $\indexededgept{p}{j}$ be a finite edge of $\stdpolyhedronletter_p$ ($1\le j<\numberofverticespt p$), 
let $\chi=\primitivelattice{\normalvertexcone{\indexededgept{p}{j}}{\stdpolyhedronletter_p}}$.
Choose $\chi'\in\normalvertexcone{\indexededgept{p}{j+1}}{\stdpolyhedronletter_p}\cap M$ 
so that $\chi$ and $\chi'$ form a 
lattice basis of $M$. 
Then 
$\chi(\indexedvertexpt{p}{j+1}-\indexedvertexpt{p}{j})=0$ and $\chi'(\indexedvertexpt{p}{j+1}-\indexedvertexpt{p}{j}))=-\latticelength{\indexededgept{p}{j}}$.\qed
%
\end{lemma}

\begin{lemma}\label{linfuncchain}
Let $\indexedvertexpt{p}{j_1}$ be a vertex of $\stdpolyhedronletter_p$, 
let $\indexededgept{p}{j_2}$ be a finite edge of $\stdpolyhedronletter_p$ ($1\le j_2<\numberofverticespt p$), 
and suppose that $j_1\le j_2$. 
Pick a degree $\chi''\in\normalvertexcone{\indexedvertexpt{p}{j_1}}{\stdpolyhedronletter_p}$.
Suppose that $\chi''\notin\normalvertexcone{\indexedvertexpt{p}{j_2+1}}{\stdpolyhedronletter_p}$. 
(Note that the contrary is possible since we allow $j_1=j_2$.)

Then $\chi''(\indexedvertexpt{p}{j_2}-\indexedvertexpt{p}{j_2+1})\le -\latticelength{\indexededgept{p}{j_2}}$.
\end{lemma}

\begin{proof}
Let $\chi=\primitivelattice{\normalvertexcone{\indexededgept{p}{j_2}}{\stdpolyhedronletter_p}}$.
Let $\chi'\in\normalvertexcone{\indexedvertexpt{p}{j_2}}{\stdpolyhedronletter_p}\cap M$ 
be a degree such that $\chi$ and $\chi'$ form a 
lattice basis of $M$. By Lemma \ref{linfuncestimate}, 
$\chi(\indexedvertexpt{p}{j_2}-\indexedvertexpt{p}{j_2+1})=0$ and $\chi'(\indexedvertexpt{p}{j_2}-\indexedvertexpt{p}{j_2+1})=-\latticelength{\indexededgept{p}{j_2}}$.
Since $\chi$ and $\chi'$ from a basis of $M$, we can write
$\chi''=a\chi+a'\chi'$. 
The line containing 
$\normalvertexcone{\indexededgept{p}{j_2}}{\stdpolyhedronletter_p}$
separates the 
normal subcones of the vertices 
$\indexedvertexpt{p}{j}$
with $j\le j_2$ 
from the 
normal subcones of the vertices 
$\indexedvertexpt{p}{j}$
with $j\ge j_2+1$. In particular, it does not separate $\normalvertexcone{\indexedvertexpt{p}{j_1}}{\stdpolyhedronletter_p}$ 
from $\normalvertexcone{\indexedvertexpt{p}{j_2}}{\stdpolyhedronletter_p}$, 
and it does not separate $\chi''$ from $\chi'$. Therefore, $a'\ge 0$.
If $a'=0$ and $a<0$, then $\chi\in\sck$, $-\chi\in\sck$, and $\sck$ cannot be a pointed cone.
If $a'=0$ and $a>0$, then $\chi''\in\normalvertexcone{\indexededgept{p}{j_2}}{\stdpolyhedronletter_p}\subset 
\normalvertexcone{\indexedvertexpt{p}{j_2+1}}{\stdpolyhedronletter_p}$, and this contradicts our assumption.
Therefore, $a' > 0$.
Then
$$
\chi''(\indexedvertexpt{p}{j_2}-\indexedvertexpt{p}{j_2+1})=a\chi(\indexedvertexpt{p}{j_2}-\indexedvertexpt{p}{j_2+1})+
a'\chi'(\indexedvertexpt{p}{j_2}-\indexedvertexpt{p}{j_2+1})
=-a'\latticelength{\indexededgept{p}{j_2}}\le -\latticelength{\indexededgept{p}{j_2}}.
$$
\end{proof}

The following lemma can be proved completely similarly using Lemma \ref{linfuncestimateii}
instead of Lemma \ref{linfuncestimate}.

\begin{lemma}\label{linfuncchainii}
Let $\indexedvertexpt{p}{j_1}$ be a vertex of $\stdpolyhedronletter_p$, 
let $\indexededgept{p}{j_2}$ be a finite edge of $\stdpolyhedronletter_p$ ($1\le j_2<\numberofverticespt p$), 
and suppose that $j_1\ge j_2+1$. 
Pick a degree $\chi''\in\normalvertexcone{\indexedvertexpt{p}{j_1}}{\stdpolyhedronletter_p}$.
Suppose that $\chi''\notin\normalvertexcone{\indexedvertexpt{p}{j_2}}{\stdpolyhedronletter_p}$. 
Then $\chi''(\indexedvertexpt{p}{j_2+1}-\indexedvertexpt{p}{j_2})\le -\latticelength{\indexededgept{p}{j_2}}$.\qed
\end{lemma}

\begin{lemma}\label{vfieldzero}
Let 
$\uidescriptionfunction{}=(0,0,0,\uidescriptionfunctiondiff{2}{1},\uidescriptionfunctiondiff{2}{2},\stdvectorfiledpletter[2],\ldots,
\uidescriptionfunctiondiff{\numberofverticespt p}{1},\uidescriptionfunctiondiff{\numberofverticespt p}{2},\stdvectorfiledpletter[\numberofverticespt p])
\in \abstractvectorspace_{1,0,p}$.
Suppose that
$\kappa_{\OO,p}(\psi_p(\uidescriptionfunction{}))\in
(\bigoplus_{j=1}^{\numberofverticespt p}\Gamma(W_p,\gvisiinv{\opensetforvertex{p,j}}))+\Gamma(W_p,\gvicircinv)
\subseteq\bigoplus_{j=1}^{\numberofverticespt p}\Gamma(W_p,\gvicircinv)$.

Then $\stdvectorfiledpletter_j=0$ for $2\le j\le \numberofverticespt p$.
\end{lemma}
\begin{proof}
Fix an index $j$, $2\le j\le \numberofverticespt p$.
For simplicity of notation, denote $\uidescriptionfunctiondiff{1}{1}=\uidescriptionfunctiondiff{1}{2}=0$, $\stdvectorfiledpletter[1]=0$.
Set $\chi=\primitivelattice{\normalvertexcone{\indexededgept{p}{j-1}}{\stdpolyhedronletter_p}}$.
It follows from the choice of the degrees $\lambda_1,\ldots,\lambda_{\numberoflatticegenerators}$ above that 
$\chi\in\{\lambda_1,\ldots,\lambda_{\numberoflatticegenerators}\}$.
Denote $a_{1,1}=\uidegree{j-1,1}^*(\chi)$, $a_{1,2}=\uidegree{j-1,2}^*(\chi)$, 
$a_{2,1}=\uidegree{j,1}^*(\chi)$, and $a_{2,2}=\uidegree{j,2}^*(\chi)$.

By Lemma \ref{giexist}, there exists a function $f\in \Gamma(\PP^1,\OO(\mathcal D(\chi)))$ defined at all ordinary points such that 
$\ord_p(\overline f)=-\mathcal D_p(\chi)$. $\chi$ is in the interior of $\sck$, 
so $\deg \mathcal D(\chi)>0$, while $\deg \div(f)=0$. Hence, there exists a point $p'\in\PP^1$ 
such that $\ord_{p'}(f)>-\mathcal D_{p'}(\chi)$. Choose a rational function $f'$ on $\PP^1$ that 
has exactly one zero of order one at $p$ and exactly one pole of order one at $p'$. Then $f'f\in \Gamma(\PP^1,\OO(\mathcal D(\chi)))$.
Note also that $df'$ is regular at $p$ and $d_p f'\ne 0$.
Set $f''=(1+f')f\in\Gamma(\PP^1,\OO(\mathcal D(\chi)))$.
Then $\overline{f''}$ is also defined at all ordinary points, and $\ord_p(\overline {f''})=-\mathcal D_p(\chi)$.

Since $\mathcal D_p(\cdot)$ is linear on $\normalvertexcone{\indexedvertexpt{p}{j-1}}{\stdpolyhedronletter_p}$, 
$\mathcal D_p(\chi)=
a_{1,1}\mathcal D_p(\uidegree{j-1,1})+a_{1,2}\mathcal D_p(\uidegree{j-1,2})$.
According to the choice 
of the functions $\uithreadfunction{i,1}$ and $\uithreadfunction{i,2}$ for all indices $i$, we have 
$-\mathcal D_p(\chi)=
a_{1,1}\ord_p(\oluithreadfunction{\opensetforvertex{p,j-1},1})+a_{1,2}\ord_p(\oluithreadfunction{\opensetforvertex{p,j-1},2})$.
Denote $i_1=\opensetforvertex{p,j-1}$, $i_2=\opensetforvertex{p,j}$. We have
$$
\ord_p\left(\frac{\oluithreadfunction{i_1,1}^{a_{1,1}}\oluithreadfunction{i_1,2}^{a_{1,2}}}{\overline f}\right)=
\ord_p\left(\frac{\oluithreadfunction{i_1,1}^{a_{1,1}}\oluithreadfunction{i_1,2}^{a_{1,2}}}{\overline{f''}}\right)=0,
$$
and it follows from Corollary \ref{diffregularcorii} that
\begin{multline*}
\textstyle
\ord_p\left(
(a_{1,1}\uidescriptionfunctiondiff{j-1}{1}+a_{1,2}\uidescriptionfunctiondiff{j-1}{2}-a_{2,1}\uidescriptionfunctiondiff{j}{1}-a_{2,2}\uidescriptionfunctiondiff{j}{2})
\vphantom{+\frac{\oluithreadfunction{i_1,1}^{a_{1,1}}\oluithreadfunction{i_1,2}^{a_{1,2}}}{\overline f}
d\left(\frac{\overline f}{\oluithreadfunction{i_1,1}^{a_{1,1}}\oluithreadfunction{i_1,2}^{a_{1,2}}}\right)\stdvectorfiledpletter[j-1]}
\right.\!\\
\textstyle
\left.\!
+\frac{\oluithreadfunction{i_1,1}^{a_{1,1}}\oluithreadfunction{i_1,2}^{a_{1,2}}}{\overline f}
d\left(\frac{\overline f}{\oluithreadfunction{i_1,1}^{a_{1,1}}\oluithreadfunction{i_1,2}^{a_{1,2}}}\right)\stdvectorfiledpletter[j-1]
-\frac{\oluithreadfunction{i_1,1}^{a_{1,1}}\oluithreadfunction{i_1,2}^{a_{1,2}}}{\overline f}
\mu_{i_2,i_1,\chi}d\left(\frac{\overline f}{\oluithreadfunction{i_2,1}^{a_{2,1}}\oluithreadfunction{i_2,2}^{a_{2,2}}}\right)\stdvectorfiledpletter[j]
\right)\ge\\
\min (0,\ord_p(\mu_{i_2,i_1,\chi})).
\end{multline*}
and
\begin{multline*}
\textstyle
\ord_p\left(
(a_{1,1}\uidescriptionfunctiondiff{j-1}{1}+a_{1,2}\uidescriptionfunctiondiff{j-1}{2}-a_{2,1}\uidescriptionfunctiondiff{j}{1}-a_{2,2}\uidescriptionfunctiondiff{j}{2})
\vphantom{+\frac{\oluithreadfunction{i_1,1}^{a_{1,1}}\oluithreadfunction{i_1,2}^{a_{1,2}}}{\overline{f''}}
d\left(\frac{\overline{f''}}{\oluithreadfunction{i_1,1}^{a_{1,1}}\oluithreadfunction{i_1,2}^{a_{1,2}}}\right)\stdvectorfiledpletter[j-1]}
\right.\!\\
\textstyle
\left.\!
+\frac{\oluithreadfunction{i_1,1}^{a_{1,1}}\oluithreadfunction{i_1,2}^{a_{1,2}}}{\overline{f''}}
d\left(\frac{\overline{f''}}{\oluithreadfunction{i_1,1}^{a_{1,1}}\oluithreadfunction{i_1,2}^{a_{1,2}}}\right)\stdvectorfiledpletter[j-1]
-\frac{\oluithreadfunction{i_1,1}^{a_{1,1}}\oluithreadfunction{i_1,2}^{a_{1,2}}}{\overline{f''}}
\mu_{i_2,i_1,\chi}d\left(\frac{\overline{f''}}{\oluithreadfunction{i_2,1}^{a_{2,1}}\oluithreadfunction{i_2,2}^{a_{2,2}}}\right)\stdvectorfiledpletter[j]
\right)\ge\\
\min (0,\ord_p(\mu_{i_2,i_1,\chi})).
\end{multline*}
By Lemma \ref{linfuncestimate}, 
$\ord_p(\mu_{i_2,i_1,\chi})=0$.
Hence, these two functions under the $\ord$ signs are regular at $p$.
Subtract the expressions under the $\ord$ signs and substitute the definition of $\mu_{i_2,i_1,\chi}$. We see that the following 
function is regular at $p$:
\begin{multline*}
\frac{\oluithreadfunction{i_1,1}^{a_{1,1}}\oluithreadfunction{i_1,2}^{a_{1,2}}}{\overline f}
d\left(\frac{\overline f}{\oluithreadfunction{i_1,1}^{a_{1,1}}\oluithreadfunction{i_1,2}^{a_{1,2}}}\right)\stdvectorfiledpletter[j-1]
-\frac{\oluithreadfunction{i_1,1}^{a_{1,1}}\oluithreadfunction{i_1,2}^{a_{1,2}}}{\overline f}
\frac{\oluithreadfunction{i_2,1}^{a_{2,1}}\oluithreadfunction{i_2,2}^{a_{2,2}}}{\oluithreadfunction{i_1,1}^{a_{1,1}}\oluithreadfunction{i_1,2}^{a_{1,2}}}
d\left(\frac{\overline f}{\oluithreadfunction{i_2,1}^{a_{2,1}}\oluithreadfunction{i_2,2}^{a_{2,2}}}\right)\stdvectorfiledpletter[j]\\
-\frac{\oluithreadfunction{i_1,1}^{a_{1,1}}\oluithreadfunction{i_1,2}^{a_{1,2}}}{\overline{f''}}
d\left(\frac{\overline{f''}}{\oluithreadfunction{i_1,1}^{a_{1,1}}\oluithreadfunction{i_1,2}^{a_{1,2}}}\right)\stdvectorfiledpletter[j-1]
+\frac{\oluithreadfunction{i_1,1}^{a_{1,1}}\oluithreadfunction{i_1,2}^{a_{1,2}}}{\overline{f''}}
\frac{\oluithreadfunction{i_2,1}^{a_{2,1}}\oluithreadfunction{i_2,2}^{a_{2,2}}}{\oluithreadfunction{i_1,1}^{a_{1,1}}\oluithreadfunction{i_1,2}^{a_{1,2}}}
d\left(\frac{\overline{f''}}{\oluithreadfunction{i_2,1}^{a_{2,1}}\oluithreadfunction{i_2,2}^{a_{2,2}}}\right)\stdvectorfiledpletter[j]=\\
\left(
\frac{\oluithreadfunction{i_1,1}^{a_{1,1}}\oluithreadfunction{i_1,2}^{a_{1,2}}}{\overline f}
d\left(\frac{\overline f}{\oluithreadfunction{i_1,1}^{a_{1,1}}\oluithreadfunction{i_1,2}^{a_{1,2}}}\right)
-\frac{\oluithreadfunction{i_1,1}^{a_{1,1}}\oluithreadfunction{i_1,2}^{a_{1,2}}}{\overline{f''}}
d\left(\frac{\overline{f''}}{\oluithreadfunction{i_1,1}^{a_{1,1}}\oluithreadfunction{i_1,2}^{a_{1,2}}}\right)
\right)
\stdvectorfiledpletter[j-1]\\
-\left(
\frac{\oluithreadfunction{i_2,1}^{a_{2,1}}\oluithreadfunction{i_2,2}^{a_{2,2}}}{\overline f}
d\left(\frac{\overline f}{\oluithreadfunction{i_2,1}^{a_{2,1}}\oluithreadfunction{i_2,2}^{a_{2,2}}}\right)
-\frac{\oluithreadfunction{i_2,1}^{a_{2,1}}\oluithreadfunction{i_2,2}^{a_{2,2}}}{\overline{f''}}
d\left(\frac{\overline{f''}}{\oluithreadfunction{i_2,1}^{a_{2,1}}\oluithreadfunction{i_2,2}^{a_{2,2}}}\right)
\right)
\stdvectorfiledpletter[j]
\end{multline*}
By a property of logarithmic derivative we can rewrite this as
$$
\frac{\overline{f''}}{\overline f}d\left(\frac{\overline f}{\overline{f''}}\right)\stdvectorfiledpletter[j-1]
-\frac{\overline{f''}}{\overline f}d\left(\frac{\overline f}{\overline{f''}}\right)\stdvectorfiledpletter[j]=
-\frac{\overline f}{\overline{f''}}d\left(\frac{\overline{f''}}{\overline f}\right)
(\stdvectorfiledpletter[j-1]-\stdvectorfiledpletter[j]).
$$
Now we can rewrite $d(\overline {f''}/\overline f)$ as 
$d(\overline{f''}/\overline f)=d(((1+f')\overline f)/\overline f)=df'$.
As we noted before, $df'$ does not have a zero or a pole at $p$. 
We have chosen $f$ and $f''$ so that $\ord_p(\overline f)=\ord_p(\overline{f''})$, hence
$\overline f/\overline{f''}$ does not have a zero or a pole at $p$ either. We conclude that 
$\stdvectorfiledpletter[j-1]-\stdvectorfiledpletter[j]$ is regular at $p$.

Now recall that $\stdvectorfiledpletter[1]=0$, therefore $\stdvectorfiledpletter[j]$ is regular at $p$ for every $j$. Finally, 
it follows from the definition of $\abstractvectorspace_{1,0,p}$ that if $\stdvectorfiledpletter[j]$ is regular at $p$, then $\stdvectorfiledpletter[j]=0$.
\end{proof}

Now we can reformulate Corollary \ref{diffregularcorii} as follows:

\begin{corollary}\label{diffregularcoriii}
Let $\uidescriptionfunction{}=(0,0,0,\uidescriptionfunctiondiff{2}{1},\uidescriptionfunctiondiff{2}{2},0,\ldots,
\uidescriptionfunctiondiff{\numberofverticespt p}{1},\uidescriptionfunctiondiff{\numberofverticespt p}{1},0)\in \abstractvectorspace_{1,0,p}$.
Suppose that 
$\kappa_{\OO,p}(\psi_p(\uidescriptionfunction{}))\in
(\bigoplus_{j=1}^{\numberofverticespt p}\Gamma(W_p,\gvisi{\opensetforvertex{p,j}}))+\Gamma(W,\gvisi{\opensetforvertex{p,j}})$.
Pick two vertices $\indexedvertexpt{p}{j_1}$ and $\indexedvertexpt{p}{j_2}$ of $\stdpolyhedronletter_p$ and denote 
$i_1=\opensetforvertex{p,j_1}$, $i_2=\opensetforvertex{p,j_2}$. Also choose 
$\chi\in\{\lambda_1,\ldots,\lambda_{\numberoflatticegenerators}\}$
and
denote
$a_{1,1}=\uidegree{i_1,1}^*(\chi)$, $a_{1,2}=\uidegree{i_1,2}^*(\chi)$, 
$a_{2,1}=\uidegree{i_2,1}^*(\chi)$, $a_{2,2}=\uidegree{i_2,2}^*(\chi)$.
Let $f\in\Gamma(\PP^1,\OO(\mathcal D(\chi)))$ be an arbitrary function.

Then
$$
\ord_p\left(\frac{\overline f}{\oluithreadfunction{i_1,1}^{a_{1,1}}\oluithreadfunction{i_1,2}^{a_{1,2}}}
(a_{1,1}\uidescriptionfunctiondiff{j_1}{1}+a_{1,2}\uidescriptionfunctiondiff{j_1}{2}-a_{2,1}\uidescriptionfunctiondiff{j_2}{1}-a_{2,2}\uidescriptionfunctiondiff{j_2}{2})\right)
\ge \min (0,\ord_p(\mu_{i_2,i_1,\chi})).
$$\qed
\end{corollary}

When we deal with elements of $\abstractvectorspace_{1,0,p}$ such that all vector fields $\stdvectorfiledpletter[j]$ are zeros, it 
is more convenient to use $U_{\numberoffixedopensets}$-descriptions instead of $U_{\opensetforvertex{p,j}}$-descriptions. 
So denote by $\abstractvectorspace_{1,1,p}$ the space of $2\numberofverticespt p$-tuples of the form 
form
$$
(0,0,\uidescriptionfunctiondiff{2}{1},\uidescriptionfunctiondiff{2}{2},\ldots,
\uidescriptionfunctiondiff{\numberofverticespt p}{1},\uidescriptionfunctiondiff{\numberofverticespt p}{2}),
$$
where
$$
\uidescriptionfunctiondiff{j}{l}=\sum_{k=-n_{j,l}}^{-1}a_{j,l,k}t_p^{k}.
$$
Denote by $\rho_p\colon \abstractvectorspace_{1,1,p}\to\abstractvectorspace_{1,0,p}$ the map that computes $U_{\opensetforvertex{p,j}}$-descriptions out of 
$U_{\numberoffixedopensets}$-descriptions, i.~e.
$\rho_p(0,0,\uidescriptionfunctiondiff{2}{1},\uidescriptionfunctiondiff{2}{2},\ldots,
\uidescriptionfunctiondiff{\numberofverticespt p}{1},\uidescriptionfunctiondiff{\numberofverticespt p}{2})=
(0,0,0,\uidescriptionfunctiondiff{2}{1}',\uidescriptionfunctiondiff{2}{2}',0,\ldots,
\uidescriptionfunctiondiff{\numberofverticespt p}{1}',\uidescriptionfunctiondiff{\numberofverticespt p}{2}',0)$, 
where
$$
\left(
\begin{array}{c}
\uidescriptionfunctiondiff{j}{1}'\\
\uidescriptionfunctiondiff{j}{2}'\\
0
\end{array}
\right)
=
\uilargetransition{\numberoffixedopensets,\opensetforvertex{p,j}}
\left(
\begin{array}{c}
\uidescriptionfunctiondiff{j}{1}\\
\uidescriptionfunctiondiff{j}{2}\\
0
\end{array}
\right).
$$
In other words,
$$
\left(
\begin{array}{c}
\uidescriptionfunctiondiff{j}{1}'\\
\uidescriptionfunctiondiff{j}{2}'\\
\end{array}
\right)=
\uismalltransition{\numberoffixedopensets,\opensetforvertex{p,j}}
\left(
\begin{array}{c}
\uidescriptionfunctiondiff{j}{1}\\
\uidescriptionfunctiondiff{j}{2}\\
\end{array}
\right).
$$
Clearly, $\rho_p$ is injective. It also follows from Lemma \ref{vfieldzero} that $\rho_p(\abstractvectorspace_{1,1,p})$ contains 
$$
\abstractvectorspace_{1,0,p}\cap \psi_p^{-1}\left(\kappa_{\OO,p}^{-1}\left(\bigg(\bigoplus_{j=1}^{\numberofverticespt p}
\Gamma\left(W_p,\gvisiinv{\opensetforvertex{p,j}}\right)\bigg)+\Gamma\left(W_p,\gvicircinv\right)\right)\right).
$$
So now we are going to find the following preimage: 
$$
\rho_p^{-1}\left(\psi_p^{-1}\left(\kappa_{\OO,p}^{-1}\left(\bigg(\bigoplus_{j=1}^{\numberofverticespt p}
\Gamma\left(W_p,\gvisiinv{\opensetforvertex{p,j}}\right)\bigg)+\Gamma\left(W_p,\gvicircinv\right)\right)\right)\right).
$$

\begin{lemma}\label{linfunczero}
Let 
$$
\uidescriptionfunction{}=(0,0,\uidescriptionfunctiondiff{2}{1},\uidescriptionfunctiondiff{2}{2},\ldots,
\uidescriptionfunctiondiff{\numberofverticespt p}{1},\uidescriptionfunctiondiff{\numberofverticespt p}{2})\in\abstractvectorspace_{1,1,p}
$$
be such that 
$$
\kappa_{\OO,p}(\psi_p(\rho_p(\uidescriptionfunction{})))\in
\bigg(\bigoplus_{j=1}^{\numberofverticespt p}\Gamma\left(W_p,\gvisi{\opensetforvertex{p,j}}\right)\bigg)+\Gamma\left(W_p,\gvicircinv\right).
$$
Pick two vertices $\indexedvertexpt{p}{j_1}$ and $\indexedvertexpt{p}{j_2}$ of $\stdpolyhedronletter_p$, choose 
$\chi\in\{\lambda_1,\ldots,\lambda_{\numberoflatticegenerators}\}\cap \normalvertexcone{\indexedvertexpt{p}{j_1}}{\stdpolyhedronletter_p}$.

Then
$$
\ord_p(\uidegree{\numberoffixedopensets,1}^*(\chi)(\uidescriptionfunctiondiff{j_1}{1}-\uidescriptionfunctiondiff{j_2}{1})+
\uidegree{\numberoffixedopensets,2}^*(\chi)(\uidescriptionfunctiondiff{j_1}{2}-\uidescriptionfunctiondiff{j_2}{2}))\ge 
\chi(\indexedvertexpt{p}{j_1}-\indexedvertexpt{p}{j_2}).
$$
\end{lemma}
\begin{proof}
Denote 
$i_1=\opensetforvertex{p,j_1}$, $i_2=\opensetforvertex{p,j_2}$,
$b_1=\uidegree{\numberoffixedopensets,1}^*(\chi)$, 
$b_2=\uidegree{\numberoffixedopensets,2}^*(\chi)$.

Denote also $a_{1,1}=\uidegree{i_1,1}^*(\chi)$, $a_{1,2}=\uidegree{i_1,2}^*(\chi)$, 
$a_{2,1}=\uidegree{i_1,1}^*(\chi)$, $a_{2,2}=\uidegree{i_1,2}^*(\chi)$.
Since $\chi=b_1 \uidegree{\numberoffixedopensets,1}+b_2 \uidegree{\numberoffixedopensets,2}$, 
we can write 
$a_{1,1}=b_1\uidegree{i_1,1}^*(\uidegree{\numberoffixedopensets,1})+b_2\uidegree{i_1,1}^*(\uidegree{\numberoffixedopensets,2})$, 
$a_{1,2}=b_1\uidegree{i_1,2}^*(\uidegree{\numberoffixedopensets,1})+b_2\uidegree{i_1,2}^*(\uidegree{\numberoffixedopensets,2})$.
These equalities can be written in a matrix form:
$$
\left(
\begin{array}{cc}
a_{1,1} & a_{1,2}
\end{array}
\right)=
\left(
\begin{array}{cc}
b_1 & b_2
\end{array}
\right)
\uismalltransition{i_1,\numberoffixedopensets}.
$$
Similarly,
$$
\left(
\begin{array}{cc}
a_{2,1} & a_{2,2}
\end{array}
\right)=
\left(
\begin{array}{cc}
b_1 & b_2
\end{array}
\right)
\uismalltransition{i_2,\numberoffixedopensets}.
$$
Denote
$$
\left(
\begin{array}{c}
\uidescriptionfunctiondiff{j_1}{1}'\\
\uidescriptionfunctiondiff{j_1}{2}'\\
\end{array}
\right)=
\uismalltransition{\numberoffixedopensets,i_1}
\left(
\begin{array}{c}
\uidescriptionfunctiondiff{j_1}{1}\\
\uidescriptionfunctiondiff{j_1}{2}\\
\end{array}
\right)
\text{\quad and\quad}
\left(
\begin{array}{c}
\uidescriptionfunctiondiff{j_2}{1}'\\
\uidescriptionfunctiondiff{j_2}{2}'\\
\end{array}
\right)=
\uismalltransition{\numberoffixedopensets,i_2}
\left(
\begin{array}{c}
\uidescriptionfunctiondiff{j_2}{1}\\
\uidescriptionfunctiondiff{j_2}{2}\\
\end{array}
\right).
$$
Then by Lemma \ref{rijcocycle},
$$
\left(
\begin{array}{c}
\uidescriptionfunctiondiff{j_1}{1}\\
\uidescriptionfunctiondiff{j_1}{2}\\
\end{array}
\right)=
\uismalltransition{i_1,\numberoffixedopensets}
\left(
\begin{array}{c}
\uidescriptionfunctiondiff{j_1}{1}'\\
\uidescriptionfunctiondiff{j_1}{2}'\\
\end{array}
\right)
\text{\quad and\quad}
\left(
\begin{array}{c}
\uidescriptionfunctiondiff{j_2}{1}\\
\uidescriptionfunctiondiff{j_2}{2}\\
\end{array}
\right)=
\uismalltransition{i_2,\numberoffixedopensets}
\left(
\begin{array}{c}
\uidescriptionfunctiondiff{j_2}{1}'\\
\uidescriptionfunctiondiff{j_2}{2}'\\
\end{array}
\right).
$$
We can write
\begin{multline*}
b_1(\uidescriptionfunctiondiff{j_1}{1}-\uidescriptionfunctiondiff{j_2}{1})+b_2(\uidescriptionfunctiondiff{j_1}{2}-\uidescriptionfunctiondiff{j_2}{2})=\\
\left(
\begin{array}{cc}
b_1 & b_2
\end{array}
\right)
\left(
\begin{array}{c}
\uidescriptionfunctiondiff{j_1}{1}\\
\uidescriptionfunctiondiff{j_1}{2}\\
\end{array}
\right)
-
\left(
\begin{array}{cc}
b_1 & b_2
\end{array}
\right)
\left(
\begin{array}{c}
\uidescriptionfunctiondiff{j_2}{1}\\
\uidescriptionfunctiondiff{j_2}{2}\\
\end{array}
\right)=\\
\left(
\begin{array}{cc}
b_1 & b_2
\end{array}
\right)
\uismalltransition{i_1,\numberoffixedopensets}
\left(
\begin{array}{c}
\uidescriptionfunctiondiff{j_1}{1}'\\
\uidescriptionfunctiondiff{j_1}{2}'\\
\end{array}
\right)
-
\left(
\begin{array}{cc}
b_1 & b_2
\end{array}
\right)
\uismalltransition{i_2,\numberoffixedopensets}
\left(
\begin{array}{c}
\uidescriptionfunctiondiff{j_2}{1}'\\
\uidescriptionfunctiondiff{j_2}{2}'\\
\end{array}
\right)=\\
\left(
\begin{array}{cc}
a_{1,1} & a_{1,2}
\end{array}
\right)
\left(
\begin{array}{c}
\uidescriptionfunctiondiff{j_1}{1}'\\
\uidescriptionfunctiondiff{j_1}{2}'\\
\end{array}
\right)
-
\left(
\begin{array}{cc}
a_{2,1} & a_{2,2}
\end{array}
\right)
\left(
\begin{array}{c}
\uidescriptionfunctiondiff{j_2}{1}'\\
\uidescriptionfunctiondiff{j_2}{2}'\\
\end{array}
\right)=\\
a_{1,1}\uidescriptionfunctiondiff{j_1}{1}'+a_{1,2}\uidescriptionfunctiondiff{j_1}{2}'-a_{2,1}\uidescriptionfunctiondiff{j_2}{1}'-a_{2,2}\uidescriptionfunctiondiff{j_2}{2}'.
\end{multline*}

By Lemma \ref{giexist}, there exists $f\in \Gamma(\PP^1,\OO(\mathcal D(\chi)))$ such that $\ord_p(\overline f)=-\mathcal D_p(\chi)$.
Since $\uidegree{i_1,1}, \uidegree{i_1,2}, \chi\in \normalvertexcone{\indexedvertexpt{p}{j_1}}{\stdpolyhedronletter_p}$, 
$\mathcal D_p(\cdot)$ is linear on $\normalvertexcone{\indexedvertexpt{p}{j_1}}{\stdpolyhedronletter_p}$, and 
$\chi=a_{1,1}\uidegree{i_1,1}+a_{1,2}\uidegree{i_1,2}$, we can write 
$\mathcal D_p(\chi)=a_{1,1}\mathcal D_p(\uidegree{i_1,1})+a_{1,2}\mathcal D_p(\uidegree{i_1,2})$. We chose 
$\uithreadfunction{i_1, 1}$ and $\uithreadfunction{i_1, 2}$ so that 
$\ord_p(\oluithreadfunction{i_1,1})=-\mathcal D_p(\uidegree{i_1,1})$, $\ord_p(\oluithreadfunction{i_1,2})=-\mathcal D_p(\uidegree{i_1,2})$.
Therefore, 
$$
\ord_p\left(\frac{\overline f}{\oluithreadfunction{i_1,1}^{a_{1,1}}\oluithreadfunction{i_1,2}^{a_{1,2}}}\right)=0.
$$

By Lemma \ref{linfuncdecomp}, $\ord_p(\mu_{i_2,i_1,\chi})=\chi(\indexedvertexpt{p}{j_1}-\indexedvertexpt{p}{j_2})$. 
Since $\chi\in\normalvertexcone{\indexedvertexpt{p}{j_1}}{\stdpolyhedronletter_p}$, 
$\indexedvertexpt{p}{j_1}$ is a point where $\chi$ attains its minimum on $\stdpolyhedronletter_p$.
Hence, $\ord_p(\mu_{i_2,i_1,\chi})\le 0$. The claim now follows from Corollary \ref{diffregularcoriii}.
\end{proof}

%

Now we are ready to formulate an exact description for 
$$
\rho_p^{-1}\left(\psi_p^{-1}\left(\kappa_{\OO,p}^{-1}\left(\bigg(\bigoplus_{j=1}^{\numberofverticespt p}\Gamma\left(W_p,\gvisi{\opensetforvertex{p,j}}\right)\bigg)
+\Gamma\left(W_p,\gvicircinv\right)\right)\right)\right).
$$
Let $\abstractvectorspace_{1,2,p}\subseteq \abstractvectorspace_{1,1,p}$ be the space of $2\numberofverticespt p$-tuples of the form
$$
(\uidescriptionfunctiondiff{1}{1},\uidescriptionfunctiondiff{1}{2},\uidescriptionfunctiondiff{2}{1},\uidescriptionfunctiondiff{2}{2},\ldots,
\uidescriptionfunctiondiff{\numberofverticespt p}{1},\uidescriptionfunctiondiff{\numberofverticespt p}{2})
$$
such that
\begin{enumerate}
\item 
\label{proplaurent}
$\uidescriptionfunctiondiff{j}{k}$ is a Laurent polynomial in $t_p$ with no terms of nonnegative degree.
\item 
\label{propfirstzero}
$\uidescriptionfunctiondiff{1}{1}=\uidescriptionfunctiondiff{1}{2}=0$.
\item 
\label{proplinfunczero}
For each $j$ such that $\indexededgept{p}{j}$ is a finite edge (i.~e. $1\le j<\numberofverticespt{p}$),
$$
\uidegree{\numberoffixedopensets,1}^*(\primitivelattice{\normalvertexcone{\indexededgept{p}{j}}{\stdpolyhedronletter_p}})(\uidescriptionfunctiondiff{j}{1}-\uidescriptionfunctiondiff{j+1}{1})+
\uidegree{\numberoffixedopensets,2}^*(\primitivelattice{\normalvertexcone{\indexededgept{p}{j}}{\stdpolyhedronletter_p}})(\uidescriptionfunctiondiff{j}{2}-\uidescriptionfunctiondiff{j+1}{2})=0.
$$
\item 
\label{proplinfuncestimate}
$\ord_p(\uidescriptionfunctiondiff{j}{1}-\uidescriptionfunctiondiff{j+1}{1})\ge -\latticelength{\indexededgept{p}{j}}$, 
$\ord_p(\uidescriptionfunctiondiff{j}{2}-\uidescriptionfunctiondiff{j+1}{2})\ge -\latticelength{\indexededgept{p}{j}}$ for all finite edges $\indexededgept{p}{j}$
($1\le j<\numberofverticespt p$).
\end{enumerate}

\begin{remark}\label{avs22pdim}
$\dim \abstractvectorspace_{1,2,p}=\latticelength{\indexededgept{p}{1}}+\ldots+\latticelength{\indexededgept{p}{\numberofverticespt p-1}}$.
\end{remark}

\begin{proposition}\label{avs22p}
$$
\abstractvectorspace_{1,2,p}=
\rho_p^{-1}\left(\psi_p^{-1}\left(\kappa_{\OO,p}^{-1}\left(\bigg(\bigoplus_{j=1}^{\numberofverticespt p}
\Gamma\left(W_p,\gvisiinv{\opensetforvertex{p,j}}\right)\bigg)+\Gamma\left(W_p,\gvicircinv\right)\right)\right)\right).
$$
\end{proposition}

\begin{proof}
The inclusion 
$$
\abstractvectorspace_{1,2,p}\supseteq 
\rho_p^{-1}\left(\psi_p^{-1}\left(\kappa_{\OO,p}^{-1}\left(\bigg(\bigoplus_{j=1}^{\numberofverticespt p}\Gamma\left(W_p,\gvisiinv{\opensetforvertex{p,j}}\right)\bigg)
+\Gamma\left(W_p,\gvicircinv\right)\right)\right)\right)
$$
follows easily from Lemmas \ref{linfunczero} and \ref{linfuncestimate}. Namely, 
let 
$$
\uidescriptionfunction{}\in
\rho_p^{-1}\left(\psi_p^{-1}\left(\kappa_{\OO,p}^{-1}\left(\bigg(\bigoplus_{j=1}^{\numberofverticespt p}\Gamma\left(W_p,\gvisiinv{\opensetforvertex{p,j}}\right)\bigg)
+\Gamma\left(W_p,\gvicircinv\right)\right)\right)\right),
$$
$$
\uidescriptionfunction{}=(\uidescriptionfunctiondiff{1}{1},\uidescriptionfunctiondiff{1}{2},\uidescriptionfunctiondiff{2}{1},\uidescriptionfunctiondiff{2}{2},\ldots,
\uidescriptionfunctiondiff{\numberofverticespt p}{1},\uidescriptionfunctiondiff{\numberofverticespt p}{2}).
$$
Properties \ref{proplaurent} and \ref{propfirstzero} in the definition of $\abstractvectorspace_{1,2,p}$ follow from the definition of $\abstractvectorspace_{1,1,p}$.
Fix a finite edge $\indexededgept{p}{j}$, $1\le j < \numberofverticespt p$.
Let 
$\chi=\primitivelattice{\normalvertexcone{\indexededgept{p}{j}}{\stdpolyhedronletter_p}}$.
According to our choice of the set $\{\lambda_1,\ldots,\lambda_{\numberoflatticegenerators}\}$, $\chi\in \{\lambda_1,\ldots,\lambda_{\numberoflatticegenerators}\}$.
There also exists a degree $\chi'\in\{\lambda_1,\ldots,\lambda_{\numberoflatticegenerators}\}$ 
such that $\chi'\in\normalvertexcone{\indexedvertexpt{p}{j}}{\stdpolyhedronletter_p}$ and $\chi$ and $\chi'$ 
form a basis of $M$. 
By Lemma \ref{linfunczero}, 
$$
\ord_p(\uidegree{\numberoffixedopensets,1}^*(\chi)(\uidescriptionfunctiondiff{j}{1}-\uidescriptionfunctiondiff{j+1}{1})+
\uidegree{\numberoffixedopensets,2}^*(\chi)(\uidescriptionfunctiondiff{j}{2}-\uidescriptionfunctiondiff{j+1}{2}))\ge 
\chi(\indexedvertexpt{p}{j}-\indexedvertexpt{p}{j+1}).
$$
By Lemma \ref{linfuncestimate}, 
$\chi(\indexedvertexpt{p}{j}-\indexedvertexpt{p}{j+1})=0$, 
in other words, 
$$
\uidegree{\numberoffixedopensets,1}^*(\chi)(\uidescriptionfunctiondiff{j}{1}-\uidescriptionfunctiondiff{j+1}{1})+
\uidegree{\numberoffixedopensets,2}^*(\chi)(\uidescriptionfunctiondiff{j}{2}-\uidescriptionfunctiondiff{j+1}{2})
$$
is a function regular at $p$. On the other hand, it is a Laurent polynomial whose terms of nonnegative degree are zeros, so 
$$
\uidegree{\numberoffixedopensets,1}^*(\chi)(\uidescriptionfunctiondiff{j}{1}-\uidescriptionfunctiondiff{j+1}{1})+
\uidegree{\numberoffixedopensets,2}^*(\chi)(\uidescriptionfunctiondiff{j}{2}-\uidescriptionfunctiondiff{j+1}{2})=0.
$$
Now, using Lemmas \ref{linfunczero} and \ref{linfuncestimate} again, 
we see that 
$$
\ord_p(\uidegree{\numberoffixedopensets,1}^*(\chi')(\uidescriptionfunctiondiff{j}{1}-\uidescriptionfunctiondiff{j+1}{1})
+\uidegree{\numberoffixedopensets,1}^*(\chi')(\uidescriptionfunctiondiff{j}{2}-\uidescriptionfunctiondiff{j+1}{2}))\ge -\latticelength{\indexededgept{p}{j}}.
$$
Since $\uidegree{\numberoffixedopensets,1}^*$ and $\uidegree{\numberoffixedopensets,2}^*$ form 
a basis of $N$, and 
$\chi$ and $\chi'$ 
form a basis of $M$, 
the matrix
$$
\left(
\begin{array}{cc}
\uidegree{\numberoffixedopensets,1}^*(\chi) & \uidegree{\numberoffixedopensets,2}^*(\chi) \\
\uidegree{\numberoffixedopensets,1}^*(\chi') & \uidegree{\numberoffixedopensets,2}^*(\chi') 
\end{array}
\right)
$$
is nondegenerate.
Therefore,
$\ord_p((\uidescriptionfunctiondiff{j}{1}-\uidescriptionfunctiondiff{j+1}{1}))\ge -\latticelength{\indexededgept{p}{j}}$ and 
$\ord_p((\uidescriptionfunctiondiff{j}{2}-\uidescriptionfunctiondiff{j+1}{2}))\ge -\latticelength{\indexededgept{p}{j}}$. 
So, the conditions \ref{proplinfunczero} and \ref{proplinfuncestimate} from the definition 
of $\abstractvectorspace_{1,2,p}$ hold, and $\uidescriptionfunction{}\in\abstractvectorspace_{1,2,p}$.

Now we are going to prove the other inclusion. 
Let 
$$
\uidescriptionfunction{}=
(\uidescriptionfunctiondiff{1}{1},\uidescriptionfunctiondiff{1}{2},\uidescriptionfunctiondiff{2}{1},\uidescriptionfunctiondiff{2}{2},\ldots,
\uidescriptionfunctiondiff{\numberofverticespt p}{1},\uidescriptionfunctiondiff{\numberofverticespt p}{2})\in\abstractvectorspace_{1,2,p}.
$$
We have to write $\psi_p(\rho_p(\uidescriptionfunction{}))$
as $f+f'$, where 
\begin{multline*}
f=(f[j]_{\chi,k})_{1\le j\le \numberofverticespt p, \chi\in\{\lambda_1,\ldots,\lambda_{\numberoflatticegenerators}\}, 1\le k\le \dim\Gamma(\PP^1,\OO(\mathcal D(\chi)))}\in
\bigoplus_{j=1}^{\numberofverticespt{p}}\gsvi{p,j},\\
f'=(f'_{j,\chi,k})_{1\le j\le \numberofverticespt p, \chi\in\{\lambda_1,\ldots,\lambda_{\numberoflatticegenerators}\}, 1\le k\le \dim\Gamma(\PP^1,\OO(\mathcal D(\chi)))}\in
\bigoplus_{j=1}^{\numberofverticespt{p}}\gsvicirc{p,j},
\end{multline*}
and, in addition, 
$$
\kappa_{\OO,p,j}((f'_{j,\chi,k})_{\chi\in\{\lambda_1,\ldots,\lambda_{\numberoflatticegenerators}\}, 
1\le k\le \dim\Gamma(\PP^1,\OO(\mathcal D(\chi)))})
$$ 
does not depend on $j$.
In other words, 
we have to find 
functions $f[j]_{\chi,k}$ 
regular at $p$ 
and 
functions
$f'_{j,\chi,k}$ such that (see the definition of $\kappa_{\OO,p,j}$)
$f'_{j_1,\chi,k}=\mu_{j_2,j_1,\chi}f'_{j_2,\chi,k}$ for each $j_1$, $j_2$.
These conditions can be verified for different degrees $\chi$ and different indices $k$ independently, so 
fix a degree 
$\chi\in\{\lambda_1,\ldots,\lambda_{\numberoflatticegenerators}\}$
and 
a generator $\dependentgeneratorsdegree{\chi,k}$
until the end of the proof. 
Denote $a_1=\uidegree{\numberoffixedopensets,1}^*(\chi)$, $a_2=\uidegree{\numberoffixedopensets,2}^*(\chi)$.

The map $\psi_p$ uses $U_{\opensetforvertex{p,j}}$-descriptions of functions and of vector fields on $U_{\numberoffixedopensets}$, 
but it follows from the definitions of
$\psi_p$, of an $U_{\opensetforvertex{p,j}}$-description and of an 
$U_{\numberoffixedopensets}$-description that instead of computing the $(j,\chi,k)$th component 
of $\psi_p(\rho_p(\uidescriptionfunction{}))$ using $\psi_p$ and $\rho_p$, we can first 
compute the $U_{\numberoffixedopensets}$-description of the derivative of $\dependentgeneratorsdegree{\chi,k}$ 
along the vector field on $U_{\numberoffixedopensets}$ whose $U_{\numberoffixedopensets}$-description is 
$(\uidescriptionfunctiondiff{j}{1},\uidescriptionfunctiondiff{j}{2},0)$, 
and then use $\mu_{\numberoffixedopensets,i_{p,j},\chi}$ to compute 
the $U_{\opensetforvertex{p,j}}$-description of the function on $U_{\numberoffixedopensets}$
whose $U_{\numberoffixedopensets}$-description we obtain this way. 
So, consider the $U_{\numberoffixedopensets}$-descriptions of the functions on $U_{\numberoffixedopensets}$ whose 
$U_{\opensetforvertex{p,j}}$-descriptions are functions $f[j]_{\chi,k}$ and $f'_{j,\chi,k}$ we are looking for.
Denote these $U_{\numberoffixedopensets}$-descriptions by by $f[j]''$ and $f'''_j$, respectively (we do not use 
indices $\chi$ and $k$ here, because they are already fixed until the end of the proof, and we do not mean 
that these functions are the same for different $\chi$ and $k$). In other words, 
$f[j]_{\chi,k}=\mu_{\numberoffixedopensets,\opensetforvertex{p,j}, \chi}f[j]''$ and
$f'_{j,\chi,k}=\mu_{\numberoffixedopensets,\opensetforvertex{p,j_1},\chi}f'''_{j_1}$.
In terms of these functions, we need to meet the following conditions:
first, $\mu_{\numberoffixedopensets,\opensetforvertex{p,j}, \chi}f[j]''$ should be regular at $p$ for each $j$, and second, 
$\mu_{\numberoffixedopensets,\opensetforvertex{p,j_1},\chi}f'''_{j_1}$ and 
$\mu_{\numberoffixedopensets,\opensetforvertex{p,j_2},\chi}f'''_{j_2}$
should be the $U_{\opensetforvertex{p,j_1}}$- and $U_{\opensetforvertex{p,j_2}}$-descriptions 
(respectively) of the same function defined on $U_{\numberoffixedopensets}$. These conditions can be reformulated as follows:
the inequality $\ord_p(\mu_{\numberoffixedopensets,\opensetforvertex{p,j},\chi}f[j]''))\ge 0$ 
should hold, and all functions $f'''_j$ should be the same function $f'''$, 
which should not depend on $j$.

Let $j_1$ be the maximal index such that $\chi\in\normalvertexcone{\indexedvertexpt{p}{j_1}}{\stdpolyhedronletter_p}$. 
(The convention that we take the maximal index is nontrivial if $\chi\in\normalvertexcone{\indexededgept{p}{j_1-1}}{\stdpolyhedronletter_p}$.) 
Fix this index $j_1$ until the end of the proof. 
Set 
$$
f'''=
\frac{\oldependentgeneratorsdegree{\chi,k}}{\oluithreadfunction{\numberoffixedopensets,1}^{a_1}\oluithreadfunction{\numberoffixedopensets,2}^{a_2}}
(a_1 \uidescriptionfunctiondiff{j_1}{1}+a_2 \uidescriptionfunctiondiff{j_1}{2}),
$$
and for each $j_2$ ($1\le j_2\le \numberofverticespt{p}$) set
$$
f[j_2]''=
\frac{\oldependentgeneratorsdegree{\chi,k}}{\oluithreadfunction{\numberoffixedopensets,1}^{a_1}\oluithreadfunction{\numberoffixedopensets,2}^{a_2}}
(a_1 \uidescriptionfunctiondiff{j_2}{1}+a_2 \uidescriptionfunctiondiff{j_2}{1})-f'''.
$$
Observe that $f[j_1]''=0$.
By Lemma \ref{uidescpsi}, $f[j_2]''+f'''$ is the $U_{\numberoffixedopensets}$-description 
of the derivative of $\wtdependentgeneratorsdegree{\chi,k}$
along the vector field whose $U_{\numberoffixedopensets}$-description 
is $(\uidescriptionfunctiondiff{j_2}{1},\uidescriptionfunctiondiff{j_2}{2},0)$. It is sufficient to prove that 
$\ord_p(\mu_{\numberoffixedopensets,\opensetforvertex{p,j_2},\chi}f[j_2]'')\ge 0$. 
Denote 
$b_{1,1}=\uidegree{\opensetforvertex{p,j_1},1}^*(\chi)$, 
$b_{1,2}=\uidegree{\opensetforvertex{p,j_1},2}^*(\chi)$,
$b_{2,1}=\uidegree{\opensetforvertex{p,j_2},1}^*(\chi)$, and 
$b_{2,2}=\uidegree{\opensetforvertex{p,j_2},2}^*(\chi)$.
Then we can write this function as follows:
\begin{multline*}
\mu_{\numberoffixedopensets,\opensetforvertex{p,j_2},\chi}f[j_2]''=
\frac{\oluithreadfunction{\numberoffixedopensets,1}^{a_1}\oluithreadfunction{\numberoffixedopensets,2}^{a_2}}
{\oluithreadfunction{\opensetforvertex{p,j_2},1}^{b_{2,1}}\oluithreadfunction{\opensetforvertex{p,j_2},2}^{b_{2,2}}}
\frac{\oldependentgeneratorsdegree{\chi,k}}
{\oluithreadfunction{\numberoffixedopensets,1}^{a_1}\oluithreadfunction{\numberoffixedopensets,2}^{a_2}}
((a_1 \uidescriptionfunctiondiff{j_2}{1}+a_2 \uidescriptionfunctiondiff{j_2}{2})-(a_1 \uidescriptionfunctiondiff{j_1}{1}+a_2 \uidescriptionfunctiondiff{j_1}{2}))
=\\
\frac{\oldependentgeneratorsdegree{\chi,k}}
{\oluithreadfunction{\opensetforvertex{p,j_2},1}^{b_{2,1}}\oluithreadfunction{\opensetforvertex{p,j_2},2}^{b_{2,2}}}
(a_1(\uidescriptionfunctiondiff{j_2}{1}-\uidescriptionfunctiondiff{j_1}{1})+a_2(\uidescriptionfunctiondiff{j_2}{2}-\uidescriptionfunctiondiff{j_1}{2}))
=\\
\frac{\oldependentgeneratorsdegree{\chi,k}}
{\oluithreadfunction{\opensetforvertex{p,j_1},1}^{b_{1,1}}\oluithreadfunction{\opensetforvertex{p,j_1},2}^{b_{1,2}}}
\frac{\oluithreadfunction{\opensetforvertex{p,j_1},1}^{b_{1,1}}\oluithreadfunction{\opensetforvertex{p,j_1},2}^{b_{1,2}}}
{\oluithreadfunction{\opensetforvertex{p,j_2},1}^{b_{2,1}}\oluithreadfunction{\opensetforvertex{p,j_2},2}^{b_{2,2}}}
(a_1(\uidescriptionfunctiondiff{j_2}{1}-\uidescriptionfunctiondiff{j_1}{1})+a_2(\uidescriptionfunctiondiff{j_2}{2}-\uidescriptionfunctiondiff{j_1}{2}))
=\\
\frac{\oldependentgeneratorsdegree{\chi,k}}
{\oluithreadfunction{\opensetforvertex{p,j_1},1}^{b_{1,1}}\oluithreadfunction{\opensetforvertex{p,j_1},2}^{b_{1,2}}}
\mu_{\opensetforvertex{p,j_1},\opensetforvertex{p,j_2},\chi}
(a_1(\uidescriptionfunctiondiff{j_2}{1}-\uidescriptionfunctiondiff{j_1}{1})+a_2(\uidescriptionfunctiondiff{j_2}{2}-\uidescriptionfunctiondiff{j_1}{2})).
\end{multline*}
Since $\dependentgeneratorsdegree{\chi,k}\in\Gamma(\PP^1,\OO(\mathcal D(\chi)))$, $\ord_p(\oldependentgeneratorsdegree{\chi,k})\ge-\mathcal D_p(\chi)$. 
We chose $\uithreadfunction{\opensetforvertex{p,j_1},1}$ and $\uithreadfunction{\opensetforvertex{p,j_1},2}$
so that $\ord_p(\oluithreadfunction{\opensetforvertex{p,j_1},1})=-\mathcal D_p(\uidegree{\opensetforvertex{p,j_1},1})$ and
$\ord_p(\oluithreadfunction{\opensetforvertex{p,j_1},2})=-\mathcal D_p(\uidegree{\opensetforvertex{p,j_1},2})$. 
We know that 
$$
\chi=b_{1,1}\uidegree{\opensetforvertex{p,j_1},1}+b_{1,2}\uidegree{\opensetforvertex{p,j_1},2},
$$ 
$\chi, \uidegree{\opensetforvertex{p,j_1},1}, \uidegree{\opensetforvertex{p,j_1},2}\in
\normalvertexcone{\indexedvertexpt{p}{j_1}}{\stdpolyhedronletter_p}$, and $\mathcal D_p(\cdot)$ 
is linear on $\normalvertexcone{\indexedvertexpt{p}{j_1}}{\stdpolyhedronletter_p}$, 
therefore 
$$
\ord_p(\oluithreadfunction{\opensetforvertex{p,j_1},1}^{b_{1,1}}\oluithreadfunction{\opensetforvertex{p,j_1},2}^{b_{1,2}})
=-b_{1,1}\mathcal D_p(\uidegree{\opensetforvertex{p,j_1},1})-b_{1,2}\mathcal D_p(\uidegree{\opensetforvertex{p,j_1},2})=-\mathcal D_p(\chi).
$$
Hence,
$$
\ord_p\left(
\frac{\oldependentgeneratorsdegree{\chi,k}}{\oluithreadfunction{\opensetforvertex{p,j_1},1}^{b_{1,1}}\oluithreadfunction{\opensetforvertex{p,j_1},2}^{b_{1,2}}}
\right)
\ge 0.
$$

So, now we are done for $j_2=j_1$. Otherwise, we have to consider two cases: $j_2>j_1$ and $j_2<j_1$.
Suppose first that $j_2>j_1$. Then 
\begin{multline*}
\mu_{\opensetforvertex{p,j_1},\opensetforvertex{p,j_2},\chi}
(a_1(\uidescriptionfunctiondiff{j_2}{1}-\uidescriptionfunctiondiff{j_1}{1})+a_2(\uidescriptionfunctiondiff{j_2}{2}-\uidescriptionfunctiondiff{j_1}{2}))=\\
\mu_{\opensetforvertex{p,j_1},\opensetforvertex{p,j_1+1},\chi}\ldots\mu_{\opensetforvertex{p,j_2-1},\opensetforvertex{p,j_2},\chi}
(a_1(\uidescriptionfunctiondiff{j_2}{1}-\uidescriptionfunctiondiff{j_2-1}{1})+a_2(\uidescriptionfunctiondiff{j_2}{2}-\uidescriptionfunctiondiff{j_2-1}{2})+\\
\ldots+
a_1(\uidescriptionfunctiondiff{j_1+1}{1}-\uidescriptionfunctiondiff{j_1}{1})+a_2(\uidescriptionfunctiondiff{j_1+1}{2}-\uidescriptionfunctiondiff{j_1}{2})).
\end{multline*}
By Lemma \ref{linfuncdecomp},
$$
\ord_p(\mu_{\opensetforvertex{p,j_1},\opensetforvertex{p,j_1+1},\chi}\ldots
\mu_{\opensetforvertex{p,j_2-1},\opensetforvertex{p,j_2},\chi})=
\chi(\indexedvertexpt{p}{j_1+1}-\indexedvertexpt{p}{j_1})+\ldots+\chi(\indexedvertexpt{p}{j_2}-\indexedvertexpt{p}{j_2-1}).
$$
Since $\chi\notin\normalvertexcone{\indexedvertexpt{p}{j}}{\stdpolyhedronletter_p}$ for all $j>j_1$, 
by Lemma \ref{linfuncchain} we have
$$
\ord_p(\mu_{\opensetforvertex{p,j_1},\opensetforvertex{p,j_1+1},\chi}\ldots
\mu_{\opensetforvertex{p,j_2-1},\opensetforvertex{p,j_2},\chi})\ge 
\latticelength{\indexededgept{p}{j_1}}+\latticelength{\indexededgept{p}{j_1+1}}+\ldots+\latticelength{\indexededgept{p}{j_2-1}}.
$$
This sum contains at least one summand since $j_2>j_1$.
By the definition of $\abstractvectorspace_{1,2,p}$, 
\begin{multline*}
\ord_p(a_1(\uidescriptionfunctiondiff{j_2}{1}-\uidescriptionfunctiondiff{j_2-1}{1})+
a_2(\uidescriptionfunctiondiff{j_2}{2}-\uidescriptionfunctiondiff{j_2-1}{2})+\\
\ldots+
a_1(\uidescriptionfunctiondiff{j_1+1}{1}-\uidescriptionfunctiondiff{j_1}{1})+
a_2(\uidescriptionfunctiondiff{j_1+1}{2}-\uidescriptionfunctiondiff{j_1}{2}))\ge\\
\min(-\latticelength{\indexededgept{p}{j_2-1}},\ldots,-\latticelength{\indexededgept{p}{j_1}})=
-\max(\latticelength{\indexededgept{p}{j_1}},\ldots,\latticelength{\indexededgept{p}{j_2-1}}).
\end{multline*}
We have 
$$\latticelength{\indexededgept{p}{j_1}}+\latticelength{\indexededgept{p}{j_1+1}}+\ldots+\latticelength{\indexededgept{p}{j_2-1}}
-\max(\latticelength{\indexededgept{p}{j_1}},\ldots,\latticelength{\indexededgept{p}{j_2-1}})\ge0,
$$
therefore
$\mu_{\numberoffixedopensets,\opensetforvertex{p,j_2},\chi}f[j_2]''$ is regular at $p$.

Now consider the case $j_2<j_1$. This time we are going to consider indices 
smaller than $j_1$, and it is possible that 
$\chi\in\normalvertexcone{\indexedvertexpt{p}{j}}{\stdpolyhedronletter_p}$ for some 
$j<j_1$, namely for $j=j_1-1$ (and this is the only possibility). So, we have to consider two 
cases: $\chi\notin\normalvertexcone{\indexedvertexpt{p}{j_1-1}}{\stdpolyhedronletter_p}$ and 
$\chi\in\normalvertexcone{\indexedvertexpt{p}{j_1-1}}{\stdpolyhedronletter_p}$. 
Suppose first that $\chi\notin\normalvertexcone{\indexedvertexpt{p}{j_1-1}}{\stdpolyhedronletter_p}$.
Then we can again write 
\begin{multline*}
\mu_{\opensetforvertex{p,j_1},\opensetforvertex{p,j_2},\chi}
(a_1(\uidescriptionfunctiondiff{j_2}{1}-\uidescriptionfunctiondiff{j_1}{1})+a_2(\uidescriptionfunctiondiff{j_2}{2}-\uidescriptionfunctiondiff{j_1}{2}))=\\
\mu_{\opensetforvertex{p,j_1},\opensetforvertex{p,j_1-1},\chi}\ldots\mu_{\opensetforvertex{p,j_2+1},\opensetforvertex{p,j_2},\chi}
(a_1(\uidescriptionfunctiondiff{j_2}{1}-\uidescriptionfunctiondiff{j_2+1}{1})
+a_2(\uidescriptionfunctiondiff{j_2}{2}-\uidescriptionfunctiondiff{j_2+1}{2})+\\
\ldots
+a_1(\uidescriptionfunctiondiff{j_1-1}{1}-\uidescriptionfunctiondiff{j_1}{1})
+a_2(\uidescriptionfunctiondiff{j_1-1}{2}-\uidescriptionfunctiondiff{j_1}{2})).
\end{multline*}
Since $\chi\notin\normalvertexcone{\indexedvertexpt{p}{j_1-1}}{\stdpolyhedronletter_p}$ 
(and $\chi\notin\normalvertexcone{\indexedvertexpt{p}{j}}{\stdpolyhedronletter_p}$ for all $j<j_1$),
we can apply Lemmas \ref{linfuncdecomp} and \ref{linfuncchainii}. We see that
\begin{multline*}
\ord_p(\mu_{\opensetforvertex{p,j_1},\opensetforvertex{p,j_1-1},\chi}\ldots
\mu_{\opensetforvertex{p,j_2+1},\opensetforvertex{p,j_2},\chi})=\\
\chi(\indexedvertexpt{p}{j_1-1}-\indexedvertexpt{p}{j_1})+\ldots
+\chi(\indexedvertexpt{p}{j_2}-\indexedvertexpt{p}{j_2+1})\ge
\latticelength{\indexededgept{p}{j_1-1}}+\ldots+\latticelength{\indexededgept{p}{j_2}}.
\end{multline*}
And again, by the definition of $\abstractvectorspace_{1,2,p}$, 
\begin{multline*}
\ord_p(a_1(\uidescriptionfunctiondiff{j_2}{1}-\uidescriptionfunctiondiff{j_2+1}{1})
+a_2(\uidescriptionfunctiondiff{j_2}{2}-\uidescriptionfunctiondiff{j_2+1}{2})+\\
\ldots
+a_1(\uidescriptionfunctiondiff{j_1-1}{1}-\uidescriptionfunctiondiff{j_1}{1})
+a_2(\uidescriptionfunctiondiff{j_1-1}{2}-\uidescriptionfunctiondiff{j_1}{2}))\ge\\
\min(-\latticelength{\indexededgept{p}{j_2}},\ldots,-\latticelength{\indexededgept{p}{j_1-1}})=
-\max(\latticelength{\indexededgept{p}{j_2}},\ldots,\latticelength{\indexededgept{p}{j_1-1}}).
\end{multline*}
Therefore, 
$$
\ord_p(\mu_{\numberoffixedopensets,\opensetforvertex{p,j_2},\chi}f[j_2]'')\ge
\latticelength{\indexededgept{p}{j_1-1}}+\ldots+\latticelength{\indexededgept{p}{j_2}}
-\max(\latticelength{\indexededgept{p}{j_2}},\ldots,\latticelength{\indexededgept{p}{j_1-1}})\ge 0.
$$

Finally, consider the case when $j_2<j_1$ and 
$\chi\in\normalvertexcone{\indexedvertexpt{p}{j_1-1}}{\stdpolyhedronletter_p}$. 
Then $\chi\in\normalvertexcone{\indexededgept{p}{j_1-1}}{\stdpolyhedronletter_p}$, and property \ref{proplinfunczero}
in the definition of $\abstractvectorspace_{1,2,p}$ guarantees that 
$$
a_1(\uidescriptionfunctiondiff{j_1-1}{1}-\uidescriptionfunctiondiff{j_1}{1})+a_2(\uidescriptionfunctiondiff{j_1-1}{2}-\uidescriptionfunctiondiff{j_1}{2})=0.
$$
It also follows from Lemmas \ref{linfuncdecomp} and \ref{linfuncestimate} that 
$$
\ord_p(\mu_{\opensetforvertex{p,j_1},\opensetforvertex{p,j_1-1},\chi})=\chi(\indexedvertexpt{p}{j_1-1}-\indexedvertexpt{p}{j_1})=0.
$$
If $j_2=j_1-1$, then we already see that 
$$
\mu_{\opensetforvertex{p,j_1},\opensetforvertex{p,j_2},\chi}
(a_1(\uidescriptionfunctiondiff{j_2}{1}-\uidescriptionfunctiondiff{j_1}{1})
+a_2(\uidescriptionfunctiondiff{j_2}{2}-\uidescriptionfunctiondiff{j_1}{2}))=0,
$$
hence $\mu_{\numberoffixedopensets,\opensetforvertex{p,j_2},\chi}f[j_2]''=0$, in particular, this function is regular at $p$.
If $j_2<j_1-1$ we write 
\begin{multline*}
\mu_{\opensetforvertex{p,j_1},\opensetforvertex{p,j_2},\chi}
(a_1(\uidescriptionfunctiondiff{j_2}{1}-\uidescriptionfunctiondiff{j_1}{1})
+a_2(\uidescriptionfunctiondiff{j_2}{2}-\uidescriptionfunctiondiff{j_1}{2}))=\\
\mu_{\opensetforvertex{p,j_1},\opensetforvertex{p,j_1-1},\chi}\ldots
\mu_{\opensetforvertex{p,j_2+1},\opensetforvertex{p,j_2},\chi}
(a_1(\uidescriptionfunctiondiff{j_2}{1}-\uidescriptionfunctiondiff{j_2+1}{1})
+a_2(\uidescriptionfunctiondiff{j_2}{2}-\uidescriptionfunctiondiff{j_2+1}{2})+\\
\ldots
+a_1(\uidescriptionfunctiondiff{j_1-1}{1}-\uidescriptionfunctiondiff{j_1}{1})
+a_2(\uidescriptionfunctiondiff{j_1-1}{2}-\uidescriptionfunctiondiff{j_1}{2}))
\end{multline*}
as previously.
This time 
\begin{multline*}
\ord_p(\mu_{\opensetforvertex{p,j_1},\opensetforvertex{p,j_1-1},\chi}\ldots
\mu_{\opensetforvertex{p,j_2+1},\opensetforvertex{p,j_2},\chi})=\\
\ord_p(\mu_{\opensetforvertex{p,j_1-1},\opensetforvertex{p,j_1-2},\chi}\ldots\mu_{\opensetforvertex{p,j_2+1},\opensetforvertex{p,j_2},\chi})=\\
\chi(\indexedvertexpt{p}{j_1-2}-\indexedvertexpt{p}{j_1-1})+\ldots+\chi(\indexedvertexpt{p}{j_2}-\indexedvertexpt{p}{j_2+1}).
\end{multline*}
And here we can apply Lemma \ref{linfuncchainii} since 
$\chi\notin \normalvertexcone{\indexededgept{p}{j}}{\stdpolyhedronletter_p}$ for all $j<j_1-1$.
We conclude that 
$$
\ord_p(\mu_{\opensetforvertex{p,j_1},\opensetforvertex{p,j_1-1},\chi}\ldots
\mu_{\opensetforvertex{p,j_2+1},\opensetforvertex{p,j_2},\chi})\ge
\latticelength{\indexededgept{p}{j_1-2}}+\ldots+\latticelength{\indexededgept{p}{j_2}}.
$$
The order of the other multiplier can be rewritten as 
\begin{multline*}
\ord_p(a_1(\uidescriptionfunctiondiff{j_2}{1}-\uidescriptionfunctiondiff{j_2+1}{1})
+a_2(\uidescriptionfunctiondiff{j_2}{2}-\uidescriptionfunctiondiff{j_2+1}{2})+\\
\ldots
+a_1(\uidescriptionfunctiondiff{j_1-1}{1}-\uidescriptionfunctiondiff{j_1}{1})
+a_2(\uidescriptionfunctiondiff{j_1-1}{2}-\uidescriptionfunctiondiff{j_1}{2}))=\\
\ord_p(a_1(\uidescriptionfunctiondiff{j_2}{1}-\uidescriptionfunctiondiff{j_2+1}{1})
+a_2(\uidescriptionfunctiondiff{j_2}{2}-\uidescriptionfunctiondiff{j_2+1}{2})+\\
\ldots
+a_1(\uidescriptionfunctiondiff{j_1-2}{1}-\uidescriptionfunctiondiff{j_1-1}{1})
+a_2(\uidescriptionfunctiondiff{j_1-2}{2}-\uidescriptionfunctiondiff{j_1-1}{2}))\ge\\
\min(-\latticelength{\indexededgept{p}{j_2}},\ldots,-\latticelength{\indexededgept{p}{j_1-2}})=
-\max(\latticelength{\indexededgept{p}{j_2}},\ldots,\latticelength{\indexededgept{p}{j_1-2}}).
\end{multline*}
Again we see that 
\begin{multline*}
\ord_p(\mu_{\opensetforvertex{p,j_1},\opensetforvertex{p,j_2},\chi}
(a_1(\uidescriptionfunctiondiff{j_2}{1}-\uidescriptionfunctiondiff{j_1}{1})
+a_2(\uidescriptionfunctiondiff{j_2}{2}-\uidescriptionfunctiondiff{j_1}{2})))\ge\\
\latticelength{\indexededgept{p}{j_1-2}}+\ldots+\latticelength{\indexededgept{p}{j_2}}
-\max(\latticelength{\indexededgept{p}{j_2}},\ldots,\latticelength{\indexededgept{p}{j_1-2}})\ge 0,
\end{multline*}
and 
$\mu_{\numberoffixedopensets,\opensetforvertex{p,j_2},\chi}f[j_2]''$ is regular at $p$.
\end{proof}

Now it is clear that 
\begin{multline*}
\dim \rho_p^{-1}(\psi_p^{-1}(\kappa_{\OO,p}^{-1}((\bigoplus_{j=1}^{\numberofverticespt p}
\Gamma(W_p,\gvisiinv{\opensetforvertex{p,j}}))+\Gamma(W_p,\gvicircinv))))=
\dim \abstractvectorspace_{1,2,p}=\\
\latticelength{\indexededgept{p}{1}}+\ldots+\latticelength{\indexededgept{p}{\numberofverticespt p-1}},
\end{multline*}
and, since $\rho_p$ is injective,  
$$
\dim (\psi_p^{-1}(\kappa_{\OO,p}^{-1}((\bigoplus_{j=1}^{\numberofverticespt p}
\Gamma(W_p,\gvisiinv{\opensetforvertex{p,j}}))+\Gamma(W_p,\gvicircinv)))\cap \abstractvectorspace_{1,0,p})=
\latticelength{\indexededgept{p}{1}}+\ldots+\latticelength{\indexededgept{p}{\numberofverticespt p-1}}.
$$
By Lemma \ref{delta0suffices}, 
\begin{multline*}
\dim\ker\left(\left.\!\Bigg(
\bigoplus_{j=1}^{\numberofverticespt p}
\left(\Gamma(W_p,\giicircinv)/
\Gamma(W_p,\giisiinv{\opensetforvertex{p,j}})\right)\Bigg)\right/\Gamma(W_p,\giicircinv)
\longrightarrow
\right.\!\\
\left.\!
\left.\!
\Bigg(
\bigoplus_{j=1}^{\numberofverticespt p}\left(\Gamma(W_p,\gvicircinv)/
\Gamma(W_p,\gvisiinv{\opensetforvertex{p,j}})\right)\Bigg)\right/\Gamma(W_p,\gvicircinv)\right)=
\latticelength{\indexededgept{p}{1}}+\ldots+\latticelength{\indexededgept{p}{\numberofverticespt p-1}}-1.
\end{multline*}
By Lemma \ref{consideronespecialpoint}, we have the following equality:
$$
\dim(\ker(\Gamma(\PP^1,\giv)\to\Gamma(\PP^1,\gviii)))=
\sum_{p\in \PP^1\text{ essential}}\left(-1+\sum_{j=1}^{\numberofverticespt p-1}\latticelength{\indexededgept{p}{j}}\right).
$$
Finally, we get the following theorem from Theorem \ref{thmsheavesdownstairs} and Proposition \ref{honepushforwardnumber}:
\begin{theorem}\label{t1aslatticelength}
We maintain the assumptions from Section \ref{probsetup}. Then the dimension of the space of equivariant first order deformations of $X$ 
can be computed as follows.
$$
\dim T^1(X)_0=\max(0, \#(\text{essential special points})-3)+\sum_{p\in\PP^1\text{ essential}}\left(-1+\sum_{j=1}^{\numberofverticespt p-1}\latticelength{\indexededgept{p}{j}}\right),
$$
where $\latticelength{\indexededgept{p}{j}}$ is the number of integer points on the edge $\indexededgept{p}{j}$ of $\stdpolyhedronletter_p$,
including exactly one of its endpoints.
\qed
\end{theorem}

Observe that the sum $\sum_{j=1}^{\numberofverticespt p-1}\latticelength{\indexededgept{p}{j}}$ can 
also be understood as follows. The integer points on the boundary of $\stdpolyhedronletter_p$
split this boundary into segments (containing no integer points in the interior). Then 
$\sum_{j=1}^{\numberofverticespt p-1}\latticelength{\indexededgept{p}{j}}$ is the amount 
of these segments in the \textbf{finite} edges of $\stdpolyhedronletter_p$.
Later, in Chapter \ref{sectksmtoric}, we will see how to construct some actual first order deformations, 
which will span a $(\dim T^1(X)_0)$-dimensional vector space.

\chapter{Connections between the graded component of degree $0$ of $T^1(X)$ and graded components of $T^1$ of 
toric varieties}
\markboth{\chaptermarkformat Connections between $T^1(X)_0$ and graded components of $T^1$ of 
toric varieties}{}

Given an affine toric 3-dimensional variety $X$, one can restrict the space the action of the 3-dimensional torus to 
a 2-dimensional subtorus, and consider $X$ as a 3-dimensional $T$-variety with an action of a 2-dimensional torus.
Toric varieties are parametrized by pointed cones of the same dimension, and $T$-varieties are parametrized 
by polyhedral divisors as described in the Introduction. These two parametrizations are related via the following
toric downgrade procedure.

Let $X$ be an affine toric 3-dimensional variety defined by a pointed cone $\tau$ in 
$\widetilde N_{\QQ}=\widetilde N\otimes_{\ZZ}\QQ$, 
where $\widetilde N$ is a 3-dimensional lattice.
Denote the dual lattice of $\widetilde N$ by $\widetilde M$, and denote the 3-dimensional torus acting on 
$X$ by $\widetilde T$.
Then two-dimensional subtori of $\widetilde T$ are parametrized by primitive vectors $\chi\in\widetilde M$.
Fix one of them until the end of this section, denote it by $\twotorusdelimiter$. We are going to 
consider the action of $T=\ker\twotorusdelimiter$ on $X$. To describe this action by 
a polyhedral divisor, choose a line $N'\subset \widetilde N$ complementary to $N=\ker \twotorusdelimiter$.
These choices are illustrated by the following diagram:
$$
\xymatrix{
0\ar[r] & N\ar[r] & \widetilde N\ar^{\twotorusdelimiter}[r] & \ZZ\ar[r]\ar@{=}[d] & 0 \\
&&& N'\ar@{_{(}-->}[ul]
}
$$
Consider also the projection from $\widetilde N_\QQ$ to $N'_\QQ=N'\otimes_\ZZ \QQ$ along 
$N_\QQ=N\otimes_\ZZ \QQ$. It maps each face of $\tau$ surjectively onto 
a cone in $N'_\QQ$
Then the variety $Y$, where the polyhedral divisor will be constructed, is defined by the coarsest fan in $N'_\QQ$ 
containing all these cones. It can be $\PP^1$, $\CC$, or $\CC^*$, depending on whether the image of $\tau$ is the 
whole line, a half-line, or a point, respectively. We are interested in the case $Y=\PP^1$, so suppose in the 
sequel that it holds. It takes place if and only if $N_\QQ$ separates $\tau$ into two nonempty two-dimensional cones, 
or, equivalently, if $\twotorusdelimiter\notin\tau^\vee$.

To construct the polyhedral divisor itself, recall that the two half-lines of $N'_\QQ$ correspond to 
the two fixed points of a torus acting on $\PP^1$, which we can denote by 0 and $\infty$. More 
exactly, let 0 (resp. $\infty$) correspond to the half-line $\{\twotorusdelimiter>0\}$ (resp. 
$\{\twotorusdelimiter<0\}$). Then the polyhedral divisor contains nontrivial polyhedra at 0 and at $\infty$ only, 
and the polyhedron at 0 (resp. at $\infty$) is the projection of $\tau\cap [\twotorusdelimiter=1]$
(resp. of $\tau\cap [\twotorusdelimiter=-1]$) to $N_\QQ$ along $N'_\QQ$. 
As previously, denote these polyhedra by $\stdpolyhedronletter_0$ and $\stdpolyhedronletter_\infty$.
The tail cone of both of these polyhedra is 
$\6=\tau\cap N_\QQ$. We only considered the cases when it was full-dimensional, and, together with the requirement 
$Y=\PP^1$, this means that $\tau$ is full-dimensional.
An example of this situation is shown by Fig. \ref{fig3dcone}

\begin{figure}[!h]
\begin{center}
\includegraphics{t1_3fold_figures-8.mps}
\end{center}
\caption{An example of toric downgrade: the three-dimensional cone $\tau$ is shown in black,
$\6$ is green, and the polyhedra $\stdpolyhedronletter_0$ and $\stdpolyhedronletter_\infty$ are shown 
in blue and red.}\label{fig3dcone}
\end{figure}

The last requirement we had says that all vertices of $\stdpolyhedronletter_0$ and $\stdpolyhedronletter_\infty$
have to be lattice points. Since $\twotorusdelimiter$ is a primitive vector, $N'\cap[\twotorusdelimiter=1]$
and $N'\cap[\twotorusdelimiter=-1]$ are lattice points, so the projections of the planes $[\twotorusdelimiter=1]$
and $[\twotorusdelimiter=-1]$ onto $N$ along $N'$ map lattice points to lattice points. Hence, the last condition we should impose
says that if a one-dimensional face of $\tau$ intersects one of the planes
$[\twotorusdelimiter=1]$
and $[\twotorusdelimiter=-1]$,
then the intersection point is a lattice point.

Now we need some notation and terminology. Call an edge of $\tau$ \textit{positive} (resp. \textit{nonnegative}, 
\textit{negative}, \textit{nonpositive}) if $\twotorusdelimiter$
takes positive (resp. nonnegative, negative, nonpositive) values on this edge 
(except the origin). Call an edge of $\tau$ \textit{orthogonal} if $\chi$ takes only zero values in this edge.
Call a facet of $\tau$ \textit{positive} (resp. \textit{negative}) if $\chi$ takes only positive (resp.
only negative) values on the interior of this facet.
Denote the edges of $\tau$ by $\indexededge{\tau}{1}, \ldots, \indexededge{\tau}{\numberofedges{\tau}}$
and the facets of $\tau$ by $\indexedfacet{\tau}{1}, \ldots, \indexedfacet{\tau}{\numberofedges{\tau}}$.
The intersections of these edges and facets with the affine planes $\twotorusdelimiter=1$ 
and $\twotorusdelimiter=-1$ are vertices and edges of $\stdpolyhedronletter_0$ and $\stdpolyhedronletter_\infty$, 
respectively, for more details see Remark \ref{intersectionnotation}.
Sometimes we can write $\indexededge{\tau}{0}$ (resp. $\indexededge{\tau}{\numberofedges{\tau}+1}$, 
$\indexedfacet{\tau}{0}$, $\indexedfacet{\tau}{\numberofedges{\tau}+1}$) instead of 
$\indexededge{\tau}{\numberofedges{\tau}}$ (resp. $\indexededge{\tau}{1}$, $\indexedfacet{\tau}{\numberofedges{\tau}}$,
$\indexedfacet{\tau}{1}$). We enumerate edges and facets so that
$\partial \indexedfacet{\tau}{i}=\indexededge{\tau}{i}\cup\indexededge{\tau}{i+1}$. We also require that 
$\indexededge{\tau}{1}$ is a positive edge, and $\indexededge{\tau}{\numberofedges{\tau}}$ is a nonpositive edge. 
This requirement allows one to choose one of exactly two enumerations of edges and facets, we choose one of 
them arbitrarily.

It is also convenient to introduce some notation for positive and negative edges separately.
Denote the number of positive edges
by $\numberofpositiveedges{\tau}$. Denote the positive edges themselves by $\indexedpositiveedge{\tau}{1}, \ldots,
\indexedpositiveedge{\tau}{\numberofpositiveedges{\tau}}$. Here the edges are enumerated in the same order as 
when we enumerated all edges, i.~e. $\indexedpositiveedge{\tau}{i}=\indexededge{\tau}{i}$ for $1\le i\le \numberofpositiveedges{\tau}$.
Similarly, denote the number of negative edges by $\numberofnegativeedges{\tau}$, and denote the 
negative edges themselves by $\indexednegativeedge{\tau}{1}, \ldots,
\indexednegativeedge{\tau}{\numberofnegativeedges{\tau}}$. This time we \textit{\textbf{reverse}} the order that we used 
when we enumerated all edges together. In other words, if $\indexednegativeedge{\tau}{1}=\indexededge{\tau}{i-1}$
for some $i$ (which can equal $\numberofedges{\tau}$ or $\numberofedges{\tau}+1$), 
then $\indexednegativeedge{\tau}{j}=\indexededge{\tau}{i-j}$ for $1\le j\le \numberofnegativeedges{\tau}$.
The notation $\indexedpositiveedge{\tau}{i}$ may look a bit redundant, but it is convenient 
to have uniform notation for positive and negative edges.

Now let us introduce notation for positive and negative facets. Denote the facet
whose boundary is $\indexedpositiveedge{\tau}{i}\cup\indexedpositiveedge{\tau}{i+1}$
(resp. $\indexednegativeedge{\tau}{i}\cup\indexednegativeedge{\tau}{i+1}$)
by $\indexedpositivefacet{\tau}{i}$ (resp. $\indexednegativefacet{\tau}{i}$) for 
$1\le i\le \numberofpositiveedges{\tau}-1$ 
(resp $1\le i\le \numberofnegativeedges{\tau}-1$).
Again we have $\indexedpositivefacet{\tau}{i}=\indexedfacet{\tau}{i}$ for $1\le i\le \numberofpositiveedges{\tau}-1$.
Extend this notation as follows. First, set $\indexedpositivefacet{\tau}{0}=\indexedfacet{\tau}{0}$ 
and $\indexedpositivefacet{\tau}{\numberofpositiveedges{\tau}}=\indexedfacet{\tau}{\numberofpositiveedges{\tau}}$.
If $\indexednegativeedge{\tau}{1}=\indexededge{\tau}{i-1}$, denote $\indexednegativefacet{\tau}{0}=\indexedfacet{\tau}{i}$
and $\indexednegativefacet{\tau}{\numberofnegativeedges{\tau}}=\indexedfacet{\tau}{i-\numberofnegativeedges{\tau}}$.
In other words, $\indexednegativefacet{\tau}{0}$ is the facet of $\tau$ with the highest 
index such that one of the edges on its boundary is negative. The other edge on its boundary is nonnegative, 
and the negative edge on the boundary of $\indexednegativefacet{\tau}{0}$
is $\indexednegativeedge{\tau}{1}$.
And $\indexednegativefacet{\tau}{\numberofnegativeedges{\tau}}$ is the facet of $\tau$ with the lowest 
index such that one of the edges on its boundary is negative.
The other edge on its boundary is nonnegative, 
and the negative edge on the boundary of $\indexednegativefacet{\tau}{\numberofnegativeedges{\tau}}$
is $\indexednegativeedge{\tau}{\numberofnegativeedges{\tau}}$.

An example of this notation is shown by Fig.~\ref{fig3dconenotation}.

\begin{figure}[!h]
\begin{center}
\includegraphics{t1_3fold_figures-9.mps}
\end{center}
\caption{An example of notation for positive and negative edges and facets. The picture shows the 
section of $\widetilde N_{\QQ}$ with an affine hyperplane that intersects all edges of $\tau$.
The only orthogonal edge here is $\indexededge{\tau}{5}$. The facet $\indexedfacet{\tau}{9}$ is 
neither negative nor positive. Here $\numberofedges{\tau}=9$, $\numberofpositiveedges{\tau}=4$, 
and $\numberofnegativeedges{\tau}=4$.}\label{fig3dconenotation}
\end{figure}

\begin{remark}\label{intersectionnotation}
Here is how the notation introduced now is related 
with the notation for edges and vertices of $\stdpolyhedronletter_p$ we 
introduced in the beginning. Namely, within the notation that we introduced now, we have
$\numberofpositiveedges{\tau}=\numberofvertices{\stdpolyhedronletter_0}$, 
$\indexedvertex{\stdpolyhedronletter_0}{i}=\indexedpositiveedge{\tau}{i}\cap[\twotorusdelimiter=1]$ for $1\le i\le \numberofpositiveedges{\tau}$, 
$\indexededge{\stdpolyhedronletter_0}{i}=\indexedpositivefacet{\tau}{i}\cap[\twotorusdelimiter=1]$ for $0\le i\le \numberofpositiveedges{\tau}$, 
$\numberofnegativeedges{\tau}=\numberofvertices{\stdpolyhedronletter_\infty}$,
$\indexedvertex{\stdpolyhedronletter_\infty}{i}=\indexednegativeedge{\tau}{i}\cap[\twotorusdelimiter=-1]$ for $1\le i\le \numberofpositiveedges{\tau}$,
and
$\indexededge{\stdpolyhedronletter_\infty}{i}=\indexednegativefacet{\tau}{i}\cap[\twotorusdelimiter=-1]$ for $0\le i\le \numberofpositiveedges{\tau}$.
\end{remark}

The faces of the cone $\tau^\vee$ dual to $\tau$ put be set into bijection with 
the faces of $\tau$. Namely, each face $\tau'$ of $\tau$ defines a face of $\tau^\vee$
consisting of all $a\in\tau^\vee$ such that $a(\tau')=0$. We call this face of $\tau^\vee$
the \textit{normal face of $\tau'$} and denote it by $\normalvertexcone{\tau'}{\tau}$.
Clearly, the normal faces of edges are facets and vice versa.

A formula for the graded components of the first-order deformation space of a toric variety was given in 
\cite{altgorenstein}. To formulate it, we need to quote also some notation from \cite{altgorenstein}.
(We slightly change the letters we use there to avoid confusion.)
First, let 
$\widetilde{\lambda_1},\ldots,\widetilde{\lambda_{\widetilde{\numberoflatticegenerators}}}$
be the Hilbert basis of $\tau^\vee$.
If $\tau'$ is an edge of $\tau$, and $\chi\in\widetilde M$ is a degree, denote
$$
\Lambda^{\chi}_{\tau'}=\{\widetilde{\lambda_i}\mid \widetilde{\lambda_i}(\primitivelattice{\tau'})<\chi(\primitivelattice{\tau'})\}.
$$
Now, if $\tau'$ is a facet of $\tau$, we set
$$
\Lambda^{\chi}_{\tau'}=\bigcap_{\substack{\text{$\tau''$ is an edge of $\tau$}\\\tau''\subset\partial\tau'}}\Lambda^{\chi}_{\tau''},
$$
and for the origin (which is also a face of $\tau$) we set
$$
\Lambda^{\chi}_0=\bigcup_{\text{$\tau'$ is an edge of $\tau$}}\Lambda^{\chi}_{\tau'}.
$$
Finally, we set
$$
\Lambda^{\chi,i}=\bigoplus_{\substack{\text{$\tau'$ is a face of $\tau$}\\\dim\tau'=i}}\Span_{\widetilde M}(\Lambda^{\chi}_{\tau'})
$$
for $i=0,1,2$.
Here $\Span_{\widetilde M}$ denotes the sublattice of $\widetilde M$ generated by the subset of $\widetilde M$ 
under the $\Span_{\widetilde M}$ sign. In the sequel we will also use notation $\Span_\QQ$ for the $\QQ$-linear subspace 
of $\widetilde M_\QQ=\widetilde M\otimes_\ZZ\QQ$ generated by a set of elements of $\widetilde M$ or of $\widetilde M_\QQ$.
generated 
Consider the complex
$$
(\Lambda^{\chi,0}\otimes_\ZZ\CC)^*\to (\Lambda^{\chi,1}\otimes_\ZZ\CC)^*\to (\Lambda^{\chi,2}\otimes_\ZZ\CC)^*,
$$
where the maps are standard Cech differentials.
Denote the graded component of $T^1(X)$ of degree $\chi$ by $T^1_\chi(X)$.
\begin{theorem}\label{t1toric}\cite[Theorem 2.1]{altgorenstein}
$$
T^1_{-\chi}(X)\cong H^1\Big((\Lambda^{\chi,\bullet}\otimes_\ZZ\CC)^*\Big).
$$
\end{theorem}

Our goal for this section is to deduce Theorem \ref{t1aslatticelength}
in the case of toric $X$ from Theorem \ref{t1toric}. 
It is known 
that 
the 0th graded component of $X$ considered as a $T$-variety
is isomorphic to
$$
\bigoplus_{a\in\ZZ}
T^1_{a\twotorusdelimiter}(X),
$$
where the degrees are understood with respect to the action of the three-dimensional torus. So, 
in the sequel we will study the spaces $T^1_\chi(X)$, where $\chi$ is a multiple of 
$\twotorusdelimiter$. 

\begin{lemma}\label{triviallambda1zero}
Let $\chi$ be a multiple of $\twotorusdelimiter$ and $\tau'$ be an edge of $\tau$.
Then $\Lambda^{\chi}_{\tau'}=\varnothing$ if one of the following conditions holds:
\begin{enumerate}
\item\label{zero3degree} $\chi=0$.
\item\label{orthogonal3degree} $\tau'$ is an orthogonal edge.
\item\label{negpos3degree} $\chi=a\twotorusdelimiter$, where $a>0$, and $\tau'$ is a negative edge.
\item\label{posneg3degree} $\chi=a\twotorusdelimiter$, where $a<0$, and $\tau'$ is a positive edge.
\end{enumerate}
\end{lemma}

\begin{proof}
Choose a Hilbert basis element $\widetilde{\lambda_i}$, where $1\le i\le \widetilde{\numberoflatticegenerators}$.
Since $\widetilde{\lambda_i}\in\tau^\vee$, we have 
$\widetilde{\lambda_i}(\primitivelattice{\tau'})\ge 0$.
On the other hand, $\chi(\primitivelattice{\tau'})=0$ if case \ref{zero3degree} or \ref{orthogonal3degree} from 
the above classification holds. If case \ref{negpos3degree} or \ref{posneg3degree} takes place, then
$\chi(\primitivelattice{\tau'})<0$. Hence, $\widetilde{\lambda_i}(\primitivelattice{\tau'})\ge \chi(\primitivelattice{\tau'})$, 
and $\widetilde{\lambda_i}\notin \Lambda^{\chi}_{\tau'}$.
\end{proof}

\begin{corollary}\label{chieq0trivial}
If $\chi=0$, then $\Lambda^{\chi,1}=0$ and $T^1_0(X)=0$.\qed
\end{corollary}

\begin{lemma}\label{delimitervalueone}
If $\tau'$ is a positive (resp. negative) edge of $\tau$, then $\twotorusdelimiter(\primitivelattice{\tau'})$
equals 1 (resp. $-1$).
\end{lemma}

\begin{proof}
If $\tau'$ is a positive edge, denote $a=\tau'\cap[\twotorusdelimiter=1]$.
If $\tau'$ is a negative edge, denote $a=\tau'\cap[\twotorusdelimiter=-1]$.
Recall that one of the requirements we have imposed on $\tau$ says that the planes $\twotorusdelimiter=1$ and
$\twotorusdelimiter=-1$ intersect edges of $\tau$ at lattice points (otherwise the polyhedral divisor we obtain from 
$\tau$ does not consist of lattice polyhedra), so $a$ is a lattice point, and 
hence $a$ is a multiple of $\primitivelattice{\tau'}$. On the other hand, if 
$a\ne \primitivelattice{\tau'}$, then $\twotorusdelimiter(\primitivelattice{\tau'})$ 
cannot be an integer. So, $a=\primitivelattice{\tau'}$, 
and $\twotorusdelimiter(\primitivelattice{\tau'})=1$ (resp. $\twotorusdelimiter(\primitivelattice{\tau'})=-1$)
if $\tau'$ is a positive (resp. negative) edge.
\end{proof}

\begin{lemma}
If $\tau'$ is a positive (resp. negative) edge of $\tau$, and $\chi=\twotorusdelimiter$ (resp. $\chi=-\twotorusdelimiter$), 
then $\Span_{\QQ}(\Lambda^\chi_{\tau'})=\Span_{\QQ}(\normalvertexcone{\tau'}{\tau})$
and $\dim\Span_{\QQ}(\Lambda^\chi_{\tau'})=2$.
\end{lemma}

\begin{proof}
Without loss of generality, suppose that $\tau'$ is a positive edge and $\chi=\twotorusdelimiter$ (the other case 
can be considered completely analogously). Then by Lemma \ref{delimitervalueone}, 
$\chi(\primitivelattice{\tau'})=1$. So, if $\widetilde{\lambda_i}\notin \normalvertexcone{\tau'}{\tau}$, 
then $\widetilde{\lambda_i}(\primitivelattice{\tau'})>0$, so $\widetilde{\lambda_i}(\primitivelattice{\tau'})\ge 1$ (this is an 
integer number), and $\widetilde{\lambda_i}(\primitivelattice{\tau'})\ge \chi(\primitivelattice{\tau'})$.
Hence, $\widetilde{\lambda_i}\notin\Lambda^\chi_{\tau'}$.
On the other hand, if $\widetilde{\lambda_i}\in \normalvertexcone{\tau'}{\tau}$, then $\widetilde{\lambda_i}(\primitivelattice{\tau'})=0$, 
and $\widetilde{\lambda_i}(\primitivelattice{\tau'})<\chi(\primitivelattice{\tau'})$. Hence, 
$\widetilde{\lambda_i}\in\Lambda^\chi_{\tau'}$. 

Therefore, $\Lambda^\chi_{\tau'}$ is the intersection of the Hilbert basis of $\tau^\vee\cap \widetilde M$ and the normal 
facet of $\tau'$, which is the Hilbert basis of $\normalvertexcone{\tau'}{\tau}\cap \widetilde M$. In 
particular, $\Lambda^\chi_{\tau'}$ generates $\Span_{\QQ}(\normalvertexcone{\tau'}{\tau})$
as a $\QQ$-vector space.
\end{proof}

\begin{lemma}
If $\tau'$ is a positive (resp. negative) edge of $\tau$, and $\chi=a\twotorusdelimiter$, where $a\ge 2$ (resp. $a\le -2$), 
then $\Span_{\QQ}(\Lambda^\chi_{\tau'})=\widetilde M_\QQ$.
\end{lemma}

\begin{proof}
Again, without loss of generality we may suppose that $\tau'$ is a positive edge and $a\ge 2$, the other 
case is completely similar.

First, let us prove that there exists a degree $\chi'\in\tau^\vee\cap \widetilde M$
such that $\chi'(\primitivelattice{\tau'})=1$. This is done by a standard continuity argument.
Namely, consider a lattice point $\chi''$ in the relative interior of $\normalvertexcone{\tau'}{\tau}$.
Consider also a line $\chi''+\QQ\twotorusdelimiter$.
This line 
cannot be contained in the plane containing $\normalvertexcone{\tau'}{\tau}$ 
since $\twotorusdelimiter(\primitivelattice{\tau'})\ne 0$. So, the intersection 
of this line and this plane is exactly $\chi''$, and $\normalvertexcone{\tau'}{\tau}$ 
splits the line 
$\chi''+\QQ\twotorusdelimiter$
into two rays, and one of these rays passes through the interior of $\tau^\vee$.
Since $\chi''(\primitivelattice{\tau'})=0$ and $\twotorusdelimiter(\primitivelattice{\tau'})>0$,
the ray passing through the interior of $\tau^\vee$ cannot be 
$\chi''+\QQ_{\le 0}\twotorusdelimiter$, and it must be 
$\chi''+\QQ_{\ge 0}\twotorusdelimiter$. Hence, if $b\in\NN$ is large enough, 
$\chi''+(1/b)\twotorusdelimiter\in\tau^\vee$. Then $b\chi''+\twotorusdelimiter\in\tau^\vee$, 
but $b\chi''+\twotorusdelimiter$ is a lattice point, and 
$(b\chi''+\twotorusdelimiter)(\primitivelattice{\tau'})=1$, so we can take $\chi'=b\chi''+\twotorusdelimiter$.

Since all $\widetilde{\lambda_i}$ form the Hilbert basis of $\tau^\vee\cap\widetilde M$,
$\chi'$ can be written as a positive integer linear combination of $\widetilde{\lambda_i}$.
Since $\widetilde{\lambda_i}(\primitivelattice{\tau'})\ge 0$, there exists $\widetilde{\lambda_i}$
such that $\widetilde{\lambda_i}(\primitivelattice{\tau'})=1$.

As we have already noted previously, the set of all $\widetilde{\lambda_i}$ such that 
$\widetilde{\lambda_i}(\primitivelattice{\tau'})=0$ form the Hilbert basis of 
$\normalvertexcone{\tau'}{\tau}\cap \widetilde M$, therefore 
they generate 
$\Span_{\QQ}(\normalvertexcone{\tau'}{\tau})$ as a $\QQ$-vector space.
Clearly, all these $\widetilde{\lambda_i}$ are in $\Lambda^\chi_{\tau'}$.
Together they generate a 2-dimensional vector space, so if we add one more vector, 
which is outside $\Span_{\QQ}(\normalvertexcone{\tau'}{\tau})$, 
all vectors together will generate a bigger vector space, but then this 
space must be $\widetilde M_\QQ$ since $\dim \widetilde M_\QQ=3$. But 
we already know that there exists a $\widetilde{\lambda_i}\in\Lambda^\chi_{\tau'}$ such that 
$\widetilde{\lambda_i}(\primitivelattice{\tau'})=1$. By the definition 
of $\normalvertexcone{\tau'}{\tau}$, all vectors from 
$\Span_{\QQ}(\normalvertexcone{\tau'}{\tau})$ vanish on $\primitivelattice{\tau'}$, 
so this $\widetilde{\lambda_i}$ cannot be in $\Span_{\QQ}(\normalvertexcone{\tau'}{\tau})$.
Therefore, $\Span_{\QQ}(\Lambda^\chi_{\tau'})=\widetilde M_\QQ$.
\end{proof}

\begin{corollary}\label{lambda1classification}
If $\chi=a\twotorusdelimiter$, $a\in\ZZ$, $a\ne 0$, then $\Lambda^{\chi,1}\otimes_\ZZ\CC$ can be written as follows:
\begin{enumerate}
\item If $a=1$, then
$$
\Lambda^{\chi,1}\otimes_\ZZ\CC=\bigoplus_{i=1}^{\numberofpositiveedges{\tau}}\Span_\QQ(\normalvertexcone{\indexedpositiveedge{\tau}{i}}{\tau})\otimes_\QQ\CC.
$$
\item If $a\ge 2$, then
$$
\Lambda^{\chi,1}\otimes_\ZZ\CC=\bigoplus_{i=1}^{\numberofpositiveedges{\tau}}\widetilde M_\QQ\otimes_\QQ\CC.
$$
\item If $a=-1$, then
$$
\Lambda^{\chi,1}\otimes_\ZZ\CC=\bigoplus_{i=1}^{\numberofnegativeedges{\tau}}\Span_\QQ(\normalvertexcone{\indexednegativeedge{\tau}{i}}{\tau})\otimes_\QQ\CC.
$$
\item If $a\le -2$, then
$$
\Lambda^{\chi,1}\otimes_\ZZ\CC=\bigoplus_{i=1}^{\numberofnegativeedges{\tau}}\widetilde M_\QQ\otimes_\QQ\CC.
$$
\end{enumerate}\qed
\end{corollary}

These lemmas also enable us to describe $\Lambda^{\chi,0}$ explicitly:

\begin{corollary}\label{lambda0classification}
If $\chi=a\twotorusdelimiter$, $a\in\ZZ$, $a\ne 0$, then $\Lambda^{\chi,0}\otimes_\ZZ\CC$ can be written as follows:
\begin{enumerate}
\item If $a=1$, then
$$
\Lambda^{\chi,0}\otimes_\ZZ\CC=\Span_\QQ\left(\bigcup_{i=1}^{\numberofpositiveedges{\tau}}\normalvertexcone{\indexedpositiveedge{\tau}{i}}{\tau}\right)\otimes_\QQ\CC.
$$
\item If $a\ge 2$, then
$$
\Lambda^{\chi,0}\otimes_\ZZ\CC=\widetilde M_\QQ\otimes_\QQ\CC.
$$
\item If $a=-1$, then
$$
\Lambda^{\chi,0}\otimes_\ZZ\CC=\Span_\QQ\left(\bigcup_{i=1}^{\numberofnegativeedges{\tau}}\normalvertexcone{\indexednegativeedge{\tau}{i}}{\tau}\right)\otimes_\QQ\CC.
$$
\item If $a\le -2$, then
$$
\Lambda^{\chi,0}\otimes_\ZZ\CC=\widetilde M_\QQ\otimes_\QQ\CC.
$$
\end{enumerate}\qed
\end{corollary}

Now we have to find $\ker ((\Lambda^{\chi,1}\otimes_\ZZ\CC)^*\to(\Lambda^{\chi,2}\otimes_\ZZ\CC)^*))$, 
where $\chi$ is a multiple of $\twotorusdelimiter$. To compute this kernel, we need some information about 
$\Lambda^{\chi,2}$.
First, let us make the following observation. An element of $(\Lambda^{\chi,2}\otimes_\ZZ\CC)^*$
can be written as a sequence $(a_1,\ldots, a_{\numberofedges{\tau}})$, where 
$a_i\in(\Span_{\widetilde M}(\Lambda^\chi_{\indexedfacet{\tau}{i}})\otimes_\ZZ\CC)^*$.
In particular, the image of an element of $(\Lambda^{\chi,1}\otimes_\ZZ\CC)^*$
can be written in this form. Consider an entry $a_i$ such that $\partial\indexedfacet{\tau}{i}$
consists only of edges such that $\Span_{\QQ}(\Lambda^\chi_{\tau'})=0$.
Observe that in this case $a_i=0$ 
since in this case $a_i$ is the difference of two elements of two vector spaces, and each of this vector spaces has dimension 0.
So, it is sufficient to 
consider only the facets whose boundary contains at least one edge $\tau'$ such that $\Span_{\QQ}(\Lambda^\chi_{\tau'})\ne 0$.
Using Corollary \ref{lambda1classification}, we can say that if $\chi=a\twotorusdelimiter$, where $a>0$ (resp. $a<0$), 
then it is sufficient to consider only the facets of $\tau$ whose boundary contains at least one positive (resp. 
negative) edge. These are exactly the facets we have denoted by $\indexedpositivefacet{\tau}{0}, \ldots, 
\indexedpositivefacet{\tau}{\numberofpositiveedges{\tau}}$ 
(resp. by $\indexednegativefacet{\tau}{0}, \ldots, \indexednegativefacet{\tau}{\numberofnegativeedges{\tau}}$).

\begin{lemma}\label{lambda2noboundaryconditionpos}
If $\chi=a\twotorusdelimiter$, where $a>0$, then $\Span_\QQ(\Lambda^\chi_{\indexedpositivefacet{\tau}{i}})=0$
for $i=0$ and $i=\numberofpositiveedges{\tau}$.
\end{lemma}

\begin{proof}
Let us consider the case $i=0$, the other case is completely similar. By the definition of 
$\indexedpositivefacet{\tau}{0}$, its boundary consists of $\indexedpositiveedge{\tau}{1}$, which
is a positive edge, and another edge $\tau'$, which is nonpositive. Hence, 
by Lemma \ref{triviallambda1zero}, $\Lambda^\chi_{\tau'}=\varnothing$, so 
$\Lambda^\chi_{\indexedpositivefacet{\tau}{0}}=\Lambda^\chi_{\tau'}\cap\Lambda^\chi_{\indexedpositiveedge{\tau}{1}}=\varnothing$
as well, and $\Span_\QQ(\Lambda^\chi_{\indexedpositivefacet{\tau}{0}})=0$.
\end{proof}

\begin{lemma}\label{lambda2noboundaryconditionneg}
If $\chi=a\twotorusdelimiter$, where $a<0$, then $\Span_\QQ(\Lambda^\chi_{\indexednegativefacet{\tau}{i}})=0$
for $i=0$ and $i=\numberofnegativeedges{\tau}$.
\end{lemma}

\begin{proof}
The proof here is again completely similar to the proof of the previous lemma, but this time we present it to ease reading.
Let us consider the case $i=\numberofnegativeedges{\tau}$, the other case is completely similar. By the definition of 
$\indexednegativefacet{\tau}{\numberofnegativeedges{\tau}}$, its boundary consists of 
$\indexednegativeedge{\tau}{\numberofnegativeedges{\tau}}$, which
is a negative edge, and another edge $\tau'$, which is nonnegative. Hence, 
by Lemma \ref{triviallambda1zero}, $\Lambda^\chi_{\tau'}=\varnothing$, so 
$\Lambda^\chi_{\indexednegativefacet{\tau}{\numberofnegativeedges{\tau}}}=
\Lambda^\chi_{\tau'}\cap\Lambda^\chi_{\indexednegativeedge{\tau}{\numberofnegativeedges{\tau}}}=\varnothing$
as well, and $\Span_\QQ(\Lambda^\chi_{\indexednegativefacet{\tau}{\numberofnegativeedges{\tau}}})=0$.
\end{proof}

To understand the behavior of $\Lambda^\chi_{\indexedpositivefacet{\tau}{i}}$,
where $1\le i\le \numberofpositiveedges{\tau}-1$,
(resp. of $\Lambda^\chi_{\indexednegativefacet{\tau}{i}}$, where $1\le i\le \numberofnegativeedges{\tau}-1$)
for degrees $\chi=a\twotorusdelimiter$ with $a>0$ (resp. $a<0$), we start with the following lemma.

\begin{lemma}\label{smalllaticepoints}
Let $\overline N$ be a two-dimensional lattice, and let $\overline M$ be its dual lattice. 
Let $a_1, a_2\in \overline N$ and $\chi\in \overline M$ be such that 
$\chi(a_1)=\chi(a_2)=1$ and $a_1\ne a_2$. Then $a_1$ and $a_2$ generate $\overline N\otimes_\ZZ\QQ$
as a $\QQ$-vector space.

Denote the primitive lattice point on the ray 
$\{\chi'\in \overline M:\chi'(a_1)>0, \chi'(a_2)=0\}$ by $\chi_1$. 
Similarly, denote by $\chi_2$ the primitive lattice point on the ray 
$\{\chi'\in \overline M:\chi'(a_1)=0, \chi'(a_2)>0\}$

Then $\chi_1(a_1)=\chi_2(a_2)=\latticelength{a_1-a_2}$.
The sets 
$$
\overline\Lambda_{\chi,a_1,a_2, b}\{\chi'\in\overline M:\chi'(a_1)\ge 0, \chi'(a_2)\ge 0, \chi'(a_1)<b, \chi'(a_2)<b\}
$$
for $b\in\NN$ behave as follows:
\begin{enumerate}
\item If $0<b\le \latticelength{a_1-a_2}$, then 
$\overline\Lambda_{\chi,a_1,a_2, b}$
is the set of all $\chi'$ of the form 
$\chi'=b'\chi$, $0\le b'< b$.
\item If $b>\latticelength{a_1-a_2}$, 
then $\overline\Lambda_{\chi,a_1,a_2, b}$
contains $\chi_1$ and $\chi_2$.
\end{enumerate}
\end{lemma}

\begin{proof}
Consider the $\QQ$-linear span of $a_1$ and $a_2$ in $\overline N\otimes_\ZZ\QQ$.
Since $\chi(a_1)\ne 0$ and $\chi(a_2)\ne 0$, this linear span can be one-dimensional only if 
$a_1$ is a $\QQ$-multiple of $a_2$. But in this case, since $\chi(a_1)=\chi(a_2)\ne 0$, 
$a_1$ and $a_2$ must coincide, and this is a contradiction.

Denote $k=\latticelength{a_1-a_2}$ and denote $a'=(1-1/k)a_1+(1/k)a_2$.
Then $a'\in \overline N$, and 
$a'-a_1$ is a 
primitive lattice vector.
Hence, there exists a function $\chi''\in\overline M$ such that $\chi''(a'-a_1)=1$.
Since $\chi(a_1)=\chi(a_2)=1$, we also have $\chi(a')=1$.
Consider the following functions $\chi'''_i$ ($i=1,2$): $\chi'''_i=\chi''-\chi''(a_i)\chi$.
We have $\chi'''_i(a_i)=\chi''(a_i)-\chi''(a_i)\chi(a_i)=0$, so $\chi'''_1$ is a multiple of 
$\chi_2$ and $\chi'''_2$ is a multiple of $\chi_1$, since $\chi_1$ and $\chi_2$ are primitive
vectors on the corresponding rays.

We also have 
$\chi'''_1(a_2)=\chi'''_1(a_1)+\chi'''_1(a_2-a_1)=k\chi'''_1(a'-a_1)=k(\chi''(a'-a_1)-\chi''(a_1)\chi(a'-a_1))=
k(1-\chi''(a_1)(1-1))=k$ 
and 
$\chi'''_2(a_1)=\chi'''_2(a_2)-\chi'''_2(a_2-a_1)=-k\chi'''_2(a'-a_1)=-k(\chi''(a'-a'_1)-\chi''(a_2)\chi(a'-a_1))=
-k(1-\chi''(a_1)(1-1))=-k$. On the other hand, $\chi_2(a_2)=\chi_2(a_1)+\chi_2(a_2-a_1)=k\chi_2(a'-a_1)$
and $\chi_1(a_1)=\chi_1(a_2)-\chi_1(a_2-a_1)=-k\chi_1(a'-a_1)$. Hence, $\chi_1(a_1)$ is a multiple of $k=\chi'''_2(a_1)$
and $\chi_2(a_2)$ is a multiple of $k=-\chi'''_1(a_2)$. Recall that 
$\chi'''_1$ is a multiple of 
$\chi_2$ and $\chi'''_2$ is a multiple of $\chi_1$. Summarizing, we conclude that $\chi_1=\pm \chi'''_2$ and 
$\chi_2=\pm \chi'''_1$. But then $\chi_1(a_1)=\pm \chi'''_2(a_1)=\pm k$ and $\chi_2(a_2)=\pm \chi'''_1(a_2)=\pm k$.
Since $\chi_1(a_1)>0$ and $\chi_2(a_2)>0$ by the definitions of $\chi_1$ and $\chi_2$, we have $\chi_1(a_1)=\chi_2(a_2)=k$.

Now fix some $b\in \NN$ and consider the set 
$$
\overline\Lambda_{\chi,a_1,a_2, b}=\{\chi'\in\overline M:\chi'(a_1)\ge 0, \chi'(a_2)\ge 0, \chi'(a_1)<b, \chi'(a_2)<b\}.
$$
If $b>\latticelength{a_1-a_2}$, then it is already clear that 
$\overline\Lambda_{\chi,a_1,a_2, b}$
contains $\chi_1$ and 
$\chi_2$ since $\chi_1(a_1)=\latticelength{a_1-a_2}$, $\chi_1(a_2)=0$, 
$\chi_1(a_2)=0$, and $\chi_2(a_2)=\latticelength{a_1-a_2}$. So suppose that $b\le \latticelength{a_1-a_2}$.
In this case it is also clear that $b'\chi\in\overline\Lambda_{\chi,a_1,a_2, b}$ for $0\le b'<b$ since 
$\chi(a_1)=\chi(a_2)=1$.

Suppose that $\chi'\in\overline\Lambda_{\chi,a_1,a_2, b}$. Without loss of generality, $\chi'(a_1)\ge \chi'(a_2)$.
Consider $\chi''=\chi'-\chi'(a_2)\chi$. We have $\chi''(a_1)=\chi'(a_1)-\chi'(a_2)\chi(a_1)=\chi'(a_1)-\chi'(a_2)\ge 0$
and $\chi''(a_2)=\chi'(a_2)-\chi'(a_2)\chi(a_2)=0$. So, $\chi''$ is a lattice point on the (closed) ray 
$\{\chi'''\in \overline M:\chi'''(a_1)\ge0, \chi'''(a_2)=0\}$. But we already know that 
the primitive lattice vector on this ray is $\chi_1$, so $\chi''$ is a (possibly zero) integer 
multiple of $\chi_1$. If $\chi'(a_1)>\chi'(a_2)$, then $\chi''\ne 0$, and we have a contradiction 
with $\chi_1(a_1)=\latticelength{a_1-a_2}$ since 
$\chi'(a_1)<b\le\latticelength{a_1-a_2}$, $\chi'(a_2)\ge 0$, and $\chi''(a_1)=\chi'(a_1)-\chi'(a_2)$.
If $\chi'(a_1)=\chi'(a_2)$, then we see that $\chi'$ and $\chi'(a_1)\chi$ take the same values on 
$a_1$ and $a_2$. Since $a_1$ and $a_2$ $\QQ$-generate $\overline N\otimes_\ZZ\QQ$, 
we can conclude that $\chi'=\chi'(a_1)\chi$ as desired.
\end{proof}

\begin{lemma}\label{shiftalongedge}
Let $\indexedfacet{\tau}{i}$ be facet of $\tau$, and let $\indexededge{\tau}{j_1}$ be an edge of $\tau$ 
on the boundary of $\indexedfacet{\tau}{i}$. Let $\indexededge{\tau}{j_2}$ be the other edge on the 
boundary of $\indexedfacet{\tau}{i}$. Suppose that we have a degree 
$\chi\in\Span_\QQ(\normalvertexcone{\indexededge{\tau}{j_1}}{\tau})\cap \widetilde M$
such that $\chi(\primitivelattice{\indexededge{\tau}{j_2}})>0$. 

Then there exists $a\in \NN$ 
such that 
$\chi+a\primitivelattice{\normalvertexcone{\indexedfacet{\tau}{i}}{\tau}}\in\normalvertexcone{\indexededge{\tau}{j_1}}{\tau}$.
\end{lemma}

\begin{proof}
Let $\indexedfacet{\tau}{k}$ be the facet of $\tau$ such that 
$\partial \normalvertexcone{\indexededge{\tau}{j_1}}{\tau}=
\normalvertexcone{\indexedfacet{\tau}{i}}{\tau}\cup\normalvertexcone{\indexedfacet{\tau}{k}}{\tau}$.
In other words, $\indexedfacet{\tau}{i}$ and $\indexedfacet{\tau}{k}$ are the two facets whose boundary 
contains $\indexededge{\tau}{j_1}$. Then $\normalvertexcone{\indexededge{\tau}{j_1}}{\tau}$ is determined 
inside $\Span_\QQ(\normalvertexcone{\indexededge{\tau}{j_1}}{\tau})$ by two inequalities corresponding 
to $\normalvertexcone{\indexedfacet{\tau}{i}}{\tau}$ and $\normalvertexcone{\indexedfacet{\tau}{k}}{\tau}$.
For an inequality corresponding to $\normalvertexcone{\indexedfacet{\tau}{i}}{\tau}$, 
we can take the restriction to $\Span_\QQ(\normalvertexcone{\indexededge{\tau}{j_1}}{\tau})$ of the inequality 
in the definition of $\tau^\vee$ corresponding to the other facet of $\tau'$ whose boundary contains 
$\normalvertexcone{\indexedfacet{\tau}{i}}{\tau}$. This other facet is $\normalvertexcone{\indexededge{\tau}{j_2}}$, 
and the corresponding inequality says that if $\chi'\in\normalvertexcone{\indexededge{\tau}{j_1}}{\tau}$, 
then $\chi'$ takes nonnegative values on $\indexededge{\tau}{j_2}$, in other words, 
$\chi'(\primitivelattice{\indexededge{\tau}{j_2}})\ge 0$.

Similarly, for an inequality corresponding to $\normalvertexcone{\indexedfacet{\tau}{k}}{\tau}$, 
we can take the restriction to $\Span_\QQ(\normalvertexcone{\indexededge{\tau}{j_1}}{\tau})$ of the inequality 
corresponding to the facet of $\tau^\vee$ different from $\normalvertexcone{\indexededge{\tau}{j_1}}{\tau}$
and whose boundary contains $\normalvertexcone{\indexedfacet{\tau}{k}}{\tau}$. This facet is the normal facet of 
the edge on the boundary of $\indexedfacet{\tau}{k}$ different from $\indexededge{\tau}{j_1}$. Denote it 
by $\indexededge{\tau}{j_3}$ so that $\partial \indexedfacet{\tau}{k}=\indexededge{\tau}{j_1}\cup \indexededge{\tau}{j_3}$.
Then the inequality corresponding to $\normalvertexcone{\indexededge{\tau}{j_3}}{\tau}$ in the 
definition of $\tau^\vee$ says that if $\chi'\in\tau^\vee$, then $\chi'$ takes nonnegative values on 
$\indexededge{\tau}{j_3}$, in other words, $\chi'(\primitivelattice{\indexededge{\tau}{j_3}})\ge 0$.
Therefore, $\normalvertexcone{\indexededge{\tau}{j_1}}{\tau}$ is determined 
inside $\Span_\QQ(\normalvertexcone{\indexededge{\tau}{j_1}}{\tau})$ by 
the restrictions to $\Span_\QQ(\normalvertexcone{\indexededge{\tau}{j_1}}{\tau})$
of the inequalities $\chi'(\primitivelattice{\indexededge{\tau}{j_2}})\ge 0$
and $\chi'(\primitivelattice{\indexededge{\tau}{j_3}})\ge 0$ for $\chi'\in\widetilde {M_\QQ}$.

Therefore, if $\chi(\primitivelattice{\indexededge{\tau}{j_3}})\ge 0$, then we can take $a=0$. 
Suppose that $\chi(\primitivelattice{\indexededge{\tau}{j_3}})<0$.

We chose $\indexedfacet{\tau}{k}$ so that 
$$\normalvertexcone{\indexedfacet{\tau}{i}}{\tau}\ne\normalvertexcone{\indexedfacet{\tau}{k}}{\tau},
$$
and we also know that 
$$
\normalvertexcone{\indexededge{\tau}{j_3}}{\tau}\cap\normalvertexcone{\indexededge{\tau}{j_1}}{\tau}=
\normalvertexcone{\indexedfacet{\tau}{k}}{\tau},
$$
so 
$$\primitivelattice{\normalvertexcone{\indexedfacet{\tau}{i}}{\tau}}\notin \normalvertexcone{\indexededge{\tau}{j_3}}{\tau}.
$$
Hence, 
$$
\primitivelattice{\normalvertexcone{\indexedfacet{\tau}{i}}{\tau}}(\primitivelattice{\indexededge{\tau}{j_3}})>0.
$$
Then there exists $a\in\NN$ such that 
$$
a\primitivelattice{\normalvertexcone{\indexedfacet{\tau}{i}}{\tau}}(\primitivelattice{\indexededge{\tau}{j_3}})>
-\chi(\primitivelattice{\indexededge{\tau}{j_3}}).
$$
In other words, 
$$
a\primitivelattice{\normalvertexcone{\indexedfacet{\tau}{i}}{\tau}}(\primitivelattice{\indexededge{\tau}{j_3}})
+\chi(\primitivelattice{\indexededge{\tau}{j_3}})>0.
$$
We have 
$$
(a\primitivelattice{\normalvertexcone{\indexedfacet{\tau}{i}}{\tau}}
+\chi)(\primitivelattice{\indexededge{\tau}{j_3}})>0.
$$
We also have 
$$
(\primitivelattice{\normalvertexcone{\indexedfacet{\tau}{i}}{\tau}})(\primitivelattice{\indexededge{\tau}{j_2}})=0
$$
since $\primitivelattice{\indexededge{\tau}{j_2}}\in\indexedfacet{\tau}{i}$. Hence, 
$$
(a\primitivelattice{\normalvertexcone{\indexedfacet{\tau}{i}}{\tau}}
+\chi)(\primitivelattice{\indexededge{\tau}{j_2}})=\chi(\primitivelattice{\indexededge{\tau}{j_2}})> 0
$$
by assumption, and
$$
a\primitivelattice{\normalvertexcone{\indexedfacet{\tau}{i}}{\tau}}
+\chi\in\normalvertexcone{\indexededge{\tau}{j_1}}{\tau}\cap\widetilde M.
$$
\end{proof}

\begin{lemma}\label{edgepluschiintau}
Let $\indexedpositivefacet{\tau}{i}$ (resp. $\indexednegativefacet{\tau}{i}$), where $1\le i\le \numberofpositiveedges{\tau}-1$
(resp. $1\le i\le \numberofnegativeedges{\tau}-1$), be a facet of $\tau$. Then 
$$
\primitivelattice{\normalvertexcone{\indexedpositivefacet{\tau}{i}}{\tau}}+\twotorusdelimiter\in\tau^\vee
\text{ (resp. }\primitivelattice{\normalvertexcone{\indexednegativefacet{\tau}{i}}{\tau}}-\twotorusdelimiter\in\tau^\vee\text{)}.
$$
\end{lemma}

\begin{proof}
Since 
$$\primitivelattice{\normalvertexcone{\indexedpositivefacet{\tau}{i}}{\tau}}\in\tau^\vee
\text{ (resp. }\primitivelattice{\normalvertexcone{\indexednegativefacet{\tau}{i}}{\tau}}\in\tau^\vee\text{)},
$$
it takes nonnegative values on the edges of $\tau$.
Since
$$
\partial \indexedpositivefacet{\tau}{i}=\indexedpositiveedge{\tau}{i}\cup \indexedpositiveedge{\tau}{i+1}
\text{ (resp. }\partial \indexednegativefacet{\tau}{i}=\indexednegativeedge{\tau}{i}\cup \indexednegativeedge{\tau}{i+1}\text{)},
$$
the only two edges of $\tau$ where 
$$
\primitivelattice{\normalvertexcone{\indexedpositivefacet{\tau}{i}}{\tau}}
\text{ (resp. }\primitivelattice{\normalvertexcone{\indexednegativefacet{\tau}{i}}{\tau}}\text{)}
$$
vanishes are
$\indexedpositiveedge{\tau}{i}$ and $\indexedpositiveedge{\tau}{i+1}$
(resp. $\indexednegativeedge{\tau}{i}$ and $\indexednegativeedge{\tau}{i+1}$).
But both of these edges are positive (resp. negative), so if $\indexededge{\tau}{j}$ is 
one of these two edges, then 
$$
\twotorusdelimiter(\primitivelattice{\indexededge{\tau}{j}})=1
\text{ (resp. }\twotorusdelimiter(\primitivelattice{\indexededge{\tau}{j}})=-1\text{)}.
$$
Hence, 
$$
(\primitivelattice{\normalvertexcone{\indexedpositivefacet{\tau}{i}}{\tau}}
+\twotorusdelimiter)(\primitivelattice{\indexededge{\tau}{j}})=1
$$
$$\text{(resp. }(\primitivelattice{\normalvertexcone{\indexednegativefacet{\tau}{i}}{\tau}}
-\twotorusdelimiter)(\primitivelattice{\indexededge{\tau}{j}})=1,\text{ observe the $-$ sign 
in front of $\twotorusdelimiter$)}
$$
for $\indexededge{\tau}{j}=\indexedpositiveedge{\tau}{i}$ or 
$\indexededge{\tau}{j}=\indexedpositiveedge{\tau}{i+1}$
(resp. $\indexededge{\tau}{j}=\indexednegativeedge{\tau}{i}$ or 
$\indexededge{\tau}{j}=\indexednegativeedge{\tau}{i+1}$).

Now suppose that $\indexededge{\tau}{j}$ is another edge, i.~e. 
$$
\indexededge{\tau}{j}\notin \partial \indexedpositivefacet{\tau}{i}
\text{ (resp. }\indexededge{\tau}{j}\notin \partial \indexednegativefacet{\tau}{i}\text{)}.
$$
Then 
$$
\primitivelattice{\normalvertexcone{\indexedpositivefacet{\tau}{i}}{\tau}}(\primitivelattice{\indexededge{\tau}{j}})>0
\text{ (resp. }\primitivelattice{\normalvertexcone{\indexednegativefacet{\tau}{i}}{\tau}}(\primitivelattice{\indexededge{\tau}{j}})>0\text{)}, 
$$
and, since $\primitivelattice{\normalvertexcone{\indexedpositivefacet{\tau}{i}}{\tau}}$ 
(resp. $\primitivelattice{\normalvertexcone{\indexednegativefacet{\tau}{i}}{\tau}}$) and $\primitivelattice{\indexededge{\tau}{j}}$ 
are lattice points, we have 
$$
\primitivelattice{\normalvertexcone{\indexedpositivefacet{\tau}{i}}{\tau}}(\primitivelattice{\indexededge{\tau}{j}})\ge 1
\text{ (resp. }\primitivelattice{\normalvertexcone{\indexednegativefacet{\tau}{i}}{\tau}}(\primitivelattice{\indexededge{\tau}{j}})\ge 1\text{)}.
$$
Now recall that if an edge of $\tau$ intersects one of the planes $[\twotorusdelimiter=1]$
and $[\twotorusdelimiter=-1]$, then the intersection point is a lattice point. This lattice 
point must be the primitive lattice vector on this edge, otherwise $\twotorusdelimiter$ would 
have taken a noninteger value at the primitive lattice vector. Therefore, if $\indexededge{\tau}{j}$
intersects one of the planes $[\twotorusdelimiter=1]$
and $[\twotorusdelimiter=-1]$, 
then $\twotorusdelimiter(\primitivelattice{\indexededge{\tau}{j}})$ can only equal 1 or $-1$. 
If $\indexededge{\tau}{j}$ intersects none of these planes, then $\twotorusdelimiter$ vanishes on 
$\indexededge{\tau}{j}$ everywhere, in particular $\twotorusdelimiter(\primitivelattice{\indexededge{\tau}{j}})=0$.
Therefore, in all cases we have 
$|\twotorusdelimiter(\primitivelattice{\indexededge{\tau}{j}})|\le 1$. But then 
$$
(\primitivelattice{\normalvertexcone{\indexedpositivefacet{\tau}{i}}{\tau}}
+\twotorusdelimiter)(\primitivelattice{\indexededge{\tau}{j}})\ge 0
$$
$$
\text{(resp. }(\primitivelattice{\normalvertexcone{\indexednegativefacet{\tau}{i}}{\tau}}
-\twotorusdelimiter)(\primitivelattice{\indexededge{\tau}{j}})\ge 0,\text{ now the sign in front of 
$\twotorusdelimiter$ does not matter)}.
$$

Summarizing, we see that if $\indexededge{\tau}{j}$ is an arbitrary edge of $\tau$, 
then 
$$
(\primitivelattice{\normalvertexcone{\indexedpositivefacet{\tau}{i}}{\tau}}
+\twotorusdelimiter)(\primitivelattice{\indexededge{\tau}{j}})\ge 0
\text{ (resp. }(\primitivelattice{\normalvertexcone{\indexednegativefacet{\tau}{i}}{\tau}}
-\twotorusdelimiter)(\primitivelattice{\indexededge{\tau}{j}})\ge 0\text{)}.
$$
Therefore, 
$$
\primitivelattice{\normalvertexcone{\indexedpositivefacet{\tau}{i}}{\tau}}
+\twotorusdelimiter\in\tau^\vee
\text{ (resp. }\primitivelattice{\normalvertexcone{\indexednegativefacet{\tau}{i}}{\tau}}
-\twotorusdelimiter\in\tau^\vee\text{)}.
$$
\end{proof}

\begin{proposition}\label{lambda2classification}
Let $\indexedpositivefacet{\tau}{i}$ (resp. $\indexednegativefacet{\tau}{i}$), where $1\le i\le \numberofpositiveedges{\tau}-1$
(resp. $1\le i\le \numberofnegativeedges{\tau}-1$), be a facet of $\tau$. Let $\chi=b\chi_0$ (resp. $\chi=-b\chi_0$), where $b\in \NN$.

\begin{enumerate}
\item If $b=1$, then 
$$
\Span_\QQ(\Lambda^\chi_{\indexedpositivefacet{\tau}{i}})=\Span_\QQ(\normalvertexcone{\indexedpositivefacet{\tau}{i}}{\tau})
\text{ (resp. }\Span_\QQ(\Lambda^\chi_{\indexednegativefacet{\tau}{i}})=\Span_\QQ(\normalvertexcone{\indexednegativefacet{\tau}{i}}{\tau})\text{)}.
$$
\item If $\latticelength{\indexedpositivefacet{\tau}{i}\cap [\twotorusdelimiter=1]}\ge 2$
(resp. $\latticelength{\indexednegativefacet{\tau}{i}\cap [\twotorusdelimiter=-1]}\ge 2$)
and $2\le b\le \latticelength{\indexedpositivefacet{\tau}{i}\cap [\twotorusdelimiter=1]}$
(resp. $2\le b\le \latticelength{\indexednegativefacet{\tau}{i}\cap [\twotorusdelimiter=-1]}$), 
then 
$$
\Span_\QQ(\Lambda^\chi_{\indexedpositivefacet{\tau}{i}})=
\Span_\QQ(\twotorusdelimiter,\normalvertexcone{\indexedpositivefacet{\tau}{i}}{\tau})
$$
$$
\text{ (resp. }\Span_\QQ(\Lambda^\chi_{\indexednegativefacet{\tau}{i}})=
\Span_\QQ(\twotorusdelimiter,\normalvertexcone{\indexednegativefacet{\tau}{i}}{\tau})\text{)}.
$$
\item If $b>\latticelength{\indexedpositivefacet{\tau}{i}\cap [\twotorusdelimiter=1]}$
(resp. $b>\latticelength{\indexednegativefacet{\tau}{i}\cap [\twotorusdelimiter=-1]}$), 
then $\Span_\QQ(\Lambda^\chi_{\indexedpositivefacet{\tau}{i}})=\widetilde M_\QQ$.
\end{enumerate}
\end{proposition}

\begin{proof}
Again, the positive and the negative cases here are completely similar. This time let us consider the 
negative case.

Consider the lattices 
$$
\overline M=\widetilde M/(\widetilde M\cap \Span_\QQ(\normalvertexcone{\indexednegativefacet{\tau}{i}}{\tau}))
$$
and 
$$
\overline N=\widetilde N\cap\Span_\QQ(\indexednegativefacet{\tau}{i}).
$$
By the definition of 
$\normalvertexcone{\indexednegativefacet{\tau}{i}}{\tau}$,
a function from $\widetilde M$ vanishes on the whole $\overline N$ 
(which is a saturated sublattice of $\widetilde N$ by construction) if and 
only if this function is contained in $\widetilde M\cap \Span_\QQ(\normalvertexcone{\indexednegativefacet{\tau}{i}}{\tau})$. 
Therefore, $\overline M$ is the dual lattice of $\overline N$, and 
the values of elements of $\overline M$
at points from $\overline N$ are well-defined.
We denote the class of a function $\chi'\in \widetilde M$ in $\overline M$ by $\overline{\chi'}$.

Denote $a_1=\primitivelattice{\indexednegativeedge{\tau}{i}}$, $a_2=\primitivelattice{\indexednegativeedge{\tau}{i+1}}$.
Recall that $\partial\indexednegativefacet{\tau}{i}=\indexednegativeedge{\tau}{i}\cup \indexednegativeedge{\tau}{i+1}$, 
so $a_1, a_2\in \overline N$. We have already seen that $\twotorusdelimiter(a_1)=\twotorusdelimiter(a_2)=-1$
and that 
$$
a_1=\indexednegativeedge{\tau}{i}\cap[\twotorusdelimiter=-1],\quad a_2=\indexednegativeedge{\tau}{i+1}\cap[\twotorusdelimiter=-1].
$$
So, $a_1$, $a_2$, and $\overline{-\twotorusdelimiter}$ satisfy the hypothesis of Lemma \ref{smalllaticepoints}, 
and $\latticelength{a_1-a_2}=\latticelength{\indexednegativefacet{\tau}{i}\cap [\twotorusdelimiter=-1]}$.
Consider the set $\overline\Lambda_{\overline{-\twotorusdelimiter},a_1,a_2,b}$ from Lemma \ref{smalllaticepoints}.
It follows directly from the definitions of $\overline\Lambda_{\overline{-\twotorusdelimiter},a_1,a_2,b}$
and of $\Lambda^\chi_{\indexednegativefacet{\tau}{i}}$ that the image of $\Lambda^\chi_{\indexednegativefacet{\tau}{i}}$ 
under the canonical projection $\widetilde M\to \overline M$ is contained in $\overline\Lambda_{\overline{-\twotorusdelimiter},a_1,a_2,b}$.
Moreover, if $\widetilde{\lambda_j}$ is an element of the Hilbert basis of $\tau^\vee$ such that 
$\chi'=\overline{\widetilde{\lambda_j}}\in \overline\Lambda_{\overline{-\twotorusdelimiter},a_1,a_2,b}$, 
then 
$$
\widetilde{\lambda_j}(a_1)=\chi'(a_1)<b=(-b)\cdot(-1)=-b\twotorusdelimiter(a_1)=\chi(a_1),
$$
so $\widetilde{\lambda_j}\in \Lambda^\chi_{\indexednegativeedge{\tau}{i}}$.
Similarly, 
$$
\widetilde{\lambda_j}(a_2)=\chi'(a_2)<b=(-b)\cdot(-1)=-b\twotorusdelimiter(a_2)=\chi(a_2),
$$
so $\widetilde{\lambda_j}\in \Lambda^\chi_{\indexednegativeedge{\tau}{i+1}}$. Hence, 
$$
\widetilde{\lambda_j}\in \Lambda^\chi_{\indexednegativeedge{\tau}{i}}\cap 
\Lambda^\chi_{\indexednegativeedge{\tau}{i+1}}=\Lambda^\chi_{\indexednegativefacet{\tau}{i}}.
$$

Consider the case $b=1$. Then by Lemma \ref{smalllaticepoints}, 
$\Lambda_{\overline{-\twotorusdelimiter},a_1,a_2,b}=\{0\}$, and all elements of 
$\Lambda^\chi_{\indexednegativefacet{\tau}{i}}$
are in $\ker(\widetilde M\to \overline M)=\widetilde M\cap \Span_\QQ(\normalvertexcone{\indexednegativefacet{\tau}{i}}{\tau})$.
On the other hand, since $\normalvertexcone{\indexednegativefacet{\tau}{i}}{\tau}$ is a face 
of $\tau^\vee$, $\primitivelattice{\normalvertexcone{\indexednegativefacet{\tau}{i}}{\tau}}$
is an element of the Hilbert basis of $\tau^\vee$, 
$\primitivelattice{\normalvertexcone{\indexednegativefacet{\tau}{i}}{\tau}}=\widetilde{\lambda_j}$ for some $j$.
As we have seen previously, this means that $\widetilde{\lambda_j}\in\Lambda^\chi_{\indexednegativefacet{\tau}{i}}$.
Hence, $\Span_\QQ(\Lambda^\chi_{\indexednegativefacet{\tau}{i}})=\Span_\QQ(\normalvertexcone{\indexednegativefacet{\tau}{i}}{\tau})$.

Now suppose that $\latticelength{a_1-a_2}\ge 2$
and $2\le b\le \latticelength{a_1-a_2}$. Then by Lemma \ref{smalllaticepoints}, 
$\overline\Lambda_{\overline{-\twotorusdelimiter},a_1,a_2,b}$ is contained in the line generated by 
$\overline{-\twotorusdelimiter}$. Hence, $\Lambda^\chi_{\indexednegativefacet{\tau}{i}}$ is contained in the plane 
generated by $\normalvertexcone{\indexednegativefacet{\tau}{i}}{\tau}$ and $-\twotorusdelimiter$.
On the other hand, we already know that 
$\primitivelattice{\normalvertexcone{\indexednegativefacet{\tau}{i}}{\tau}}$
is an element of the Hilbert basis of $\tau^\vee$, 
and, since it represents the zero class in $\overline M$ and $0\in \overline\Lambda_{\overline{-\twotorusdelimiter},a_1,a_2,b}$, 
it is also contained in $\Lambda^\chi_{\indexednegativefacet{\tau}{i}}$.
By Lemma \ref{edgepluschiintau}, 
$\chi''=\primitivelattice{\normalvertexcone{\indexednegativefacet{\tau}{i}}{\tau}}-\twotorusdelimiter\in\tau^\vee$.
If $\chi''$ is not an element of the Hilbert basis of $\tau^\vee$, it can be decomposed into an integer positive 
linear combination of elements of the Hilbert basis. Since $\chi''(a_1)=\chi''(a_2)=1$, the elements of the 
Hilbert basis present in this combination may only take values 0 or 1 at $a_1$ and $a_2$ (in arbitrary order).
But if there exists $\widetilde{\lambda_k}$ such that $\widetilde{\lambda_k}(a_1)=1$ and $\widetilde{\lambda_k}(a_2)=0$, then 
$\overline{\widetilde{\lambda_k}}\in\overline\Lambda_{\overline{-\twotorusdelimiter},a_1,a_2,b}$, and this is a contradiction
with Lemma \ref{smalllaticepoints}. Similarly, one cannot have $\widetilde{\lambda_k}(a_1)=0$ and $\widetilde{\lambda_k}(a_2)=1$.
Hence, there exist an element $\widetilde{\lambda_k}$ of the Hilbert basis such that 
$\widetilde{\lambda_k}(a_1)=\widetilde{\lambda_k}(a_2)=1$. By Lemma \ref{smalllaticepoints}, 
$a_1$ and $a_2$ $\QQ$-generate $\overline N\otimes_\ZZ\QQ$. Therefore, 
elements of $\overline M$ are determined by their values at $a_1$ and $a_2$, and
$\overline{\widetilde{\lambda_k}}=\overline{-\twotorusdelimiter}\in\overline\Lambda_{\overline{-\twotorusdelimiter},a_1,a_2,b}$.
We already know that this means that $\widetilde{\lambda_k}\in \Lambda^\chi_{\indexednegativefacet{\tau}{i}}$.
Since $\overline{\widetilde{\lambda_k}}=\overline{-\twotorusdelimiter}$, 
$$
-\twotorusdelimiter-\widetilde{\lambda_k}\in
\Span_\QQ(\normalvertexcone{\indexednegativefacet{\tau}{i}}{\tau}),
$$
and $\widetilde{\lambda_k}$ and 
$\normalvertexcone{\indexednegativefacet{\tau}{i}}{\tau}$ together $\QQ$-generate the same plane as 
$-\twotorusdelimiter$ and $\normalvertexcone{\indexednegativefacet{\tau}{i}}{\tau}$ $\QQ$-generate, i.~e.
they $\QQ$-generate 
$\Span_\QQ(\twotorusdelimiter,\normalvertexcone{\indexednegativefacet{\tau}{i}}{\tau})$.
Therefore, 
$$
\Span_\QQ(\Lambda^\chi_{\indexednegativefacet{\tau}{i}})=
\Span_\QQ(\twotorusdelimiter,\normalvertexcone{\indexednegativefacet{\tau}{i}}{\tau}).
$$

Finally, let us consider the case $b>\latticelength{a_1-a_2}$.
By Lemma \ref{smalllaticepoints}, there exist 
$\chi_1,\chi_2\in \overline\Lambda_{\overline{-\twotorusdelimiter},a_1,a_2,b}$ such that 
$\chi_1(a_1)>0$, $\chi_1(a_2)=0$, $\chi_2(a_1)=0$, and $\chi_2(a_2)>0$.
Pick arbitrary $\chi'_1, \chi'_2\in \widetilde M$ such that 
$\overline{\chi'_1}=\chi_1$ and $\overline{\chi'_2}=\chi_2$.
We have $\chi'_1(a_1)>0$, $\chi'_1(a_2)=0$, $\chi'_2(a_1)=0$, and $\chi'_2(a_2)>0$, 
so, by the definitions of $\normalvertexcone{\indexednegativeedge{\tau}{i}}{\tau}$
and of $\normalvertexcone{\indexednegativeedge{\tau}{i+1}}{\tau}$, we have 
$\chi'_1\in \normalvertexcone{\indexednegativeedge{\tau}{i+1}}{\tau}$ 
and $\chi'_2\in \normalvertexcone{\indexednegativeedge{\tau}{i}}{\tau}$.
Therefore, we can apply Lemma \ref{shiftalongedge} to the facet 
$\indexednegativefacet{\tau}{i}$ of $\tau$, to the edge 
$\indexednegativeedge{\tau}{i+1}$ of $\tau$, and to the degree $\chi'_1$
and find another degree $\chi''_1$ such that $\chi''_1-\chi'_1$ is a multiple of 
$\primitivelattice{\indexednegativefacet{\tau}{i}}$ and $\chi''_1\in \normalvertexcone{\indexednegativeedge{\tau}{i+1}}{\tau}$.
Similarly, by Lemma \ref{shiftalongedge} applied to 
$\indexednegativefacet{\tau}{i}$, to $\indexednegativeedge{\tau}{i}$, and to 
$\chi'_2$, there exists a degree $\chi''_2\in \normalvertexcone{\indexednegativeedge{\tau}{i}}{\tau}$
such that $\chi''_2-\chi'_2$ is a multiple of 
$\primitivelattice{\indexednegativefacet{\tau}{i}}$.
In other words, $\overline{\chi''_1}=\overline{\chi'_1}=\chi_1$
and $\overline{\chi''_2}=\overline{\chi'_2}=\chi_2$.

Now we have degrees $\chi''_1, \chi''_2\in\tau^\vee$ satisfying the following conditions:
$\chi''_1(a_2)=\chi''_2(a_1)=0$, $0<\chi''_1(a_1)=\chi_1(a_1)<b$, 
$0<\chi''_2(a_2)=\chi_2(a_2)<b$. Decompose $\chi''_1$ into a positive integer linear combination
of $\widetilde{\lambda_j}$. The elements $\widetilde{\lambda_j}$ of the Hilbert basis 
occurring in this decomposition satisfy $\widetilde{\lambda_j}(a_2)=0$
and $0\le \widetilde{\lambda_j}(a_1)<b$, and for at least one of them we have 
$\widetilde{\lambda_j}(a_1)>0$. Similarly, there exists $\widetilde{\lambda_k}$
satisfying $\widetilde{\lambda_k}(a_1)=0$ and $0<\widetilde{\lambda_k}(a_2)<b$.
We can write this as 
$$
\widetilde{\lambda_j}(a_1)<b=(-b)\cdot(-1)=
-b\twotorusdelimiter(a_1)=\chi(a_1)
$$
and 
$$\widetilde{\lambda_j}(a_2)=0<b=(-b)\cdot(-1)=
-b\twotorusdelimiter(a_2)=\chi(a_2),
$$
so 
$\widetilde{\lambda_j}\in\Lambda^\chi_{\indexednegativefacet{\tau}{i}}$.
Similarly, $\widetilde{\lambda_k}\in\Lambda^\chi_{\indexednegativefacet{\tau}{i}}$.
Finally, as we saw previously, 
$\primitivelattice{\normalvertexcone{\indexednegativefacet{\tau}{i}}{\tau}}$
is an element of the Hilbert basis, its class in $\overline M$ is 
$0\in\overline\Lambda_{\overline{-\twotorusdelimiter},a_1,a_2,b}$, 
so 
$$
\primitivelattice{\normalvertexcone{\indexednegativefacet{\tau}{i}}{\tau}}
\in\Lambda^\chi_{\indexednegativefacet{\tau}{i}}.
$$

Now we claim that $\widetilde{\lambda_j}$, $\widetilde{\lambda_k}$, and 
$\primitivelattice{\normalvertexcone{\indexednegativefacet{\tau}{i}}{\tau}}$
$\QQ$-generate $\widetilde M_\QQ$. Indeed, $\widetilde{\lambda_j}(a_1)\ne 0$, 
while 
$$
\primitivelattice{\normalvertexcone{\indexednegativefacet{\tau}{i}}{\tau}}(a_1)=0
$$
by the definition of $\normalvertexcone{\indexednegativefacet{\tau}{i}}{\tau}$.
Hence, $\widetilde{\lambda_j}$ 
and $\primitivelattice{\normalvertexcone{\indexednegativefacet{\tau}{i}}{\tau}}$
are linearly independent and
$\QQ$-generate $\Span_\QQ(\normalvertexcone{\indexednegativeedge{\tau}{i+1}}{\tau})$.
Similarly, $\widetilde{\lambda_k}$ 
and $\primitivelattice{\normalvertexcone{\indexednegativefacet{\tau}{i}}{\tau}}$
$\QQ$-generate $\Span_\QQ(\normalvertexcone{\indexednegativeedge{\tau}{i}}{\tau})$.
The linear span of these two planes can be two-dimensional only if these two planes coincide, 
but $\normalvertexcone{\indexednegativefacet{\tau}{i}}{\tau}$ and 
$\normalvertexcone{\indexednegativeedge{\tau}{i+1}}{\tau}$ are two different facets of 
$\tau^\vee$, so 
$$
\Span_\QQ(\widetilde{\lambda_j}, \widetilde{\lambda_k}, 
\primitivelattice{\normalvertexcone{\indexednegativefacet{\tau}{i}}{\tau}})=\widetilde M_\QQ,
$$
and 
$$
\Span_\QQ(\Lambda^\chi_{\indexednegativefacet{\tau}{i}})=\widetilde M_\QQ.
$$
\end{proof}

\begin{corollary}\label{kerlambda1}
If $\chi=\twotorusdelimiter$ (resp. $\chi=-\twotorusdelimiter$), then 
$\ker((\Lambda^{\chi,1}\otimes_\ZZ\CC)^*\to(\Lambda^{\chi,2}\otimes_\ZZ\CC)^*)$ equals the space of 
sequences of the form $(\uidescriptionfunction1,\ldots,\uidescriptionfunction{\numberofpositiveedges{\tau}})$
(resp. $(\uidescriptionfunction1,\ldots,\uidescriptionfunction{\numberofnegativeedges{\tau}})$), 
where $\uidescriptionfunction i$ is a linear function on 
$\Span_\QQ(\normalvertexcone{\indexedpositiveedge{\tau}{i}}{\tau})\otimes_\QQ\CC$ 
(resp on $\Span_\QQ(\normalvertexcone{\indexednegativeedge{\tau}{i}}{\tau})\otimes_\QQ\CC$), 
and where 
$$
\uidescriptionfunction i|_{\Span_\QQ(\normalvertexcone{\indexedpositivefacet{\tau}{i}}{\tau})\otimes_\QQ\CC}=
\uidescriptionfunction{i+1}|_{\Span_\QQ(\normalvertexcone{\indexedpositivefacet{\tau}{i}}{\tau})\otimes_\QQ\CC}
$$
for $1\le i<\numberofpositiveedges{\tau}$
(resp. 
$$
\uidescriptionfunction i|_{\Span_\QQ(\normalvertexcone{\indexednegativefacet{\tau}{i}}{\tau})\otimes_\QQ\CC}=
\uidescriptionfunction{i+1}|_{\Span_\QQ(\normalvertexcone{\indexednegativefacet{\tau}{i}}{\tau})\otimes_\QQ\CC}
$$
for $1\le i<\numberofnegativeedges{\tau}$).
\end{corollary}

\begin{proof}
The claim follows directly from Corollary \ref{lambda1classification}, Lemma \ref{lambda2noboundaryconditionpos}, 
Lemma \ref{lambda2noboundaryconditionneg}, and Proposition \ref{lambda2classification}.
\end{proof}

\begin{corollary}\label{kerlambda2}
If $\chi=a\twotorusdelimiter$ (resp. $\chi=-a\twotorusdelimiter$), where $a\in\NN$, $a\ge 2$, then 
$\ker((\Lambda^{\chi,1}\otimes_\ZZ\CC)^*\to(\Lambda^{\chi,2}\otimes_\ZZ\CC)^*)$ equals the space of 
sequences of the form $(\uidescriptionfunction1,\ldots,\uidescriptionfunction{\numberofpositiveedges{\tau}})$
(resp. $(\uidescriptionfunction1,\ldots,\uidescriptionfunction{\numberofnegativeedges{\tau}})$), 
where $\uidescriptionfunction i$ are linear functions on 
$\widetilde M_\QQ\otimes_\QQ\CC$ 
satisfying the following conditions for $1\le i<\numberofpositiveedges{\tau}$ 
(resp. for $1\le i<\numberofnegativeedges{\tau}$):
\begin{enumerate}
\item If $b\le \latticelength{\indexedpositivefacet{\tau}{i}\cap [\twotorusdelimiter=1]}$
(resp. $b\le \latticelength{\indexednegativefacet{\tau}{i}\cap [\twotorusdelimiter=-1]}$), then 
$$
\uidescriptionfunction i|_{\Span_\QQ(\twotorusdelimiter, \normalvertexcone{\indexedpositivefacet{\tau}{i}}{\tau})\otimes_\QQ\CC}=
\uidescriptionfunction{i+1}|_{\Span_\QQ(\twotorusdelimiter, \normalvertexcone{\indexedpositivefacet{\tau}{i}}{\tau})\otimes_\QQ\CC}
$$
(resp. 
$$
\uidescriptionfunction i|_{\Span_\QQ(\twotorusdelimiter, \normalvertexcone{\indexednegativefacet{\tau}{i}}{\tau})\otimes_\QQ\CC}=
\uidescriptionfunction{i+1}|_{\Span_\QQ(\twotorusdelimiter, \normalvertexcone{\indexednegativefacet{\tau}{i}}{\tau})\otimes_\QQ\CC}
$$
).
\item If $b> \latticelength{\indexedpositivefacet{\tau}{i}\cap [\twotorusdelimiter=1]}$
(resp. $b> \latticelength{\indexednegativefacet{\tau}{i}\cap [\twotorusdelimiter=-1]}$), then 
$\uidescriptionfunction i=\uidescriptionfunction{i+1}$.
\end{enumerate}
\end{corollary}

\begin{proof}
The claim follows directly from Corollary \ref{lambda1classification}, Lemma \ref{lambda2noboundaryconditionpos}, 
Lemma \ref{lambda2noboundaryconditionneg}, and Proposition \ref{lambda2classification}.
\end{proof}

Now we construct a less invariant, but more explicit vector space isomorphic to 
$\ker((\Lambda^{\chi,1}\otimes_\ZZ\CC)^*\to(\Lambda^{\chi,2}\otimes_\ZZ\CC)^*)$.
Namely, denote by $\abstractvectorspace_{2,1,1}$ (resp. by $\abstractvectorspace_{2,1,-1}$)
the space of sequences of the form $(\uidescriptionfunction0',\ldots, \uidescriptionfunction{\numberofpositiveedges{\tau}}')$
(resp. $(\uidescriptionfunction0',\ldots, \uidescriptionfunction{\numberofnegativeedges{\tau}}')$), where 
$\uidescriptionfunction{i}'$ is a linear function on $\Span_\QQ(\normalvertexcone{\indexedpositivefacet{\tau}{i}}{\tau})\otimes_\QQ\CC$
(resp. on $\Span_\QQ(\normalvertexcone{\indexednegativefacet{\tau}{i}}{\tau})\otimes_\QQ\CC$).
For $a\in\NN$, $a\ge 2$, denote by 
$\abstractvectorspace_{2,1,a}$ (resp. by $\abstractvectorspace_{2,1,-a}$)
the space of sequences of the form 
$(\uidescriptionfunction1',\ldots, \uidescriptionfunction{\numberofpositiveedges{\tau}}')$
(resp. $(\uidescriptionfunction1',\ldots, \uidescriptionfunction{\numberofnegativeedges{\tau}}')$), where 
\begin{enumerate}
\item $\uidescriptionfunction1'$ is a linear function on $\widetilde M_\QQ\otimes_\QQ\CC$.
\item If $1<i\le\numberofpositiveedges{\tau}$ (resp. $1<i\le\numberofnegativeedges{\tau}$)
and $a\le\latticelength{\indexededge{\stdpolyhedronletter_0}{i-1}}$
(resp. $a\le\latticelength{\indexededge{\stdpolyhedronletter_\infty}{i-1}}$), then
$\uidescriptionfunction{i}'$ is a linear function on 
$$
(\widetilde M_\QQ\otimes_\QQ\CC)/(\Span_\QQ(\twotorusdelimiter, \normalvertexcone{\indexedpositivefacet{\tau}{i-1}}{\tau})\otimes_\QQ\CC)
$$
$$
\text{(resp. on }
(\widetilde M_\QQ\otimes_\QQ\CC)/(\Span_\QQ(\twotorusdelimiter, \normalvertexcone{\indexednegativefacet{\tau}{i-1}}{\tau})\otimes_\QQ\CC)\text{)}.
$$
\item If $1<i\le\numberofpositiveedges{\tau}$ (resp. $1<i\le\numberofnegativeedges{\tau}$)
and $a>\latticelength{\indexededge{\stdpolyhedronletter_0}{i-1}}$
(resp. $a>\latticelength{\indexededge{\stdpolyhedronletter_\infty}{i-1}}$), then
$\uidescriptionfunction{i}'=0$.
\end{enumerate}

\begin{lemma}\label{kerlambda1bruteforce}
If $\chi=\twotorusdelimiter$ (resp. $\chi=-\twotorusdelimiter$), then 
$\ker((\Lambda^{\chi,1}\otimes_\ZZ\CC)^*\to(\Lambda^{\chi,2}\otimes_\ZZ\CC)^*)$
is isomorphic to $\abstractvectorspace_{2,1,1}$ (resp. to $\abstractvectorspace_{2,1,-1}$).
After this identification, the map 
$(\Lambda^{\chi,0}\otimes_\ZZ\CC)^*\to(\Lambda^{\chi,1}\otimes_\ZZ\CC)^*$
(in fact, the map $(\Lambda^{\chi,0}\otimes_\ZZ\CC)^*\to
\ker((\Lambda^{\chi,1}\otimes_\ZZ\CC)^*\to(\Lambda^{\chi,2}\otimes_\ZZ\CC)^*)$)
becomes the following map: it maps $\uidescriptionfunction{}\in(\Lambda^{\chi,0}\otimes_\ZZ\CC)^*$
to the sequence of restrictions of $\uidescriptionfunction{}$
to the lines
$\Span_\QQ(\normalvertexcone{\indexedpositivefacet{\tau}{i}}{\tau})\otimes_\QQ\CC$
for $0\le i\le \numberofpositiveedges{\tau}$
(resp. $\Span_\QQ(\normalvertexcone{\indexednegativefacet{\tau}{i}}{\tau})\otimes_\QQ\CC$
for $0\le i\le \numberofnegativeedges{\tau}$).
\end{lemma}

\begin{proof}
Again, the positive and the negative cases are completely analogous, 
so we consider only one of them, for example, the case when 
$\chi=-\twotorusdelimiter$.

First, we should note that a function from $(\Lambda^{\chi,0}\otimes_\ZZ\CC)^*$
is really defined on all lines $\Span_\QQ(\normalvertexcone{\indexednegativefacet{\tau}{i}}{\tau})\otimes_\QQ\CC$
(and the restriction mentioned in the statement of the Lemma really exists) by Lemma \ref{lambda0classification}
since each normal cone $\normalvertexcone{\indexednegativefacet{\tau}{i}}{\tau}$ (for $0\le i\le \numberofnegativeedges{\tau}$)
is contained in (the boundary of) a cone $\normalvertexcone{\indexednegativeedge{\tau}{j}}{\tau}$
for some $j$, $1\le j\le \numberofnegativeedges{\tau}$.

The isomorphism is constructed as follows. Given a sequence 
$$
(\uidescriptionfunction1,\ldots,\uidescriptionfunction{\numberofnegativeedges{\tau}})\in
\ker((\Lambda^{\chi,1}\otimes_\ZZ\CC)^*\to(\Lambda^{\chi,2}\otimes_\ZZ\CC)^*),
$$
we set 
$$
\uidescriptionfunction0'=\uidescriptionfunction1|_{\Span_\QQ(\normalvertexcone{\indexednegativefacet{\tau}{0}}{\tau})\otimes_\QQ\CC}
$$
and
$$
\uidescriptionfunction i'=\uidescriptionfunction i|_{\Span_\QQ(\normalvertexcone{\indexednegativefacet{\tau}{i}}{\tau})\otimes_\QQ\CC}
$$
for $0<i\le \numberofnegativeedges{\tau}$ and say that 
$(\uidescriptionfunction1,\ldots,\uidescriptionfunction{\numberofnegativeedges{\tau}})\mapsto 
(\uidescriptionfunction0',\ldots,\uidescriptionfunction{\numberofnegativeedges{\tau}}')$. Observe that 
by Corollary \ref{kerlambda1}, we also have
$$
\uidescriptionfunction {i-1}'=\uidescriptionfunction i|_{\Span_\QQ(\normalvertexcone{\indexednegativefacet{\tau}{i-1}}{\tau})\otimes_\QQ\CC}
$$
for $0< i\le\numberofnegativeedges{\tau}$. Since 
$\Span_\QQ(\normalvertexcone{\indexednegativeedge{\tau}{i}}{\tau})\otimes_\QQ\CC$
is a two-dimensional space, and 
$\Span_\QQ(\normalvertexcone{\indexednegativefacet{\tau}{i-1}}{\tau})\otimes_\QQ\CC$
and 
$\Span_\QQ(\normalvertexcone{\indexednegativefacet{\tau}{i}}{\tau})\otimes_\QQ\CC$
are its noncoinciding one\-/dimensional subspaces, a linear function on 
$\Span_\QQ(\normalvertexcone{\indexednegativeedge{\tau}{i}}{\tau})\otimes_\QQ\CC$
is uniquely determined by its restrictions to 
$\Span_\QQ(\normalvertexcone{\indexednegativefacet{\tau}{i}}{\tau})\otimes_\QQ\CC$
and
$\Span_\QQ(\normalvertexcone{\indexednegativeedge{\tau}{i}}{\tau})\otimes_\QQ\CC$,
and these restrictions can be arbitrary linear functions.
Therefore, the map we have constructed is really an isomorphism.
The correctness of the explicit description of the map 
$(\Lambda^{\chi,0}\otimes_\ZZ\CC)^*\to\abstractvectorspace_{2,1,-1}$ 
in the statement of the lemma
follows directly 
from the definition of 
the map $(\Lambda^{\chi,0}\otimes_\ZZ\CC)^*\to(\Lambda^{\chi,1}\otimes_\ZZ\CC)^*$
and of the isomorphism between 
$\ker((\Lambda^{\chi,1}\otimes_\ZZ\CC)^*\to(\Lambda^{\chi,2}\otimes_\ZZ\CC)^*)$
and $\abstractvectorspace_{2,1,-1}$.
\end{proof}

\begin{lemma}\label{kerlambda2bruteforce}
If $\chi=a\twotorusdelimiter$ (resp. $\chi=-a\twotorusdelimiter$), where $a\in \NN$, $a\ge 2$, then 
$\ker((\Lambda^{\chi,1}\otimes_\ZZ\CC)^*\to(\Lambda^{\chi,2}\otimes_\ZZ\CC)^*)$
is isomorphic to $\abstractvectorspace_{2,1,a}$ (resp. to $\abstractvectorspace_{2,1,-a}$).
After this identification, the map 
$(\Lambda^{\chi,0}\otimes_\ZZ\CC)^*\to(\Lambda^{\chi,1}\otimes_\ZZ\CC)^*$
becomes the following map: it maps $\uidescriptionfunction{}\in(\Lambda^{\chi,0}\otimes_\ZZ\CC)^*=(\widetilde M\otimes_\QQ\CC)^*$
to $(\uidescriptionfunction{},0,\ldots,0)$.
\end{lemma}

\begin{proof}
This time let us consider the case $\chi=a\twotorusdelimiter$, the other case is completely similar.

First, let us construct a map from $\ker((\Lambda^{\chi,1}\otimes_\ZZ\CC)^*\to(\Lambda^{\chi,2}\otimes_\ZZ\CC)^*)$
to $\abstractvectorspace_{2,1,a}$.
Given a sequence
$$
(\uidescriptionfunction1,\ldots,\uidescriptionfunction{\numberofpositiveedges{\tau}})\in
\ker((\Lambda^{\chi,1}\otimes_\ZZ\CC)^*\to(\Lambda^{\chi,2}\otimes_\ZZ\CC)^*),
$$
we set 
$$
\uidescriptionfunction1'=\uidescriptionfunction1
$$
and
$$
\uidescriptionfunction i'=\uidescriptionfunction i-\uidescriptionfunction{i-1}
$$
for $1<i\le \numberofpositiveedges{\tau}$. By Corollary \ref{kerlambda2}, 
$$
\uidescriptionfunction i|_{\Span_\QQ(\twotorusdelimiter, \normalvertexcone{\indexedpositivefacet{\tau}{i-1}}{\tau})\otimes_\QQ\CC}=
\uidescriptionfunction{i-1}|_{\Span_\QQ(\twotorusdelimiter, \normalvertexcone{\indexedpositivefacet{\tau}{i-1}}{\tau})\otimes_\QQ\CC}
$$
if $a\le\latticelength{\indexedpositivefacet{\tau}{i}\cap [\twotorusdelimiter=1]}$, 
and 
$\uidescriptionfunction i=\uidescriptionfunction{i-1}$ if $a>\latticelength{\indexedpositivefacet{\tau}{i}\cap [\twotorusdelimiter=1]}$.
(here $1<i\le \numberofpositiveedges{\tau}$. Recall that 
$\indexededge{\stdpolyhedronletter_0}{i}=\indexedpositivefacet{\tau}{i}\cap[\twotorusdelimiter=1]$.
So, we can say that 
$$
\uidescriptionfunction i'|_{\Span_\QQ(\twotorusdelimiter, 
\normalvertexcone{\indexedpositivefacet{\tau}{i-1}}{\tau})\otimes_\QQ\CC}=
(\uidescriptionfunction i-\uidescriptionfunction{i-1})|_{\Span_\QQ(\twotorusdelimiter, 
\normalvertexcone{\indexedpositivefacet{\tau}{i-1}}{\tau})\otimes_\QQ\CC}=0
$$
if $a\le\latticelength{\indexededge{\stdpolyhedronletter_0}{i-1}}$, 
and 
$\uidescriptionfunction i'=\uidescriptionfunction i-\uidescriptionfunction{i-1}=0$
if $a>\latticelength{\indexededge{\stdpolyhedronletter_0}{i-1}}$.
Therefore, 
$(\uidescriptionfunction1',\ldots,\uidescriptionfunction{\numberofpositiveedges{\tau}}')$
really defines an element of $\abstractvectorspace_{2,1,a}$, and we say that 
$(\uidescriptionfunction1,\ldots,\uidescriptionfunction{\numberofpositiveedges{\tau}})\mapsto
(\uidescriptionfunction1',\ldots,\uidescriptionfunction{\numberofpositiveedges{\tau}}')$.

The inverse map can be constructed by induction on $i$. 
Let $(\uidescriptionfunction1',\ldots,\uidescriptionfunction{\numberofpositiveedges{\tau}}')\in\abstractvectorspace_{2,1,a}$.
First, set $\uidescriptionfunction1=\uidescriptionfunction1'$. Now suppose that 
we already have $\uidescriptionfunction{i-1}\in (\widetilde M_\QQ\otimes_\QQ\CC)^*$.
If $a>\latticelength{\indexededge{\stdpolyhedronletter_0}{i-1}}$, set 
$\uidescriptionfunction i=\uidescriptionfunction{i-1}$. Otherwise, 
$\uidescriptionfunction i'$ is a linear function on
$(\widetilde M_\QQ\otimes_\QQ\CC)/(\Span_\QQ(\twotorusdelimiter, 
\normalvertexcone{\indexednegativefacet{\tau}{i-1}}{\tau})\otimes_\QQ\CC)$.
It gives rise to a function on $\widetilde M_\QQ\otimes_\QQ\CC$, 
which vanishes on 
$\Span_\QQ(\twotorusdelimiter, 
\normalvertexcone{\indexednegativefacet{\tau}{i-1}}{\tau})\otimes_\QQ\CC$
and which we also denote by $\uidescriptionfunction i'$. Set 
$\uidescriptionfunction i=\uidescriptionfunction{i-1}+\uidescriptionfunction i'$. 
Then 
$$
(\uidescriptionfunction i-\uidescriptionfunction{i-1})|_{\Span_\QQ(\twotorusdelimiter, 
\normalvertexcone{\indexedpositivefacet{\tau}{i-1}}{\tau})\otimes_\QQ\CC}=0.
$$
Now we have a sequence $(\uidescriptionfunction1,\ldots,\uidescriptionfunction{\numberofpositiveedges{\tau}})$
of functions on $\widetilde M_\QQ\otimes_\QQ\CC$, 
and by Corollary \ref{kerlambda2}, 
$(\uidescriptionfunction1,\ldots,\uidescriptionfunction{\numberofpositiveedges{\tau}})
\in
\ker((\Lambda^{\chi,1}\otimes_\ZZ\CC)^*\to(\Lambda^{\chi,2}\otimes_\ZZ\CC)^*)$.
So, we have constructed a map $\abstractvectorspace_{2,1,a}\to
\ker((\Lambda^{\chi,1}\otimes_\ZZ\CC)^*\to(\Lambda^{\chi,2}\otimes_\ZZ\CC)^*)$. It is clear from the construction
that the two maps we have are mutually inverse.

By Corollary \ref{lambda0classification}, $(\Lambda^{\chi,0}\otimes_\ZZ\CC)^*=(\widetilde M\otimes_\QQ\CC)^*$.
Again, it is clear from the definition of the map 
$(\Lambda^{\chi,0}\otimes_\ZZ\CC)^*\to(\Lambda^{\chi,1}\otimes_\ZZ\CC)^*$
and from the construction of the isomorphism between 
$\ker((\Lambda^{\chi,1}\otimes_\ZZ\CC)^*\to(\Lambda^{\chi,2}\otimes_\ZZ\CC)^*)$
and $\abstractvectorspace_{2,1,a}$ that after this identification 
$\ker((\Lambda^{\chi,1}\otimes_\ZZ\CC)^*\to(\Lambda^{\chi,2}\otimes_\ZZ\CC)^*)\cong
\abstractvectorspace_{2,1,a}$
the map 
$(\Lambda^{\chi,0}\otimes_\ZZ\CC)^*\to(\Lambda^{\chi,1}\otimes_\ZZ\CC)^*$
becomes the map 
$$
\Big(\uidescriptionfunction{}\in (\widetilde M\otimes_\QQ\CC)^*\Big)\mapsto 
\Big((\uidescriptionfunction{},0,\ldots,0)\in\abstractvectorspace_{2,1,a}\Big).
$$
\end{proof}

\begin{corollary}\label{kerlambda1dim}
If $\chi=\twotorusdelimiter$ (resp. $\chi=-\twotorusdelimiter$) and $\numberofpositiveedges{\tau}=1$
(resp. $\numberofpositiveedges{\tau}=1$), then $\dim T^1_{-\chi}(X)=0$.

If $\chi=\twotorusdelimiter$ (resp. $\chi=-\twotorusdelimiter$) and $\numberofpositiveedges{\tau}\ge 2$
(resp. $\numberofpositiveedges{\tau}\ge 2$), then $\dim T^1_{-\chi}(X)=\numberofpositiveedges{\tau}-2$
(resp. $\dim T^1_{-\chi}(X)=\numberofnegativeedges{\tau}-2$).
\end{corollary}

\begin{proof}
We consider the case $\chi=\twotorusdelimiter$, the other case is completely similar. Note that 
$\dim\Span_\QQ(\normalvertexcone{\indexedpositivefacet{\tau}{i}}{\tau})\otimes_\QQ\CC=1$, so 
$\dim\abstractvectorspace_{2,1,1}=\numberofpositiveedges{\tau}+1$. Also note that
it follows from the description of $\Span_{\widetilde M}(\Lambda^{0,\chi})\otimes_\ZZ\CC$
in Corollary \ref{lambda0classification} and from Lemma \ref{kerlambda1bruteforce}
that the map $(\Lambda^{\chi,0}\otimes_\ZZ\CC)^*\to\abstractvectorspace_{2,1,1}$ is in fact an embedding, 
so $\dim T^1_\chi(X)=\dim\abstractvectorspace_{2,1,1}-\dim(\Lambda^{\chi,0}\otimes_\ZZ\CC)^*$.
Now, since $\normalvertexcone{\indexedpositivefacet{\tau}{i}}$ for different $i$ are different edges of 
$\tau^\vee$, we have $\dim(\Lambda^{\chi,0}\otimes_\ZZ\CC)^*=\min(3,\numberofpositiveedges{\tau}+1)$.
Thus, $\dim(\Lambda^{\chi,0}\otimes_\ZZ\CC)^*=2$ if $\numberofpositiveedges{\tau}=1$ and
$\dim(\Lambda^{\chi,0}\otimes_\ZZ\CC)^*=3$ if $\numberofpositiveedges{\tau}\ge 2$.
Finally, we have $\dim\abstractvectorspace_{2,1,1}=1+1-2=0$ if $\numberofpositiveedges{\tau}=1$ and
$\dim\abstractvectorspace_{2,1,1}=\numberofpositiveedges{\tau}+1-3=\numberofpositiveedges{\tau}-2$ if 
$\numberofpositiveedges{\tau}\ge 2$.
\end{proof}

\begin{proposition}\label{kerlambda2dim}
If $\chi=a\twotorusdelimiter$ (resp. $\chi=-a\twotorusdelimiter$), where $a\in \NN$, $a\ge 2$, then 
$\dim T^1_{-\chi}(X)$ equals the number of indices $i$ such that $1\le i<\numberofpositiveedges{\tau}$
(resp. $1\le i<\numberofnegativeedges{\tau}$) and 
$a\le\latticelength{\indexededge{\stdpolyhedronletter_0}{i}}$
(resp. $a\le\latticelength{\indexededge{\stdpolyhedronletter_\infty}{i}}$).
\end{proposition}

\begin{proof}
This follows directly from the definition of $\abstractvectorspace_{2,1,a}$ and 
Lemma \ref{kerlambda2bruteforce}.
\end{proof}

Now it is already easy to deduce Theorem \ref{t1aslatticelength} in the case when $X$ is in fact toric 
from Theorem \ref{t1toric}. First, let us compute the sum
$$
\sum_{a=2}^\infty \dim T^1_{-a\twotorusdelimiter}(X).
$$
By Proposition \ref{kerlambda2dim}, this sum can be decomposed into $\numberofpositiveedges{\tau}-1=
\numberofvertices{\stdpolyhedronletter_0}-1$ 
sums (indexed by $i=1,\ldots, \numberofpositiveedges{\tau}-1$), and each 
of these sums contributes 1 for $2\le a\le \latticelength{\indexededge{\stdpolyhedronletter_0}{i}}$
and 0 for larger values of $a$. Therefore, the $i$th of these sums equals 
$\latticelength{\indexededge{\stdpolyhedronletter_0}{i}}-1$, and we have
$$
\sum_{a=2}^\infty \dim T^1_{-a\twotorusdelimiter}(X)=
\sum_{i=1}^{\numberofvertices{\stdpolyhedronletter_0}-1}(\latticelength{\indexededge{\stdpolyhedronletter_0}{i}}-1).
$$
Observe that this sum vanishes if $\numberofvertices{\stdpolyhedronletter_0}=1$ 
(i.~e. if $0\in\PP^1$ is a removable special point).
Similarly, 
$$
\sum_{a=-2}^{-\infty} \dim T^1_{-a\twotorusdelimiter}(X)=
\sum_{i=1}^{\numberofvertices{\stdpolyhedronletter_\infty}-1}(\latticelength{\indexededge{\stdpolyhedronletter_\infty}{i}}-1).
$$
And again, this sum vanishes if $\numberofvertices{\stdpolyhedronletter_\infty}=1$, 
i.~e. if $\infty\in\PP^1$ is a removable special point.
Now, by Corollary \ref{kerlambda1dim}, $\dim T^1_{-\twotorusdelimiter}(X)=0$ if 
$0\in\PP^1$ is a removable special point, and $\dim T^1_{-\twotorusdelimiter}(X)=\numberofvertices{\stdpolyhedronletter_0}-2$
otherwise. Similarly, 
$\dim T^1_{\twotorusdelimiter}(X)=0$ if 
$\infty\in\PP^1$ is a removable special point, 
$\dim T^1_{-\twotorusdelimiter}(X)=\numberofvertices{\stdpolyhedronletter_\infty}-2$
otherwise.
Hence, if $0\in\PP^1$ is a removable special point, then 
$$
\sum_{a=1}^\infty \dim T^1_{-a\twotorusdelimiter}(X)=0,
$$
and if $0\in\PP^1$ is an essential special point, then
\begin{multline*}
\sum_{a=1}^\infty \dim T^1_{-a\twotorusdelimiter}(X)=
\numberofvertices{\stdpolyhedronletter_0}-2+\sum_{i=1}^{\numberofvertices{\stdpolyhedronletter_0}-1}(\latticelength{\indexededge{\stdpolyhedronletter_0}{i}}-1)\\
=-1+\numberofvertices{\stdpolyhedronletter_0}-1+\sum_{i=1}^{\numberofvertices{\stdpolyhedronletter_0}-1}(\latticelength{\indexededge{\stdpolyhedronletter_0}{i}}-1)
=-1+\sum_{i=1}^{\numberofvertices{\stdpolyhedronletter_0}-1}(\latticelength{\indexededge{\stdpolyhedronletter_0}{i}}).
\end{multline*}
Similarly, if $\infty\in\PP^1$ is a removable special point, then 
$$
\sum_{a=-1}^{-\infty} \dim T^1_{-a\twotorusdelimiter}(X)=0,
$$
and if $\infty\in\PP^1$ is an essential special point, then
$$
\sum_{a=-1}^{-\infty} \dim T^1_{-a\twotorusdelimiter}(X)
=-1+\sum_{i=1}^{\numberofvertices{\stdpolyhedronletter_\infty}-1}(\latticelength{\indexededge{\stdpolyhedronletter_\infty}{i}}).
$$
Finally, recall that by Corollary \ref{chieq0trivial}, $\dim T^1_{0\in \widetilde M}(X)=0$, and we get the 
formula from Theorem \ref{t1aslatticelength}.

\chapter{A formally versal $T$-equivariant deformation over affine space}\label{sectksmtoric}

\section{Construction of the deformation}\label{sectversalconstruction}

In this section we construct a formally versal $T$-equivariant deformation of a $T$-variety $X$ over an affine space used as the parameter space. 
The Kodaira-Spencer map of this deformation maps the tangent space to this parameter space surjectively onto 
the zeroth graded component of $T^1(X)$. These properties of the deformation 
will enable us to prove
that the deformation has some good 
versality properties, namely so-called formal versality.

We maintain the notations and the assumptions from the Introduction. Recall that we have a polyhedral 
divisor $\mathcal D=\sum_{i=1}^{\numberofdivisorpoints}p_i\otimes \Delta_i$, 
where $\Delta_i\subset N_\QQ$ are polyhedra, and all their vertices are lattice points. From now on, without loss of generality, we may and will 
suppose that the point with coordinate $\infty$ on $\PP^1$ is a removable special point, $p_{\numberofdivisorpoints}=\infty$. 
Recall that if we add a principal polyhedral divisor to $\mathcal D$, the $T$-variety 
will not change. So, after we add several principal polyhedral divisors, each of which has two (removable)
special points, $p_i$ ($1\le i< \numberofdivisorpoints$) and $\infty$, to $\mathcal D$, we may suppose that 
$\indexedvertexpt{p}{1}=0$ (the origin in $N$) for all special points $p$ except $\infty$. In other words, 
$\indexededgept{p}{0}$ is always a ray, which begins at the origin.

\begin{remark}\label{finitepointstrivial}
After these changes, all special points except $p_{\numberofdivisorpoints}$
will be either essential or trivial.
\end{remark}

\begin{lemma}\label{finitedivisornonpos}
If $\stdpolyhedronletter\subset N_\QQ$ is a polyhedron with tail cone $\6$ and 
such that $\indexedvertex{\stdpolyhedronletter}{1}=0$,
then its individual evaluation function takes only nonpositive values.
\end{lemma}

\begin{proof}
It is clear that if $\indexedvertex{\stdpolyhedronletter}{1}=0$, then $\6\subseteq \stdpolyhedronletter$.
\end{proof}

\begin{lemma}\label{inftydivisornonneg}
In the assumptions stated above, the individual evaluation function of $\stdpolyhedronletter_{p_\numberofdivisorpoints}$
takes only nonnegative values and takes positive values on the interior of $\sck$. Therefore, 
$\stdpolyhedronletter_{p_\numberofdivisorpoints}\subset \6$.
\end{lemma}

\begin{proof}
If $\chi\in\sck\cap M$, then $\deg\mathcal D(\chi)=\sum_{i=1}^\numberofdivisorpoints\eval_{\stdpolyhedronletter_{p_i}}(\chi)\ge 0$
since $\mathcal D(\chi)$ is semiample. Since $\eval_{\stdpolyhedronletter_{p_i}}(\chi)\le 0$ for $1\le i<\numberofdivisorpoints$, 
$\eval_{\stdpolyhedronletter_{p_\numberofdivisorpoints}}(\chi)\ge 0$. If $\chi$, moreover, is in the interior of $\sck$
then $\deg\mathcal D(\chi)>0$ since $\mathcal D(\chi)$ is big. So, $\eval_{\stdpolyhedronletter_{p_\numberofdivisorpoints}}(\chi)>0$.
\end{proof}

Denote $\totalpolyhedronletter=\sum_{i=1}^{\numberofdivisorpoints}\stdpolyhedronletter_{p_i}$. 

\begin{remark}
If $\chi\in\sck\cap M$, then
$\deg\mathcal D(\chi)=\eval_{\totalpolyhedronletter}(\chi)$ and $\dim\Gamma(\PP^1,\OO(\mathcal D(\chi)))=\eval_{\totalpolyhedronletter}(\chi)+1$.
\end{remark}

\begin{definition}
We call a polyhedron $\stdprimitivepolyhedronletter\subset N_\QQ$ with tail cone $\6$ \textit{primitive} if:
\begin{enumerate}
\item $\numberofvertices{\stdprimitivepolyhedronletter}=2$.
\item $\latticelength{\indexededge{\stdprimitivepolyhedronletter}{1}}=1$.
\item $\indexedvertex{\stdprimitivepolyhedronletter}{1}=0\in N$.
\end{enumerate}
\end{definition}

In other words, a primitive polyhedron is the Minkowski sum of $\6$ and a specially chosen primitive lattice segment. One of the endpoints of 
the segment should be the origin, but this segment cannot be chosen totally arbitrarily, otherwise we could also obtain a polyhedron 
$\stdprimitivepolyhedronletter$ with $\indexedvertex{\stdprimitivepolyhedronletter}{2}=0\in N$, not 
$\indexedvertex{\stdprimitivepolyhedronletter}{1}=0\in N$, or we could also get $\6$ itself, if the segment is inside $\6$.
Fig.~\ref{figprimpolyexample} shows an example of a primitive polyhedron.

\begin{figure}
\begin{center}
\includegraphics{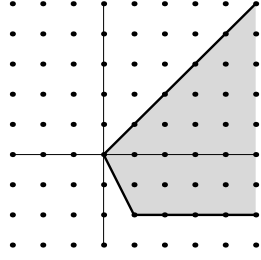}
\end{center}
\caption{An example of a primitive polyhedron.}
\label{figprimpolyexample}
\end{figure}

\begin{remark}\label{primitivenonpos}
If $\stdprimitivepolyhedronletter\subset N_\QQ$ is a primitive polyhedron with tail cone $\6$, 
then its individual evaluation function takes only nonpositive values.
\end{remark}

Clearly, if $\stdpolyhedronletter\subset N_\QQ$ is a lattice polyhedron with tail cone $\6$ and with 
$\indexedvertex{\stdpolyhedronletter}{1}=0$, then $\stdpolyhedronletter$ can be decomposed into a Minkowski 
sum of several primitive polyhedra (each of them can be taken several times). 
Decompose each polyhedron $\stdpolyhedronletter_{p_i}$ ($1\le i<\numberofdivisorpoints$)
into a sum of primitive polyhedra. Denote by 
$\stdprimitivepolyhedronletter_1, \ldots, \stdprimitivepolyhedronletter_\numberofprimitivepolyhedra$ 
all non-isomorphic primitive polyhedra we have.
We have $\stdpolyhedronletter_{p_j}=\sum_i \primitivepolyhedronlocalmultiplicity{i}{j}\stdprimitivepolyhedronletter_i$ for $1\le j<\numberofdivisorpoints$
and for some numbers $\primitivepolyhedronlocalmultiplicity{i}{j}\in\ZZ_{\ge 0}$. Denote 
$\primitivepolyhedronmultiplicity{i}=\sum_j\primitivepolyhedronlocalmultiplicity{i}{j}$. 
In other words, $\primitivepolyhedronmultiplicity{i}$ is the total number of times when a
polyhedron $\stdprimitivepolyhedronletter_i$ occurs in the decomposition of 
some of the polyhedra $\stdpolyhedronletter_{p_j}$ (for some $j$, $1\le j<\numberofdivisorpoints$) into a Minkowski
sum of primitive polyhedra.

\begin{remark}
For each $i$ ($1\le i\le \numberofprimitivepolyhedra$), the individual evaluation function of 
$\stdprimitivepolyhedronletter_i$ is linear on each of the cones in the total normal fan of $\mathcal D$.
\end{remark}

\begin{remark}\label{bardeltadecomp}
$\stdpolyhedronletter_{p_\numberofdivisorpoints}+\sum_{i=1}^\numberofprimitivepolyhedra\primitivepolyhedronmultiplicity{i}\stdprimitivepolyhedronletter_i
=\sum_{i=1}^{\numberofdivisorpoints}\stdpolyhedronletter_{p_i}=\totalpolyhedronletter$.
\end{remark}

First, let us construct an affine variety $\bigtotalspace$, which will be the total space
of the deformation. It will also be a variety with an action of a torus of 
dimension 2,
and we use the general construction of such varieties 
outlined in the Introduction. 
Consider a vector space $\coefficientspace$ with coordinates $\stdpolynomialcoefficient_{1,0},\ldots,\stdpolynomialcoefficient_{1,\primitivepolyhedronmultiplicity{1}-1},
\stdpolynomialcoefficient_{2,0},\ldots,\stdpolynomialcoefficient_{2,\primitivepolyhedronmultiplicity{2}-1},\ldots,
\stdpolynomialcoefficient_{\numberofprimitivepolyhedra,0},\ldots,
\stdpolynomialcoefficient_{\numberofprimitivepolyhedra,\primitivepolyhedronmultiplicity{\numberofprimitivepolyhedra}-1}$
Here we take $\primitivepolyhedronmultiplicity{i}$ coordinates for each primitive polyhedron 
$\stdprimitivepolyhedronletter_i$ we have.
Consider also a projective line $\PP^1$ with coordinate $t_0$.
Set $Y=\coefficientspace\times \PP^1$. For divisors
$Z_i$ ($1\le i\le \numberofprimitivepolyhedra$) we take the vanishing loci 
of the polynomials $t_0^{\primitivepolyhedronmultiplicity{i}}+\sum_{k=0}^{\primitivepolyhedronmultiplicity{i}-1} \stdpolynomialcoefficient_{i,k}t_0^k$ 
(these are polynomials in $\primitivepolyhedronmultiplicity{i}+1$ variables 
$\stdpolynomialcoefficient_{i,0}, \ldots, \stdpolynomialcoefficient_{i,\primitivepolyhedronmultiplicity{i}-1}, t_0$, 
not just in one variable $t_0$). 
Consider one more divisor $Z_0=\{t_0=\infty\}$.
Finally, for a polyhedral divisor we take 
$\mathfrak D=Z_0\otimes \Delta_{\numberofdivisorpoints}+\sum_{i=1}^{\numberofprimitivepolyhedra} Z_i\otimes \stdprimitivepolyhedronletter_i$.

It is not very easy to check directly that this polyhedral divisor is proper, but it is clear that it defines a (possibly non-finitely generated) 
algebra $A$. We will find an easy description of this algebra and then see directly that it is finitely generated.

The easiest way to describe the algebra $A$
is to embed 
it
into an algebra of polynomials.
First, choose a $\ZZ$-basis $\{\chi_1,\chi_2\}$ of $M$ so that all points of $\sck\cap M$ are 
linear combinations of $\chi_1$ and $\chi_2$ \textit{with positive coefficients}. In other words, 
the cone generated by $\chi_1$ and $\chi_2$ contains $\sck$.
Denote the dual basis of $N$ by $\chi_1^*$ and $\chi_2^*$.
Then we will embed $A$ into 
$\CC[\stdpolynomialcoefficient_{1,0},\ldots,\stdpolynomialcoefficient_{1,\primitivepolyhedronmultiplicity{1}-1},
\ldots,
\stdpolynomialcoefficient_{\numberofprimitivepolyhedra,0},\ldots,
\stdpolynomialcoefficient_{\numberofprimitivepolyhedra,\primitivepolyhedronmultiplicity{\numberofprimitivepolyhedra}-1},
t_0,t_1,t_2]$, where the variables $t_1$ and $t_2$ determine a grading (i.~e. we introduce an $M$-grading 
on this algebra, and $\deg(t_1)=\chi_1$ and $\deg(t_2)=\chi_2$, 
while the degrees of all other variables equal zero).
For each $\chi\in\sck\cap M$ denote by $\defparpolynomial_\chi$ the following polynomial:
$$
\defparpolynomial_\chi=\prod_{i=1}^{\numberofprimitivepolyhedra}
(t_0^{\primitivepolyhedronmultiplicity{i}}
+\sum_{k=0}^{\primitivepolyhedronmultiplicity{i}-1} \stdpolynomialcoefficient_{i,k}t_0^k)^{-\eval_{\stdprimitivepolyhedronletter_i}(\chi)}.
$$

\begin{lemma}\label{totalalgebragradecomp}
Let $\chi\in\sck\cap M$ be a degree. 
Then 
$\eval_{\stdpolyhedronletter_{p_\numberofdivisorpoints}}(\chi)
+\sum_{i=1}^\numberofprimitivepolyhedra\primitivepolyhedronmultiplicity{i}\eval_{\stdprimitivepolyhedronletter_i}(\chi)\ge 0$, and
the $\chi$th graded component of 
$A$
is a free 
$\CC[\coefficientspace]$-module
generated by 
$$
\defparpolynomial_\chi t_1^{\chi_1^*(\chi)}t_2^{\chi_2^*(\chi)}, \defparpolynomial_\chi t_1^{\chi_1^*(\chi)}t_2^{\chi_2^*(\chi)} t_0, \ldots, 
\defparpolynomial_\chi t_1^{\chi_1^*(\chi)}t_2^{\chi_2^*(\chi)} t_0^{\eval_{\totalpolyhedronletter}(\chi)}.
$$
\end{lemma}
\begin{proof}
$\mathfrak D(\chi)$ is a linear combination of the divisors $Z_i$ ($1\le i\le \numberofprimitivepolyhedra$) 
with nonpositive coefficients (Remark \ref{primitivenonpos}) and of 
the divisor $Z_0$ with a nonnegative coefficient (Lemma \ref{inftydivisornonneg}).
Therefore, $\Gamma(\OO(\mathfrak D(\chi)))$ is a subspace in the polynomial ring 
in the variables 
$$
\stdpolynomialcoefficient_{1,0},\ldots,\stdpolynomialcoefficient_{1,\primitivepolyhedronmultiplicity{1}-1},
\ldots,
\stdpolynomialcoefficient_{\numberofprimitivepolyhedra,0},\ldots,
\stdpolynomialcoefficient_{\numberofprimitivepolyhedra,\primitivepolyhedronmultiplicity{\numberofprimitivepolyhedra}-1},
t_0.
$$
Namely, these polynomials are of degree at most $\eval_{\stdpolyhedronletter_{p_\numberofdivisorpoints}}(\chi)$ with respect to $t_0$, 
and they are divisible by each of the polynomials 
$$
(t_0^{\primitivepolyhedronmultiplicity{i}}
+\sum_{k=0}^{\primitivepolyhedronmultiplicity{i}-1} \stdpolynomialcoefficient_{i,k}t_0^k)^{-\eval_{\stdprimitivepolyhedronletter_i}(\chi)}.
$$
To get the corresponding graded component of $A$, we have to multiply them by $t_1^{\chi_1^*(\chi)}t_2^{\chi_2^*(\chi)}$.
Therefore, the $\chi$th graded component is generated by 
$$
\defparpolynomial_\chi t_1^{\chi_1^*(\chi)}t_2^{\chi_2^*(\chi)}, \defparpolynomial_\chi t_1^{\chi_1^*(\chi)}t_2^{\chi_2^*(\chi)} t_0, \ldots, 
\defparpolynomial_\chi t_1^{\chi_1^*(\chi)}t_2^{\chi_2^*(\chi)} t_0^{\eval_{\stdpolyhedronletter_{p_\numberofdivisorpoints}}(\chi)
+\sum_{i=1}^\numberofprimitivepolyhedra\primitivepolyhedronmultiplicity{i}\eval_{\stdprimitivepolyhedronletter_i}(\chi)}
$$
as a
$\CC[\coefficientspace]$-module.
The claim follows from Remark \ref{bardeltadecomp}.
\end{proof}

Recall that in Chapter \ref{sectcohomformula} we chose a set of degrees $\{\lambda_1,\ldots,\lambda_{\numberofdivisorpoints}\}$, which 
contained Hilbert bases of all cones in the total normal fan of $\mathcal D$.

\begin{lemma}\label{totalalgebragenerators}
The algebra $A$ is finitely generated. More precisely, the generators from Lemma \ref{totalalgebragradecomp}
for degrees $\lambda_i$ generate $A$.
\end{lemma}
\begin{proof}
Fix a degree $\chi\in\sck\cap M$, and let $\tau$ be a 
cone from the total normal fan of $\mathcal D$ containing $\chi$.
Then all individual evaluation functions of polyhedra $\stdprimitivepolyhedronletter_i$
are linear on $\tau$, and the individual evaluation function of 
$\stdpolyhedronletter_{p_\numberofdivisorpoints}$
is also linear on $\tau$ (and even on $\sck$) since 
$p_\numberofdivisorpoints=\infty$ is a removable special point.
If $\chi',\chi''\in \tau\cap M$ and $\chi'+\chi''=\chi$, then $\defparpolynomial{\chi'}\defparpolynomial_{\chi''}=\defparpolynomial_{\chi}$, and 
each element of the basis of the $\chi$th graded component of $A$ from Lemma \ref{totalalgebragradecomp}
is a product of an element of the basis of the $\chi'$th graded component and 
of an element of the basis of the $\chi''$th graded component.
Therefore, since $\{\lambda_i\}$ contains the Hilbert basis of $\tau$, all components of $A$ of degrees $\lambda_i$ 
generate the $\chi$th graded component.
\end{proof}

So, $\bigtotalspace=\Spec A$ is an algebraic variety, and $T$ acts on it. We have 
$\primitivepolyhedronmultiplicity{1}+\ldots+\primitivepolyhedronmultiplicity{\numberofprimitivepolyhedra}$
global $T$-invariant functions $a_{i,j}$ on $X$, so we have a $T$-invariant 
map $\bigtotalspace\to \coefficientspace$, which we denote by $\bigflatmorphism$. It follows 
directly from Lemma \ref{totalalgebragradecomp} that this morphism is flat.
We can consider both 
$\coefficientspace\times \Spec \CC[t_0,t_1,t_2]$
and $\bigtotalspace$ as varieties over the base 
$\coefficientspace$.

\begin{lemma}\label{bigbiratisstabdom}
The morphism 
$\coefficientspace\times \Spec \CC[t_0,t_1,t_2]\to \bigtotalspace$
of 
$\coefficientspace$-varieties
is stably dominant.
\end{lemma}

\begin{proof}
We will construct a graded
$\CC[\coefficientspace]$-module
$K\subset \CC[\coefficientspace]\otimes \CC[t_0,t_1,t_2]$
such that 
$K\cap \CC[\bigtotalspace]=0$
and 
$K\oplus\CC[\bigtotalspace]=\CC[\coefficientspace]\otimes\CC[t_0,t_1,t_2]$.
We are going to construct each $M$-graded component of $K$ separately. Fix a degree 
$\chi\in\sck\cap M$.
Note that if we consider $\defparpolynomial_\chi$ as a polynomial in $t_0$ only, 
its leading coefficient will be equal 1, and its degree will be 
$-\sum \primitivepolyhedronmultiplicity{i}\eval_{\stdprimitivepolyhedronletter_i}(\chi)=
\eval_{\stdpolyhedronletter_{p_\numberofdivisorpoints}}(\chi)-\eval_{\totalpolyhedronletter}(\chi)$.
Now it follows from Lemma \ref{totalalgebragradecomp} that 
for the $\chi$th graded component of $K$ we can take the module generated by 
the following generators: $t_0^kt_1^{\chi_1^*(\chi)}t_2^{\chi_2^*(\chi)}$, where 
$0\le k <\eval_{\stdpolyhedronletter_{p_\numberofdivisorpoints}}(\chi)-\eval_{\totalpolyhedronletter}(\chi)$
or $k>\eval_{\stdpolyhedronletter_{p_\numberofdivisorpoints}}(\chi)$.
The claim follows from Remark \ref{directsummandinj}.
\end{proof}

%

Now we are ready to compute the fibers of $\bigflatmorphism$ using Lemma \ref{computefiber}.
For an arbitrary point 
$$
(\stdpolynomialcoefficient_{1,0}^{(0)},\ldots,\stdpolynomialcoefficient_{1,\primitivepolyhedronmultiplicity{1}-1}^{(0)},
\ldots,
\stdpolynomialcoefficient_{\numberofprimitivepolyhedra,0}^{(0)},\ldots,
\stdpolynomialcoefficient_{\numberofprimitivepolyhedra,\primitivepolyhedronmultiplicity{\numberofprimitivepolyhedra}-1}^{(0)})\in \coefficientspace
$$
we define a divisor
$$
\mathfrak D_{\stdpolynomialcoefficient_{1,0}^{(0)},\ldots,\stdpolynomialcoefficient_{1,\primitivepolyhedronmultiplicity{1}-1}^{(0)},
\ldots,
\stdpolynomialcoefficient_{\numberofprimitivepolyhedra,0}^{(0)},\ldots,
\stdpolynomialcoefficient_{\numberofprimitivepolyhedra,\primitivepolyhedronmultiplicity{\numberofprimitivepolyhedra}-1}^{(0)}}
$$
on a projective line as follows.
Consider a projective line $\PP^1$ with a coordinate function $t$.
For each $i$ ($1\le i\le \numberofprimitivepolyhedra$) denote by 
$\stdpolynomialroot_{i,1},\ldots,\stdpolynomialroot_{i,\primitivepolyhedronmultiplicity{i}}$
the zeros of the function 
$$
t^{\primitivepolyhedronmultiplicity{i}}
+\sum_{k=0}^{\primitivepolyhedronmultiplicity{i}-1} \stdpolynomialcoefficient_{i,k}^{(0)}t^k
$$
on $\PP^1$ with multiplicities (i.~e. if we have a zero of order more than one, we write 
the same point several times, for example, $\stdpolynomialroot_{i,1}$ and $\stdpolynomialroot_{i,2}$ 
can be the same point).
Then, for each $i$ ($1\le i\le \numberofprimitivepolyhedra$) and for each $j$ 
($1\le j\le \primitivepolyhedronmultiplicity{i}$)
we put $\stdprimitivepolyhedronletter_i$ at the point $b_{i,j}$.
If we put several polyhedra at the same point of $\PP^1$
(for example, if we had a zero of order more than one, or if the functions different values of $i$ 
vanish at the same point), we take the Minkowski sum of them instead.
Finally, we put $\stdpolyhedronletter_{p_\numberofdivisorpoints}$ at the point of $\PP^1$ 
with coordinate $\infty$.

\begin{lemma}\label{fiberoflargedeformation}
Let 
$$
(\stdpolynomialcoefficient_{1,0}^{(0)},\ldots,\stdpolynomialcoefficient_{1,\primitivepolyhedronmultiplicity{1}-1}^{(0)},
\ldots,
\stdpolynomialcoefficient_{\numberofprimitivepolyhedra,0}^{(0)},\ldots,
\stdpolynomialcoefficient_{\numberofprimitivepolyhedra,\primitivepolyhedronmultiplicity{\numberofprimitivepolyhedra}-1}^{(0)})\in \coefficientspace
$$
be an arbitrary point.
The fiber of $\bigflatmorphism$
over this point 
is the T-variety 
defined by 
$$
\mathfrak D_{\stdpolynomialcoefficient_{1,0}^{(0)},\ldots,\stdpolynomialcoefficient_{1,\primitivepolyhedronmultiplicity{1}-1}^{(0)},
\ldots,
\stdpolynomialcoefficient_{\numberofprimitivepolyhedra,0}^{(0)},\ldots,
\stdpolynomialcoefficient_{\numberofprimitivepolyhedra,\primitivepolyhedronmultiplicity{\numberofprimitivepolyhedra}-1}^{(0)}},
$$
a polyhedral divisor on $\PP^1$.

More precisely, the construction of a T-variety out of a polyhedral divisor identifies the global functions on the T-variety 
with sections of line bundles on $\PP^1$. In this case, this identification works "in the natural way", namely, as follows.
Fix a degree $\chi\in\sck\cap M$.
The restriction of a function 
$\defparpolynomial_\chi t_1^{\chi_1^*(\chi)}t_2^{\chi_2^*(\chi)} t_0^k$ 
to the fiber is identified with the rational function
$$
\defparpolynomial_\chi\Bigg|_{\substack{t_0=t\\\stdpolynomialcoefficient_{i,j}=\stdpolynomialcoefficient_{i,j}^{(0)}}} t^k\in 
\OO_{\PP^1}(\mathfrak D_{\stdpolynomialcoefficient_{1,0}^{(0)},\ldots,\stdpolynomialcoefficient_{1,\primitivepolyhedronmultiplicity{1}-1}^{(0)},
\ldots,
\stdpolynomialcoefficient_{\numberofprimitivepolyhedra,0}^{(0)},\ldots,
\stdpolynomialcoefficient_{\numberofprimitivepolyhedra,\primitivepolyhedronmultiplicity{\numberofprimitivepolyhedra}-1}^{(0)}}(\chi)).
$$
\end{lemma}

\begin{proof}
First, let us check that 
$$
\mathfrak D_{\stdpolynomialcoefficient_{1,0}^{(0)},\ldots,\stdpolynomialcoefficient_{1,\primitivepolyhedronmultiplicity{1}-1}^{(0)},
\ldots,
\stdpolynomialcoefficient_{\numberofprimitivepolyhedra,0}^{(0)},\ldots,
\stdpolynomialcoefficient_{\numberofprimitivepolyhedra,\primitivepolyhedronmultiplicity{\numberofprimitivepolyhedra}-1}^{(0)}}
$$
is a proper polyhedral divisor. Since this is a polyhedral divisor on $\PP^1$, 
it is sufficient to check that for each $\chi\in\sck\cap M$ 
$$
\deg \mathfrak D_{\stdpolynomialcoefficient_{1,0}^{(0)},\ldots,\stdpolynomialcoefficient_{1,\primitivepolyhedronmultiplicity{1}-1}^{(0)},
\ldots,
\stdpolynomialcoefficient_{\numberofprimitivepolyhedra,0}^{(0)},\ldots,
\stdpolynomialcoefficient_{\numberofprimitivepolyhedra,\primitivepolyhedronmultiplicity{\numberofprimitivepolyhedra}-1}^{(0)}}(\chi)=\deg\mathcal D(\chi),
$$
where $\mathcal D$ is the original divisor on $\PP^1$ describing the variety we are going to deform.
We have 
\begin{multline*}
\deg \mathfrak D_{\stdpolynomialcoefficient_{1,0}^{(0)},\ldots,\stdpolynomialcoefficient_{1,\primitivepolyhedronmultiplicity{1}-1}^{(0)},
\ldots,
\stdpolynomialcoefficient_{\numberofprimitivepolyhedra,0}^{(0)},\ldots,
\stdpolynomialcoefficient_{\numberofprimitivepolyhedra,\primitivepolyhedronmultiplicity{\numberofprimitivepolyhedra}-1}^{(0)}}(\chi)=
\deg \stdpolyhedronletter_{p_\numberofdivisorpoints}(\chi)+
\sum_{i=1}^{\numberofprimitivepolyhedra}\primitivepolyhedronmultiplicity{i}\deg\stdprimitivepolyhedronletter_i(\chi)=\\
\deg \stdpolyhedronletter_{p_\numberofdivisorpoints}(\chi)+
\sum_{i=1}^{\numberofprimitivepolyhedra}(\sum_{j=1}^{\numberofdivisorpoints-1}\primitivepolyhedronlocalmultiplicity{i}{j}\deg\stdprimitivepolyhedronletter_i(\chi))=
\deg \stdpolyhedronletter_{p_\numberofdivisorpoints}(\chi)+
\sum_{j=1}^{\numberofdivisorpoints-1}\deg\stdpolyhedronletter_{p_j}(\chi)=\deg\mathcal D(\chi),
\end{multline*}
and 
$$
\mathfrak D_{\stdpolynomialcoefficient_{1,0}^{(0)},\ldots,\stdpolynomialcoefficient_{1,\primitivepolyhedronmultiplicity{1}-1}^{(0)},
\ldots,
\stdpolynomialcoefficient_{\numberofprimitivepolyhedra,0}^{(0)},\ldots,
\stdpolynomialcoefficient_{\numberofprimitivepolyhedra,\primitivepolyhedronmultiplicity{\numberofprimitivepolyhedra}-1}^{(0)}}
$$
is a proper polyhedral divisor.

Fix a degree $\chi\in\sck\cap M$. Let us compute 
$$
\Gamma(\OO_{\PP^1}(\mathfrak D_{\stdpolynomialcoefficient_{1,0}^{(0)},\ldots,\stdpolynomialcoefficient_{1,\primitivepolyhedronmultiplicity{1}-1}^{(0)},
\ldots,
\stdpolynomialcoefficient_{\numberofprimitivepolyhedra,0}^{(0)},\ldots,
\stdpolynomialcoefficient_{\numberofprimitivepolyhedra,\primitivepolyhedronmultiplicity{\numberofprimitivepolyhedra}-1}^{(0)}}(\chi))).
$$
Recall that Minkowski addition of two polyhedra leads to summation of their individual evaluation 
functions. 
Denote all distinct points among $\stdpolynomialroot_{i,j}$ by $\stdpolynomialroot'_1,\ldots,\stdpolynomialroot'_l$.
For each $j$ ($1\le j\le l$) and for each $i$ ($1\le i\le \numberofprimitivepolyhedra$) 
denote by $\stdpolynomialrootmultiplicity_{j,i}$ the order of zero of the function 
$$
t^{\primitivepolyhedronmultiplicity{i}}
+\sum_{k=0}^{\primitivepolyhedronmultiplicity{i}-1} \stdpolynomialcoefficient_{i,k}^{(0)}t^{k}
$$
at the point $\stdpolynomialroot'_j$.
In other words, $\stdpolynomialrootmultiplicity_{j,i}$ is the amount of 
indices $j'$ ($1\le j'\le \primitivepolyhedronmultiplicity{i}$) such that $\stdpolynomialroot_{i,j}=\stdpolynomialroot'_{j'}$.
Then the polyhedron in 
$$
\mathfrak D_{\stdpolynomialcoefficient_{1,0}^{(0)},\ldots,\stdpolynomialcoefficient_{1,\primitivepolyhedronmultiplicity{1}-1}^{(0)},
\ldots,
\stdpolynomialcoefficient_{\numberofprimitivepolyhedra,0}^{(0)},\ldots,
\stdpolynomialcoefficient_{\numberofprimitivepolyhedra,\primitivepolyhedronmultiplicity{\numberofprimitivepolyhedra}-1}^{(0)}}
$$
above a point $\stdpolynomialroot'_j$ is 
$\sum_{i=1}^{\numberofprimitivepolyhedra}\stdpolynomialrootmultiplicity_{j,i}\stdprimitivepolyhedronletter_i$.

Then the global sections of 
$$
\OO_{\PP^1}(\mathfrak D_{\stdpolynomialcoefficient_{1,0}^{(0)},\ldots,\stdpolynomialcoefficient_{1,\primitivepolyhedronmultiplicity{1}-1}^{(0)},
\ldots,
\stdpolynomialcoefficient_{\numberofprimitivepolyhedra,0}^{(0)},\ldots,
\stdpolynomialcoefficient_{\numberofprimitivepolyhedra,\primitivepolyhedronmultiplicity{\numberofprimitivepolyhedra}-1}^{(0)}}(\chi))
$$
are the polynomials in $t$ 
of degree at most $\eval_{\stdpolyhedronletter_{p_\numberofdivisorpoints}}(\chi)$
divisible by 
\begin{multline*}
\prod_{k=1}^l(t-t(\stdpolynomialroot'_k))^{\sum_i\stdpolynomialrootmultiplicity_{k,i}\eval_{\stdprimitivepolyhedronletter_i}(\chi)}=
\prod_{i=1}^{\numberofprimitivepolyhedra}\prod_{k=1}^l(t-t(\stdpolynomialroot'_k))^{\stdpolynomialrootmultiplicity_{k,i}\eval_{\stdprimitivepolyhedronletter_i}(\chi)}=\\
\prod_{i=1}^{\numberofprimitivepolyhedra}\bigg(\prod_{j=1}^{\primitivepolyhedronmultiplicity{i}}
(t-t(\stdpolynomialroot_{i,j}))^{\eval_{\stdprimitivepolyhedronletter_i}(\chi)}\bigg)=
\prod_{i=1}^{\numberofprimitivepolyhedra}\bigg(t^{\primitivepolyhedronmultiplicity{i}}
+\sum_{k=0}^{\primitivepolyhedronmultiplicity{i}-1} \stdpolynomialcoefficient_{i,k}^{(0)}t^k\bigg)^{\eval_{\stdprimitivepolyhedronletter_i}(\chi)}=\\
\defparpolynomial_\chi|_{t_0=t\text{ and }\stdpolynomialcoefficient_{i,j}=\stdpolynomialcoefficient_{i,j}^{(0)}\text{ for }
1\le i\le \numberofprimitivepolyhedra, 1\le j\le \primitivepolyhedronmultiplicity{i}}.
\end{multline*}

On the other hand, if $\chi=m_1\chi_1+m_2\chi_2$, then, by Lemmas \ref{computefiber} and \ref{bigbiratisstabdom},
the functions of degree $\chi$ on 
$$
\bigflatmorphism^{-1}(\stdpolynomialcoefficient_{1,0}^{(0)},\ldots,\stdpolynomialcoefficient_{1,\primitivepolyhedronmultiplicity{1}-1}^{(0)},
\ldots,
\stdpolynomialcoefficient_{\numberofprimitivepolyhedra,0}^{(0)},\ldots,
\stdpolynomialcoefficient_{\numberofprimitivepolyhedra,\primitivepolyhedronmultiplicity{\numberofprimitivepolyhedra}-1}^{(0)})
$$
are $\CC$-generated by
the images of polynomials
$$
\defparpolynomial_\chi t_1^{m_1}t_2^{m_2}, \defparpolynomial_\chi t_1^{m_1}t_2^{m_2} t_0, \ldots, 
\defparpolynomial_\chi t_1^{m_1}t_2^{m_2} t_0^{\eval_{\totalpolyhedronletter}(\chi)}
$$
under the quotient map 
$$
(\CC[\coefficientspace]\otimes \CC[t_0,t_1,t_2])\to 
(\CC[\coefficientspace]\otimes\CC[t_0,t_1,t_2])/(I(\CC[\coefficientspace]\otimes\CC[t_0,t_1,t_2])),
$$
where $I$ is the ideal in 
$\CC[\coefficientspace]$
generated by equations $\stdpolynomialcoefficient_{i,j}=\stdpolynomialcoefficient_{i,j}^{(0)}$.
In other words, 
the polynomials 
$$
\defparpolynomial_\chi t_1^{m_1}t_2^{m_2}, \defparpolynomial_\chi t_1^{m_1}t_2^{m_2} t_0, \ldots, 
\defparpolynomial_\chi t_1^{m_1}t_2^{m_2} t_0^{\eval_{\totalpolyhedronletter}(\chi)}
$$
after the substitutions $\stdpolynomialcoefficient_{i,j}=\stdpolynomialcoefficient_{i,j}^{(0)}$
$\CC$-generate the space of the functions of degree $\chi$ on the fiber.
\end{proof}

\begin{corollary}
$\dim \bigtotalspace=\dim \coefficientspace+3$.\qed
\end{corollary}

For each $i$ ($1\le i\le \numberofprimitivepolyhedra$), denote by
$\stdpolynomialcoefficient_{i,0}^{(1)}, \ldots, \stdpolynomialcoefficient_{i,\primitivepolyhedronmultiplicity{i}-1}^{(1)}$
the coefficients of the polynomial with leading coefficient 1 and with 
roots at the points $p_j$, where the root at $p_j$ has multiplicity 
$\primitivepolyhedronlocalmultiplicity{i}{j}$. In other words, 
$$
t^{\primitivepolyhedronmultiplicity{i}}
+\sum_{k=0}^{\primitivepolyhedronmultiplicity{i}-1} \stdpolynomialcoefficient_{i,k}^{(1)}t^k=
\prod_{j=1}^{\numberofdivisorpoints}(t-t(p_j))^{\primitivepolyhedronlocalmultiplicity{i}{j}}
$$
as polynomials in $t$. Fix this notation until the end of Chapter \ref{sectksmtoric}.

\begin{corollary}\label{originfibercorrect}
The fiber 
$$
\bigflatmorphism^{-1}(\stdpolynomialcoefficient_{1,0}^{(1)},\ldots,\stdpolynomialcoefficient_{1,\primitivepolyhedronmultiplicity{1}-1}^{(1)},
\ldots,
\stdpolynomialcoefficient_{\numberofprimitivepolyhedra,0}^{(1)},\ldots,
\stdpolynomialcoefficient_{\numberofprimitivepolyhedra,\primitivepolyhedronmultiplicity{\numberofprimitivepolyhedra}-1}^{(1)})
$$
is isomorphic to the original $T$-variety $X$, which we are deforming, and which was constructed from the 
polyhedral divisor $\mathcal D$.\qed
\end{corollary}

Therefore, we have constructed a deformation of $X$ over $\coefficientspace$. Now we are going to compute the Kodaira-Spencer map of this 
deformation.

\section{Kodaira-Spencer map for a deformation given by perturbation of generators}\label{ksmgens}


We are going to consider the following general situation. 
Suppose that we have a deformation of a normal variety $X$ over an affine line. Denote the total space of the deformation by $S$ and 
the function from $S$ to $\CC^1$ by $\stdflatmorphism$.
Note that by definition this means that the \textit{scheme-theoretic} fiber $\stdflatmorphism^{-1}(0)$ is $X$.
Suppose also that $S$ is embedded into a vector space, where $\stdflatmorphism$ is one of the coordinates, 
and the other coordinates are $\abstractindependentgenerator1, \ldots,\abstractindependentgenerator{\numberofabstractgenerators}$.
These data canonically define a set of generators of the algebra $\CC[S]$, denote them by 
$\stdflatmorphism, \abstractdependentgenerator1,\ldots,\abstractdependentgenerator{\numberofabstractgenerators}$.
Suppose also that we have another vector space $\CC^{k+1}$
with coordinates $\stdbirationalgenerator_0,\stdbirationalgenerator_1,\ldots, \stdbirationalgenerator_k$, and 
a dominant morphism $\stdbirationalmap\colon \CC^{k+1}\to S$.
Note that it already follows that $S$ is irreducible.
Suppose that $\stdbirationalmap$ 
satisfies the following two conditions:
\begin{enumerate}
\item $\stdflatmorphism\circ \stdbirationalmap=\stdbirationalgenerator_0$, in other words, the coordinate 
$\stdflatmorphism$ of a point $b(\stdbirationalgenerator_0,\ldots, \stdbirationalgenerator_k)$
equals $\stdbirationalgenerator_0$. This condition implies that $\stdbirationalmap^{-1}(X)=
(\stdflatmorphism\circ\stdbirationalmap)^{-1}(0)=\{\stdbirationalgenerator_0=0\}$, moreover, the scheme-theoretic fiber 
also equals $\{\stdbirationalgenerator_0=0\}$.
\item The restriction of $\stdbirationalmap$ to the hyperplane $\stdbirationalgenerator_0=0$ 
is a dominant morphism to $X$. This condition implies that $X$ is irreducible.
\item The restriction of $\stdbirationalmap$ to the hyperplane $\stdbirationalgenerator_0=0$ maps it 
\textit{birationally} to $X$.
\end{enumerate}

In algebraic terms, the existence of such a morphism $\stdbirationalmap$ and these conditions mean 
the following.
Suppose that $\stdbirationalmap$ is given by polynomials: 
$\abstractdependentgenerator i=P_i(\stdbirationalgenerator_0,\ldots,\stdbirationalgenerator_k)$.
Then $\CC[S]$ can be understood as the \emph{subalgebra} of $\CC[\stdbirationalgenerator_0,\ldots,\stdbirationalgenerator_k]$ 
generated by $\stdbirationalgenerator_0$ and $P_1,\ldots P_k$.
For each $a\in\CC$ such that $\stdbirationalmap|_{\stdbirationalgenerator_0=a}$ is dominant, 
the algebra of functions on $\stdflatmorphism^{-1}(a)$ can be understood as the 
subalgebra of $\CC[\stdbirationalgenerator_1,\ldots,\stdbirationalgenerator_k]$
generated by $P_i|_{\stdbirationalgenerator_0=a}$.
In particular, $\CC[\stdflatmorphism^{-1}(0)]=\CC[X]$ becomes the subalgebra
of $\CC[\stdbirationalgenerator_1,\ldots,\stdbirationalgenerator_k]$
generated by $P_i|_{\stdbirationalgenerator_0=0}$, and 
then, informally speaking, 
when $X$ is deformed, 
the generators of the subalgebra are also deformed, and the algebra of functions on the deformed variety
is the subalgebra generated by these deformed generators. 

First, let us prove in this setting the following easy corollary of Hilbert Nullstellensatz.
\begin{lemma}\label{hilbdivisibility}
Let $f\colon S\to \CC$ be a regular function such that $f\circ\stdbirationalmap$ (which is a regular function on 
$\CC^{k+1}$) can be written as $\stdbirationalgenerator_0 h$, where $h\in \CC[\stdbirationalgenerator_0,\ldots,\stdbirationalgenerator_k]$.
Then there exists $f_1\in \CC[S]$ such that $f=\stdflatmorphism f_1$ and $h=f_1\circ \stdbirationalmap$.
\end{lemma}
\begin{proof}
For each point $x$ of $X$ of the form $x=\stdbirationalmap (0,\stdbirationalgenerator_1,\ldots,\stdbirationalgenerator_k)$
we have 
$$
f(x)=f(\stdbirationalmap (0,\stdbirationalgenerator_1,\ldots,\stdbirationalgenerator_k))=0.
$$
The set $\stdbirationalmap(\{\stdbirationalgenerator_0=0\})$ is open and dense in $X$, so $f|_X=0$.
By Hilbert Nullstellensatz, some power of $f$ is divisible by $\stdflatmorphism$, 
but since $X$ is the scheme-theoretic fiber of $\stdflatmorphism$ above zero, the
ideal $\stdflatmorphism\CC[S]$ is radical, and $f$ itself is divisible by $\stdflatmorphism$. 
Let $f_1\in \CC[S]$ be such that $f=\stdflatmorphism f_1$.
Then $\stdbirationalgenerator_0 h=f\circ \stdbirationalmap=(\stdflatmorphism\circ\stdbirationalmap)(f_1\circ\stdbirationalmap)=
\stdbirationalgenerator_0(f_1\circ\stdbirationalmap)$.
Therefore, $h=f_1\circ\stdbirationalmap$.
\end{proof}

Let $U\subseteq X$ be a smooth open sunset such that $\codim_X(X\setminus U)\ge 2$. And let 
$U'\subseteq U$ be a subset such that $(\stdbirationalmap|_{\stdbirationalgenerator_0=0})^{-1}$
is defined on $U'$. Consider the following section of 
$\Theta_{\CC^{1+\numberofabstractgenerators}=
\Spec \CC[\stdflatmorphism,\abstractindependentgenerator1,\ldots,\abstractindependentgenerator{\numberofabstractgenerators}]}|_{U'}$: 
at each point $x\in U'$ we have
$$
v(x)=d_{(\stdbirationalmap|_{\stdbirationalgenerator_0=0})^{-1}(x)}\stdbirationalmap\left(\frac{\partial}{\partial\stdbirationalgenerator_0} \right).
$$
The first coordinate of this vector (the coefficient in front of $\partial/\partial \stdflatmorphism$) 
is always one.

Consider the restriction of the deformation $\stdflatmorphism \colon S\to\CC^1$ 
to the double point at the origin corresponding to the vector $\partial/\partial \stdflatmorphism$.
Denote the total space of the deformation by $\widetilde S$ and the flat morphism by 
$\varepsilon\colon \widetilde S\to\Spec \CC[\varepsilon]/\varepsilon^2$
Denote the restrictions of functions $\abstractdependentgenerator i$ to 
$\widetilde S$ by $\widetilde{\abstractdependentgenerator i}$.

Let $I\subset \CC[\abstractindependentgenerator1,\ldots,\abstractindependentgenerator{\numberofabstractgenerators}]$
be the ideal of equations of $X$, i.~e. 
$\CC[X]=\CC[\abstractindependentgenerator1,\ldots,\abstractindependentgenerator{\numberofabstractgenerators}]/I$.
Since we have chosen lifts of generators of $\CC[X]$ to $\CC[\widetilde S]$, we have a uniquely determined map 
$I/I^2\to \CC[X]$ representing the deformation.

\begin{proposition}\label{mushift}
For each function $g\in I$, 
denote by $g^\circ\in \CC[\stdflatmorphism,\abstractindependentgenerator1,\ldots,\abstractindependentgenerator{\numberofabstractgenerators}]$
the image of $g$ under the natural embedding 
$\CC[\abstractindependentgenerator1,\ldots,\abstractindependentgenerator{\numberofabstractgenerators}]
\hookrightarrow\CC[\stdflatmorphism,\abstractindependentgenerator1,\ldots,\abstractindependentgenerator{\numberofabstractgenerators}]$
($g^\circ$ actually does not depend on $\stdflatmorphism$ and, informally speaking, equals $g$ as it is written).
Set $\mu(g)=dg^\circ(v)$ 
(we apply the differential of a function to a rational vector field and get a rational function). 

\begin{enumerate}
\item For each $g\in I$, $\mu(g)$ is a regular function on the whole $X$ (by definition we only 
know that it is defined on $U'$)
\item $\mu$ is a well-defined $\CC[X]$-linear morphism $I/I^2\to \CC[X]$.
\item $\mu$ represents the deformation 
$\varepsilon\colon \widetilde S\to\Spec \CC[\varepsilon]/\varepsilon^2$
in $T^1(X)$.
\end{enumerate}
\end{proposition}

\begin{proof}
The function $\mu(g)$ is a rational function on $X$, so it can be written as a ratio of two 
polynomials in $\abstractdependentgenerator1,\ldots,\abstractdependentgenerator{\numberofabstractgenerators}$, 
and the second of them has no zeros inside $U'$.
Fix these two polynomials and consider them now as polynomials in 
$\stdflatmorphism,\abstractdependentgenerator1,\ldots,\abstractdependentgenerator{\numberofabstractgenerators}$.
Then we will get two regular functions on $S$, denote them by $P$ and $Q$, respectively. 
Then $Q(x)\ne 0$ if $x\in U'\subseteq X\subset S$, and $P(x)/Q(x)=\mu(g)(x)$ if $x\in U'$.

Now consider a rational function $(P/Q)\circ\stdbirationalmap$ on $\CC^{k+1}$. 
Let $(0,\stdbirationalgenerator_1,\ldots,\stdbirationalgenerator_k)\in\CC^{k+1}$ be a point such that 
$\stdbirationalmap(0,\stdbirationalgenerator_1,\ldots,\stdbirationalgenerator_k)\in U'$.
Then 
\begin{multline*}
(P/Q)(\stdbirationalmap(0,\stdbirationalgenerator_1,\ldots,\stdbirationalgenerator_k))=
\mu(g)(\stdbirationalmap(0,\stdbirationalgenerator_1,\ldots,\stdbirationalgenerator_k))=\\
d_{\stdbirationalmap(0,\stdbirationalgenerator_1,\ldots,\stdbirationalgenerator_k)}g^\circ
(v(\stdbirationalmap(0,\stdbirationalgenerator_1,\ldots,\stdbirationalgenerator_k)))=\\
d_{\stdbirationalmap(0,\stdbirationalgenerator_1,\ldots,\stdbirationalgenerator_k)}g^\circ
((\partial/\partial \stdbirationalgenerator_0\stdbirationalmap)((\stdbirationalmap|_{\stdbirationalgenerator_0=0})^{-1}
(\stdbirationalmap(0,\stdbirationalgenerator_1,\ldots,\stdbirationalgenerator_k))))=\\
d_{\stdbirationalmap(0,\stdbirationalgenerator_1,\ldots,\stdbirationalgenerator_k)}g^\circ
(\partial/\partial \stdbirationalgenerator_0\stdbirationalmap(0,\stdbirationalgenerator_1,\ldots,\stdbirationalgenerator_k))=
(\partial/\partial \stdbirationalgenerator_0)(g^\circ\circ\stdbirationalmap).
\end{multline*}
Therefore, 
the functions $(P/Q)\circ \stdbirationalmap$ and 
$(\partial/\partial \stdbirationalgenerator_0)(g^\circ\circ\stdbirationalmap)$
coincide on $\stdbirationalmap^{-1}(U')$. Then the functions 
$P\circ \stdbirationalmap$ and $(Q\circ\stdbirationalmap)((\partial/\partial \stdbirationalgenerator_0)(g^\circ\circ\stdbirationalmap))$
(both of them are regular) also coincide on $\stdbirationalmap^{-1}(U')$, which is an 
open subset of the hyperplane $\{\stdbirationalgenerator_0=0\}$, 
and their difference 
$P\circ \stdbirationalmap-(Q\circ\stdbirationalmap)((\partial/\partial \stdbirationalgenerator_0)(g^\circ\circ\stdbirationalmap))$
is a polynomial divisible by $\stdbirationalgenerator_0$.

Consider the following regular function on $\CC^{k+1}$: 
$g^\circ\circ\stdbirationalmap-\stdbirationalgenerator_0 \partial/\partial\stdbirationalgenerator_0(g^\circ\circ\stdbirationalmap)$.
Clearly, it vanishes if $\stdbirationalgenerator_0=0$. Moreover, 
$$
\partial/\partial \stdbirationalgenerator_0 (g^\circ\circ\stdbirationalmap-\stdbirationalgenerator_0 \partial/\partial\stdbirationalgenerator_0(g^\circ\circ\stdbirationalmap))=
-\stdbirationalgenerator_0\partial^2/\partial\stdbirationalgenerator_0^2(g^\circ\circ\stdbirationalmap),
$$ 
so 
$$(\partial/\partial \stdbirationalgenerator_0 
(g^\circ\circ\stdbirationalmap-\stdbirationalgenerator_0 \partial/\partial\stdbirationalgenerator_0(g^\circ\circ\stdbirationalmap)))|_{\stdbirationalgenerator_0=0}=0,
$$
and 
$g^\circ\circ\stdbirationalmap-\stdbirationalgenerator_0 \partial/\partial\stdbirationalgenerator_0(g^\circ\circ\stdbirationalmap)$
is a polynomial divisible by $\stdbirationalgenerator_0^2$.
Hence, 
$(Q\circ\stdbirationalmap)(g^\circ\circ\stdbirationalmap-\stdbirationalgenerator_0 \partial/\partial\stdbirationalgenerator_0(g^\circ\circ\stdbirationalmap))$
is also divisible by $\stdbirationalgenerator_0^2$.
We also know that 
$(P\circ \stdbirationalmap)\stdbirationalgenerator_0
-(Q\circ\stdbirationalmap)\stdbirationalgenerator_0((\partial/\partial \stdbirationalgenerator_0)(g^\circ\circ\stdbirationalmap))$
is divisible by $\stdbirationalgenerator_0^2$, 
so 
$(Q\circ\stdbirationalmap)(g^\circ\circ\stdbirationalmap)-(P\circ \stdbirationalmap)\stdbirationalgenerator_0=
(Qg^\circ-P\stdflatmorphism)\circ\stdbirationalmap$
is divisible by $\stdbirationalgenerator_0^2$.
Then by Lemma \ref{hilbdivisibility} applied twice, 
$Qg^\circ-P\stdflatmorphism$ is divisible by $\stdflatmorphism^2$ in $\CC[S]$.

Recall that we have lifts $\widetilde{\abstractdependentgenerator i}\in \CC[\widetilde S]$ of the 
functions $\abstractdependentgenerator i$ on $X$.
So, the restriction of the deformation $\stdflatmorphism\colon S\to \CC^1$ to the double point 
at the origin can be represented by 
an element of 
$\Hom_{\CC[X]}(I/I^2,\CC[X])$. In particular, there exists
a polynomial $g_1\in \CC[\abstractindependentgenerator 1,\ldots, \abstractindependentgenerator{\numberofabstractgenerators}]$
such that 
$g(\widetilde{\abstractdependentgenerator1},\ldots,\widetilde{\abstractdependentgenerator{\numberofabstractgenerators}})=
\varepsilon g_1(\widetilde{\abstractdependentgenerator1},\ldots,\widetilde{\abstractdependentgenerator{\numberofabstractgenerators}})$
in $\CC[\widetilde S]$.
By the definition of $\CC[\widetilde S]$ this means that 
$g(\abstractdependentgenerator1,\ldots,\abstractdependentgenerator{\numberofabstractgenerators})-
\stdflatmorphism g_1(\abstractdependentgenerator1,\ldots,\abstractdependentgenerator{\numberofabstractgenerators})$
is divisible by $\stdflatmorphism$ in $\CC[S]$.
Denote the function $g_1(\abstractdependentgenerator1,\ldots,\abstractdependentgenerator{\numberofabstractgenerators})$
understood as a function on the whole $S$ by $g_1^\circ$.
Then $Qg^\circ-Q\stdflatmorphism g_1^\circ$ is also divisible by $\stdflatmorphism^2$.
Therefore, $Q\stdflatmorphism g_1^\circ-P\stdflatmorphism$ is divisible by $\stdflatmorphism^2$ in $\CC[S]$.
Since $S$ is an irreducible variety, $Qg_1^\circ-P$ is divisible by $\stdflatmorphism$. 
So, $(Qg_1^\circ-P)|_X=0$, and this by definition of the field of rational functions means that 
$g_1=\mu(g)$ in $\CC(X)$, and $\mu(g)$ is in fact a regular function on $X$.

Let us check that the map $\mu$ is 
$\CC[\abstractindependentgenerator1,\ldots,\abstractindependentgenerator{\numberofabstractgenerators}]$-linear.
The additivity of $\mu$ is clear. Choose a polynomial 
$h\in \CC[\abstractindependentgenerator1,\ldots,\abstractindependentgenerator{\numberofabstractgenerators}]$
and denote by $h^\circ$ the polynomial $h$ understood as a polynomial in 
$\stdflatmorphism,\abstractindependentgenerator1,\ldots,\abstractindependentgenerator{\numberofabstractgenerators}$.
Then $\mu(hg)=d(h^\circ g^\circ)(v)=h^\circ dg^\circ(v)+g^\circ dh^\circ(v)$, but 
$\mu(hg)$ is a function on $X$, and the second summand on $X$ equals zero, and 
the first summand on $X$ equals $h dg^\circ(v)=h\mu (v)$.

Now, since $\mu$ is a
$\CC[\abstractindependentgenerator1,\ldots,\abstractindependentgenerator{\numberofabstractgenerators}]$-linear
map from $I$ to $\CC[X]$, it vanishes on $I/I^2$ and induces a $\CC[X]$-linear map from $I/I^2$ to $\CC[X]$.
And we have already seen before that $\mu(g)$ coincides with the image of $g$ 
under the map $I/I^2\to \CC[X]$ corresponding to the first order deformation in $T^1(X)$.
\end{proof}

We keep the notation $\mu$ introduced in Proposition \ref{mushift} for further usage.
Recall that the first coordinate of any vector $v(x)$, where $x\in U'$, equals 1, 
and that $g^\circ$ does not actually depend on $\stdflatmorphism$.
Denote the projection of $v$ to 
$\Theta_{\CC^{\numberofabstractgenerators}=\Spec\CC[\abstractindependentgenerator1,\ldots,\abstractindependentgenerator{\numberofabstractgenerators}]}$
by $\overline v$.
So, If we replace $v$ by 
$\overline v$
in the definition of $\mu(g)$, we will get the same function.
We will call 
$\overline v$
the \emph{field of deformation speeds} of the deformation $\stdflatmorphism \colon S\to \CC^1$.

To formulate the next proposition, recall that $U\subseteq U'$.

\begin{proposition}\label{ksmcomputation}
There exists a section $v'\in \Gamma(U, \mathcal N_{X\subseteq \CC^{\numberofabstractgenerators}})$
such that $v'|_{U'}$ coincides with the image of the field of deformation speeds under the canonical map of sheaves 
$\Theta_{\CC^{\numberofabstractgenerators}}|_{U'}\to \mathcal N_{X\subseteq \CC^{\numberofabstractgenerators}}|_{U'}$.
Denote the image of $v'$ under the snake map for the exact sequence of sheaves 
$$
0\to \Theta_{U}\to \Theta_{\CC^{\numberofabstractgenerators}}|_{U}\to \mathcal N_{X\subseteq \CC^{\numberofabstractgenerators}}|_{U}\to 0
$$
by $\nu\in H^1(U, \Theta_{U})$. Then in the sense of Theorem \ref{schlessgen}, $\nu$ represents the 
deformation $\varepsilon\colon \widetilde S\to \Spec \CC[\varepsilon]/\varepsilon^2$.
\end{proposition}
\begin{proof}
Recall the sheaf $\mathcal I^\vee$ on $X$, which was used in the proof of Theorem \ref{schlessgen}.
Since $X$ is affine, each sheaf on $X$ is determined by the $\CC[X]$-module of its global sections, 
and $\Gamma(X,\mathcal I^\vee)=\Hom_{\CC[X]}(I/I^2,\CC[X])$.
By Proposition \ref{mushift}, $\mu$ represents an element of $\Hom_{\CC[X]}(I/I^2,\CC[X])=\Gamma(X,\mathcal I^\vee)$,
and this element represents the deformation $\varepsilon\colon \widetilde S\to \Spec \CC[\varepsilon]/\varepsilon^2$.
Denote the restriction of this element of $\Gamma(X,\mathcal I^\vee)$ to $U$ by $\mu|_{U}$.

The subset $U$ satisfies the conditions of Theorem \ref{schlessgen}.
Recall one more exact sequence of sheaves we have seen in the proof of Theorem \ref{schlessgen}:
$$
0\to\Theta_X|_{U}\stackrel{\psi|_{U}}\longrightarrow \mathcal O_X^{\oplus \numberofabstractgenerators}|_{U}\stackrel{\widetilde{\phi}|_{U}}\longrightarrow \mathcal I^\vee|_{U}\to 0,
$$
The sheaves $\Theta_{\CC^{\numberofabstractgenerators}}|_{U}$ and 
$\mathcal O_X^{\oplus \numberofabstractgenerators}|_{U}$ are isomorphic, and 
this isomorphism identifies $\psi|_{U}$ and the embedding of the tangent vector bundles.
So, we have an isomorphism 
$\mathcal I^\vee|_{U}\to \mathcal N_{X\subseteq \CC^{\numberofabstractgenerators}}|_{U}$.
By construction, this isomorphism identifies the quotient map of vector bundles and $\psi$.
A direct diagram-chase computation shows that this isomorphism identifies $\mu|_{U}$ 
with a section $v'\in \Gamma(U, \mathcal N_{X\subseteq \CC^{\numberofabstractgenerators}}|_{U})$
whose restriction to $U'$ coincides with $\overline v$.

It follows from the proof of Theorem \ref{schlessgen} that the element $\nu\in \ker(H^1(U,\Theta_X)\to 
H^1(U,\mathcal O_X^{\oplus n}))$ representing the first order deformation
is obtained from $\mu|_{U}$ via the snake map for the short exact sequence 
$$
0\to\Theta_X|_{U}\stackrel{\psi|_{U}}\longrightarrow \mathcal O_X^{\oplus \numberofabstractgenerators}|_{U}\stackrel{\widetilde{\phi}|_{U}}\longrightarrow \mathcal I^\vee|_{U}\to 0.
$$
But we have identified this exact sequence with the short exact sequence 
$$
0\to \Theta_{U}\to \Theta_{\CC^{\numberofabstractgenerators}}|_{U}\to \mathcal N_{X\subseteq \CC^{\numberofabstractgenerators}}|_{U}\to 0
$$
so that $\mu|_{U}$ is identified with $v'$, therefore, 
$\nu$ is also obtained from $v'$ via the snake map for this sequence.
\end{proof}

\section{Kodaira-Spencer map in the particular case of a deformation of a T-variety}\label{ksmtvar}

Let us apply the results of Section \ref{ksmgens} to the deformation of the T-variety we have.
Section \ref{ksmgens} speaks about one-parameter deformations, and we have a deformation 
over a $(k_1+\ldots+k_{\numberofprimitivepolyhedra})$-dimensional space $\coefficientspace$. Moreover, the variety $X$
we want to deform is the fiber over the point 
$$
\stdpolynomialcoefficient^{(1)}=(\stdpolynomialcoefficient_{1,0}^{(1)},\ldots,\stdpolynomialcoefficient_{1,\primitivepolyhedronmultiplicity{1}-1}^{(1)},
\ldots,
\stdpolynomialcoefficient_{\numberofprimitivepolyhedra,0}^{(1)},\ldots,
\stdpolynomialcoefficient_{\numberofprimitivepolyhedra,\primitivepolyhedronmultiplicity{\numberofprimitivepolyhedra}-1}^{(1)}),
$$
not above the origin. We are going to restrict the deformation to lines (with a fixed coordinate, which we will denote by $\stdflatmorphism$)
passing through this point, and then restrict it further to the double point corresponding to the tangent vector 
$\partial/\partial\stdflatmorphism$ at the origin of this line. So we will get a map 
$\Theta_{a^{(1)}}\coefficientspace\to T^1(X)$, which is called Kodaira-Spencer map and which is linear.
Since this map is linear, it is sufficient to compute it for the lines parallel to the 
coordinate axes in $\coefficientspace$ only.

So, until the end of Section \ref{ksmtvar}, fix two indices, $j$ and $k$, $1\le j\le \numberofprimitivepolyhedra$, 
$0\le k\le \primitivepolyhedronmultiplicity{j}-1$ 
and consider 
the following map from an affine line $\CC^1$ with coordinate $\stdflatmorphism$ to $\coefficientspace$:
$$
\stdflatmorphism\mapsto 
(\stdpolynomialcoefficient_{1,0}^{(1)},\ldots,
\stdpolynomialcoefficient_{1,\primitivepolyhedronmultiplicity{1}-1}^{(1)},
\ldots,
\stdpolynomialcoefficient_{j,k}^{(1)}+\stdflatmorphism
\ldots,
\stdpolynomialcoefficient_{\numberofprimitivepolyhedra,0}^{(1)},\ldots,
\stdpolynomialcoefficient_{\numberofprimitivepolyhedra,\primitivepolyhedronmultiplicity{\numberofprimitivepolyhedra}-1}^{(1)}).
$$


Now let us apply the base change $-\times_{\coefficientspace} \CC^1$ to the $\coefficientspace$-varieties 
$\coefficientspace\times \Spec\CC[t_0,t_1,t_2]$ and $\bigtotalspace$ and to the morphism
$\coefficientspace\times \Spec\CC[t_0,t_1,t_2]\to \bigtotalspace$, which was stably dominant by Lemma \ref{bigbiratisstabdom}.
We will get two new $\CC^1$-varieties, $\Spec \CC[\stdflatmorphism,t_0,t_1,t_2]$ and $\bigtotalspace\times_{\coefficientspace} \CC^1$ (denote 
$\bigtotalspace\times_{\coefficientspace} \CC^1=S$) and a morphism $\Spec \CC[\stdflatmorphism,t_0,t_1,t_2]\to S$ 
(denote it by $\stdbirationalmap$). 
By Lemma \ref{basechangestabdom}, this morphism is a stably dominant morphism of $\CC^1$-varieties.
Since $S$ is a $\CC^1$-variety, we have a morphism $S\to \CC^1$, which we will denote in Section \ref{ksmtvar} by $\stdflatmorphism$, because 
it computes the coordinate on $\CC^1$, which is $\stdflatmorphism$. In the subsequent sections, where the indices $j$ and $k$ will not be fixed anymore, 
we will denote this morphism 
by $\stdflatmorphism_{j,k}$.


The fact that 
$S=\bigtotalspace\times_{\coefficientspace} \CC^1$
means that $\stdflatmorphism\colon S\to \CC^1$ is 
the pullback of the deformation $\bigflatmorphism\colon\bigtotalspace\to\coefficientspace$ to the affine line.
Informally speaking, we restrict the deformation to an affine line (with a fixed coordinate function)
in $\coefficientspace$. 
By Corollary \ref{originfibercorrect}, 
$\stdflatmorphism^{-1}(0)=X$.
We are going to 
reformulate Lemmas \ref{totalalgebragradecomp} and \ref{totalalgebragenerators} and 
describe $\CC[S]$.

For each $\chi\in \sck\cap M$, denote
\begin{multline*}
\overline{\defparpolynomial_\chi}=
\defparpolynomial_\chi\Bigg|_{\substack{\stdpolynomialcoefficient_{j,k}=\stdpolynomialcoefficient_{j,k}^{(1)}+\stdflatmorphism\\
\stdpolynomialcoefficient_{j',k'}=\stdpolynomialcoefficient_{j',k'}^{(1)}\text{ if $j'\ne j$ or $k'\ne k$}}}
=\\
\Big(t_0^{\primitivepolyhedronmultiplicity{j}}+(\stdpolynomialcoefficient_{j,k}^{(1)}+\stdflatmorphism)t_0^k
+\sum_{\substack{0\le k'<\primitivepolyhedronmultiplicity{j}\\k'\ne k}}
\stdpolynomialcoefficient_{j,k'}^{(1)}t_0^{k'}\Big)^{-\eval_{\stdprimitivepolyhedronletter_{j}}(\chi)}
\prod_{\substack{1\le j'\le \numberofprimitivepolyhedra\\ j'\ne j}}
\left(t_0^{\primitivepolyhedronmultiplicity{j'}}
+\sum_{k'=0}^{\primitivepolyhedronmultiplicity{j'}-1} 
\stdpolynomialcoefficient_{j',k'}^{(1)}t_0^{k'}\right)^{-\eval_{\stdprimitivepolyhedronletter_{j'}}(\chi)}.
\end{multline*}
These are polynomials in $t_0$ and $\stdflatmorphism$.
\begin{lemma}\label{onedirectionalgebradecomp}
For each $\chi\in\sck\cap M$, 
the $\chi$th graded component of $\CC[S]$
(understood as a subalgebra of $\CC[\stdflatmorphism, t_0,t_1,t_2]$)
is the free 
$\CC[\stdflatmorphism]$-module
generated by 
$$
\overline{\defparpolynomial_\chi} t_1^{\chi_1^*(\chi)}t_2^{\chi_2^*(\chi)}, \overline{\defparpolynomial_\chi} t_1^{\chi_1^*(\chi)}t_2^{\chi_2^*(\chi)} t_0, \ldots, 
\overline{\defparpolynomial_\chi} t_1^{\chi_1^*(\chi)}t_2^{\chi_2^*(\chi)} t_0^{\eval_{\totalpolyhedronletter}(\chi)}.
$$
\end{lemma}
\begin{proof}
We are going to use Lemma \ref{computefiber}.
We treat $\CC[S]$ as a subalgebra of $\CC[\stdflatmorphism, t_0,t_1,t_2]$, 
and $\CC[\bigtotalspace]$ 
as a subalgebra of $\CC[\coefficientspace]\otimes \CC[t_0,t_1,t_2]$. 
Then, by Lemma \ref{computefiber}, 
$\CC[S]$ is the image of $\CC[\bigtotalspace]$ under the 
map
$\CC[V]\otimes \CC[t_0,t_1,t_2]\to 
(\CC[V]\otimes \CC[t_0,t_1,t_2])\otimes_{\CC[\coefficientspace]}\CC[\stdflatmorphism]$,
$f\mapsto f\otimes 1_{\CC[\stdflatmorphism]}$ for all $f\in \CC[V]\otimes \CC[t_0,t_1,t_2]$.
By a standard argument for tensor products, this map works as follows:
given a polynomial $f$ in variables 
$$
\stdpolynomialcoefficient_{1,0},\ldots,\stdpolynomialcoefficient_{1,\primitivepolyhedronmultiplicity{1}-1},
\ldots,
\stdpolynomialcoefficient_{\numberofprimitivepolyhedra,0},\ldots,
\stdpolynomialcoefficient_{\numberofprimitivepolyhedra,\primitivepolyhedronmultiplicity{\numberofprimitivepolyhedra}-1},
t_0,t_1,t_2,
$$
one should substitute
$$
\stdpolynomialcoefficient_{1,0}^{(1)},\ldots,
\stdpolynomialcoefficient_{1,\primitivepolyhedronmultiplicity{1}-1}^{(1)},
\ldots,
\stdpolynomialcoefficient_{j,k}^{(1)}+\stdflatmorphism
\ldots,
\stdpolynomialcoefficient_{\numberofprimitivepolyhedra,0}^{(1)},\ldots,
\stdpolynomialcoefficient_{\numberofprimitivepolyhedra,\primitivepolyhedronmultiplicity{\numberofprimitivepolyhedra}-1}^{(1)}
$$
instead of the variables
$$
\stdpolynomialcoefficient_{1,0},\ldots,\stdpolynomialcoefficient_{1,\primitivepolyhedronmultiplicity{1}-1},
\ldots,
\stdpolynomialcoefficient_{\numberofprimitivepolyhedra,0},\ldots,
\stdpolynomialcoefficient_{\numberofprimitivepolyhedra,\primitivepolyhedronmultiplicity{\numberofprimitivepolyhedra}-1},
$$
respectively.
So, the polynomials $\defparpolynomial_\chi t_1^{\chi_1^*(\chi)}t_2^{\chi_2^*(\chi)} t_0^k$
become exactly $\overline{\defparpolynomial_\chi} t_1^{\chi_1^*(\chi)}t_2^{\chi_2^*(\chi)} t_0^k$, 
and the claim follows from Lemma \ref{totalalgebragradecomp}.
\end{proof}

For each $i$, $1\le i\le \numberoflatticegenerators$, 
let us introduce the following notation.
Set
$$
\dependentgeneratorsdegree{i,0}=\overline{\defparpolynomial_{\lambda_i}} t_1^{\chi_1^*(\lambda_i)}t_2^{\chi_2^*(\lambda_i)}, 
\dependentgeneratorsdegree{i,1}=\overline{\defparpolynomial_{\lambda_i}} t_1^{\chi_1^*(\lambda_i)}t_2^{\chi_2^*(\lambda_i)} t_0, \ldots,
\dependentgeneratorsdegree{i,\eval_{\totalpolyhedronletter}(\lambda_i)}=
\overline{\defparpolynomial_{\lambda_i}} t_1^{\chi_1^*(\lambda_i)}t_2^{\chi_2^*(\lambda_i)} t_0^{\eval_{\totalpolyhedronletter}(\lambda_i)}.
$$
Denote the total number of these generators by $\numberoffixedgenerators$.
\begin{lemma}\label{onedirectionalgebragenerators}
$\CC[S]$ can be embedded into the algebra of polynomials in variables 
$\stdflatmorphism, t_0, t_1, t_2$ as the subalgebra
generated by $\stdflatmorphism$ and all $\dependentgeneratorsdegree{i,i'}$, where 
$1\le i\le \numberoflatticegenerators$ and $0\le i'\le \eval_{\totalpolyhedronletter}(\lambda_i)$.
\end{lemma}
\begin{proof}
We can use Lemma \ref{computefiber} in the same way as in the proof of the previous Lemma.
Then the claim
follows 
from Lemma 
\ref{totalalgebragenerators}.
\end{proof}

In other words, $S$ is now embedded into an $(\numberoffixedgenerators+1)$-dimensional vector space
with coordinates $\stdflatmorphism$ and $\dependentgeneratorsdegree{i,i'}$, 
and $X$ is the intersection of $S$ and the hyperplane $\stdflatmorphism=0$.

\begin{lemma}\label{xisratparam}
The preconditions of Section \ref{ksmgens} are satisfied for $\stdbirationalmap$, namely:
\begin{enumerate}
\item The first coordinate of a point $\stdbirationalmap(\stdflatmorphism,t_0,t_1,t_2)$
equals $\stdflatmorphism$.
\item The restriction of $\stdbirationalmap$ to the hyperplane $\stdflatmorphism=0$ is a dominant map 
to $X$.
\item The restriction of $\stdbirationalmap$ to the hyperplane $\stdflatmorphism=0$ maps it 
birationally to $X$.
\end{enumerate}
\end{lemma}

\begin{proof}
The first claim is clear. Since $\stdbirationalmap$ is stably dominant, 
the second claim follows from Lemma \ref{basechangestabdom} 
(applied to the change of base from $\CC^1$ to the point $\stdflatmorphism=0$ in $\CC^1$).

For the last claim, let us suppose that $\stdflatmorphism=0$
and express $t_0$, $t_1$, and $t_2$ as rational functions 
on $X=S\cap \stdflatmorphism^{-1}(0)$.

First, let us express $t_0$. Choose a degree $\chi$ in the interior of $\sck$.
Since $\mathcal D$ is an integral proper polyhedral divisor, 
$\dim \Gamma(\PP^1,\OO(\mathcal D(\chi)))\ge 2$, so $\eval_{\totalpolyhedronletter}(\lambda_i)\ge 1$.
Then 
$$
\overline{\defparpolynomial_\chi} t_1^{\chi_1^*(\chi)}t_2^{\chi_2^*(\chi)}, \overline{\defparpolynomial_\chi} t_1^{\chi_1^*(\chi)}t_2^{\chi_2^*(\chi)} t_0\in \CC[S],
$$
and $\overline{\defparpolynomial_\chi}|_{\stdflatmorphism=0} t_1^{\chi_1^*(\chi)}t_2^{\chi_2^*(\chi)}$ 
and
$\overline{\defparpolynomial_\chi}|_{\stdflatmorphism=0} t_1^{\chi_1^*(\chi)}t_2^{\chi_2^*(\chi)} t_0$ are nonzero 
functions on $X$ since they are nonzero polynomials in $\CC[t_0,t_1,t_2]$.
Their ratio is a rational function on $X$, and it equals $t_0$.

After we have an expression for $t_0$, take two degrees $\lambda_i$ and $\lambda_{i'}$ that form a basis of $M$
(such degrees exist by the definition of the set $\{\lambda_1,\ldots,\lambda_{\numberoflatticegenerators}\}$).
Write $\lambda_i=b_3\chi_1+b_4\chi_2$ and $\lambda_{i'}=b_5\chi_1+b_6\chi_2$.
Then $t_1^{b_3}t_2^{b_4}=(\dependentgeneratorsdegree{i,0}|_{\stdflatmorphism=0})/(\overline{\defparpolynomial_{\lambda_i}}|_{\stdflatmorphism=0})$, 
and $\overline{\defparpolynomial_{\lambda_i}}|_{\stdflatmorphism=0}$ is a nonzero function on $X$ since it is a nonzero element of 
$\CC[t_0,t_1,t_2]$.
Similarly, 
$t_1^{b_5}t_2^{b_6}=(\dependentgeneratorsdegree{i',0}|_{\stdflatmorphism=0})/(\overline{\defparpolynomial_{\lambda_{i'}}}|_{\stdflatmorphism=0})$.
So we have rational expressions for 
$t_1^{b_3}t_2^{b_4}$ and $t_1^{b_5}t_2^{b_6}$, and, 
since $\lambda_i$ and $\lambda_{i'}$ form a basis of $M$, 
we can also get rational expressions for $t_1$ and $t_2$ on $X$.
\end{proof}

Now we can apply the results of Subsection \ref{ksmgens}.

\begin{lemma}\label{speedfieldcomputation}
For each $i$ ($1\le i\le \numberoflatticegenerators$) and 
for each $i'$ ($1\le i'\le \eval_{\totalpolyhedronletter}(\lambda_i)$),
the $(i,i')$th coordinate of the field of deformation speeds 
(i.~e. the coordinate in front of $\partial/\partial (\dependentgeneratorsdegree{i,i'}|_{\stdflatmorphism=0})$)
equals
$$
\frac{-\eval_{\stdprimitivepolyhedronletter_j}(\lambda_i)t_0^k}
{t_0^{\primitivepolyhedronmultiplicity{j}}
+\sum_{k'=0}^{\primitivepolyhedronmultiplicity{j}-1}
\stdpolynomialcoefficient_{j,k'}^{(1)}t_0^{k'}}\dependentgeneratorsdegree{i,i'}|_{\stdflatmorphism=0}.
$$
\end{lemma}
\begin{proof}
A direct computation of 
$$
\left(\frac{\partial}{\partial\stdflatmorphism}\overline{\defparpolynomial_{\lambda_i}}\right)|_{\stdflatmorphism=0}
$$
proves this. The powers of $t_0$, $t_1$, and $t_2$ do not depend on $\stdflatmorphism$ in $\CC[\stdflatmorphism,t_0,t_1,t_2]$, 
so multiplication by these powers multiplies the derivative by the same powers.
\end{proof}

Denote this field of deformation speeds by $\stdvectorfiledxletter$. Recall that we have a rational map 
$\pi\colon X\to \PP^1$, which is defined on an open set of $X$, which we have denoted by $U_0$.
By Lemma \ref{xisratparam}, $t_0$ can be considered as a rational function on $X$.
Also recall that we have a coordinate function $t$ on $\PP^1$.

\begin{lemma}\label{ttocorrespondence}
$t_0$ is defined on $U_0\setminus \pi^{-1}(t=\infty)$, and, if $x\in U_0\setminus \pi^{-1}(t=\infty)$, 
then $t(\pi(x))=t_0(x)$.
If $x\in U_0\cap \pi^{-1}(\infty)$, then $1/t_0$ is defined at $x$, and $(1/t_0)(x)=0$.
\end{lemma}
\begin{proof}
Choose an arbitrary degree $\chi$ in the interior of $\sck$.
As we have already seen in the proof of Lemma \ref{xisratparam},
$t_0$ can be expressed as the ratio of two regular functions of degree $\chi$ on $X$, namely,
$$
t_0=\frac{\overline{\defparpolynomial_\chi}|_{\stdflatmorphism=0} t_1^{\chi_1^*(\chi)}t_2^{\chi_2^*(\chi)} t_0}
{\overline{\defparpolynomial_\chi}|_{\stdflatmorphism=0} t_1^{\chi_1^*(\chi)}t_2^{\chi_2^*(\chi)}}.
$$
By Lemma \ref{fiberoflargedeformation}, these generators
of the $\chi$th graded component of $\CC[X]$
are identified with 
$\overline{\defparpolynomial_\chi}|_{\stdflatmorphism=0, t_0=t}t\in\Gamma(\PP^1,\OO(\mathcal D(\chi)))$ 
and 
$\overline{\defparpolynomial_\chi}|_{\stdflatmorphism=0, t_0=t}\in\Gamma(\PP^1,\OO(\mathcal D(\chi)))$, respectively.

Let $x\in U_0$ be a point. By Proposition \ref{quotmorph}, 
$$
t_0=\frac{\overline{\defparpolynomial_\chi}|_{\stdflatmorphism=0} t_1^{\chi_1^*(\chi)}t_2^{\chi_2^*(\chi)} t_0}
{\overline{\defparpolynomial_\chi}|_{\stdflatmorphism=0} t_1^{\chi_1^*(\chi)}t_2^{\chi_2^*(\chi)}}
$$
is defined at $x$ if and only if 
$$
t=\frac{\overline{\defparpolynomial_\chi}|_{\stdflatmorphism=0, t_0=t}t}
{\overline{\defparpolynomial_\chi}|_{\stdflatmorphism=0, t_0=t}}
$$
is defined at 
$\pi(x)$, i.~e. if $t(\pi(x))\ne \infty$.
So, if $t(\pi(x))\ne\infty$, then $t_0$ is defined at $x$,
and in this case Proposition \ref{quotmorph}
also says that $t_0(x)=t(\pi(x))$. 

If $t(\pi(x))=\infty$, then the rational function 
$$
\frac1t=\frac{\overline{\defparpolynomial_\chi}|_{\stdflatmorphism=0, t_0=t}}
{\overline{\defparpolynomial_\chi}|_{\stdflatmorphism=0, t_0=t}t}
$$
is defined at $\pi(x)$, and $(1/t)(\pi(x))=0$. 
By Proposition \ref{quotmorph}, 
$$
\frac1{t_0}=\frac{\overline{\defparpolynomial_\chi}|_{\stdflatmorphism=0} t_1^{\chi_1^*(\chi)}t_2^{\chi_2^*(\chi)}}
{(\overline{\defparpolynomial_\chi}|_{\stdflatmorphism=0} t_1^{\chi_1^*(\chi)}t_2^{\chi_2^*(\chi)} t_0}
$$
is defined at $x$, and $(1/t_0)(x)=0$.
%
%
%
\end{proof}

\begin{lemma}\label{speedfielddefined}
If $x\in U_0$ and $\pi(x)$ is not an essential special point, 
then $\stdvectorfiledxletter$ is defined at $x$.

If $x\in U_0$, $\pi(x)$ is an essential special point, $p=p_{j'}$, and 
$\primitivepolyhedronlocalmultiplicity{j}{j'}=0$
(i.~e. the decomposition of $\stdpolyhedronletter_p$
into a Minkowski sum of polyhedra $\stdprimitivepolyhedronletter_i$
does not contain $\stdprimitivepolyhedronletter_j$), then 
$\stdvectorfiledxletter$ is also defined at $x$.
\end{lemma}

\begin{proof}
Denote $p=\pi(x)$. 
First, suppose that $t$ is defined at $p$ (in other words, $t(p)\ne \infty$.
Then $t_0$ is defined at $x$ and $t_0(x)=t(p)$.
Recall that if $p$ is a 
removable special point, 
then it must be trivial 
(Remark \ref{finitepointstrivial}).

We chose the numbers $\stdpolynomialcoefficient^{(1)}_{i,i'}$ so that 
the function
$$
t^{\primitivepolyhedronmultiplicity{j}}
+\sum_{k'=0}^{\primitivepolyhedronmultiplicity{j}-1}
\stdpolynomialcoefficient_{j,k'}^{(1)}t^{k'}
$$
only has zeros at special points.
More precisely, at a special point $p_{j'}$
this function has a zero of order 
$\primitivepolyhedronlocalmultiplicity{j}{j'}$, 
where the numbers $\primitivepolyhedronlocalmultiplicity{i}{j'}$
satisfy
$\stdpolyhedronletter_p=\sum_i \primitivepolyhedronlocalmultiplicity{i}{j'}\stdprimitivepolyhedronletter_i$ 
But if $p=p_{j'}$ is a \emph{trivial}
special point ($j'\ne \numberofdivisorpoints$), then $\primitivepolyhedronlocalmultiplicity{i}{j'}=0$ 
for all $i$.
So, if $p$ is either a trivial special point, or an essential
special point such that $\primitivepolyhedronlocalmultiplicity{j}{j'}$
still equals zero, then the function 
$$
t^{\primitivepolyhedronmultiplicity{j}}
+\sum_{k'=0}^{\primitivepolyhedronmultiplicity{j}-1}
\stdpolynomialcoefficient_{j,k'}^{(1)}t^{k'}
$$
has zero of order 0, i.~e. does not have a zero at all, at $p$.

In other words, if $t(p)\ne\infty$,
then 
$$t(p)^{\primitivepolyhedronmultiplicity{j}}
+\sum_{k'=0}^{\primitivepolyhedronmultiplicity{j}-1}
\stdpolynomialcoefficient_{j,k'}^{(1)}t(p)^{k'}\ne 0.
$$
But then 
$$
t_0(x)^{\primitivepolyhedronmultiplicity{j}}
+\sum_{k'=0}^{\primitivepolyhedronmultiplicity{j}-1}
\stdpolynomialcoefficient_{j,k'}^{(1)}t_0(x)^{k'}\ne 0,
$$
and the rational function
$$
\frac{t_0^k}
{t_0^{\primitivepolyhedronmultiplicity{j}}
+\sum_{k'=0}^{\primitivepolyhedronmultiplicity{j}-1}
\stdpolynomialcoefficient_{j,k'}^{(1)}t_0^{k'}}
$$
is defined at $x$.

Now suppose that $t$ is not defined at $p$, or, informally speaking, $t(p)=\infty$.
We can write
$$
\frac{t_0^k}
{t_0^{\primitivepolyhedronmultiplicity{j}}
+\sum_{k'=0}^{\primitivepolyhedronmultiplicity{j}-1}
\stdpolynomialcoefficient_{j,k'}^{(1)}t_0^{k'}}
=
\frac1{t_0^{\primitivepolyhedronmultiplicity{j}-k}}
\frac1{1+\sum_{k'=0}^{\primitivepolyhedronmultiplicity{j}-1}
\stdpolynomialcoefficient_{j,k'}^{(1)}(1/t_0)^{\primitivepolyhedronmultiplicity{j}-k'}}
$$
By Lemma \ref{ttocorrespondence}, the rational function $1/t_0$ is defined at $x$, and $(1/t_0)(x)=0$.
Since $0\le k < \primitivepolyhedronmultiplicity{j}$, 
the rational function 
$$
\frac{t_0^k}
{t_0^{\primitivepolyhedronmultiplicity{j}}
+\sum_{k'=0}^{\primitivepolyhedronmultiplicity{j}-1}
\stdpolynomialcoefficient_{j,k'}^{(1)}t_0^{k'}}
$$
is also defined (and takes value 0) at $x$.
\end{proof}

We need to construct some tangent vector fields on $X$ (i.~e. sections of $\Theta_X$).
Let $f$ be a linear function on $M$ with values in $\QQ$. In other words, let $f$ be a point of $N_\QQ$.
Recall that $X$ is embedded into an $\numberoffixedgenerators$-dimensional 
vector space $\CC^{\numberoffixedgenerators}$ 
with coordinates $\dependentgeneratorsdegree{i,i'}|_{\stdflatmorphism=0}$.
Consider the following section of $\Theta_{\CC^{\numberoffixedgenerators}}|_{X}$
and denote it by $\stdvectorfiledxletter'_f$: 
$$
\stdvectorfiledxletter'_f=\sum_{i=1}^{\numberoflatticegenerators}\sum_{i'=0}^{\eval_{\totalpolyhedronletter}(\lambda_i)}
f(\lambda_i)\dependentgeneratorsdegree{i,i'}|_{\stdflatmorphism=0}\frac\partial{\partial \dependentgeneratorsdegree{i,i'}|_{\stdflatmorphism=0}}.
$$
\begin{lemma}\label{ksmtangentcorrections}
In fact, $\stdvectorfiledxletter'_f$ consists of vectors tangent to $X$, 
i.~e. $\stdvectorfiledxletter'_f\in\Gamma(X,\Theta_X)$.
\end{lemma}

\begin{proof}
It is sufficient to verify the condition $\stdvectorfiledxletter'_f\in\Gamma(X,\Theta_X)$
on an open subset of $X$. For such an open subset we can use the open set where 
$(\stdbirationalmap|_{\stdflatmorphism=0})^{-1}$
is defined.

So, consider the following vector field on $\Spec\CC[t_0,t_1,t_2]$: 
$\stdvectorfiledpletter=f(\chi_1)t_1\partial/\partial t_1+f(\chi_2)t_2\partial/\partial t_2$.
The differential of $\stdbirationalmap|_{\stdflatmorphism=0}$ maps it to
\begin{multline*}
\!\!\!\sum_{i=1}^{\numberoflatticegenerators}\sum_{k'=0}^{\eval_{\totalpolyhedronletter}(\lambda_i)}
\left(\left(
t_1f(\chi_1)\frac{\overline{\partial\defparpolynomial_{\lambda_i}}|_{\stdflatmorphism=0}t_0^{k'}t_1^{\chi_1^*(\lambda_i)}t_2^{\chi_2^*(\lambda_i)}}{\partial t_1}
+
t_2f(\chi_2)\frac{\overline{\partial\defparpolynomial_{\lambda_i}}|_{\stdflatmorphism=0}t_0^{k'}t_1^{\chi_1^*(\lambda_i)}t_2^{\chi_2^*(\lambda_i)}}{\partial t_2}
\right)\frac{\partial}{\partial\dependentgeneratorsdegree{i,k'}}\right)=\\
\sum_{i=1}^{\numberoflatticegenerators}
\left(
(f(\chi_1)\chi_1^*(\lambda_i)+f(\chi_2)\chi_2^*(\lambda_i))
\sum_{k'=0}^{\eval_{\totalpolyhedronletter}(\lambda_i)}
\overline{\defparpolynomial_{\lambda_i}}|_{\stdflatmorphism=0}t_0^{k'}t_1^{\chi_1^*(\lambda_i)}t_2^{\chi_2^*(\lambda_i)}
\frac{\partial}{\partial\dependentgeneratorsdegree{i,k'}}
\right)=\stdvectorfiledxletter'_f.
\end{multline*}
\end{proof}

Now recall that we have a sufficient system $\{U_i\}$ of $X$. We have $\numberoffixedopensets$ of these sets, 
and each set $U_i$, except $U_{\numberoffixedopensets}$, corresponds to a pair 
$(p,j')$, where $p\in \PP^1$ is a special point, and $1\le j'\le \numberofverticespt{p}$. 
Sometimes we have two open sets $U_i$ corresponding to one such pair, this happens if 
and only if $p$ is removable special point and 
$\deg\mathcal D(\sckboundarybasis{0})>0$ and $\deg\mathcal D(\sckboundarybasis{1})>0$.
We have $U_{\numberoffixedopensets}\subseteq U_i$ for $1\le i<\numberoffixedopensets$.
The union $\cup_{i=1}^{\numberoffixedopensets-1}U_i$ was denoted by $U$, and $U\subseteq U_0$.
$U$ is smooth, and $\codim_X(X\setminus U)\ge 2$.
We also have an affine covering of $\PP^1$, which consists of 
the sets $W_p=W\cup \{p\}$ for all special points $p$, where
$W$ is the set of all ordinary points.

We are going to define tangent vector fields $\stdvectorfiledxletter_i$ (one for each set $U_i$) defined on some
open subsets of $X$
so that 
$\stdvectorfiledxletter-\stdvectorfiledxletter_i\in \Gamma(U_i,\Theta_{\CC^{\numberoffixedgenerators}}|_{X})$ 
for each $i$ ($1\le i\le \numberoffixedopensets$). Note that $U_{\numberoffixedopensets}\subseteq \pi^{-1}(W)$, 
so, by Lemma \ref{speedfielddefined}, $\stdvectorfiledxletter$ is already defined on $U_{\numberoffixedopensets}$, 
and we can (and we will) set $\stdvectorfiledxletter_{\numberoffixedopensets}=0$.

Now suppose that an open set $U_i$ ($1\le i<\numberoffixedopensets$) 
corresponds to a special point $p$ and a vertex $\indexedvertexpt p{j'}$.
Then $U_i\subseteq \pi^{-1}(W_p)$.
If $p$ is a removable special point (including the point with $t(p)=\infty$),
then by Lemma \ref{speedfielddefined}, $\stdvectorfiledxletter$ is defined on $U_i$, 
and we set $\stdvectorfiledxletter_i=0$. We do the same if $p$ is an essential 
special point, $p=p_{j''}$, but $\primitivepolyhedronlocalmultiplicity{j}{j''}=0$.

Finally, let us consider the case when $p$ is an essential special point,
$p=p_{j''}$, and $\primitivepolyhedronlocalmultiplicity{j}{j''}\ne 0$.
This means that the convex piecewise-linear function 
$\eval_{\stdpolyhedronletter_p}\colon \sck\to \QQ$
can be decomposed into a sum of several convex piecewise-linear functions, 
and one of these summands is $\primitivepolyhedronlocalmultiplicity{j}{j''}\eval_{\stdprimitivepolyhedronletter_j}$.
Addition of convex piecewise-linear functions can only split maximal subcones of
linearity into a smaller cones, and $\normalvertexcone{\indexedvertexpt p{j'}}{\stdpolyhedronletter_p}$
is a maximal subcone of linearity of $\eval_{\stdpolyhedronletter_p}$.
Therefore, $\normalvertexcone{\indexedvertexpt p{j'}}{\stdpolyhedronletter_p}$
is a subcone of one of the maximal subcones of linearity of the function 
$\eval_{\stdprimitivepolyhedronletter_j}$.
The function $\eval_{\stdprimitivepolyhedronletter_j}$
has two maximal subcones of linearity, they are the normal 
vertex cones of the two vertices of $\stdprimitivepolyhedronletter_j$, 
$\normalvertexcone{\indexedvertexpt p{j'}}{\stdpolyhedronletter_p}\subseteq 
\normalvertexcone{\indexedvertex {\stdprimitivepolyhedronletter_j}l}{\stdprimitivepolyhedronletter_j}$
for some $l\in\{0,1\}$. Recall that points of $N$, in particular, vertices of $\stdprimitivepolyhedronletter_j$, 
can be considered as functions on $M$, and set 
$$
\stdvectorfiledxletter_i=
\frac{t_0^k}
{t_0^{\primitivepolyhedronmultiplicity{j}}
+\sum_{k'=0}^{\primitivepolyhedronmultiplicity{j}-1}
\stdpolynomialcoefficient_{j,k'}^{(1)}t_0^{k'}}
\stdvectorfiledxletter'_{-\indexedvertex {\stdprimitivepolyhedronletter_j}l}.
$$

\begin{lemma}\label{ksmregularrepresentatives}
For each $i$, $1\le i\le \numberoffixedopensets$ 
we have 
$\stdvectorfiledxletter-\stdvectorfiledxletter_i\in \Gamma(U_i,\Theta_{\CC^{\numberoffixedgenerators}}|_{X})$.
\end{lemma}

\begin{proof}
The only nontrivial cases we have to consider are the cases when $i$ satisfies the following conditions:
\begin{enumerate}
\item $i<\numberoffixedopensets$, and hence $U_i$ corresponds to a pair $(p,j')$, where $p$ is an essential special point,
and $1\le j'\le \numberofverticespt{p}$.
\item If $p=p_{j''}$, then $\primitivepolyhedronlocalmultiplicity{j}{j''}\ne 0$.
\end{enumerate}
Under these conditions, $\normalvertexcone{\indexedvertexpt p{j'}}{\stdpolyhedronletter_p}$
is contained in some of the cones
$\normalvertexcone{\indexedvertex {\stdprimitivepolyhedronletter_j}l}{\stdprimitivepolyhedronletter_j}$
(for some $l\in\{0,1\}$), 
and
$$
\stdvectorfiledxletter_i=
\frac{t_0^k}
{t_0^{\primitivepolyhedronmultiplicity{j}}
+\sum_{k'=0}^{\primitivepolyhedronmultiplicity{j}-1}
\stdpolynomialcoefficient_{j,k'}^{(1)}t_0^{k'}}
\stdvectorfiledxletter'_{-\indexedvertex {\stdprimitivepolyhedronletter_j}l}.
$$
Then
$$
\stdvectorfiledxletter-\stdvectorfiledxletter_i=
\sum_{i'=1}^{\numberoflatticegenerators}\sum_{i''=0}^{\eval_{\totalpolyhedronletter}(\lambda_{i'})}
\frac{t_0^k}
{t_0^{\primitivepolyhedronmultiplicity{j}}
+\sum_{k'=0}^{\primitivepolyhedronmultiplicity{j}-1}
\stdpolynomialcoefficient_{j,k'}^{(1)}t_0^{k'}}
(-\eval_{\stdprimitivepolyhedronletter_j}(\lambda_{i'})+\lambda_{i'}(\indexedvertex {\stdprimitivepolyhedronletter_j}l))
\dependentgeneratorsdegree{i',i''}|_{\stdflatmorphism=0}\frac\partial{\partial \dependentgeneratorsdegree{i',i''}|_{\stdflatmorphism=0}}.
$$
This section of $\Theta_{\CC^{\numberoffixedgenerators}}|_{X}$ is defined on $U_i$ if and only if each 
function in front of $\partial/\partial(\dependentgeneratorsdegree{i',i''}|_{\stdflatmorphism=0})$
is defined on $U_i$. So, let us fix indices $i'$ ($1\le i'\le \numberoflatticegenerators$) and 
$i''$ ($0\le i'\le \eval_{\totalpolyhedronletter}(\lambda_{i'})$) until the end of the proof
and check that the function
$$
\frac{t_0^k}
{t_0^{\primitivepolyhedronmultiplicity{j}}
+\sum_{k'=0}^{\primitivepolyhedronmultiplicity{j}-1}
\stdpolynomialcoefficient_{j,k'}^{(1)}t_0^{k'}}
(-\eval_{\stdprimitivepolyhedronletter_j}(\lambda_{i'})+\lambda_{i'}(\indexedvertex {\stdprimitivepolyhedronletter_j}l))
\dependentgeneratorsdegree{i',i''}|_{\stdflatmorphism=0}
$$
is defined on $U_i$.
Denote this (a priori rational) function on $U_i$ by $f$.

First, if $\lambda_{i'}\in\normalvertexcone{\indexedvertex {\stdprimitivepolyhedronletter_j}l}{\stdprimitivepolyhedronletter_j}$, 
then $\eval_{\stdprimitivepolyhedronletter_j}(\lambda_{i'})=\lambda_{i'}(\indexedvertex {\stdprimitivepolyhedronletter_j}l)$, 
and 
$$
\frac{t_0^k}
{t_0^{\primitivepolyhedronmultiplicity{j}}
+\sum_{k'=0}^{\primitivepolyhedronmultiplicity{j}-1}
\stdpolynomialcoefficient_{j,k'}^{(1)}t_0^{k'}}
(-\eval_{\stdprimitivepolyhedronletter_j}(\lambda_{i'})+\lambda_{i'}(\indexedvertex {\stdprimitivepolyhedronletter_j}l))
\dependentgeneratorsdegree{i',i''}|_{\stdflatmorphism=0}=0.
$$

Now suppose that $\lambda_{i'}\notin\normalvertexcone{\indexedvertex {\stdprimitivepolyhedronletter_j}l}{\stdprimitivepolyhedronletter_j}$.
We are going to use Lemma \ref{uiregularity}.
By Lemma \ref{ttocorrespondence}, $t_0=t\circ \pi$ as a rational function on $X$, 
so, if we denote 
$$
f_2=
\frac{t^k}
{t^{\primitivepolyhedronmultiplicity{j}}
+\sum_{k'=0}^{\primitivepolyhedronmultiplicity{j}-1}
\stdpolynomialcoefficient_{j,k'}^{(1)}t^{k'}},
$$
then 
$$
f=(f_2\circ \pi)
(-\eval_{\stdprimitivepolyhedronletter_j}(\lambda_{i'})+\lambda_{i'}(\indexedvertex {\stdprimitivepolyhedronletter_j}l))
\dependentgeneratorsdegree{i',i''}|_{\stdflatmorphism=0}
$$
as a rational function on $X$.
Denote by $f_1$ the following section of 
$\OO_{\PP^1}(\mathcal D(\lambda_{i'}))$:
$$
f_1=(-\eval_{\stdprimitivepolyhedronletter_j}(\lambda_{i'})+\lambda_{i'}(\indexedvertex {\stdprimitivepolyhedronletter_j}l))
\overline{\defparpolynomial_{\lambda_{i'}}}|_{\stdflatmorphism=0, t_0=t}t^{i''}
$$
Then, by Lemma \ref{fiberoflargedeformation}, 
$$
(-\eval_{\stdprimitivepolyhedronletter_j}(\lambda_{i'})+\lambda_{i'}(\indexedvertex {\stdprimitivepolyhedronletter_j}l))
\dependentgeneratorsdegree{i',i''}|_{\stdflatmorphism=0}=\widetilde{f_1},
$$
and $f=(f_2\circ \pi)\widetilde{f_1}$.

Let us verify the conditions of Lemma \ref{uiregularity}.
By construction, $\overline{f_1}$ is defined at all points of $\PP^1$ except $t=\infty$, 
in particular, it is defined at all ordinary point.
And the denominator of $f_2$ does not have zeros at ordinary points, so 
$f_2\overline{f_1}$ is regular at all ordinary points, i.~e. at all points of $V_i$ 
except, possibly, $p$.
Recall that 
$$
\overline{\defparpolynomial_{\lambda_{i'}}}|_{\stdflatmorphism=0,t_0=t}
\prod_{1\le j'''\le \numberofprimitivepolyhedra}
\left(t^{\primitivepolyhedronmultiplicity{j'}}
+\sum_{k'=0}^{\primitivepolyhedronmultiplicity{j'}-1} 
\stdpolynomialcoefficient_{j''',k'}^{(1)}t^{k'}\right)^{-\eval_{\stdprimitivepolyhedronletter_{j'''}}(\lambda_{i'})}.
$$
By choice of the coefficients 
$\stdpolynomialcoefficient_{j''',k'}^{(1)}$, 
$$
\ord_p(\overline{\defparpolynomial_{\lambda_{i'}}}|_{\stdflatmorphism=0,t_0=t})=
-\sum_{j'''=0}^{\numberofprimitivepolyhedra}\primitivepolyhedronlocalmultiplicity{j'''}{j''}\eval_{\stdprimitivepolyhedronletter_{j'''}}(\lambda_{i'})=
-\eval_{\stdpolyhedronletter_p}(\lambda_{i'}).
$$
So, $\ord_p(\overline{f_1})\ge -\eval_{\stdpolyhedronletter_p}(\lambda_{i'})$.
For $f_2$, we have 
$$
\ord_p(f_2)\ge \ord_p\left(\frac1{t^{\primitivepolyhedronmultiplicity{j}}
+\sum_{k'=0}^{\primitivepolyhedronmultiplicity{j}-1}
\stdpolynomialcoefficient_{j,k'}^{(1)}t^{k'}}\right)=-\primitivepolyhedronlocalmultiplicity{j}{j''}.
$$

It suffices to prove that 
$$
-\eval_{\stdpolyhedronletter_p}(\lambda_{i'})-\primitivepolyhedronlocalmultiplicity{j}{j''}\ge
-\uidegree{i,1}^*(\lambda_{i'})\mathcal D_p(\uidegree{i,1})
-\uidegree{i,2}^*(\lambda_{i'})\mathcal D_p(\uidegree{i,2}).
$$
By construction of the sets $U_i$, we have 
$\uidegree{i,1},\uidegree{i,2}\in\normalvertexcone{\indexedvertexpt p{j'}}{\stdpolyhedronletter_p}$, 
so $\mathcal D_p(\uidegree{i,1})=\uidegree{i,1}(\indexedvertexpt p{j'})$ and 
$\mathcal D_p(\uidegree{i,2})=\uidegree{i,2}(\indexedvertexpt p{j'})$.
So, 
$$
-\uidegree{i,1}^*(\lambda_{i'})\mathcal D_p(\uidegree{i,1})
-\uidegree{i,2}^*(\lambda_{i'})\mathcal D_p(\uidegree{i,2})=
-\uidegree{i,1}^*(\lambda_{i'})\uidegree{i,1}(\indexedvertexpt p{j'})
-\uidegree{i,2}^*(\lambda_{i'})\uidegree{i,2}(\indexedvertexpt p{j'})=
-\lambda_{i'}(\indexedvertexpt p{j'}),
$$
and it suffices to prove that 
$-\eval_{\stdpolyhedronletter_p}(\lambda_{i'})-\primitivepolyhedronlocalmultiplicity{j}{j''}\ge
-\lambda_{i'}(\indexedvertexpt p{j'})$.

Since $\stdpolyhedronletter_p=\sum_{j'''=1}^{\numberofprimitivepolyhedra}\primitivepolyhedronlocalmultiplicity{j'''}{j''}\stdprimitivepolyhedronletter_{j'''}$, 
for each polyhedron $\stdprimitivepolyhedronletter_{j'''}$ such that 
$\primitivepolyhedronlocalmultiplicity{j'''}{j''}\ne 0$, the cone 
$\normalvertexcone{\indexedvertexpt p{j'}}{\stdpolyhedronletter_p}$
is contained in a maximal cone of linearity of the function 
$\eval_{\stdprimitivepolyhedronletter_{j'''}}$, which is 
the normal vertex cone of a vertex of $\eval_{\stdprimitivepolyhedronletter_{j'''}}$. 
Denote this vertex by $b_{j'''}\in N$. In other words, 
$\normalvertexcone{\indexedvertexpt p{j'}}{\stdpolyhedronletter_p}\subseteq \normalvertexcone{b_{j'''}}{\stdprimitivepolyhedronletter_{j'''}}$.
(Note that $b_j=\indexedvertex {\stdprimitivepolyhedronletter_j}l$
according to previously chosen notation.)
If $\primitivepolyhedronlocalmultiplicity{j'''}{j''}=0$, denote by $b_{j'''}$
an arbitrary vertex of $\stdprimitivepolyhedronletter_{j'''}$.
Then the points $\indexedvertexpt p{j'}$ and 
$\sum_{j'''=1}^{\numberofprimitivepolyhedra}\primitivepolyhedronlocalmultiplicity{j'''}{j''}b_{j'''}$
define the same function on the two-dimensional cone 
$\normalvertexcone{\indexedvertexpt p{j'}}{\stdpolyhedronletter_p}$.
Therefore, 
$\indexedvertexpt p{j'}=\sum_{j'''=1}^{\numberofprimitivepolyhedra}\primitivepolyhedronlocalmultiplicity{j'''}{j''}b_{j'''}$
in $M$.

Now we can write
$$
-\eval_{\stdpolyhedronletter_p}(\lambda_{i'})=-\sum_{j'''=1}^{\numberofprimitivepolyhedra}
\primitivepolyhedronlocalmultiplicity{j'''}{j''}\eval_{\stdprimitivepolyhedronletter_{j'''}}(\lambda_{i'})
\text{ and }
-\lambda_{i'}(\indexedvertexpt p{j'})=-\sum_{j'''=1}^{\numberofprimitivepolyhedra}
\primitivepolyhedronlocalmultiplicity{j'''}{j''}\lambda_{i'}(b_{j'''}).
$$
Recall that $\eval_{\stdprimitivepolyhedronletter_{j'''}}(\lambda_{i'})$
is the minimum among the values that the function $\lambda_{i'}$
takes at the vertices of $\stdprimitivepolyhedronletter_{j'''}$, 
so 
$-\eval_{\stdprimitivepolyhedronletter_{j'''}}(\lambda_{i'})\ge -\lambda_{i'}(b_{j'''})$.
Moreover, since 
$\lambda_{i'}\notin\normalvertexcone{\indexedvertex {\stdprimitivepolyhedronletter_j}l}{\stdprimitivepolyhedronletter_j}=
\normalvertexcone{b_j}{\stdprimitivepolyhedronletter_j}$, 
$-\eval_{\stdprimitivepolyhedronletter_{j}}(\lambda_{i'})>-\lambda_{i'}(b_{j})$.
These numbers are integer, so
$-\eval_{\stdprimitivepolyhedronletter_{j}}(\lambda_{i'})-1\ge-\lambda_{i'}(b_{j})$.
Therefore, 
$$
-\primitivepolyhedronlocalmultiplicity{j}{j''}\eval_{\stdprimitivepolyhedronletter_{j}}(\lambda_{i'})-\primitivepolyhedronlocalmultiplicity{j}{j''}
\ge
\primitivepolyhedronlocalmultiplicity{j}{j''}\lambda_{i'}(b_{j}),
$$
and
$$
-\eval_{\stdpolyhedronletter_p}(\lambda_{i'})-\primitivepolyhedronlocalmultiplicity{j}{j''}\ge
-\lambda_{i'}(\indexedvertexpt p{j'}).
$$
\end{proof}

\begin{lemma}\label{h1qcomputed}
The image of the deformation $\stdflatmorphism\colon S\to \CC^1$ under the Kodaira-Spencer map in 
$H^1(U,\Theta_U)$
is represented by the following Cech class: on each intersection $U_i\cap U_{i'}$ ($i<i'$) 
we have vector field $\stdvectorfiledxletter_i-\stdvectorfiledxletter_{i'}$. Here $1\le i<i'<\numberoffixedopensets$ 
(resp. $1\le i<i\le \numberoffixedopensets$) if we use $U_1,\ldots, U_{\numberoffixedopensets-1}$
(resp. $U_1,\ldots, U_{\numberoffixedopensets}$) as the affine covering of $U$.
\end{lemma}

\begin{proof}
This follows directly from Proposition \ref{ksmcomputation} and Lemmas \ref{speedfieldcomputation}, 
\ref{ksmtangentcorrections}, and \ref{ksmregularrepresentatives}.
\end{proof}

\begin{corollary}
The isomorphisms
\begin{multline*}
H^1(U,\Theta_U)=
\left(\ker\Bigg(\bigoplus_{i=1}^\numberoffixedopensets\Big(H^0(U_{\numberoffixedopensets},\Theta_U)/H^0(U_i,\Theta_U)\Big)
\vphantom{\bigoplus_{1\le i<i'\le \numberoffixedopensets}}
\right.\!
\\
\left.\!
\left.\!
\longrightarrow
\bigoplus_{1\le i<i'\le \numberoffixedopensets}\Big(H^0(U_{\numberoffixedopensets},\Theta_U)/H^0(U_i\cap U_{i'},\Theta_U)\Big)\Bigg)\right)\right/H^0(U_{\numberoffixedopensets},\Theta_U)
\end{multline*}
and 
\begin{multline*}
H^1(U,\Theta_U)=
\left(\ker\Bigg(\bigoplus_{i=1}^{\numberoffixedopensets-1}\Big(H^0(U_{\numberoffixedopensets},\Theta_U)/H^0(U_i,\Theta_U)\Big)
\vphantom{\bigoplus_{1\le i<i'\le \numberoffixedopensets-1}}
\right.\!
\\
\left.\!
\left.\!
\longrightarrow 
\bigoplus_{1\le i<i'\le \numberoffixedopensets-1}\Big(H^0(U_{\numberoffixedopensets},\Theta_U)/H^0(U_i\cap U_{i'},\Theta_U)\Big)\Bigg)\right)\right/H^0(U_{\numberoffixedopensets},\Theta_U)
\end{multline*}
(respectively)
from Corollary \ref{h1computegen}
identifies the image of the deformation $\stdflatmorphism\colon S\to \CC^1$ 
under the Kodaira-Spencer map
with the classes of 
$$
(\stdvectorfiledpletter_i)_{1\le i\le \numberoffixedopensets}\in\bigoplus_{i=1}^\numberoffixedopensets H^1(U_{\numberoffixedopensets},\Theta_U)
$$
and 
$$
(\stdvectorfiledpletter_i)_{1\le i\le \numberoffixedopensets-1}\in\bigoplus_{i=1}^{\numberoffixedopensets-1} H^1(U_{\numberoffixedopensets},\Theta_U)
$$
(respectively).
\end{corollary}\label{h1qquotcomputed}
\begin{proof}
This follows from Lemma \ref{h1qcomputed} and the construction of isomorphisms in the proof of Proposition \ref{hnasquotient}.
\end{proof}

\begin{lemma}
Let $f$ be a homogeneous function of degree $\chi\in\sck\cap M$ on $X$, and let $1\le i\le \numberoffixedopensets$.

If $U_i$ corresponds to an essential special point $p=p_{j''}$ and a vertex $\indexedvertexpt{p}{j'}$, then
$$
df(\stdvectorfiledxletter_i)=-\frac{t_0^k}
{t_0^{\primitivepolyhedronmultiplicity{j}}
+\sum_{k'=0}^{\primitivepolyhedronmultiplicity{j}-1}
\stdpolynomialcoefficient_{j,k'}^{(1)}t_0^{k'}}
\chi(\indexedvertex {\stdprimitivepolyhedronletter_j}l) f,
$$
where $l\in\{0,1\}$ is such that 
$\normalvertexcone{\indexedvertexpt p{j'}}{\stdpolyhedronletter_p}\subseteq 
\normalvertexcone{\indexedvertex {\stdprimitivepolyhedronletter_j}l}{\stdprimitivepolyhedronletter_j}$.

Otherwise, $df(\stdvectorfiledxletter_i)=0$ (recall that $\stdvectorfiledxletter_i=0$ in this case).
\end{lemma}

\begin{proof}
Each function of degree $\chi$ on $X$ is a polynomial in variables $\dependentgeneratorsdegree{i',i''}|_{\stdflatmorphism=0}$.
Let us first consider the case when $f$ is a monomial.
Then we prove the lemma by induction on 
the number of variables in this monomial.

First, if $f=\dependentgeneratorsdegree{i',i''}|_{\stdflatmorphism=0}$, then $\chi=\lambda_{i'}$, 
and the statement of Lemma holds by the definition of 
$\stdvectorfiledxletter_i$.

Suppose that $f=\dependentgeneratorsdegree{i',i''}|_{\stdflatmorphism=0}f_1$, where $f_1$ is another monomial
of degree $\chi-\lambda_{i'}$, and the statement follows from Leibniz rule.

Finally, the statement of lemma for arbitrary polynomials follows by linearity.
\end{proof}

\begin{lemma}\label{ksmuidesc}
Let $1\le i\le \numberoffixedopensets$.

If $U_i$ corresponds to an essential special point $p=p_{j''}$ and a vertex $\indexedvertexpt{p}{j'}$, then
the $U_i$-description of $\stdvectorfiledxletter_i$ equals
$$
\left(
-\frac{t^k}
{t^{\primitivepolyhedronmultiplicity{j}}
+\sum_{k'=0}^{\primitivepolyhedronmultiplicity{j}-1}
\stdpolynomialcoefficient_{j,k'}^{(1)}t^{k'}}
\uidegree{i,1}(\indexedvertex {\stdprimitivepolyhedronletter_j}l),
-\frac{t^k}
{t^{\primitivepolyhedronmultiplicity{j}}
+\sum_{k'=0}^{\primitivepolyhedronmultiplicity{j}-1}
\stdpolynomialcoefficient_{j,k'}^{(1)}t^{k'}}
\uidegree{i,2}(\indexedvertex {\stdprimitivepolyhedronletter_j}l),
0\right),
$$
where $l\in\{0,1\}$ is such that 
$\normalvertexcone{\indexedvertexpt p{j'}}{\stdpolyhedronletter_p}\subseteq 
\normalvertexcone{\indexedvertex {\stdprimitivepolyhedronletter_j}l}{\stdprimitivepolyhedronletter_j}$.

Otherwise, the $U_i$-description of $\stdvectorfiledxletter_i$ is $(0,0,0)$ (recall that $\stdvectorfiledxletter_i=0$
in this case).
\end{lemma}

\begin{proof}
Suppose that $U_i$ corresponds to an essential special point $p=p_{j''}$ and a vertex $\indexedvertexpt{p}{j'}$.

Let $p'\in\PP^1$ be an ordinary point, and let $x$ be the canonical point in $\pi^{-1}(p')\cap U_i$.
Then the rational function 
$$
-\frac{t^k}
{t^{\primitivepolyhedronmultiplicity{j}}
+\sum_{k'=0}^{\primitivepolyhedronmultiplicity{j}-1}
\stdpolynomialcoefficient_{j,k'}^{(1)}t^{k'}}
\uidegree{i,1}(\indexedvertex {\stdprimitivepolyhedronletter_j}l)
$$
is defined at $p'$, $t_0$ is defined at $x$, and $t_0(x)=t(p')$ (Lemma \ref{ttocorrespondence}).

By definition, the values of the first two components of the $U_i$-description
of $\stdvectorfiledxletter_i$ at $p'$ equal 
$d_x\wtuithreadfunction{i,1}(\stdvectorfiledpletter_i)$ and 
$d_x\wtuithreadfunction{i,2}(\stdvectorfiledpletter_i)$, respectively.
By the previous lemma, 
$$
d_x\wtuithreadfunction{i,1}(\stdvectorfiledpletter_i)=
-\frac{t_0(x)^k}
{t_0(x)^{\primitivepolyhedronmultiplicity{j}}
+\sum_{k'=0}^{\primitivepolyhedronmultiplicity{j}-1}
\stdpolynomialcoefficient_{j,k'}^{(1)}t_0(x)^{k'}}
\uidegree{i,1}(\indexedvertex {\stdprimitivepolyhedronletter_j}l)\wtuithreadfunction{i,1}(x)
$$
and
$$
d_x\wtuithreadfunction{i,2}(\stdvectorfiledpletter_i)=
-\frac{t_0(x)^k}
{t_0(x)^{\primitivepolyhedronmultiplicity{j}}
+\sum_{k'=0}^{\primitivepolyhedronmultiplicity{j}-1}
\stdpolynomialcoefficient_{j,k'}^{(1)}t_0(x)^{k'}}
\uidegree{i,2}(\indexedvertex {\stdprimitivepolyhedronletter_j}l)\wtuithreadfunction{i,2}(x).
$$
But by the definition of a canonical point, 
$\wtuithreadfunction{i,1}(x)=\wtuithreadfunction{i,2}(x)=1$.
So, the values of the first two components of the $U_i$-description
equal 
$$
-\frac{t(p')^k}
{t(p')^{\primitivepolyhedronmultiplicity{j}}
+\sum_{k'=0}^{\primitivepolyhedronmultiplicity{j}-1}
\stdpolynomialcoefficient_{j,k'}^{(1)}t(p')^{k'}}
\uidegree{i,1}(\indexedvertex {\stdprimitivepolyhedronletter_j}l)
$$
and
$$
-\frac{t(p')^k}
{t(p')^{\primitivepolyhedronmultiplicity{j}}
+\sum_{k'=0}^{\primitivepolyhedronmultiplicity{j}-1}
\stdpolynomialcoefficient_{j,k'}^{(1)}t(p')^{k'}}
\uidegree{i,2}(\indexedvertex {\stdprimitivepolyhedronletter_j}l),
$$
respectively.

To compute the third component of the $U_i$-description, note that by Lemma \ref{ttocorrespondence}, 
$t_0$ can be considered as follows. Consider the affine chart $t\ne \infty$ on $\PP^1$. It is an affine line, 
and $t$ is a coordinate on it. Then $t_0$ is a function on $U_0\cap \pi^{-1}(\{t\ne \infty\})$
that computes the coordinate $t$ of the image of a point $x'\in U_0\cap \pi^{-1}(\{t\ne \infty\})$.
In these terms, $d_x\pi(\stdvectorfiledxletter_i)=d_xt_0(\stdvectorfiledxletter_i)(\partial/\partial t)$.

Let us compute $d_xt_0(\stdvectorfiledxletter_i)$. Choose a degree $\chi$ in the interior of $\sck$.
As we have seen in the proof of Lemma \ref{xisratparam}, there exist global functions $f_1$ and $f_2$ of degree $\chi$ on $X$
such that $t_0=f_1/f_2$.
Using the previous lemma again, we can write
\begin{multline*}
d_x\frac{f_1}{f_2}(\stdvectorfiledxletter_i)=\frac{f_2(x)d_x{f_1}(\stdvectorfiledxletter_i)-f_1(x)d_x{f_2}(\stdvectorfiledxletter_i)}{f_2(x)^2}=\\
\frac{-\frac{t_0^k}
{t_0^{\primitivepolyhedronmultiplicity{j}}
+\sum_{k'=0}^{\primitivepolyhedronmultiplicity{j}-1}
\stdpolynomialcoefficient_{j,k'}^{(1)}t_0^{k'}}
\chi(\indexedvertex {\stdprimitivepolyhedronletter_j}l)(f_2(x)f_1(x)-f_1(x)f_2(x))}{f_2(x)^2}=0.
\end{multline*}
\end{proof}

\begin{corollary}\label{ksmuqdesc}
Let $1\le i\le \numberoffixedopensets$.

If $U_i$ corresponds to an essential special point $p=p_{j''}$ and a vertex $\indexedvertexpt{p}{j'}$, then
the $U_{\numberoffixedopensets}$-description of $\stdvectorfiledxletter_i$ equals
$$
\left(
-\frac{t^k}
{t^{\primitivepolyhedronmultiplicity{j}}
+\sum_{k'=0}^{\primitivepolyhedronmultiplicity{j}-1}
\stdpolynomialcoefficient_{j,k'}^{(1)}t^{k'}}
\uidegree{\numberoffixedopensets,1}(\indexedvertex {\stdprimitivepolyhedronletter_j}l),
-\frac{t^k}
{t^{\primitivepolyhedronmultiplicity{j}}
+\sum_{k'=0}^{\primitivepolyhedronmultiplicity{j}-1}
\stdpolynomialcoefficient_{j,k'}^{(1)}t^{k'}}
\uidegree{\numberoffixedopensets,2}(\indexedvertex {\stdprimitivepolyhedronletter_j}l),
0\right),
$$
where $l\in\{0,1\}$ is such that 
$\normalvertexcone{\indexedvertexpt p{j'}}{\stdpolyhedronletter_p}\subseteq 
\normalvertexcone{\indexedvertex {\stdprimitivepolyhedronletter_j}l}{\stdprimitivepolyhedronletter_j}$.

Otherwise, the $U_{\numberoffixedopensets}$-description of $\stdvectorfiledxletter_i$ is $(0,0,0)$.
\end{corollary}

\begin{proof}
This follows directly from Lemma \ref{vfieldtransition}.
\end{proof}

So, we have computed the Kodaira-Spencer map for a set of basis vectors of $\Theta_{a^{(1)}}\coefficientspace$, and therefore by linearity 
we can now compute it for an arbitrary vector from $\Theta_{a^{(1)}}\coefficientspace$.
Let us prove the surjectivity of this map.

\section{Surjectivity of the Kodaira-Spencer map}

To prove the surjectivity, we use the results of Chapter \ref{combformula}.
We will prove that the composition
$\Theta_{a^{(1)}}\coefficientspace\to T^1_0(X)\to \ker(H^0(\PP^1,\givinv)\to H^0(\PP^1,\gviiiinv))$
is surjective and that 
$\im(\Theta_{a^{(1)}}\coefficientspace\to T^1_0(X))$
contains $\im(H^1(\PP^1,\gi)\to T^1_0(X))$.

Denote by $\abstractvectorspace_{3,0}$ the space of $3(\numberoffixedopensets-1)$-tuples of the form
$$
(\uidescriptionfunctiondiff11,\uidescriptionfunctiondiff12,\stdvectorfiledpletter[1],\ldots, 
\uidescriptionfunctiondiff{\numberoffixedopensets-1}1,\uidescriptionfunctiondiff{\numberoffixedopensets-1}2,\stdvectorfiledpletter[\numberoffixedopensets-1]),
$$
where each $\uidescriptionfunctiondiff ij$ is a rational function on $\PP^1$, each $\stdvectorfiledpletter[i]$
is a rational vector field on $\PP^1$, and each triple $(\uidescriptionfunctiondiff i1,\uidescriptionfunctiondiff i2, \stdvectorfiledpletter[i])$
is the $U_i$-description of a $T$-invariant vector field defined on $U_{\numberoffixedopensets}$.
This space $\abstractvectorspace_{3,0}$ can be identified (using the notion of an $U_i$-description)
with the zeroth graded component of 
$$
\bigoplus_{i=1}^{\numberoffixedopensets-1} H^0(U_{\numberoffixedopensets},\Theta_U).
$$
Hence, we have a map
$$
\abstractvectorspace_{3,0}\to \bigoplus_{i=1}^{\numberoffixedopensets-1}(H^0(U_{\numberoffixedopensets},\Theta_U)/H^0(U_i,\Theta_U)).
$$
Denote the preimage under this map of 
$$
\ker\Bigg(\bigoplus_{i=1}^{\numberoffixedopensets-1}\Big(H^0(U_{\numberoffixedopensets},\Theta_U)/H^0(U_i,\Theta_U)\Big)\to 
\bigoplus_{1\le i<i'\le \numberoffixedopensets-1}\Big(H^0(U_{\numberoffixedopensets},\Theta_U)/H^0(U_i\cap U_{i'},\Theta_U)\Big)\Bigg)
$$
by $\abstractvectorspace_{3,1}\subseteq\abstractvectorspace_{3,0}$.
By Corollary \ref{h1computegen},
\begin{multline*}
H^1(U,\Theta_U)=
\left(\ker\Bigg(\bigoplus_{i=1}^{\numberoffixedopensets-1}\Big(H^0(U_{\numberoffixedopensets},\Theta_U)/H^0(U_i,\Theta_U)\Big)
\vphantom{\bigoplus_{1\le i<i'\le \numberoffixedopensets-1}}
\right.\!
\\
\left.\!
\left.\!
\longrightarrow 
\bigoplus_{1\le i<i'\le \numberoffixedopensets-1}\Big(H^0(U_{\numberoffixedopensets},\Theta_U)/H^0(U_i\cap U_{i'},\Theta_U)\Big)\Bigg)\right)\right/H^0(U_{\numberoffixedopensets},\Theta_U).
\end{multline*}
Therefore, we have a surjective map from $\abstractvectorspace_{3,1}$ to the zeroth graded component of $H^1(U,\Theta_U)$, 
and each element of $\abstractvectorspace_{3,1}$ can be interpreted as an element of the zeroth graded component of $H^1(U,\Theta_U)$.

Recall that if $p$ is an essential special point and $1\le j\le \numberofverticespt p$, then we have denoted by $\opensetforvertex{p,j}$
the index such that $U_{\opensetforvertex{p,j}}$ is the set among $U_i$ that corresponds to $(p,j)$.
Now we extend this notation so that we could use it also for removable special points. First, if $p$ is an essential 
special point, denote $\numberofverticesptdiff p=\numberofverticespt p$. If $p$ is a removable special point, 
denote by $\numberofverticesptdiff p$ the amount of sets $U_i$ corresponding to $p$ (there can be one or two such sets).
Recall that we enumerate the sets $U_i$ in such an order that if we have two sets $U_i$ corresponding 
to the same removable special point $p$, then they are consequent, i.~e. they are $U_i$ and $U_{i+1}$
for some $i$. Then denote this $i$ by $\opensetforvertex{p,1}$, and set $\opensetforvertex{p,2}=i+1$.
If we have only one set $U_i$ corresponding to a removable special point $p$, 
denote this $i$ by $\opensetforvertex{p,1}$.
Now we can say that in general, we enumerate the sets $U_i$ so that the sequence
$$
\opensetforvertex{p_1,1},\ldots,\opensetforvertex{p_1,\numberofverticesptdiff{p_1}},\ldots,\ldots,\ldots,
\opensetforvertex{p_{\numberofdivisorpoints},1},\ldots,\opensetforvertex{p_{\numberofdivisorpoints},\numberofverticesptdiff{p_{\numberofdivisorpoints}}}
$$
is just
$$
1,2,\ldots,\numberoffixedopensets-1.
$$

In Chapter \ref{combformula}, we also needed one coordinate function on $\PP^1$ (i.~e. a function with one zero and one pole) for each special point $p$.
This function was denoted by $t_p$ and it had its single zero at $p$. Now let us set $t_p=t-t(p)$ for all special points where $t$ is defined, 
and if $t(p)=\infty$ (then $p$ is a removable special point), then set $t_p=1/t$.

We will need one more notation. Fix a primitive polyhedron $\stdprimitivepolyhedronletter_i$ ($1\le i\le \numberofprimitivepolyhedra$).
Let $p_j$ be a special point. If $\primitivepolyhedronlocalmultiplicity{i}{j}=0$, set $\shiftstartingpoint{p_j}{i}=\numberofverticesptdiff{p_j}$.
Otherwise, $p_j$ is an essential special point, and for each vertex $\indexedvertexpt{p_j}{k}$ ($1\le k\le \numberofverticespt{p_j}$)
its normal vertex cone $\normalvertexcone{\stdpolyhedronletter_{p_j}}{\indexedvertexpt{p_j}{k}}$
is a subcone of one of two cones $\normalvertexcone{\stdprimitivepolyhedronletter_i}{\indexedvertex{\stdprimitivepolyhedronletter_i}{0}}$
or $\normalvertexcone{\stdprimitivepolyhedronletter_i}{\indexedvertex{\stdprimitivepolyhedronletter_i}{1}}$.
Moreover, the vertices of $\stdpolyhedronletter_{p_j}$ whose normal vertex cones are subcones of one of these two cones
are consequent, more precisely, the values of $j'$ such that 
$\normalvertexcone{\stdpolyhedronletter_{p_j}}{\indexedvertexpt{p_j}{k}}\subseteq
\normalvertexcone{\stdprimitivepolyhedronletter_i}{\indexedvertex{\stdprimitivepolyhedronletter_i}{0}}$
are all integers between $1$ and some $k_0$ ($1\le j_0<\numberofverticespt{p_j}$), inclusively. 
This $k_0$ is precisely the index such that $\indexededgept{p_j}{k_0}$ is the edge of $\stdpolyhedronletter_{p_j}$
parallel to the finite edge of $\stdprimitivepolyhedronletter_i$. Set $\shiftstartingpoint{i}{p_j}=k_0$.
Then $\normalvertexcone{\stdpolyhedronletter_{p_j}}{\indexedvertexpt{p_j}{k}}\subseteq
\normalvertexcone{\stdprimitivepolyhedronletter_i}{\indexedvertex{\stdprimitivepolyhedronletter_i}{0}}$
if and only if $1\le k\le \shiftstartingpoint{i}{p_j}$, 
and 
$\normalvertexcone{\stdpolyhedronletter_{p_j}}{\indexedvertexpt{p_j}{k}}\subseteq
\normalvertexcone{\stdprimitivepolyhedronletter_i}{\indexedvertex{\stdprimitivepolyhedronletter_i}{1}}$
if and only if $\shiftstartingpoint{i}{p_j}< k\le \numberofverticespt{p_j}$.

\begin{lemma}\label{multiplicityisedgelength}
If $\stdprimitivepolyhedronletter_i$ is a primitive polyhedron, 
$p_j$ is an essential special point, and $\primitivepolyhedronlocalmultiplicity{i}{j}\ne 0$, 
then $|\indexededgept{p_j}{\shiftstartingpoint{i}{p_j}}|=\primitivepolyhedronlocalmultiplicity{i}{j}$.
\end{lemma}

\begin{proof}
This follows from the definition of $\shiftstartingpoint{i}{p_j}$
and the fact that $\stdpolyhedronletter_{p_j}=\sum_i \primitivepolyhedronlocalmultiplicity{i}{j}\stdprimitivepolyhedronletter_i$.
\end{proof}

Now note that in Lemma \ref{ksmuidesc} and in Corollary \ref{ksmuqdesc}, if $l=0$, then 
$\indexedvertex{\stdprimitivepolyhedronletter_j}{l}=0$, so the $U_i$-description and 
the $U_{\numberoffixedopensets}$-description are both zero. So, the image of the Kodaira-Spencer map 
computed in the previous section can be written as follows.

\begin{lemma}
The image of the Kodaira-Spencer map for the deformation $\stdflatmorphism_{j,k}$ in $H^1(U,\Theta_U)$
is represented by the following class in $\abstractvectorspace_{3,1}$:
$$
s_{3,2,j,k}=(\uidescriptionfunctiondiff11,\uidescriptionfunctiondiff12,\stdvectorfiledpletter[1],\ldots, 
\uidescriptionfunctiondiff{\numberoffixedopensets-1}1,\uidescriptionfunctiondiff{\numberoffixedopensets-1}2,\stdvectorfiledpletter[\numberoffixedopensets-1]),
$$
where:
\begin{enumerate}
\item $\stdvectorfiledpletter[i]=0$ for all $i$ ($1\le i\le \numberoffixedopensets-1$).
\item $\uidescriptionfunctiondiff{\opensetforvertex{p,j'}}{j''}=0$ if $p$ is a special point, $1\le j'\le \shiftstartingpoint jp$, and $j''=1,2$.
\item 
$$
\uidescriptionfunctiondiff{\opensetforvertex{p,j'}}{j''}=
-\frac{t^k}
{t^{\primitivepolyhedronmultiplicity{j}}
+\sum_{k'=0}^{\primitivepolyhedronmultiplicity{j}-1}
\stdpolynomialcoefficient_{j,k'}^{(1)}t^{k'}}
\uidegree{\opensetforvertex{p,j'},j''}(\indexedvertex {\stdprimitivepolyhedronletter_j}1),
$$
if $p$ is a special point, $\shiftstartingpoint jp<j'\le\numberofverticesptdiff p$, and $j''=1,2$.\qed
\end{enumerate}
\end{lemma}
We keep the notation $s_{3,2,j,k}$ 
for further usage.


Denote the subspace of $\abstractvectorspace_{3,1}$ 
spanned by 
all $s_{3,2,j,k}$ 
for $1\le j\le \numberofprimitivepolyhedra$ and $0\le k\le \primitivepolyhedronmultiplicity{j}-1$
by $\abstractvectorspace_{3,2}$. 

Now it suffices to prove that the map $\abstractvectorspace_{3,2}\to
\ker(H^0(\PP^1,\givinv)\to H^0(\PP^1,\gviiiinv))$ is a surjection, and that 
$\im(\abstractvectorspace_{3,2}\to H^1(U,\Theta_U))$
contains $\im(H^1(\PP^1,\gi)\to H^1(U,\Theta_U))$.

Let us start with $\abstractvectorspace_{3,2}\to
\ker(H^0(\PP^1,\givinv)\to H^0(\PP^1,\gviiiinv))$.
First, we need to understand the map 
$\abstractvectorspace_{3,2}\to
H^0(\PP^1,\givinv)$.
Recall that we interpret $\givinv$ as
$$
\givinv=\left.\!\left(\ker\Bigg(\bigoplus_{i=1}^{\numberoffixedopensets-1}(\giicircinv/\giisiinv{i})\to
\bigoplus_{1\le i<j\le \numberoffixedopensets-1}(\giicircinv/\giipsijinv{i,j})\Bigg)\right)\right/\giicircinv.
$$
In particular, global sections of $\bigoplus_{i=1}^{\numberoffixedopensets-1}\giicircinv$
that project down to the appropriate kernel define global sections of $\givinv$.
Now it follows from the discussion in the end of Section \ref{sectleray} that 
the map $\abstractvectorspace_{3,2}\to
H^0(\PP^1,\givinv)$
is induced by the following map 
$\abstractvectorspace_{3,2}\to \bigoplus_{i=1}^{\numberoffixedopensets-1}\Gamma(\PP^1,\giicircinv)$:
Given
$$
(\uidescriptionfunctiondiff11,\uidescriptionfunctiondiff12,\stdvectorfiledpletter[1],\ldots, 
\uidescriptionfunctiondiff{\numberoffixedopensets-1}1,\uidescriptionfunctiondiff{\numberoffixedopensets-1}2,\stdvectorfiledpletter[\numberoffixedopensets-1])
\in \abstractvectorspace_{3,2},
$$
let 
$$
(\stdvectorfiledxletter[1],\ldots,\stdvectorfiledpletter[\numberoffixedopensets-1])\in \bigoplus_{i=1}^{\numberoffixedopensets-1}\Gamma(\PP^1,\giicircinv)
$$
be such that each $\stdvectorfiledxletter[i]$ is the vector field of degree 0 on $U_{\numberoffixedopensets}$
with $U_i$-description
$(\uidescriptionfunctiondiff i1,\uidescriptionfunctiondiff i2,\stdvectorfiledpletter[i])$.

We will follow the argument from Chapter \ref{combformula}.
There we have introduced the notion of an excessive index $i$, 
which is one of two indices corresponding to a removable special point $p$.
We checked that we can replace $\bigoplus_{i=1}^{\numberoffixedopensets-1} \giicircinv$
with 
$$
\bigoplus_{\substack{1\le i\le\numberoffixedopensets-1\\\text{$i$ is not exessive}}} \giicircinv
$$
(the morphism 
$$
\bigoplus_{\substack{1\le i\le\numberoffixedopensets-1\\\text{$i$ is not exessive}}} \giicircinv\to \bigoplus_{i=1}^{\numberoffixedopensets-1} \giicircinv
$$
duplicates the entries with the non-excessive indices $i$ corresponding to the same removable special points
to get the entries with excessive indices)
so that 
$$
\bigoplus_{\substack{1\le i\le\numberoffixedopensets-1\\\text{$i$ is not exessive}}} \giicircinv
$$
is mapped surjectively onto
$\givinv$.
Since in each element of $\abstractvectorspace_{3,2}$ 
all entries with indices (both excessive and non-excessive) corresponding to 
essential special points are zeros, we can just forget entries corresponding to 
the excessive indices to get the map
$$\abstractvectorspace_{3,2}\to \bigoplus_{\substack{1\le i\le\numberoffixedopensets-1\\\text{$i$ is not exessive}}}\Gamma(\PP^1,\giicircinv)$$
from the map
$\abstractvectorspace_{3,2}\to \bigoplus_{i=1}^{\numberoffixedopensets-1}\Gamma(\PP^1,\giicircinv)$
described above.

After that, we have split the interpretation of $\Gamma(\PP^1, \givinv)$ as
$$
\Gamma\left(\PP^1,
\left.\!\Bigg(\bigoplus_{\substack{1\le i\le \numberoffixedopensets-1\\ i\text{ is not excessive}}}(\giicircinv/\giisiinv{i})\Bigg)\right/\giicircinv
\right)
$$
into $\numberofdivisorpoints$ direct summands, each of them 
consisted of sections over $W_p$ for a special point $p$ (Corollary \ref{g4g8directsum} ad discussion below).
The morphism between these two interpretations was the restriction map for sheaves from 
$\PP^1$ to $W_p$. More important, the kernel $\ker \Gamma(\PP^1,\givinv)\to \Gamma(\PP^1,\gviiiinv)$
also gets split into $\numberofdivisorpoints$ direct summands, in other words, 
the kernel equals the sum of its intersections with each direct summand.

Then some of the summands of the double direct sum turned out to be zero, and 
we got the following interpretation of $\Gamma(\PP^1, \givinv)$ (Lemma \ref{consideronespecialpoint}):
$$
\Gamma(\PP^1,\givinv)\cong\bigoplus_{\substack{p\text{ special}\\\text{point}}}\left(\left.\!\Bigg(
\bigoplus_{\substack{1\le i\le \numberoffixedopensets-1\\ i\text{ is not excessive}\\ V_i=W_p}}\left(\Gamma(W_p,\giicircinv)/
\Gamma(W_p,\giisiinv{i\vphantom{\opensetforvertex{p,j}}})\right)\Bigg)\right/\Gamma(W_p,\giicircinv)\right).
$$
The isomorphism between these two interpretations of $\Gamma(\PP^1,\givinv)$ works as follows:
given an element $g$ of 
$$
\Gamma\left(\PP^1,
\left.\!\Bigg(\bigoplus_{\substack{1\le i\le \numberoffixedopensets-1\\ i\text{ is not excessive}}}(\giicircinv/\giisiinv{i})\Bigg)\right/\giicircinv
\right),
$$
each entry of its image in 
$$
\bigoplus_{p\text{ special point}}\left(\left.\!\Bigg(
\bigoplus_{\substack{1\le i\le \numberoffixedopensets-1\\ i\text{ is not excessive}\\ V_i=W_p}}\left(\Gamma(W_p,\giicircinv)/
\Gamma(W_p,\giisiinv{i\vphantom{\opensetforvertex{p,j}}})\right)\Bigg)\right/\Gamma(W_p,\giicircinv)\right).
$$
with index $i$ in the inner direct sum (in fact, there is only one such entry 
for each $i$, where $1\le i\le \numberoffixedopensets-1$, $i$ is not excessive)
is (locally on $\PP^1$) the $i$th entry the direct sum for $g$. 


In particular, if $g$ originates from a global section of 
$$
\bigoplus_{\substack{1\le i\le \numberoffixedopensets-1\\ i\text{ is not excessive}}}\giicircinv,
$$
then the entry of the result with index $i$ in the inner sum is the restriction of the $i$th entry 
of $g$ from $\PP^1$ to $W_p$. In fact, this restriction is a trivial operation since $\pi^{-1}(W_p)\cap U_{\numberoffixedopensets}=U_{\numberoffixedopensets}$.

Hence, the map 
$$
\abstractvectorspace_{3,2}\to \bigoplus_{p\text{ special point}}\left(\left.\!\Bigg(
\bigoplus_{\substack{1\le i\le \numberoffixedopensets-1\\ i\text{ is not excessive}\\ V_i=W_p}}\left(\Gamma(W_p,\giicircinv)/
\Gamma(W_p,\giisiinv{i\vphantom{\opensetforvertex{p,j}}})\right)\Bigg)\right/\Gamma(W_p,\giicircinv)\right).
$$
works as follows. Given
$$
(\uidescriptionfunctiondiff11,\uidescriptionfunctiondiff12,\stdvectorfiledpletter[1],\ldots, 
\uidescriptionfunctiondiff{\numberoffixedopensets-1}1,\uidescriptionfunctiondiff{\numberoffixedopensets-1}2,\stdvectorfiledpletter[\numberoffixedopensets-1])
\in \abstractvectorspace_{3,2},
$$
for each $i$ ($1\le i\le \numberoffixedopensets-1$, $i$ is not excessive), the entry of the result with index $i$ 
in the inner direct sum is the vector field on $U_{\numberoffixedopensets}$ with 
the $U_i$-description $(\uidescriptionfunctiondiff i1,\uidescriptionfunctiondiff i2,\stdvectorfiledpletter[i])$.

After that, we proved that the summands of the outer direct sum where $p$ is a removable special point
are in fact zero, and removed them, rewriting $\Gamma(\PP^1,\givinv)$ as
$$
\bigoplus_{\text{$p$ essential special point}}\left(\left.\!
\Bigg(\bigoplus_{j=1}^{\numberofverticespt p}
\left(\Gamma(W_p,\giicircinv)/
\Gamma(W_p,\giisiinv{\opensetforvertex{p,j}})\right)\Bigg)
\right/\Gamma(W_p,\giicircinv)\right).
$$
Again, the kernel $\ker \Gamma(\PP^1,\givinv)\to \Gamma(\PP^1,\gviiiinv)$
is also split in the direct sum of its intersections with each of the direct summands.
The map from $\abstractvectorspace_{3,2}$ again computes the $(p,j)$th entry out of the 
$\opensetforvertex{p,j}$th entry of an element of $\abstractvectorspace_{3,2}$ 
treated as an $U_{\opensetforvertex{p,j}}$-description.

Recall that for each essential special point $p$ and for each $j$ ($1\le j\le \numberofverticespt p$) we have denoted 
by $\gsiicirc{p,j}$ the space of triples of two regular functions and one vector field defined on $W\subset \PP^1$. 
We also have denoted by $\kappa_{\Theta,p,j}$ the map (actually, it is an isomorphism) 
$\kappa_{\Theta,p,j}\colon \gsiicirc{p,j}\to \Gamma(W_p,\giicircinv)$
that computes a vector field defined on $U_{\numberoffixedopensets}$ out if its 
$U_{\opensetforvertex{p,j}}$-description. Note also that 
$\Gamma(W_p,\giicircinv)=\Gamma(\PP^1,\giicircinv)$
since both spaces are the spaces of $T$-invariant vector fields defined on 
$U_{\numberoffixedopensets}=U_{\opensetforvertex{p,j}}\cap U_{\numberoffixedopensets}$.
The direct sum of maps $\kappa_{\Theta,p,j}$ for a fixed essential special point $p$ 
and for all $j$ ($1\le j\le \numberofverticespt p$) was denoted by $\kappa_{\Theta,p}$

%

Let us also denote 
$$
\gsiicircgen{p}=\bigoplus_{j=1}^{\numberofverticespt p} \gsiicirc{p,j}\text{ and }
\gsiicircgengen=\bigoplus_{\text{$p$ essential special point}}\gsiicircgen{p}.
$$
Denote
the direct sum of all isomorphisms 
$\kappa_{\Theta,p}$
for all essential special points $p$ by
by 
$$
\kappa_\Theta\colon 
\gsiicircgengen\to
\bigoplus_{\text{$p$ essential special point}}\bigoplus_{j=1}^{\numberofverticespt p} \Gamma(\PP^1,\giicircinv).
$$
Using this isomorphism, we can say that 
we have a map 
$\abstractvectorspace_{3,2}\to \gsiicircgengen$,
and this map works as follows:
the $(p,j)$th entry of the image of an element $g$ of $\abstractvectorspace_{3,2}$ is the $\opensetforvertex{p,j}$th triple 
(i.~e. the $(3\opensetforvertex{p,j}-2)$th, the $(3\opensetforvertex{p,j}-1)$th, and the $3\opensetforvertex{p,j}$th entries) of $g$.

To proceed, we will need a more convenient basis for $\abstractvectorspace_{3,2}$.
Recall that for each primitive polyhedron $\stdprimitivepolyhedronletter_i$ ($1\le i\le\numberofprimitivepolyhedra$),
$$
t^{\primitivepolyhedronmultiplicity{i}}
+\sum_{k=0}^{\primitivepolyhedronmultiplicity{i}-1}
\stdpolynomialcoefficient_{i,k}^{(1)}t^{k}=
\prod_{j=1}^{\numberofdivisorpoints}(t-t(p_j))^{\primitivepolyhedronlocalmultiplicity{i}{j}}.
$$

\begin{lemma}\label{partialfraction}
Fix a primitive polyhedron $\stdprimitivepolyhedronletter_i$ ($1\le i\le\numberofprimitivepolyhedra$)
and denote temporarily $f=t^{\primitivepolyhedronmultiplicity{i}}
+\sum_{k'=0}^{\primitivepolyhedronmultiplicity{i}-1}
\stdpolynomialcoefficient_{i,k}^{(1)}t^{k}=
\prod_{j=1}^{\numberofdivisorpoints}(t-t(p_j))^{\primitivepolyhedronlocalmultiplicity{i}{j}}$
Consider the rational functions 
$$
\frac{t^{\primitivepolyhedronmultiplicity{i}}
+\sum_{k=0}^{\primitivepolyhedronmultiplicity{i}-1}
\stdpolynomialcoefficient_{i,k'}^{(1)}t^{k'}}{(t-t(p_j))^{k}}
$$
for all $j$ and $k$ such that $1\le j\le \numberofdivisorpoints$ and 
$1\le k\le \primitivepolyhedronlocalmultiplicity{i}{j}$.

All these functions together span the same subspace in rational functions on $\PP^1$ as 
$1,t,\ldots, t^{\primitivepolyhedronmultiplicity{i}-1}$.
\end{lemma}

\begin{proof}
By partial fraction decomposition theorem, the functions
$$
\frac1{(t-t(p_j))^{k}}
$$
for all $j$ and $k$ such that $1\le j\le \numberofdivisorpoints$ and 
$1\le k\le \primitivepolyhedronlocalmultiplicity{i}{j}$
form a basis of all rational functions of the form
$$
\frac f{\prod_{j=1}^{\numberofdivisorpoints}(t-t(p_j))^{\primitivepolyhedronlocalmultiplicity{i}{j}}}=
\frac f
{t^{\primitivepolyhedronmultiplicity{i}}
+\sum_{k=0}^{\primitivepolyhedronmultiplicity{i}-1}
\stdpolynomialcoefficient_{i,k}^{(1)}t^{k}},
$$
where $f$ is a polynomial in $t$ of degree at most $-1+\sum_{j=1}^{\numberofprimitivepolyhedra} \primitivepolyhedronlocalmultiplicity{i}{j}=
\primitivepolyhedronmultiplicity{i}-1$.
After the multiplication by 
$t^{\primitivepolyhedronmultiplicity{i}}
+\sum_{k=0}^{\primitivepolyhedronmultiplicity{i}-1}
\stdpolynomialcoefficient_{i,k}^{(1)}t^{k}$,
we get the claim of the lemma.
\end{proof}

For each $j$ ($1\le j\le \numberofprimitivepolyhedra$) and for each pair $(j',k)$, where $1\le j'\le \numberofdivisorpoints$
and $1\le k\le \primitivepolyhedronlocalmultiplicity{j}{j'}$ denote 
$$
s'_{3,2,j,j',k}=(\uidescriptionfunctiondiff11,\uidescriptionfunctiondiff12,\stdvectorfiledpletter[1],\ldots, 
\uidescriptionfunctiondiff{\numberoffixedopensets-1}1,\uidescriptionfunctiondiff{\numberoffixedopensets-1}2,\stdvectorfiledpletter[\numberoffixedopensets-1]),
$$
where:
\begin{enumerate}
\item $\stdvectorfiledpletter[i]=0$ for all $i$.
\item $\uidescriptionfunctiondiff{\opensetforvertex{p,j''}}1=\uidescriptionfunctiondiff{\opensetforvertex{p,j''}}2=0$
if $p$ is a special point and $1\le j''\le \shiftstartingpoint jp$.
\item 
$$
\uidescriptionfunctiondiff{\opensetforvertex{p,j''}}{j'''}=
-\frac{1}
{(t-t(p_{j'}))^k}
\uidegree{\opensetforvertex{p,j''},j'''}(\indexedvertex {\stdprimitivepolyhedronletter_j}1),
$$
if $p$ is a special point, $\shiftstartingpoint jp<j''\le\numberofverticesptdiff p$, and $j'''=1,2$.
\end{enumerate}

\begin{lemma}\label{s3pspans}
$\abstractvectorspace_{3,2}$ is spanned by all functions $s'_{3,2,j,j',k}$, where 
$1\le j\le \numberofprimitivepolyhedra$, 
$1\le j'\le \numberofdivisorpoints$,
and $1\le k\le \primitivepolyhedronlocalmultiplicity{j}{j'}$.
\end{lemma}

\begin{proof}
By Lemma \ref{partialfraction}, we may replace the numerators of functions in $s_{3,2,j,k}$
from $1,t,\ldots, t^{\primitivepolyhedronmultiplicity{j}}$
to the functions from the statement of Lemma \ref{partialfraction}. After doing this, we get exactly
$s_{3,2,j,j',k}$.
\end{proof}

%
%
%


Now let us recall that for each essential special point $p$ we had a vector space $\abstractvectorspace_{1,2,p}$
and a map $\rho_p\colon \abstractvectorspace_{1,2,p}\to \gsiicircgen{p}$
(actually, it maps $\abstractvectorspace_{1,2,p}$ to a subspace of $\gsiicircgen{p}$, which 
was called $\abstractvectorspace_{1,0,p}$)
such that the composition $\abstractvectorspace_{1,2,p}\to \gsiicircgen{p}\to \Gamma(\PP^1,\givinv)$
maps $\abstractvectorspace_{1,2,p}$ surjectively onto the graded component corresponding to $p$ of the kernel 
$\ker (\Gamma(\PP^1,\givinv)\to \Gamma(\PP^1,\gviiiinv))$ 
(see Proposition \ref{avs22p} and Remark \ref{commdiagremark}).

So, denote $\abstractvectorspace_{1,2}=\bigoplus_{\text{$p$ essential special point}}\abstractvectorspace_{1,2,p}$,
and denote the direct sum of all maps $\rho_p$ by $\rho$. Then 
the composition 
$\abstractvectorspace_{1,2}\to 
\gsiicircgengen\to 
\Gamma(\PP^1,\givinv)$
maps $\abstractvectorspace_{1,2}$ surjectively onto 
$\ker (\Gamma(\PP^1,\givinv)\to \Gamma(\PP^1,\gviiiinv))$.
By the definition of the maps $\rho_p$, $\rho$ actually works as follows:
it computes $U_{\opensetforvertex{p,j}}$-descriptions out of $U_{\numberoffixedopensets}$-descriptions 
and adds some zeros.

It would be sufficient to prove that the image of $\abstractvectorspace_{3,2}$ in 
$\gsiicircgengen$
contains $\rho(\abstractvectorspace_{3,2})$, but this is not true in general.
Instead, we will construct another vector space $\abstractvectorspace_{3,3}$, whose elements 
will define the same classes in $\Gamma(\PP^1,\givinv)$ (and even in $H^1(U,\Theta_U)$) as elements of 
$\abstractvectorspace_{3,2}$, and whose image in 
$\gsiicircgengen$
will contain $\abstractvectorspace_{1,2}$.

Namely, $\abstractvectorspace_{3,3}\subseteq \abstractvectorspace_{3,0}$ will consist of some of the $3(\numberoffixedopensets-1)$-tuples
of the form
$$
(\uidescriptionfunctiondiff11,\uidescriptionfunctiondiff12,\stdvectorfiledpletter[1],\ldots, 
\uidescriptionfunctiondiff{\numberoffixedopensets-1}1,\uidescriptionfunctiondiff{\numberoffixedopensets-1}2,\stdvectorfiledpletter[\numberoffixedopensets-1]),
$$
where each $\uidescriptionfunctiondiff ij$ is a regular function on $W\subset \PP^1$, each $\stdvectorfiledpletter[i]$
is a vector field on $W\subset \PP^1$. We will take only some of these $3(\numberoffixedopensets-1)$-tuples,
not all of them. More precisely, $\abstractvectorspace_{3,3}$ will be spanned by the following elements 
$s_{3,3,j,j',k}$, where 
$1\le j\le \numberofprimitivepolyhedra$, $1\le j'\le \numberofdivisorpoints$,
and $1\le k\le \primitivepolyhedronlocalmultiplicity{j}{j'}$:
$$
s_{3,3,j,j',k}=(\uidescriptionfunctiondiff11,\uidescriptionfunctiondiff12,\stdvectorfiledpletter[1],\ldots, 
\uidescriptionfunctiondiff{\numberoffixedopensets-1}1,\uidescriptionfunctiondiff{\numberoffixedopensets-1}2,\stdvectorfiledpletter[\numberoffixedopensets-1]),
$$
where:
\begin{enumerate}
\item $\stdvectorfiledpletter[i]=0$ for all $i$.
\item 
\begin{enumerate}
\item If $i=\opensetforvertex{p_{j'},j''}$, where $\shiftstartingpoint j{p_{j'}}<j''\le \numberofverticesptdiff{p_{j'}}$, then
$$
\uidescriptionfunctiondiff{i}{j'''}=
-\frac{1}
{(t-t(p_{j'}))^k}
\uidegree{i,j'''}(\indexedvertex {\stdprimitivepolyhedronletter_j}1),
$$
for $j'''=1,2$.
\item 
Otherwise (if $i$ is not of this form),
$\uidescriptionfunctiondiff{i}1=\uidescriptionfunctiondiff{i}2=0$.
\end{enumerate}
\end{enumerate}

\begin{lemma}
For each $j$, $j'$, and $k$ ($1\le j\le \numberofprimitivepolyhedra$, $1\le j'\le \numberofdivisorpoints$,
and $1\le k\le \primitivepolyhedronlocalmultiplicity{j}{j'}$),
$s'_{3,2,j,j',k}\in\abstractvectorspace_{3,2}$ and $s_{3,3,j,j',k}\in\abstractvectorspace_{3,3}$
define the same class in 
$\bigoplus_{i=1}^{\numberofabstractopensets-1}(\Gamma(U_{\numberoffixedopensets},\Theta_U)/\Gamma(U_i,\Theta_U))$.
\end{lemma}

\begin{proof}
Write
$$
s'_{3,2,j,j',k}=(\uidescriptionfunctiondiff11,\uidescriptionfunctiondiff12,\stdvectorfiledpletter[1],\ldots, 
\uidescriptionfunctiondiff{\numberoffixedopensets-1}1,\uidescriptionfunctiondiff{\numberoffixedopensets-1}2,\stdvectorfiledpletter[\numberoffixedopensets-1]),
$$
$$
s_{3,3,j,j',k}=(\uidescriptionfunctiondiff11',\uidescriptionfunctiondiff12',\stdvectorfiledpletter[1]',\ldots, 
\uidescriptionfunctiondiff{\numberoffixedopensets-1}1',\uidescriptionfunctiondiff{\numberoffixedopensets-1}2',\stdvectorfiledpletter[\numberoffixedopensets-1]').
$$

Choose an index $i$ such that the $i$th entries of the images of 
$s'_{3,2,j,j',k}$ and of $s_{3,3,j,j',k}$ in 
$\bigoplus_{i=1}^{\numberofabstractopensets-1}\Gamma(U_{\numberoffixedopensets},\Theta_U)$
differ. First, note that $i$ corresponds to an essential special point $p$, otherwise 
both entries are zeros since $\shiftstartingpoint jp=\numberofverticesptdiff p$
for removable special points $p$.

So, $p$ is an essential special point and there exists an index $j''$ ($1\le j''\le\numberofverticespt p$)
such that $i=\opensetforvertex{p,j''}$.
such that the $(p,j'')$th entries of the images of 

Now the fact that 
the $\opensetforvertex{p,j''}$th entries of the images of 
$s'_{3,2,j,j',k}$ and of $s_{3,3,j,j',k}$ in 
$\bigoplus_{i=1}^{\numberofabstractopensets-1}\Gamma(U_{\numberoffixedopensets},\Theta_U)$
differ
means that $p\ne p_{j'}$, $\shiftstartingpoint jp<j''\le\numberofverticesptdiff p$, 
and the different entries 
have the following $U_i$-descriptions:
$$
(\uidescriptionfunctiondiff{\opensetforvertex{p,j''}}1,\uidescriptionfunctiondiff{\opensetforvertex{p,j''}}2,\stdvectorfiledpletter[\opensetforvertex{p,j''}])
=\left(-\frac{1}{(t-t(p_{j'}))^k}
\uidegree{\opensetforvertex{p,j''},1}(\indexedvertex {\stdprimitivepolyhedronletter_j}1),
-\frac{1}{(t-t(p_{j'}))^k}
\uidegree{\opensetforvertex{p,j''},2}(\indexedvertex {\stdprimitivepolyhedronletter_j}1),
0\right)
$$
and
$$
(\uidescriptionfunctiondiff{\opensetforvertex{p,j''}}1',\uidescriptionfunctiondiff{\opensetforvertex{p,j''}}2',\stdvectorfiledpletter[\opensetforvertex{p,j''}]')
=(0,0,0).
$$
Since $p\ne p_{j'}$, the functions 
$\uidescriptionfunctiondiff{\opensetforvertex{p,j''}}1$ and $\uidescriptionfunctiondiff{\opensetforvertex{p,j''}}2$
are defined on $W_p$, 
so 
the vector fields with these descriptions are defined on $U_i$, and 
$s'_{3,2,j,j',k}\in\abstractvectorspace_{3,2}$ and $s_{3,3,j,j',k}\in\abstractvectorspace_{3,3}$
define the same class in 
$\bigoplus_{i=1}^{\numberofabstractopensets-1}(\Gamma(U_{\numberoffixedopensets},\Theta_U)/\Gamma(U_i,\Theta_U))$.
\end{proof}

\begin{corollary}
$\abstractvectorspace_{3,3}$ is contained in $\abstractvectorspace_{3,1}$ and therefore defines a subspace of 
$H^1(U,\Theta_U)$. Moreover, images of $\abstractvectorspace_{3,2}$ and of $\abstractvectorspace_{3,3}$
in $H^1(U,\Theta_U)$ are the same.\qed
\end{corollary}

The map $\abstractvectorspace_{3,3}\to \gsiicircgengen$ is defined in the same way as the map 
$\abstractvectorspace_{3,2}\to \gsiicircgengen$: Given
$$
g=(\uidescriptionfunctiondiff11,\uidescriptionfunctiondiff12,\stdvectorfiledpletter[1],\ldots, 
\uidescriptionfunctiondiff{\numberoffixedopensets-1}1,\uidescriptionfunctiondiff{\numberoffixedopensets-1}2,\stdvectorfiledpletter[\numberoffixedopensets-1])\in\abstractvectorspace_{3,3},
$$
if $p$ is an essential special point and $1\le j\le\numberofverticespt p$, then the $(p,j)$th entry of the image of $g$ 
is $(\uidescriptionfunctiondiff{\opensetforvertex{p,j}}1,\uidescriptionfunctiondiff{\opensetforvertex{p,j}}2,\stdvectorfiledpletter[\opensetforvertex{p,j}])$.

\begin{corollary}
The images of $\abstractvectorspace_{3,2}$ and of $\abstractvectorspace_{3,3}$ in $\gsiicircgengen$
define the same subspace of $\Gamma(\PP^1,\givinv)$.\qed
\end{corollary}

We will need a slightly different set of generators for $\abstractvectorspace_{3,3}$.
If $p_{j'}$ is an essential special point, $1\le l<\numberofverticespt{p_{j'}}$, and 
$1\le k\le |\indexededgept{p_{j'}}{l}|$, 
denote
$$
s'_{3,3,l,j',k}=(\uidescriptionfunctiondiff11,\uidescriptionfunctiondiff12,\stdvectorfiledpletter[1],\ldots, 
\uidescriptionfunctiondiff{\numberoffixedopensets-1}1,\uidescriptionfunctiondiff{\numberoffixedopensets-1}2,\stdvectorfiledpletter[\numberoffixedopensets-1]),
$$
where:
\begin{enumerate}
\item $\stdvectorfiledpletter[i]=0$ for all $i$.
\item 
\begin{enumerate}
\item If $i=\opensetforvertex{p_{j'},j''}$, where $l<j''\le \numberofverticespt{p_{j'}}$, then
$$
\uidescriptionfunctiondiff{i}{j'''}=
-\frac{1}
{(t-t(p_{j'}))^k}
\uidegree{i,j'''}(\indexedvertexpt{p_{j'}}{l+1}-\indexedvertexpt{p_{j'}}{l}),
$$
for $j'''=1,2$.
\item 
Otherwise (if $i$ is not of this form),
$\uidescriptionfunctiondiff{i}1=\uidescriptionfunctiondiff{i}2=0$.
\end{enumerate}
\end{enumerate}
(we do not claim a priori that $s'_{3,3,l,j',k}\in\abstractvectorspace_{3,3}$).

\begin{lemma}
If $p_{j'}$ is an essential special point, $1\le j\le \numberofprimitivepolyhedra$, and
$1\le k\le \primitivepolyhedronlocalmultiplicity{j}{j'}$, then
$|\indexededgept{p_{j'}}{\shiftstartingpoint{j}{p_{j'}}}|s_{3,3,j,j',k}=s'_{3,3,\shiftstartingpoint{j}{p_{j'}},j',k}$.

Moreover, all $s'_{3,3,l,j',k}$ (for all essential special points $p_{j'}$, 
$1\le l<\numberofverticespt{p_{j'}}$, and 
$1\le k\le |\indexededgept{p_{j'}}{l}|$) are elements of $\abstractvectorspace_{3,3}$ and generate 
$\abstractvectorspace_{3,3}$.
\end{lemma}

\begin{proof}
The first claim follows from the following equality:
$$
\indexedvertexpt {p_{j'}}{\shiftstartingpoint{j}{p_{j'}}+1}-\indexedvertexpt {p_{j'}}{\shiftstartingpoint{j}{p_{j'}}}=
|\indexededgept{p_{j'}}{\shiftstartingpoint{j}{p_{j'}}}|\indexedvertex {\stdprimitivepolyhedronletter_j}1.
$$
This is true because $\indexededgept{p_{j'}}{\shiftstartingpoint{j}{p_{j'}}}$
is the edge of $\stdpolyhedronletter_{p_{j'}}$ parallel to 
$\indexededge {\stdprimitivepolyhedronletter_j}1$, 
$\stdpolyhedronletter_{p_{j'}}=\sum_i \primitivepolyhedronlocalmultiplicity{i}{j'}\stdprimitivepolyhedronletter_i$, 
$|\indexededgept{p_{j'}}{\shiftstartingpoint{j}{p_{j'}}}|=\primitivepolyhedronlocalmultiplicity{j}{j'}$ (Lemma \ref{multiplicityisedgelength}), 
and $\indexedvertex {\stdprimitivepolyhedronletter_j}0=0$.


To check the second claim, we need to do the following. For each essential special point $p_{j'}$
and for each $l$ ($1\le l<\numberofverticespt{p_{j'}}$)
we have to check that there exists a primitive polyhedron $\stdprimitivepolyhedronletter_j$ 
($1\le j\le \numberofprimitivepolyhedra$)
such that $l=\shiftstartingpoint{j}{p_{j'}}$. But this follows from the equality
$\stdpolyhedronletter_{p_{j'}}=\sum_j \primitivepolyhedronlocalmultiplicity{j}{j'}\stdprimitivepolyhedronletter_j$
and the fact that if $\primitivepolyhedronlocalmultiplicity{j}{j'}\ne 0$, then 
$\indexededgept{p_{j'}}{\shiftstartingpoint{j}{p_{j'}}}$
is the edge of $\stdpolyhedronletter_{p_{j'}}$
parallel to $\indexededge{\stdprimitivepolyhedronletter_j}1$.

To finish the proof, note that for removable special points $p_{j'}$, 
$\primitivepolyhedronlocalmultiplicity{j}{j'}=0$ for all $j$, 
so it is not possible to take $k$ so that $1\le k\le \primitivepolyhedronlocalmultiplicity{j}{j'}$, 
so there are no generators $s_{3,3,j,j',k}$, where $p_{j'}$ is a removable special point.
\end{proof}


\begin{remark}
For each $l$, $j$, and $k$ ($p_j$ is an essential special point, $1\le l<\numberofverticespt{p_j}$,
and $1\le k\le |\indexededgept{p_j}l|$),
the image of $s'_{3,3,l,j,k}$ in $\gsiicircgengen$ actually belongs to $\gsiicircgen{p_j}\subseteq \gsiicircgengen$.
\end{remark}

The space $\abstractvectorspace_{1,2}$ consists of functions, which are interpreted as $U_{\numberoffixedopensets}$-descriptions
of vector fields. The functions on $\abstractvectorspace_{3,3}$ are interpreted as $U_i$-descriptions
for different values of $i$. To work with $\abstractvectorspace_{1,2}$ easier, let us construct 
another space $\abstractvectorspace_{3,4}$, whose elements will be interpreted as $U_{\numberoffixedopensets}$-descriptions.
So, by definition $\abstractvectorspace_{3,4}$ 
consists of some of the $3(\numberoffixedopensets-1)$-tuples of the form
$$
(\uidescriptionfunctiondiff11,\uidescriptionfunctiondiff12,\stdvectorfiledpletter[1],\ldots, 
\uidescriptionfunctiondiff{\numberoffixedopensets-1}1,\uidescriptionfunctiondiff{\numberoffixedopensets-1}2,\stdvectorfiledpletter[\numberoffixedopensets-1]),
$$
where each $\uidescriptionfunctiondiff ij$ is a regular function on $W\subset \PP^1$, each $\stdvectorfiledpletter[i]$
is a vector field on $W\subset \PP^1$.
More precisely, $\abstractvectorspace_{3,4}$
is spanned by $3(\numberoffixedopensets-1)$-tuples called
$s_{3,4,l,j,k}$, where 
$p_j$ is an essential special point, 
$1\le l\le \numberofverticespt{p_j}$,
$1\le k\le |\indexededgept{p_j}l|$.
By definition,
$$
s_{3,4,l,j,k}=
(\uidescriptionfunctiondiff11,\uidescriptionfunctiondiff12,\stdvectorfiledpletter[1],\ldots, 
\uidescriptionfunctiondiff{\numberoffixedopensets-1}1,\uidescriptionfunctiondiff{\numberoffixedopensets-1}2,\stdvectorfiledpletter[\numberoffixedopensets-1]),
$$
where:
\begin{enumerate}
\item $\stdvectorfiledpletter[i]=0$ for all $i$.
\item 
\begin{enumerate}
\item If $i=\opensetforvertex{p_j,j'}$, where $l<j'\le \numberofverticespt{p_j'}$, then
$$
\uidescriptionfunctiondiff{i}{j''}=
-\frac{1}
{(t-t(p_{j'}))^k}
\uidegree{\numberoffixedopensets,j''}(\indexedvertexpt{p_j}{l+1}-\indexedvertexpt{p_j}{l}),
$$
for $j''=1,2$.
\item 
Otherwise (if $i$ is not of this form),
$\uidescriptionfunctiondiff{i}1=\uidescriptionfunctiondiff{i}2=0$.
\end{enumerate}
\end{enumerate}
Clearly, $\abstractvectorspace_{3,4}$ is isomorphic to $\abstractvectorspace_{3,3}$, 
and the isomorphism computes the $(3i-2)$th, $(3i-1)$th, and $3i$th entries of an element of 
$\abstractvectorspace_{3,3}$
as the $U_i$-description of the vector field with the $U_{\numberoffixedopensets}$-description
consisting of the $(3i-2)$th, $(3i-1)$th, and $3i$th entries of the corresponding 
element of $\abstractvectorspace_{3,4}$.

The resulting morphism $\abstractvectorspace_{3,4}\to \gsiicircgengen$ works as follows. Given $g\in \abstractvectorspace_{3,4}$, 
the $(p,j)$th entry of its image in $\gsiicircgengen$ is the $U_{\opensetforvertex{p,j}}$-description 
of the vector field on $U_{\numberoffixedopensets}$ with the 
$U_{\numberoffixedopensets}$-description formed by the 
$(3\opensetforvertex{p,j}-2)$th, the $(3\opensetforvertex{p,j}-1)$th, and the $3\opensetforvertex{p,j}$th 
entries of $g$.

Recall that $\rho\colon \abstractvectorspace_{1,2}\to \gsiicircgengen$ works in a similar way:
the $(p,j)$th entry of the image is the 
$U_{\opensetforvertex{p,j}}$-description computed from the 
$U_{\numberoffixedopensets}$-description formed by the $(p,2j-1)$th and the $(p,2j)$th entries of an element of 
$\abstractvectorspace_{1,2}$ and 0. 
So, if we construct a morphism that reorders entries of am element of $\abstractvectorspace_{3,4}$ in the appropriate way 
(and removes some zeros)
and maps $\abstractvectorspace_{3,4}$ to $\abstractvectorspace_{1,2}$, 
we will factor the map $\abstractvectorspace_{3,4}\to \gsiicircgengen$ through $\abstractvectorspace_{1,2}$.

\begin{lemma}\label{avs33to22goescorrectly}
Let 
$$
g=(\uidescriptionfunctiondiff11,\uidescriptionfunctiondiff12,\stdvectorfiledpletter[1],\ldots, 
\uidescriptionfunctiondiff{\numberoffixedopensets-1}1,\uidescriptionfunctiondiff{\numberoffixedopensets-1}2,\stdvectorfiledpletter[\numberoffixedopensets-1])
\in\abstractvectorspace_{3,4}.
$$
Fix a special point $p$ and 
consider the following sequences, which we will denote by $\rho_p'(g)$:
$$
\rho_p'(g)=(\uidescriptionfunctiondiff{\opensetforvertex{p,1}}1,\uidescriptionfunctiondiff{\opensetforvertex{p,1}}2,\ldots,
\uidescriptionfunctiondiff{\opensetforvertex{p,\numberofverticespt{p}}}1,\uidescriptionfunctiondiff{\opensetforvertex{p,\numberofverticespt{p}}}2).
$$
Then $\rho_p'(g)\in\abstractvectorspace_{1,2,p}$.
\end{lemma}

\begin{proof}
Without loss of generality, $g=s_{3,4,l,j,k}$
for some essential special point $p_j$, $1\le l\le \numberofverticespt{p_j}$, 
$1\le k\le |\indexededgept{p_j}l|$.
Moreover, if this is true, and $p\ne p_{j'}$, then 
$\rho_p'(g)$ is the zero sequence by the definition of 
$s_{3,4,l,j,k}$. So suppose that $p=p_{j'}$.
Then 
$$
\rho_p'(g)=(\underbrace{0,0,\ldots,0,0}_{\text{$2l$ zeros}},
\underbrace{g'_1,g'_2,g'_1,g'_2,\ldots,g'_1,g'_2}_{\text{$2(\numberofverticespt p-l)$ entries}}),
$$
where
$$
g'_1=
-\frac{1}
{(t-t(p_{j'}))^k}
\uidegree{\numberoffixedopensets,1}(\indexedvertexpt{p_j}{l+1}-\indexedvertexpt{p_j}{l}),
g'_2=
-\frac{1}
{(t-t(p_{j'}))^k}
\uidegree{\numberoffixedopensets,2}(\indexedvertexpt{p_j}{l+1}-\indexedvertexpt{p_j}{l}).
$$

Let us check the conditions in the definition of $\abstractvectorspace_{1,2,p}$ one by one.
Condition \ref{proplaurent} is satisfied by the choice of functions $t_p$ in this section.
Condition \ref{propfirstzero} is satisfied since 
$l\ge 1$.
Condition \ref{proplinfunczero} is only nontrivial
for the edge $\indexededgept{p}{l}$. For this edge, 
it suffices to check that 
$$
\uidegree{\numberoffixedopensets,1}^*(\primitivelattice{\normalvertexcone{\indexededgept{p}{l}}{\stdpolyhedronletter_p}})
\uidegree{\numberoffixedopensets,1}(\indexedvertexpt{p_j}{l+1}-\indexedvertexpt{p_j}{l})+
\uidegree{\numberoffixedopensets,2}^*(\primitivelattice{\normalvertexcone{\indexededgept{p}{l}}{\stdpolyhedronletter_p}})
\uidegree{\numberoffixedopensets,2}(\indexedvertexpt{p_j}{l+1}-\indexedvertexpt{p_j}{l})=0.
$$
The expression at the left is the way of writing in coordinates of 
$\primitivelattice{\normalvertexcone{\indexededgept{p}{l}}{\stdpolyhedronletter_p}}(\indexedvertexpt{p_j}{l+1}-\indexedvertexpt{p_j}{l})$.
And by a property of the normal cone of an edge of a polyhedron, if we shift an argument of 
$\primitivelattice{\normalvertexcone{\indexededgept{p}{l}}{\stdpolyhedronletter_p}}$
along the edge, the value will not change. So, 
$\primitivelattice{\normalvertexcone{\indexededgept{p}{l}}{\stdpolyhedronletter_p}}(\indexedvertexpt{p_j}{l+1})=
\primitivelattice{\normalvertexcone{\indexededgept{p}{l}}{\stdpolyhedronletter_p}}(\indexedvertexpt{p_j}{l})$, 
and 
$$
\uidegree{\numberoffixedopensets,1}^*(\primitivelattice{\normalvertexcone{\indexededgept{p}{l}}{\stdpolyhedronletter_p}})
\uidegree{\numberoffixedopensets,1}(\indexedvertexpt{p_j}{l+1}-\indexedvertexpt{p_j}{l})+
\uidegree{\numberoffixedopensets,2}^*(\primitivelattice{\normalvertexcone{\indexededgept{p}{l}}{\stdpolyhedronletter_p}})
\uidegree{\numberoffixedopensets,2}(\indexedvertexpt{p_j}{l+1}-\indexedvertexpt{p_j}{l})=0.
$$
Finally, Condition \ref{proplinfuncestimate} is again nontrivial only for the edge $\indexededgept{p}{l}$, 
and for this edge we have $k\le |\indexededgept{p}{l}|$.
\end{proof}

So, we have defined a map $\rho_p'\colon \abstractvectorspace_{3,4}\to \abstractvectorspace_{1,2,p}$.
Now, if $g\in \abstractvectorspace_{3,4}$, denote 
$$
\rho'(g)=(\rho'_p(g))_{\text{$p$ essential special point}}.
$$
Then we have a map $\rho'\colon \abstractvectorspace_{3,4}\to\abstractvectorspace_{1,2}$, and 
it follows directly from the definitions of these maps that the map 
$\abstractvectorspace_{3,4}\to \gsiicircgengen$ we have is the composition of $\rho'$ and $\rho$.
Now it suffices to check that $\rho'$ is surjective.

If $p_j$ is an essential special point, $1\le l<\numberofverticespt{p_j}$, and 
$1\le k\le |\indexededgept{p_j}{l}|$, denote the following element of $\abstractvectorspace_{1,2,p_j}$
by $s_{2,l,j,k}$:
$$
s_{2,l,j,k}=
(\underbrace{0,0,\ldots,0,0}_{\text{$2l$ zeros}},
\underbrace{g_1,g_2,g_1,g_2,\ldots,g_1,g_2}_{\text{$2(\numberofverticespt p-l)$ entries}}),
$$
where
$$
g_1=
-\frac{1}
{(t-t(p_j))^k}
\uidegree{\numberoffixedopensets,1}(\indexedvertexpt {p_j}{l+1}-\indexedvertexpt {p_j}l),
g_2=
-\frac{1}
{(t-t(p_j))^k}
\uidegree{\numberoffixedopensets,2}(\indexedvertexpt {p_j}{l+1}-\indexedvertexpt {p_j}l).
$$

\begin{corollary}
If $p_j$ is an essential special point, $1\le l<\numberofverticespt{p_j}$, and
$1\le k\le |\indexededgept{p_j}{l}|$, then
$\rho'_{p_j}(s_{3,4,l,j,k})=s_{2,l,j,k}$, 
and $\rho'_p(s_{3,4,l,j,k})=0$ if $p\ne p_j$.\qed
\end{corollary}

\begin{lemma}
Let $p_j$ be an essential special point.
Then all $s_{2,l,j,k}$ for all possible $l$ and $k$ 
($1\le l< \numberofverticespt{p_j}$, and 
$1\le k\le |\indexededgept{p_j}{l}|$)
span (and even form a basis of)
$\abstractvectorspace_{1,2,p_j}$.
\end{lemma}

\begin{proof}
Clearly, all these sequences $s_{2,l,j,k}$
are nonzero and linearly independent.
The amount of them is 
$|\indexededgept{p_j}1|+\ldots+|\indexededgept{p_j}{\numberofverticespt{p_j}-1}|=\dim \abstractvectorspace_{1,2,p_j}$ (Remark \ref{avs22pdim}).
\end{proof}

\begin{corollary}
$\rho'\colon \abstractvectorspace_{3,4}\to \abstractvectorspace_{1,2}$ is surjective.\qed
\end{corollary}

Finally, we get the following proposition:
\begin{proposition}\label{secondsurjectivity}
The map $\abstractvectorspace_{3,2}\to 
\ker(H^0(\PP^1,\givinv)\to H^0(\PP^1,\gviiiinv))$
is surjective.\qed
\end{proposition}

Now we continue with $\im(\abstractvectorspace_{3,2}\to H^1(U,\Theta_U))$. We will have to prove that 
it contains $\im(H^1(\PP^1,\gi)\to H^1(U,\Theta_U))$. Let us start with the following lemma.

\begin{lemma}\label{recalcregularity}
Let $p$ be an essential special point, $1\le j_1\le\numberofverticespt p$, $1\le j_2\le\numberofverticespt p$.
Let $\stdvectorfiledxletter$ be the vector field on $U_\numberoffixedopensets$
with $U_{\opensetforvertex{p,j_1}}$-description $(0,0,\stdvectorfiledpletter)$, 
where $\stdvectorfiledpletter=\partial/\partial t$ (recall that $t$ is defined at all \emph{essential}
special points).

Let $(\uidescriptionfunction{1},\uidescriptionfunction{2},\stdvectorfiledpletter)$ be the 
$U_{\opensetforvertex{p,j_2}}$-description of $\stdvectorfiledxletter$.
Then 
$$
\uidescriptionfunction{1}-\uidegree{\opensetforvertex{p,j_2},1}(\indexedvertexpt{p}{j_1}-\indexedvertexpt{p}{j_2})\frac1{t-t(p)}
$$
and
$$
\uidescriptionfunction{2}-\uidegree{\opensetforvertex{p,j_2},2}(\indexedvertexpt{p}{j_1}-\indexedvertexpt{p}{j_2})\frac1{t-t(p)}
$$
are rational functions on $\PP^1$ regular at $p$.
\end{lemma}

\begin{proof}
By Lemma \ref{vfieldtransition},
$$
\uidescriptionfunction i=
\frac
{\oluithreadfunction{\opensetforvertex{p,j_1},1}^{\uidegree{\opensetforvertex{p,j_1},1}^*(\uidegree{\opensetforvertex{p,j_2},i})}
\oluithreadfunction{\opensetforvertex{p,j_1},2}^{\uidegree{\opensetforvertex{p,j_1},2}^*(\uidegree{\opensetforvertex{p,j_2},i})}}
{\oluithreadfunction{\opensetforvertex{p,j_2},i}} 
d\left(\frac
{\oluithreadfunction{\opensetforvertex{p,j_2},i}}
{\oluithreadfunction{\opensetforvertex{p,j_1},1}^{\uidegree{\opensetforvertex{p,j_1},1}^*(\uidegree{\opensetforvertex{p,j_2},i})}
\oluithreadfunction{\opensetforvertex{p,j_1},2}^{\uidegree{\opensetforvertex{p,j_1},2}^*(\uidegree{\opensetforvertex{p,j_2},i})}}\right)
\stdvectorfiledpletter
$$
for $i=1,2$. Denote 
$$
f_i=
\frac
{\oluithreadfunction{\opensetforvertex{p,j_2},i}}
{\oluithreadfunction{\opensetforvertex{p,j_1},1}^{\uidegree{\opensetforvertex{p,j_1},1}^*(\uidegree{\opensetforvertex{p,j_2},i})}
\oluithreadfunction{\opensetforvertex{p,j_1},2}^{\uidegree{\opensetforvertex{p,j_1},2}^*(\uidegree{\opensetforvertex{p,j_2},i})}},
$$
then
$$
\uidescriptionfunction i=\frac{df_i}{f_i}\stdvectorfiledpletter.
$$
Let us find $\ord_p(f_i)$. We have 
$$
\ord_p(f_i)=\ord_p(\oluithreadfunction{\opensetforvertex{p,j_2},i})
-\uidegree{\opensetforvertex{p,j_1},1}^*(\uidegree{\opensetforvertex{p,j_2},i})\ord_p(\oluithreadfunction{\opensetforvertex{p,j_1},1})
-\uidegree{\opensetforvertex{p,j_1},2}^*(\uidegree{\opensetforvertex{p,j_2},i})\ord_p(\oluithreadfunction{\opensetforvertex{p,j_1},2}).
$$
Since $p\in W_p=V_{\opensetforvertex{p,j_1}}=V_{\opensetforvertex{p,j_2}}$,
we have 
$\ord_p(\oluithreadfunction{\opensetforvertex{p,j_1},i})=-\mathcal D_p(\uidegree{\opensetforvertex{p,j_1},i})=
-\eval_{\stdpolyhedronletter_p}(\uidegree{\opensetforvertex{p,j_1},i})$
and 
$\ord_p(\oluithreadfunction{\opensetforvertex{p,j_2},i})=-\mathcal D_p(\uidegree{\opensetforvertex{p,j_2},i})=
-\eval_{\stdpolyhedronletter_p}(\uidegree{\opensetforvertex{p,j_2},i})$.
By the definition of $\opensetforvertex{p,j_1}$ and of $\opensetforvertex{p,j_2}$, 
$\uidegree{\opensetforvertex{p,j_1},i}\in \normalvertexcone{\stdpolyhedronletter_p}{\indexedvertexpt p{j_1}}$ for $i=1,2$ 
and 
$\uidegree{\opensetforvertex{p,j_1},i}\in \normalvertexcone{\stdpolyhedronletter_p}{\indexedvertexpt p{j_2}}$ for $i=1,2$.
So, 
$\ord_p(\oluithreadfunction{\opensetforvertex{p,j_1},i})=-\uidegree{\opensetforvertex{p,j_1},i}(\indexedvertexpt p{j_1})$
and
$\ord_p(\oluithreadfunction{\opensetforvertex{p,j_2},i})=-\uidegree{\opensetforvertex{p,j_2},i}(\indexedvertexpt p{j_2})$.
Now,
\begin{multline*}
\ord_p(f_i)=-\uidegree{\opensetforvertex{p,j_2},i}(\indexedvertexpt p{j_2})
+\uidegree{\opensetforvertex{p,j_1},1}^*(\uidegree{\opensetforvertex{p,j_2},i})\uidegree{\opensetforvertex{p,j_1},1}(\indexedvertexpt p{j_1})
+\uidegree{\opensetforvertex{p,j_1},2}^*(\uidegree{\opensetforvertex{p,j_2},i})\uidegree{\opensetforvertex{p,j_1},2}(\indexedvertexpt p{j_1})=\\
-\uidegree{\opensetforvertex{p,j_2},i}(\indexedvertexpt p{j_2})+\uidegree{\opensetforvertex{p,j_2},i}(\indexedvertexpt p{j_1})=
\uidegree{\opensetforvertex{p,j_2},i}(\indexedvertexpt p{j_1}-\indexedvertexpt p{j_2}).
\end{multline*}

Consider also functions 
$$
f'_i=(t-t(p))^{-\uidegree{\opensetforvertex{p,j_2},i}(\indexedvertexpt p{j_1}-\indexedvertexpt p{j_2})}.
$$
Its logarithmic derivative equals
$$
\frac{df'_i}{f'_i}=-\uidegree{\opensetforvertex{p,j_2},i}(\indexedvertexpt p{j_1}-\indexedvertexpt p{j_2})\frac{dt}{t-t(p)},
$$
and
$$
\frac{df'_i}{f'_i}\stdvectorfiledpletter=-\uidegree{\opensetforvertex{p,j_2},i}(\indexedvertexpt p{j_1}-\indexedvertexpt p{j_2})\frac{1}{t-t(p)}.
$$
Clearly, $\ord_p(f'_i)=\uidegree{\opensetforvertex{p,j_2},i}(\indexedvertexpt p{j_1}-\indexedvertexpt p{j_2})$,
and $\ord_p(f_if'_i)=0$, so the logarithmic derivative of $f_if'_i$ is regular at $p$. We have
$$
\frac{df_if'_i}{f_if'_i}\stdvectorfiledpletter=
\frac{df_i}{f_i}+\frac{df'_i}{f'_i}=
\uidescriptionfunction i-\uidegree{\opensetforvertex{p,j_2},i}(\indexedvertexpt p{j_1}-\indexedvertexpt p{j_2})\frac{1}{t-t(p)}.
$$
\end{proof}

\begin{corollary}\label{recalcregularitycor}
Let $p$ be an essential special point, $1\le j_1\le\numberofverticespt p$, $1\le j_2\le\numberofverticespt p$.
Let $\stdvectorfiledxletter$ be the vector field on $U_\numberoffixedopensets$
with $U_{\opensetforvertex{p,j_1}}$-description $(\uidescriptionfunction{j_1,1},\uidescriptionfunction{j_1,2},\stdvectorfiledpletter)$, 
where $\stdvectorfiledpletter=\partial/\partial t$ (recall that $t$ is defined at all \emph{essential}
special points) and $\uidescriptionfunction{j_1,i}$ are regular at $p$.

Let $(\uidescriptionfunction{j_2,1},\uidescriptionfunction{j_2,2},\stdvectorfiledpletter)$ be the 
$U_{\opensetforvertex{p,j_2}}$-description of $\stdvectorfiledxletter$.
Then 
$$
\uidescriptionfunction{j_2,1}-\uidegree{\opensetforvertex{p,j_2},1}(\indexedvertexpt{p}{j_1}-\indexedvertexpt{p}{j_2})\frac1{t-t(p)}
$$
and
$$
\uidescriptionfunction{j_2,2}-\uidegree{\opensetforvertex{p,j_2},2}(\indexedvertexpt{p}{j_1}-\indexedvertexpt{p}{j_2})\frac1{t-t(p)}
$$
are rational functions on $\PP^1$ regular at $p$.
\end{corollary}

\begin{proof}
Set
$$
\left(
\begin{array}{c}
\uidescriptionfunction{j_2,1}'\\
\uidescriptionfunction{j_2,2}'\\
\end{array}
\right)=
\uismalltransition{\opensetforvertex{p,j_1},\opensetforvertex{p,j_2}}
\left(
\begin{array}{c}
\uidescriptionfunction{j_1,1}\\
\uidescriptionfunction{j_1,2}\\
\end{array}
\right)
$$
and
$$
\left(
\begin{array}{c}
\uidescriptionfunction{j_2,1}''\\
\uidescriptionfunction{j_2,2}''\\
\stdvectorfiledpletter
\end{array}
\right)=
\uilargetransition{\opensetforvertex{p,j_1},\opensetforvertex{p,j_2}}
\left(
\begin{array}{c}
0\\
0\\
\stdvectorfiledpletter
\end{array}
\right).
$$
Then $\uidescriptionfunction{j_2,i}=\uidescriptionfunction{j_2,i}'+\uidescriptionfunction{j_2,i}''$.
Since the entries of $\uismalltransition{\opensetforvertex{p,j_1},\opensetforvertex{p,j_2}}$
are constants, the functions $\uidescriptionfunction{j_2,1}'$ and $\uidescriptionfunction{j_2,2}'$
are regular at $p$. The claim follows from Lemma \ref{recalcregularity}.
\end{proof}

We need to introduce a notation. Let $p$ be an essential special point. 
Let $v=\partial/\partial t$ be a vector field on $\PP^1$.
Let $(\uidescriptionfunction1,\uidescriptionfunction2,\stdvectorfiledpletter)$
be the $U_{\numberoffixedopensets}$-description of the vector field on $U_{\numberoffixedopensets}$
with the $U_{\opensetforvertex{p,1}}$-description $(0,0,\stdvectorfiledpletter)$.
It follows from the form of the matrix 
$\uilargetransition{\numberoffixedopensets,\opensetforvertex{p,1}}$
and properties of logarithmic derivatives
that the functions 
$\uidescriptionfunction1$ and $\uidescriptionfunction2$ have poles of order at most one at $p$. 
So, functions $(t-t(p))\uidescriptionfunction i$ ($i=1,2$) are regular at $p$. 
Denote their values at $p$ by $a^{(2)}_{p,1}$ and $a^{(2)}_{p,2}$, respectively.
Then functions $\uidescriptionfunction i-a^{(2)}_{p,i}(t-t(p))^{-1}$ ($i=1,2$) are regular at $p$.
We keep this notation $a^{(2)}_{p,1}$ and $a^{(2)}_{p,2}$
until the end of the section, while $p$, $\stdvectorfiledpletter$, 
$\uidescriptionfunction1$, and $\uidescriptionfunction2$ will be used in the sequel to denote other objects.


\begin{lemma}\label{firstisregular}
Let $p$ be an essential special point, $\stdvectorfiledpletter=\partial/\partial t$, 
and let $\stdvectorfiledxletter$ be the vector field on $U_{\numberoffixedopensets}$
with $U_{\numberoffixedopensets}$-description $(a^{(2)}_{p,1}(t-t(p))^{-1}, a^{(2)}_{p,2}(t-t(p))^{-1}, \stdvectorfiledpletter)$.
Let $(\uidescriptionfunction1,\uidescriptionfunction2, \stdvectorfiledpletter)$
be the $U_{\opensetforvertex{p,1}}$-description of $\stdvectorfiledxletter$.
Then $\uidescriptionfunction1$ and $\uidescriptionfunction2$ are regular at $p$.
\end{lemma}

\begin{proof}
Set
$$
\left(
\begin{array}{c}
\uidescriptionfunction{\numberoffixedopensets,1}\\
\uidescriptionfunction{\numberoffixedopensets,2}\\
\stdvectorfiledpletter
\end{array}
\right)=
\uilargetransition{\opensetforvertex{p,1},\numberoffixedopensets}
\left(
\begin{array}{c}
0\\
0\\
\stdvectorfiledpletter
\end{array}
\right).
$$
Then
$$
\left(
\begin{array}{c}
0\\
0\\
\stdvectorfiledpletter
\end{array}
\right)=
\uilargetransition{\numberoffixedopensets,\opensetforvertex{p,1}}
\left(
\begin{array}{c}
\uidescriptionfunction{\numberoffixedopensets,1}\\
\uidescriptionfunction{\numberoffixedopensets,2}\\
\stdvectorfiledpletter
\end{array}
\right)
$$
and
$$
\left(
\begin{array}{c}
\uidescriptionfunction1\\
\uidescriptionfunction2\\
\stdvectorfiledpletter
\end{array}
\right)=
\uilargetransition{\numberoffixedopensets,\opensetforvertex{p,1}}
\left(
\begin{array}{c}
(a^{(2)}_{p,1}(t-t(p))^{-1}\\
(a^{(2)}_{p,2}(t-t(p))^{-1}\\
\stdvectorfiledpletter
\end{array}
\right),
$$
so
$$
\left(
\begin{array}{c}
\uidescriptionfunction1\\
\uidescriptionfunction2\\
\end{array}
\right)=
\uismalltransition{\numberoffixedopensets,\opensetforvertex{p,1}}
\left(
\begin{array}{c}
(a^{(2)}_{p,1}(t-t(p))^{-1}-\uidescriptionfunction{\numberoffixedopensets,1}\\
(a^{(2)}_{p,2}(t-t(p))^{-1}-\uidescriptionfunction{\numberoffixedopensets,2}\\
\end{array}
\right).
$$
Functions $(a^{(2)}_{p,i}(t-t(p))^{-1}-\uidescriptionfunction{\numberoffixedopensets,i}$
are regular at $p$, the entries of $\uismalltransition{\numberoffixedopensets,\opensetforvertex{p,1}}$
are constants,
so $\uidescriptionfunction1$ and $\uidescriptionfunction2$ are regular at $p$.
\end{proof}

Consider the following elements of $\abstractvectorspace_{3,3}$: For each essential 
special point $p_j$ set
$$
s_{3,5,j}=\sum_{l=1}^{\numberofverticespt{p_j}-1}s_{l,j,1}.
$$
Denote by $\abstractvectorspace_{3,5}$ the subspace of $\abstractvectorspace_{3,3}$ spanned by all $s_{3,5,j}$.

We are going to prove that $\im (\abstractvectorspace_{3,5}\to H^1(U,\Theta_U))=\im(H^1(\PP^1,\gi)\to H^1(U,\Theta_U))$.
As before, we will replace $\abstractvectorspace_{3,5}$ by another vector space that will represent the same subspace 
of $H^1(U,\Theta_U)$.
Namely, for each essential special point $p_j$ denote by $s_{3,6,j}$ the following element of $\abstractvectorspace_{3,0}$.
$$
s_{3,6,j}=
(\uidescriptionfunctiondiff11,\uidescriptionfunctiondiff12,\stdvectorfiledpletter[1],\ldots, 
\uidescriptionfunctiondiff{\numberoffixedopensets-1}1,\uidescriptionfunctiondiff{\numberoffixedopensets-1}2,\stdvectorfiledpletter[\numberoffixedopensets-1]),
$$
where:
\begin{enumerate}
\item If $i=\opensetforvertex{p_j,k}$, where $1\le k\le \numberofverticespt{p_j'}$, then
$(\uidescriptionfunctiondiff i1,\uidescriptionfunctiondiff i2,\stdvectorfiledpletter[i])$
is the $U_i$-description of the vector field on $U_{\numberoffixedopensets}$
with the $U_{\numberoffixedopensets}$-description $(a^{(2)}_{p_j,1}(t-t(p_j))^{-1},a^{(2)}_{p_j,2}(t-t(p_j))^{-1},\partial/\partial t)$.
\item Otherwise (if $i$ is not of this form),
$\uidescriptionfunctiondiff{i}1=\uidescriptionfunctiondiff{i}2=\stdvectorfiledpletter[i]=0$.
\end{enumerate}

\begin{lemma}
For each essential special point $p_j$, 
$s_{3,5,j}$ and $s_{3,6,j}$
define the same class in 
$\bigoplus_{i=1}^{\numberofabstractopensets-1}(\Gamma(U_{\numberoffixedopensets},\Theta_U)/\Gamma(U_i,\Theta_U))$.
\end{lemma}

\begin{proof}
Let
$$
s_{3,5,j}=
(\uidescriptionfunctiondiff11,\uidescriptionfunctiondiff12,\stdvectorfiledpletter[1],\ldots, 
\uidescriptionfunctiondiff{\numberoffixedopensets-1}1,\uidescriptionfunctiondiff{\numberoffixedopensets-1}2,\stdvectorfiledpletter[\numberoffixedopensets-1])
$$
and
$$
s_{3,6,j}=
(\uidescriptionfunctiondiff11',\uidescriptionfunctiondiff12',\stdvectorfiledpletter[1]',\ldots, 
\uidescriptionfunctiondiff{\numberoffixedopensets-1}1',\uidescriptionfunctiondiff{\numberoffixedopensets-1}2',\stdvectorfiledpletter[\numberoffixedopensets-1]').
$$
It is sufficient to prove the following: for each $k$ ($1\le k\le \numberofverticespt{p_j}$), 
$$
(\uidescriptionfunctiondiff{\opensetforvertex{p_j,k}}1'-\uidescriptionfunctiondiff{\opensetforvertex{p_j,k}}1,
\uidescriptionfunctiondiff{\opensetforvertex{p_j,k}}2'-\uidescriptionfunctiondiff{\opensetforvertex{p_j,k}}2,
\stdvectorfiledpletter[\opensetforvertex{p_j,k}]'-\stdvectorfiledpletter[\opensetforvertex{p_j,k}])
$$
is the $U_{\opensetforvertex{p_j,k}}$-description of a vector field defined on $U_{\opensetforvertex{p_j,k}}$.
In other words, we have to check that the functions 
$\uidescriptionfunctiondiff{\opensetforvertex{p_j,k}}1'-\uidescriptionfunctiondiff{\opensetforvertex{p_j,k}}1$
and 
$\uidescriptionfunctiondiff{\opensetforvertex{p_j,k}}2'-\uidescriptionfunctiondiff{\opensetforvertex{p_j,k}}2$
are regular at $p_j$
(for $\stdvectorfiledpletter[\opensetforvertex{p_j,k}]'-\stdvectorfiledpletter[\opensetforvertex{p_j,k}]=\partial/\partial t$
this is clear).

First, let us find a precise expression for $\uidescriptionfunctiondiff{\opensetforvertex{p_j,k}}1$ and $\uidescriptionfunctiondiff{\opensetforvertex{p_j,k}}2$.
By the definition of $s'_{3,3,k,j,1}$, we have
$$
\uidescriptionfunctiondiff{\opensetforvertex{p_j,k}}{j'}=
-\frac{1}
{t-t(p_j)}
\sum_{l=1}^{k-1}
\uidegree{\opensetforvertex{p_j,k},j'}(\indexedvertexpt{p_j}{l+1}-\indexedvertexpt{p_j}{l})=
-\frac{1}
{t-t(p_j)}
\uidegree{\opensetforvertex{p_j,k},j'}(\indexedvertexpt{p_j}k-\indexedvertexpt{p_j}1).
$$
(Note that for $k=1$ we get 
$\uidescriptionfunctiondiff{\opensetforvertex{p_j,1}}1=\uidescriptionfunctiondiff{\opensetforvertex{p_j,k}}2=0$.)

For $k=1$, the functions 
$\uidescriptionfunctiondiff{\opensetforvertex{p_j,k}}1'-\uidescriptionfunctiondiff{\opensetforvertex{p_j,k}}1=\uidescriptionfunctiondiff{\opensetforvertex{p_j,k}}1'$
and 
$\uidescriptionfunctiondiff{\opensetforvertex{p_j,k}}2'-\uidescriptionfunctiondiff{\opensetforvertex{p_j,k}}2=\uidescriptionfunctiondiff{\opensetforvertex{p_j,k}}2'$
are regular at $p$ by Lemma \ref{firstisregular}.
For other values of $k$, we have
$$
\uidescriptionfunctiondiff{\opensetforvertex{p_j,k}}{j'}'-\uidescriptionfunctiondiff{\opensetforvertex{p_j,k}}{j'}=
\uidescriptionfunctiondiff{\opensetforvertex{p_j,k}}{j'}'
-\frac{1}
{t-t(p_j)}
\uidegree{\opensetforvertex{p_j,k},j'}(\indexedvertexpt{p_j}1-\indexedvertexpt{p_j}k).
$$
These functions are regular at $p$ by Corollary \ref{recalcregularitycor}
since $(\uidescriptionfunctiondiff{\opensetforvertex{p_j,k}}1', \uidescriptionfunctiondiff{\opensetforvertex{p_j,k}}2', \partial/\partial t)$
is the $U_{\opensetforvertex{p_j,k}}$-description of 
a vector field on $U_{\numberoffixedopensets}$
with $U_{\opensetforvertex{p_1,1}}$-description 
$(\uidescriptionfunctiondiff{\opensetforvertex{p_j,1}}1', \uidescriptionfunctiondiff{\opensetforvertex{p_j,1}}2', \partial/\partial t)$,
and functions 
$\uidescriptionfunctiondiff{\opensetforvertex{p_j,k}}1'$
and 
$\uidescriptionfunctiondiff{\opensetforvertex{p_j,k}}2'$
are regular at $p$.
\end{proof}

\begin{corollary}\label{avs35issubspace}
For each essential special point $p_j$, 
$s_{3,6,j}\in \abstractvectorspace_{3,1}$.
Moreover, 
$s_{3,5,j}$ and $s_{3,6,j}$
define the same classes in $H^1(U,\Theta_U)$.\qed
\end{corollary}

Denote the subspace of $\abstractvectorspace_{3,1}$
generated by all $s_{3,6,j}$ by $\abstractvectorspace_{3,6}$.
By Corollary \ref{avs35issubspace},
$\im(\abstractvectorspace_{3,6}\to H^1(U,\Theta_U))$
is a subspace of 
$\im(\abstractvectorspace_{3,2}\to H^1(U,\Theta_U))$.
We will prove that 
$\im(\abstractvectorspace_{3,6}\to H^1(U,\Theta_U))=\im(H^1(\PP^1,\gi)\to H^1(U,\Theta_U))$.

Let us recall the results of Chapter \ref{combformula} related with $H^1(\PP^1,\gi)$.
There we have introduced vector spaces 
$$
\abstractvectorspace_{0,1}=\bigoplus_{\text{$p$ essential special point}}\Theta_{\PP^1,p}
$$
and (for each special point $p$) $\abstractvectorspace_{0,0,p}$, which was the space of triples of Laurent polynomials in $t_p$
of a certain form, where the first two polynomials were rational functions on $\PP^1$, and 
the last one was a rational vector field on $\PP^1$. $\abstractvectorspace_{0,1}$ was mapped to 
$\bigoplus_{\text{$p$ special point}}\abstractvectorspace_{0,0,p}$, 
namely, a sequence of tangent vectors $(\uidescriptionfunction p\partial/\partial t_p)_{\text{$p$ essential special point}}$,
where $\uidescriptionfunction p\in \CC$, was mapped to a sequence of rational functions and vector fields on $\PP^1$, 
where all functions are zeros, and the vector fields on $\PP^1$ defined by the same formulas 
(plus zero vector fields for removable special points). 
$\bigoplus_{\text{$p$ special point}}\abstractvectorspace_{0,0,p}$ 
was further mapped to $H^1(\PP^1,\gi)$. Then we proved (Lemma \ref{honepushforwardnumber})
that $\abstractvectorspace_{0,1}$ is mapped to $H^1(\PP^1,\gi)$
surjectively.

The space $\abstractvectorspace_{0,1}$
has the following obvious basis: for each $j$ such that $p_j$ is an essential 
special point, 
let $s_{1,j}\in \abstractvectorspace_{0,1}$ be the sequence with the $p_j$th 
entry $\partial/\partial t$ (recall that $t_{p_j}=t-t(p_j)$ for essential special points)
and all other entries are zeros. The image of $s_{1,j}$ in $\abstractvectorspace_{0,0,p_j}$
is $(0,0,\partial/\partial t)$, and the image of $s_{1,j}$ in $\abstractvectorspace_{0,0,p}$
with $p\ne p_j$ is $(0,0,0)$.

We also checked (Lemma \ref{functreduction})
that if we change the first two entries of an element of $\abstractvectorspace_{0,0,p}$, 
where $p$ is a special point, arbitrarily, then the class of this element in 
$H^1(\PP^1,\gi)$ will not change. So, 
if $p_j$ is an essential special point, set
$s'_{1,j}=(a^{(2)}_{p_j,1}(t-t(p_j))^{-1},a^{(2)}_{p_j,2}(t-t(p_j))^{-1},\partial/\partial t)\in \abstractvectorspace_{0,0,p_j}$.
Denote the subspace of $\bigoplus_{j=1}^{\numberofdivisorpoints}\abstractvectorspace_{0,0,p_j}$
spanned by all $s'_{1,j}$ by $\abstractvectorspace_{0,2}$. 
By Lemma \ref{functreduction}, $\abstractvectorspace_{0,2}$ is mapped surjectively onto $H^1(\PP^1,\gi)$.

The sheaf $\gi$ was constructed as follows. Its sections on an open set $V\subseteq \PP^1$
were sequences of length $2\numberoffixedopensets+1$, where the first $2\numberoffixedopensets$ entries were 
rational functions on $\PP^1$ and the last entry was a rational vector field on $\PP^1$. More 
precisely, for each $i$ ($1\le i\le \numberoffixedopensets$) the $(2i-1)$th, the $2i$th, and the $(2\numberoffixedopensets+1)$th
entries from the $U_i$-description of the same (i.~e. not depending on $i$) vector field on $U\cap \pi^{-1}(V)$.
For each special point $p$ we had a morphism $\abstractvectorspace_{0,0,p}\to \Gamma(W,\gi)$, 
which computed all $U_i$-descriptions of a vector field by its $U_{\numberoffixedopensets}$-description.
Then these morphisms were summed up to a map
\begin{multline*}
\bigoplus_{j=1}^{\numberofdivisorpoints}\abstractvectorspace_{0,0,p}\to \bigoplus_{j=1}^{\numberofdivisorpoints} \Gamma(W,\gi)\\
\longrightarrow 
\left.\!\left(\bigoplus_{j=1}^{\numberofdivisorpoints}\Big(\Gamma(W,\gi)/\Gamma(W_{p_j},\gi)\Big)\right)\right/\Gamma(W,\gi)=H^1(\PP^1,\gi),
\end{multline*}
where the second arrow is the canonical projection.

We also had a sheaf $\giinv$, which was the zeroth graded component of 
$\pi_*\Theta_U$. And we had an isomorphism $\gi\to \giinv$, which computed vector fields out of 
their $U_i$-descriptions.

Finally, we need to understand the map 
$H^1(\PP^1,\giinv)\to H^1(U,\Theta_U)$. We have affine 
coverings $\{W_p\}_{\text{$p$ special point}}$ of $\PP^1$ and $\{U_i\}_{1\le i\le \numberoffixedopensets-1}$
of $U$.
We interpret 
$H^1(\PP^1,\giinv)$ as a quotient of 
$\bigoplus_{j=1}^{\numberofdivisorpoints}\Gamma(W,\giinv)$
and 
$H^1(U,\Theta_U)$
as a subquotient of 
$\bigoplus_{i=1}^{\numberoffixedopensets-1}\Gamma(U_{\numberoffixedopensets},\Theta_U)$.
As it was explained in Section \ref{sectleray},
to describe the map 
$H^1(\PP^1,\giinv)\to H^1(U,\Theta_U)$,
we need to enumerate the sets $U_i$ by pairs of indices so that the 
first index in such a pair corresponds to one of the open sets from the
affine covering of $\PP^1$.
For such an enumeration, we use the notation $\opensetforvertex{p,j}$.
Namely, recall that for each $i$ ($1\le i\le \numberoffixedopensets-1$)
there exists a (removable or essential) special point $p$ 
and an index $j$ ($1\le j\le \numberofverticesptdiff p$)
such that $i=\opensetforvertex{p,j}$. 
So, denote $U_{(p,j)}=U_{\opensetforvertex{p,j}}$ for all special points $p$ and
for all $j$ ($1\le j\le \numberofverticesptdiff p$).
Then $U_{(p,j)}\subseteq\pi^{-1}(W_p)$, and the conditions of Section \ref{sectleray} are satisfied.
(Note that the set that was denoted in the "generic" situation of Section \ref{sectleray}
by $U$ is now $U_{\numberoffixedopensets}$, and the set that was denoted in the 
the "generic" situation of Section \ref{sectleray}
by $V$ is now $W$).
After we have introduced these notations, 
we can say that the map $H^1(\PP^1,\giinv)\to H^1(U,\Theta_U)$
is induced by the following map 
$\bigoplus_{i=1}^{\numberofdivisorpoints}\Gamma(W,\giinv)\to 
\bigoplus_{i=1}^{\numberofdivisorpoints}\bigoplus_{j=1}^{\numberofverticesptdiff{p_i}}\Gamma(U_{\numberoffixedopensets},\Theta_U)$.
The $(p_i,j)$th entry of the result is the $i$th entry of the preimage restricted to $U_{\numberoffixedopensets}$
(originally it was a vector field on $\pi^{-1}(W)\supseteq U_{\numberoffixedopensets}$).

Summarizing, we see that the map 
$\bigoplus_{j=1}^{\numberofdivisorpoints}\abstractvectorspace_{0,0,p_j}\to H^1(U,\Theta_U)$
is induced by
the following map 
$\bigoplus_{j=1}^{\numberofdivisorpoints}\abstractvectorspace_{0,0,p_j}\to \bigoplus_{i=1}^{\numberoffixedopensets} \Gamma(U_{\numberoffixedopensets},\Theta_U)$.
The $\opensetforvertex{p_j,k}$th entry of the result is the vector field whose $U_{\numberoffixedopensets}$-description is the $j$th entry 
of the preimage.
In particular, each $s'_{1,j}$ (for each essential special point $p_j$) is mapped to the following 
sequence. If $i=\opensetforvertex{p_j,k}$ for some $k$ ($1\le k\le \numberofverticespt{p_j}$), 
then the $i$th entry of the result is the vector field on $U_{\numberoffixedopensets}$
with the $U_{\numberoffixedopensets}$-description 
$(a^{(2)}_{p_j,1}(t-t(p_j))^{-1},a^{(2)}_{p_j,2}(t-t(p_j))^{-1},\partial/\partial t)$.
Otherwise (for other values of $i$), the $i$th entry of the result is 0.
By the definition of $s_{3,6,j}$, the image of $s_{3,6,j}$
in $\bigoplus_{i=1}^{\numberoffixedopensets} \Gamma(U_{\numberoffixedopensets},\Theta_U)$
is the same. Therefore, 
$\im(\abstractvectorspace_{3,6}\to H^1(U,\Theta_U))=\im(\abstractvectorspace_{0,2}\to H^1(U,\Theta_U))$, 
and we get the following proposition.

\begin{proposition}\label{firstsurjectivity}
$\im(\abstractvectorspace_{3,2}\to H^1(U,\Theta_U))$
contains $\im(H^1(\PP^1,\gi)\to H^1(U,\Theta_U))$.\qed
\end{proposition}

The following proposition follows from Propositions \ref{secondsurjectivity}
and \ref{firstsurjectivity}.

\begin{proposition}\label{propksmsurj}
The deformation $\bigflatmorphism\colon\bigtotalspace\to \coefficientspace$ of $X$ constructed in Section \ref{sectversalconstruction} 
has surjective Kodaira-Spencer map.\qed
\end{proposition}

Finally, let us recall the definition of a formally versal deformation. A deformation 
$\bigflatmorphism'\colon\bigtotalspace'\to \coefficientspace'$ of $X$ with the basepoint 
$\stdpolynomialcoefficient^{(3)}\in\coefficientspace'$
is called \textit{formally versal in the class of $T$-equivariant
deformations} if the following holds.

Let $\bigflatmorphism''\colon\bigtotalspace''\to \coefficientspace''$
be another $T$-equivariant deformation of $X$, and let $\stdpolynomialcoefficient^{(4)}\in\coefficientspace''$
be the basepoint of this deformation.

Denote by $\widetilde{\coefficientspace'}$
the formal neighborhood of $\stdpolynomialcoefficient^{(3)}$ in $\coefficientspace'$.
Denote by $\widetilde{\bigflatmorphism'}\colon \widetilde{\bigtotalspace'}\to \widetilde{\coefficientspace'}$
the restriction of the deformation $\bigflatmorphism'$ to $\widetilde{\coefficientspace'}$.

Similarly, let $\widetilde{\coefficientspace''}$ be the formal neighborhood of 
$\stdpolynomialcoefficient^{(4)}$ in $\coefficientspace''$, 
and let 
$\widetilde{\bigflatmorphism''}\colon \widetilde{\bigtotalspace''}\to \widetilde{\coefficientspace''}$
be the restriction of the deformation $\bigflatmorphism''$ to $\widetilde{\coefficientspace''}$.

Then fromal versality means that there exists a morphism $f\colon \widetilde{\coefficientspace''}\to \widetilde{\coefficientspace'}$
such that the deformation $\widetilde{\bigflatmorphism''}$
is the pullback of the deformation $\widetilde{\bigflatmorphism'}$
via this map $f$.

\begin{proposition}\label{ksmsurjimpliesversal}
Let $X$ be a $T$-variety, let $\coefficientspace$ be a vector space, 
and let $\bigflatmorphism\colon\bigtotalspace\to \coefficientspace$ be an equivariant
deformation. Suppose that the marked point of this deformation is the origin in $\coefficientspace$. 
Suppose that the Kodaira-Spencer map for this deformation is surjective onto $T^1(X)_0$, 
which is finite dimensional.

Then $\bigflatmorphism\colon\bigtotalspace\to \coefficientspace$ is an equivariant formally versal deformation of $X$.
\end{proposition}

\begin{proof}[Idea of a proof]
First, one can check that a formally versal deformation exists using
\cite[Theorem 2.11]{schlessdef}. 
The conditions ($\mathrm H_1$) and ($\mathrm H_2$) are verified exactly 
in the same way as they are verified for non-equivariant deformation, see 
Section 3.7 of \cite{schlessdef}. One has to use graded algebras and equivariant maps between 
them, but the arguments stay the same. Condition ($\mathrm H_3$)
is our assumption that $T^1(X)_0$
is finite dimensional.
The parameter space (denote it by $Y$) of a formally versal deformation we can obtain this way 
is the spectrum of a complete Noetherian local algebra. By Cohen structure theorem, 
$\CC[Y]$ is a quotient of a formal power series ring over $\CC$ in finitely many variables. 
Note that it is not true in general that $\CC[Y]$ is a finitely generated $\CC$-algebra 
(i.~e. a quotient of a polynomial ring).
In the proof of this proposition, choose and denote by $b_1,\ldots, b_m$ a set of variables
such that $\CC[Y]$ is a quotient of $\CC[[b_1,\ldots, b_m]]$, 
and the maximal ideal of $\CC[Y]$ is the image of $(b_1,\ldots, b_m)$.

$T^1(X)_0$ can be identified with the tangent space of $Y$ at the geometric point (see \cite[Definition 2.7]{schlessdef}).
Denote by $W$ the vector space with coordinates $b_1. \ldots, b_m$.
Then the tangent space of $Y$ at the geometric point
becomes a subspace of $W$. After a linear change of variables we may 
suppose that this tangent space
is defined by the equations $b_{n+1}=\ldots=b_m=0$. 
Then $b_1,\ldots, b_n$ are coordinates on $T^1(X)_0$.

%


In the proof of this proposition, we denote the dimension of $\dim T^1(X)_0$ by $n$.

Also, in the proof of this proposition we can suppose without loss of generality 
that the Kodaira-Spencer map is an isomorphism 
(otherwise we can replace $\coefficientspace$ with a complement to the kernel of the Kodaira-Spencer map).

 
Let $f$ be a morphism from the formal neighborhood of zero in $\coefficientspace$ to 
$Y$ such that $\bigflatmorphism\colon\bigtotalspace\to \coefficientspace$ is the 
pullback via $f$ of the formally versal equivariant deformation over $Y$. Then $df$ is the Kodaira-Spencer map.

Choose coordinates $a_1,\ldots, a_n$ in $\coefficientspace$. Then $f$ can be written using $m$ 
power series in the variables $a_i$. Denote these power series by $f_1,\ldots, f_m$
so that $b_i=f_i(a_1,\ldots, a_n)$. These series do not have constant terms. 
The first $n$ of them have nontrivial linear terms, the last $m-n$ power series 
do not have terms of degree less than two.

Since the Kodaira-Spencer map is an isomorphism, without loss of generality (after a suitable linear change of coordinates 
in $\coefficientspace$) we may suppose that the linear term in $f_i$, where $1\le i\le n$, is exactly $a_i$.
In other words, $b_i=a_i+(\text{terms of degree ${}\ge 2$})$ for $1\le i\le n$.

Now, using iterated corrections in higher and higher degrees, we can find power series $g_1,\ldots, g_n$ in $b_1,\ldots, b_n$
(the variables $b_{n+1},\ldots, b_m$ will not appear there)
such that $g_i(f_1,\ldots, f_n)=a_i$. In other words, the map $f$ between 
the formal neighborhoods of the marked points is invertible, in other words, it is an isomorphism. 
Hence, $\bigflatmorphism\colon\bigtotalspace\to \coefficientspace$ is also an equivariant formally versal deformation.
\end{proof}

\begin{remark}
In fact, this proposition holds true if $\coefficientspace$ is smooth, but not necessarily a vector space. The proof is more complicated 
in this case.
\end{remark}

Therefore, we get the following theorem from Theorem \ref{t1aslatticelength}, Proposition \ref{propksmsurj}, 
and Proposition \ref{ksmsurjimpliesversal}.

\begin{theorem}
The deformation $\bigflatmorphism\colon\bigtotalspace\to \coefficientspace$ of $X$ constructed in Section \ref{sectversalconstruction} 
is formally versal in the class of $T$-equivariant deformations.\qed
\end{theorem}

\appendix

\chapter*{Acknowledgments}\addchaptertocentry{}{Acknowledgments}

I thank a lot my academic supervisor Klaus Altmann for bringing my attention to the problem, for 
useful discussions, and for the attention he paid to my work.
I thank Robert Vollmert for a good introduction to T-varieties.
I also thank Sergey Loktev from Moscow for paying attention to my work and Valentina Kiritchenko 
from Moscow for bringing my attention to papers about T-varieties in the very beginning of (and even before) my PhD studies.
I thank Alexander Schmidt and H\'el\`ene Esnault for answering my questions in algebraic geometry.
I thank Jan Stevens, Dmitriy Kaledin, Duco van Straten, and Jan Christophersen
for useful discussions on deformation theory.
Finally, I thank members of my research group, namely 
Nikolai Beck,
Ana Maria Botero,
Alexandru Constantinescu, 
Joana Cirici,
Maria Donten-Bury,
Matej Filip, 
Alejandra Rinc\'on Hidalgo,
Victoria Hoskins, 
Lars Kastner, 
Marianne Merz, 
Mateusz Michalek,
Lars Petersen, 
Irem Portakal,
Eva Mart\'inez Romero,
Giangiacomo Sanna,
Richard Sieg, and
Anna-Lena Winz
for answering various mathematical and technical questions.

\end{document}